\documentclass[11pt,reqno]{amsart}
\usepackage{amssymb,amsmath,amsfonts,amsthm,enumerate,stmaryrd,tensor,mathtools,dsfont,upgreek,bbm,mathrsfs,nameref,microtype,xcolor,bm,tikz}
\usetikzlibrary{arrows}
\usepackage[left=1 in, right=1 in,top=1 in, bottom=1 in]{geometry}

\usepackage{imakeidx}%This allows the index to be referenced
\usepackage{hyperref,cleveref}%

\numberwithin{equation}{subsection}
\definecolor{popblue}{RGB}{55,115,255}
\definecolor{lightbl}{RGB}{155,205,255}
\definecolor{depthbl}{RGB}{145,215,255}
\definecolor{fancyre}{RGB}{225,55,115}
\definecolor{darkblu}{RGB}{15,75,185}
\definecolor{mellowy}{RGB}{225,225,35}
\renewcommand{\j}[1]{\tjump{#1}}
\renewcommand{\tilde}[1]{\widetilde{#1}}
\renewcommand{\Bar}{\overline}

\renewcommand{\S}{\mathbb{S}}
\newcommand{\R}{\mathbb{R}}
\newcommand{\N}{\mathbb{N}}

\newcommand{\C}{\mathbb{C}}
\newcommand{\X}{\mathbb{X}}
\newcommand{\Y}{\mathbb{Y}}
\newcommand{\W}{\mathbb{W}}
\newcommand{\E}{\mathbb{E}}
\newcommand{\F}{\mathbb{F}}
\newcommand{\m}{\mathrm}

\newcommand{\lv}{\lVert}
\newcommand{\rv}{\rVert}

\newcommand{\al}{\alpha}
\newcommand{\be}{\beta}
\newcommand{\es}{\varnothing}
\newcommand{\lra}{\;\Leftrightarrow\;}
\newcommand{\ep}{\varepsilon}
\newcommand{\f}{\frac}
\newcommand{\sig}{\sigma}
\newcommand{\gam}{\gamma}
\newcommand{\del}{\delta}

\newcommand{\bn}{\binom}

\newcommand{\pd}{\partial}

\newcommand{\grad}{\nabla}
\newcommand{\bpm}{\begin{pmatrix}}
\newcommand{\epm}{\end{pmatrix}}

\newcommand{\loc}{\m{loc}}
\renewcommand{\bar}{\overline}

\newcommand{\emb}{\hookrightarrow}
\newcommand{\res}{\restriction}
\renewcommand{\le}{\leqslant}
\renewcommand{\ge}{\geqslant}
\newcommand{\tjump}[1]{\llbracket#1\rrbracket}

\newcommand{\norm}[1]{\left\lv#1\right\rv}
\newcommand{\bnorm}[1]{\Big\lv#1\Big\rv}
\newcommand{\snorm}[1]{\big\lv#1\big\rv}
\newcommand{\tnorm}[1]{\lv#1\rv}
\newcommand{\p}[1]{\left(#1\right)}
\newcommand{\bp}[1]{\Big(#1\Big)}
\renewcommand{\sp}[1]{\big(#1\big)}
\newcommand{\tp}[1]{(#1)}
\newcommand{\tfloor}[1]{\lfloor #1\rfloor}
\newcommand{\sfloor}[1]{\big\lfloor#1\big\rfloor}
\newcommand{\bfloor}[1]{\Big\lfloor#1\Big\rfloor}

\newcommand{\abs}[1]{\left|#1\right|}
\newcommand{\babs}[1]{\Big|#1\Big|}

\newcommand{\tabs}[1]{|#1|}
\newcommand{\bsb}[1]{\Big[{#1}\Big]}
\newcommand{\ssb}[1]{\big[{#1}\big]}
\newcommand{\tsb}[1]{[{#1}]}
\newcommand{\cb}[1]{\left\{{#1}\right\}}
\newcommand{\scb}[1]{\big\{{#1}\big\}}
\newcommand{\bcb}[1]{\Big\{{#1}\Big\}}
\newcommand{\tcb}[1]{\{{#1}\}}
\newcommand{\br}[1]{\left\langle #1 \right\rangle}
\providecommand{\bbr}[1]{\Big\langle #1 \Big\rangle}

\providecommand{\tbr}[1]{\langle #1 \rangle}
\renewcommand{\bf}[1]{\mathbf{#1}}
\newcommand{\ii}{\m{i}}
\DeclareMathOperator{\supp}{supp}

\newtheorem{prop}{\color{popblue}{Proposition}}[section]
\newtheorem{thm}[prop]{\color{popblue}{Theorem}}
\newtheorem{defn}[prop]{\color{popblue}{Definition}}
\newtheorem{lem}[prop]{\color{popblue}{Lemma}}
\newtheorem{coro}[prop]{\color{popblue}{Corollary}}
\newtheorem{rmk}[prop]{\color{popblue}{Remark}}
\newtheorem{exa}[prop]{\color{popblue}{Example}}

\newtheorem{innercustomthm}{\color{popblue}{Theorem}}
\newenvironment{customthm}[1]
{\renewcommand\theinnercustomthm{#1}\innercustomthm}
{\endinnercustomthm}
\newtheorem{innercustomcoro}{\color{popblue}{Corollary}}
\newenvironment{customcoro}[1]
{\renewcommand\theinnercustomcoro{#1}\innercustomcoro}
{\endinnercustomcoro}
\author{Noah Stevenson}
\address{
	Department of Mathematics\\
	Princeton University\\
	Princeton, NJ 08544, USA
}
\email[N. Stevenson]{stevenson@princeton.edu}
\thanks{N. Stevenson was supported by an NSF Graduate Research Fellowship}
\author{Ian Tice}
\address{
	Department of Mathematical Sciences\\
	Carnegie Mellon University\\
	Pittsburgh, PA 15213, USA
}
\email[I. Tice]{iantice@andrew.cmu.edu}
\thanks{I. Tice was supported by an NSF Grant (DMS \#2204912). }

\title[Compressible traveling waves]{
    Well-posedness of the traveling wave problem for the free boundary compressible Navier-Stokes equations
 }
\renewcommand{\indexname}{Notation Index}
\indexsetup{level=\section*,toclevel=section,headers = {}{}}
\makeindex[title={Notation Index\label{index}}]
\begin{document}

\begin{abstract}

We prove that traveling waves in viscous compressible liquids are a generic phenomenon.  The setting for our result is a horizontally infinite, finite depth layer of compressible, barotropic, viscous fluid, modeled by the free boundary compressible Navier-Stokes equations in dimension $n \ge 2$.  The bottom boundary of the fluid is flat and rigid, while the top is a moving free boundary.  A constant gravitational field acts normal to the flat bottom.  We allow external forces to act in the fluid's bulk and external stresses to act on its free surface.  These are posited to be in traveling wave form, i.e. time-independent when viewed in a coordinate system moving at a constant, nontrivial velocity parallel to the lower rigid boundary.  

In the absence of such external sources of stress and force, the fluid system reverts to equilibrium,  which corresponds to a flat, quiescent fluid layer with vertically stratified density.  In contrast, when such sources of stress or force are present, the system admits traveling wave solutions.  We establish a small data well-posedness theory for this problem by proving that for every nontrivial traveling wave speed there exists a nonempty open set of stress and forcing data that give rise to unique traveling wave solutions, and that these solutions depend continuously on the data and the wave speed.  When $n \ge 3$ we prove this with surface tension accounted for at the free boundary, while in the case $n=2$ we prove this with or without surface tension. To the best of our knowledge, this result constitutes  the first general construction of traveling wave solutions to any free boundary compressible fluid equations.

The traveling wave formulation of the equations is a quasilinear system of mixed type. The interaction of the hyperbolic and elliptic parts leads to derivative loss in the linearizations of the system. As such, we are compelled to construct solutions via an inverse function theorem of Nash-Moser type. Our well-posedness proof has a number of novelties and elements of broader interest, including: a new Nash-Moser variant that works in Banach scales but guarantees minimal regularity loss in the existence as well as continuity of the local inverse; a host of results about steady transport equations and their elliptic regularizations; new results about and uses of a scale of anisotropic Sobolev spaces suited for constructing traveling waves; and a robust, streamlined, and flexible approach for constructing solutions to our family of linearized free boundary problems.
\end{abstract}

\maketitle

\pagebreak
%%% This forces the subsections to indent
\makeatletter \def\l@subsection{\@tocline{2}{0pt}{1pc}{5pc}{}} \def\l@subsection{\@tocline{2}{0pt}{2pc}{6pc}{}} \makeatother
{\small\tableofcontents}
%{\tableofcontents}

\pagebreak
	%%%%%%%%%%%%%%%%%%%%%%%%%%%%%%%%%%%%%%%%%%%%%%%%%%%%%%%%%%%%%%%%%%%%%%%%%%%%%%%%%%%%%%%%%%%%%%%%%%%%%%%%%%%%%%%%%%%%%%%%%%%%%%%%%%%%%%%%%%%%%%%%%%%%%%%%%%%%%%%%%%%%%%%%%%%%%%%%%%%%%%%%%%%%%%%%%%%%%%%%%%%%%%%%%%%%%%%%%%%%%%%%%%%%%%%%%%%%%%%%%%%%%%%%%%%%%%%%%%%%%%%%%%%%%%%%%%%%%%%%%%%%%%%%%%%%%%%%%%%%%%%%%%%%%%%%%%%%%%%%%%%%%%%%%%%%%%%%%%%%%%%%%%%%%%%%%%%%%%%%%%%%%%%%%%%%%%%%%%%%%%%%%%%%%%%%%%%%%%%%%%%%%%%%%%%%%%%%%%%%%%%%%%%%%%%%%%%%%%%%%%%%%%%%%%%%%%%%%%%%%%%%%%%%%%%%%%%%%%%%%%%%%%%%%%%%%%%%%%%%%%%%%%%%%%%%%%%%%%%%%%%%%%%%%%%%%%%%%%%%%%%%%%%%%%%%%%%%%%%%%%%%%%%%%%%%%%%%%%%%%%%%%%%%%%%%%%%%%%%%%%%%%%%%%%%%%%%%%%%%%%%%%%%%%%%%%%%%%%%%%%%%%%%%%%%%%%%%%%%%%%%%%%%%%%%%
	\section{Introduction}
	%%%%%%%%%%%%%%%%%%%%%%%%%%%%%%%%%%%%%%%%%%%%%%%%%%%%%%%%%%%%%%%%%%%%%%%%%%%%%%%%%%%%%%%%%%%%%%%%%%%%%%%%%%%%%%%%%%%%%%%%%%%%%%%%%%%%%%%%%%%%%%%%%%%%%%%%%%%%%%%%%%%%%%%%%%%%%%%%%%%%%%%%%%%%%%%%%%%%%%%%%%%%%%%%%%%%%%%%%%%%%%%%%%%%%%%%%%%%%%%%%%%%%%%%%%%%%%%%%%%%%%%%%%%%%%%%%%%%%%%%%%%%%%%%%%%%%%%%%%%%%%%%%%%%%%%%%%%%%%%%%%%%%%%%%%%%%%%%%%%%%%%%%%%%%%%%%%%%%%%%%%%%%%%%%%%%%%%%%%%%%%%%%%%%%%%%%%%%%%%%%%%%%%%%%%%%%%%%%%%%%%%%%%%%%%%%%%%%%%%%%%%%%%%%%%%%%%%%%%%%%%%%%%%%%%%%%%%%%%%%%%%%%%%%%%%%%%%%%%%%%%%%%%%%%%%%%%%%%%%%%%%%%%%%%%%%%%%%%%%%%%%%%%%%%%%%%%%%%%%%%%%%%%%%%%%%%%%%%%%%%%%%%%%%%%%%%%%%%%%%%%%%%%%%%%%%%%%%%%%%%%%%%%%%%%%%%%%%%%%%%%%%%%%%%%%%%%%%%%%%%%%%%%%%%%%

The study of traveling wave solutions to the free boundary problems of fluid mechanics has been of fundamental interest in mathematics for nearly two centuries.  During this time, tremendous progress has been made in the analysis of such solutions for incompressible fluids.  Throughout most of this period, the primary focus was inviscid, irrotational fluids; only in the past two decades has progress been made on models that account for more robust phenomena such as vorticity and viscosity.  However, to the best of our knowledge, there are no rigorous results in the literature that account for the fundamental fluid mechanical effect of  compressibility.  It is this effect that we aim to study in the present paper.

All real fluids experience some degree of compressibility and viscosity, even if small, so it is physically important to verify that traveling waves remain a generic phenomenon when these effects are accounted for.  From a mathematical perspective, the development of the compressible viscous theory also opens the door to studying incompressible or inviscid limits, which may then shed light on the zoo of incompressible inviscid solutions that have been constructed in the literature.  We emphasize that in this work we only study compressible fluids for which the density does not vanish at the free boundary; these are often referred to as compressible liquids in the literature.

The rest of the introduction proceeds as follows.  In Section~\ref{dynamics_and_equilibria} we formulate the dynamical equations for a compressible viscous fluid with free boundary and identify the equilibrium solutions, which correspond to stratified layers of quiescent fluid.  Section~\ref{sec_tw_and_force_role} is concerned with the traveling wave ansatz and a discussion of the role played by external stresses and forces.  Previous work on traveling waves is discussed in  Section~\ref{sitting in an english garden wating for the sun}.  Further reformulations of the equations, made in the interest of identifying `good unknowns,' are recorded in Section~\ref{its just a state of mind}.  Our main results are stated in Section~\ref{sec_main_results} along with some discussion of their implications.  Section~\ref{sec_strategy_label_evil_duck} contains a high-level summary of the difficulties in the proof and our strategies for overcoming them.  Finally, Section~\ref{sec_notational_conventions} records the notational conventions we employ throughout the paper.  We emphasize that for  convenience we have included a~\hyperref[index]{\indexname} at the end of the paper, just before the references.

\subsection{Dynamics in Eulerian coordinates and stratified equilibria} \label{dynamics_and_equilibria}

We begin by formulating the free boundary compressible Navier-Stokes equations, which govern the dynamics of a layer of viscous, compressible, barotropic (isentropic) fluid.  First we must set some notation.  We let $n\in\N\setminus\cb{0,1}$ denote the spatial dimension; the cases $n\in\tcb{2,3}$ are the physically relevant ones, but our analysis works in general dimension.  The parameter $b\in\R^+$ designates the equilibrium depth of the fluid.  If $\eta:\R^{n-1}\to\R$ is a continuous function satisfying $\eta+b>0$, then we define the open set
\begin{equation}\label{Omega_eta_def}\index{\textbf{Fluid mechanical terms}!0@$\Omega[\eta]$}
\Omega[\eta]=\cb{\p{x,y}\in\R^{n-1}\times\R\;:\;0<y<b+\eta\p{x}}    
\end{equation}
as well as the interfacial sets 
\begin{equation}\label{Sigma_defs} 
\Sigma[\eta]=\cb{\p{x,y}\in\R^{n-1}\times\R\;:\;y=b+\eta\p{x}}   \index{\textbf{Fluid mechanical terms}!2@$\Sigma[\eta]$}
\text{ and } 
\Sigma_0=\R^{n-1}\times\tcb{0}.  \index{\textbf{Fluid mechanical terms}!4@$\Sigma_0$}
\end{equation}
Note that $\pd\Omega[\eta]=\Sigma[\eta]\sqcup\Sigma_0$. See Figure~\ref{fig:the sets} for a depiction of these sets. Throughout the paper we will extensively use the shorthand notation
\begin{equation}
\Omega=\Omega[0] = \R^{n-1} \times (0,b) \index{\textbf{Fluid mechanical terms}!1@$\Omega$}
\text{ and }
\Sigma=\Sigma[0] = \R^{n-1} \times \{b\} \index{\textbf{Fluid mechanical terms}!3@$\Sigma$}.
\end{equation}

\begin{figure}[!h]
    \centering
    \includegraphics{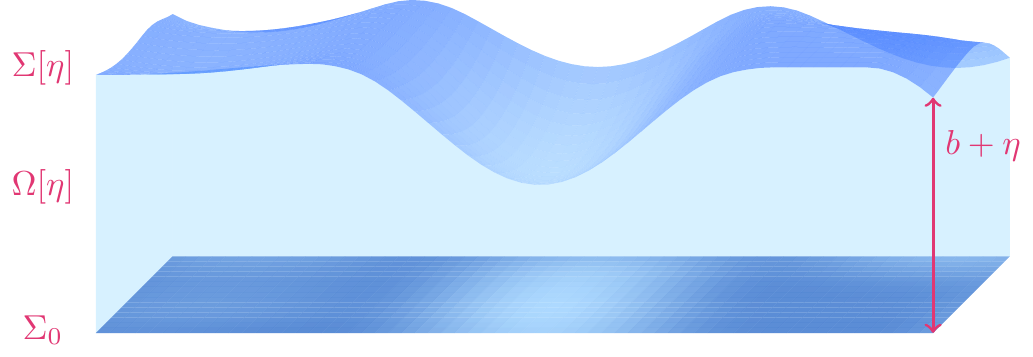}
    \caption{Depiction of the domain $\Omega[\eta]$ and its boundary, $\Sigma[\eta]$ and $\Sigma_0$, when $n=3$.}
    \label{fig:the sets}
\end{figure}
	
The fluid is assumed to occupy a semi-infinite layer of finite depth that changes in time; more precisely, at time $t$ the fluid occupies the set $\Omega[\zeta\p{t,\cdot}]\subset \R^n$ for an unknown continuous free surface function $\zeta\p{t,\cdot}:\R^{n-1}\to\R$ satisfying $\zeta\p{t,\cdot}+b>0$.  The upper boundary of this set, $\Sigma[\zeta\p{t,\cdot}]$, is called the free boundary, while the lower boundary $\Sigma_0$ is referred to as the fixed boundary.  The fluid is described by its velocity vector field $w\p{t,\cdot}:\Omega[\zeta\p{t,\cdot}]\to\R^n$ and its scalar density $\tau\p{t,\cdot}:\Omega[\zeta\p{t,\cdot}]\to\R^+$.  Associated to $w$ and $\tau$ are two crucial fluid mechanical quantities: the pressure and the viscous stress tensor.  The pressure within the fluid is given by $P(\tau)$, where $P\in C^\infty\p{\R^+;\R}$\index{\textbf{Fluid mechanical terms}!8@$P$} is a given pressure law that is strictly increasing and satisfies $P'>0$.  The assumption that the pressure depends only on the density is what makes the fluid barotropic; this can be viewed as a consequence of assuming the fluid flow is isentropic, i.e. entropy remains constant.  The viscous stress tensor within the fluid is the symmetric tensor 
\begin{equation}\index{\textbf{Fluid mechanical terms}!20@$\mathbb{S}^\tau$}
		\S^{\tau}w=\upmu\p{\tau}\bp{\grad w+\grad w^{\m{t}}-\f{2}{n}\grad\cdot wI}+\uplambda\p{\tau}(\grad\cdot w)I=\upmu(\tau)\mathbb{D}^0w+\uplambda(\tau)(\grad\cdot w)I,
\end{equation}
where the shear and bulk viscosity coefficients are $\upmu,\uplambda\in C^\infty\p{\R^+;[0,\infty)}$.  The non-negativity of  the viscosity coefficients is a requirement of the Clausius–Duhem inequality from continuum thermodynamics, but we will place more assumptions on these below.  Note that the specific forms of the functions $P$, $\upmu$, and $\uplambda$ can be thought of as characterizing the specific material comprising the fluid.

The dynamics of $\zeta$, $w$, and $\tau$ are then coupled to the various forces and stresses acting on the fluid through the free boundary compressible Navier-Stokes equations:
\begin{equation}\label{compressible navier-stokes in Eulerian coordinates}
		\begin{cases}
			\pd_t\tau+\grad\cdot\p{\tau w}=0&\text{in }\Omega[\zeta\p{t,\cdot}]\\
			\tau\p{\pd_tw+w\cdot\grad w}+\grad(P\p{\tau})-\grad\cdot\S^{\tau}w=-\mathfrak{g}\tau e_n+\tau G+F&\text{in }\Omega[\zeta\p{t,\cdot}]\\
			-\tp{P\p{\tau}-\S^{\tau}w}\nu_\zeta+P_{\m{ext}}\nu_\zeta-\varsigma\mathscr{H}\p{\zeta}\nu_\zeta=T\nu_\zeta&\text{on }\Sigma[\zeta\p{t,\cdot}]\\
			\pd_t\zeta+w\cdot\tp{\grad_{\|}\zeta,-1}=0&\text{on }\Sigma[\zeta\p{t,\cdot}]\\
			w=0&\text{on }\Sigma_0.
		\end{cases}
\end{equation}
Note that in the above we have written $\grad_{\|}=(\pd_1,\dots,\pd_{n-1})$\index{\textbf{Miscellaneous}!02@$\grad_{\parallel}$} to refer to the `tangential gradient'.  

We now enumerate these forces and stresses.  The term $-\mathfrak{g} \tau e_n$ is the gravitational force acting on the fluid, with gravitational strength $\mathfrak{g} >0$ \index{\textbf{Fluid mechanical terms}!11@$\mathfrak{g}$} and unit vector $e_n = (0,\dots,1)\in \R^n$ perpendicular to $\Sigma_0$.  The vector fields $F\p{t,\cdot},G(t,\cdot):\Omega[\zeta\p{t,\cdot}]\to\R^n$ are the applied bulk and specific bulk forces, respectively.  The parameter $P_{\m{ext}}\in\R$\index{\textbf{Fluid mechanical terms}!9@$P_{\m{ext}}$} is a constant external pressure,  $T\p{t,\cdot}:\Sigma[\zeta\p{t,\cdot}]\to\R^{n\times n}$ is the applied surface stress, and $\nu_\zeta$ is the outward unit normal to the surface $\Sigma[\zeta\p{t,\cdot}]$.  We note that in continuum mechanics it is usually the case that $T\p{t,\cdot}$ is symmetric, but this condition plays no role in our analysis, so have allowed for the most general case.  The mean curvature operator is 
\begin{equation}\label{mean_curvature_def} \index{\textbf{Fluid mechanical terms}!10@$\mathscr{H}$}
    \mathscr{H}\p{\zeta} = \grad_{\|} \cdot ( (1+ \abs{\grad_{\|} \zeta}^2)^{-1/2}    \grad_{\|}   \zeta),
\end{equation}
and the parameter $\varsigma \ge 0$ \index{\textbf{Fluid mechanical terms}!12@$\varsigma$} is called the coefficient of surface tension.  For technical reasons that will be discussed later, we make the following assumptions about the viscosity coefficients and the coefficient of surface tension:
\begin{equation}\label{parameter_assumptions}
\begin{cases}
    \upmu >0, \uplambda \ge 0 \text{ in } \R^+, \text{ and } \varsigma >0 & \text{if } n \ge 3 \\
    \upmu >0, \uplambda > 0  \text{ in } \R^+, \text{ and } \varsigma \ge 0 & \text{if } n = 2. 
\end{cases}
\end{equation}

The first equation in~\eqref{compressible navier-stokes in Eulerian coordinates} is the continuity equation, which asserts conservation of mass.  The next is the momentum equation, and it dictates a Newtonian balance of forces in the fluid bulk.  After this is the dynamic boundary condition, which enforces a balance of stresses acting on the free surface.  The penultimate equation is the kinematic boundary condition, which determines how the free surface evolves according to the fluid velocity.  The final equation in~\eqref{compressible navier-stokes in Eulerian coordinates} is simply the no-slip boundary condition for the velocity on the rigid bottom.  For a more thorough introduction to the compressible Navier-Stokes equations, including their derivation, we refer to the books of  Wehausen and Laitone~\cite{MR0119656},   Feireisl~\cite{MR2040667}, Lions~\cite{MR1422251},  Novotn\'{y} and Stra\v{s}kraba~\cite{MR2084891},  Gurtin, Fried, and Anand \cite{GFA_2010}, and Plotnikov and Soko\l owski~\cite{MR2963679}. 
	
The compressible Navier-Stokes system \eqref{compressible navier-stokes in Eulerian coordinates} admits a vertically stratified equilibrium solution, provided that the barotropic pressure law $P$, the external depth $b$, the external pressure $P_{\m{ext}}$, and the gravitational field strength $\mathfrak{g}$ satisfy some compatibility conditions.  Indeed, suppose that the fluid experiences no external forces or stresses, i.e. $F=G=0$ and $T =0$, and that the fluid is quiescent and occupies a flat slab of depth $b$, i.e. $w=0$ and $\zeta =0$.  Finally, suppose that $\partial_t \tau =0$; then \eqref{compressible navier-stokes in Eulerian coordinates} reduces to $\tau(t,x,y) = \varrho(y)$, where $\varrho : [0,b] \to \R^+$ is a smooth function solving the Cauchy problem
\begin{equation}\label{equilibrium cauchy problem}
	\begin{cases}
		(P\circ\varrho)'=-\mathfrak{g}\varrho&\text{in }(0,b)\\
		P\circ\varrho(b)=P_{\m{ext}}.
	\end{cases}
\end{equation}
In order to guarantee that a solution exists, we henceforth assume that the following pair of compatibility conditions are satisfied (these conditions are actually necessary and sufficient):
\begin{equation}\label{equilibrium_ccs}
    P_{\m{ext}}\in P(\R^+) \text{ and } (0,\infty]\ni\int_{P^{-1}(P_{\m{ext}})}^\infty t^{-1}P'(t)\;\m{d}t>\mathfrak{g}b.
\end{equation}
With \eqref{equilibrium_ccs} in hand, we can conveniently solve \eqref{equilibrium cauchy problem} by introducing the enthalpy $H:\R^+\to\R$, which is the smooth increasing function defined via
\begin{equation}\label{enthalpy_def}  \index{\textbf{Fluid mechanical terms}!60@$H$}
	H(s)=-\mathfrak{g}b+\int_{P^{-1}(P_{\m{ext}})}^st^{-1}P'(t)\;\m{d}t.  
\end{equation}
Note that since $P' >0$ we have that $H' >0$ as well.  We can also calculate the image $H(\R^+) = (H_{\m{min}},H_{\m{max}})\subseteq\R$\index{\textbf{Fluid mechanical terms}!61@$H_{\m{min}},H_{\m{max}}$}, where
\begin{equation}\label{domain of the inverse enthalpy}
		H_{\m{min}}=-\mathfrak{g}b-\int_0^{P^{-1}(P_{\m{ext}})}t^{-1}P'(t)\;\m{d}t < -\mathfrak{g}b  \text{ and } 
		H_{\m{max}}=-\mathfrak{g}b+\int_{P^{-1}(P_{\m{ext}})}^\infty t^{-1}P'(t)\;\m{d}t >0.
\end{equation}
Since we now know that $H:\R^+\to(H_{\m{min}},H_{\m{max}})$ is a smooth diffeomorphism and $[-\mathfrak{g}b,0] \subseteq H(\R^+)$, we may realize $\varrho$ as the smooth decreasing function defined by
\begin{equation}\label{varrho_def} \index{\textbf{Fluid mechanical terms}!62@$\varrho$}
    \varrho(y)=H^{-1}(-\mathfrak{g}y) \text{ for } y \in [0,b],
\end{equation}
which is the unique solution to \eqref{equilibrium cauchy problem} in light of the construction of $H$.  The equilibrium density $\varrho$ will play a crucial role in our subsequent analysis.  As a concrete example, if the pressure satisfies the well-known polytropic law $P(t) = K t^\alpha$ for $\alpha \ge 1$ and $K \in \R^+$, then 
\begin{equation}
    \varrho(y) = 
\begin{cases}
    P_{\m{ext}} K^{-1} \exp\tp{ \mathfrak{g} K^{-1} (b-y)} &\text{if } \alpha =1 \\
    \tp{(P_{\m{ext}} K^{-1} )^{(\alpha-1)/\alpha} + (\alpha-1)\alpha^{-1} K^{-1} \mathfrak{g}(b-y)}^{1/(\alpha-1)} &\text{if }\alpha >1.
\end{cases}    
\end{equation}

\subsection{Traveling waves and the role of the stress and forces}	\label{sec_tw_and_force_role}
	
The main thrust of this paper is the study of traveling wave solutions to the system~\eqref{compressible navier-stokes in Eulerian coordinates}. These are solutions that are time-independent when viewed in an inertial coordinate system obtained from the above Eulerian coordinates through a Galilean transformation. In order for time-independence to hold, the moving coordinate system must travel at a constant velocity parallel to $\Sigma_0$.  Without loss of generality (we can always apply a rigid rotation that fixes the vector $e_n$ to change coordinates), we may assume that the traveling coordinate system moves at constant velocity $\gam e_1$ for a wave speed $\gam\in\R^+$. 
	
In the new coordinates, the stationary free boundary is described by $\eta\p{x-\gam te_1}=\zeta\p{t,x}$ for a new unknown free surface function $\eta : \R^{n-1} \to (-b,\infty)$, which then determines the fluid domain $\Omega[\eta]$ and the free boundary $\Sigma[\eta]$.  We then posit that the other quantities are also in traveling wave form: $\gam v\p{x-t\gam e_1,y}=w\p{t,x,y}$, $\sig(x-\gam te_1,y)=\tau(t,x,y)$, 	$\mathcal{F}\p{x-t\gam e_1,y}=F\p{t,x,y}$, $\mathcal{G}\p{x-t\gam e_1,y}=G\p{t,x,y}$, and $\mathcal{T}\p{x-t\gam e_1,y}=T\p{t,x,y}$, where $v:\Omega[\eta]\to\R^n$, $\sig:\Omega[\eta]\to\R$, $\mathcal{F},\mathcal{G}:\Omega[\eta]\to\R^n$, and $\mathcal{T}:\Sigma[\eta]\to\R^{n\times n}$ define the stationary velocity field, density, external and specific forces, and external stresses, respectively.  Under these assumptions, \eqref{compressible navier-stokes in Eulerian coordinates} is equivalent to the following traveling compressible Navier-Stokes system for unknowns $\p{\sig,v,\eta}$ and data $\p{\mathcal{T},\mathcal{G},\mathcal{F}}$:
 	\begin{equation}\label{compressible navier-stokes traveling wave equations rescaled}
		\begin{cases}
			-\pd_1\sig+\grad\cdot\p{\sig v}=0&\text{in }\Omega[\eta]\\
			\gam^2\sig\p{v- e_1}\cdot\grad v+\grad(P\p{\sig})-\gam\grad\cdot\S^{\sig}v=-\mathfrak{g}\sig e_n+\sig \mathcal{G}+\mathcal{F}&\text{in }\Omega[\eta]\\
-\tp{P\p{\sig}-\gam\S^{\sig}v}\nu_\eta+P_{\m{ext}}\nu_\eta-\varsigma\mathscr{H}\p{\eta}\nu_\eta=\mathcal{T}\nu_\eta&\text{on }\Sigma[\eta]\\
			-\pd_1\eta+v\cdot\tp{\grad_{\|}\eta,-1}=0&\text{on }\Sigma[\eta]\\
			v=0&\text{on }\Sigma_0.
		\end{cases}
	\end{equation}
 Note that in changing unknowns, we have rescaled the velocity vector by $\gam$. This has the effect of nondimensionalizing the vector field $v$ and, more importantly, removing the $\gam$-dependence from the continuity equation and the kinematic boundary condition.

 With the traveling wave system \eqref{compressible navier-stokes traveling wave equations rescaled} formulated, we turn to a discussion of the role played by the surface stress, specific bulk force, and bulk force data triple, $(\mathcal{T},\mathcal{G},\mathcal{F})$.   We have chosen to study this general form of $\p{\mathcal{T},\mathcal{G},\mathcal{F}}$ in order to allow these to model a variety of physical effects.  As a specific example, the bulk force term $\mathcal{G}$ can be thought of as a localized perturbation of the gravitational field caused by a massive object translating above the fluid (a primitive model of the ocean-moon system).  Similarly, a simple example of the surface stress occurs when  $\mathcal{T} = -\varphi I_{n \times n}$ for a given scalar function $\varphi: \R^{n} \to \R$; in this configuration, $\varphi$ can be viewed as a spatially localized source of pressure translating above the fluid. See Figure~\ref{various wave speeds} for depictions of the free surface for this latter case of applied stress.

 If $(\mathcal{T},\mathcal{G},\mathcal{F})=0$, then it is a simple matter to verify that the stratified equilibrium solution,  $v=0$, $\eta =0$, and $\sigma = \varrho$ (defined by \eqref{varrho_def}), provides a solution to \eqref{compressible navier-stokes traveling wave equations rescaled} with any value of $\gamma$.  This suggests that we should seek solutions to \eqref{compressible navier-stokes traveling wave equations rescaled} as perturbations of this stratified equilibrium.  An elementary formal calculation (we will state and prove a rigorous version later after a further reformulation of the problem: see Appendix~\ref{lift my head, im still yawning} and, in particular Corollary~\ref{flat_diss_power_id}) reveals that if $\sigma - \varrho$, $v$, and $\eta$ are in a Sobolev-type framework, then  
  \begin{equation}\label{she said I know what its like to be dead}
     \int_{\Omega[\eta]}\f{\upmu(\sig)}{2}|\mathbb{D}^0v|^2+\uplambda(\sig)|\grad\cdot v|^2=\int_{\Omega[\eta]}(\sig\mathcal{G}+\mathcal{F})\cdot v+\int_{\Sigma[\eta]}\mathcal{T}\nu_\eta\cdot v.
 \end{equation}
 The physical interpretation of this identity is that if a traveling wave solution exists, then the power supplied by the forces and stress (the right side of \eqref{she said I know what its like to be dead}) must be in exact balance with the energy dissipation rate due to viscosity (the left side of~\eqref{she said I know what its like to be dead}). 
 
  The identity \eqref{she said I know what its like to be dead} reveals even more if we assume there are no applied forces or stress, i.e. $(\mathcal{T},\mathcal{G},\mathcal{F})=0$.  Without a source of external power, \eqref{she said I know what its like to be dead} requires that the left integral vanishes, and so the assumptions  \eqref{parameter_assumptions} together with the Korn inequality (see Propositions~\ref{prop on Korn's inequality} and~\ref{prop on deviatoric Korn's inequality} for Korn in $\Omega$, but similar results hold in $\Omega[\eta]$ if $\eta$ is sufficiently regular) imply that  $v=0$.  In turn, the momentum equation implies that $\grad(P(\sig))=-\mathfrak{g}\sig e_n$ and hence $\sig=H^{-1}(-\mathfrak{g}\m{id}_{\R^n}\cdot e_n+c)$ for some constant $c$, but we must have $c=0$ since $\sig-\varrho$ vanishes at infinity.  The normal part of the dynamic boundary condition then requires
 \begin{equation}
     P_{\m{ext}}-P\circ H^{-1}(-\mathfrak{g}(b+\eta))-\varsigma\grad_{\|}\cdot(\tp{1+|\grad_{\|}\eta|^2}^{-1/2}\grad_{\|}\eta)=0,
 \end{equation}
and by multiplying this equation by $\eta$  and integrating by parts in the mean curvature term, we deduce that
\begin{equation}\label{songs that linger on my}
    \int_{\R^{n-1}}(P_{\m{ext}}-P\circ H^{-1}(-\mathfrak{g}(b+\eta)))\eta\le0.
\end{equation}
Since $P_{\m{ext}}-P\circ H^{-1}(-\mathfrak{g}(b+\cdot))$ is a strictly increasing function vanishing at zero, the integrand in~\eqref{songs that linger on my} is nonnegative and must thus vanish pointwise.  Hence, $\eta =0$, and we have deduced that $(\sig,v,\eta)=(\varrho,0,0)$ when $(\mathcal{T},\mathcal{G},\mathcal{F})=0$.

The above formal computation suggests that the dissipative nature of viscosity prohibits the existence of nontrivial  traveling wave solutions (in Sobolev-type spaces) without applied stress or forcing.  This shows that the triple $(\mathcal{T},\mathcal{G},\mathcal{F})$ plays an essential role in the study of traveling wave solutions, as the stress and forcing data are necessary for solutions to exist.  We emphasize, however, that the above argument does not preclude the existence of nontrivial solutions with $(\mathcal{T},\mathcal{G},\mathcal{F})=0$ in non-Sobolev functional frameworks.

Now that the importance of the data triple in the traveling wave theory is evident, we can roughly summarize our goal for the paper: we aim to prove that for every traveling wave speed $\gamma \in \R^+$ the problem \eqref{compressible navier-stokes traveling wave equations rescaled} admits a small-data well-posedness theory.  That is, we aim to identify a nontrivial open set, $\mathcal{U}_\gamma$, of stress and forcing data in a Sobolev-type framework such that for every  $(\mathcal{T},\mathcal{G},\mathcal{F}) \in \mathcal{U}_\gamma$ there exists a locally unique solution triple $(\sigma, v, \eta)$ to \eqref{compressible navier-stokes traveling wave equations rescaled}, also in a Sobolev-type framework, that depends continuously on $(\mathcal{T},\mathcal{G},\mathcal{F})$.  However, for technical reasons that we will explain at the beginning of Section~\ref{its just a state of mind}, the variables $(\sigma,v,\eta)$ are not suitable for this task, and we must introduce a further reformulation of \eqref{compressible navier-stokes traveling wave equations rescaled} with a new set of `good unknowns' in order to achieve our goal.  Interestingly, this new formulation has the added benefit of allowing us to establish the continuity of solutions with respect to the wave speed $\gamma$ as well.

\subsection{Previous work}\label{sitting in an english garden wating for the sun}

The time dependent free boundary compressible fluid equations of~\eqref{compressible navier-stokes in Eulerian coordinates} and their variants have received much attention in the literature.  A full review is beyond the scope of the paper, so we will settle for a brief survey of some results closely related to ours.   

A significant portion of the literature on the dynamic problem concerns the inviscid analog of~\eqref{compressible navier-stokes in Eulerian coordinates}, i.e. the free boundary compressible Euler equations.  Lindblad~\cite{Lindblad_2003,MR2177323} proved local well-posedness of the liquid droplet problem.  Jang and Masmoudi \cite{Jang_Masmoudi_2009, Jang_Masmoudi_2015} and Coutand and Shkoller~\cite{Coutand_Shkoller_2011,MR2980528}  proved local well-posedness for the vacuum droplet problem.  Well-posedness of the liquid droplet problem with surface tension was studied by Coutand, Hole, and Shkoller~\cite{MR3139610} and by Disconzi and Kukavica \cite{Disconzi_Kukavica_2019}.  The incompressible limit for the liquid droplet problem was derived by Lindblad and Luo \cite{Lindblad_Luo_2018} without surface tension and by Disconzi and Luo  \cite{Disconzi_Luo_2020} with surface tension. Trakhinin~\cite{MR2560044} and  Luo and Zhang~\cite{MR4439376} proved local well-posedness for inviscid liquid layers of infinite depth.

There are also a number of dynamical studies of the free boundary compressible Navier-Stokes equations.  Denisova proved local well-posedness for the compressible bubble in a compressible fluid problem in Sobolev spaces \cite{Denisova_1997} and in weighted H\"older spaces \cite{Denisova_2003}.  Denisova  \cite{Denisova_2000} also proved similar results for the compressible bubble in an incompressible fluid problem.  The viscous liquid droplet problem was studied by Secchi and Valli~\cite{MR697305}, who proved local well-posedness with heat conduction, Solonnikov and Tani~\cite{MR1226506}, who proved local well-posedness with surface tension, and by Denisova and Solonnikov \cite{Denisova_Solonnikov_2018}, who developed a well-posedness theory with and without surface tension.  The well-posedness of the viscous liquid droplet problem has also been proved under various conditions using maximal regularity techniques: see the work of Enomoto, von Below, and Shibata \cite{EvBS_2014}, Shibata \cite{Shibata_2016}, and Burczak, Shibata, and Zaj\c{a}czkowski \cite{BSZ_2018}.   

The viscous literature also has several studies of layer geometries.  Jin~\cite{MR2164990} and Jin and Padula~\cite{MR1903004} proved local and global well-posedness for a layer of periodic barotropic fluid with surface tension.  Tanaka and Tani~\cite{MR2004291} gave local and global well-posedness results for layers of heat conducting fluids. Jang, Tice, and Wang~\cite{MR3537008,MR3488552} proved local and global well-posedness for multiple layers of barotropic fluid.  Huang and Luo~\cite{MR4266110} proved global existence for a layer of heat conducting fluid without surface tension.

It is also possible to use compressible fluids with free boundaries as a simple model of stellar structure, in which case the fluid is subject to the force of its own gravitational field.  When gravity is assumed to be Newtonian rather than relativistic, this problem is called the Euler-Poisson system for inviscid fluids and the Navier-Stokes-Poisson system for viscous ones.   Makino~\cite{MR882389} proved an early local existence result for Euler-Poisson under special assumptions on the data.  Makino \cite{Makino_1993} and Matu\v{s}u-Ne\v{c}asov\'{a}, Okada, and Makino  \cite{MOM_1995} studied the viscous problem with spherical symmetry outside a solid inner-core.  Jang \cite{Jang_2010} proved local existence for the Navier-Stokes-Poisson problem with vacuum boundary.  Luo, Xin, and Zeng~\cite{MR3218831} studied local well-posedness for radial solutions to the inviscid problem with vacuum boundary. Ginsberg, Lindblad, and Luo~\cite{MR4072680} studied the local well-posedness of a self gravitating compressible liquid.  Jang and  Had\v{z}i\'{c} \cite{Hadzic_Jang_2019} constructed global expanding solutions for Euler-Poisson.   The self-gravitating problem admits nontrivial radial steady states known as Lane-Emden solutions (see, for instance, the book of Chandrasekhar \cite{Chandrasekhar_1957}).  More recent work has rigorously constructed rotating solutions: for a variational approach  we refer to Auchmuty and Beals \cite{Auchmuty_Beals_1971} and Li \cite{Li_1991}, and for a perturbative approach we refer to Jang and Makino \cite{Jang_Makino_2017} and Strauss and Wu \cite{Strauss_Wu_2019}.

In stark contrast to the above discussion, the compressible traveling problem~\eqref{compressible navier-stokes traveling wave equations rescaled} has, to the best of our knowledge, not received any prior attention in the PDE literature either with or without viscosity.  Perhaps the closest work, though still rather distant, concerns the construction of stationary (but not traveling) solutions to some free boundary problems for viscous compressible fluids.  Pileckas and Zaj\c{a}czkowski~\cite{MR1046283} found stationary solutions to a bounded viscous compressible fluid droplet with surface tension under symmetry considerations. Jin and Padula~\cite{MR2076684} considered steady flows of viscous compressible fluids in a bounded rigid container with partially free boundary.

The incompressible analogs of~\eqref{compressible navier-stokes traveling wave equations rescaled} have, however, received attention in the literature.  The incompressible and inviscid analog of~\eqref{compressible navier-stokes traveling wave equations rescaled}, which is also known as the traveling water wave problem, has received enormous attention in the mathematics literature for more than a century.  We refer to the surveys of Toland~\cite{Toland_1996}, Groves~\cite{Groves_2004}, Strauss~\cite{Strauss_2010}, and Haziot, Hur, Strauss, Toland, Wahl\'en, Walsh, and Wheeler~\cite{MR4406719} and the references therein for a thorough review of this extensive literature.  On the other hand, progress on the traveling wave theory for the free boundary incompressible Navier-Stokes equations only began quite recently.  A small data well-posedness theory was first developed by Leoni and Tice~\cite{leoni2019traveling}.  This result was subsequently generalized by Stevenson and Tice~\cite{MR4337506} and Koganemaru and Tice~\cite{koganemaru2022traveling} to multi-layered and inclined geometries, respectively.  A similar well-posedness and stability theory for traveling wave solutions to the one-phase Muskat problem  was developed by Nguyen and Tice \cite{nguyen_tice_2022}.  Viscous traveling waves were also empirically identified recently in experiments with a tube of air, translating uniformly above a wave tank, blowing onto a single layer of viscous fluid. For details, we refer to the works of Akylas, Cho, Diorio, and Duncan \cite{CDAD_2011,DCDA_2011}, Masnadi and Duncan \cite{MD_2017}, and Park and Cho \cite{PC_2016,PC_2018}.

As we have already mentioned, there are only a few recent works in the literature that rigorously treat traveling waves with viscosity and none that treat compressibility.  Viscosity and compressibility are important physical effects to account for in studying traveling waves because all real fluids experience some degree of compressibility and viscosity, even if small.  Within the applied literature there are works that study the observable effects of  compressibility in simplified fluid models.  For example,  Longuet-Higgins~\cite{MR40887} and Kadri~\cite{MR3281921} studied how the compressibility of water leads to microseisms in the ocean, and  Long and Morton \cite{Long_Morton_1966} and Miesen, Kamp, and Sluijter \cite{MKS_1990,MKS_1990_2} used asymptotic expansions and numerics to study the role of compressibility in atmospheric solitary traveling waves.
 
Our proof of well-posedness for the traveling wave problem~\eqref{compressible navier-stokes traveling wave equations rescaled} uses a novel variation on the Nash-Moser inverse function theorem.  A literature review concerning our version is delayed until Sections \ref{sec_strategy_label_evil_duck} and \ref{section on NMH}, but we will conclude this subsection with a sampling of interesting results from the fluids PDE literature that have been proved with Nash-Moser.  
Beale~\cite{MR445136} used it to prove the existence of steady water waves for the irrotational incompressible free boundary Euler equations in two dimensions. Plotnikov and Toland~\cite{MR1854060} found time periodic standing water waves. Lindblad~\cite{MR2177323,MR2178961} used Nash-Moser to study the droplet problem for incompressible and compressible Euler.  Chen and Wang~\cite{MR2372810} studied the existence and stability of compressible current-vortex sheets in three-dimensional magnetohydrodynamics.  Trakhinin~\cite{MR2560044} considered the infinite depth surface wave problem for compressible Euler. Makino~\cite{MR3379135,MR3569407} studied spherically symmetric motions of a planet's compressible atmosphere and the vacuum boundary problem for gaseous stars.  Buffoni and Wahl\'{e}n~\cite{MR3892402} used a version of Nash-Moser to produce steady three dimensional rotation flows for incompressible Euler.  Chen, Secchi, and Wang~\cite{MR3925528} studied relativistic vortex sheets in three-dimensional Minkowski spacetime for compressible, relativistic Euler.  Chen, Hu, Wang, Wang, and Yuan \cite{CHWWY_2020} used Nash-Moser to study compressible vortex sheets in two-dimensional elastodynamics.  Trakhinin and Wang~\cite{MR4444136} studied the ideal compressible magnetohydrodynamic equations with surface tension via Nash-Moser.

\subsection{Enthalpy and flattened reformulations}\label{its just a state of mind}

The goal of this subsection is to further reformulate the traveling wave system \eqref{compressible navier-stokes traveling wave equations rescaled} so that it is more convenient to analyze.  The motivations for this are three-fold:  two common difficulties in free boundary problems and a third more subtle issue specific to the problem at hand.  The first issue is that we wish to establish a well-posedness theory in Sobolev-type spaces.  However, as we discussed at the end of Section \ref{sec_tw_and_force_role}, if $(\mathcal{T},\mathcal{G},\mathcal{F})=0$ then the solution reduces to the stratified equilibrium, but we cannot expect $\sigma = \varrho$ to belong to Sobolev-type spaces on sets of infinite measure.  This suggests that we should rewrite \eqref{compressible navier-stokes traveling wave equations rescaled} as a perturbation of the equilibrium solution $(\varrho,0,0)$.  The second issue is that, even if we rewrite in perturbed form, the resulting equations are still posed in an unknown domain $\Omega[\eta]$.  In order to conveniently employ standard PDE toolboxes, it is advantageous to recast the system in a fixed, known domain.  The traveling wave formulation precludes the common choice of Lagrangian coordinates, so we instead employ a flattening into the equilibrium domain $\Omega$ based only on the free surface function.  The third issue arises because the traveling wave structure ultimately forces the free surface function $\eta$ to belong to a scale of anisotropic Sobolev-type spaces (see Appendix \ref{appendix on anisotropic Sobolev spaces}).  After the perturbation and flattened reformulations, $\eta$ will end up appearing in various nonlinearities in a form that, due to the strange properties of the anisotropic spaces, is quite difficult or impossible to control in a Sobolev-type framework.  Roughly speaking, our way around this problem is to make a nonlinear change of unknown, shifting from the density to the perturbed enthalpy $h=H(\sigma)-H(\varrho)$.  The unknown $h$ turns out to serve as a sort of `good unknown' in that it recasts the worst nonlinearity, which comes from the term $\grad (P(\sigma) ) + \mathfrak{g}\sig e_n$ in \eqref{compressible navier-stokes traveling wave equations rescaled}, in the form $\sigma \grad(H(\sigma) - H(\varrho)) = H^{-1}(h+H(\varrho)) \grad h$, which shifts the nonlinearity outside of the gradient and permits simple Sobolev estimates.  We will also employ a nonlinear change of the velocity in order to similarly linearize the kinematic boundary condition.  The cost of these changes is that the continuity equation becomes more cumbersome, but fortunately we can handle its new form.

We now turn to the execution of these reformulations of \eqref{compressible navier-stokes traveling wave equations rescaled}.  We start by switching to a perturbed enthalpy reformulation by defining the new unknown $h=H(\sig)-H(\varrho)$.  Note that $\sigma$ can be recovered from $h$ via $\sigma = H^{-1}(h+H(\varrho))$, provided $h+H(\varrho)$ takes values in $(H_{\m{min}},H_{\m{max}})$ from \eqref{domain of the inverse enthalpy}, which will always hold for the solutions we construct.  Then the perturbative enthalpy reformulation of the traveling free boundary compressible Navier-Stokes equations is the following system for $(h,v,\eta)$ with data $(\mathcal{T},\mathcal{G},\mathcal{F})$:
	\begin{equation}\label{unflattened PDE enthalpy formulation}
		\begin{cases}
			-\pd_1(H^{-1}(h+H(\varrho)))+\grad\cdot(H^{-1}(h+H(\varrho))v)=0&\text{in }\Omega[\eta]\\
			\gam^2H^{-1}(h+H(\varrho))(v-e_1)\cdot\grad v+H^{-1}(h+H(\varrho))\grad h-\gam\grad\cdot\S^{H^{-1}(h+H(\varrho))}v&\\\quad=H^{-1}(h+H(\varrho))\mathcal{G}+\mathcal{F}&\text{in }\Omega[\eta]\\
			-((P-P_{\m{ext}})\circ H^{-1}(h+H(\varrho))-\gam\S^{H^{-1}(h+H(\varrho))}v)\nu_\eta-\varsigma\mathscr{H}(\eta)\nu_\eta=\mathcal{T}\nu_\eta&\text{on }\Sigma[\eta]\\
			-\pd_1\eta+v\cdot(\grad_{\|}\eta,-1)=0&\text{on }\Sigma[\eta]\\
			v=0&\text{on }\Sigma_0.
		\end{cases}
	\end{equation}
	
Next we turn our attention to reformulating \eqref{unflattened PDE enthalpy formulation} in the fixed domain $\Omega= \R^{n-1} \times (0,b)$.  We first define a Poisson-like extension operator that takes functions defined on $\Sigma$ and extends them to functions defined on $\Omega$.  The auxiliary mapping $\mathcal{E}_0:H^{s-1/2}(\Sigma)\to H^{s}(\Omega)\cap{_0}H^1(\Omega)$\index{\textbf{Linear maps}!05@$\mathcal{E}_0$} (see~\eqref{zero on the left}), for $s\in\N^+$, is defined via $\varphi\mapsto\mathcal{E}_0\varphi$, where $\mathcal{E}_0\varphi$ is the unique solution to the PDE
	\begin{equation}\label{Nevada}
		\begin{cases}
			-\Delta\mathcal{E}_0\varphi=0&\text{in }\Omega,\\
			\mathcal{E}_0\varphi=\varphi&\text{on }\Sigma,\\
			\mathcal{E}_0\varphi=0&\text{on }\Sigma_0.
		\end{cases}
	\end{equation}
	Our Poisson-like extension operator $\mathcal{E}$\index{\textbf{Linear maps}!06@$\mathcal{E}$} is then defined for appropriate functions (see Lemma~\ref{lem on mapping properties of the Poisson extension operator variants}) $\eta:\Sigma\to\R$ through the assignment 
	\begin{equation}\label{North Dakota}
		\mathcal{E}\eta(x,y)=y\cdot b^{-1}\Uppi^1_{\m{L}}\eta(x)+\mathcal{E}_0\Uppi^1_{\m{H}}\eta(x,y)
		\text{ for }
		(x,y)\in\R^{n-1}\times(0,b),
	\end{equation}
	where the Fourier projectors are given by  $\Uppi^1_{\m{L}}\eta=\mathscr{F}^{-1}[\mathds{1}_{B(0,1)}\mathscr{F}[\eta]]$ and  $\Uppi^1_{\m{H}}\eta=\eta-\Uppi^1_{\m{L}}\eta$.
 
 Our flattening map is then built from $\mathcal{E}$ as follows.  For $\eta: \Sigma \to \R$ satisfying $\eta>-b$ and belonging to an appropriate function space (see~\eqref{in the introduction we have an ansiobro} with $s>1+d/2$ and $d=n-1$) we define $\mathfrak{F}_\eta : \Bar{\Omega} \to \R^{n}$ via 
 \begin{equation}\label{flattening_map_def} \index{\textbf{Nonlinear maps}!110@$\mathfrak{F}_\eta$}
     \mathfrak{F}_\eta(x,y)=(x,y+\mathcal{E}\eta(x,y)) \text{ for }(x,y)\in \Omega =\R^{n-1}\times(0,b).
 \end{equation}
 We associate to $\mathfrak{F}_\eta$ two crucial quantities: the Jacobian $J_\eta:\Omega\to\R^+$ 
 and (when $J_\eta$ is nowhere vanishing) the geometry matrix $\mathcal{A}_\eta:\Omega\to\R^{n\times n}$, which are defined via
 \begin{equation}\label{geometry_and_jacobian_def}
 J_\eta=\det(\grad\mathfrak{F}_\eta)=1+\pd_n\mathcal{E}\eta = \pd_n(\mathfrak{F}_\eta\cdot e_n) \index{\textbf{Nonlinear maps}!130@$J_\eta$}
 \text{ and }
 \mathcal{A}_\eta=(\grad\mathfrak{F}_\eta))^{-\m{t}}   \index{\textbf{Nonlinear maps}!120@$\mathcal{A}_\eta$}.
 \end{equation}
 Then, provided that
 \begin{equation}
     J_\eta>0\text{ and }J_\eta,1/J_\eta\in L^\infty(\Omega),
 \end{equation}
 $\mathfrak{F}_\eta(\Omega)=\Omega[\eta]$ and $\mathfrak{F}_\eta$ is a bi-Lipschitz homeomorphism from $\Bar{\Omega}$ to $\Bar{\Omega[\eta]}$ such that its restriction to $\Omega$ defines a smooth diffeomorphism to $\Omega[\eta]$.  Moreover,  $\mathfrak{F}_\eta(\Sigma) = \Sigma[\eta]$ and $\mathfrak{F}_\eta$ is the identity on  $\Sigma_0$.

 For our penultimate change of equations, we set $\m{h}=h\circ\mathfrak{F}_\eta$ and $u=J_\eta\mathcal{A}^{\m{t}}_\eta v\circ\mathfrak{F}_\eta$ to be our new unknowns defined in the fixed domain $\Omega$. Equations~\eqref{unflattened PDE enthalpy formulation} then transform to the following equivalent system for unknowns $(\m{h},u,\eta)$ with data $(\mathcal{T},\mathcal{G},\mathcal{F})$:
	\begin{equation}\label{flattened PDE enthalpy formulation I}
		\begin{cases}
			
			\grad\cdot\tp{H^{-1}(\m{h}+H\circ\varrho(\mathfrak{F}_\eta\cdot e_n))(u-J_\eta\mathcal{A}^{\m{t}}_\eta e_1)}=0&\text{in }\Omega\\
			H^{-1}(\m{h}+H\circ\varrho(\mathfrak{F}_\eta\cdot e_n))(\gam^2(\mathcal{A}^{-\m{t}}_\eta u/J_\eta- e_1)\cdot\mathcal{A}_\eta\grad (\mathcal{A}_\eta^{-\m{t}}u/J_\eta)+\mathcal{A}_\eta\grad\m{h})&\\\quad-\gam(\mathcal{A}_\eta\grad)\cdot\mathbb{S}_{\mathcal{A}_\eta}^{H^{-1}(\m{h}+H\circ\varrho(\mathfrak{F}_\eta\cdot e_n))}(\mathcal{A}_\eta^{-\m{t}}u/J_\eta)=H^{-1}(\m{h}+H\circ\varrho(\mathfrak{F}_\eta\cdot e_n))\mathcal{G}\circ\mathfrak{F}_\eta&\\\qquad+\mathcal{F}\circ\mathfrak{F}_\eta&\text{in }\Omega\\
			-((P-P_{\m{ext}})\circ H^{-1}(\m{h}+H\circ\varrho(\mathfrak{F}_\eta\cdot e_n))-\gam\mathbb{S}_{\mathcal{A}_\eta}^{H^{-1}(\m{h}+H\circ\varrho(\mathfrak{F}_\eta\cdot e_n))}(\mathcal{A}_\eta^{-\m{t}}u/J_\eta))\mathcal{N}_\eta&\\\qquad-\varsigma\mathscr{H}(\eta)\mathcal{N}_\eta=\mathcal{T}\circ\mathfrak{F}_\eta\mathcal{N}_\eta&\text{on }\Sigma\\
			u\cdot e_n+\pd_1\eta=0&\text{on }\Sigma\\
			u=0&\text{on }\Sigma_0,
		\end{cases}
	\end{equation}
	where for $\uptau=H^{-1}(\m{h}+H\circ\varrho(\mathfrak{F}_\eta\cdot e_n))$,  $\mathcal{M}=\mathcal{A}_\eta$, and $w=\mathcal{A}_\eta^{-\m{t}}u/J_\eta$ we have used the notation
	\begin{equation}\label{flat_stress_def}
		\S^{\uptau}_{\mathcal{M}}w=\upmu(\uptau)\bp{\grad w\mathcal{M}^{\m{t}}+\mathcal{M}\grad w^{\m{t}}-\f{2}{n}((\mathcal{M}\grad)\cdot w)I}+\uplambda(\uptau)((\mathcal{M}\grad)\cdot w)I\index{\textbf{Fluid mechanical terms}!21@$\mathbb{S}^\tau_{\mathcal{A}}$},
	\end{equation}
 and we also denote $\mathcal{N}_\eta=(-\grad_{\|}\eta,1)$\index{\textbf{Nonlinear maps}!140@$\mathcal{N}_\eta$}.

 Note that in obtaining~\eqref{flattened PDE enthalpy formulation I} we did not only flatten the velocity vector, we also multiplied by the matrix $J_\eta\mathcal{A}_{\eta}^{\m{t}}$. This has the following important consequences. First, we are able to maintain the continuity equation in `perfect divergence' form. Second, the kinematic boundary condition transforms to a linear equation. These properties are indispensable in our analysis.

We have one more change of unknowns left to make before we reach the desired formulation of the problem.  The issue is that the formulation \eqref{flattened PDE enthalpy formulation I} will not quite let us ensure that  $\m{h}$  belongs to a standard Sobolev space.  Rather than develop more nonstandard Sobolev theory, we circumvent this issue with a final change of unknowns and equations by defining $q=\m{h}-\mathfrak{g}\eta$, and multiplying the momentum equation by $\mathcal{A}_\eta^{-1}$, which leads to considerable simplifications in our analysis.  We are then left with the following final equivalent form of our equations for the unknowns $(q,u,\eta)$ with data $(\mathcal{T},\mathcal{G},\mathcal{F})$:
	\begin{equation}\label{The nonlinear equations in the right form}
		\begin{cases}
			\grad\cdot(\sig_{q,\eta}(u-M_\eta e_1))=0&\text{in }\Omega,\\
			\gam^2\sig_{q,\eta}M_{\eta}^{-\m{t}}(((u-M_\eta e_1)\cdot\grad)(M_\eta^{-1}u))+\sig_{q,\eta}\grad(q+\mathfrak{g}\eta)&\\\quad-\gam M_{\eta}^{-\m{t}}(\grad\cdot(\S^{\sig_{q,\eta}}_{\mathcal{A}_\eta}(M_{\eta}^{-1}u)M_{\eta}^{\m{t}}))=J_\eta\sig_{q,\eta}M_{\eta}^{-\m{t}}\mathcal{G}\circ\mathfrak{F}_\eta+J_\eta M_{\eta}^{-\m{t}}\mathcal{F}\circ\mathfrak{F}_\eta&\text{in }\Omega,\\
			-\tp{(P-P_{\m{ext}})\circ\sig_{q,\eta}-\gam\S^{\sig_{q,\eta}}_{\mathcal{A}_\eta}(M_\eta^{-1}u)}M_\eta^{\m{t}}e_n-\varsigma\mathscr{H}(\eta)M_\eta^{\m{t}}e_n=\mathcal{T}\circ\mathfrak{F}_\eta M_\eta^{\m{t}}e_n&\text{on }\Sigma,\\
			u\cdot e_n+\pd_1\eta=0&\text{on }\Sigma,\\
			u=0&\text{on }\Sigma_0,
		\end{cases}
	\end{equation}
	where we have set
	\begin{equation}\label{Mississippi}  \index{\textbf{Nonlinear maps}!150@$M_\eta$}
		M_{\eta}=J_\eta\mathcal{A}_{\eta}^{\m{t}}=\bpm(1+\pd_n\mathcal{E}\eta) I_{(n-1)\times(n-1)}&0_{(n-1)\times 1}\\-\mathcal{E}(\grad_{\|}\eta)&1\epm
		: \Omega \to \R^{n \times n}
	\end{equation}
	and 
	\begin{equation}\label{sigma_q_eta_def}	\index{\textbf{Nonlinear maps}!160@$\sig_{q,\eta}$}
	\sig_{q,\eta}=H^{-1}(q+\mathfrak{g}\eta+H\circ\varrho(\mathfrak{F}_\eta\cdot e_n))=H^{-1}(-\mathfrak{g}\m{id}_{\R^n}\cdot e_n+q+\mathfrak{g}(I-\mathcal{E})\eta)
	: \Omega \to \R.    
	\end{equation}

 As we discussed in Section~\ref{dynamics_and_equilibria}, the original system~\eqref{compressible navier-stokes in Eulerian coordinates} admits a stratified equilibrium solution.  This solution is still encoded in the new formulation~\eqref{The nonlinear equations in the right form} in the sense that for any $\gam\in\R^+$, $(q,u,\eta)=(0,0,0)$ is a solution to~\eqref{The nonlinear equations in the right form} when there are no additional stress or forces present, i.e. when $(\mathcal{T},\mathcal{G},\mathcal{F})=(0,0,0)$.  Furthermore, we show in Corollary \ref{trivial_solns_unique} that these trivial solutions are unique among triples $(q,u,\eta)$ satisfying 
 \begin{equation}
     \sqrt{\tnorm{q}_{H^{4+\tfloor{n/2}}}^2+\tnorm{u}_{H^{5+\tfloor{n/2}}}^2+\tnorm{\eta}^2_{\mathcal{H}^{11/2+\tfloor{n/2}}}} < \rho,
 \end{equation}
 where $\rho$ is a constant depending only on the equilibrium depth $b$, gravity $\mathfrak{g}$, the pressure law $P$, the viscosity coefficients $\upmu$ and $\uplambda$, and the dimension $n$.  We emphasize, though, that this result is recorded in this form for simplicity but could be improved. This reflects the fact that we only expect traveling wave solutions to exist for viscous fluids if they are generated by stress and force.  Solutions to \eqref{The nonlinear equations in the right form} also obey a balance of power and dissipation analogous to \eqref{she said I know what its like to be dead}; this is recorded in Corollary \ref{flat_diss_power_id}.

 \subsection{Statement of main result}\label{sec_main_results}

We now state our main results. To do so, we first need to introduce a bit of notation to describe the functional framework.  In our work, as in the previous work of Leoni and Tice~\cite{leoni2019traveling}, Stevenson and Tice~\cite{MR4337506}, and Koganemaru and Tice~\cite{koganemaru2022traveling} on traveling wave solutions to the incompressible analog of~\eqref{compressible navier-stokes traveling wave equations rescaled}, and Nguyen and Tice~\cite{nguyen_tice_2022} on traveling wave solutions to the one-phase Muskat problem, the traveling wave structure forces the free surface functions to belong to a scale of nonstandard anisotropic Sobolev spaces, the properties of which end up playing a crucial role in the analysis. Therefore, in order to properly state our main theorem, we first introduce these spaces.

For $\R\ni s\ge 0$ and $d\in\N^+$ we define the anisotropic Sobolev space
	\begin{equation}\label{in the introduction we have an ansiobro}
		\mathcal{H}^s(\R^d)=\tcb{f\in\mathscr{S}^\ast(\R^d;\R)\;:\;\mathscr{F}[f]\in L^1_{\m{loc}}(\R^d;\C),\;\tnorm{f}_{\mathcal{H}^s} <\infty},
	\end{equation}
	equipped with the norm
	\begin{equation}
 \tnorm{f}_{\mathcal{H}^s}
 =\bp{\int_{\R^d}\sp{|\xi|^{-2}(\xi_1^2+|\xi|^4)\mathds{1}_{B(0,1)}(\xi) + \tbr{\xi}^{2s}\mathds{1}_{\R^d\setminus B(0,1)}(\xi)} |\mathscr{F}[f](\xi)|^2\;\m{d}\xi}^{1/2}.
	\end{equation}
 We refer to Appendix~\ref{appendix on anisotropic Sobolev spaces} for more information on these function spaces, but for the purposes of stating our main theorem, we note here that  $\mathcal{H}^s(\R^d)$ is a Hilbert space and $H^s(\R^d) \emb \mathcal{H}^s(\R^d) \emb H^s(\R^d) + C^\infty_0(\R^d)$, with equality in the first embedding if and only if $d=1$.
 
 For the following statement of the main theorem, we set $r=10+2\tfloor{n/2}$ and for $s\in\N$ we define the sets
\begin{equation}
    U_s=H^{1+s+r}(\R^n;\R^{n\times n})\times H^{s+r}(\R^n;\R^n)\times H^{s+r}(\R^n;\R^n)\times\R^+,
\end{equation}
\begin{equation}\label{its me, hi, im the problem its me}
    V_s=H^{1+s+r}(\Omega)\times H^{2+s+r}(\Omega;\R^n)\times\mathcal{H}^{5/2+s+r}(\R^{n-1}),
\end{equation}
\begin{equation}\label{the label master strikes again! - RETURN OF HTAT!}
    V^0=\tcb{(q,u,\eta)\in V_0\;:\;\m{Tr}_{\Sigma_0}(u)=0,\;\m{Tr}_{\Sigma}(u\cdot e_n)+\pd_1\eta=0}.
\end{equation}
We can now state our main theorem, which establishes the well-posedness of the traveling wave formulation of free boundary compressible Navier-Stokes equations~\eqref{The nonlinear equations in the right form}.

\begin{customthm}{1}[Proved in Theorem~\ref{main thm 2 on traveling wave solutions to the free boundary compressible Navier-Stokes equations}]\label{thm on main}
Assume that the parameters $\upmu$, $\uplambda$, and $\varsigma$ satisfy \eqref{parameter_assumptions} and that $P$ satisfies $P'>0$ and~\eqref{equilibrium_ccs}.   There exist a collection $\{\mathscr{V}(\gam)\}_{\gam \in \R^+}$ of open subsets of $V^0$ and a  nonincreasing sequence $\tcb{\mathscr{U}_s}_{s=0}^\infty$ of open subsets of $U_0$ such that the following hold.
    \begin{enumerate}
        \item Nondegeneracy: We have that $\tcb{0}\times\R^+\subseteq\bigcap_{s=0}^\infty\mathscr{U}_s$ and $0\in\bigcap_{\gam\in\R^+}\mathscr{V}(\gam)$.
        \item Existence and uniqueness: For all tuples $(\mathcal{T},\mathcal{G},\mathcal{F},\gam)\in\mathscr{U}_0$ of applied stress, specific bulk force, bulk force, and wave speed, there exists a unique solution $(q,u,\eta)\in\mathscr{V}(\gam)$ such that the traveling wave reformulation for the free boundary compressible Navier-Stokes equations, system~\eqref{The nonlinear equations in the right form}, is classically satisfied with data $(\mathcal{T},\mathcal{G},\mathcal{F})$ and wave speed $\gam$.
        \item Regularity, given low norm smallness: If $s\in\N^+$ and $(\mathcal{T},\mathcal{G},\mathcal{F},\gam)\in\mathscr{U}_s\cap U_s$, then the corresponding solution satisfies $(q,u,\eta)\in\mathscr{V}(\gam)\cap V_s$.
        \item Continuous dependence: For any $s\in\N$, the solution map 
        \begin{equation}\label{thm on main_soln map}
            \mathscr{U}_s\cap U_s\ni(\mathcal{T},\mathcal{G},\mathcal{F},\gam)\mapsto(q,u,\eta)\in V_s \cap V^0
        \end{equation}
        is continuous with respect to the $U_s$ and $V_s\cap V^0$ topologies.
        \item No vacuum formation: There exists positive constants $c,C\in\R^+$ such that for all $(q,u,\eta)\in\bigcup_{\gam\in\R^+}\mathscr{V}(\gam)$ we have that $c\le \sig_{q,\eta}\le C$, where $\sig_{q,\eta}$ is defined in~\eqref{sigma_q_eta_def}.
        \item Flattening map diffeomorphism: For any $s\in\N$ and $(q,u,\eta)\in V_s\cap\bigcup_{\gam\in\R^+}\mathscr{V}(\gam)$, we have that the flattening map $\mathfrak{F}_\eta$ from~\eqref{flattening_map_def} is a smooth diffeomorphism from $\Omega$ to $\Omega[\eta]$ that extends to a         $C^{12+\tfloor{n/2}+s}$ diffeomorphism from $\Bar{\Omega}$ to $\Bar{\Omega[\eta]}$. 
    \end{enumerate}

\end{customthm}

Before enumerating some corollaries, we pause for a few comments and remarks.  First, we emphasize that a high-level summary of our theorem is that traveling waves for the free boundary compressible Navier-Stokes system are generic: they exist for all nontrivial wave speeds $\gamma \in \R^+$, and for a fixed regularity index $s \in \N$ the set of stress-force-speed data, $\mathscr{U}_s\cap U_s$, is open.  In particular, for each fixed wave speed $\gamma \in \R^+$ and $s \in \N$, the set of stress and force data for which we can solve~\eqref{The nonlinear equations in the right form} is an open set containing the origin, so the existence of small-data solutions is a generic phenomenon.  The utility of working in $\mathscr{U}_s\cap U_s$ is that it allows us to prove the joint continuity of our solutions with respect to both the stress-forcing data $(\mathcal{T}, \mathcal{G},\mathcal{F})$ and the wave speed $\gamma$.  In fact, we can say a bit more: we show in Theorem~\ref{main thm 1 on traveling wave solutions to the free boundary compressible Navier-Stokes equations} and Remark \ref{abstract_higher_reg_remark} that the solution map~\eqref{thm on main_soln map} enjoys a certain form of continuous differentiability, and even higher regularity. Stating this precisely entails carefully dealing with more issues related to derivative loss, so we have skipped this technical point in the statement of Theorem~\ref{thm on main} for the sake of brevity.

\begin{figure}[!h]
    \centering
    \includegraphics{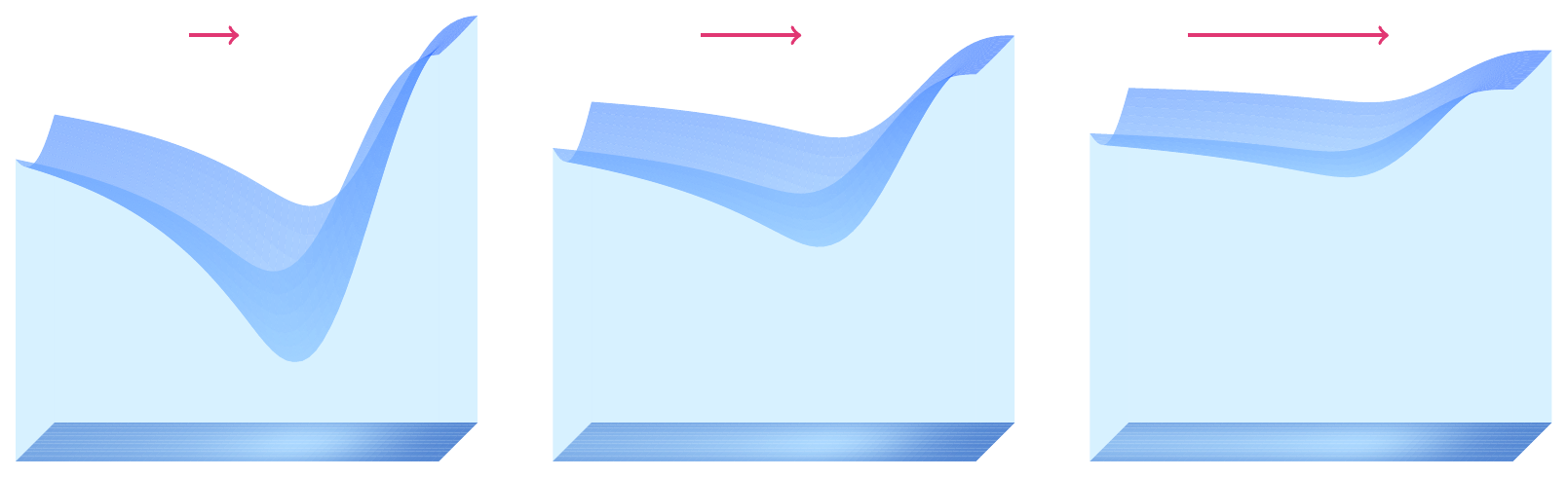}
    \caption{A depiction of three traveling wave free surfaces generated by the same applied stress tensor $\mathcal{T}=-\varphi I_{3\times 3}$, where $\varphi\ge0$ is compactly supported, with increasing wave speed $\gam$ moving from left to right, indicated by the red arrows.}
    \label{various wave speeds}
\end{figure}

Second, it is worth highlighting that our theorem does slightly different things when $n=2$ as compared to when $n \ge 3$.  In the case $n \ge 3$, the hypothesis~\eqref{parameter_assumptions} requires that the coefficient of surface tension is positive, $\varsigma >0$, and our analysis crucially uses this and the ellipticity of the mean curvature operator to gain regularity for $\eta$.  When $n=2$ we only require $\varsigma \ge 0$, as in this case there is another mechanism built into the equations for regularity gain, namely the equation $\partial_1 \eta = - u \cdot e_n$ on $\Sigma$, which is elliptic when $n=2$ since then $\partial_1$ is elliptic as a differential operator on $\Sigma \simeq \R$.  When $n \ge 3$ we have to contend with the free surface function $\eta$ belonging to the anisotropic spaces $\mathcal{H}^{5/2+s}(\R^{n-1})$, which are strictly larger than $H^{5/2+s}(\R^{n-1})$, but when $n =2$ we have that $\mathcal{H}^{5/2+s}(\R)= H^{5/2+s}(\R)$, and so our solutions live in standard Sobolev spaces.  Interestingly, when $n=2$ there is one condition that must be stronger than when $n \ge 3$; indeed, \eqref{parameter_assumptions} requires positive bulk viscosity, $\uplambda >0$, when $n=2$ and only $\uplambda \ge 0$ when $n \ge 3$.  This is ultimately related to technical issues with the deviatoric Korn inequality; we refer to Remark~\ref{rmk that even a watched pot boils twice per day} for further details.

In spite of the above dimensional differences in the precise functional setting, the space $V^0\cap V_s$ from~\eqref{its me, hi, im the problem its me} and~\eqref{the label master strikes again! - RETURN OF HTAT!} always satisfies the embedding (thanks to Proposition \ref{proposition on frequency splitting} and standard Sobolev embeddings)
\begin{equation}
        V^0\cap V_s\emb C^{s+k}_0(\Omega)\times C^{s+k+1}_0(\Omega;\R^n)\times C^{s+k+2}_0(\R^{n-1})
\end{equation}
for $k=10+\tfloor{n/2}$.  Since our theorem guarantees the inclusion $(q,u,\eta) \in V_0\cap V_s$, we see that our solutions always decay to zero at infinity, which means that our solutions are what are known as solitary waves in the parlance of the traveling wave literature.  At the level of generality of our well-posedness theory, there is not much more qualitative information that can be deduced about our solutions.  However, our result opens the door to more detailed qualitative studies given specific forms of the stress-forcing data tuple $(\mathcal{T},\mathcal{G},\mathcal{F})$.

We also point out that our techniques only allow us to construct traveling wave solutions ($\gamma\in \R^+)$ and not stationary $(\gamma=0)$ solutions. The strict sign condition on $\gamma$ plays a key role in our analysis by, in part, allowing us to use a nondegenerate norm on the collection of free surface functions. Our selected functional framework simply does not work with $\gamma =0$.  We face a similar issue with the gravitational constant in our analysis; indeed, we can only construct solutions in the case that $\mathfrak{g}>0$.

We now turn our attention to two corollaries of Theorem~\ref{thm on main}.  The first formalizes the above discussion about the open set of stress-force data for a given wave speed $\gamma \in \R^+.$

\begin{customcoro}{2}[Proved in Corollary~\ref{coro on rain}]
    With $r=10+2\tfloor{n/2}$ as before, for each $\gam\in\R^+$ there exists a nonempty open set
    \begin{equation}
    (0,0,0)\in\mathscr{W}(\gam)\subset H^{1+r}(\R^n;\R^{n\times n})\times H^r(\R^n;\R^n)\times H^r(\R^n;\R^n)
    \end{equation}
with the property that for all stress-force data tuples $(\mathcal{T},\mathcal{G},\mathcal{F})\in\mathscr{W}(\gam)$ there exists a unique $(q,u,\eta)\in\mathscr{V}(\gam)$ (where the latter open set is from the statement of Theorem~\ref{thm on main}) such that system~\eqref{The nonlinear equations in the right form} is satisfied with solution $(q,u,\eta)$, wave speed $\gam$, and data $(\mathcal{T},\mathcal{G},\mathcal{F})$.
\end{customcoro}

Although the natural formulation for the traveling wave problem from the perspective of well-posedness is in a flattened domain as in~\eqref{The nonlinear equations in the right form}, we can also switch our solutions back to the Eulerian formulation~\eqref{compressible navier-stokes traveling wave equations rescaled}.  We record this in our second corollary.

\begin{customcoro}{3}[Proved in Corollary~\ref{blowing in the wind}]
    Each solution to the flattened perturbative enthalpy formulation for the traveling wave problem for free boundary compressible Navier-Stokes, i.e system~\eqref{The nonlinear equations in the right form}, produced by Theorem~\ref{thm on main} gives rise to a classical solution to the traveling Eulerian formulation of the problem given by system~\eqref{compressible navier-stokes traveling wave equations rescaled}.
\end{customcoro}

 \subsection{Summary of strategy and layout of paper}\label{sec_strategy_label_evil_duck}

In this subsection we aim to summarize the principal difficulties in proving Theorem \ref{thm on main} and our strategies for overcoming them.  This will also serve to outline the structure of the paper.

\textbf{High level summary of difficulties.}  The boundary value problem~\eqref{The nonlinear equations in the right form} is posed in an unbounded domain with infinite measure and  non-compact boundary, the equations are quasilinear, and there is no variational structure; as such, compactness, Fredholm, and variational techniques are unavailable.  This, along with our expectation of a robust linear theory, suggests that the construction of solutions should proceed through perturbative techniques such as the implicit function theorem, or more fundamentally, an iteration scheme based on some sort of linearization.  Indeed, an implicit function theorem strategy, based on the linearization of the equations around vanishing stress-force data and trivial solution triple for a fixed arbitrary wave speed $\gamma \in \R^+$, proved successful in recent work \cite{koganemaru2022traveling,leoni2019traveling,MR4337506} on the incompressible version of \eqref{The nonlinear equations in the right form}, so it is enlightening to begin our discussion by stating the corresponding linearization of \eqref{The nonlinear equations in the right form}:
	\begin{equation}\label{linearization_at_zero}
		\begin{cases}
			\grad\cdot(\varrho u)+\mathfrak{g}^{-1}\varrho'\pd_1(q+\mathfrak{g}\eta)=g&\text{in }\Omega,\\
			-\gam^2\varrho\pd_1 u+\varrho\grad(q+\mathfrak{g}\eta)-\gam\grad\cdot\S^{\varrho}u=f&\text{in }\Omega,\\
			-(\varrho q-\gam\S^{\varrho}u)e_n-\varsigma\Delta_{\|}\eta e_n=k&\text{on }\Sigma,\\
			u\cdot e_n+\pd_1\eta=0&\text{on }\Sigma,\\
			u=0&\text{on }\Sigma_0,
		\end{cases}
	\end{equation}
where the linearized unknowns are still labeled $(q,u,\eta)$ but the linearized data is now the triple $(g,f,k)$.

The reader familiar with the elliptic structure of the incompressible Stokes problem will recognize a fundamental difficulty appearing already in the first two equations of \eqref{linearization_at_zero}: even if we ignore $\eta$ or view it as given, these two equations do not constitute an elliptic system for $(q,u)$ in the sense of Agmon, Douglis, and Nirenberg \cite{MR162050}, due to the appearance of $\partial_1 q$ in the first equation.  Without this elliptic structure to serve as a base for the analysis, it is not obvious that \eqref{linearization_at_zero} will give rise to an isomorphism between Banach spaces, a necessary ingredient for the perturbation strategy.  Remarkably, in spite of this ellipticity failure, we are able to show in Theorem \ref{thm on existence for the principal part} that the forward linear map $(q,u,\eta) \mapsto (g,f,k)$ defined by \eqref{linearization_at_zero} actually does induce an isomorphism between the Sobolev-type Hilbert spaces $\overset{0,0,0}{\X^s}$ and $\mathbb{Y}^s$ for $s \in \N$, as defined by \eqref{Oregon} and \eqref{Y^s_def}, respectively.  Unpacking the details of the space $\overset{0,0,0}{\X^s}$ reveals the fundamental difficulties lurking in \eqref{The nonlinear equations in the right form} and motivates our overall strategy.  

A cursory glance at the spaces $\overset{0,0,0}{\X^s}$ and $\mathbb{Y}^s$ shows that the regularity count essentially matches that of the incompressible problem if we formally identify $q$ with the pressure in the incompressible problem: in the domain space, if $q$ gets $1+s$ derivatives, then $u$ gets $2+s$, and $\eta$ gets $5/2+s$, while in the codomain $g$ gets $1+s$ derivatives, $f$ gets $s$, and $k$ gets $1/2 +s$.  However, a closer inspection reveals two crucial complications with these spaces and this counting scheme.  The first, which was already present in the analysis of the incompressible problem, is that the structure of the operators hitting $\eta$ in \eqref{linearization_at_zero} only allows for the recovery of estimates of
\begin{equation}\label{aisumimasen}
    \Delta_{\|}\eta \in H^{1/2+s}(\R^{n-1}), \grad_{\|}\eta \in H^s(\R^{n-1};\R^{n-1}), \text{ and } \partial_1 \eta \in \dot{H}^{-1}(\R^{n-1}),
\end{equation}
where $\dot{H}^{-1}(\R^{n-1})$ is defined by \eqref{dotHminus1_def}, which is not enough to guarantee the inclusion $\eta \in H^{5/2+s}(\R^{n-1})$ for general $n$.  Instead, as we prove in Proposition \ref{proposition on spatial characterization of anisobros}, these inclusions essentially characterize the anisotropic inclusion $\eta \in \mathcal{H}^{5/2+s}(\R^{n-1})$.  The takeaway is that the anisotropic spaces are inextricably linked to the traveling wave problem through the structure of the differential operators (and also through the positivity $\mathfrak{g}>0$ and $\gam>0$, which give us the latter two estimates of~\eqref{aisumimasen}). The second complication is more severe and new to the compressible problem:  knowing that $q \in H^{1+s}(\Omega)$, $u \in H^{2+s}(\Omega;\R^n)$, and $\eta \in \mathcal{H}^{5/2+s}(\R^{n-1})$ alone is not enough to guarantee that $g \in H^{1+s}(\Omega)$.  With this count, the continuity equation in \eqref{linearization_at_zero} only yields $g \in H^s(\Omega)$.  To achieve the higher regularity inclusion $g \in H^{1+s}(\Omega)$ we have to build the extra condition that $\partial_1 q \in H^{1+s}(\Omega)$ into the domain space $\overset{0,0,0}{\X^s}$, and conversely, with $g \in H^{1+s}(\Omega)$ we are able to recover that $\partial_1 q \in H^{1+s}(\Omega)$ through the linearized continuity equation.  This is what the $0,0,0$ adornment actually indicates for the space $\overset{0,0,0}{\X^s}$: the $q$ elements in this space enjoy some `bonus partial regularity' whose precise form is dictated by the structure of the linearized operator around the trivial triple $(0,0,0)$.  The bonus partial regularity can also be viewed as another manifestation of anisotropy in the problem.

The fact that \eqref{linearization_at_zero} induces an isomorphism $\overset{0,0,0}{\X^s} \ni (q,u,\eta) \mapsto (g,f,k) \in \mathbb{Y}^s$ is certainly encouraging, but in reality it exposes a much deeper complication with the nonlinear problem \eqref{The nonlinear equations in the right form}.  Indeed, if we attempt to formulate \eqref{The nonlinear equations in the right form} as a nonlinear mapping problem on a space in which $(q,u,\eta) \in \overset{0,0,0}{\X^s}$, then we immediately see that the bonus partial regularity is lost by the nonlinearity.  In more concrete terms: any attempt to solve \eqref{The nonlinear equations in the right form} through an iteration scheme based on the isomorphism from \eqref{linearization_at_zero} will suffer from derivative loss, rendering the scheme useless.
 
This isomorphism issue is actually more generic.  We also prove in Theorem \ref{thm on existence for the principal part} that for any appropriately small triple $(q_0,u_0,\eta_0)$, the linearization of \eqref{The nonlinear equations in the right form} around $(q_0,u_0,\eta_0)$ for $(\mathcal{T},\mathcal{G},\mathcal{F}) = 0$ and $\gamma \in \R^+$  induces an isomorphism $\overset{q_0,u_0,\eta_0}{\X^s} \ni (q,u,\eta) \mapsto (g,f,k) \in \mathbb{Y}^s$, where now the `adapted space' $\overset{q_0,u_0,\eta_0}{\X^s}$ encodes the bonus partial regularity  $v_{q_0,u_0,\eta_0} \cdot \grad q \in H^{1+s}(\Omega)$  for a vector field $v_{q_0,u_0,\eta_0}$ that is determined by the linearization location $(q_0,u_0,\eta_0)$ (the field is collinear with $e_1$ at $(0,0,0)$).  Once more, this can be viewed as a sort of anisotropy, but now it is clear that the favored direction depends on the background triple $(q_0,u_0,\eta_0)$.  These general adapted spaces face the same problem described above: the bonus regularity is lost by the nonlinearity, leading to derivative loss in any iteration scheme.  The failure of the nonlinearity to preserve the adapted spaces can ultimately be traced to the fact that the bonus regularity is not perturbative: we cannot use control of  $v_{q_0,u_0,\eta_0} \cdot \grad q$ to say anything useful about $v_{q_1,u_1,\eta_1} \cdot \grad q$ for general distinct triples  $(q_i,u_i,\eta_i)$ with $i  \in \{0,1\}$.  In other words, in general the adapted spaces $\overset{q_0,u_0,\eta_0}{\X^s}$ and $\overset{q_1,u_1,\eta_1}{\X^s}$ are inequivalent.

While these issues preclude the use of elementary perturbation techniques, they also reveal the potential utility of more sophisticated Nash-Moser techniques.  Indeed, for any appropriate triple $(q_0,u_0,\eta_0)$, we have the natural inclusion $\X^{s+1} \hookrightarrow \overset{q_0,u_0,\eta_0}{\X^s}$, where the former space is defined by \eqref{domain banach scales} and encodes no location-specific bonus regularity. This suggests that we may pose the nonlinear mapping from \eqref{The nonlinear equations in the right form} on a scale of spaces, indexed by $s$, in which the codomain involves $\mathbb{Y}^s$ but, to compensate for the derivative loss, the domain scale is shifted and requires $(q,u,\eta) \in \X^{s+1}$. The above isomorphism results then suggest that the maps $\mathbb{Y}^s \ni (g,f,k) \mapsto (q,u,\eta) \in \overset{q_0,u_0,\eta_0}{\X^s} \hookrightarrow \X^s$ will allow us to construct the right and left inverse to the derivative of the nonlinear map, provided we expand our view to scales of Banach spaces and accept the reality of derivative loss.  This is precisely the purview of the Nash-Moser technique \cite{nash_1956,Moser_1966}, which we have thus chosen as the engine to prove Theorem \ref{thm on main}.  

\textbf{Nash-Moser framework.}  Our goal in studying \eqref{The nonlinear equations in the right form} is not just to show the existence and uniqueness of solutions, but to establish a proper well-posedness theory that shows the solutions depend continuously on the data in the optimal topology. To the best of our knowledge, the only Nash-Moser inverse function theorems in the literature with this capability are the formulations of Sergeraert~\cite{Sergeraert_1972} and Hamilton~\cite{MR656198}, which actually yield smoothness of the inverse map.  The Sergeraert and Hamilton Nash-Moser theorems work in the context of smooth tame maps between Fr\'echet spaces. Roughly speaking, one can think of tameness as a family of structured estimates associated to the derivatives of the maps; these bounds play an essential role in the use of a modification of Newton's method to overcome the derivative loss in proving surjectivity. The Fr\'echet spaces that serve as the domain and codomain  of the nonlinear operator in these theorems can be thought of as the intersection of all of the spaces in a Banach scale (like $H^\infty(\R^n) = \bigcap_{k \in \N} H^k(\R^n)$ vis-\`a-vis the Banach scale $\{H^k(\R^n)\}_{k \in \N}$). The hypotheses of the Sergeraert and Hamilton Nash-Moser variants require, among other things, a family of right inverses to the derivative mapping into the domain Fr\'echet space.  Unfortunately, for reasons that are ultimately attributable to the failure of hypoellipticity for the hyperbolic structure appearing in the traveling wave formulation of the continuity equation, in our context we can only verify that in a given open neighborhood of the trivial solution, the derivatives' inverses only map into \emph{finite} regularity spaces in the Banach scale, and so we cannot satisfy the basic hypotheses of Sergeraert's or Hamilton's formulation.

There are Nash-Moser theorems in the literature that allow for this finite invertibility range.  The oldest we are aware of is found in the work of Schwartz \cite{Schwartz_1960,Schwartz_1969}, but this is formulated only for a very specific Banach scale, produces solutions in a suboptimal space, and has no mechanism for regularity promotion.  The version due to H\"ormander~\cite{MR602181,MR802486,MR1039355} works for general Banach scales and produces solutions in the optimal or nearly-optimal space in an associated `weak Banach scale,' which in some cases coincides with the original (e.g. H\"older scales) but in general is slightly larger (e.g. if the original is the Sobolev scale $\{H^k(\R^n)\}_{k \in \N}$, then the weak scale is the larger Besov scale $\{B^k_{2,\infty}(\R^n)\}_{k \in \N}$).  This result also has no regularity promotion mechanism.  The recent Nash-Moser theorem of Baldi and Haus \cite{MR3711883} produces solutions in the optimal space and also has a regularity promotion mechanism.  Famously, the Nash-Moser technique overcomes the problem of derivative loss by employing a family of smoothing operators that satisfy a host of precise quantitative estimates. Baldi and Haus achieve their significant improvement by placing more strenuous conditions on these smoothing operators than in the other Nash-Moser formulations, which allows them to port techniques from Littlewood-Paley theory and the paradifferential calculus into the abstract setting.

In the context of the function spaces we employ, it is relatively easy to construct smoothing operators that satisfy the hypotheses of, say H\"ormander's Nash-Moser theorem~\cite{MR802486}.  However, in our context it is rather delicate to show that these operators satisfy the stronger hypotheses of Baldi and Haus~\cite{MR3711883}. Rather than focus effort on this construction, which would mostly apply to our specific problem, we have chosen to use a more general, abstract approach, which is a Banach scale generalization of the notion of a tame Fr\'echet space as defined in Hamilton~\cite{MR656198}, and which we believe may be of use in other problems. Roughly speaking, the idea is to view a given scale of Banach spaces, in which it is hard or impossible to construct the smoothing operators, as being a retract of another Banach scale in which the smoothing operators are known to exist. In the setting of Banach spaces, we can think of the retract property in terms of the former spaces being direct summands, or complemented subspaces, of the latter.  We emphasize that this idea does not pull the smoothing operators back to the initial scale, but rather pushes the nonlinear map forward to the larger scale, and so in some sense our method shifts the focus from constructing smoothing operators to identifying the direct summand structure, which is more amenable to PDE techniques.

In light of the above discussion, we have opted to craft another version of the abstract Nash-Moser inverse function theorem, synthesizing the desired elements of the  Sergeraert,  Hamilton, and Baldi and Haus formulations, that is capable of working within our finite range of invertibility context, produces solutions in the optimal space, provides a regularity promotion mechanism, and provides some degree of regularity for the inverse map, in particular a continuity assertion in an optimally strong norm.  Our new version employs the direct summand method to sidestep the smoothing operator construction.  This new Nash-Moser formulation, the precise statement of which is given in Theorems~\ref{thm on nmh} and~\ref{thm on further conclusions of the inverse function theorem}, is essential for our proof of Theorem~\ref{thm on main}, but we believe it is likely to be of broader interest and applicability due to the flexibility of its hypotheses and improved conclusions.  We refer to  Section~\ref{section on NMH} for further exposition.

With our Nash-Moser variant in hand, our strategy for proving Theorem \ref{thm on main} is simple to state: encode the conclusions of the theorem as properties of a nonlinear map associated the system~\eqref{The nonlinear equations in the right form} that can be granted by our Nash-Moser inverse function theorem, and then verify the hypotheses of Nash-Moser. These hypotheses, which are stated precisely in Definition~\ref{defn of the mapping hypotheses} and Theorem~\ref{thm on nmh}, are divided into two categories: nonlinear and linear.  For the former, we need to verify that the nonlinear operator satisfies certain differentiability conditions and that the derivatives obey certain tame estimates.  At the linear level we need to study the linearization of~\eqref{The nonlinear equations in the right form} around a generic triple $(q_0,u_0,\eta_0)$ in an open neighborhood of zero, and in particular we need to construct the family of left and right inverses to these derivatives and verify they also obey a set of tame estimates.   

The diagram in Figure~\ref{fig:flow_diagram} represents the logical flow of dependencies for our strategy as it is implemented in this paper. The gray boxes correspond to the abstract nonlinear analysis in Section~\ref{section on NMH}, which culminates with the inverse function theorem. The green boxes correspond to the main features of our nonlinear analysis, which are found in Section~\ref{icelandic arctic char}. The blue boxes show our linear analysis strategy, which is then executed in Sections~\ref{Section: Analysis of Regularized Steady Transport Equations}, \ref{chilean sea bass}, and~\ref{alaskan black cod}. Finally, at the bottom, we have the red box representing our conclusion and final proof of Theorem~\ref{thm on main}, appearing in Section~\ref{she could steal, but she could not rob}.

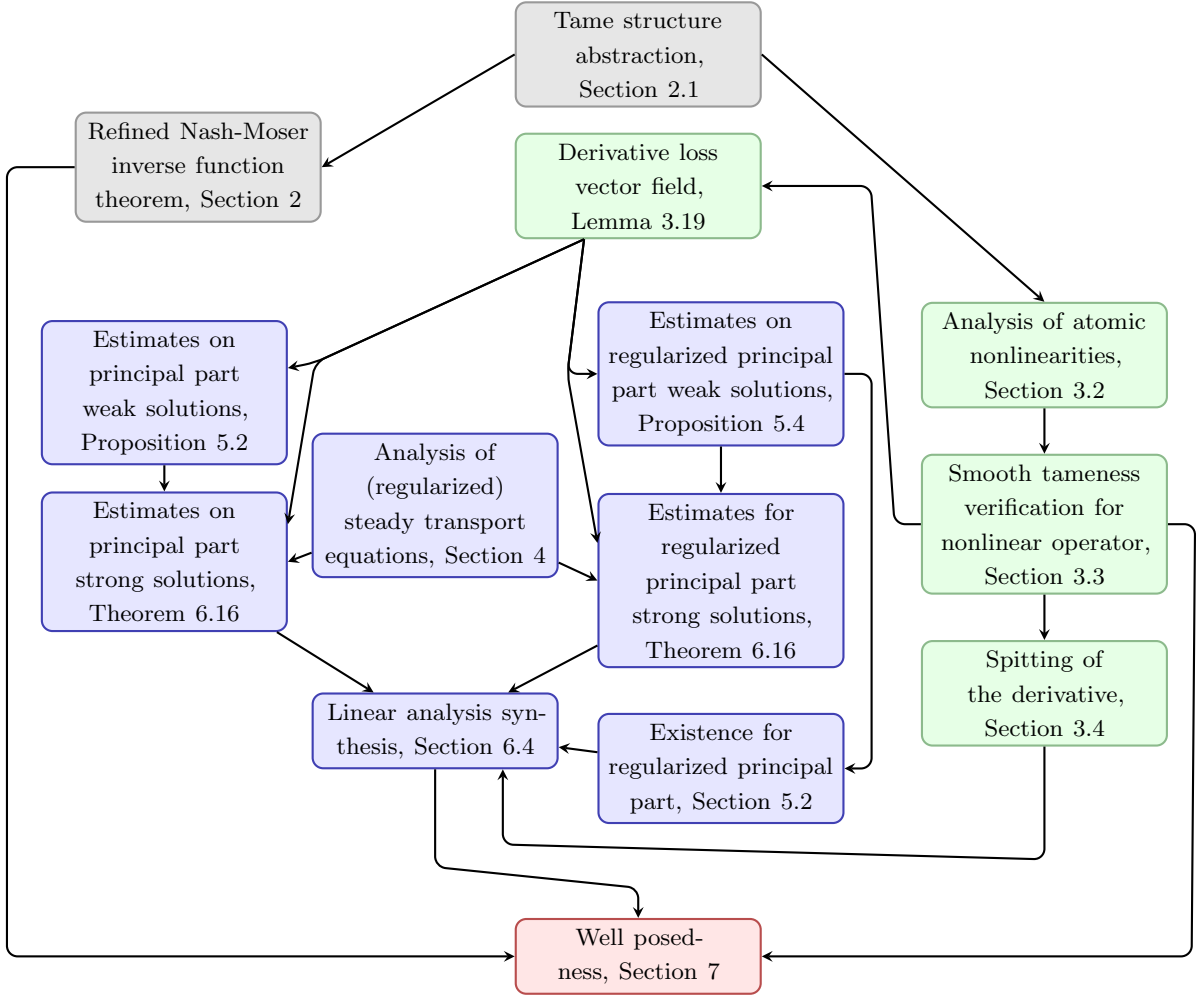
\begin{figure}[!h]
    \centering

\usetikzlibrary {calc,positioning,shapes.misc}
\begin{tikzpicture}[scale = 1,every node/.style={scale=1},abstract/.style={rectangle,rounded corners,thick,draw=black!40,fill=black!10,text width=3cm,align=center}, nonlinear/.style={rectangle,rounded corners,thick,draw=green!40!black!45,fill=green!10,text width=3cm,align=center},linear/.style={rectangle,rounded corners,thick,draw=blue!60!black!75,fill=blue!10,text width=3cm,align=center},combo/.style={rectangle,rounded corners,thick,draw=red!60!black!70,fill=red!10,text width=3cm,align=center}]

\node[anchor = center] at (0,-1) [combo] (conclusion) {\footnotesize Well posedness, Section~\ref{she could steal, but she could not rob}};

\node[anchor = center] at (1.1,1.5) [linear] (ERPP) {\footnotesize Existence for regularized principal part, Section~\ref{please dont spoil my day, im miles away and}};

\node[anchor = center] at (-2.7,2) [linear] (analysis) {\footnotesize Linear analysis synthesis, Section~\ref{dotting his socks in the night}};

\node[anchor = center] at (-5.85,9.5) [abstract] (nmh) {\footnotesize Refined Nash-Moser inverse function theorem, Section~\ref{section on NMH}};

\node[anchor = center] at (5.4,4.75) [nonlinear] (tamesmooth) {\footnotesize Smooth tameness verification for nonlinear operator, Section~\ref{section on smooth tameness of the nonlinear map}};

\node[anchor = center] at (0,11) [abstract] (tame) {\footnotesize Tame structure abstraction, Section~\ref{Michigan}};

\node[anchor = center] at (0,9.25) [nonlinear] (vector) {\footnotesize Derivative loss vector field, Lemma~\ref{properties of the principal parts vector field}};

\node[anchor = center] at (5.4,7) [nonlinear] (atom) {\footnotesize Analysis of atomic nonlinearities, Section~\ref{only place he's ever been}};

\node[anchor = center] at (5.4,2.5) [nonlinear] (split) {\footnotesize Spitting of the derivative, Section~\ref{subsection on derivative splitting and notation for linear analysis}};

\node[anchor = center] at (-2.7,5.0) [linear] (steady) {\footnotesize Analysis of (regularized) steady transport equations, Section~\ref{Section: Analysis of Regularized Steady Transport Equations}};

\node[anchor = center] at (-6.3,4.25) [linear] (SPP) {\footnotesize Estimates on principal part strong solutions, Theorem~\ref{thm on a priori estimates for the principal part and the regularization}};

\node[anchor = center] at (1.1,4) [linear] (SRPP) {\footnotesize Estimates for regularized principal part strong solutions, Theorem~\ref{thm on a priori estimates for the principal part and the regularization}};

\node[anchor = center] at (1.1,6.75) [linear] (WRPP) {\footnotesize Estimates on regularized principal part weak solutions, Proposition~\ref{prop on a priori estimates for regularized weak solutions}};

\node[anchor = center] at (-6.3,6.5) [linear] (WPP) {\footnotesize Estimates on principal part weak solutions, Proposition~\ref{prop on a priori estimates for weak solutions}};

\draw[-stealth,thick,rounded corners] (WPP) -- (SPP);
\draw[-stealth,thick,rounded corners] (WRPP) -- (SRPP);
\draw[-stealth,thick,rounded corners] (steady) -- (SRPP.west);
\draw[-stealth,thick,rounded corners] (steady) -- (SPP.east);
\draw[-stealth,thick,rounded corners] (SPP) -- (analysis);
\draw[-stealth,thick,rounded corners] (SRPP) -- (analysis);
\draw[-stealth,thick,rounded corners] (ERPP) -- (analysis);
\draw[-stealth,thick,rounded corners] (WRPP.east) -- ($(WRPP.east)+(0.36,0)$) -- ($(ERPP.east)+(0.36,0)$) -- (ERPP.east);
\draw[-stealth,thick,rounded corners] ($(vector.south)+(-0.72,0)$) -- ($(WRPP.west)+(-0.4,0)$) -- (WRPP);
\draw[-stealth,thick,rounded corners] ($(vector.south)+(-0.72,0)$) -- ($(WRPP.west)+(-0.4,0)$) -- ($(SRPP.west)+(0,0.5)$);
\draw[-stealth,thick,rounded corners] ($(vector.south)+(-0.72,0)$) -- ($(WPP.east)+(0.36,0.4)$) -- (WPP);
\draw[-stealth,thick,rounded corners] ($(vector.south)+(-0.72,0)$) -- ($(WPP.east)+(0.36,0.4)$) -- ($(SPP.east)+(0,0.5)$);
\draw[-stealth,thick,rounded corners] (tame.east) -- (atom.north);
\draw[-stealth,thick,rounded corners] (tame.west) -- (nmh.east);
\draw[-stealth,thick,rounded corners] (atom) -- (tamesmooth);
\draw[-stealth,thick,rounded corners] (tamesmooth) -- (split);
\draw[-stealth,thick,rounded corners] (split.south) -- ($(split.south)+(0,-1.5)$) -- ($(analysis.south)+(.9,-1)$) -- ($(analysis.south)+(.9,-0)$);
\draw[-stealth,thick,rounded corners] (tamesmooth.west) -- ($(tamesmooth.west)+(-0.35,0)$) -- ($(vector.east)+(1.45,0)$) -- (vector.east); 
\draw[-stealth,thick,rounded corners] (tamesmooth.east) -- ($(tamesmooth.east)+(0.36,0)$)  -- ($(conclusion.east)+(5.715,0)$) -- (conclusion.east);
\draw[-stealth,thick,rounded corners] (analysis.south) -- ($(analysis.south)+(0,-1.3)$) -- ($(conclusion.north)+(0,0.4)$) -- (conclusion.north);
\draw[-stealth,thick,rounded corners] (nmh.west) -- ($(nmh.west)+(-.9,0)$) -- ($(conclusion.west)+(-6.75,0)$) -- (conclusion.west);
\end{tikzpicture}

    \caption{Logical flow of the paper}
    \label{fig:flow_diagram}
\end{figure}

\textbf{Nonlinear analysis.}  Our abstract nonlinear analysis begins in Section \ref{section on NMH} with the study of the tame structure of differentiable maps between scales of Banach spaces, which form the basic framework of our Nash-Moser theorem.  In contrast to the  Sergeraert~\cite{Sergeraert_1972} and Hamilton~\cite{MR656198} frameworks, we study tameness in a finite regularity context and on possibly finite Banach scales rather than Fr\'echet spaces.  We develop a calculus of such maps, showing closure under various operations such as sums, products, and compositions.  These results turn out to be essential in our subsequent verification of the `nonlinear hypotheses' of our Nash-Moser inverse function theorem. Once this is done, we then formulate and prove our version of the inverse function theorem.

In Section~\ref{section on scales of Banach spaces} we precisely define the nonlinear operator associated to the system~\eqref{The nonlinear equations in the right form}, namely $\Bar{\Psi}$ defined by~\eqref{the nonlinear map equation}, and formulate the Banach scales that serve as its domain and codomain.  We verify, in Lemma~\ref{lem on tameness of domain and codomain}, that the domain Banach scale is tame and that the codomain Banach scale is a tame direct summand of the domain. Then we endeavor to check that $\Bar{\Psi}$ is tamely twice continuously differentiable and has order one derivative loss.  Within a standard Sobolev framework this could be accomplished, more or less, with off-the-shelf tools.  Unfortunately, in our context the free surface functions $\eta$ belong to the anisotropic spaces $\mathcal{H}^{5/2+s}(\R^{n-1})$, and their ubiquity in the nonlinearities of~\eqref{The nonlinear equations in the right form} then requires a much more delicate and customized analysis.   To make this task less arduous, we first identify within $\Bar{\Psi}$ a number of simpler `atomic' nonlinearities that can be analyzed separately.  

For the most part, these atoms are handled via elementary high-low estimates from Appendix~\ref{appendix on tools for tame estimates}, combined with some results from the abstraction of tame structure.  We emphasize that our choice of the enthalpy formulation makes the atoms that comprise the momentum equation all of this form.  By contrast, the nonlinearity arising from the continuity equation in~\eqref{The nonlinear equations in the right form}, namely
\begin{equation}
 (q,u,\eta) \mapsto  \Xi(q,u,\eta) =  \grad\cdot(\sig_{q,\eta}(u- M_{\eta}e_1)),
\end{equation}
is significantly more delicate and requires more sophisticated ideas.  The principal difficulty here is that our functional setting requires that 
\begin{equation}
     \int_0^b \Xi(q,u,\eta)(\cdot,y)\;\m{d}y    \in \dot{H}^{-1}(\R^{n-1}),
\end{equation}
where $\dot{H}^{-1}(\R^{n-1})$ is defined by \eqref{dotHminus1_def}.  At first glance, this inclusion appears to follow readily from the identity 
\begin{equation}
    \int_0^b \Xi(q,u,\eta)(\cdot,y)\;\m{d}y = (\grad_{\|},0) \cdot \int_0^b (\sig_{q,\eta}(u- M_{\eta}e_1)+\varrho e_1)(\cdot,y)\;\m{d}y,
\end{equation}
but a closer inspection reveals that, because of the anisotropic spaces, the integral argument of the divergence on the right does not belong to $L^2(\R^{n-1})$ in general.  To get around this problem we employ the `Taylor expansion trick' of Lemma~\ref{lem on taylor expansion trick} in conjunction with the subtle vector field  decomposition  from Lemma~\ref{lem on fundamental decomposition of continuity equation-like vector fields}.  These rely crucially on various nontrivial algebraic properties of the anisotropic spaces.

We then synthesize the analysis of the atomic nonlinearities with our analysis of tame structure to complete the verification of most of the `nonlinear hypotheses' of the inverse function theorem. This is done in Theorem~\ref{thm on smooth tameness of the nonlinear operator}. 

\textbf{Principal part identification.}   We next pass through some applications of our nonlinear results to set up the linear analysis in a simpler form.  We identify the manifestation of the derivative loss at the linear level as the vector field $v_{q_0,u_0,\eta_0}$ in the linearized continuity equation, as described above. The relevant properties of this derivative loss vector field are enumerated in Lemma~\ref{properties of the principal parts vector field}.  This understanding allows us to perform a `derivative splitting' for the nonlinear operator $\Bar{\Psi}$ of the form $D\Bar{\Psi} = D\Bar{\Psi}_{\m{prin}} + D\Bar{\Psi}_{\m{rem}}$, where the linear operator $D\Bar{\Psi}_{\m{prin}}$ is the `principal part' and $D\Bar{\Psi}_{\m{rem}}$ is the perturbative remainder.  Associated to the principal part operator $D\Bar{\Psi}_{\m{prin}}$ is the following principal part PDE system for linearized unknowns $(q,u,\eta)$ with given data $(g,f,k)$ and wave speed $\gamma \in \R^+$:
\begin{equation}\label{principal_part_intro}
		\begin{cases}
			\grad\cdot(\varrho u)+\grad\cdot(v_{q_0,u_0,\eta_0} (q+\mathfrak{g}\eta))=g&\text{in }\Omega,\\
			-\gam^2\varrho\pd_1 u+\varrho\grad(q+\mathfrak{g}\eta)-\gam\grad\cdot\S^{\varrho}u=f&\text{in }\Omega,\\
			-(\varrho q-\gam\S^{\varrho}u)e_n-\varsigma\Delta_{\|}\eta e_n=k&\text{on }\Sigma,\\
			u\cdot e_n+\pd_1\eta=0&\text{on }\Sigma,\\
			u=0&\text{on }\Sigma_0.
		\end{cases}
\end{equation}
This splitting provides two key benefits.  First, the problem \eqref{principal_part_intro} is as close as possible to the linearization around the trivial background, namely \eqref{linearization_at_zero}, while retaining the entirety of the derivative loss information. Second, the remainder piece of the linearization has no derivative loss and is effectively small so that it can be handled perturbatively.   We refer to Propositions~\ref{prop on properties of the A+P decomposition} and~\ref{prop on the Q+R decomposition of DPhi} for the precise details of this derivative splitting. As a result of this careful splitting, we effectively reduce most of our linear analysis to the study of this family of principal part linear equations, \eqref{principal_part_intro}.

\textbf{Linear analysis.}  Now we discuss the strategy for the verification of the `linear hypotheses' for our Nash-Moser inverse function theorem.  The majority of the work is devoted to studying the principal part system \eqref{principal_part_intro} with $v_{q_0,u_0,\eta_0}$ obtained from a general background triple $(q_0,u_0,\eta_0)$ in an open neighborhood of zero, and in particular showing that it induces the aforementioned isomorphism  $\overset{q_0,u_0,\eta_0}{\X^s} \ni (q,u,\eta) \mapsto (g,f,k) \in \mathbb{Y}^s$ and obeys related tame estimates.  The fundamental difficulty here, as in \eqref{linearization_at_zero}, is the lack of ellipticity in the base Stokes system for $q$ and $u$; indeed, one should view \eqref{principal_part_intro} as an unhappy marriage of elliptic (Lam\'e and mean-curvature type) and hyperbolic (steady transport type) operators whose individual regularity theories are incompatible and appear not to combine without substantial difficulty. To deal with the difficulties inherent in the system \eqref{principal_part_intro}, it is convenient to initially decouple the problems of estimates and existence and only combine them at the last moment.

The key to this strategy is the introduction of a regularizing term in \eqref{principal_part_intro} that makes the elliptic parts interface better with the hyperbolic steady transport part.  Unfortunately, the natural technique of applying a smoothing operator to the steady transport term in the continuity equation of \eqref{principal_part_intro} does not work well in $\Omega$: operators that preserve the good elliptic energy structure seemingly lack good commutators and so fail to give high regularity estimates, while operators that have good commutators do not seem to respect the energy structure.  We are thus led to employ an elliptic regularization by replacing the system  \eqref{principal_part_intro} with
	\begin{equation}\label{regularization_intro}
		\begin{cases}
			\grad\cdot(\varrho u)+\tau\grad\cdot(v_{q_0,u_0,\eta_0}(q+\mathfrak{g}\eta))+N^{-1}L_m(q+\mathfrak{g}\eta)=g&\text{in }\Omega,\\
			-\gam^2\varrho\pd_1u+\varrho\grad(q+\mathfrak{g}\eta)-\gam\grad\cdot\S^\varrho u=f&\text{in }\Omega,\\
			-(\varrho q-\gam\S^\varrho u)e_n-\varsigma\Delta_{\|}\eta e_n=k&\text{on }\Sigma,\\
			u\cdot e_n+\pd_1\eta=N^{-1}(-\Delta_{\|})^{m-1/4}\eta&\text{on }\Sigma,\\
			u=0&\text{on }\Sigma_0,\\
			\pd_n^mq=\cdots=\pd_n^{2m-1}q=0&\text{on }\pd\Omega,
		\end{cases}
	\end{equation}
where the regularization parameter is $N^{-1}$ for $N \in \N^+$,  $L_m$ is the $2m^{\m{th}}$-order linear elliptic differential operator
\begin{equation}
    L_m=(-1)^m\sum_{j=1}^n\pd_j^{2m},
\end{equation}
$(-\Delta_{\|})^{m-1/4}$ is a standard fractional power of the Laplace operator on $\Sigma \simeq \R^{n-1}$, $m \in \N^+$ is a tunable regularity parameter, and $\tau \in [0,1]$ is an operator homotopy parameter ($\tau =1$ corresponds to a regularization of \eqref{principal_part_intro}).  Note that the Neumann boundary conditions for $q$ recorded in the final equation of \eqref{regularization_intro} are new relative to \eqref{principal_part_intro}; their presence is dictated by our introduction of $L_m$, but the specific choice of the Neumann conditions plays a crucial role later.

The benefits of the elliptic regularization \eqref{regularization_intro} are manifold.  First, the domain space for \eqref{regularization_intro} is $\X^s_{m,N}$, as defined by \eqref{the regularized spaces}.  Crucially, this space is independent of the background triple $(q_0,u_0,\eta_0)$ but continuously embeds into the background-dependent space $\overset{q_0,u_0,\eta_0}{\X^s}$.  This makes it an ideal setting for using arguments based on the method of continuity to extend the solvability theory from $\tau =0$ to $\tau=1$, which is the form of the problem we actually care about.  Second, the regularization operators preserve the energy structure of the original problem \eqref{principal_part_intro} while giving a relatively simple regularity gain mechanism.  Third, and perhaps most important, they are compatible with the derivation of $N-$independent estimates at high-regularity, which we will employ to solve~\eqref{principal_part_intro} in $\overset{q_0,u_0,\eta_0}{\X^s}$ via weak compactness arguments.    Note that in doing so we will always need $m \ge 1+s$ so that the artificial Neumann conditions for $q$ in \eqref{regularization_intro} do not pass to the limit.  Of course, regularization does not come without its downsides, and a good amount of work is needed to deal with technical complications it introduces.  We now turn to a somewhat more detailed account of how we use \eqref{regularization_intro} to complete the linear analysis.

We study a priori estimates for the principal part and its regularization in tandem; for the former the goal is to develop the desired tame estimates, but for the latter the goal is high regularity $N-$independent estimates with only a weaker form of tameness with respect to the background tuple $(q_0,u_0,\eta_0)$. The starting point for the principal part is estimates for weak solutions, which are proved in the usual manner modulo some minor technical complications due to some nonlinear expressions involving members of anisotropic Sobolev spaces.  See Proposition~\ref{prop on a priori estimates for weak solutions} for more details.  With the weak estimates in hand, we then turn our attention to high regularity estimates for strong solutions.  The idea here is to exploit the fact that we can apply tangential derivatives to the system \eqref{principal_part_intro} without changing the basic structure of the equations; this allows us to employ the weak solution estimates to get bounds on these tangential derivatives as in Theorem~\ref{thm on tangential derivative analysis}.  After we achieve this control of the tangential derivatives, it remains to recover control of the normal, or vertical derivatives.   Here we implement a version of the classic technique of Matsumura and Nishida~\cite{MR713680}, originally developed for fixed domains with Dirichlet boundary conditions, to reveal a subtle dissipative structure for the normal derivatives of the density.  This, together with some supplementary analysis of steady transport equations (see Proposition~\ref{proposition on a priori estimates for steady transport}), provides an estimate on the high norms of the solution in terms of the data and the tangentially differentiated solution alone.  This is recorded in the first conclusion of Theorem~\ref{third synthesis of normal system results}.  Synthesizing, we then derive the first conclusion of Theorem~\ref{thm on a priori estimates for the principal part and the regularization}, which is the desired a priori estimates for the principal part equations.
	
The strategy for the a priori estimates of the regularized problem \eqref{regularization_intro} is similar at a bird's eye view, although the specifics of the argument are rather distinct.  We again start by proving a priori estimates for weak solutions and then study the equations satisfied by the tangentially differentiated solution: see Proposition~\ref{prop on a priori estimates for regularized weak solutions} and the second conclusion of Theorem~\ref{thm on tangential derivative analysis}, respectively.  The argument used to derive estimates for the normal system, the second conclusion of Theorem~\ref{third synthesis of normal system results}, now substantially deviates from that used for the principal part case because of the need to pass through analysis of the regularized steady transport equations.  It is here that the regularization $L_m$ serves as a liability rather than an asset.   Indeed, deriving estimates with good dependence on $N$ is technically delicate and requires a very careful use of the bilinear form associated to $L_m$ and the precise Neumann boundary conditions imposed on $q$.  The majority of Section~\ref{Section: Analysis of Regularized Steady Transport Equations} is devoted to this regularized transport analysis.  In the end, we obtain the uniform in $N$ estimates as stated in the second conclusion of Theorem~\ref{thm on a priori estimates for the principal part and the regularization}.

The existence theory for the regularized problem \eqref{regularization_intro} for all $\tau \in [0,1]$  is developed by first establishing the existence of weak solutions in Theorem \ref{thm on existence of regularized weak solutions}.  This is proved through the method of continuity, based on the a priori estimates from Proposition  \ref{prop on a priori estimates for regularized weak solutions} and an existence result for the problem with $\tau =0$, which itself requires a further two-parameter regularization and compactness argument to establish.  The weak solutions are then promoted to strong solutions in Corollary \ref{coro on regularity of the regularization} via elliptic regularity arguments.  We emphasize that while this result shows that \eqref{regularization_intro} induces an isomorphism, it does not establish $N-$uniform estimates.   The existence theory for the principal part problem \eqref{principal_part_intro} is then established by way of the regularized existence theory, the $N-$uniform a priori estimates, and another weak limiting argument. This is our Theorem~\ref{thm on existence for the principal part}.

For the conclusion of the linear analysis, we combine the work that allowed us to identify the principal part equations, namely the splitting of the derivative, with the previously discussed principal part analysis on estimates and existence. This is the synthesis of linear analysis result of Theorem~\ref{thm on analysis of the linearization 1}.

We conclude with a couple remarks on our linear strategy.  First, we emphasize that our existence strategy for~\eqref{principal_part_intro} intentionally bypasses direct regularity promotion of weak solutions in favor of high regularity inherited by weak limits of solutions to the regularized equations~\eqref{regularization_intro}.  This is advantageous since the mixed elliptic-hyperbolic nature of~\eqref{principal_part_intro} creates substantial technical difficulties in attempting to implement the standard techniques for regularity promotion (e.g. finite differences or mollification and commutators). The second feature we wish to highlight is that our methods of handling the free surface unknown in the linear existence theory are significantly different and simpler than in the previous work on the incompressible problem \cite{leoni2019traveling, MR4337506, koganemaru2022traveling}. In these works the free surface unknown is constructed in terms of the data alone via a pseudodifferential equation, the symbol for which is inverted after a careful asymptotic analysis.  In our work, we entirely circumvent the use of these delicate pseudodifferential techniques in our existence theory, replacing them with regularizations, a priori estimates, weak limits, and a new spatial characterization of the anisotropic Sobolev spaces $\mathcal{H}^s(\R^d)$.

\textbf{Conclusion and appendices.}  Once we have completed the nonlinear and linear analysis, the hypotheses of our Nash-Moser theorem are verified.  We then invoke the result with a few minor additional arguments in Theorems~\ref{main thm 1 on traveling wave solutions to the free boundary compressible Navier-Stokes equations} and~\ref{main thm 2 on traveling wave solutions to the free boundary compressible Navier-Stokes equations} to complete the proof of well-posedness.

The remainder of the paper consists of four appendices.  These mostly consist of various analytical and PDE tools that are customized or optimized for our particular needs in the paper.  However, some of the results there appear to be new and may be of independent interest for use in other problems.  Appendix  \ref{appendix_std_sob_tools} records tools related to standard Sobolev spaces.  In contrast, Appendix \ref{appendix_nonstandard_fn_spaces} is concerned with properties of the nonstandard spaces we employ in this paper, namely the anisotropic spaces and the adapted spaces.  Appendix \ref{appendix_pde_tools} focuses on PDE tools, with an emphasis on the specified divergence problem in various contexts.  Appendix \ref{appendix_nlin_analysis} contains a selection of nonlinear analysis tools that are essential in our abstract tame analysis.

 \subsection{Conventions of notation} \label{sec_notational_conventions}
 
 We recall that we have included a \hyperref[index]{\indexname} at the end of the paper, which catalogs the numerous operators, function spaces, and other symbols used throughout.

 We write $\N$ for the set of nonnegative integers, $\N^+=\N\setminus\tcb{0}$, and $\R^+=(0,\infty)$. Whenever $\al\lesssim\be$ appears in a result, it means there is a constant $C \in \R^+$ depending only on the parameters mentioned in the formulation of the result such that $\al\le C\be$. To emphasize this dependence, we will sometimes write $\alpha \lesssim_{a,\dotsc,b} \beta$ to indicate the parameters $a,\dotsc,b$.  We will also write $\al\asymp\be$ to mean $\al\lesssim \be$ and $\be\lesssim \al$.  We will use the bracket notation
 \begin{equation}\label{j_bracket_def}      \index{\textbf{Miscellaneous}!00@$\tbr{\cdot}$}
 \tbr{x_1,\dots,x_p}=\tp{1+x_1^2+\cdots+x_p^2}^{1/2}    
 \end{equation}
  for $p\in\N^+$ and $x_1,\dots,x_p\in\R$.  Given $x \in \R^d$, we will often abbreviate $\tbr{x}= \tbr{x_1,\dotsc,x_d}$.    If $U$ and $V$ are some open sets we write $U\Subset V$ to mean that the closure $\Bar{U}$ is compact and $\Bar{U}\subset V$.

We denote the gradient and its tangential counterpart by $\grad=(\pd_1,\dots,\pd_n)$ and $\grad_{\|}=(\pd_1,\dots,\pd_{n-1})$, respectively. The divergence and tangential divergence operators are written $\grad\cdot f=\sum_{j=1}^n\pd_j(f\cdot e_j)$ and $(\grad_{\|},0)\cdot f=\sum_{k=1}^{n-1}\pd_k(f\cdot e_k)$ for appropriate $\R^n$-valued functions $f$. We will also use subscript `$\|$' to indicate that a differential operator depends only on $\pd_1,\dots,\pd_{n-1}$, e.g. $\Delta_{\|}=\sum_{\ell=1}^{n-1}\pd_\ell^2$.

If $\tnorm{\cdot}$ is a norm on a product of normed spaces, $X_1\times\cdots\times X_q$, we will typically write $\tnorm{x_1,\dots, x_q}$ in place of $\tnorm{(x_1,\dots,x_q)}$, where $x_i \in X_i$ for $i \in \{1,\dotsc,q\}$.   Given $\Lambda \subseteq \R$, we say a decreasing collection of Banach spaces $\tcb{X_s}_{s\in\Lambda}$ with norms $\tcb{\tnorm{\cdot}_{X_s}}_{s\in\Lambda}$ is log-convex if for all $s_0,s_1,s \in \Lambda$ such that  $s_0 < s_1$ and $s = (1-\sigma)s_0 + \sigma s_1$ for some $\sigma \in [0,1]$ we have that
	\begin{equation}
		\tnorm{x}_{X_s}\lesssim_{s_0,s_1,\sigma} \tnorm{x}_{s_0}^{1-\sig}\tnorm{x}_{s_1}^{\sig} \text{ for all } x\in X_{s_1}.
	\end{equation}

 Suppose that $k\in\N$, $U\subseteq\R^d$ is open, and $V$ is a finite dimensional real vector space.  For $p \in [1,\infty]$, we write $W^{k,p}(U;V)$ for the usual $L^p$-based Sobolev space of order $k$ with functions valued in $V$,  and we write $H^k(U;V) = W^{k,2}(U;V)$.  We write  $C^{k}_b(U;V)=C^{k}(U;V)\cap W^{k,\infty}(U;V)$, \index{\textbf{Function spaces}!900@$C^k_b$}
 endowed with the obvious norm.  The space $C^{k}_0(U;V)$ is the closure of $C^\infty_c(\R^d;V)$\index{\textbf{Function spaces}!910@$C^k_0$} in $C^{k}_b(U;V)$.  We also write $C^\infty_b(U;V)=\bigcap_{k\in\N} C^k_b(U;V)$\index{\textbf{Function spaces}!920@$C^\infty_b$} and $C^\infty_0(U;V)=\bigcap_{k\in\N}C_0^k(U;V)$\index{\textbf{Function spaces}!930@$C^\infty_0$}. Similarly, we write $H^\infty(U;V)=\bigcap_{k\in\N}H^k(U;V)$\index{\textbf{Function spaces}!890@$H^\infty$} and $W^{\infty,\infty}(U;V)=\bigcap_{k\in\N}W^{k,\infty}(U;V)$\index{\textbf{Function spaces}!895@$W^{\infty,\infty}$}.  Note that $W^{\infty,\infty}(U;V) = C^\infty_b(U;V)$.

 For $S\in\tcb{\Sigma_0,\Sigma}$, we write $\m{Tr}_{S}$ for the trace operator that maps appropriate functions defined on $\Omega$ to functions on $S$. The following closed subspace of $H^1(\Omega;\R^n)$ is frequently used:
 \begin{equation}\label{zero on the left}\index{\textbf{Function spaces}!164@${_0}H^{1}$}
     {_0}H^1(\Omega;\R^n) = \tcb{u\in H^1(\Omega;\R^n)\;:\;\m{Tr}_{\Sigma_0}(u)=0}.
 \end{equation}
 
 The Fourier transform, which we normalize to be unitary on $L^2$, is denoted by $\mathscr{F}$.   We will utilize the homogeneous Sobolev space of order $-1$, which is defined as
 \begin{equation}\label{dotHminus1_def}   \index{\textbf{Function spaces}!165@$\dot{H}^{-1}$}
     \dot{H}^{-1}(\R^d) = \tcb{f\in\mathscr{S}^\ast(\R^d;\R)\;:\;\mathscr{F}[f]\in L^1_{\m{loc}}(\R^d\setminus\tcb{0};\C) \text{ and }\tsb{f}_{\dot{H}^{-1}}<\infty}
 \end{equation}
 for the seminorm $\tsb{f}_{\dot{H}^{-1}}=\tnorm{\mathds{1}_{\R^d\setminus\tcb{0}}|\cdot|^{-1}\mathscr{F}[f]}_{L^2}$.  Finally, we emphasize that we frequently identify $\Sigma$ and $\R^{n-1}$ in the canonical way when working with function spaces defined on $\Sigma$.

	% - space - space - outer - % - space - space - outer - % - space - space - outer - % - space - space - outer - % - space - space - outer - % - space - space - outer - % - space - space - outer - % - space - space - outer - % - space - space - outer - % - space - space - outer - % - space - space - outer - % - space - space - outer - 
	
\section{A variation on the Nash-Moser inverse function theorem}\label{section on NMH}

As we discussed in Section \ref{sec_strategy_label_evil_duck}, our strategy for proving the well-posedness of the system \eqref{The nonlinear equations in the right form} is to employ a new version of the Nash-Moser inverse function theorem.  The goal of this section is to prove this theorem and develop its abstract framework, which will be employed generally in all of our subsequent analysis.  Before we begin, we provide a brief overview of the Nash-Moser strategy and some variants of the theorem that have been developed in the literature.

The abstract setting of the classical inverse function theorem is $C^1$ maps $\Psi : U \to F$, where $E$ and $F$ are Banach spaces and $U \subseteq E$ is an open set, for which $D\Psi(u) \in \mathcal{L}(E;F)$ is invertible for some $u \in U$.  The typical proof leverages the invertibility of $D\Psi(u)$ to prove the local bijectivity of $\Psi$ by way of the contraction mapping principle.  In particular, surjectivity is established via a Picard iteration scheme that crucially employs the map $D\Psi(u)^{-1}$.

The Nash-Moser approach aims to handle the case when $D\Psi(u)$ fails to be invertible but remains `nearly invertible' in the sense that it admits a right inverse $L(u)$, defined on $F$ but only mapping into some larger vector space $E_0\supset E$.  This phenomenon is typically regarded as `derivative loss,' as in practice $E_0$ often consists of functions, and the subspace $E$ consists of functions of higher regularity than those in $E_0$. To fully take advantage of this, the Nash-Moser strategy generalizes the setting of the standard inverse function theorem to maps between one-parameter scales of Banach spaces, say $\{E^s\}_{s \in S}$ and $\{F^s\}_{s \in S}$ for some $S \subseteq \R$, where roughly speaking one should think of the parameter $s$ as measuring the regularity of the elements of the space. The derivative loss is then required to be uniform along the scale in the sense that $D\Psi$ maps from $E^{s+\mu}$ to $F^s$, for some $\mu >0$ measuring the extent of the derivative loss, but its right inverse  $L$ (which is now required to exist in the entirety of an appropriate open set) only maps from $F^{s}$ to $E^s$.  See Figure \ref{fig:NM commutative diagram} for a diagram.  The profound idea of Nash~\cite{nash_1956}, which was expanded upon by Moser~\cite{Moser_1966}, was to establish local surjectivity not via Picard iteration, which is not available due to the derivative loss, but instead with an iteration scheme, based on Newton's method, that employs a family of smoothing operators that increase the regularity along the Banach scales.  The extreme speed of convergence of Newton's method is needed to make the derivative loss and smoothing operators cooperate, and in order to properly implement this idea various precise quantitative estimates are required for the map and the smoothing operators.  

The above approach is extremely flexible and customizable, and has thus become more of a strategy than a specific theorem.  Indeed, many variants of the Nash-Moser theorem have appeared in the literature.  Most of these are really local surjectivity theorems rather than full inverse function theorems, which must establish injectivity as well as continuity or higher regularity of the induced local inverse map.  We are aware of a few exceptions in the literature.  Sergeraert~\cite{Sergeraert_1972} and Hamilton~\cite{MR656198} prove smoothness of the inverse map by working in the more restrictive context of smooth tame maps between Fr\'echet spaces.  Berti, Bolle, and Procesi~\cite{MR2580515} study an implicit function theorem with parameters, and they prove continuous differentiability of the local solution map only with respect to a finite dimensional space of parameters.  Ekeland~\cite{MR2765512} and Ekeland and S\'er\'e~\cite{MR4275475} prove an inverse function theorem and deduce some suboptimal Lipschitz estimates of the inverse map.

Unfortunately, the well-posedness problem for ~\eqref{The nonlinear equations in the right form} is not amenable to any of the above results, so we have endeavored to develop a new Nash-Moser inverse function theorem that works well for our problem, produces solutions in optimal spaces, provides a regularity promotion mechanism, and provides an optimal continuity and even some higher regularity results for the inverse map. Our approach is to synthesize ideas from Hamilton~\cite{MR656198} with the iteration scheme of Baldi and Haus~\cite{MR3711883}, which is an improvement on the results of H\"ormander~\cite{MR602181,MR802486,MR1039355} that employs ideas from Littlewood-Paley theory.  

When compared to the standard inverse function theorem, the Nash-Moser variants have hypotheses that are much more involved, but are similar at a high level. To guarantee a local inverse for a nonlinear operator, one needs to check that: 1) the nonlinear operator is defined on tame scales of Banach spaces; 2) the operator is tamely $C^2$, with a fixed derivative loss; and, 3) the derivative of the operator admits a tame family of inverses in an open neighborhood of a point.  As all of these hypotheses involve the adjective `tame,' we devote Section~\ref{Michigan} to defining tame structures and developing a calculus of tame maps. Once this is done, we use Section~\ref{did you remeber to  start your sourdough starter? no im a fish} to state the precise hypotheses and conclusions of our version of the Nash-Moser inverse function theorem. Afterward, in Sections~\ref{billy shears}, \ref{what would you do if i sang}, and~\ref{section on refinements} the theorem is then proved in several parts.  For the reader looking to take the inverse function theorem as a `black box' and more readily proceed to the analysis of the PDE~\eqref{The nonlinear equations in the right form}, we suggest restricting to the abstract tame structure and the statement of our Nash-Moser theorem, but initially skipping over Sections~\ref{billy shears}, \ref{what would you do if i sang}, and~\ref{section on refinements}.

\subsection{Tame structure abstraction}\label{Michigan}
 
 In this subsection we are first concerned with tame mappings between scales of Banach spaces. Our presentation here primarily inspired by Hamilton~\cite{MR656198} and Baldi and Haus~\cite{MR3711883}.  Our initial concern is scales of Banach spaces, starting with some notation for indexing them.
	
	\begin{defn}[Subsets of $\N$]\label{defn on subsets of N}
		Given $N\in\N\cup\tcb{\infty}$ we define $\j{N}\subseteq\N$ to be the set $\j{N}=\tcb{n\in\N\;:\;n\le N}$\index{\textbf{Miscellaneous}!10@$\j{N}$}.
	\end{defn}
	
    Now we define scales of Banach spaces.

	\begin{defn}[Banach scales]\label{definition of Banach scales}
		Let $N\in\N^+ \cup\tcb{\infty}$ and let $\bf{E} = \{E^s\}_{s\in \j{N}}$ be a collection of Banach spaces over a common field, either $\R$ or $\C$.
		\begin{enumerate}
			\item We say that $\bf{E}$ is a Banach scale if for each $s \in \j{N-1}$ we have the non-expansive inclusion $E^{s+1} \hookrightarrow E^s$, i.e. $\tnorm{\cdot}_{E^s}\le\tnorm{\cdot}_{E^{s+1}}$.
			\item If $\bf{E}$ is a Banach scale, then we define the scale's terminal space to be $E^N = \bigcap_{s \in \j{N}} E^s$ and endow it with the Fr\'{e}chet topology induced by the collection of norms $\{\norm{\cdot}_{E^{s}}\}_{s\in \j{N}}$. Note that if $N<\infty$, then $E^N$ has the standard Banach topology from its norm.
			\item If $\bf{E}$ is a Banach scale, then we write $B_{E^r}(u,\delta) \subseteq E^r$ for the $E^r$-open ball of radius $\delta>0$, centered at $u \in E^r$.
			\item If $\bf{E}$ is a Banach scale and $E^N$ is dense in $E^s$ for each $s \in \j{N}$, then we say that $\bf{E}$ is terminally dense.
		\end{enumerate}
	\end{defn}
	
	We note that finite scales of Banach spaces correspond to the case $N \in \N$ in this definition.   Next we record a quick remark regarding the product of Banach scales.

 	\begin{rmk}[Products of Banach scales]\label{scale_prod_rmk}
	Suppose that $\bf{E}_i = \{E^s_i\}_{s\in \j{N}}$ is a Banach scale for $i \in \{1,\dotsc,n\}$, each over the same field. Then $\bf{F} = \prod_{i=1}^n \bf{E}_i = \{ \prod_{i=1}^n E^s_i \}_{s\in \j{N}}$ is a Banach scale when each $\prod_{i=1}^n E^s_i$ is endowed with the norm
		\begin{equation}
			\norm{u_1,\dotsc,u_n}_{\prod_{i=1}^n E^s_i} = 
				\bp{\sum_{i=1}^n \norm{u_i}_{E^s_i}^2}^{1/2}.
		\end{equation}
	\end{rmk}

	  We now consider tame mappings between Banach scales in our next definition.
	
	\begin{defn}[Tame maps]\label{defn of smooth tame maps}
		Let $\bf{E}=\tcb{E^s}_{s\in\j{N}}$ and $\bf{F}=\tcb{F^s}_{s\in\j{N}}$ be Banach scales over the same field, $\mu,r\in\j{N}$ with $\mu\le r$, $U\subseteq E^r$ be an open set, and $P:U\to F^0$.
		\begin{enumerate}
			\item We say that $P$ satisfies tame estimates of order $\mu$ and base $r$ (with respect to the Banach scales $\bf{E}$ and $\bf{F}$) if for all $\j{N}\ni s\ge r$, there exists a constant $C_s\in\R^+$ such that for all $f\in U\cap E^s$ we have the inclusion $P(f)\in F^{s-\mu}$ as well as the estimate
			\begin{equation}
				\tnorm{P(f)}_{F^{s-\mu}}\le C_s\tbr{\tnorm{f}_{E^{s}}},
			\end{equation}
			where $\tbr{\cdot}$ is defined by \eqref{j_bracket_def}.
			
			\item We say that $P$ is $\mu$-tame with base $r$ if for each $f\in U$, there exists an $E^r$-open set $V\subseteq U$ such that $f \in V$ and the restricted map $P\res V:V \to F^0$ satisfies tame estimates of order $\mu$ and base $r$.
			
			\item We say that $P$ is strongly $\mu$-tame with base $r$ if on every $E^r$-open and bounded subset $V\subseteq U$, the restricted map $P\res V$ satisfies tame estimates of order $\mu$ and base $r$.
			
			\item We say that $P$ is (strongly) $\mu$-tamely $C^0$ with base $r$ if $P$ is (strongly) $\mu$-tame with base $r$ and if for every $\j{N}\ni s\ge r$ we have that $P:U\cap E^s\to F^{s-\mu}$
			is continuous as a map from $E^s$ to $F^{s-\mu}$. The collections of such maps will be denoted by $T^0_{\mu,r}(U,\bf{E};\bf{F})$ and ${_{\m{s}}}T^0_{\mu,r}(U,\bf{E};\bf{F})$, respectively.
			
			\item For $k\in\N^+$ we say that $P$ is (strongly) $\mu$-tamely $C^k$ with base $r$ if for all $\j{N}\ni s\ge r$ the map $P:U\cap E^{s}\to F^{s-\mu}$ is  $C^k$ and for all $j\in\tcb{0,1,\dots,k}$ the $j^{\m{th}}$ derivative map, thought of as mapping $D^jP:U\times\prod_{\ell=1}^jE^r \to F^0$, is (strongly) $\mu$-tame with base $r$ with respect to the Banach scales $\bf{E}^{1+j}$ and $\bf{F}$. In other words, $D^jP\in T^0_{\mu,r}(U\times\prod_{\ell=1}^jE^r,\bf{E}^{1+j};\bf{F})$ or, in the strong case, $D^jP\in {_{\m{s}}}T^0_{\mu,r}(U\times\prod_{\ell=1}^jE^r,\bf{E}^{1+j};\bf{F})$. The collections of such maps will be denoted by $T^k_{\mu,r}(U,\bf{E};\bf{F})$\index{\textbf{Function spaces}!65@$T^k_{\mu,r}(U,\bf{E};\bf{F})$} and ${_{\m{s}}}T^k_{\mu,r}(U,\bf{E};\bf{F})$\index{\textbf{Function spaces}!66@${_{\m{s}}}T^k_{\mu,r}(U,\bf{E};\bf{F})$}, respectively.
			
			\item We denote the collections of (strongly) $\mu$-tamely $C^\infty$ with base $r$ maps  by $T^\infty_{\mu,r}(U,\bf{E};\bf{F})=\bigcap_{k\in\N} T^k_{\mu,r}(U,\bf{E};\bf{F})$ and ${_{\m{s}}}T^\infty_{\mu,r}(U,\bf{E};\bf{F})=\bigcap_{k\in\N} {_{\m{s}}}T^k_{\mu,r}(U,\bf{E};\bf{F})$.
			
			\item    When $U= E^r$ we will often use the abbreviated notation $T^k_{\mu,r}(\bf{E},\bf{F})$ and ${_{\m{s}}}T^k_{\mu,r}(\bf{E};\bf{F})$ in place of $T^k_{\mu,r}(E^r,\bf{E};\bf{F})$ and ${_{\m{s}}}T^k_{\mu,r}(U,\bf{E};\bf{F})$, respectively.
		\end{enumerate}
	\end{defn}

	The following result is a useful characterization of tameness when there is multilinear structure present in the map.
	
	\begin{lem}[Multilinearity and tame maps]\label{lem on multilinear tame maps} Let $k\in\N^+$ and  $r,\mu\in\j{N}$ with $\mu\le r$.  Let $\bf{F}=\tcb{F^s}_{s\in\j{N}}$, $\bf{G}=\tcb{G^s}_{s\in\j{N}}$, and $\bf{E}_j=\tcb{E^s_j}_{s\in\j{N}}$, for $j\in\tcb{1,\dots,k}$, be Banach scales over the same field.  Let $U\subseteq G^r$ be an open set.  Then the following are equivalent for all maps $P:U\times\prod_{j=1}^kE^r_j\to F^0$ such that $P(g,\cdot)$ is $k$-multilinear for all $g\in U$.
		\begin{enumerate}
			\item $P\in T^0_{\mu,r}(U\times\prod_{j=1}^kE^r_j;\bf{G}\times\prod_{j=1}^k\bf{E}_j;\bf{F})$.
			\item The restriction of $P$ to $(U\cap G^s)\times\prod_{j=1}^kE^s_j$ is continuous as a map into $F^{s-\mu}$ for all $\j{N}\ni s\ge r$, and for all $g\in U$ there exists a $G^r$-open subset $V\subseteq U$ such that $g \in V$ and whenever $\j{N}\ni s\ge r$, $f\in V\cap G^s$, and $h_i\in E_i^s$ for $i\in\tcb{1,\dots,k}$ we have the estimate
			\begin{equation}\label{which he ate and donated to the national trust}
				\tnorm{P(f,h_1,\dots,h_k)}_{F^{s-\mu}}\lesssim\tbr{\tnorm{f}_{G^s}}\prod_{i=1}^k\tnorm{h_i}_{E_i^r}+\sum_{j=1}^k\tnorm{h_j}_{E_j^s}\prod_{i\neq j}\tnorm{h_i}_{E_i^r}.
			\end{equation}
		\end{enumerate}
		A similar equivalence holds for maps $P\in{_{\m{s}}}T^0_{\mu,r}(U\times\prod_{j=1}^k E^r_j;\bf{G}\times\prod_{j=1}^k\bf{E}_j;\bf{F})$ if we change the space in the first item to the space of strongly tame maps and we change the second item's  quantification of $V$ to `for all bounded $G^r$-open sets $V \subseteq U$'.
	\end{lem}
	\begin{proof}
		The second item implies the first by noting that if $\tcb{h_i}_{i=1}^k$ lies within a bounded subset of $\prod_{i=1}^kE^r_i$, then we immediately obtain the required tame estimates on $P$ from inequality~\eqref{which he ate and donated to the national trust}.
		
		Now we look to the converse, fixing $g\in U$.  The first item provides an open set $\tilde{V}\subseteq U\times\prod_{j=1}^k E^r_j$ such that $(g,0,\dots,0)\in\tilde{V}$ and if $(f,h_1,\dots,h_k)\in\tilde{V}\cap (G^s\times E^s_1\times\cdots\times E^s_k)$ for $\j{N}\ni s\ge r$, then
		\begin{equation}\label{i hear the clock a tickin on the mantel shelf}
			\tnorm{P(f,h_1,\dots,h_k)}_{F^{s-\mu}}\lesssim\tbr{\tnorm{f}_{G^s},\tnorm{h_1,\dots,h_k}_{E^s_1\times\cdots\times E^s_{k}}}.
		\end{equation}
		Now, since $\tilde{V}$ is open, there exists $\del\in\R^+$ such that $V_{\del}=\Bar{B_{G^r}(0,\del)}\times\prod_{j=1}^k\Bar{B_{E^r_j}(0,\del)}\subseteq\tilde{V}$. Hence, if $f\in\Bar{B_{G^r}(0,\del)}$ and $h_1,\dots,h_k\in\prod_{j=1}^k(E^r_k\setminus\tcb{0})$ are such that $(g,h_1,\dots,h_k)\in G^s\times E^s_1\times\cdots\times E^s_k$, then we have that
		\begin{equation}
			(f,\del h_1/\tnorm{h_1}_{E^r_1},\dots,\del h_k/\tnorm{h_k}_{E^r_k})\in V_{\del}
		\end{equation}
		and so we can invoke estimate~\eqref{i hear the clock a tickin on the mantel shelf} and multiply through by $\del^{-k}\tnorm{h_1}_{E^r_1}\cdot\dots\cdot\tnorm{h_k}_{E^r_k}$ to acquire the desired bound \eqref{which he ate and donated to the national trust}.  A similar argument applies in the case of strongly tame maps.
	\end{proof}
	
	\begin{rmk}\label{remark on purely multlinear maps}
		If $P\in T^0_{\mu,r}(\prod_{j=1}^k\bf{E}_j;\bf{F})$ is a $k$-multilinear mapping, then a simple modification of the proof of Lemma~\ref{lem on multilinear tame maps} shows that $P$ is actually strongly tame.  Moreover, by multilinearity, we have that $P$ is automatically a smooth function whose derivatives are also multilinear maps. This combines with the previous fact to show that actually $P\in {_{\m{s}}}T^\infty_{\mu,r}(\prod_{j=1}^k\bf{E}_j;\bf{F})$.
	\end{rmk}
	
	As a consequence of the previous lemma, we have structured estimates of the derivatives of tame maps.
	
	\begin{coro}[Derivative estimates on tame maps]\label{lem on derivative estimates on tame maps}
		Let $\bf{E}=\tcb{E^s}_{s\in\j{N}}$ and $\bf{F}=\tcb{F^s}_{s\in\j{N}}$ be Banach scales over a common field. Let $r,\mu\in\j{N}$ with $\mu\le r$, $U\subseteq E^r$ be an open set, and $P\in T^k_{\mu,r}(U,\bf{E};\bf{F})$.  Then for each $f_0\in U$, there exists an open set $f_0\in V\subseteq U$ with the property that for every $\j{N}\ni s\ge r$ there exists a constant $C_{s,V} \in \R^+$, depending on $s$ and $V$, such that for all $f\in V\cap E^s$ and all $g_1,\dots,g_k\in E^s$ we have the estimate
		\begin{equation}
			\tnorm{D^kP(f)[g_1,\dots,g_k]}_{F^{s-\mu}}\le C_{s,V} \sum_{\ell=1}^k\tnorm{g_\ell}_{E^s}\prod_{j\neq\ell}\tnorm{g_j}_{E^r}
			+ C_{s,V} \tbr{\tnorm{f}_{E^s}}\prod_{j=1}^k\tnorm{g_j}_{E^r}.
		\end{equation}
		Moreover, if $P\in{_{\m{s}}}T^k_{\mu,r}(U,\bf{E};\bf{F})$, then a similar assertion holds for every bounded and open subset $V \subseteq U$.
	\end{coro}
	\begin{proof}
		This is a direct application of Lemma~\ref{lem on multilinear tame maps}.
	\end{proof}
	
	Our next result studies the interaction of tame maps via composition.
	
	\begin{lem}[Composition of tame maps]\label{lem on composition of tame maps}
		Let  $\bf{E}=\tcb{E^s}_{s\in\j{N}}$, $\bf{F}=\tcb{F^s}_{s\in\j{N}}$, and $\bf{G}=\tcb{G^s}_{s\in\j{N}}$ be a triple of Banach scales over a common field ($\R$ or $\C$), and let $k\in\N$ and $\mu,\mu', r,r'\in\j{N}$ be such that $\mu\le r$, $\mu'\le r'$, $r'+\mu\in\j{N}$.  Suppose that $U\subseteq E^r$, $U'\subseteq F^{r'}$ are open sets,  $P\in T^k_{\mu,r}(U,\bf{E};\bf{F})$, and $Q\in T^k_{\mu',r'}(U',\bf{F};\bf{G})$. Set $V=U\cap E^{\max\tcb{r,r'+\mu}}$.   If $P(V)\subseteq U'$, then $Q\circ P\in T^k_{\mu+\mu',\max\tcb{r,\mu+r'}}(V,\bf{E};\bf{G})$. A similar assertion holds for the composition of strongly tame maps.
	\end{lem}
	\begin{proof}
		We proceed by induction on $k \in \N$.  Consider first the case that $k=0$. Fix $f_0\in V$. Since $P(f_0)\in U'$, we can appeal to the tameness of $Q$ to obtain an open set $P(f_0)\in W'\subseteq U'$ such that $Q$ satisfies a tame estimate of order $\mu'$ and base $r'$ in $W'$. The map $P$ is continuous and hence $P^{-1}(W')\subseteq V$ is open and contains $f_0$. In light of the tameness of $P$, there exists an open set $f_0\in W\subseteq V$ in which $P$ satisfies a tame estimate of order $\mu$ and base $r$. Now, the open set $W\cap P^{-1}(W')\subseteq V$ contains $f_0$ and is such that whenever $\j{N}\ni s\ge\max\tcb{r,r'+\mu}$ and $f \in E^s \cap W \cap P^{-1}(W')$ we may estimate
		\begin{equation}
			\tnorm{Q\circ P(f)}_{G^{s-(\mu+\mu')}}\lesssim\tbr{\tnorm{P(f)}_{F^{s-\mu}}}\lesssim\tbr{\tnorm{f}_{E^{s}}}.
		\end{equation}
		Hence, $Q\circ P$ is indeed $(\mu+\mu')$-tamely $C^0$ with base $\max\tcb{r,r'+\mu}$.
		
		Now suppose that for $1 \le k \in \N$ the result has been proved at the level $k-1$, and $P$ and $Q$ satisfy the hypotheses at level $k$.  Applying the induction hypothesis handles the tameness of all derivatives up to order $k-1$, so it suffices to show that $D^k(Q\circ P)$ is $(\mu+\mu')$-tame with base $\max\tcb{r,r'+\mu}$. Fix $f_0\in V$. For $\ell\in\tcb{0,1,\dots,k}$ we have that $D^\ell Q$ is $\mu$-tame with base $r$ and hence there exists an open set $P(f_0)\in W'_\ell\subseteq U'$ such that in $W'_\ell$ we have that $D^\ell Q$ satisfies an $\ell$-multilinear tame estimate of order $\mu'$ and base $r'$ (see Lemma~\ref{lem on derivative estimates on tame maps}). The map $P$ is continuous and hence $\bigcap_{\ell=1}^kP^{-1}(W'_\ell)$ is an open subset of $V$ containing $f_0$. By appealing to the tameness of $P$ and Lemma~\ref{lem on derivative estimates on tame maps} again, for $\ell\in\tcb{0,1,\dots,k}$ there exists an open set $f_0\in W_\ell\subseteq V$ such that in $W_\ell$ we have that $D^\ell P$ satisfies an $\ell$-multilinear tame estimate of order $\mu$ and base $r$.
		
		For each $R>0$, the set $O_{f_0,R}=\bigcap_{\ell=1}^k(W_\ell\cap P^{-1}(W'_\ell))\times B_{(E^{\max\tcb{r,r'+\mu}})^k}(0,R)$ is an open $V\times\prod_{p=1}^kE^{\max\tcb{r,r'+\mu}}$ containing $(f_0,0,\dots,0)$. For $\j{N}\ni s\ge\max\tcb{r,r'+\mu}$ and $(f,g_1,\dots,g_k)\in O_{f_0,R}\cap E^s$,  we use the Fa\`a di Bruno theorem (see, for instance, Section 2.4A in Abraham, Marsden, and Ratiu~\cite{AbMaRa_1988}) to compute
		\begin{multline}\label{waterloo}
			D^k(Q\circ P)(f)[g_1,\dots,g_k]\\=\sum_{\sig\in S_k}\sum_{\ell=1}^k\sum_{j_1+\cdots +j_\ell=k}c_{\sig,\ell,k,j_1,\dots,j_\ell}D^\ell Q\circ P(f)\tcb{D^{j_1}P(f),\dots,D^{j_\ell}P(f)}[g_{\sig(1)},\dots,g_{\sig(k)}],
		\end{multline}
		where $S_k$ denotes the permutation group on $\tcb{1,\dots,k}$, and $c_\ast$ denotes some combinatorial constant.  We will estimate the norm in $G^{s-(\mu+\mu')}$ of each term in the sum above. Fix $\ell\in\tcb{1,\dots,k}$ and $j_1,\dots,j_\ell\in\tcb{1,\dots,k}$ such that $j_1+\cdots+j_\ell=k$. Thanks to Lemma~\ref{lem on derivative estimates on tame maps} and the fact that $\max_{1 \le i \le k}\tnorm{g_i}_{E^r}\le R$, we are free to estimate
		\begin{equation}
			\begin{cases}
				\tnorm{D^{j_1}P(f)[g_1,\dots,g_{j_1}]}_{F^{s-\mu}}\\
				\qquad\vdots\\
				\tnorm{D^{j_\ell}P(f)[g_{j_{\ell-k+1}},\dots,g_k]}_{F^{s-\mu}}
			\end{cases}
			\lesssim_R 
			\tbr{\tnorm{f,g_1,\dots,g_k}_{E^s\times\cdots\times E^{s}}}.
		\end{equation}
		By the same argument,
		\begin{equation}
			\begin{cases}
				\tnorm{D^{j_1}P(f)[g_1,\dots,g_{j_1}]}_{F^{r'}}\\
				\qquad\vdots\\
				\tnorm{D^{j_\ell}P(f)[g_{j_{\ell-k+1}},\dots,g_k]}_{F^{r'}}
			\end{cases}
			\lesssim_R\tbr{\tnorm{f,g_1,\dots,g_k}_{E^{\max\tcb{r'+\mu}}\times\cdots\times E^{\tcb{r'+\mu}}}}\lesssim_R1.
		\end{equation}
		Thus, applying the estimate from Lemma~\ref{lem on derivative estimates on tame maps} once more yields
		\begin{multline}\label{sunset}
			\tnorm{D^\ell Q\circ P(f)\tcb{D^{j_1}P(f)[g_1,\dots,g_{j_1}],\dots,D^{j_\ell}P(f)[g_{j_\ell-k+1},\dots,g_{k}]}}_{G^{s-(\mu+\mu')}}\\\lesssim_R\tbr{\tnorm{f,g_1,\dots,g_k}_{E^s\times\cdots\times E^{s}}}.
		\end{multline}
		Then \eqref{waterloo} and \eqref{sunset} combine to give us the estimate
		\begin{equation}
			\tnorm{D^k(Q\circ P)(f)[g_1,\dots g_k]}_{G^{s-(\mu+\mu')}}\lesssim_R\tbr{\tnorm{f,g_1,\dots,g_k}_{E^s\times\cdots\times E^s}},
		\end{equation}
		which is the desired tame bound at level $k$.  The result thus holds for all $k \in \N$ by induction.
	\end{proof}
	
	A convenient application of the previous result is that tame products of tame maps result in a tame map. More precisely, we have the following. 

	\begin{coro}[Tame products of tame $C^k$ maps]\label{lem on product of smooth tame maps} Let  $\ell,k\in\N$ with $\ell\ge1$.  For $j\in\tcb{1,\dots,\ell}$ let  $\bf{E}_j=\tcb{E^s_j}_{s\in\j{N}}$ and $\bf{F}_j=\tcb{E^s_j}_{s\in\j{N}}$ be Banach scales over a common field, and let $\bf{G}=\tcb{G^s}_{s\in\j{N}}$ be a Banach scale over the same field.  Let $r_1,\dots,r_\ell\in\j{N}$, $\mu_j\in\j{r_j}$ for $j\in\tcb{1,\dots,\ell}$, $r'\in\j{N}$, and $\mu'\in\j{r'}$.  Finally, suppose that $U_j\subseteq E^{r_j}_j$ is open, $P_j\in T^k_{\mu_j,r_j}(U_j,\bf{E}_j;\bf{F}_j)$, and $B\in T^0_{\mu',r'}(\prod_{j=1}^\ell\bf{F}_j;\bf{G})$ is $k$-multilinear. If we set
		\begin{equation}
			\mu=\mu'+\max\tcb{\mu_1,\dots,\mu_\ell}
			\text{ and }
			r=\max\tcb{r_1,\dots,r_\ell,\mu_1+r',\dots,\mu_\ell+r'}
		\end{equation}
		and assume that $r\in\j{N}$, then the product map satisfies the inclusion
		\begin{equation}
			B(P_1,\dots,P_\ell)\in T^k_{\mu,r}\bp{\prod_{j=1}^\ell U_j\cap E^{r}_j,\prod_{j=1}^\ell\bf{E}_j;\bf{G}}.
		\end{equation}
		A similar assertion holds for products of strongly tame maps.
	\end{coro}
	\begin{proof} We first note that by Remark~\ref{remark on purely multlinear maps}, $B$ in fact belongs to ${_{\m{s}}}T^\infty_{\mu',r'}(\prod_{j=1}^\ell\bf{F}_j;\bf{G})$. The inclusion
		\begin{equation}
			(P_1,\dots,P_\ell)\in T^k_{\max\tcb{\mu_1,\dots,\mu_\ell},\max\tcb{r_1,\dots,r_\ell}}\bp{\prod_{j=1}^\ell U_j\cap E^{\max\tcb{r_1,\dots,r_\ell}}_{j},\prod_{j=1}^\ell\bf{E}_j;\prod_{j=1}^\ell\bf{F}_j}
		\end{equation}
		is clear, so the conclusion then follows from Lemma~\ref{lem on composition of tame maps}.
	\end{proof}
	
	We can also combine families of tame maps by integrating over parameters. Although more general results hold, we will need only the following simple realization of this fact.
	
	\begin{lem}[Integrals of one-parameter families of tame $C^k$ maps] \label{lemma on addition of smooth tame maps}
		Let $k\in\N\cup\tcb{\infty}$, $\bf{E}=\tcb{E^s}_{s\in\j{N}}$ and $\bf{F}=\tcb{F^s}_{s\in\j{N}}$ be a pair of Banach scales over the same field, and let $\mu,r\in\j{N}$ satisfy $\mu\le r$.  Suppose that $U\subseteq E^r$ is an open set, and for all $t\in[0,1]$ let $P_t\in T^k_{\mu,r}(U,\bf{E};\bf{F})$.  Suppose additionally that the defining inequalities for these tame estimates (as in the first item of Definition~\ref{defn of smooth tame maps} for the map and its derivatives) are satisfied uniformly in $t\in[0,1]$ and that for every $j\in\tcb{0,\dots,k}$, $\j{N}\ni s\ge r$, $f\in U\cap E^s$, and $f_1,\dots,f_j\in E^s$ the map
		\begin{equation}\label{integrability hypotheses}
			[0,1]\ni t\mapsto D^jP_t(f)[f_1,\dots,f_j]\in F^{s-\mu}
		\end{equation}
		is continuous. Then the integral map $\int_0^1P_t\;\m{d}t: U \to F^0$  given by $f\mapsto\int_0^1P_t(f)\;\m{d}t$ is well-defined and satisfies $\int_0^1 P_{t}\;\m{d}t\in T^k_{\mu,r}(U,\bf{E};\bf{F})$. A similar assertion holds for integrals of strongly tame maps.
	\end{lem}
	\begin{proof}
		The proof essentially amounts to checking definitions, so we will only sketch the argument.  Hypothesis \eqref{integrability hypotheses} ensures that the map $\int_0^1P_t\;\m{d}t$ is well-defined, while the assumed uniform tame estimates ensure that the integrands are uniformly bounded with respect to $t\in[0,1]$.  The map is $C^k$ as a map from $U\cap E^s$ to  $F^{s-\mu}$ thanks to \eqref{integrability hypotheses} and standard dominated convergence arguments.  The fact that $\int_0^1 P_t\;\m{d}t$ obeys the defining inequalities to be $\mu$-tamely $C^k$ with base $r$ follows from uniformity in $t$ and the fact that the norm of the integral is at most the integral of the norm.
	\end{proof}
	
	The latter half of this subsection is concerned with various specialized classes of Banach scales. The nicest classes of these are introduced in the subsequent definition, which closely follows Baldi and Haus \cite{MR3711883}.

	\begin{defn}[Smoothable and LP-smoothable Banach scales]\label{defn of smoothable and LP-smoothable Banach scales}
		Let $\bf{E} = \{E^s\}_{s\in \j{N}}$ be a Banach scale.
		\begin{enumerate}
			\item We say that $\bf{E}$ is smoothable if for each $j \in \N^+$ there exists a linear map $S_j : E^0 \to E^N$\index{\textbf{Linear maps}!100@$S_j$} satisfying the following smoothing conditions for every $u \in E^0$:
			\begin{equation}\label{smoothing_1}
				\norm{S_j u}_{E^s} \lesssim \norm{u}_{E^s} \text{ for all }s \in \j{N},
			\end{equation}
			\begin{equation}\label{smoothing_2}
				\norm{S_j u}_{E^t} \lesssim 2^{j(t-s)} \norm{S_j u}_{E^s}  \text{ for all } s,t \in \j{N} \text{ with } s < t,  
			\end{equation}
			\begin{equation}\label{smoothing_3}
				\norm{(I-S_j)u}_{E^s} \lesssim 2^{-j(t-s)} \norm{(I-S_j) u}_{E^t} \text{ for all } s,t \in \j{N} \text{ with } s < t,
			\end{equation}
			\begin{equation}\label{smoothing_4}
				\norm{(S_{j+1}-S_j)u}_{E^t} \lesssim 2^{j(t-s)}    \norm{(S_{j+1}-S_j)u}_{E^s} \text{ for all }s,t \in \j{N},
			\end{equation}
			where the implicit constants are independent of $j$ and are increasing in $s$ and $t$.

			\item We say that $\bf{E}$ is LP-smoothable if $\bf{E}$ is smoothable and the smoothing operators satisfy the following Littlewood-Paley condition:   for $s \in \j{N}$ and $u \in E^s$ we have that
			\begin{equation}\label{smoothing_LP_1}
				\lim_{j\to \infty} \tnorm{(I-S_j)u}_{E^s} =0, 
			\end{equation}
			and there exists a constant $A >0$, possibly depending on $s$, such that 
			\begin{equation}\label{smoothing_LP_2}
				A^{-1} \norm{u}_{E^s} \le \bp{ \sum_{j=0}^\infty \norm{\Updelta_j u}_{E^s}^2 }^{1/2} \le A \norm{u}_{E^s},
			\end{equation}
			where the operators $\{\Updelta_j\}_{j=0}^\infty$ are defined via $\Updelta_j=S_{j+1}-S_j$\index{\textbf{Linear maps}!101@$\Updelta_j$} with the convention that $S_0=0$.
		\end{enumerate}
	\end{defn}
	
	We now give some examples of LP-smoothable Banach scales.  The first is trivial, but instructive.  
	
	\begin{exa}[A fixed Banach space]\label{example of a fixed Banach space}
		Suppose that $X$ is a Banach space and $N \in \N \cup \{\infty\}$. Then the scale $\tcb{X}_{s \in \j{N}}$, generated by $X$ alone, is LP-smoothable in the sense of Definition~\ref{defn of smoothable and LP-smoothable Banach scales}, provided we take $S_j=I$ for all $j\in\N^+$.
	\end{exa}
	
	We also have the less trivial example of Sobolev spaces on all of Euclidean space.
	
	\begin{exa}[Sobolev spaces on $\R^d$]\label{example of Euclidean Soblev spaces}
		Let $V$ be a finite dimensional real vector space and $N \in \N \cup \{\infty\}$.  The real Banach scale of Sobolev spaces $\tcb{H^s(\R^d;V)}_{s\in \j{N}}$ is LP-smoothable for the smoothing operators $\tcb{S_j}_{j=0}^\infty$ given by $S_j=\mathds{1}_{B(0,2^j)}(\grad/2\pi\ii)$.
	\end{exa}
	
	Unfortunately, the (LP-)smoothable Banach scales are insufficiently general for our purposes. As such, we introduce a broader class that captures Banach scales which are essentially closed and complemented subspaces of LP-smoothable Banach scales.  This is analogous to Hamilton's notion of a tame Fr\'echet space, which is given in Definition 1.3.2 of~\cite{MR656198}.

	\begin{defn}[Tame direct summands and tame Banach scales]\label{defn of tame direct summands and tame Banach scales}
		Let $\bf{E} = \{E^s\}_{s\in \j{N}}$ be a Banach scale.
		\begin{enumerate}
			\item Suppose that  $\bf{F} = \{F^s\}_{s\in \j{N}}$ is a Banach scale over the same field as $\bf{E}$. We say that $\bf{E}$ is a tame direct summand of $\bf{F}$ if there exist bounded linear maps $\lambda : E^0 \to F^0$ and $\rho : F^0 \to E^0$ such that the following hold.
			\begin{enumerate}
				\item $\rho \lambda u = u$ for all $u \in E^0$.
				\item For $s \in \j{N}$ we have that $\lambda(E^s) \subseteq F^s$ and $\rho(F^s) \subseteq E^s$, and the induced maps $\lambda : E^s \to F^s$ and $\rho : F^s \to E^s$ are bounded and linear.
			\end{enumerate}
			In this case, we say $\lambda$ is the lifting map and $\rho$ is the restriction map.
			\item We say that $\bf{E}$ is tame if there exists an LP-smoothable Banach scale $\bf{F}$ such that $\bf{E}$ is a tame direct summand of $\bf{F}$.
		\end{enumerate}
	\end{defn}
	
	\begin{rmk}\label{Massachusetts}
		We note that if  $\bf{E}$ is a tame direct summand of $\bf{F}$, then basic functional analysis shows that $\lambda(E^s) \subseteq F^s$ is a closed and complemented subspace of $F^s$, and $\pi = \lambda \circ \rho : F^s \to F^s$ is bounded and linear projection onto $\lambda(E^s)$.  Consequently, for each $s \in \j{N}$ there exists a closed subspace $G^s \subseteq F^s$ such that $F^s = \lambda(E^s) \oplus G^s$.  This is the motivation for calling $\bf{E}$ a direct summand of $\bf{F}$. Moreover, if $\bf{E}$ is a direct summand of $\bf{F}$, then for each $s \in \N$ we have the equivalence $\norm{u}_{E^s} \asymp \norm{\lambda u}_{F^s}$ for $u\in E^s$.  
	\end{rmk}
	
	We have the following example of a tame Banach scale.
	
	\begin{exa}[Sobolev spaces on domains]\label{example on Sobolev spaces on domains}
		Let $U\subset\R^d$ be an open set that is a Stein extension domain in the sense of Definition~\ref{defn Stein-extension operator}, and let $\mathfrak{E}_U$ denote the associated extension operator.  Let $V$ be a finite dimensional real vector space and $N \in \N \cup \{\infty\}$.  The real Banach scale of Sobolev spaces $\tcb{H^s(U;V)}_{s\in \j{N}}$ is tame in the sense of Definition~\ref{defn of tame direct summands and tame Banach scales}. Indeed, we realize this scale as a tame direct summand of $\tcb{H^s(\R^d;V)}_{s \in \j{N} }$, which is LP-smoothable by Example~\ref{example of Euclidean Soblev spaces}, with the lifting operator   $\lambda=\mathfrak{E}_U$ and restriction operator given by standard restriction, $\rho=\mathfrak{R}_U$.
	\end{exa}
	
	We now record a result on products of (LP-)smoothable and tame Banach scales, the proof of which is straightforward and thus omitted.
	
	\begin{lem}[Products of Banach scales]\label{lemma on products of types of Banach scales}
		The product of a finite family of (smoothable, LP-smoothable, tame) Banach scales over a common field is again a (smoothable, LP-smoothable, tame) Banach scale.
	\end{lem}
	
	Consider now the following useful inequalities about smoothing in tame Banach scales.
	
	\begin{lem}[Smoothing in tame Banach scales]\label{lem on useful smoothing inequalities in tame Banach scales}
		Let $\bf{E}=\tcb{E^s}_{s\in\j{N}}$ be a tame Banach scale in the sense of Definition~\ref{defn of tame direct summands and tame Banach scales}.  There exist a sequence of smoothing operators $\tcb{T_j}_{j=0}^\infty\subset\mathcal{L}(E^0;E^N)$ and, for $s\in\j{N}$, sequences of seminorms $\tcb{\bf{m}_j^s}_{j=0}^\infty,\tcb{\bf{n}_j^s}_{j=0}^\infty$ on $E^s$ such that the following hold.
		\begin{enumerate}
			\item For all $g \in E^s$ we have the estimates
			\begin{equation}
				\tnorm{(T_{j+1}-T_j)g}_{E^s}\lesssim\bf{m}^s_j(g)
				\text{ and }
				\tnorm{(I-T_j)g}_{E^s}\lesssim\bf{n}^s_j(g),
			\end{equation}
			where the implicit constants depend only on $s$.
			
			\item For $s,t\in\j{N}$ we have that
			\begin{equation}
				\begin{cases}
					\bf{m}_j^s\asymp 2^{j(s-t)}\bf{m}_j^t&\text{for all }s,t,\\
					\bf{n}_j^s\lesssim 2^{j(s-t)}\bf{n}_j^t&\text{for all }s\le t,
				\end{cases}
			\end{equation}
			with the implicit constants depending only on $s$ and $t$.
			\item We have the equivalence
			\begin{equation}
				\tnorm{\cdot}_{E^s}^2\asymp\sum_{j=0}^\infty(\bf{m}_j^s)^2
			\end{equation}
			with implicit constants depending only on $s$.
			\item For every $g \in E^s$ we have that $\bf{n}_j^s(g)\lesssim\tnorm{g}_{E^s}$ (with implicit constants depending only on $s$), and $\bf{n}_j^s(g) \to 0$ as $j \to \infty$.
			\item $\bf{E}$ is terminally dense in the sense of the fourth item of Definition~\ref{definition of Banach scales}.
		\end{enumerate}
	\end{lem}
	\begin{proof}
		Let $\bf{F}$ be an LP-smoothable Banach scale witnessing the definition of tameness for $\bf{E}$ with lifting $\lambda\in\mathcal{L}(E^0;F^0)$, restriction $\rho\in\mathcal{L}(F^0;E^0)$, and smoothing operators $\tcb{S_j}_{j=0}^\infty\subseteq\mathcal{L}(F^0;F^N)$. We define $T_j=\rho\circ S_j\circ\lambda$ and for $s\in\j{N}$ and $g\in E^s$ set 
		\begin{equation}
			\bf{m}_j^s(g)=\tnorm{(S_{j+1}-S_j)\circ \lambda g}_{F^s} 
			\text{ and }
			\bf{n}^s_j(g)=\tnorm{(I-S_j)\circ\lambda g}_{F^s}.
		\end{equation}
		Then the first item follows from the boundedness of $\rho$. The second, third, and fourth items are immediate consequences of Definition~\ref{defn of smoothable and LP-smoothable Banach scales} and Remark~\ref{Massachusetts}. The fifth assertions follows from the first and the fourth.
	\end{proof}
	
	Next we give an interpolation result.

	\begin{lem}[Log-convexity in tame Banach scales]\label{lem on log-convexity in tame Banach scales}
		Let $\bf{E}=\tcb{E^s}_{s\in\N}$ be a tame Banach scale in the sense of Definition~\ref{defn of tame direct summands and tame Banach scales} over either $\R$ or $\C$. For $r,s,t \in \j{N}$ with $r < s < t$ we have that 
		\begin{equation}\label{smoothing_lemma_02}
			\norm{u}_{E^s} \lesssim \norm{u}_{E^r}^{\frac{t-s}{t-r}} \norm{u}_{E^t}^{\frac{s-r}{t-r}}
		\end{equation}
		for all $u \in E^t$, where the implicit constant is increasing $r,t$.
	\end{lem}
	\begin{proof}
		The bound is trivial if $u=0$, so assume $u \neq 0$.  Let $\bf{F}$ be an LP-smoothable Banach scale witnessing the definition of tameness for $\bf{E}$, with lifting, restriction, and smoothing operators $\lambda$, $\rho$, and $\tcb{S_j}_{j=0}^\infty$, respectively. By Remark~\ref{Massachusetts},  we have that $\tnorm{u}_{E^p}\asymp\tnorm{\lambda u}_{F^p}$ for $p\in\tcb{s,r,t}$. Now we invoke the properties of the smoothing operators from Definition~\ref{defn of smoothable and LP-smoothable Banach scales} to see that for any $j\in\N$
		\begin{equation}
			\tnorm{\lambda u}_{F^s}\le\tnorm{(I-S_j)\lambda u}_{F^s}+\tnorm{S_j\lambda u}_{F^s}\lesssim 2^{j(s-t)}\tnorm{\lambda u}_{F^t}+2^{j(s-r)}\tnorm{\lambda u}_{F^r}
		\end{equation}
		and hence $\tnorm{ u}_{E^s}\lesssim 2^{j(s-t)}\tnorm{ u}_{E^t}+2^{j(s-r)}\tnorm{u}_{E^r}$. Now, since $1\le \tnorm{u}_{E^t}/\tnorm{u}_{E^r}$, we can choose  $j = \lfloor  \log(\tnorm{u}_{E^t}/\tnorm{u}_{E^r}) /  ((t-r) \log 2)  \rfloor \in \N$ to obtain the desired inequality.
	\end{proof}
	
	\subsection{Mapping hypotheses and statement of the inverse function theorem}\label{did you remeber to  start your sourdough starter? no im a fish}
	
	We now introduce  a lengthy definition that records a number of conditions that must be placed on the nonlinear map in our version of the Nash-Moser inverse function theorem.  

	\begin{defn}[Mapping hypotheses]\label{defn of the mapping hypotheses}
		We say that a triple $(\bf{E},\bf{F},\Psi)$ satisfies the RI mapping hypotheses with parameters $(\mu,r,R) \in \j{N}^3$ if $\bf{E} = \{E^s\}_{s\in \j{N}}$ and $\bf{F} = \{F^s\}_{s\in \j{N}}$ are Banach scales over a common field, $1 \le \mu \le r < R < \infty$, $R+\mu \in \j{N}$, and there exists $0 < \delta_r \in \R$ such that  $\Psi : B_{E^r}(0,\delta_r) \to F^{r-\mu}$ is a map satisfying the following.
		\begin{enumerate}
			\item $\Psi(0)=0$.
			
			\item $\mu$-tamely $C^2$: For every $r-\mu \le s \in \j{N-\mu}$ we have that $\Psi : B_{E^r}(0,\delta_r) \cap E^{s+\mu} \to F^s$ is $C^2$, and for every $u_0 \in B_{E^r}(0,\delta_r) \cap E^{s+\mu}$ we have the tame estimate 
			\begin{equation}\label{C2 tame estimates}
				\norm{D^2 \Psi(u_0)[v,w] }_{F^s} \le C_1(s)\tp{\norm{v}_{E^{s+\mu}} \norm{w}_{E^r} 
					+ \norm{v}_{E^r} \norm{w}_{E^{s+\mu}}  
					+  \br{\norm{u_0}_{E^{s+\mu}} } \norm{v}_{E^r} \norm{w}_{E^r} }.
			\end{equation}
			Here the constant $C_1(s)$ is increasing in $s$. In other words, we have the inclusion $\Psi \in{_{\m{s}}}T^2_{\mu,r}(B_{E^r}(0,\del_r),\bf{E};\bf{F})$ according to the notation from Definition~\ref{defn of smooth tame maps}.
			
			\item Derivative inversion: There exists $\delta_R \in \R$ satisfying $0 < \delta_R \le \delta_r$ such that for every $u_0 \in B_{E^r}(0,\delta_R) \cap E^{N}$ there exists a bounded linear operator $L(u_0) : F^r \to E^r$ satisfying the following three conditions.
			\begin{enumerate}
				\item For every $s \in \N \cap [r,R]$  we have that the restriction of $L(u_0)$ to $F^s$ defines a bounded linear operator with values in $E^s$, i.e. $L(u_0) \in \mathcal{L}(F^s; E^s)$.
				\item $D\Psi(u_0) L(u_0) f = f$ for every $f \in F^r$.

				\item We have the tame estimate 
				\begin{equation}\label{tame estimates on the right inverse}
					\norm{L(u_0) f}_{E^s} \le C_2(s) \tp{\norm{f}_{F^s} + \br{\norm{u_0}_{E^{s+\mu}}} \norm{f}_{F^r}}
				\end{equation}
				for every $f \in F^s$ and $r \le s \le R$, where again the constant $C_2(s)$ is increasing in $s$.
			\end{enumerate}
		\end{enumerate}
		Here the use of the prefix RI- is meant to indicate that the maps $L(u_0)$ are only required to be right inverses.  We say that the triple $(\bf{E},\bf{F},\Psi)$ satisfies the LRI mapping hypotheses if condition $(b)$ in the third item is augmented by the left-inverse condition  $(b'):$  
		\begin{equation}
			L(u_0) D\Psi(u_0) v = v \text{ for every }v \in E^{r+\mu}.
		\end{equation}
  See Figure~\ref{fig:NM commutative diagram} for a diagrammatic depiction of how  $L(u_0)$ and $D\Psi(u_0)$ interact with $\bf{E}$ and $\bf{F}$.
	\end{defn}
	
 \begin{figure}[h]
    \centering
    \usetikzlibrary {calc,positioning,shapes.misc}
\begin{tikzpicture}[space/.style={align=center,inner sep = 0.2 cm,anchor = center}, map/.style={fill=white,align=center,inner sep = 0.2 cm,anchor = center}]

    \node at (0,0) [space] (Fs1){
    $F^{s+\mu}$
    };
    \node at (0,3) [space] (Es1){
    $E^{s+\mu}$
    };
    \draw[-stealth] (Fs1) -- (Es1);

    \node at (0,1.5) [map] (Fs1Es1){
    $L(u_0)$
    };

    \node at (4,0) [space] (Fs){
    $F^s$
    };
    \node at (4,3) [space] (Es){
    $E^s$
    };
    \draw[-stealth] (Fs) -- (Es);
    \node at (4,1.5) [map] (FsEs){
    $L(u_0)$
    };

    \draw[-stealth] (Es1) -- (Fs);
    \node at (2,1.5) [map] (Es1Fs){
    $D\Psi(u_0)$
    };

    \draw[right hook-stealth] (Es1) -- (Es);
    \draw[right hook-stealth] (Fs1) -- (Fs);

    \node at (8,0) [space] (Fs-1){
    $F^{s-\mu}$
    };
    \node at (8,3) [space] (Es-1){
    $E^{s-\mu}$
    };
    \draw[-stealth] (Fs-1) -- (Es-1);
    \node at (8,1.5) [map] (Fs-1Es-1){
    $L(u_0)$
    };

    \draw[-stealth] (Es) -- (Fs-1);
    \node at (6,1.5) [map] (EsFs-1){
    $D\Psi(u_0)$
    };

    \draw[right hook-stealth] (Fs) -- (Fs-1);
    \draw[right hook-stealth] (Es) -- (Es-1);

    \node at (10,0) [space] (F0){
    $F^0$
    };
    \node at (10,3) [space] (E0){
    $E^0$
    };
    \draw[right hook-stealth] (Es-1) -- (E0);
    \draw[right hook-stealth] (Fs-1) -- (F0);

    \node at (-2,0) [space] (FN){
    $F^N$
    };
    \node at (-2,3) [space] (EN){
    $E^N$
    };
    \draw[right hook-stealth] (FN) -- (Fs1);
    \draw[right hook-stealth] (EN) -- (Es1);

\end{tikzpicture}
    \caption{Commutative diagram arising from the LRI mapping hypotheses for $u_0\in E^N\cap B_{E^r}(0,\del_R)$ and $r+\mu\le s\le R-\mu$. The `$\hookrightarrow$' are the inclusion maps.}
    \label{fig:NM commutative diagram}
\end{figure}

	To conclude this subsection, we state our version of the Nash-Moser inverse function theorem, which is divided into two parts.  
	
	\begin{thm}[Inverse function theorem]\label{thm on nmh}
		Let $\bf{E} = \{E^s\}_{s\in \j{N}}$ and $\bf{F} = \{F^s\}_{s\in \j{N}}$ be Banach scales over the same field, and assume that $\bf{E}$ and $\bf{F}$ satisfy one of the following two conditions:
		\begin{enumerate}[I:]
			\item $\bf{E}$ and $\bf{F}$ are LP-smoothable (see Definition~\ref{defn of smoothable and LP-smoothable Banach scales}), or 
			\item $\bf{E}$ is tame and $\bf{F}$ is a tame direct summand of $\bf{E}$ (see Definition~\ref{defn of tame direct summands and tame Banach scales}).
		\end{enumerate} 
		Assume the triple $(\bf{E},\bf{F},\Psi)$ satisfies the LRI mapping hypotheses of Definition \ref{defn of the mapping hypotheses} with parameters $(\mu,r,R)$ satisfying $2(r+\mu) + 1 < (r+R)/2$, and set $\be=2(r+\mu)+1\in\j{N}$.   Then there exist $\ep,\kappa_1,\kappa_2 >0$ such that the following hold.
		\begin{enumerate}
			\item Existence of local inverse: For every $g \in B_{F^\beta}(0,\ep)$ there exists a unique $u \in B_{E^\beta}(0,\kappa_1 \ep)$ such that $\Psi(u) =g$. 

			\item Estimates of local inverse:  The induced bijection $\Psi^{-1} :  B_{F^\beta}(0,\ep) \to \Psi^{-1}(B_{F^\beta}(0,\ep)) \cap  B_{E^\beta}(0,\kappa_1 \ep)$ obeys the estimate
			\begin{equation}\label{nmh_01}
				\tnorm{\Psi^{-1}(g)}_{E^\beta} \le \kappa_1 \tnorm{g}_{F^\beta} \text{ for all }g \in B_{F^\beta}(0,\ep).
			\end{equation}
			Moreover, if $\N\ni\nu\le R+r-2\be$, then we have that
			\begin{equation}\label{derivatives_gain}
				\Psi^{-1}:B_{F^\be}(0,\ep)\cap F^{\be+\nu}\to E^{\be+\nu}
			\end{equation}
			with the estimate
			\begin{equation}\label{when the rain comes they run and hide}
				\tnorm{\Psi^{-1}(g)}_{E^{\be+\nu}}\lesssim\tnorm{g}_{F^{\be+\nu}},
			\end{equation}
			where the implied constant is independent of $g$.
			
			\item Basic continuous dependence: The map $\Psi^{-1}$ obeys the Lipschitz bound
			\begin{equation}\label{hey cuz}
				\tnorm{\Psi^{-1}(g_0)-\Psi^{-1}(g_1)}_{E^{\be-\mu}}\le\kappa_2\tnorm{g_0-g_1}_{F^{\be-\mu}}
			\end{equation}
			for all $g_0,g_1\in B_{F^\be}(0,\ep)$.
		\end{enumerate}
	\end{thm}
	\begin{rmk}
		The inequality $3r +4\mu + 2< R$ is equivalent to $\be=2(r+\mu) +1 < (r+R)/2$.  Furthermore, when $\j{N}$ is finite, the mapping hypotheses require that $R+\mu\le N$, and so we obtain the necessary relation $3r + 5 \mu + 2< N$.
	\end{rmk}

	We present the proof of Theorem~\ref{thm on nmh} in Section \ref{what would you do if i sang} after first establishing two other theorems that prove separate components of the theorem under different hypotheses.  We now turn to the statement of the second part of our inverse function theorem, beginning with some notation.  Due to the derivative loss in the nonlinear operators under consideration, the higher regularity of the local inverse map is most conveniently phrased in terms of some variation on Gateaux derivatives rather than the usual Fr\'echet notion of differentiability. This is analogous to what is done in Section I.3 and elsewhere in Hamilton~\cite{MR656198}.

	\begin{defn}[Continuous Gateaux differentiability]\label{definition of gateaux}
		Let $X$ and $Y$ be Banach spaces over a common field, $U\subseteq X$ an open set, and $f:U\to Y$. 
		\begin{enumerate}
			\item We say that $f$ is continuously Gateaux differentiable on $U$ if there exists a continuous map $Df:U\times X\to Y$ such that for all $x\in U$ we have that $Df(x)\in\mathcal{L}(X;Y)$ and for all $z\in X$
			\begin{equation}
				\lim_{t\to0}\;t^{-1}\tp{f(x+tz)-f(x)}=Df(x)z.
			\end{equation}
			\item For $\N\ni\ell\ge2$, we say that $f$ is $\ell$-times continuously Gateaux differentiable if $f$ is continuously Gateaux differentiable and $Df:U\times X\to Y$ is $(\ell-1)$-times continuously Gateaux differentiable.
		\end{enumerate}
	\end{defn}
	
	We can now state the second part of our Nash-Moser inverse function theorem, the proof of which is in Section \ref{section on refinements}.

	\begin{thm}[Further conclusions of the inverse function theorem]\label{thm on further conclusions of the inverse function theorem}
		Assume the hypotheses of Theorem~\ref{thm on nmh} and additionally that the Banach scale $\bf{E}$ consists of reflexive spaces. The following additional conclusions hold for the local inverse map $\Psi^{-1}:B_{F^\be}(0,\ep)\to E^\be$.
		\begin{enumerate}
			\item Continuity: For $s\in[\be,R+r-\be-\mu)\cap\N$ the map
			\begin{equation}
				\Psi^{-1}:B_{F^\be}(0,\ep)\cap F^s\to E^s
			\end{equation}
			is continuous.
			
			\item Continuous differentiability:  For $s\in[\be,R+r-\be-\mu)\cap\N$,  the map
			\begin{equation}
				\Psi^{-1}:B_{F^\be}(0,\ep)\cap F^s\to E^{s-\mu}
			\end{equation}
			is differentiable in the Fr\'echet sense with $D\Psi^{-1}=L\circ\Psi^{-1}$. Moreover, when viewing $D\Psi^{-1}$ as map
			\begin{equation}
				D\Psi^{-1}:(B_{F^\be}(0,\ep)\cap F^s)\times F^{s-\mu}\to E^{s-\mu}
			\end{equation}
			we have that it is continuous.
			
			\item Higher regularity: Let $\N\ni\ell\ge 2$ and assume that the map $\Psi$ is $\ell$-times continuously Gateaux differentiable. For $s\in[\be+(\ell(\ell+1)/2-1)\mu,R+r-\be-\mu)\cap\N$ we have that the map
			\begin{equation}\label{Georgia}
				\Psi^{-1}:B_{F^\be}(0,\ep)\cap F^s\to E^{s-\f{\ell(\ell+1)}{2}\mu}
			\end{equation}
			is also $\ell$-times continuously Gateaux differentiable.
		\end{enumerate}
	\end{thm}
	
	\subsection{Local surjectivity and injectivity}\label{billy shears}
	
	We begin this section by proving an existence result modeled on the main theorem of Baldi and Haus \cite{MR3711883}, but adapted to our specific setting.  We emphasize that the following theorem only requires the existence of right inverses in the mapping hypotheses.
	
	\begin{thm}[Local surjectivity]\label{iteration_thm}
		Let $\bf{E} = \{E^s\}_{s\in \j{N}}$ and $\bf{F} = \{F^s\}_{s\in \j{N}}$ be Banach scales over the same field, and suppose that $(\bf{E},\bf{F},\Psi)$ satisfies the RI mapping hypotheses of Definition \ref{defn of the mapping hypotheses} with parameters $(\mu,r,R)$ such that $2(r+\mu) +1 < (r+R)/2$.  Set $\be=2(r+\mu)+1\in\j{N}$ and let $A_\beta >0$ be such that $\norm{\Updelta_j f}_{F^\beta} \le A_\beta \norm{f}_{F^\beta}$ for all $j\in\N$ and $f\in F^\be$. Let $C_i = C_i(R)$ denote the constants from the mapping hypotheses for $i \in \{1,2\}$.
		
		Then there exist $K_1,K_2,K_3,K_4 \in \R^+$ such that for every $g \in F^\beta$ satisfying
		\begin{equation}\label{it_0_gbnd}
			\norm{g}_{F^\beta} < \frac{\delta_R}{2[(1+A_\beta) K_1 + K_2 + K_3 + K_4 + (1+A_\beta)^2 K_1(K_1 + K_3) ]}
		\end{equation}
		there exists a sequence  $\{(u_j,v_j,h_j,y_j,f_j,e_j)\}_{j=0}^\infty \subseteq (B_{E^r}(0,\delta_R) \cap E^R) \times E^N \times E^R \times F^N \times F^N \times F^{R-\mu}$ satisfying the following.
		\begin{enumerate}
			\item When $j=0$ we have the identities
			\begin{align}\label{it_0_seed}
				u_0 &= 0, &
				v_0 &= S_0u_0 = 0, \nonumber\\
				y_0 &= 0, &
				f_0 &= \Updelta_0g = S_1g, \\
				h_0 &= L(u_0)f_0 = L(0)S_1g, &
				e_0 &= \Psi(u_0+h_0)-\Psi(u_0)-f_0\nonumber,
			\end{align}
			while for $j \ge 1$ we have the recursive relations
			\begin{align}\label{it_0_iteration}
				u_j &= u_{j-1} + h_{j-1}, &
				v_j &= S_j u_j, \nonumber\\
				y_j &= -S_j \sum_{n=0}^{j-1} e_n - \sum_{n=0}^{j-1} y_n, &
				f_j &=\Updelta_j g + y_j, \\
				h_j &= L(v_j) f_j, &
				e_j &= \Psi(u_j +h_j) - \Psi(u_j) - f_j\nonumber.
			\end{align}
			
			\item For $j \in \N$ we have the estimates 
			\begin{equation}\label{it_0_h}
				\tnorm{h_j}_{E^s} \le K_1 \tp{ \norm{g}_{F^\beta} 2^{-j} + \norm{\Updelta_j g}_{F^\beta} } 2^{j(s-\beta)} \text{ for all } s \in [r,R] \cap \N,
			\end{equation}
			\begin{equation}\label{it_0_v}
				\tnorm{v_j}_{E^s} \le K_2 \tnorm{g}_{F^\beta} 2^{j(s-\beta)} \text{ for all } \beta \le s \in \j{N},
			\end{equation}
			\begin{equation}\label{it_0_u-v}
				\tnorm{u_j - v_j}_{E^s} \le K_3 \tnorm{g}_{F^\beta} 2^{j(s-\beta)} \text{ for all } s \in [0,R] \cap \N,
			\end{equation}
			\begin{equation}\label{it_0_u}
				\tnorm{u_j}_{E^\beta} \le K_4 \tnorm{g}_{F^\beta}.
			\end{equation}
			and 
			\begin{equation}\label{it_0_e}
				\tnorm{e_j}_{F^s} \lesssim C_1 \tnorm{g}_{F^\beta} 2^{j(s+\mu+r - 2\beta)} \text{ for all } s \in [r-\mu,R-\mu] \cap \N.
			\end{equation}
			
			\item There exists $u \in E^\beta$ such that $u_j \to u$ in $E^\beta$ as $j \to \infty$, $\norm{u}_{E^\beta} \le K_4 \norm{g}_{F^\beta}$, and $\Psi(u) =g$. 
			
			\item If we know additionally that $\nu\in\N$ is such that $\nu\le R+r-2\be$ and $g\in F^{\be+\nu}$, then actually $u_j  \to u$ in $E^{\be+\nu}$ as $j \to \infty$ and we have the estimate $\tnorm{u}_{E^{\be+\nu}}\lesssim\tnorm{g}_{F^{\be+\nu}}$. Here the implied constant depends on $K_1,\dots,K_4$,  $A_\be, A_{\be+\nu}$,  $\mu$, $r$, $\delta_R$, and $R$.
		\end{enumerate}
	\end{thm}
	\begin{proof}
		
		We divide the proof into several steps. The first three steps establish a trio of crucial claims that will be used in the fourth step to inductively construct the sequence. The convergence result is then proved in the fifth step. The higher regularity assertions of the fourth item are then proved in the sixth step. Throughout the proof we will utilize the sequence $\{\xi_n\}_{n=0}^\infty \subseteq [0,\infty)$ defined by  
		\begin{equation}\label{it_xi}
			\xi_n = \norm{g}_{F^\beta} 2^{-n} + \norm{\Updelta_n g}_{F^\beta}.
		\end{equation}

		\textbf{Step 1}: An estimate in the style of Littlewood-Paley theory. Let $\ell,k \in \N$ satisfy $0 \le \ell \le k$.  We claim that if  $\{h_j\}_{j=0}^k \subseteq E^R$ are given and satisfy \eqref{it_0_h} for $0 \le j \le k$, then 
		\begin{equation}\label{it_4}
			\bnorm{\sum_{n=\ell}^k h_n}_{E^\beta} \lesssim K_1 \bp{\sum_{n=\ell}^k \xi_n^2}^{1/2},
		\end{equation}
		where $\xi_n$ is defined by \eqref{it_xi}

		To prove the claim, we begin by noting that for any $j \in \N$, $0 \le n \le k$, and $s=\be+\m{sgn}(j-n)\in[r,R]\cap\N$ we may bound, via~\eqref{smoothing_4} and~\eqref{it_0_h},
		\begin{equation}\label{it_1}
			\tnorm{\Updelta_j h_n}_{E^\beta} \lesssim 2^{j(\beta-s)} \tnorm{h_n}_{E^s} \lesssim K_1 \xi_n 2^{j(\beta-s)}2^{n(s-\beta)} = K_1 \xi_n 2^{(j-n)(\beta-s)}=K_1\xi_n2^{-|j-n|}.
		\end{equation}
		We estimate the sum on the left hand side of~\eqref{it_4} via the `Littlewood-Paley' characterization~\eqref{smoothing_LP_2}, the  bound~\eqref{it_1}, and  Young's convolution inequality (setting $\xi_n =0$ for $n \in \mathbb{Z}$ such that $0 < n$):
		\begin{multline}\label{yes very good at spelling}
			\bnorm{\sum_{n=\ell}^k h_n}_{E^\be}^2 \lesssim
			\sum_{j=0}^\infty \bnorm{\Updelta_j \sum_{n=\ell}^k h_n}_{E^\be}^2 
			\le \sum_{j=0}^\infty \bp{\sum_{n=\ell}^k \tnorm{\Updelta_j h_n}_{E^\be}}^2 \lesssim K_1^2 \sum_{j=0}^\infty \bp{\sum_{n=\ell}^k\xi_n2^{-|j-n|}}^2 \\
			\le K_1^2 \sum_{j \in \mathbb{Z}} \bp{\sum_{n \in \mathbb{Z}} \xi_n \mathds{1}_{[\ell,k]}(n)  2^{-|j-n|}}^2 
			\le K_1^2 \bp{\sum_{n \in \mathbb{Z}}  \xi_n^2 \mathds{1}_{[\ell,k]}(n) }
			\bp{\sum_{n \in \mathbb{Z}} 2^{-\abs{n}}  }^2
			\lesssim K_1^2\sum_{n=\ell}^k\xi_n^2.
		\end{multline}
		This is~\eqref{it_4}.

		\textbf{Step 2}: An estimate from Taylor's theorem.   Let $k \in \N$.  We claim that if $g \in F^\beta$ satisfies~\eqref{it_0_gbnd} and $\{(u_j,v_j,h_j)\}_{j=0}^k \subseteq (B_{E^r}(0,\delta_R) \cap E^R) \times E^N\times E^R$ are given and satisfy~\eqref{it_0_h}--\eqref{it_0_u} for $0 \le j \le k$,  as well as the conditions $u_0 = v_0 = 0$ and $u_j = u_{j-1} + h_{j-1}$ if $1 \le j \le k$, then 
		\begin{equation}\label{it_5_inc}
			\{u_j\}_{j=0}^k,     \{u_j+h_j\}_{j=0}^k,     \{v_j\}_{j=0}^k \subseteq B_{E^r}(0,\delta_R) \subseteq B_{E^r}(0,\delta_r)
		\end{equation}
		and
		\begin{equation}\label{it_5_bnd}
			\tnorm{ \Psi(u_j +h_j) - \Psi(u_j) - D\Psi(v_j) h_j }_{F^s} \lesssim C_{1} \tnorm{g}_{F^\beta} 2^{j(s+\mu+r - 2\beta)}  
		\end{equation}
		for all $s \in [r-\mu,R-\mu] \cap \N$.
		
		To prove this claim we begin by employing \eqref{it_0_gbnd}.  Indeed, \eqref{it_0_gbnd} and \eqref{it_0_u} imply that $\norm{u_j}_{E^r} \le  \norm{u_j}_{E^\beta} \le K_4 \norm{g}_{F^\beta} < \frac{\delta_R}{2} \le \frac{\delta_r}{2}$, \eqref{it_0_v} and \eqref{it_0_gbnd} imply that $\norm{v_j}_{E^r} \le  \norm{v_j}_{E^\beta} \le K_2 \norm{g}_{F^\beta} < \frac{\delta_R}{2} \le \frac{\delta_r}{2}$ for $0 \le j \le k$,  and \eqref{it_0_h} and~\eqref{it_0_gbnd} imply that $\norm{h_j}_{E^r} \le \norm{h_j}_{E^\beta} \le K_1 \left( \norm{g}_{F^\beta} + A_\beta \norm{g}_{F^\beta}\right) = K_1(1+A_\beta) \norm{g}_{F^\beta} <  \frac{\delta_R}{2} \le \frac{\delta_r}{2}$ for the same range of $j$.  Together, these three bounds imply the inclusions~\eqref{it_5_inc}.
		
		Now that~\eqref{it_5_inc} is established, we know that $\Psi(u_j+h_j)$, $\Psi(u_j)$, and $D\Psi(v_j)$ are well-defined.  We may then employ the fundamental theorem of calculus, Taylor's theorem, and the convexity of $B_{E^r}(0,\delta_r)$ to write 
		\begin{multline}
			\Psi(u_j+h_j) - \Psi(u_j)  - D\Psi(v_j) h_j = \tp{D\Psi(u_j) - D \Psi(v_j)}h_j + \int_0^1 (1-t) D^2 \Psi(u_j + th_j)[h_j,h_j] dt \\
			= \int_0^1 D^2 \Psi((1-t) v_j + t u_j)[h_j,u_j - v_j] dt + \int_0^1 (1-t) D^2 \Psi(u_j + th_j)[h_j,h_j] dt 
			= \bf{I}_j + \bf{II}_j.
		\end{multline}
		By using the tame $C^2$ estimate from the mapping hypotheses of Definition~\ref{defn of the mapping hypotheses}, we readily deduce that for $r-\mu \le s \in \j{N-\mu}$, 
		\begin{multline}
			\tnorm{\bf{I}_j}_{F^s} \lesssim C_1\tbr{\norm{u_j - v_j}_{E^r}}  \norm{u_j - v_j}_{E^{s+\mu}} \norm{h_j}_{E^r} \\
			+ C_1 \tbr{ \norm{v_j}_{E^{s+\mu}}} \norm{u_j - v_j}_{E^r} \norm{h_j}_{E^r} 
			+ C_1 \norm{u_j - v_j}_{E^r} \norm{h_j}_{E^{s+\mu}} 
		\end{multline}
		and 
		\begin{equation}
			\norm{\bf{II}_j}_{F^s} \lesssim C_1 \tbr{ \norm{h_j}_{E^r}} \norm{h_j}_{E^{s+\mu}} \norm{h_j}_{E^r} + C_1 \tbr{ \norm{u_j}_{E^{s+\mu}}}\norm{h_j}_{E^r}^2.
		\end{equation}
		Synthesizing these bounds, we find that 
		\begin{multline}\label{it_6}
			\norm{    \Psi(u_j+h_j) - \Psi(u_j)  - D\Psi(v_j) h_j}_{F^s} \lesssim
			C_1 \tbr{ \norm{u_j - v_j}_{E^r} }  \norm{u_j - v_j}_{E^{s+\mu}} \norm{h_j}_{E^r} \\ 
			+ C_1 \norm{u_j - v_j}_{E^r} \norm{h_j}_{E^{s+\mu}} 
			+ C_1 \tbr{  \norm{h_j}_{E^r}}\norm{h_j}_{E^{s+\mu}} \norm{h_j}_{E^r}     \\
			+C_1 \tbr{ \norm{u_j,v_j}_{E^{s+\mu}} }\norm{h_j}_{E^r} (\norm{h_j}_{E^r} + \norm{u_j-v_j}_{E^r})   
		\end{multline}
		for $r-\mu \le s \in \j{N-\mu}$.
		
		Next we turn our attention to estimates for $u_j$ and $v_j$.  Note that if $1 \le j \le k$, then the identity $u=\sum_{n=0}^{j-1}h_n$ and~\eqref{it_0_h} imply that 
		\begin{equation}
			\norm{u_j}_{E^R} \le \sum_{n=0}^{j-1} \norm{h_n}_{E^R} \le K_1(1+A_\beta)\norm{g}_{F^\beta} \sum_{n=0}^{j-1} 2^{n(R-\beta)} \lesssim K_1(1+A_\beta) \norm{g}_{F^\beta}  2^{j(R-\beta)}.
		\end{equation}
		In turn, when we couple this with the smoothing estimates from Definition~\ref{defn of smoothable and LP-smoothable Banach scales}, the interpolation result from Lemma~\ref{lem on log-convexity in tame Banach scales}, and~\eqref{it_0_u}, this implies that 
		\begin{equation}
			\norm{v_j}_{E^s} \lesssim \norm{u_j}_{E^s} \lesssim \norm{u_j}_{E^\beta}^{\frac{R-s}{R-\beta}} \norm{u_j}_{E^R}^{\frac{s-\beta}{R-\beta}}  \lesssim  K_4^{\frac{R-s}{R-\beta}} \tp{K_1(1+A_\beta)}^{\frac{s-\beta}{R-\beta}} \norm{g}_{F^\beta} 2^{j(s-\beta)}
		\end{equation}
		for all $s\in\N\cap[\be,R]$.
		However, \eqref{it_0_gbnd} implies that $K_4^{\frac{R-s}{R-\beta}} \tp{K_1(1+A_\beta)}^{\frac{s-\beta}{R-\beta}} \norm{g}_{F^\beta} \le 1$, so we may further bound $\norm{v_j}_{E^s} \lesssim \norm{u_j}_{E^s} \lesssim 2^{j(s-\beta)}$  for all $s \in \N \cap [\beta,R]$.
		
		On the other hand, \eqref{it_0_u} and \eqref{it_0_gbnd} give us that $\norm{v_j}_{E^s} \lesssim \norm{u_j}_{E^s} \lesssim \norm{u_j}_{E^\beta} \lesssim K_4 \norm{g}_{F^\beta} \lesssim 1$ for all $s \in \N \cap [r,\beta]$. Upon combining these, we find that 
		\begin{equation}\label{it_7}
			\norm{u_j}_{E^s} + \norm{v_j}_{E^s} \lesssim  2^{j \max\{0,s-\beta\}}  \text{ for all }s \in \N \cap [r,R] \text{ and } 0 \le j \le k,
		\end{equation}
		where we note that these bounds are trivial for $j=0$ since $u_0 = v_0 =0$.
		
		Now we combine \eqref{it_6} and \eqref{it_7} with \eqref{it_0_h}--\eqref{it_0_u} to bound, for $0 \le j \le k$ and $s \in \N \cap [r-\mu,R-\mu]$, 
		\begin{multline}
			\norm{    \Psi(u_j+h_j) - \Psi(u_j)  - D\Psi(v_j) h_j}_{F^s} \\
			\lesssim C_1 \tbr{  K_3 \norm{g}_{F^\beta} 2^{j(r-\beta)}} K_3 \norm{g}_{F^\beta} 2^{j(s+\mu-\beta)} \cdot K_1(1+A_\beta) \norm{g}_{F^\beta} 2^{j(r-\beta)} \\
			+ C_1 K_3 \norm{g}_{F^\beta} 2^{j(r-\beta)} \cdot K_1(1+A_\beta) \norm{g}_{F^\beta} 2^{j(s+\mu-\beta)} \\
			+ C_1 \tbr{  K_1(1+A_\beta)\norm{g}_{F^\beta} 2^{j(r-\beta)}} K_1 (1+A_\beta) 2^{j(s+\mu-\beta)} \cdot K_1 (1+A_\beta) 2^{j(r-\beta)} \\
			+ C_1 \tbr{ 2^{j\max\{0,s+\mu-\beta\}} } K_1(1+A_\beta) \norm{g}_{F^\beta} 2^{j(r-\beta)} \cdot ( K_1(1+A_\beta) \norm{g}_{F^\beta} 2^{j(r-\beta)}  + K_3 \norm{g}_{F^\beta} 2^{j(r-\beta)} ).
		\end{multline}
		For the last term we note that 
		\begin{equation}
			\begin{cases}
				r -\mu  \le s \le \beta -\mu \Rightarrow 2r -2\beta \le s+\mu +r - 2 \beta \\
				\beta-\mu \le s \le R -\mu \Rightarrow s+\mu-\beta + 2r-2\beta \le s+\mu+r-2\beta.
			\end{cases}
		\end{equation}
		Thus, upon regrouping and using \eqref{it_0_gbnd}, we then see that
		\begin{multline}
			\norm{    \Psi(u_j+h_j) - \Psi(u_j)  - D\Psi(v_j) h_j}_{F^s} 
			\lesssim C_1(1+A_\beta)^2 K_1(K_1 + K_3) \norm{g}_{F^\beta}^2 2^{j(s+\mu+r-2\beta)} \\
			\lesssim  C_1   \norm{g}_{F^\beta}  2^{j(s+\mu+r-2\beta)} 
		\end{multline}
		for all $s \in \N \cap [r-\mu,R-\mu]$ and $0\le j \le k$.  This is \eqref{it_5_bnd}, and so the proof of the claim is complete.

		\textbf{Step 3}: A recursive identity.  Suppose that $\{(y_j,e_j)\}_{j=0}^k \subseteq F^N \times F^{R-\mu}$ is given for $2 \le k \in \N$ and satisfies $y_0 =0$ and the recursive condition \begin{equation}\label{recursive_condition_must_label_using_underscores_instead_of_spaces_because_it_is_cooler_sorry_}
			y_j = -S_j \sum_{n=0}^{j-1}e_n - \sum_{n=0}^{j-1}y_n \text{ for }1 \le j \le k.     
		\end{equation}
		We claim that 
		\begin{equation}\label{_u_n_d_e_r_s_c_o_r_e_}
			y_j = -S_j e_{j-1} - \Updelta_{j-1} \sum_{n=0}^{j-2} e_n \text{ for } 2\le j \le k.
		\end{equation}
		To see this, note that from~\eqref{recursive_condition_must_label_using_underscores_instead_of_spaces_because_it_is_cooler_sorry_} we deduce that for any $1\le j\le k$ the identity $\sum_{n=0}^j y_n=-S_j\sum_{n=0}^{j-1}e_n$. Hence, for $2\le j\le k$ we have $y_j=-S_j\sum_{n=0}^{j-1}e_n-S_{j-1}\sum_{n=0}^{j-2}e_n$, but by the definitions of the $\Updelta_{n}$ and $e_n$, this is the same as $y_j=-S_je_{j-1}-\Updelta_{j-1}\sum_{n=0}^{j-2}e_n$,
		and therefore~\eqref{_u_n_d_e_r_s_c_o_r_e_} holds.
		
		\textbf{Step 4}: Inductive construction of the sequence.  We now aim to inductively construct the desired sequence under some assumptions on the constants $K_1$, $K_2$, $K_3$, $K_4$, which we will work out as we proceed.  
		
		We begin by seeding the sequence, i.e. constructing its elements with $j=0$.  Define $u_0 =0 \in B_{E^r}(0,\delta_R) \cap E^R$, $v_0 = S_0 u_0 = 0 \in E^N$, $y_0 =0 \in F^N$, $f_0 = \Updelta_0 g = S_1 g \in F^N$,  $h_0 = L(u_0) f_0 = L(0) S_1 g \in E^R$, and $e_0 = \Psi(h_0) - \Psi(0) - f_0 \in F^{R-\mu}$.  By construction, the bounds \eqref{it_0_v}, \eqref{it_0_u-v}, and \eqref{it_0_u} hold trivially when $j=0$ for all possible choices of $K_2,K_3, K_4 > 0$.  On the other hand, we can apply the tame estimate for $L(0)$ from the mapping hypotheses to see that for $s \in \N \cap [r,R]$, \begin{multline}
			\norm{h_0}_{E^s} \le C_2 \norm{S_1 g}_{F^s} + C_2(1 + \norm{0}_{E^{s+\mu}} ) \norm{S_1 g}_{F^r} \lesssim C_2 \norm{\Updelta_0 g}_{F^\beta} +  C_2 \norm{g}_{F^\beta} \\
			\le  K_1 \left( \norm{g}_{F^\beta} 2^{-0} + \norm{\Updelta_j g}_{F^\beta} \right) 2^{0(s-\beta)},
		\end{multline}
		provided the constant $K_1$ satisfies the bound
		\begin{equation}\label{it_K_1}
			C_2  \lesssim K_1,
		\end{equation}
		which we henceforth assume holds. We then have that \eqref{it_0_h} is satisfied for $j=0$.
		
		We now proceed to the inductive step.  Let $k \in \N$ and suppose that 
		\begin{equation}
			\{(u_j,v_j,h_j,y_j,f_j,e_j)\}_{j=0}^k \subseteq (B_{E^r}(0,\delta_R) \cap E^R) \times E^N \times E^R \times F^N \times F^N \times F^{R-\mu}    
		\end{equation}
		are given and satisfy \eqref{it_0_seed},  \eqref{it_0_h}--\eqref{it_0_u} for $0 \le j \le k$,   and \eqref{it_0_iteration} if $1 \le j \le k$.  We will construct $(u_{k+1},v_{k+1},h_{k+1},y_{k+1},f_{k+1},e_{k+1}) \in  (B_{E^r}(0,\delta_R) \cap E^R) \times E^N \times E^R \times F^N\times F^N\times F^{R-\mu}$ and show that \eqref{it_0_h}--\eqref{it_0_u} continue to hold for $j=k+1$.
		
		First, we define $u_{k+1} = u_k + h_k \in E^R$.  The claim established in Step 2 guarantees that $u_{k+1} \in B_r(0,\delta_R) \subseteq B_r(0,\delta_r)$.  Moreover, we may use a telescoping argument to see that $u_{k+1} = u_{k+1} - u_0 = \sum_{n=0}^k h_n$. Thus, the claim established in Step 1 shows that 
		\begin{equation}\label{it_8}
			\norm{u_{k+1}}_{E^\beta} \lesssim K_1\bp{ \sum_{n=0}^\infty \xi_n^2}^{1/2} \lesssim K_1 \norm{g}_{F^\beta},
		\end{equation}
		where in the last inequality we have used the Littlewood-Paley bound~\eqref{smoothing_LP_2}.  Thus, $u_{k+1} \in B_{E^r}(0,\delta_R) \cap E^R$ satisfies~\eqref{it_0_u} with $j=k+1$, provided that $K_1$ and $K_4$ satisfy 
		\begin{equation}\label{it_K_2}
			K_1 \lesssim K_4,
		\end{equation}
		which we henceforth assume.

		Second, we define $v_{k+1} = S_{k+1} u_{k+1} \in E^N$.  We then use~\eqref{it_0_h} and the smoothing bounds~\eqref{smoothing_1}--\eqref{smoothing_4} to estimate 
		\begin{multline}
			\norm{v_{k+1} - u_{k+1}}_{E^R} =  \norm{(I-S_{k+1}) u_{k+1}}_{E^R} \lesssim \norm{u_{k+1}}_{E^R}  \le \sum_{n=0}^k \norm{h_n}_{E^R} \le K_1 \sum_{n=0}^k \xi_n 2^{n(R-\beta)} \\
			\le K_1 \bp{\sum_{n=0}^\infty \xi_n^2 }^{1/2} \bp{\sum_{n=0}^k 2^{2n(R-\beta)}}^{1/2} 
			\lesssim K_1 \norm{g}_{F^\beta} 2^{(k+1)(R-\beta)}.
		\end{multline}
		On the other hand, \eqref{it_8} and the smoothing bounds~\eqref{smoothing_1}--\eqref{smoothing_4} show that
		\begin{multline}
			\norm{v_{k+1} - u_{k+1}}_{E^0} = \norm{(I-S_{k+1})u_{k+1}}_{E^0} \lesssim 2^{-(k+1) \beta} \norm{(I-S_{k+1})u_{k+1}}_{E^\beta} \\
			\lesssim 2^{-(k+1)\beta} \norm{u_{k+1}}_{E^\beta} \lesssim K_1 \norm{g}_{F^\beta} 2^{-(k+1) \beta}.
		\end{multline}
		Upon combining these two estimates and interpolating with the help of Lemma~\ref{lem on log-convexity in tame Banach scales}, we find that for $s \in \N \cap [0,R]$,
		\begin{equation}
			\norm{v_{k+1} - u_{k+1}}_{E^s} \lesssim     \norm{v_{k+1} - u_{k+1}}_{E^0}^{1-s/R}    \norm{v_{k+1} - u_{k+1}}_{E^R}^{s/R} \lesssim K_1 \norm{g}_{F^\beta} 2^{(k+1)(s-\beta)},
		\end{equation}
		and so $v_{k+1}\in E^N$ satisfies \eqref{it_0_u-v} with $j=k+1$ provided that $K_1$ and $K_3$ satisfy 
		\begin{equation}\label{it_K_3}
			K_1 \lesssim K_3,
		\end{equation}
		which we henceforth assume.

		Continuing with $v_{k+1}$, we observe that for $\beta \le s \in \j{N}$ we can use the smoothing bounds of Definition~\ref{defn of smoothable and LP-smoothable Banach scales} to estimate 
		\begin{multline}
			\norm{v_{k+1}}_{E^s} = \norm{S_{k+1} u_{k+1}}_{E^s} \lesssim 2^{(k+1)(s-\beta)} \norm{S_{k+1} u_{k+1}}_{E^\beta} \lesssim 2^{(k+1)(s-\beta)} \norm{u_{k+1}}_{E^\beta} \\
			\lesssim K_1 \norm{g}_{F^\beta} 2^{(k+1)(s-\beta)},
		\end{multline}
		where we have again used \eqref{it_8}.  Thus, $v_{k+1}$ obeys the bound \eqref{it_0_v} for $j=k+1$ provided that $K_1$ and $K_2$ satisfy 
		\begin{equation}\label{it_K_4}
			K_1 \lesssim K_2,
		\end{equation}
		which we henceforth assume, and in this case \eqref{it_0_gbnd} in turn implies that
		\begin{equation}\label{it_12}
			\norm{v_{k+1}}_{E^r} \le   \norm{v_{k+1}}_{E^\beta} \le K_2 \norm{g}_{F^\beta} < \delta_R.
		\end{equation}
		
		Third, we introduce some useful estimates for the terms $\{e_j\}_{j=0}^k$.  For $0 \le j \le k$ we know that $h_j = L(v_j) f_j$, and so $f_j = D\Psi(v_j) h_j$ by the RI mapping hypotheses.  We may thus invoke the claim established in Step 2 in order to see that for $0 \le j \le k$ \eqref{it_5_bnd} implies that
		\begin{multline}\label{it_9}
			\norm{e_j}_{F^s} = \norm{\Psi(u_j + h_j) - \Psi(u_j) - f_j}_{F^s} = \norm{\Psi(u_j + h_j) - \Psi(u_j) - D\Psi(v_j) h_j}_{F^s} \\
			\lesssim C_{1} \norm{g}_{F^\beta} 2^{j(s+\mu+r - 2\beta)}
		\end{multline} 
		for all $s \in \N \cap [r-\mu,R-\mu]$.
		
		Fourth, we turn our attention to defining $y_{k+1}$.  If $k=0$, then we simply set $y_1 = - S_1 e_0 \in F^N$, while if $k \ge 1$ then we set $y_{k+1} = -S_{k+1} \sum_{n=0}^k e_n - \sum_{n=0}^k y_n \in F^N$. We may then invoke the claim of Step 3 to see that the formula 
		\begin{equation}\label{it_10}
			y_j = -S_{j} e_{j-1} - \Updelta_{j-1} \sum_{n=0}^{j-2} e_n 
		\end{equation}
		holds whenever $1 \le j \le k+1$, provided that we understand sums over empty ranges to mean zero.  Next we use \eqref{it_9} to estimate $y_{k+1}$.   Initially we use the smoothing operator properties together with \eqref{it_9} to estimate 
		\begin{equation}\label{it_10_1}
			\norm{S_{k+1} e_k}_{F^s} \lesssim 
			\begin{cases}
				\norm{e_k}_{F^s} &\text{if } r-\mu  \le s \le R-\mu \\
				2^{(k+1)(s-R+\mu)} \norm{e_k}_{F^{R-\mu}} &\text{if } R-\mu \le s \in \j{N} 
			\end{cases}
			\lesssim 
			C_{1} \norm{g}_{F^\beta} 2^{k(s+\mu+r - 2\beta)}
		\end{equation}
		for every $r-\mu \le s \in \j{N} $.  Similarly, if $k \ge 1$ then we can use the fact that $R+r-2\beta >0$ to bound 
		\begin{multline}\label{it_10_2}
			\bnorm{\sum_{j=0}^{k-1} \Updelta_k e_j   }_{F^s} \le \sum_{j=0}^{k-1} \norm{\Updelta_k e_j}_{F^s}  \lesssim 2^{k(s-R+\mu)} \sum_{j=0}^{k-1}  \norm{\Updelta_k e_j}_{F^{R-\mu}} \lesssim  2^{k(s-R+\mu)} \sum_{j=0}^{k-1}  \norm{ e_j}_{F^{R-\mu}} \\
			\lesssim C_{1} \norm{g}_{F^\beta} 2^{k(s-R+\mu)}  \sum_{j=0}^{k-1} 2^{j(R+r-2\beta)} 
			\lesssim C_{1} \norm{g}_{F^\beta} 2^{k(s+\mu+r -2\beta )} 
		\end{multline}
		for every $r-\mu\le s \in \j{N}$.  Synthesizing \eqref{it_10}, \eqref{it_10_1}, and \eqref{it_10_2}, we deduce that 
		\begin{equation}\label{it_11}
			\norm{y_{k+1}}_{F^s} \lesssim C_{1} \norm{g}_{F^\beta} 2^{k(s+\mu+r -2\beta )}
		\end{equation}
		for every $r-\mu \le s \in \j{N}$.

		As the penultimate update we define $f_{k+1} = \Updelta_{k+1}g + y_{k+1} \in F^N$. We know that $v_{k+1}$ satisfies \eqref{it_12}, so the operator $L(v_{k+1})$ exists, and we may make the final update by setting $h_{k+1} = L(v_{k+1}) f_{k+1} \in E^R$. The tame estimates for $L(v_{k+1})$ then provide for the bound 
		\begin{equation}\label{it_13}
			\norm{h_{k+1}}_{E^s} \le C_2 \tp{\norm{\Updelta_{k+1} g}_{F^s} + \norm{y_{k+1}}_{F^s}  } + C_2\tbr{\norm{v_{k+1}}_{E^{s+\mu}}}\tp{\norm{\Updelta_{k+1} g}_{F^r} + \norm{y_{k+1}}_{F^r}}
		\end{equation}
		for $s \in \N \cap [r,R]$.  From \eqref{it_0_v}, which we established above holds for $j=k+1$, and \eqref{it_0_gbnd} we may estimate $\norm{v_{k+1}}_{E^{s+\mu}} \lesssim K_4 \norm{g}_{F^\beta} 2^{(k+1) (s+\mu-\beta)} \lesssim 2^{(k+1) (s+\mu-\beta)}$. By plugging this and \eqref{it_11} into \eqref{it_13}, we then find that 
		\begin{multline}
			\norm{h_{k+1}}_{E^s} \lesssim C_2\left( \norm{\Updelta_{k+1} g}_{F^\beta} 2^{(k+1)(s-\beta)} + C_{1} \norm{g}_{F^\beta} 2^{k(s+\mu+r-2\beta)}   \right) \\
			+ C_2 \tbr{2^{(k+1) (s+\mu-\beta)}}  \left(\norm{\Updelta_{k+1} g}_{F^\beta} 2^{(k+1)(r-\beta)} + C_{1} \norm{g}_{F^\beta} 2^{k(2r+\mu-2\beta)}   \right) \\
			\lesssim  C_{2} (1+C_{1}) 2^{(k+1)(s-\beta)}    ( \norm{\Updelta_{k+1} g}_{F^\beta} (2^{(k+1)(r-s)} + 2^{(k+1)(\mu +r-\beta)} )  \\
			+ \norm{g}_{F^\beta} ( 2^{(k+1)(\mu+r-\beta)} + 2^{(k+1)(\mu+2r-s-\beta)}  + 2^{(k+1)(2r+2\mu -\beta)} ). 
		\end{multline}
		To consolidate all of the exponents we recall that that $r-s \le 0$ and $1 = \beta - 2r -2\mu >0$, which show that $\mu + r-\beta \le 2\mu + 2r - \beta = -1 \le 0$ and $\mu + 2r -s -\beta \le \mu + r -\beta \le -1$.	Hence, the previous estimate implies that $\norm{h_{k+1}}_{E^s} \lesssim C_{2} (1+C_{1}) 2^{(k+1)(s-\beta)}\tp{\norm{\Updelta_{k+1} g}_{F^\beta} + \norm{g}_{F^\beta} 2^{- (k+1)} }$, and we deduce that $h_{k+1}$ obeys \eqref{it_0_h} for $j=k+1$, provided that $K_1$ satisfies 
		\begin{equation}\label{it_K_5}
			C_{2} (1+C_{1}) \lesssim K_1,
		\end{equation}
		which we assume holds.
		
		We have now established conditions on $K_1$, $K_2$, $K_3$, and $K_4$ that are sufficient for the construction of $(u_{k+1},v_{k+1},h_{k+1},y_{k+1},f_{k+1},e_{k+1}) \in  (B_{E^r}(0,\delta_R) \cap E^R) \times E^N \times E^R \times F^N \times F^N \times F^{R-\mu}$, namely \eqref{it_K_1}, \eqref{it_K_2}, \eqref{it_K_3}, \eqref{it_K_4}, and \eqref{it_K_5}.  It is a simple matter to choose parameters satisfying these conditions, and so the inductive step is complete provided these parameters are chosen.  We thus have the desired sequence.  Note that \eqref{it_0_e} follows as above from the claim established in Step 2.

		\textbf{Step 5}: Convergence of the sequence.  With the sequence in hand from Step 4, we use the claim from Step 1 with $0 \le \ell \le  k$ and the fact that $u_{k+1}-u_\ell=\sum_{n=\ell}^k h_\ell$ to see that 
		\begin{equation}
			\tnorm{u_{k+1} - u_\ell}_{E^\beta} \le \sum_{n=\ell}^k \norm{h_n}_{E^\beta} \lesssim K_1 \bp{\sum_{n=\ell}^\infty \norm{g}_{F^\beta}^2 2^{-2 n} + \norm{\Updelta_n g}_{F^\beta}^2 }^{1/2} \to 0
		\end{equation}
		as $\ell\to\infty$. Thus, $\{u_j\}_{j=0}^\infty$ is a Cauchy sequence in $E^\beta$, and hence convergent to some $u \in E^\beta$.  Sending $j \to \infty$ in \eqref{it_0_u} shows that $\norm{u}_{E^\beta} \le K_4 \norm{g}_{F^\beta}$.  
		
		It remains only to show that $\Psi(u) =g$.  To this end, we again telescope and use that $\Psi(0)=\Psi(u_0) =0$ and \eqref{it_0_iteration} to write, for $k \ge 2$, 
		\begin{multline}
			\Psi(u_{k+1}) = \sum_{j=0}^k \tp{\Psi(u_{k+1}) - \Psi(u_k)} =  \sum_{j=0}^k \tp{\Psi(u_{k} +h_k) - \Psi(u_k)} = \sum_{j=0}^k \tp{e_j + f_j} \\
			= \sum_{j=0}^k \Updelta_j g + \sum_{j=0}^k e_j + \sum_{j=0}^k y_j  = S_{k+1}g + \sum_{j=0}^k e_j - S_k \sum_{j=0}^{k-1}e_j = S_{k+1}g + e_k + \tp{I-S_k}\sum_{j=0}^{k-1}e_j.
		\end{multline}
		Now, we know from the properties of the smoothing operators, which are given in Definition~\ref{defn of smoothable and LP-smoothable Banach scales}, that $\norm{(I-S_{k+1})g}_{F^\beta} \to 0$ as $k\to\infty$,
		while \eqref{it_0_e} implies that
		\begin{equation}
			\norm{e_k}_{F^{\beta-\mu}} \lesssim C_{1} \norm{g}_{F^\beta} 2^{k(r-\beta)} \to 0 \text{ as } k \to \infty
		\end{equation}
		and (observing that $\beta  -\mu -1 \ge 2r+\mu \ge 1$)
		\begin{multline}
			\bnorm{(I-S_k) \sum_{j=0}^{k-1}e_j}_{F^{\beta-\mu-1}}  \le \sum_{j=0}^{k-1} \norm{(I-S_k) e_j}_{F^{\beta-\mu -1}} 
			\lesssim 2^{-k} \sum_{j=0}^{k-1} \tnorm{(I-S_k) e_j}_{F^{\beta-\mu}} \\
			\lesssim C_{1} \tnorm{g}_{F^\beta}  2^{-k} \sum_{j=0}^{k-1}2^{j(r-\beta)} 
			\lesssim C_{1} \tnorm{g}_{F^\beta}  2^{-k}  \to 0 \text{ as } k \to \infty. 
		\end{multline}
		Upon combining these, we find that $\norm{\Psi(u_{k+1}) - g}_{F^{\beta-\mu-1}} \to 0$ as $k\to\infty$ but since $u_k \to u$ in $E^\beta$ as $k \to \infty$, the continuity of $\Psi$ guarantees that $\Psi(u_{k+1}) \to \Psi(u)$ in $F^{\beta-\mu}$.  Thus, $\Psi(u) = g$.
		
		\textbf{Step 6}: Higher regularity. Let $\N\ni\upnu\le R+r-2\be$ and now suppose additionally that $g\in F^{\be+\upnu}$. We wish to prove that the solution $u\in E^\be$ constructed in Step 5 actually belongs to $E^{\be+\upnu}$ and satisfies $\tnorm{u}_{E^{\be+\upnu}}\lesssim\tnorm{g}_{F^{\be+\upnu}}$. We will establish this via finite induction.  The proposition to be proved inductively is the following statement, depending on $\nu\in\tcb{0,1,\dots,\upnu}$:  if $j \in \N$ and $s\in\N\cap[r,R]$, then
		\begin{equation}\label{the_inductive_hypothesiesta}
			\tnorm{h_j}_{E^s}\lesssim 2^{j(s-(\be+\nu))}\tp{\tnorm{\Updelta_jg}_{F^{\be+\nu}}+2^{-j}\tnorm{g}_{F^{\be+\nu}}}
		\end{equation}
		for an implicit constant depending on $\nu$, as well as $K_1,\dots,K_4$,  $A_\be, A_{\be+\nu}$,  $\mu$, $r$, $\delta_R$, and $R$.  The case $\nu=0$ was established in the fourth step, in particular in the verification of~\eqref{it_0_h}. Assume now that for some $\nu\in\tcb{0,1,\dots,\upnu-1}$ we have that the induction hypothesis~\eqref{the_inductive_hypothesiesta} holds at level $\nu$. We wish to prove it at level $\nu+1$.
		
		We begin by mimicking Step 1 and deriving improved bounds on the sequence $\tcb{u_j}_{j=0}^\infty$. For $n\in\N$ we set $\xi^\nu_n=\tnorm{\Updelta_ng}_{F^{\be+\nu}}+2^{-n}\tnorm{g}_{F^{\be+\nu}}$. By arguing as in the first step, we can show that for any $0\le j\le k$ we have the estimate
		\begin{equation}\label{knock knock whose shoe showers with the dock for which the boat docks indeed if you want me too}
			\bnorm{\sum_{n=\ell}^kh_n}_{E^{\be+\nu}}\lesssim\bp{\sum_{n=\ell}^k\tp{\xi^\nu_n}^2}^{1/2}.
		\end{equation}
		Since $g\in F^{\be+\upnu}$, we deduce from~\eqref{smoothing_LP_2} that $\tcb{\xi^\nu_n}_{n=0}^\infty\in\ell^2(\N)$; hence, the sequence of partial sums $\tcb{\sum_{n=0}^jh_n}_{j=0}^\infty\subset E^{\be+\nu}$ is Cauchy. We have already established convergence of the scheme, giving that $u=\sum_{n=0}^\infty h_j$ in $E^\be$. Therefore, $u\in E^{\be+\nu}$ and, from~\eqref{knock knock whose shoe showers with the dock for which the boat docks indeed if you want me too}, we deduce that for all $j\in\N$
		\begin{equation}\label{improved bounds on u_j}
			\tnorm{u_j}_{E^{\be+\nu}}\lesssim\bp{\sum_{n=0}^{j-1}(\xi^\nu_n)^2}^{1/2}\lesssim\tnorm{g}_{F^{\be+\nu}}.
		\end{equation}
		
		Now we derive improved bounds on the sequence $\tcb{v_j}_{j\in\N}$. From properties~\eqref{smoothing_1} and~\eqref{smoothing_2} of the smoothing operators and the improved bounds of~\eqref{improved bounds on u_j}, we deduce for $\be+\nu\le s\in\j{N}$ that
		\begin{equation}\label{improved bounds on v_j}
			\tnorm{v_{j}}_{E^s}=\tnorm{S_ju_j}_{E^s}\lesssim 2^{j(s-(\be+\nu))}\tnorm{u_j}_{E^{\be+\nu}}\lesssim 2^{j(s-(\be+\nu))}\tnorm{g}_{F^{\be+\nu}}.
		\end{equation}
		
		Next, we bound the sequence $\tcb{u_j-v_j}_{j\in\N}$. Thanks to Lemma~\ref{lem on log-convexity in tame Banach scales}, for $s\in\N\cap[0,R]$ we have the bound 
		\begin{equation}\label{tight lipped}
			\tnorm{u_j-v_j}_{E^s}\lesssim\tnorm{u_j-v_j}_{E^0}^{1-s/R}\tnorm{u_j-v_j}_{E^R}^{s/R}.
		\end{equation}
		The $E^0$-norm term we bound using properties~\eqref{smoothing_1} and~\eqref{smoothing_3} of the smoothing operators and estimate~\eqref{improved bounds on u_j}:
		\begin{equation}\label{condescending}
			\tnorm{u_j-v_j}_{E^0}=\tnorm{(I-S_j)u_j}_{E^0}\lesssim2^{-j(\be+\nu)}\tnorm{u_j}_{E^{\be+\nu}}\lesssim 2^{-j(\be+\nu)}\tnorm{g}_{E^{\be+\nu}}.
		\end{equation}
		On the other hand, the $E^R$ term is handled via~\eqref{smoothing_1}, the telescoping identity  $u_j=\sum_{n=0}^{j-1}h_n$, induction hypothesis~\eqref{the_inductive_hypothesiesta}, and the bound $\xi^\nu_n\lesssim\tnorm{g}_{F^{\be+\nu}}$:
		\begin{equation}\label{mama's little chauvinist}
			\tnorm{u_j-v_j}_{E^R} \le \sum_{n=0}^{j-1}\tnorm{(I-S_j)h_n}_{E^R}  \lesssim \sum_{n=0}^{j-1}\tnorm{h_n}_{E^R} \lesssim \tnorm{g}_{F^{\be+\nu}}\sum_{n=0}^{j-1}2^{n(R-(\be+\nu))}
			\lesssim 2^{j(R-(\be+\nu))}\tnorm{g}_{F^{\be+\nu}}.
		\end{equation}
		We synthesize~\eqref{tight lipped}, \eqref{condescending}, and~\eqref{mama's little chauvinist} to get
		\begin{equation}\label{improved bound on u_j-v_j}
			\tnorm{u_j-v_j}_{E^s}\lesssim 2^{j(s-(\be+\nu))}\tnorm{g}_{F^{\be+\nu}}.
		\end{equation}
		
		Our next endeavor is to derive improved bounds on the sequence $\tcb{e_j}_{j\in\N}$. We turn to estimate~\eqref{it_6} for $s\in[r-\mu,R-\mu]\cap\N$. To handle the right hand side, we invoke the following bounds:
		\begin{equation}
			\begin{split}
				\max\tcb{\tnorm{u_j-v_j}_{E^r},\tnorm{h_j}_{E^r}}&\lesssim2^{j(r-\be)}\min\tcb{1,2^{-j\nu}\tnorm{g}_{F^{\be+\nu}}}, \\ 
				\max\tcb{\tnorm{u_j-v_j}_{E^{s+\mu}},\tnorm{h_j}_{E^{s+\mu}}}&\lesssim2^{j(s+\mu-\be)}\min\tcb{1,2^{-j\nu}\tnorm{g}_{F^{\be+\nu}}}, \\
				\tnorm{u_j,v_j}_{E^{s+\mu}}&\lesssim 2^{j\max\tcb{0,s+\mu-\be}},
			\end{split}
		\end{equation}
		which are consequences of~\eqref{it_0_h}, \eqref{it_0_u-v}, \eqref{it_0_gbnd}, \eqref{the_inductive_hypothesiesta}, \eqref{improved bound on u_j-v_j}, and finally~\eqref{it_7}. In this way we acquire the estimate
		\begin{equation}\label{improved bounds on e_j}
			\tnorm{e_j}_{F^s}\lesssim2^{j(s+\mu+r-2\be-\nu)}\tnorm{g}_{F^{\be+\nu}}.
		\end{equation}
		
		Now we estimate the sequence $\tcb{y_j}_{j\in\N}$. For this we recall identity~\eqref{_u_n_d_e_r_s_c_o_r_e_}. By arguing as in~\eqref{it_10_1}, but using the bounds established in~\eqref{improved bounds on e_j} for $\tcb{e_j}_{j\in\N}$, we learn that for $j\in\N$ and $r-\mu\le s\in\j{N}$
		\begin{equation}\label{of this}
			\tnorm{S_{j+1}e_{j}}_{F^s}\lesssim 2^{j(s+\mu+r-2\be-\nu)}\tnorm{g}_{F^{\be+\nu}}.
		\end{equation}
		Similarly, by arguing as in~\eqref{it_10_2}, but instead using~\eqref{improved bounds on e_j} with $s=R-\mu$ and the fact that $R+r-2\be-\nu>0$ (since $\nu<\upnu$), we gain the bound
		\begin{equation}\label{the one the only}
			\bnorm{\sum_{n=0}^{j-1}\Updelta_je_{n}}_{F^s}\lesssim 2^{j(s+\mu+r-2\be-\nu)}\tnorm{g}_{F^{\be+\nu}}.
		\end{equation}
		We combine~\eqref{of this} and~\eqref{the one the only} to see that for $j\in\N$ and $r-\mu\le s\in\j{N}$,
		\begin{equation}\label{improved bounds on y_j}
			\tnorm{y_j}_{F^s}\lesssim 2^{(j-1)(s+\mu+r-2\be-\nu)}\tnorm{g}_{F^{\be+\nu}}.
		\end{equation}
		
		At last, we are ready to obtain an improved estimate on the sequence $\tcb{h_j}_{j\in\N}$ and close the induction. By using the identity $h_j=L(v_j)(\Updelta_jg+y_j)$ with the right inverse estimates of equation~\eqref{tame estimates on the right inverse}, we find that for $j\in\N$,
		\begin{equation}\label{wah-wah you made me such a big star}
			\tnorm{h_j}_{E^s}\lesssim\tnorm{\Updelta_jg}_{F^s}+\tnorm{y_j}_{F^s}+\tbr{\tnorm{v_j}_{E^{s+\mu}}}\tp{\tnorm{\Updelta_jg}_{F^r}+\tnorm{y_j}_{E^r}} \text{ for } s\in[r,R]\cap\N.
		\end{equation}
		We estimate $\tnorm{y_j}_{F^s}$ and $\tnorm{y_j}_{F^r}$ according to the improved estimates of~\eqref{improved bounds on y_j}, the $\Updelta_jg$ norms are handled via smoothing estimate~\eqref{smoothing_4}, i.e. $\tnorm{\Updelta_jg}_{F^k}=2^{j(k-(\be+\nu+1))}\tnorm{\Updelta_jg}_{F^{\be+\nu+1}}$ for $k\in\tcb{r,s}$, and the $v_j$-term is estimated according to~\eqref{smoothing_1} and \eqref{it_0_u} in the case $s\le\be$ and via \eqref{it_0_v}. Hence estimate~\eqref{wah-wah you made me such a big star} yields
		\begin{multline}
			\tnorm{h_j}_{E^s}\lesssim 2^{j(s-(\be+\nu+1))}\tnorm{\Updelta_jg}_{F^{\be+\nu+1}}+2^{j(s+\mu+r-2\be-\nu)}\tnorm{g}_{F^{\be+\nu}}\\+\tbr{2^{j(s-\be)}}\tp{2^{j(r-(\be+\nu+1))}\tnorm{\Updelta_jg}_{F^{\be+\nu+1}}+2^{j(\mu+2r-2\be-\nu)}\tnorm{g}_{F^{\be+\nu}}}.
		\end{multline}
		Upon regrouping, we acquire the bound
		\begin{multline}\label{her mind can blow those clouds away}
			\tnorm{h_j}_{E^s}\lesssim 2^{j(s-(\be+\nu+1))}\tp{1+\tbr{2^{j(s-\be)}}2^{j(r-s)}}\tnorm{\Updelta_jg}_{F^{\be+\nu+1}}\\
			+2^{j(s-(\be+\nu+1))}2^{-j}\tp{2^{j(2+\mu+r-\be)}+\tbr{2^{j(s-\be)}}2^{j(2+\mu+2r-\be-s))}}\tnorm{g}_{F^{\be+\nu}}.
		\end{multline}
		Since $r\le\min\tcb{s,\be}$, we have that 
		\begin{equation}\label{its not always}
			1+\tbr{2^{j(s-\be)}}2^{j(r-s)}\lesssim1.
		\end{equation}
		The inequalities $1\le\mu+r$ and $2(\mu+r)<\be$ imply that
		\begin{equation}\label{gonna be this}
			2^{j(2+\mu+r-\be)}\lesssim 1.
		\end{equation}
		The inequality $\mu\ge 1$ implies that $\min\tcb{2\be,s+\be}\ge\be\ge1+2(\mu+r)\ge 2+\mu+2r$ and hence
		\begin{equation}\label{grey}
			\tbr{2^{j(s-\be)}}2^{j(2+\mu+2r-\be-s)}\lesssim1.
		\end{equation}
		We combine inequalities~\eqref{her mind can blow those clouds away}, \eqref{its not always}, \eqref{gonna be this}, and~\eqref{grey} to see that
		\begin{equation}
			\tnorm{h_j}_{E^s}\lesssim 2^{j(s-(\be+\nu+1))}\tp{\tnorm{\Updelta_jg}_{F^{\be+\nu+1}}+2^{-j}\tnorm{g}_{F^{\be+\nu}}}.
		\end{equation}
		This means that the inductive proposition~\eqref{the_inductive_hypothesiesta} has been verified for $\nu+1$, and hence the induction is complete.  We therefore know that~\eqref{the_inductive_hypothesiesta} holds for $\nu=\upnu$. We then argue as in~\eqref{knock knock whose shoe showers with the dock for which the boat docks indeed if you want me too} and~\eqref{improved bounds on u_j} with the sequence $\tcb{\xi^\upnu_n}_{n\in\N}$ to deduce that $u\in E^{\be+\upnu}$ with the estimate $\tnorm{u}_{E^{\be+\upnu}}\lesssim\tnorm{g}_{F^{\be+\upnu}}$. 
	\end{proof}
	
	We now complement the previous local surjectivity result with the following local injectivity result, which has no analog in Baldi and Haus \cite{MR3711883} but is analogous to a result of Hamilton~\cite{MR656198}.  We emphasize that in the following we only need left inverses and weaker forms of the tame estimates than in the LRI mapping hypotheses.
	
	\begin{thm}[Local injectivity]\label{local_inj}
		Let $\bf{E} = \{E^s\}_{s\in \j{N}}$ and $\bf{F} = \{F^s\}_{s\in \j{N}}$ be Banach scales over the same field, and suppose that $\bf{E}$ is terminally dense in the sense of Definition~\ref{definition of Banach scales}.   Let $\sigma,\mu \in \N$ be such that $\sigma+\mu \in \j{N}$.  Assume that there exists $\gamma_\sigma \in \R^+$ and a $C^2$ map $\Psi : B_{E^{\sigma+\mu}}(0,\gamma_\sigma) \to F^{\sig}$ such that the following hold.
		\begin{enumerate}
			\item For every $u_0 \in B_{E^{\sigma+\mu}}(0,\gamma_\sigma)$ we have the bound 
			\begin{equation}
				\norm{D^2 \Psi(u_0)[v,w] }_{F^\sigma} \lesssim \br{\norm{u_0}_{E^{\sigma+\mu}} } (\norm{v}_{E^{\sigma+\mu}} \norm{w}_{E^\sigma} + \norm{v}_{E^\sigma} \norm{w}_{E^{\sigma+\mu}}  ).
			\end{equation}
			
			\item  For every $u_0 \in B_{E^{\sigma+\mu}}(0,\gamma_\sigma) \cap E^N$ there exists a bounded linear operator $L(u_0) : F^\sigma \to E^\sigma$ satisfying the following two conditions.
			\begin{enumerate}
				\item $L(u_0) D\Psi(u_0) v = v$ for every $v \in E^{\sigma+\mu}$.
				\item We have the estimate $\norm{L(u_0) f}_{E^\sigma} \lesssim   \br{\norm{u_0}_{E^{\sigma+\mu}} } \norm{f}_{F^\sigma}$ for every $f \in F^\sigma$.
			\end{enumerate}
		\end{enumerate}
		Then there exists $0 < \delta_{\operatorname*{inj},\sig} \le \gamma_\sigma/4$ such that if $u_0-u_1 \in B_{E^{\sigma+\mu}}(0,2\delta_{\operatorname*{inj},\sig})$, with $u_i\in B_{E^{\sig+\mu}}(0,\gam_\sig)$ for $i\in\tcb{0,1}$, then 
		\begin{equation}\label{local_inj_0}
			\norm{u_1-u_0}_{E^\sigma} \lesssim \norm{\Psi(u_1)-\Psi(u_0)}_{F^\sigma}. 
		\end{equation}
		In particular, the restriction $\Psi : B_{E^{\sigma+\mu}}(0,\delta_{\operatorname*{inj},\sig}) \to F^\sigma$ is injective. 
	\end{thm}
	\begin{proof}

		Suppose initially that $u_0-u_1 \in B_{E^{\sigma+\mu}}(0,2\delta) \cap E^N$ and $u_i\in B_{E^{\sig+\mu}}(0,\gam_\sig)$ for some $0 < \delta \le \gam_\sig/4$.  We may use Taylor's theorem  to write
		\begin{equation}
			\Psi(u_1) - \Psi(u_0) = D \Psi(u_0)(u_1 - u_0) + \int_0^1(1-t) D^2 \Psi((1-t)u_0 + t u_1) [u_1-u_0,u_1-u_0]\;\m{d}t,
		\end{equation}
		with the understanding that this equality holds in the space $F^\sigma$.   We then apply the bounded linear map $L(u_0) : F^\sigma \to E^\sigma$ and rearrange to see that 
		\begin{equation}\label{local_inj_1}
			u_1 - u_0 = L(u_0)\bp{\Psi(u_1) - \Psi(u_0) -  \int_0^1(1-t) D^2 \Psi((1-t)u_0 + t u_1) [u_1-u_0,u_1-u_0]\;\m{d}t}.
		\end{equation}
		
		Next we couple the identity \eqref{local_inj_1} to the estimate for $L(u_0)$ and the trivial bound $\br{\norm{u_0}_{E^{\sigma+\mu}}}  \lesssim 1$ to see that
		\begin{equation}\label{local_inj_2}
			\norm{u_1 - u_0}_{E^\sigma} \lesssim \norm{\Psi(u_1) - \Psi(u_0)}_{F^\sigma} +   \int_0^1(1-t) \norm{D^2 \Psi((1-t)u_0 + t u_1) [u_1-u_0,u_1-u_0]}_{F^\sigma}\;\m{d}t. 
		\end{equation}
		For the latter term we use the $C^2$ estimate to bound
		\begin{multline}
			\norm{D^2 \Psi((1-t)u_0 + t u_1) [u_1-u_0,u_1-u_0]}_{F^\sigma} 
			\lesssim 
			\\
			\br{  \norm{(1-t)u_0 + t u_1}_{E^{\sigma+\mu}} } \norm{u_1-u_0}_{E^{\sigma+\mu}} \norm{u_1-u_0}_{E^\sigma}   
			\lesssim  \norm{u_1-u_0}_{E^{\sigma+\mu}} \norm{u_1-u_0}_{E^\sigma}
		\end{multline}
		for every $t \in [0,1]$. We then plug this bound into \eqref{local_inj_2} to see that 
		\begin{multline}
			\norm{u_1 - u_0}_{E^\sigma} \lesssim \norm{\Psi(u_1) - \Psi(u_0)}_{F^\sigma} +  \norm{u_1-u_0}_{E^{\sigma+\mu}} \norm{u_1-u_0}_{E^\sigma} \\
			\lesssim \norm{\Psi(u_1) - \Psi(u_0)}_{F^\sigma} +  \delta \norm{u_1-u_0}_{E^\sigma}.       
		\end{multline}
		From this we readily deduce the existence of $0 <\delta_{\operatorname*{inj}}\le \gamma_\sig/4$ such that if $\delta \le \delta_{\operatorname*{inj}}$ then  we can absorb the right-most term onto the left to conclude that 
		\begin{equation}\label{local_inj_3}
			\norm{u_1 - u_0}_{E^\sigma} \lesssim \norm{\Psi(u_1) - \Psi(u_0)}_{F^\sigma} \text{ for all }u_0-u_1 \in B_{\sigma+\mu}(0,2\delta_{\operatorname*{inj}}) \cap E^N.
		\end{equation}
		
		Estimate \eqref{local_inj_3} is not quite the desired result since it requires $u_0,u_1 \in E^N$, but we can use the fact that $\bf{E}$ is terminally dense to promote this result.  Indeed, given $u_0,u_1 \in B_{E^{\sigma+\mu}}(0,\gam_\sig)$ such that $u_0-u_1\in B_{E^{\sig+\mu}}(0,2\delta_{\operatorname*{inj}})$ we can pick 
		$\{u^{n}_i\}_{n\in \N} \subseteq E^N$ such that $u^{n}_i \to u_i$ in $E^{\sigma+\mu}$ for $i \in \{0,1\}$. Since $B_{E^{\sigma+\mu}}(0,2\delta_{\operatorname*{inj}})$ and $B_{E^{\sig+\mu}}(0,\gam_\sig)$ are open, we may assume without loss of generality, that $\{u^{n}_0-u^n_1\}_{n\in \N} \subseteq B_{E^{\sigma+\mu}}(0,2\delta_{\operatorname*{inj}})$ and for $i\in\tcb{0,1}$, $\tcb{u_i^n}_{n\in\N}\subset B_{E^{\sig+\mu}}(0,\gam_\sig)\cap E^N$.  We then apply \eqref{local_inj_3} to the sequence to see that $\norm{u_1^{n} - u_0^{n}}_{E^\sigma} \lesssim \norm{\Psi(u_1^{n}) - \Psi(u_0^{n})}_{F^\sigma}$ for all $n\in\N$. By sending $n \to \infty$ and using the continuity of $\Psi : B_{E^{\sigma+\mu}}(0,\gam_\sig) \to F^\sigma$,  we deduce that $\norm{u_1 - u_0}_{E^\sigma} \lesssim \norm{\Psi(u_1) - \Psi(u_0)}_{F^\sigma}$ for all $u_0-u_1 \in B_{E^{\sigma+\mu}}(0,2\delta_{\operatorname*{inj}})$, with $\;u_0,u_1\in B_{E^{\sig+\mu}}(0,\gam_\sig)$ which is \eqref{local_inj_0}.
	\end{proof}
	
	\subsection{Proof of the inverse function theorem}\label{what would you do if i sang}
	
	We are now ready for the proofs of the inverse function theorem. The proofs are given under conditions $I$ and $II$ separately.
	
	\begin{proof}[Proof of Theorem~\ref{thm on nmh}, assuming $I$]
		Assume that condition $I$ is satisfied, i.e. $\bf{E}$ and $\bf{F}$ are LP-smoothable.  
		
		First note that Lemma~\ref{lem on useful smoothing inequalities in tame Banach scales} implies that $\bf{E}$ is terminally dense.  This and the LRI mapping hypotheses from Definition~\ref{defn of the mapping hypotheses} imply the hypotheses of Theorem~\ref{local_inj} with $\sigma= \beta -\mu$ and  $\gamma_\sigma = \delta_{R} \le \delta_r.$  Let $\delta_{\operatorname*{inj},\sig}>0$ be the constant from Theorem \ref{local_inj}.  Then the theorem tells us that $\Psi : B_{E^{\beta}}(0,\delta_{\operatorname*{inj},\sig}) \to F^{\beta-\mu}$ is injective and obeys the estimate \eqref{local_inj_0}. 
		
		On the other hand, by hypothesis, we know that the inequality $2(r+\mu) < \beta < (r+R)/2$ is satisfied.  Let $K_1$, $K_2$, $K_3$, and $K_4$ be the constants from Theorem \ref{iteration_thm} and let $\ep_{\operatorname*{surj}}>0$ denote the constant appearing on the right side of \eqref{it_0_gbnd}.  Then Theorem \ref{iteration_thm} guarantees that for every $g \in B_{F^\beta}(0,\ep_{\operatorname*{surj}})$ there exists $u \in E^\beta$ satisfying $\Psi(u) =g$ as well as the bound 
		\begin{equation}\label{nmh_b_1}
			\norm{u}_{E^\beta} \le K_4 \norm{g}_{F^\beta}.
		\end{equation}
		
		Set $\ep = \min\{\ep_{\operatorname*{surj}},  \delta_{\operatorname*{inj},\sig} / K_4\}$, $\kappa_1 = K_4$, and $\kappa_2 >0$ to be the constant on the right side of \eqref{local_inj_0}.  Since $\kappa_1 \ep \le  \delta_{\operatorname*{inj},\sig}$, the above analysis shows that for every $g \in B_{F^\beta}(0,\ep)$ there exists a unique  $u \in B_{E^\beta}(0,\kappa_1 \ep)$ such that $\Psi(u) = g$.  Moreover, the induced map $\Psi^{-1} : B_{F^\beta}(0,\ep) \to \Psi^{-1}(B_{F^\beta}(0,\ep)) \cap B_{E^\beta}(0,\kappa_1 \ep)$ satisfies~\eqref{nmh_01} and~\eqref{hey cuz} in light of~\eqref{nmh_b_1} and~\eqref{local_inj_0}.
		
		Now suppose that $\N\ni\nu\le R+r-2\be$. From the fourth item of Theorem~\ref{iteration_thm}, we deduce~\eqref{derivatives_gain} and~\eqref{when the rain comes they run and hide}, i.e. that the inverse $\Psi^{-1}$ sends $B_{F^\be}(0,\ep)\cap F^{\be+\nu}$ to $E^{\be+\nu}$ with the tame estimate $\tnorm{\Psi^{-1}(g)}_{E^{\be+\nu}}\lesssim\tnorm{g}_{F^{\be+\nu}}$.
	\end{proof}

	A similar, but slightly more involved argument is needed to prove the inverse function theorem under assumption $II$.

	\begin{proof}[Proof of Theorem~\ref{thm on nmh}, assuming $II$]
		Assume that condition $II$ is satisfied, i.e. $\bf{E}$ is tame  and $\bf{F}$ is a direct summand of $\bf{E}$.  
		
		Since $\bf{E}$ is tame, there exists an LP-smoothable Banach scale $\bf{G} = \{G^s\}_{s\in\j{N}}$ such that $\bf{E}$ is a tame direct summand of $\bf{G}$ with lifting and restriction operators $\lambda_{\bf{E}}: E^0 \to G^0$ and $\rho_{\bf{E}} : G^0 \to E^0$.  Similarly, since $\bf{F}$ is a tame direct summand of $\bf{E}$ we can pick associated lifting and restriction operators $\lambda_{\bf{F}} : F^0 \to E^0$ and $\rho_{\bf{F}} : E^0 \to F^0$.  
		
		Before proceeding, we need to introduce three bits of notation.  First, we define  $\bf{H} = \bf{G}^3$, endowed with the $2-$norm for the sake of definiteness, which makes $\bf{H}$ into an LP-smoothable Banach scale thanks to Lemma~\ref{lemma on products of types of Banach scales}.  Next, for $s \in \j{N}$ write $Q_s =\norm{\rho_{\bf{E}}}_{\mathcal{L}(G^s;E^s)}$.  Then by construction we have that 
		\begin{equation}\label{nmhs_1}
			\rho_{\bf{E}}(B_{G^s}(0,\delta/Q_s)) \subseteq B_{E^s}(0,\delta) \text{ for all }s \in \j{N}, 0< \delta.
		\end{equation}
		Third, we define the map $\Phi : B_{H^r}(0,\delta_r/Q_r) \to G^{r-\mu}$ via 
		\begin{equation}
			\Phi(u,v,w) = \lambda_{\bf{E}}(  \lambda_{\bf{F}} \Psi(\rho_{\bf{E}} u)  +  (I-\lambda_{\bf{F}} \rho_{\bf{F}}) \rho_{\bf{E}}v ) + (I-\lambda_{\bf{E}} \rho_{\bf{E}}) w,
		\end{equation}
		which is well-defined in light of \eqref{nmhs_1} and the fact that $\Psi : B_{E^r}(0,\delta_r) \to F^{r-\mu}$.
		
		We now claim that $(\bf{H},\bf{G},\Phi)$ satisfy the RI mapping hypotheses with parameters $(\mu,r,R)$.  Clearly, $\Phi(0) =0$.  Now let $r-\mu \le s \in \j{N-\mu}$.  By the mapping properties of $\Psi$ and the lifting and restriction operators, we readily deduce that $\Phi : B_{H^r}(0,\delta_r/Q_r) \cap H^{s+1} \to G^s$ is $C^2$ and satisfies 
		\begin{equation}\label{nmhs_deriv_1}
			D \Phi(u_0,v_0,w_0)(a,b,c) =  \lambda_{\bf{E}}(  \lambda_{\bf{F}} D\Psi(\rho_{\bf{E}} u_0) \rho_{\bf{E}} a  +  (I-\lambda_{\bf{F}} \rho_{\bf{F}}) \rho_{\bf{E}} b ) + (I-\lambda_{\bf{E}} \rho_{\bf{E}}) c
		\end{equation}
		and
		\begin{equation}\label{nmhs_deriv_2}
			D^2 \Phi(u_0,v_0,w_0)[(a_1,b_1,c_1), (a_2,b_2,c_2) ] = \lambda_{\bf{E}} \lambda_{\bf{F}} D^2 \Psi(\rho_{\bf{E}} u_0)[\rho_{\bf{E}} a_1, \rho_{\bf{E}} a_2  ].
		\end{equation}
		In turn, the tame $C^2$ estimate for $D^2 \Psi$ and \eqref{nmhs_deriv_2} imply that
		\begin{multline}
			\norm{D^2 \Phi(u_0,v_0,w_0)[(a_1,b_1,c_1), (a_2,b_2,c_2) ]  }_{H^s}  \\ \le C'_1(s)[\norm{(a_1,b_1,c_1)}_{H^{s+\mu}} \norm{(a_2,b_2,c_2)}_{H^r} + \norm{(a_1,b_1,c_1)}_{H^r} \norm{(a_2,b_2,c_2)}_{H^{s+\mu}}  ] \\
			+ C'_1(s) \br{ \norm{(u_0,v_0,w_0)}_{H^{s+\mu}} } \norm{(a_1,b_1,c_1)}_{H^r} \norm{(a_2,b_2,c_2)}_{H^r}
		\end{multline}
		for  
		\begin{equation}\label{nmhs_map_2}
			C'_1(s) = C_1(s) \norm{\lambda_{\bf{E}} \lambda_{\bf{F}} }_{\mathcal{L}(F^s;G^s)} (1+ Q_{s+\mu}) (Q_r + Q_{r}^2).
		\end{equation}
		This proves the first and second items of the RI mapping hypotheses.
		
		To complete the proof of the claim, it remains to show that the third item of the RI mapping hypotheses is satisfied.  Let  $0 < \delta_R \le \delta_r$ be given by the RI mapping hypotheses for $(\bf{E},\bf{F},\Psi)$.  For $(u_0,v_0,w_0) \in B_{H^r}(0,\delta_R / Q_R) \cap H^N$ we then define the bounded linear operator $\Lambda(u_0,v_0,w_0) : G^s \to H^s$, for $r \le s \le R$,  via 
		\begin{equation}
			\Lambda(u_0,v_0,w_0) \xi = (\lambda_{\bf{E}} L(\rho_{\bf{E}} u_0 ) \rho_{\bf{F}} \rho_{\bf{E}} \xi, \lambda_{\bf{E}} \rho_{\bf{E}} \xi , \xi    ),
		\end{equation}
		which is well-defined since $\rho_{\bf{E}} u_0 \in B_{E^r}(0,\delta_R) \cap E^N$ whenever $(u_0,v_0,w_0) \in  B_{H^r}(0,\delta_R / Q_R) \cap H^N$.  We now use \eqref{nmhs_deriv_1} to verify that $\Lambda(u_0,v_0,w_0)$ is a right inverse of $D\Phi(u_0,v_0,w_0)$, using the identities $\rho_{\bf{F}} \lambda_{\bf{F}} = 1$ and $\rho_{\bf{E}} \lambda_{\bf{E}}=1$:   
		\begin{multline}
			D\Phi(u_0,v_0,w_0) L(u_0,v_0,w_0) \xi 
			= 
			\lambda_{\bf{E}}(  \lambda_{\bf{F}} D\Psi(\rho_{\bf{E}} u_0) \rho_{\bf{E}} \lambda_{\bf{E}} L(\rho_{\bf{E}} u_0 ) \rho_{\bf{F}} \rho_{\bf{E}} \xi  +  (I-\lambda_{\bf{F}} \rho_{\bf{F}}) \rho_{\bf{E}} \lambda_{\bf{E}} \rho_{\bf{E}} \xi ) \\ + (I-\lambda_{\bf{E}} \rho_{\bf{E}}) \xi  
			= 
			\lambda_{\bf{E}}\tp{  \lambda_{\bf{F}}  \rho_{\bf{F}} \rho_{\bf{E}} \xi  +  (I-\lambda_{\bf{F}} \rho_{\bf{F}})  \rho_{\bf{E}} \xi } + (I-\lambda_{\bf{E}} \rho_{\bf{E}}) \xi   =  \lambda_{\bf{E}} \rho_{\bf{E}}\xi +  (I-\lambda_{\bf{E}} \rho_{\bf{E}}) \xi = \xi
		\end{multline}
		for all $\xi \in G^r$.  Moreover, for $\xi \in G^s$, the tame estimate for $L(u_0)$ allows us to bound
		\begin{multline}
			\norm{\Lambda(u_0,v_0,w_0) \xi}_{H^s} = \norm{\lambda_{\bf{E}} L(\rho_{\bf{E}} u_0 ) \rho_{\bf{F}} \rho_{\bf{E}} \xi  }_{G^s} + \norm{\lambda_{\bf{E}} \rho_{\bf{E}} \xi }_{G^s} + \norm{\xi}_{G^s} \\
			\le   C_2(s) \norm{\lambda_{\bf{E}}}_{\mathcal{L}(E^s;G^s)} 
			\tp{ 
			\norm{\rho_{\bf{F}} \rho_{\bf{E}} \xi}_{F^s} 
			+  \br{ \norm{\rho_{\bf{E}} u_0}_{E^{s+\mu}} } \norm{\rho_{\bf{F}} \rho_{\bf{E}} \xi}_{F^r}
			}
			\\+\tbr{\norm{\lambda_{\bf{E}} \rho_{\bf{E}} }_{\mathcal{L}(G^s) }}  \norm{\xi}_{G^s}     
			\le  C_2'(s) \norm{\xi}_{G^s} +  \br{ \norm{(u_0,v_0,w_0)}_{H^{s+\mu}} } \norm{\xi}_{F^r},
		\end{multline}
		where the constant $C_2'(s)$ depends on $C_2(s)$  as well as on the quantities $\norm{\lambda_{\bf{E}} \rho_{\bf{E}} }_{\mathcal{L}(G^s) }$, $\norm{\lambda_{\bf{E}}  }_{\mathcal{L}(E^s;G^s) }$, $\norm{\rho_{\bf{F}} \rho_{\bf{E}} }_{\mathcal{L}(G^s;F^s)}$, and $\norm{\rho_{\bf{E}}}_{\mathcal{L}(G^{s+\mu};E^{s+\mu})}$.  Thus, the third item of the RI hypotheses is satisfied by the triple $(\bf{H},\bf{G},\Phi)$ with parameters $(\mu,r,R)$, and the claim is proved.  We emphasize, though, that we are not asserting that the LRI hypotheses are satisfied, as the left inverse condition fails in general for $D\Phi$ and $\Lambda$.
		
	With the claim in hand, we now consider $\beta = 2(r+\mu) + 1 \in \j{N}$ and  invoke Theorem \ref{iteration_thm} for the triple $(\bf{H},\bf{G},\Phi)$.  Let $K_1$, $K_2$, $K_3$, and $K_4$ be the constants from Theorem \ref{iteration_thm} and let $\ep'_{\operatorname*{surj}}>0$ denote the constant appearing on the right side of \eqref{it_0_gbnd}.  Then Theorem \ref{iteration_thm} guarantees that for every $\xi \in B_{G^\beta}(0,\ep'_{\operatorname*{surj}})$ there exists $(u',v',w') \in H^\beta = (G^\beta)^3$ satisfying $\Phi(u',v',w') =\xi$ as well as the bound 
		\begin{equation}\label{nmh_s_1}
			\tnorm{u'}_{G^\beta} + \tnorm{v'}_{G^\beta} + \tnorm{w'}_{G^\beta} = \tnorm{(u',v',w')}_{H^\beta} \le K_4 \tnorm{\xi}_{G^\beta}.
		\end{equation}
		Set $\ep_{\operatorname*{surj}} = \ep'_{\operatorname*{surj}} / \tnorm{\lambda_{\bf{E}} \lambda_{\bf{F}} }_{\mathcal{L}(F^\beta;G^\beta)}$ and note that 
		\begin{equation}
			\lambda_{\bf{E}} \lambda_{\bf{F}}(B_{F^\beta}(0,\ep_{\operatorname*{surj}} ) ) \subseteq B_{G^\beta}(0,\ep'_{\operatorname*{surj}})
		\end{equation}
		by construction.  Consequently, for any $g \in B_{F^\beta}(0,\ep_{\operatorname*{surj}} )$ there exists $(u',v',w') \in H^\beta$ such that $\Phi(u',v',w') = \lambda_{\bf{E}} \lambda_{\bf{F}} g$, which unravels to 
		\begin{equation}\label{twisted}
			\lambda_{\bf{E}}(  \lambda_{\bf{F}} \Psi(\rho_{\bf{E}} u')  +  (I-\lambda_{\bf{F}} \rho_{\bf{F}}) \rho_{\bf{E}}v' ) + (I-\lambda_{\bf{E}} \rho_{\bf{E}}) w' = \lambda_{\bf{E}} \lambda_{\bf{F}} g.
		\end{equation}
		Applying $\rho_{\bf{E}}$ and using the identity $\rho_{\bf{E}} \lambda_{\bf{E}}=1$, this implies the identity
		\begin{equation}
			\lambda_{\bf{F}} \Psi(\rho_{\bf{E}} u')  +  (I-\lambda_{\bf{F}} \rho_{\bf{F}}) \rho_{\bf{E}}v' = \lambda_{\bf{F}} g,
		\end{equation}
		to which we apply  $\rho_{\bf{F}}$ and using the identity $\rho_{\bf{F}} \lambda_{\bf{F}}=1$ to see that 
		\begin{equation}\label{shaken}
			\Psi(\rho_{\bf{E}}u') = g.
		\end{equation}
		Thus, if we set $u = \rho_{\bf{E}}u' \in E^{\beta}$, then $\Psi(u) =g$, and \eqref{nmh_s_1} implies that
		\begin{equation}\label{nmh_s_2}
			\norm{u}_{E^\beta} \le \kappa_1 \norm{g}_{F^\beta} 
		\end{equation}
		for $\kappa_1 = K_4 \norm{\rho_{\bf{E}}}_{\mathcal{L}(G^\beta;E^\beta) }  \norm{\lambda_{\bf{E}} \lambda_{\bf{F}}  }_{\mathcal{L}(F^\beta;G^\beta)}$, which in particular means that $u \in B_{E^\beta}(0,\kappa_1 \ep_{\operatorname*{surj}})$.
		
		On the other hand, Lemma~\ref{lem on useful smoothing inequalities in tame Banach scales}  implies that $\bf{E}$ is terminally dense.  This and the LRI mapping hypotheses imply that the hypotheses of Theorem \ref{local_inj} are satisfied by the triple $(\bf{E}, \bf{F},\Psi)$ with $\sigma =\beta -\mu$ and  $\gamma_\sigma = \delta_{R} \le \delta_r.$   Let $\delta_{\operatorname*{inj},\sig}>0$ be the constant from Theorem \ref{local_inj}.  Then the theorem tells us that $\Psi : B_{E^{\beta}}(0,\delta_{\operatorname*{inj},\sig}) \to F^{\beta-\mu}$ is injective and obeys the estimate \eqref{local_inj_0}. 
		
		Set $\ep = \min\tcb{\ep_{\operatorname*{surj}},  \delta_{\operatorname*{inj},\sig} / \kappa_1}$ and $\kappa_2 >0$ to be the constant on the right side of \eqref{local_inj_0}.  Since $\kappa_1 \ep \le  \delta_{\operatorname*{inj},\sig}$, the above analysis shows that for every $g \in B_{F^\beta}(0,\ep)$ there exists a unique  $u \in B_{E^\beta}(0,\kappa_1 \ep)$ such that $\Psi(u) = g$.  Moreover, the induced map $\Psi^{-1} : B_{F^\beta}(0,\ep) \to \Psi^{-1}(B_{F^\beta}(0,\ep)) \cap B_{E^\beta}(0,\kappa_1 \ep)$ satisfies \eqref{nmh_01} and \eqref{hey cuz} in light of \eqref{nmh_s_2} and \eqref{local_inj_0}.
		
		Finally, if we assume that $\N\ni\nu\le R+r-2\be$ and $g\in B_{F^\be}(0,\ep)\cap F^{\be+\nu}$, then we are assured by the fourth item of Theorem~\ref{iteration_thm} that there exists $(u',v',w')\in H^{\be+\nu}$ such that $\Phi(u',v',w')=\lambda_{\bf{E}}\lambda_{\bf{F}}g$. By unraveling as in~\eqref{twisted}--\eqref{shaken}, we find that for $u=\rho_{\bf{E}}u'\in E^{\be+\nu}$ we have that $\Psi(u)=g$ and $\tnorm{u}_{E^{\be+\nu}}\lesssim\tnorm{g}_{F^{\be+\nu}}$.
	\end{proof}
	
	\subsection{Refinements}\label{section on refinements}  In this subsection we aim to strengthen the conclusions of Theorem~\ref{thm on nmh}.  In particular, we will study the continuity and higher order smoothness of the inverse map $\Psi^{-1}$ provided by the theorem. First, we analyze the right and left linear inverse map $L$ by making an extension to backgrounds outside of the terminal space $E^N$, and then proving various continuity and differentiability assertions. Second, we return to the map $\Psi^{-1}$ and show a more refined continuity estimate than the basic assertion of~\eqref{hey cuz}.  Then we prove differentiability of $\Psi^{-1}$ and relate the derivative to the operator $L$.  Once this is done, we conclude by reading off higher regularity assertions. 
	
	We now enumerate further properties of the family of inverses $L$.
	
	\begin{thm}[Extension and regularity of the right and left inverse]\label{thm on extension and regularity of the right and left inverse}
		Under the LRI mapping hypotheses set forth in Definition~\ref{defn of the mapping hypotheses} and the additional assumptions that $r+\mu\le R$ and that the Banach scale $\tcb{E^s}_{s\in\j{N}}$ consists of reflexive spaces, we have that the following properties of $L$ hold.
		\begin{enumerate}
			\item Existence of extension: There exists a family of bounded linear maps $\Bar{L}:B_{E^r}(0,\del_R)\cap E^{r+\mu}\to\mathcal{L}(F^r;E^r)$ such that the following extension properties are satisfied.
			\begin{enumerate}
				\item $\Bar{L}=L$ on the subset $B_{E^r}(0,\del_R)\cap E^N$.
				\item If $g\in F^r$ and $u_0\in B_{E^r}(0,\del_R)\cap E^{r+\mu}$, then we have that $D\Psi(u_0)\Bar{L}(u_0)g=g$. Additionally, if $u\in E^{r+\mu}$ then $\Bar{L}(u_0)D\Psi(u_0)u=u$.
				\item For $s\in[r,R]\cap\N$ we have that $\Bar{L}:B_{E^r}(0,\del_R)\cap E^{s+\mu}\to\mathcal{L}(F^s;E^s)$ with the tame estimate $\tnorm{\Bar{L}(u_0)g}_{E^s}\lesssim\tnorm{g}_{F^s}+\tbr{\tnorm{u_0}_{E^{s+\mu}}}\tnorm{g}_{F^r}$.   \end{enumerate}
			
			\item Continuity: For $s\in[r,R-\mu]\cap\N$, if we view $\Bar{L}$ as mapping $\Bar{L}:(B_{E^r}(0,\del_R)\cap E^{s+\mu})\times F^s\to E^s$, then this map is continuous.
			
			\item Higher regularity: Assume that, for some $\N\ni\ell\ge 2$, the map $\Psi$ is $\ell$-times continuously Gateaux differentiable in the sense of Definition~\ref{definition of gateaux}.  If for $s\in[r,R-\ell\mu]\cap\N$ we view $\Bar{L}$ as a mapping
			\begin{equation}
				\Bar{L}:(B_{E^r}(0,\del_R)\cap E^{s+\ell\mu})\times F^{s+(\ell-1)\mu}\to E^s,
			\end{equation}
			then $\Bar{L}$ is $(\ell-1)$-times continuously Gateaux-differentiable.
		\end{enumerate}
	\end{thm}
	\begin{proof}

		We divide the proof into several steps.
		
		\textbf{Step 1}: Constructing $\Bar{L}$.  First, we provide a continuity estimate on the right and left inverse map $L$ (a priori only defined for $E^N$-backgrounds). Suppose that $u_0,w_0\in B_{E^r}(0,\del_R)\cap E^N$, $s\in[r,R-\mu]\cap\N$, and $g\in F^{s+\mu}$. We will estimate the difference
		\begin{multline}\label{believe me when I tell you}
			L(u_0)g-L(w_0)g=-L(u_0)(D\Psi(u_0)-D\Psi(w_0))L(w_0)g\\
			=-\int_0^1L(u_0)D^2\Psi((1-t)w_0+tu_0)(u_0-w_0,L(w_0)g)\;\m{d}t
		\end{multline}
		in $E^{s}$. By applying the tame estimates of~\eqref{C2 tame estimates} and~\eqref{tame estimates on the right inverse} repeatedly and the embedding inequalities of the first item of Definition~\ref{definition of Banach scales}, we get the somewhat crude estimate
		\begin{multline}\label{somewhat crude}
			\tnorm{L(u_0)g-L(w_0)g}_{E^s}\lesssim\tbr{\tnorm{u_0,w_0}_{E^{s+\mu}}}^3\tnorm{u_0-w_0}_{E^{s+\mu}}\tnorm{g}_{F^r}\\+\tbr{\tnorm{u_0,w_0}_{E^{s+\mu}}}\tnorm{u_0-w_0}_{E^r}\tp{\tnorm{g}_{F^{s+\mu}}+\tbr{\tnorm{w_0}_{E^{s+2\mu}}}\tnorm{g}_{F^r}}.
		\end{multline}
		While forgoing tameness, this estimate is strong enough to allow us to define our extension. 
		
		Indeed, let $u_0\in B_{E^r}(0,\del_R)\cap E^{r+\mu}$ and fix some $g\in F^r$. We claim that the sequence $\tcb{L(T_ju_0)T_jg}_{j=\ell}^\infty\subset E^r$ is Cauchy,  where the operators $\tcb{T_j}_{j=0}^\infty$ are from Lemma~\ref{lem on useful smoothing inequalities in tame Banach scales} and $\ell\in\N$ is the first index for which $T_ju_0\in B_{E^r}(0,\del_R)$ for all $\N\ni j\ge\ell$.  We verify this by first  estimating
		\begin{multline}\label{by sending k off into the subset}
			\tnorm{L(T_{j+k}u_0)T_{j+k}g-L(T_ju_0)T_jg}_{E^r}\le\tnorm{L(T_{j+k}u_0)(T_{j+k}g-T_jg)}_{F^r}\\+\tnorm{(L(T_{j+k}u_0)-L(T_ju_0))T_jg}_{E^r}=\bf{I}_{j,k}+\bf{II}_{j,k}.
		\end{multline}
		For $\bf{I}_{j,k}$ we simply apply estimate~\eqref{tame estimates on the right inverse} and Lemma~\ref{lem on useful smoothing inequalities in tame Banach scales}:
		\begin{equation}\label{in search of}
			\bf{I}_{j,k}\lesssim\tbr{\tnorm{T_{j+k}u_0}_{E^{r+\mu}}}\tnorm{T_{j+k}g-T_jg}_{F^r}\lesssim\tbr{\tnorm{u_0}_{E^{r+\mu}}}\tp{\bf{n}^r_j(g)+\bf{n}^r_{j+k}(g)}.
		\end{equation}
		Hence $\lim_{j,k\to\infty}\bf{I}_{j,k}=0$. On the other hand, for $\bf{II}_{j,k}$, we first employ estimate~\eqref{somewhat crude}:
		\begin{multline}\label{in search of those}
			\bf{II}_{j,k}\lesssim\tbr{\tnorm{T_{j+k}u_0,T_ju_0}_{E^{r+\mu}}}^3\tnorm{T_{j+k}u_0-T_j u_0}_{E^{r+\mu}}\tnorm{T_jg}_{F^r}\\
			+\tbr{\tnorm{T_{j+k}u_0,T_ju_0}_{E^{r+\mu}}}\tnorm{T_{j+k}u_0-T_ju_0}_{E^r}\tp{\tnorm{T_jg}_{F^{r+\mu}}+\tbr{\tnorm{T_{j}u_0}_{E^{r+2\mu}}}\tnorm{T_jg}_{F^r}}.
		\end{multline}
		Thanks again to Lemma~\ref{lem on useful smoothing inequalities in tame Banach scales}, we are free to make the following bounds:
		\begin{align}\label{mahi mahi}
			\tnorm{T_{j+k}u_0}_{E^{r+\mu}},\tnorm{T_ju_0}_{E^{r+\mu}}&\lesssim\tnorm{u_0}_{F^{r+\mu}}, &
			\tnorm{T_jg}_{F^r}&\lesssim\tnorm{g}_{F^r},\nonumber\\
			\tnorm{T_{j+k}u_0-T_ju_0}_{E^r}&\lesssim 2^{-j\mu}\tp{\bf{n}^{r+\mu}_{j}(u_0)+\bf{n}^{r+\mu}_{j+k}(u_0)}, & \tnorm{T_jg}_{F^{r+\mu}}&\lesssim 2^{j\mu}\tnorm{g}_{F^r},\nonumber\\
			\tnorm{T_{j+k}u_0-T_ju_0}_{E^{r+\mu}}&\lesssim\bf{n}^{r+\mu}_{j}(u_0)+\bf{n}^{r+\mu}_{j+k}(u_0)
			& \tnorm{T_ju_0}_{E^{r+2\mu}}&\lesssim 2^{j\mu}\tnorm{u_0}_{E^{r+\mu}}.
		\end{align}
		Upon combining~\eqref{in search of those} and~\eqref{mahi mahi}, we acquire the estimate
		\begin{equation}
			\bf{II}_{j,k}\lesssim\tbr{\tnorm{u_0}_{E^{r+\mu}}}^3\tnorm{g}_{F^r}\tp{\bf{n}^{r+\mu}_{j}(u_0)+\bf{n}^{r+\mu}_{j+k}(u_0)},
		\end{equation}
		and hence $\lim_{j,k\to\infty}\bf{II}_{j,k}=0$. We deduce that the sequence $\tcb{L(T_ju_0)T_jg}_{j=\ell}^\infty\subset E^r$ is Cauchy. Hence there exists 
		\begin{equation}
			\Bar{L}(u_0)g=\lim_{j\to\infty}L(T_ju_0)T_jg,
		\end{equation}
		and this defines $\Bar{L}$ as a family of linear maps.
		
		\textbf{Step 2}: Properties of $\Bar{L}$.  We now examine the restriction of $\Bar{L}$ to higher regularity (larger $s$) spaces in the scale. Suppose that $u_0\in B_{E^r}(0,\del_R)\cap E^{s+\mu}$ and that $g\in F^{s}$ for some $s\in[r,R]\cap \N$. For $\N\ni j$ sufficiently large we can apply estimate~\eqref{tame estimates on the right inverse} and again Lemma~\ref{lem on useful smoothing inequalities in tame Banach scales}  to see that
		\begin{equation}\label{before I am a frail old man}
			\tnorm{L(T_ju_0)T_jg}_{E^s}\lesssim\tnorm{T_jg}_{F^s}+
			\tbr{\tnorm{T_ju_0}_{E^{s+\mu}}}\tnorm{T_jg}_{F^r}\lesssim\tnorm{g}_{F^s}+\tbr{\tnorm{u_0}_{E^{s+\mu}}}\tnorm{g}_{F^r}.
		\end{equation}
		Hence, the sequence $\tcb{L(T_ju_0)T_jg}_{j=\ell}^\infty\subset E^s$ is bounded by the right hand expression above. The space $E^s$ is reflexive, and this sequence already converges in $E^r$, so the limit $\Bar{L}(u_0)g$ belongs to $E^s$ and has norm bounded above by the right side of~\eqref{before I am a frail old man} thanks to the weak sequential lower semicontinuity of the norm.

		We next prove that $\Bar{L}$ is a family of right and left inverses for $D\Psi$. First suppose that $g\in F^r$ and $u_0\in B_{E^r}(0,\del_{R})\cap E^{r+\mu}$. For $j \in \N$ sufficiently large, we have that 
		\begin{equation}\label{ziglet_the_piglet}
    		D\Psi(T_ju_0)L(T_ju_0)T_jg=T_jg. 		    
		\end{equation}
		Since $\Psi:B_{E^r}(0,\del_r)\cap E^{r}\to F^{r-\mu}$ is $C^2$, $\tcb{L(T_ju_0)T_jg}_{j=0}^\infty\subset E^r$ converges to $\Bar{L}(u_0)g$, and we have the convergences $T_ju_0\to u_0$ in $E^{r+\mu}\emb E^r$ and $T_jg\to g$ in $F^r$ as $j\to\infty$, we may send $j\to\infty$ in \eqref{ziglet_the_piglet} to see that $D\Psi(u_0)\Bar{L}(u_0)g=g$.
		
		On the other hand, if we assume that $u\in E^{r+\mu}$ and $u_0\in B_{E^r}(0,\del_R)\cap E^{r+\mu}$, then we  have 
		\begin{equation}\label{all i got lookin for was somebody who was lookin like you}
			L(T_ju_0)T_jD\Psi(u_0)u=u+L(T_ju_0)(T_jD\Psi(u_0)-D\Psi(T_ju_0))u.
		\end{equation}
		The left hand side converges in $E^r$ to $\Bar{L}(u_0)D\Psi(u_0)u$ as $j\to\infty$. For the right hand side we may estimate
		\begin{multline}
			\tnorm{L(T_ju_0)(T_jD\Psi(u_0)-D\Psi(T_ju_0))u}_{E^r}\lesssim\tbr{\tnorm{u_0}_{E^{r+\mu}}}\tnorm{(T_jD\Psi(u_0)-D\Psi(T_ju_0))u}_{F^r}\\
			\lesssim\tbr{\tnorm{u_0}_{E^{r+\mu}}}\tp{\tnorm{(I-T_j)D\Psi(u_0)u}_{F^r}+\tnorm{(D\Psi(u_0)-D\Psi(T_ju_0))u}_{F^r}},
		\end{multline}
		and the right hand side of this evidently converges to zero as $j\to\infty$ thanks again to the properties of $I-T_j$ from Lemma~\ref{lem on useful smoothing inequalities in tame Banach scales}, the continuity properties of $D\Psi$, and the fact that $u,u_0\in E^{r+\mu}$.  Hence by sending $j\to\infty$ in \eqref{all i got lookin for was somebody who was lookin like you} we learn that $\Bar{L}(u_0)D\Psi(u_0)u=u$.
		
		Now we are ready to prove that $\Bar{L}$ is actually an extension of $L$. Suppose that $u_0\in B_{E^r}(0,\del_R)\cap E^{N}$ and that $g\in F^r$. Then by the right and left inverse properties, for every $j\in\N$ we have the identity
		\begin{equation}
			\Bar{L}(u_0)T_jg=\Bar{L}(u_0)D\Psi(u_0)L(u_0)T_jg=L(u_0)T_jg.
		\end{equation}
		Upon sending $j\to\infty$ and using that $\Bar{L}(u_0)$ and $L(u_0)$ are bounded, we deduce that $\Bar{L}(u_0)g=L(u_0)g$.  This completes the proof of all of the assertions of the first item.

		\textbf{Step 3}: Continuity.  We now study continuity. We have established that for $s\in[r,R-\mu]\cap\N$, $u_0,w_0\in B_{E^r}(0,\del_R)\cap E^{s+2\mu}$, and $g\in F^{s+\mu}$ the sequence $\tcb{L(T_ju_0)T_jg-L(T_jw_0)T_jg}_{j=0}^\infty\subset E^{s+\mu}$ converges weakly up to a subsequence in $E^{s+\mu}$ and strongly in $E^r$ to the limit $\Bar{L}(u_0)g-\Bar{L}(w_0)g$. Therefore, upon invoking the log-convexity of Lemma~\ref{lem on log-convexity in tame Banach scales}, we obtain strong convergence in $E^s$ up to a subsequence.  By passing along this subsequence in~\eqref{somewhat crude} and then taking the supremum over $\tnorm{g}_{F^{s+\mu}}\le 1$, we obtain the Lipschitz estimate
		\begin{equation}\label{yes to the}
			\tnorm{\Bar{L}(u_0)-\Bar{L}(w_0)}_{\mathcal{L}(F^{s+\mu},E^s)}\lesssim\tbr{\tnorm{u_0,w_0}_{E^{s+2\mu}}}^3\tnorm{u_0-w_0}_{E^{s+\mu}}.
		\end{equation}	
		
		We use \eqref{yes to the} as an intermediate step in deriving the stated continuity in the second item of the Theorem statement. Let $u_0,w_0\in B_{E^r}(0,\del_R)\cap E^{s+\mu}$ and $g,h\in F^s$. For $j\in\N$ sufficiently large, we consider the difference
		\begin{multline}
			\Bar{L}(u_0)g-\Bar{L}(w_0)h=(\Bar{L}(u_0)g-\Bar{L}(T_ju_0)T_jg)\\+(\Bar{L}(T_ju_0)T_jg-\Bar{L}(T_jw_0)T_jh)+(\Bar{L}(T_jw_0)T_jh-\Bar{L}(w_0)h)=\bf{I}_j+\bf{II}_j+\bf{III}_j
		\end{multline}
		and estimate the three terms individually. For $\bf{I}_j$, we note that \eqref{by sending k off into the subset}, \eqref{in search of}, \eqref{in search of those}, and~\eqref{mahi mahi} hold with $r$ replaced by $s$ (by the same proof), and hence we may send  $k\to\infty$ to acquire the bound
		\begin{equation}\label{insert song lyric here that sounds weird out of contex so modify it to sound meaningless}
			\tnorm{\bf{I}_j}_{E^s}\lesssim\tbr{\tnorm{u_0}_{E^{s+\mu}}}^3\sp{\tnorm{g}_{F^s}\bf{n}_j^{s+\mu}(u_0)+\bf{n}^s_j(g)}.
		\end{equation}
		For $\bf{II}_j$ we simply apply the local Lipschitz estimate~\eqref{yes to the} and use properties of the operators $T_j$:
		\begin{equation}
			\tnorm{\bf{II}_j}_{E^s}\lesssim\tbr{2^{j\mu}\tnorm{g}_{F^s}}\tbr{2^{j\mu}\tnorm{u_0,w_0}_{E^{s+\mu}}}^3\tnorm{u_0-w_0,g-h}_{E^{s+\mu}\times F^s}.
		\end{equation}
		For $\bf{III}_j$, we begin by bounding as in~\eqref{insert song lyric here that sounds weird out of contex so modify it to sound meaningless}:
		\begin{equation}
			\tnorm{\bf{III}_j}_{E^s}\lesssim\tbr{\tnorm{w_0}_{E^{s+\mu}}}^3\sp{\tnorm{h}_{F^s}\bf{n}_j^{s+\mu}(w_0)+\bf{n}^s_j(h)}.
		\end{equation}
		Now, by Lemma~\ref{lem on useful smoothing inequalities in tame Banach scales}, we may estimate
		\begin{equation}
			\bf{n}_j^{s+\mu}(w_0)\lesssim\bf{n}_j^{s+\mu}(u_0)+\tnorm{u_0-w_0}_{E^{s+\mu}}
			\text{ and }
			\bf{n}_j^s(h)\lesssim\bf{n}_j^{s}(g)+\norm{g-h}_{F^s},
		\end{equation}
		and hence deduce that
		\begin{multline}
			\tnorm{\Bar{L}(u_0)g-\Bar{L}(w_0)h}_{E^s}\lesssim2^{3\mu j}\tbr{\tnorm{u_0,w_0}_{E^{s+\mu}},\tnorm{g,h}_{F^s}}^3\tnorm{u_0-w_0,g-h}_{E^{s+\mu}\times F^s}\\
			+\tbr{\tnorm{u_0,w_0}_{E^{s+\mu}},\tnorm{g,h}_{F^s}}^3\tp{\bf{n}_j^{s+\mu}(u_0)+\bf{n}^s_j(g)}.
		\end{multline}
		By taking $j$ large relative to $u_0$ and $g$ and then taking $w_0$ and $h$ sufficiently close $u_0$ and $g$ to we see that $\Bar{L}$ is continuous as a map from $E^{s+\mu}\times F^s$ to $E^s$, but not necessarily uniformly so.  This completes the proof of the second item.

		\textbf{Step 4}: Higher regularity.  Finally, we prove the third assertion. By arguing as in the derivation of \eqref{before I am a frail old man}, we learn that for $s\in[r,R-2\mu]\cap\N$, $w_0,h_0,w_0+h_0\in B_{E^r}(0,\del_R)\cap E^{s+2\mu}$,  $g\in F^{s+\mu}$, and $\tau\in(0,1)$ we have the decomposition
		\begin{equation}
			\tau^{-1}\tp{\Bar{L}(w_0+\tau h_0)g-\Bar{L}(w_0)g}+\Bar{L}(w_0)D^2\Psi(w_0)(h_0,\Bar{L}(w_0)g)=\bf{I}_\tau+\bf{II}_\tau
		\end{equation}
		where 
		\begin{equation}
			\bf{I}_\tau=-(\Bar{L}(w_0+\tau h_0)-\Bar{L}(w_0))\int_0^1D^2\Psi(w_0+\tau h_0)(h_0,\Bar{L}(w_0)g)\;\m{d}t
		\end{equation}
		and
		\begin{equation}
			\bf{II}_\tau=-\Bar{L}(w_0)\int_0^1(D^2\Psi(w_0+t\tau h_0)-D^2\Psi(w_0))(h_0,\bar{L}(w_0)g)\;\m{d}t.
		\end{equation}
  
		Since $\Psi$ is $\mu$-tamely $C^2$,  $\Bar{L}$ satisfies the continuity assertions of the second item, and  $\Bar{L}$ obeys the same tame estimates as $L$, it holds that $\tnorm{\bf{I}_\tau,\bf{II}_\tau}_{E^s} \to 0$ as $\tau\to0$. This proves that the map
		\begin{equation}
			\Bar{L}:(B_{E^r}(0,\del_R)\cap  E^{s+2\mu})\times F^{s+\mu}\to E^s
		\end{equation}
		is Gateaux differentiable with derivative given by
		\begin{equation}\label{florida}
			D\Bar{L}(u_0,g_0)[w,h]=-\Bar{L}(u_0)D^2\Psi(u_0)(w,L(u_0)g_0)+\Bar{L}(u_0)h
		\end{equation}
		for $u_0\in B_{E^r}(0,\del_R)\cap E^{s+2\mu}$, $w\in E^{s+2\mu}$, and $g_0,h\in F^{s+\mu}$. Now suppose that $\N\ni\ell\ge 3$ and that $\Psi$ is $\ell$-times continuously Gateaux differentiable. By a simple induction argument using the formula~\eqref{florida}, we find the remaining conclusions of the third item.
	\end{proof}

	Now that we have a refined understanding of the mapping properties of the family of right and left inverses $\Bar{L}$, we return to studying the local inverse map $\Psi^{-1}$, which we recall is granted by the conclusions of Theorem~\ref{thm on nmh}.  Indeed, we now prove Theorem \ref{thm on further conclusions of the inverse function theorem}.

	\begin{proof}[Proof of Theorem \ref{thm on further conclusions of the inverse function theorem}]
		Throughout the proof we will use the operator $\Bar{L}$ from Theorem~\ref{thm on extension and regularity of the right and left inverse}, which is an extension of the operator $L$ from the LRI mapping hypotheses. By a very mild abuse of notation, we write $L$ in place of $\Bar{L}$ in what follows.

		Under either hypothesis $I$ or $II$, the Banach scales $\bf{E}$ and $\bf{F}$ are tame, and so there exist smoothing operators $\tcb{T_j}_{j=0}^\infty$ as in Lemma~\ref{lem on useful smoothing inequalities in tame Banach scales}. We aim to establish that if $s\in[\be,R+r-\be-\mu)\cap\N$ and $g\in B_{F^\be}(0,\ep)\cap F^s$, then  $\Psi^{-1}(T_jg)\to\Psi^{-1}(g)$ in the space $E^{s}$ as $j\to\infty$. By Taylor expanding $\Psi$ at $T_jg$ to second order as in the proof of Lemma~\ref{local_inj} and using the left invertibility of $D\Psi(T_jg)$ by $L\circ\Psi^{-1}(T_jg)$, we arrive at the identity
		\begin{multline}
			\Psi^{-1}(T_{j+1}g)-\Psi^{-1}(T_jg)=L\circ\Psi^{-1}(T_jg)(T_{j+1}-T_j)g\\-L\circ\Psi^{-1}(T_jg)\int_0^1(1-t)D^2\Psi((1-t)\Psi^{-1}(T_jg)+t\Psi^{-1}(T_{j+1}g))(\Psi^{-1}(T_{j+1}g)-\Psi^{-1}(T_jg))^{\otimes 2}\;\m{d}t\\=\bf{I}_j+\bf{II}_j.
		\end{multline}
		We will estimate the right hand side of the above in the spaces $E^{s+\sig}$ for $\sig\in\tcb{-1,0,1}$. From the assumed tame structure, we have the estimate
		\begin{equation}
			\tnorm{L\circ\Psi^{-1}(T_jg)h}_{E^{s+\sig}}\lesssim\tbr{\tnorm{\Psi^{-1}(T_jg)}_{E^{s+\mu+\sig}}}\tnorm{h}_{F^r}+\tnorm{h}_{F^{s+\sig}}
			\lesssim\tbr{2^{j(\mu+\sig)}\tnorm{g}_{F^s}}\tnorm{h}_{F^r}+\tnorm{h}_{F^{s+\sig}}.
		\end{equation}
		For $h=(T_{j+1}-T_j)g$ we estimate
		\begin{equation}
			\begin{cases}
				\tnorm{(T_{j+1}-T_j)g}_{F^r}\lesssim 2^{j(r-s)}\bf{m}_j^{s}(g),\\
				\tnorm{(T_{j+1}-T_j)g}_{F^{s+\sig}}\lesssim 2^{j\sig}\bf{m}_j^s(g),
			\end{cases}
		\end{equation}
		and hence (since $\mu+r\le s$) $\tnorm{\bf{I}_j}_{E^{s+\sig}}\lesssim2^{j\sig}\tbr{\tnorm{g}_{F^s}}\bf{m}_j^s(g)$, where we recall that the seminorms $\bf{m}_j^s$ are from Lemma~\ref{lem on useful smoothing inequalities in tame Banach scales}. On the other hand, for some $t\in[0,1]$, we take 
		\begin{equation}
			h_t=D^2\Psi((1-t)\Psi^{-1}(T_jg)+t\Psi^{-1}(T_{j+1}g))(\Psi^{-1}(T_{j+1}g)-\Psi^{-1}(T_jg))^{\otimes 2}
		\end{equation}
		and estimate for $\ell\in\tcb{r,s+\sig}$
		\begin{multline}
			\tnorm{h_t}_{E^\ell}\lesssim\tbr{\tnorm{T_jg,T_{j+1}g}_{F^{\ell+\mu}}}\tnorm{\Psi^{-1}(T_{j+1}g)-\Psi^{-1}(T_jg)}_{E^{r}}^2\\+\tnorm{\Psi^{-1}(T_{j+1}g)-\Psi^{-1}(T_jg)}_{E^{r}}\tnorm{\Psi^{-1}(T_{j+1}g)-\Psi^{-1}(T_jg)}_{E^{\ell+\mu}}.
		\end{multline}
		If $j$ is sufficiently large, say $j\ge J(g)$, then we have that $T_jg,T_{j+1}g\in B_{F^\be}(0,\ep)$, and hence we can apply the estimate from the third conclusion of Theorem~\ref{thm on nmh}, to bound
		\begin{equation}\label{referenced}
			\begin{cases}
				\tnorm{\Psi^{-1}(T_{j+1}g)-\Psi^{-1}(T_jg)}_{E^r}\\
				\tnorm{\Psi^{-1}(T_{j+1}g)-\Psi^{-1}(T_jg)}_{E^{r+\mu}}
			\end{cases}\lesssim\tnorm{(T_{j+1}-T_j)g}_{E^{\be-\mu}}\lesssim 2^{j(\be-\mu-s)}\bf{m}_j^s(g).
		\end{equation}
		On the other hand, we trivially bound
		\begin{equation}
			\tnorm{\Psi^{-1}(T_{j+1}g)-\Psi^{-1}(T_jg)}_{E^{s+\sig+\mu}}\lesssim\tnorm{T_{j+1}g,T_jg}_{F^{s+\sig+\mu}}\lesssim 2^{j(\mu+\sig)}\tnorm{g}_{F^s}.
		\end{equation}
		Therefore, we get that
		\begin{equation}
			\tnorm{h_t}_{E^{s+\sig}}\lesssim2^{2j(\be-\mu-s)+j(\mu+\sig)}\tbr{\tnorm{g}_{F^s}}\bf{m}_j^s(g)^2+2^{j(\be+\sig-s)}\tnorm{g}_{F^s}\bf{m}_j^s(g)
			\lesssim2^{j\sig}\tbr{\tnorm{g}_{F^s}}^2\bf{m}_j^s(g)
		\end{equation}
		and $\tbr{2^{j(\mu+\sig)}\tnorm{g}_{F^s}}\tnorm{h_t}_{E^r}\lesssim2^{j\sig}\tbr{\tnorm{g}_{F^s}}^3\bf{m}_j^s(g)$. Upon synthesizing these bounds we arrive at the estimate $\tnorm{\bf{II}_j}_{E^{s+\sig}}\lesssim 2^{j\sig}\tbr{\tnorm{g}_{F^s}}^3\bf{m}_j^s(g)$,	and hence (for $\sig\in\tcb{-1,0,1}$ and $\N\ni j\ge J(g)$) we have
		\begin{equation}\label{the good estimate}
			\tnorm{\Psi^{-1}(T_{j+1}g)-\Psi^{-1}(T_jg)}_{E^{s+\sig}}\lesssim2^{j\sig}\tbr{\tnorm{g}_{F^s}}^3\bf{m}_j^s(g).
		\end{equation}
		
		We now use~\eqref{the good estimate} to prove that the sequence $\tcb{\sum_{j=J(g)}^n\tp{\Psi^{-1}(T_{j+1}g)-\Psi^{-1}(T_jg)}}_{n=J(g)}^\infty$ is Cauchy in $E^s$. Again, the Banach scales $\bf{E}$ and $\bf{F}$ are tame in the sense of Definition~\ref{defn of tame direct summands and tame Banach scales}, and so there exist  LP-smoothable Banach scales $\Bar{\bf{E}}$ and $\Bar{\bf{F}}$ with lifting and restriction pairs $\lambda_{\bf{E}},\rho_{\bf{E}}$ and $\lambda_{\bf{F}},\rho_{\bf{F}}$ witnessing the definition of tameness for $\bf{E}$ and $\bf{F}$, respectively. Let $\xi_j(g)=\lambda_{\bf{E}}(\Psi^{-1}(T_{j+1}g)-\Psi^{-1}(T_jg))\in\Bar{E}^{R+r-\be}$ and $\eta_j(g)=\tbr{\tnorm{g}_{F^s}}^3(S_{j+1}-S_j)\circ\lambda_{\bf{F}}g\in\Bar{F}^{R+r-\be}$. By the continuity of $\lambda_{\bf{E}}$, \eqref{the good estimate}, and properties of the LP-smoothing operators from Definition~\ref{defn of smoothable and LP-smoothable Banach scales}, we deduce that for $j,k\in\N$ with $j\ge J(g)$ we have the bound
		\begin{equation}
			\tnorm{\Updelta_k\xi_j(g)}_{\Bar{E}^{s}}\lesssim 2^{-|j-k|}\tnorm{\eta_j(g)}_{\Bar{F}^s}.
		\end{equation}
		Hence, we obtain the following estimates for $\ell,n\in\N$ with $\ell\ge J(g)$:
		\begin{equation}\label{the above}
			\bnorm{\sum_{j=\ell}^{\ell+n}\xi_j(g)}_{\Bar{E}^s}^2\lesssim\sum_{k=0}^\infty\bp{\sum_{j=\ell}^{\ell+n}\tnorm{\Updelta_k\xi_j(g)}_{\Bar{E}^s}}^2
			\lesssim\sum_{k=0}^\infty\bp{\sum_{j=\ell}^{\ell+n}2^{-|j-k|}\tnorm{\eta_j(g)}_{\Bar{F}^s}}^2
			\lesssim\sum_{j=\ell}^{\ell+n}\tnorm{\eta_j(g)}_{\Bar{F}^s}^2,
		\end{equation}
		where in the last inequality above we have employed Young's convolution inequality, as was done in~\eqref{yes very good at spelling}. Since $g\in F^s$, we have the inclusion $\tcb{\tnorm{\eta_j(g)}_{\Bar{F}^s}}_{j=0}^\infty\in\ell^2(\N)$ and hence from~\eqref{the above}, we deduce that $\tcb{\sum_{j=J(g)}^n\xi_j(g)}_{n=J(g)}^\infty\subset\Bar{E}^s$ is Cauchy. The map $\rho_{\bf{E}}$ is continuous and linear and thus by taking the image of this sequence we learn that
            \begin{equation}
			\bcb{\sum_{j=J(g)}^n(\Psi^{-1}(T_{j+1}g)-\Psi^{-1}(T_jg))}_{n=J(g)}^\infty=\tcb{\Psi^{-1}(T_{n+1}g)-\Psi^{-1}(T_{J(g)}g)}_{n=0}^\infty\subset E^s
		\end{equation}
		is Cauchy, as desired. On the other hand, the third conclusion of Theorem~\ref{thm on nmh} yields:
		\begin{equation}
			\tnorm{\Psi^{-1}(T_{n+1}g)-\Psi^{-1}(g)}_{E^{\be-\mu}}\lesssim\tnorm{(I-T_{n+1})g}_{F^{\be-\mu}}\to0\quad\text{as}\quad n\to\infty.
		\end{equation}
		Therefore, the limit in $E^s$ of the above sequence, which exists by the satisfaction of the Cauchy condition, necessarily is $\Psi^{-1}(g)$. By tracing back through the estimates, we find the following quantitative rate of convergence for $n\ge J(g)$:
		\begin{equation}\label{quantiative rate of convergence}
			\tnorm{\Psi^{-1}(T_{n+1}g)-\Psi^{-1}(g)}_{E^s}\lesssim\tbr{\tnorm{g}_{F^s}}^3\bp{\sum_{j=n}^\infty\bf{m}_j^s(g)^2}^{1/2}.
		\end{equation}

		We now have all the tools we need to establish that the map $\Psi^{-1}:B_{F^\be}(0,\ep)\cap F^s\to E^s$ is continuous for every $s\in[\be,R+r-\be-\mu)\cap\N$. Let $g,h,g+h\in B_{F^\be}(0,\ep)\cap F^s$, and assume that $\tnorm{h}_{F^s}$ is sufficiently small so that $J(g+h)\le J(g)+1$, where once more we let $J(g+h)$ denote the first index for which $j\ge J(g+h)$ implies that $T_j(g+h)\in B_{F^\be}(0,\ep)$. For any $\N\ni n\ge J(g)+1$, we may then estimate
		\begin{multline}
			\tnorm{\Psi^{-1}(g)-\Psi^{-1}(g+h)}_{E^s}\le\tnorm{\Psi^{-1}(T_{n+1}g)-\Psi^{-1}(g)}_{E^s}\\+\tnorm{\Psi^{-1}(T_{n+1}g)-\Psi^{-1}(T_{n+1}(g+h))}_{E^s}+\tnorm{\Psi^{-1}(T_{n+1}(g+h))-\Psi^{-1}(g+h)}_{E^s}\\
			=\bf{I}_n+\bf{II}_n+\bf{III}_n.
		\end{multline}
		For $\bf{I}_n$ and $\bf{III}_n$ we employ the quantitative rate of convergence~\eqref{quantiative rate of convergence} along with the inequality (which is true by Lemma~\ref{lem on useful smoothing inequalities in tame Banach scales})
		\begin{equation}
			\bp{\sum_{j=n}^\infty\bf{m}_j^s(g+h)^2}^{1/2}\lesssim\tnorm{h}_{F^s}+\bp{\sum_{j=n}^\infty\bf{m}_j^s(g)^2}^{1/2}
		\end{equation}
		to see that
		\begin{equation}
            \bf{I}_n+\bf{III}_n \lesssim \tbr{\tnorm{g,h}_{F^s}}^3 \bp{\tnorm{h}_{F^s} + \bp{\sum_{j=n}^\infty\bf{m}_j^s(g)^2}^{1/2}}.
		\end{equation}
		It remains to handle $\bf{II}_n$. As before, we have
		\begin{multline}\label{para exemplar}
			\Psi^{-1}(T_{n+1}(g+h))-\Psi^{-1}(T_{n+1}g)=L\circ\Psi^{-1}(T_{n+1}g)\Big(T_{n+1}h-\\\int_0^1(1-t)D^2\Psi((1-t)\Psi^{-1}(T_{n+1}g)+t\Psi^{-1}(T_{n+1}(g+h)))(\Psi^{-1}(T_{n+1}(g+h))-\Psi^{-1}(T_{n+1}g))^{\otimes 2}\;\m{d}t\Big)
		\end{multline}
		as a consequence of Taylor's theorem. Hence we have the estimate
		\begin{multline}
			\bf{II}_n\lesssim\tbr{2^{\mu n}\tnorm{g}_{F^s}}\big(\tnorm{h}_{F^s}+\\\tbr{2^{\mu n}\tnorm{g,h}_{F^s}}\tnorm{\Psi^{-1}(T_{n+1}(g+h))-\Psi^{-1}(T_{n+1}g)}_{E^{s+\mu}}\tnorm{\Psi^{-1}(T_{n+1}(g+h))-\Psi^{-1}(T_{n+1}g)}_{E^{r}}\big).
		\end{multline}
		We then crudely estimate
		\begin{equation}
			\tnorm{\Psi^{-1}(T_{n+1}(g+h))-\Psi^{-1}(T_{n+1}g)}_{E^{s+\mu}}\lesssim2^{n\mu}\tnorm{g,h}_{F^{s}}
		\end{equation}
		and also apply the third conclusion of Theorem~\ref{thm on nmh} to see that
		\begin{equation}
			\tnorm{\Psi^{-1}(T_{n+1}(g+h))-\Psi^{-1}(T_{n+1}g)}_{E^{r}}\le\tnorm{T_{n+1}h}_{F^{\be-\mu}}\lesssim\tnorm{h}_{F^s}.
		\end{equation}
		Synthesizing these bounds shows that $\bf{II}_n\lesssim\tbr{2^{\mu n}\tnorm{g,h}_{F^s}}^3\tnorm{h}_{F^s}$, and hence
		\begin{equation}
			\tnorm{\Psi^{-1}(g)-\Psi^{-1}(g+h)}_{E^s}\lesssim\tbr{\tnorm{g,h}_{F^s}}^3\bp{\tnorm{h}_{F^s}+\bp{\sum_{j=n}^\infty\bf{m}_j^s(g)^2}^{1/2}}+\tbr{2^{\mu n}\tnorm{g,h}_{F^s}}^3\tnorm{h}_{F^s}.
		\end{equation}
		By taking $n$ large relative to $g$ and then taking $h$ small, we see that the above estimate proves that $\Psi^{-1}$ is continuous at $g$. This completes the proof of the first item.

		We continue by studying the differentiability of the inverse map $\Psi^{-1}$. Suppose that $s\in[\be,R+r-\be-\mu)\cap\N$, $g,h,g+h\in B_{F^\be}(0,\ep)\cap F^s$. By expanding $\Psi$ to second order via Taylor's theorem with integral remainder and utilizing the left invertibility of $L$ (see e.g.~\eqref{para exemplar}), we derive the equality
		\begin{multline}
			\Psi^{-1}(g+h)-\Psi^{-1}(g)-L\circ\Psi^{-1}(g)h=\\
			- L\circ\Psi^{-1}(g)\int_0^1(1-t)D^2\Psi((1-t)\Psi^{-1}(g)+t\Psi^{-1}(g+h))(\Psi^{-1}(g+h)-\Psi^{-1}(g))^{\otimes 2}\;\m{d}t.
		\end{multline}
		By taking the norm of both sides in the space $E^{s-\mu}$, employing tame estimates, and utilizing the third item of Theorem~\ref{thm on nmh}, we find that 
		\begin{multline}\label{let it down let it down}
			\tnorm{\Psi^{-1}(g+h)-\Psi^{-1}(g)-L\circ\Psi^{-1}(g)h}_{E^{s-\mu}}\\\lesssim\tbr{\tnorm{g,h}_{F^s}}^2\tnorm{\Psi^{-1}(g+h)-\Psi^{-1}(g)}_{E^s}\tnorm{\Psi^{-1}(g+h)-\Psi^{-1}(g)}_{E^r}\\
			\lesssim\tbr{\tnorm{g,h}_{F^s}}^2\tnorm{\Psi^{-1}(g+h)-\Psi^{-1}(g)}_{E^s}\tnorm{h}_{F^{\be-\mu}}.
		\end{multline}
		Since $F^s\emb F^{\be-\mu}$ and we have already established that $\Psi^{-1}$ is continuous with respect to the  $F^s$ and $E^s$ topologies, estimate~\eqref{let it down let it down} shows that $\Psi^{-1}$ is differentiable at $g$ when viewed as a map from $F^s$  to $E^{s-\mu}$, and we have the derivative formula $D\Psi^{-1}(g)h=L\circ\Psi^{-1}(g)h$.  Moreover, since $\Psi^{-1}$ is continuous  relative to the $F^s$ and $E^s$ topologies, and $L$ is continuous relative to the  $E^s\times F^{s-\mu}$ and $E^{s-\mu}$ topologies, we find, by composition of continuity, that $D\Psi^{-1}$ is continuous with respect to the $F^s\times F^{s-\mu}$ and $E^{s-\mu}$ topologies.  This proves the second item.
		
		The third item now follows by pairing the final item of Theorem~\ref{thm on extension and regularity of the right and left inverse} and the derivative formula $D\Psi^{-1}=L\circ\Psi^{-1}$ with a simple induction argument.
	\end{proof}
	
	% - space - space - outer - % - space - space - outer - % - space - space - outer - % - space - space - outer - % - space - space - outer - % - space - space - outer - % - space - space - outer - % - space - space - outer - % - space - space - outer - % - space - space - outer - % - space - space - outer - % - space - space - outer -
	
	\section{Nonlinear analysis of traveling free boundary compressible Navier-Stokes}\label{icelandic arctic char}

    With our tools from nonlinear analysis now established in abstract form, we return to our main goal: the analysis of the traveling wave free boundary compressible Navier-Stokes equations, \eqref{The nonlinear equations in the right form}. In this section we verify most of the `nonlinear' hypotheses for our Nash-Moser inverse function theorem and make preparations for the linear analysis that follows in subsequent sections.
	
	We select a nonlinear mapping with tame Banach scales for the domain and codomain. In Section~\ref{section on scales of Banach spaces} we verify condition II of Theorem~\ref{thm on nmh}. To check that our rather complicated operator is $1$-tamely  $C^2$, we analyze each of the atomic nonlinearities individually in Section~\ref{only place he's ever been} and then synthesize in Section~\ref{section on smooth tameness of the nonlinear map} via the calculus of tame maps from Section~\ref{Michigan}. Once this is done, in Section~\ref{subsection on derivative splitting and notation for linear analysis} we then use our newly developed understanding of the nonlinear map to decompose its derivative into a principal part and a remainder term.   We then conclude in Section~\ref{subsection on notation and overview for linear analysis} with some preliminary results for our linear analysis.
	
	\subsection{Banach scales for the traveling wave problem}\label{section on scales of Banach spaces}
	
	We begin by defining some function spaces that will comprise our Banach scales.  First, we introduce spaces that will play a role in the domain of our nonlinear map.  For $s\in\tcb{-1,0}\cup\R^+$  we define
	\begin{equation}\label{Illinois}  \index{\textbf{Function spaces}!00@$\X_s$}
		\X_s=H^{1+s}(\Omega)\times H^{2+s}(\Omega;\R^n)\times\mathcal{H}^{5/2+s}(\Sigma),
	\end{equation}
	where $\mathcal{H}^{5/2+s}(\Sigma)$ denotes the anisotropic Sobolev space defined in \eqref{how has this not yet been labeled}, with $\Sigma$ identified with $\R^{n-1}$, and we endow $\X_s$ with the Hilbert norm
	\begin{equation}
		\tnorm{q,u,\eta}_{\X_s}=\sqrt{\tnorm{q}_{H^{1+s}}^2+\tnorm{u}_{H^{2+s}}^2+\tnorm{\eta}^2_{\mathcal{H}^{5/2+s}}}.
	\end{equation}
     We single out an important closed subspace of $\X_s$:
	\begin{equation}\label{domain banach scales}\index{\textbf{Function spaces}!10@$\X^s$}
		\X^s=\tcb{(q,u,\eta)\in\X_s\;:\;\m{Tr}_{\Sigma_0}(u)=0,\;\m{Tr}_{\Sigma}(u\cdot e_n)+\pd_1\eta=0}.
	\end{equation}
	
	Second, we define some spaces that will play a role in the codomain of our nonlinear map.  For $s\in\tcb{-1,0}\cup\R^+$ we define the space 
	\begin{equation}\label{Y_s_def} \index{\textbf{Function spaces}!20@$\Y_s$}
		\Y_s=\begin{cases}
			L^2(\Omega)\times({_0}H^1(\Omega;\R^n))^\ast&\text{if }s=-1,\\
			H^{1+s}(\Omega)\times H^s(\Omega;\R^n)\times H^{1/2+s}(\Sigma;\R^n)&\text{if }s \ge 0
		\end{cases}
	\end{equation}
	and endow it with the norm
	\begin{equation}
		\begin{cases}
			\tnorm{g,F}_{\Y_{s}}=\sqrt{\tnorm{g}^2_{L^2}+\tnorm{F}_{({_0}H^1)^\ast}^2}&\text{if }s=-1,\\
			\tnorm{g,f,k}_{\Y_s}=\sqrt{\tnorm{g}^2_{H^{1+s}}+\tnorm{f}^2_{H^s}+\tnorm{k}^2_{H^{1/2+s}}}&\text{if }s\ge 0.
		\end{cases}
	\end{equation}
	We also single out an important subspace of $\Y_s$:
	\begin{equation}\label{Y^s_def}\index{\textbf{Function spaces}!30@$\Y^s$}
		\Y^s=\begin{cases}
			\bcb{(g,F)\in\Y_s\;:\;\int_0^b g(\cdot,y)\;\m{d}y\in\dot{H}^{-1}(\Sigma)}&\text{if }s=-1,\\
			\bcb{(g,f,k)\in\Y_s\;:\;\int_0^b g(\cdot,y)\;\m{d}y\in\dot{H}^{-1}(\Sigma)}&\text{if }s \ge 0,
		\end{cases}
	\end{equation}
	which we endow with the norm
	\begin{equation}
		\begin{cases}
			\tnorm{g,F}_{\Y^{s}}=\sqrt{\tnorm{g,F}_{\Y_{-1}}^2+\tsb{\int_0^bg(\cdot,y)\;\m{d}y}^2_{\dot{H}^{-1}}}&\text{if }s=-1,\\
			\tnorm{g,f,k}_{\Y^s}=\sqrt{\tnorm{g,f,k}^2_{\Y_s}+\tsb{\int_0^bg(\cdot,y)\;\m{d}y}^2_{\dot{H}^{-1}}}&\text{if }s\ge 0.
		\end{cases}
	\end{equation}
	The spaces $\Y_s$ and $\Y^s$ are clearly complete and are Hilbert spaces in the case $s \ge 0$.

 \begin{rmk}\label{remarkable remark is named what else other than shark}
     In the notation of~\eqref{Rhode Island}, we have that $\Y^s=\hat{H}^{1+s}(\Omega)\times H^s(\Omega;\R^n)\times H^{1/2+s}(\Sigma;\R)$ for $\R\ni s\ge 0$, while $\Y^{-1}=\hat{H}^0(\Omega)\times({_0}H^{1}(\Omega;\R^n))^\ast$.
 \end{rmk}
 
 Third, we introduce spaces that will contain the stress and forcing data tuple $(T,G,F)$.  For $\R \ni s \ge 0$ we define the space
	\begin{equation}\index{\textbf{Function spaces}!40@$\W_s$}
		\W_s=H^{1+s}(\R^n;\R^{n\times n})\times H^{s}(\R^n;\R^n)\times H^s(\R^n;\R^n)
	\end{equation}
	and endow it with the norm
	\begin{equation}
		\tnorm{\mathcal{T},\mathcal{G},\mathcal{F}}_{\W_s}=\sqrt{\tnorm{\mathcal{T}}_{H^{1+s}}^2+\tnorm{\mathcal{G}}_{H^s}^2+\tnorm{\mathcal{F}}^2_{H^s}}.
	\end{equation}

     Finally, for $s\in\N$ we set
	\begin{equation}\label{North Carolina}
		\E^s=\X^s\times\W_{1+s}\times\R\index{\textbf{Function spaces}!50@$\E^s$}\quad\text{and}\quad\F^s=\Y^s\times\W_{1+s}\times\R\index{\textbf{Function spaces}!60@$\F^s$}
	\end{equation}
	and endow these with the norm from Remark~\ref{scale_prod_rmk}. 
	
	From these spaces we now define the following Banach scales:\index{\textbf{Function spaces}!200@$\pmb{\X}$}\index{\textbf{Function spaces}!201@$\pmb{\Y}$}\index{\textbf{Function spaces}!202@$\pmb{\W}$}\index{\textbf{Function spaces}!202@$\pmb{\E}$}\index{\textbf{Function spaces}!203@$\pmb{\F}$}
	\begin{align}\label{main_thm_banach_scales_def}
		\pmb{\X}&=\tcb{\X^s}_{s\in\N}, & \pmb{\Y}&=\tcb{\Y^s}_{s\in\N}, & \pmb{\W}&=\tcb{\W_{1+s}}_{s\in\N}, \nonumber\\
		\pmb{\E}&=\tcb{\E^s}_{s\in\N}, & \pmb{\F}&=\tcb{\F^s}_{s\in\N}.
	\end{align}
	It is a simple matter to check that these all satisfy the conditions required by Definition \ref{definition of Banach scales} to be Banach scales.
	
	In the next result we check the second condition in the statement of Theorem~\ref{thm on nmh} for these scales.
	
	\begin{lem}[Tameness of domain and codomain]\label{lem on tameness of domain and codomain}
		The Banach scale $\pmb{\E}$ is tame, and the Banach scale $\pmb{\F}$ is a tame direct summand of $\pmb{\E}$ in the sense of Definition~\ref{defn of tame direct summands and tame Banach scales}.
	\end{lem}
	\begin{proof}
		We begin by proving that $\pmb{\E}$ is a tame Banach scale. The scale $\pmb{\W}$ is tame in light of  Lemma \ref{lemma on products of types of Banach scales} and Example~\ref{example of Euclidean Soblev spaces}, as it is the product of scales of $L^2$-based Sobolev spaces on all of $\R^n$.  Since $\pmb{\W}$ is tame, this lemma also shows that it suffices to prove that $\R$ and $\pmb{\X} =\tcb{\X^s}_{s\in\N}$ are tame. $\R$ is handled by Example~\ref{example of a fixed Banach space}, so it remains only to handle  $\pmb{\X}$. 
		
		For $s \in \N$ write $\Bar{\X}_s=H^{1+s}(\R^n)\times H^{2+s}(\R^n;\R^{n})\times\mathcal{H}^{5/2+s}(\Sigma)$.  Thanks to Lemma~\ref{lem on lp smoothability of anisotropic Sobolev spaces}, Example~\ref{example of Euclidean Soblev spaces}, and Lemma~\ref{lemma on products of types of Banach scales}  the Banach scale $\tcb{\Bar{\X}_s}_{s\in\N}$ is LP smoothable.   We now show that $\tcb{\X^s}_{s\in\N}$ is tame by showing it is a tame direct summand of $\tcb{\Bar{\X}_s}_{s\in\N}$.   To this end, we define an auxiliary mapping $\mathcal{E}_1:H^{1/2+s}(\Sigma_0;\R^n)\times H^{1/2+s}(\Sigma)\to H^{1+s}(\Omega;\R^n)$ via $\mathcal{E}_1(\phi,\varphi)=w$, where $w$ is the unique solution to the PDE 
		\begin{equation}
			\begin{cases}
				-\Delta w=0&\text{in }\Omega,\\
				w=\phi&\text{on }\Sigma_0,\\
				(I-e_n\otimes e_n)w=0&\text{on }\Sigma,\\
				w\cdot e_n=\varphi&\text{on }\Sigma.
			\end{cases}
		\end{equation}
		Also, let $\mathfrak{E}_\Omega$  denote a Stein extension operator, in the sense of Definition~\ref{defn Stein-extension operator}, and write $\mathfrak{R}_\Omega$ for the linear operator corresponding to restriction from $\R^n$ to $\Omega$.   We then define the lifting operators $\lambda_{\X}:\X^s\to\Bar{\X}_s$ and the restriction operators $\rho_{\X}:\Bar{\X}_s\to\X^s$ via 
		\begin{equation}
			\lambda_{\X}(q,u,\eta)=(\mathfrak{E}_\Omega q,\mathfrak{E}_\Omega u,\eta)
			\text{ and }
			\rho_{\X}(q,u,\eta)=(\mathfrak{R}_\Omega q,\mathfrak{R}_\Omega u-\mathcal{E}_1(\m{Tr}_{\Sigma_0}u,\m{Tr}_{\Sigma}(u\cdot e_n)+\pd_1\eta),\eta).
		\end{equation}
		In light of Proposition \ref{proposition on spatial characterization of anisobros}, it is straightforward to check that $\lambda_\X$ and $\rho_\X$ indeed define bounded linear maps between their stated domains and codomains.  Additionally, we have that $\rho_{\X}\lambda_{\X}=\m{id}_{\X^0}$,  which completes the proof that $\tcb{\X^s}_{s\in\N}$ is a tame direct summand of $\tcb{\Bar{\X}_s}_{s\in\N}$, and hence is tame.  In turn, this completes the proof that $\pmb{\E}$ is tame.
		
		We continue by showing that $\pmb{\F}$ is a tame direct summand of $\pmb{\E}$. For this, it is clearly sufficient to prove that $\pmb{\Y}$ is a tame direct summand of $\pmb{\X}$.  In fact, we have the stronger result that these spaces are isomorphic.  To see this, consider the map $\Gamma:\X^s\to\Y^s$  defined by $\Gamma(q,u,\eta)=(g,f,k)$, where
		\begin{equation}
			\begin{cases}
				\grad\cdot u=g&\text{in }\Omega,\\
				-\pd_1u+\grad(q+\eta)-\Delta u=f&\text{in }\Omega,\\
				-(q-\mathbb{D}u)e_n-\Delta_{\|}\eta e_n=k&\text{on }\Sigma,\\
				u\cdot e_n+\pd_1\eta=0&\text{on }\Sigma,\\
				u=0&\text{on }\Sigma_0.
			\end{cases}
		\end{equation}
		This is well-defined by virtue of Proposition \ref{proposition on spatial characterization of anisobros}.  The proof of Theorem 6.6 in Leoni and Tice~\cite{leoni2019traveling} establishes that this map is a Banach isomorphism, although the theorem itself states the isomorphism between slightly different spaces, which is needed in~\cite{leoni2019traveling} due to a change of unknown made by taking a particular linear combination of $q$ and $\eta$.
	\end{proof}
	
	\subsection{Analysis of atomic nonlinearities}\label{only place he's ever been}
	
	In this subsection we put to use the abstract analysis of smooth tame structures from Section~\ref{Michigan} and the tools for verifying smoothness and tame estimates from Appendices~\ref{subsection on analysis of superposition nonlinearities} and~\ref{appendix on tools for tame estimates} in the study of the nonlinear expressions appearing in the equations under consideration, namely system~\eqref{The nonlinear equations in the right form}.	The majority of the results in this subsection are a recasting of the tame calculus estimates from Appendix~\ref{appendix on tools for tame estimates} in the language of Section~\ref{Michigan}. As such, some of the proofs are not much more than mere referrals; it is the statements that are important as we move forward. Once these are established, we conclude this subsection by considering two nonlinearities that play a distinguished role in our analysis of~\eqref{The nonlinear equations in the right form}. These aid in understanding the mapping properties of the continuity equation, which is rather subtle.

    The Banach scales we work with in this subsection are built from combinations of the following atomic scales: 
	\index{\textbf{Function spaces}!300@$\bf{H}$}\index{\textbf{Function spaces}!301@$\bm{\mathcal{H}}$}\index{\textbf{Function spaces}!302@$\bf{W}$}\index{\textbf{Function spaces}!303@$\bm{\mathcal{H}}_{(\kappa)}$}\index{\textbf{Function spaces}!304@$\bf{W}\bm{\mathcal{H}}_{(\kappa)}$}
	\begin{align}
		\bf{H}(U;V) &=\tcb{H^s(U;V)}_{s\in\N},& \bm{\mathcal{H}}(\R^d)&=\tcb{\mathcal{H}^s(\R^{d})}_{s\in\N},\nonumber\\
		\bf{W}(U;V)&=\tcb{W^{s,\infty}(U;V)}_{s\in\N},  & \bf{W}\bm{\mathcal{H}}_{(\kappa)}&=\tcb{W^{s,\infty}((0,b);\mathcal{H}^{s}_{(\kappa)}(\R^d))}_{s\in\N},\nonumber\\ 
        \bm{\mathcal{H}}_{(\kappa)}(\R^d)&=\tcb{\mathcal{H}_{(\kappa)}^s(\R^d)}_{s\in\N},
	\end{align}
	where $\kappa\in\R^+$, $d\in\N^+$, $U\subseteq\R^d$ is a Stein extension domain (see Definition~\ref{defn Stein-extension operator}), $V$ is a finite dimensional vector space over $\R$, and the spaces $\mathcal{H}^s$, $\mathcal{H}_{(\kappa)}^s$ are defined in Appendix~\ref{appendix on anisotropic Sobolev spaces}. We will often suppress the domain and codomain in this notation when they are clear from context. It is a simple matter to check that these are indeed Banach scales.  
	
	We begin by looking at products. Recall the spaces of (strongly)-tame maps introduced in Definition~\ref{defn of smooth tame maps}.
	
	\begin{lem}[Smooth tameness of products, 1]\label{lemma on tameness of products 1}
		Let $V_1,V_2$ and $W$ be real finite dimensional vector spaces, $B:V_1\times V_2 \to W$ be a bilinear map, and $U\subseteq\R^d$ be a Stein-extension domain (see Definition~\ref{defn Stein-extension operator}). The following inclusions hold when viewing $B$ as a product for functions taking values in $V_1$ and $V_2$.
		\begin{enumerate}
			\item $B\in {_{\m{s}}}T^\infty_{0,1+\tfloor{d/2}}(\bf{H}(U;V_1)\times\bf{H}(U;V_2);\bf{H}(U;W))$.
			\item $B\in {_{\m{s}}}T^\infty_{0,0}(\bf{H}(U;V_1)\times\bf{W}(U;V_2);\bf{H}(U;W))$.
		\end{enumerate}
	\end{lem}
	\begin{proof}
		By using the universality of tensor products to factor $B=L\circ\otimes$ for a linear map $L:V_1\otimes V_2\to W$ and then working component-wise, we see that this is a direct application of the high-low product estimates of Corollary~\ref{corollary on tame estimates on simple multipliers}.
	\end{proof}
	
	We now handle more complicated products also involving sums.
	
	\begin{lem}[Smooth tameness of products, 2]\label{lem on smooth tameness of products 2}
		Let $m\in\N^+$, $V$, $W$, and $V_1,\dots,V_m$ be real finite dimensional vector spaces, $B:V_1\times\cdots\times V_m\times V\to W$ be $(m+1)$-linear, and $U\subseteq\R^d$ be a Stein extension domain (see Definition~\ref{defn Stein-extension operator}). The $(m+1)$-linear map $P_B$ defined via $P_B\tp{(g_j,\psi_j)_{j=1}^m,\varphi}=B(g_1+\psi_1,\dots,g_m+\psi_m,\varphi)$ satisfies the inclusion
		\begin{equation}
			P_B\in{_{\m{s}}}T^\infty_{0,1+\tfloor{d/2}}\bp{\prod_{j=1}^m\tp{\bf{W}(U;V_j)\times\bf{H}(U;V_j)}\times\bf{H}(U;V);\bf{H}(U;W)}.
		\end{equation}
	\end{lem}
	\begin{proof}
		We proceed via induction on $m\in\N^+$. The base case, $m=1$, is a simple application of Lemma~\ref{lemma on tameness of products 1}. Now suppose that $m\in\N^+$ and the result holds for $m$.  We wish to prove it for $m+1$, so let $B:V_1\times\cdots V_m\times V_{m+1}\times V\to W$ be an $(m+2)$-linear map on some finite dimensional real vector spaces. We claim that there exists a bilinear map $B_1:V_1\times\tp{V_2\otimes\cdots\otimes V_{m+1}\otimes V}\to W$ and an $(m+1)$-linear map $B_2:V_2\times\cdots\times V_{m+1}\times V\to V_2\otimes\cdots \otimes V_{m+1}\otimes V$ such that
		\begin{equation}\label{factoring of B}
			B(v_1,\dots,v_{m+1},v)=B_1(v_1,B_2(v_2,\dots,v_{m+1},v)).
		\end{equation}
		Indeed, by the universality property of tensor products there exists a linear map $L:V_1\otimes\cdots\otimes V_{m+1}\otimes V\to W$ defined via
		\begin{equation}
			 L(v_1\otimes\cdots\otimes v_{m+1}\otimes v)=B(v_1,\dots,v_{m+1},v),
		\end{equation}
		so we set $B_1(v_1,\tau)=L(v_1\otimes \tau)$ and $B_2(v_2,\dots,v_{m+1},v)=v_2\otimes\cdots\otimes v_{m+1}\otimes v$. We then apply the induction hypothesis to the operator $P_{B_2}$ to see that
		\begin{equation}
			P_{B_2}\in{_{\m{s}}}T^\infty_{0,1+\tfloor{d/2}}\bp{\prod_{j=2}^{m+1}\tp{\bf{W}(U;V_j)\times\bf{H}(U;V_j)}\times\bf{H}(U;V);\bf{H}(U;V_2\otimes\cdots\otimes V_{m+1}\otimes V)}
		\end{equation}
		Similarly, we apply the base case to the operator $P_{B_1}$ and acquire the inclusion
		\begin{equation}
			P_{B_1}\in{_{\m{s}}}T^{\infty}_{0,1+\tfloor{d/2}}(\bf{W}(U;V_1)\times\bf{H}(U;V_1)\times\bf{H}(U;V_2\otimes\cdots\otimes V_{m+1}\otimes V);\bf{H}(U;W)).
		\end{equation}
		Equation~\eqref{factoring of B} implies that $P_B\tp{(g_j,\psi_j)_{j=1}^{m+1},\varphi}=P_{B_1}\tp{(g_1,\varphi_1),P_{B_2}\tp{(g_j,\psi_j)_{j=2}^{m+1},\varphi}}$. Hence, the inclusion $P_B\in{_{\m{s}}}T^\infty_{0,1+\tfloor{d/2}}$ is a consequence of the composition of smooth tame maps, Lemma~\ref{lem on composition of tame maps}.
	\end{proof}

	We also consider products with members of the anisotropic Sobolev spaces, which are defined in Appendix~\ref{appendix on anisotropic Sobolev spaces}.

	\begin{lem}[Smooth tameness of products, 3]\label{lem onsmooth tameness of products 3}
		The following hold.
		\begin{enumerate}
			\item  Let $r_d \in \N^+$ be as defined in the second item of Proposition~\ref{proposition on algebra properties}.  We have that the map $\prod_{j=1}^{r_d}\mathcal{H}^0_{(1)}(\R^d)\ni(\eta_1,\dots,\eta_{r_d})\mapsto \prod_{j=1}^{r_d}\eta_j\in H^0(\R^d)$ belongs to ${_{\m{s}}}T^\infty_{0,0}(\prod_{j=1}^{r_d}\bm{\mathcal{H}}_{(1)};\bf{H})$.
			
			\item For any $\chi\in W^{\infty,\infty}(0,b)$ and $\upnu\in\N^+$ we have that the map $\prod_{j=1}^{\upnu}\mathcal{H}^0_{(1)}(\R^d)\ni(\eta_1,\dots,\eta_\upnu)\mapsto\chi\prod_{j=1}^\upnu\eta_j\in W^{0,\infty}((0,b);\mathcal{H}^0_{(\upnu)}(\R^d))$ belongs to ${_{\m{s}}}T^\infty_{0,0}(\prod_{j=1}^\nu\bm{\mathcal{H}}_{(1)};\bf{W}\bm{\mathcal{H}}_{(\nu)})$.
		\end{enumerate}
	\end{lem}
	\begin{proof}
		The fact that the maps in the statement are well-defined and smooth follows from multilinearity and Proposition~\ref{proposition on algebra properties}. 
		Note that for $s \in \N$, the product map of the first item actually smoothly maps  $\prod_{j=1}^{r_d}\mathcal{H}^0_{(1)}(\R^d)\to H^s(\R^d)$
		and the map of the second item smoothly maps  $\prod_{j=1}^{\upnu}\mathcal{H}^0_{(1)}(\R^d)\to W^{s,\infty}((0,b);\mathcal{H}^s_{(\upnu)}(\R^d))$.  Then the strong tame estimates on all of the derivatives follow from Remark~\ref{remark on purely multlinear maps} and these improved mapping properties. 
	\end{proof}
	
	The next key nonlinear structure we handle is superposition as a Sobolev multiplier.

	\begin{lem}[Smooth tameness of superposition multipliers]\label{lem on smooth tameness of superposition multipliers}
		Let $V_1$, $V_2$, $W$, and $Z$ be real finite dimensional vector spaces, and $B:V_1\times V_2\to W$ be a bilinear map.  Let $U\subseteq\R^n$ and $O\subseteq Z$ be Stein extension domains (see Definition~\ref{defn Stein-extension operator}), and  $f\in C^\infty_b(O;V_1)$.  Finally, suppose that $Y\subseteq W^{2+\tfloor{n/2},\infty}(U;Z)\times H^{2+\tfloor{n/2}}(U;Z)$ is an open set with the property that if $(g,\psi)\in Y$, then $(g+\psi)(U)\subseteq O$. Then the nonlinear operator $P$ defined via $P(g,\psi,\varphi)=B(f(g+\psi),\varphi)$ belongs to
		\begin{equation}
			{_{\m{s}}}T^\infty_{0,2+\tfloor{n/2}}(Y\times H^{2+\tfloor{n/2}}(U;V_2),\bf{W}(U;Z)\times\bf{H}(U;Z)\times\bf{H}(U;V_2);\bf{H}(U;W)).
		\end{equation}
	\end{lem}
	\begin{proof}
		By working component-wise and using the universality of tensor products, we see that as a consequence of Theorem~\ref{comp_C2_fixed_variant} and the smoothness of the Sobolev multiplier times Sobolev to Sobolev pairing, for any $s\ge 2+\tfloor{n/2}$ the map
		\begin{equation}\label{can you peel me some garlic}
			P:(Y\times H^{2+\tfloor{n/2}}(U,V_2))\cap\tp{W^{s,\infty}(U;Z)\times H^s(U;Z)\times H^s(U;V_2)}\to H^s(U;W)
		\end{equation}
		is well-defined and smooth. Thus, we need only check that the derivatives of $P$ obey the required tame estimates. Consider first the case of zero derivatives, which we refer to as the base case. Employing the third item of Corollary~\ref{coro on tame estimates on superposition multipliers}, working component-wise, and using the universality of tensor products, we readily deduce that the map in~\eqref{can you peel me some garlic} belongs to the space ${_{\m{s}}}T^0_{0,2+\tfloor{n/2}}$.

		Now for $k\in\N^+$ we consider the case of $k$-derivatives of $P$. A short computation reveals the derivative formula
		\begin{equation}
			D^kP(g,\psi,\varphi)(g_j,\psi_j,\varphi_j)_{j=1}^k=B(D^kf(g+\psi)(g_j+\psi_j)_{j=1}^k,\varphi)+\sum_{i=1}^kB\tp{D^{k-1}f(g+\psi)(g_j+\psi_j)_{j\neq i},\varphi_i}.
		\end{equation}
		Again by the universality of tensor products, there are linear maps
  \begin{equation}
      L^k:\mathcal{L}^k(Z;V_1)\otimes\bp{\bigotimes_{j=1}^k Z}\otimes V_2 \to W\text{ and }L_{k-1}:\mathcal{L}^{k-1}(Z;V_1)\otimes\bp{\bigotimes_{j=1}^{k-1} Z}\otimes V_2 \to W
  \end{equation}
(depending only on $k$ and $B$) such that
		\begin{equation}\label{operator k}
			B(D^kf(g+\psi)(g_j+\psi_j)_{j=1}^k,\varphi)=L^k\bp{D^kf(g+\psi)\otimes\bigotimes_{j=1}^k(g_j+\psi_j)\otimes\varphi}
		\end{equation}
		and
		\begin{equation}\label{operator k-1}
			B\tp{D^{k-1}f(g+\psi)(g_j+\psi_j)_{j\neq i},\varphi_i}=L_{k-1}\bp{D^{k-1}f(g+\psi)\otimes\bigotimes_{j\neq i}(g_j+\psi_j)\otimes\varphi_i}.
		\end{equation}
		We then find that the operators in~\eqref{operator k} and~\eqref{operator k-1} belong to ${_{\m{s}}}T^0_{0,2+{\tfloor{n/2}}}$ by applying the base case, the tameness of composition from Lemma~\ref{lem on composition of tame maps}, and the second version of tameness of products, Lemma~\ref{lem on smooth tameness of products 2}.  This shows that, indeed, $D^kP$ belongs to ${_{\m{s}}}T^0_{0,2+{\tfloor{n/2}}}$ for every $k \in \N$.
	\end{proof}
	
	The next nonlinear structure we examine is superposition.

	\begin{lem}[Smooth tameness of superposition]\label{lem on smooth tameness of superposition}
		Let $\N\ni k\ge 2+\tfloor{n/2}$. The following hold.
		\begin{enumerate}
			\item Let $V$, $W$ be real finite dimensional vector spaces, $0\in O\subseteq V$ be a Stein extension domain (see Definition~\ref{defn Stein-extension operator}) that is star shaped with respect to the origin, $U\subseteq\R^n$ be a Stein extension domain, and $f\in C^\infty(O;W)$ be such that $f(0)=0$ and $Df\in C^\infty_b(O;\mathcal{L}(V;W))$.  Let $Y\subseteq H^{2+\tfloor{n/2}}(U;V)$ be an open set with the property that if $\varphi\in Y$, then $\varphi(U)\subseteq O$. Then the operator $\varphi\mapsto f(\varphi)$
			belongs to ${_{\m{s}}}T^\infty_{0,2+\tfloor{n/2}}(Y,\bf{H}(U;V);\bf{H}(U;W))$.
			\item Suppose that $\Phi\in C^\infty(\R^n;\R^n)$ is a bi-Lipschitz homeomorphism and a $C^1$ diffeomorphism such that $D\Phi\in C^\infty_b(\R^n;\R^{n\times n})$. There exists an open set $0\in U_{\Phi}\subseteq W^{2+\tfloor{n/2},\infty}(\R^n;\R^n)\times H^{2+\tfloor{n/2}}(\R^n;\R^n)$, depending only on $\Phi$ and the dimension, such that for any $m\in\N$, the operator $P$ defined via $P(f,g,h)=f(\Phi+g+h)$
			satisfies
			\begin{equation}
				P\in{_{\m{s}}}T^m_{0,2+\tfloor{n/2}}(H^{m+2+\tfloor{n/2}}(\R^n;\R^d)\times U_\Phi,\bf{I}_m;\bf{H}(\R^n;\R^d)),
			\end{equation}
			where $\bf{I}_m=\tcb{H^{m+s}(\R^n;\R^d)\times W^{s,\infty}(\R^n,\R^n)\times H^s(\R^n;\R^n)}_{s\in\N}$.
		\end{enumerate}
	\end{lem}
	\begin{proof}
		We begin by proving the first item. Note that by using the fundamental theorem of calculus, the map from the first item has the equivalent formula
		\begin{equation}
			\varphi\mapsto f(\varphi)=\int_0^1Df(t\varphi)(\varphi)\;\m{d}t.
		\end{equation}
		For each $t\in[0,1]$, the map $\varphi\mapsto Df(t\varphi)(\varphi)$ is seen to be ${_{\m{s}}}T^\infty_{0,2+\tfloor{n/2}}$ as a consequence of our previous result on the smooth tameness of superposition multipliers, Lemma~\ref{lem on smooth tameness of superposition multipliers}, and the composition of smooth tame maps, Lemma~\ref{lem on composition of tame maps}. In fact, it is a simple matter to verify that we actually have uniformity of the defining inequalities with respect to $t$ as well as the satisfaction of the remaining hypotheses of Lemma~\ref{lemma on addition of smooth tame maps}. Hence, the first item follows by applying the lemma.
		
		We next consider the second item. Theorem~\ref{comp_C2} establishes the existence of an open set $U_{\Phi}\subseteq W^{2+\tfloor{n/2},\infty}(\R^n;\R^n)\times H^{2+\tfloor{n/2}}(\R^n;\R^n)$ such that for every $m\in\N$ and $s\ge 2+\tfloor{n/2}$ the map
		\begin{equation}
			P:H^{m+s}(\R^n;\R^d)\times\tp{U_\Phi\cap\tp{W^{s,\infty}(\R^n;\R^n)\times H^s(\R^n;\R^n)}}\to H^s(\R^n;\R^d)
		\end{equation}
		is well-defined and $C^m$.  Thus we need only verify that the derivatives of $P$, which are enumerated in~\eqref{comp_C2_02}, are ${_{\m{s}}}T^0_{0,2+\tfloor{n/2}}$. The case of the $0^{\m{th}}$ derivative ($m=0$) follows immediately from the second item of Corollary~\ref{coro on tame estimates on superposition} and the bound $\tnorm{\det D(g+h)^{-1}}_{L^\infty}\le 2^n\tnorm{D\Phi^{-1}}_{L^\infty}^n$  for all $(g,h)\in U_\Phi$, which holds as a consequence of Lemma~\ref{comp_bilip_lem} and Theorem~\ref{comp_C2}.  When $m>0$, the formula \eqref{comp_C2_02} shows that the $m^{\m{th}}$ derivative is built from simple products and the $0^{\m{th}}$ derivative structure.  Consequently, the result follows in this case by supplementing this observation with Lemmas~\ref{lem on composition of tame maps} and~\ref{lem on smooth tameness of products 2}.
	\end{proof}
	
	The next two results, which happen to be the most subtle part of the nonlinear analysis, deal with superposition-like nonlinearities whose argument contains a member of the anisotropic Sobolev space. Recall that Appendix~\ref{appendix on anisotropic Sobolev spaces} gives the notation and an enumeration of basic properties for these specialized spaces. In the following statement $r_d\in \N^+$, for $d=n-1$, refers to the number defined in the second item of Proposition~\ref{proposition on algebra properties}.

	\begin{lem}[Taylor expansion trick]\label{lem on taylor expansion trick}
		Let $I\subseteq\R$ be an open interval containing $[-\mathfrak{g}b,0]$ and $\phi\in C^\infty(I)$. There exists a $\rho_I\in\R^+$,
		\begin{equation}
			\mathcal{N}_\phi^{(1)}\in {_{\m{s}}}T^\infty_{0,2+\tfloor{n/2}}(B_{H^{2+\tfloor{n/2}}}(0,\rho_I)\times B_{\mathcal{H}^{2+\tfloor{n/2}}}(0,\rho_I),\bf{H}(\Omega)\times\bm{\mathcal{H}}(\Sigma);\bf{H}(\Omega))\index{\textbf{Nonlinear maps}!80@$\mathcal{N}_\phi^{(1)}$},    
		\end{equation}
		and
		\begin{equation}
			\mathcal{N}_{\phi}^{(2)}\in {_{\m{s}}}T^\infty_{0,0}(\bm{\mathcal{H}}_{(1)}(\Sigma);\bf{W}\bm{\mathcal{H}}_{(r_{n-1}-1)})\index{\textbf{Nonlinear maps}!85@$\mathcal{N}_\phi^{(2)}$}
		\end{equation}
		such that $\mathcal{N}_\phi^{(1)}(0,0)=0$, $\mathcal{N}^{(2)}_\phi(0)=0$, and
		\begin{equation}\label{taylor expansion trick}
			\mathcal{N}_\phi(q,\eta)-\mathcal{N}_\phi(0,0)=\mathcal{N}^{(1)}_\phi(q,\eta)+\mathcal{N}^{(2)}_\phi(\Uppi^{1}_{\m{L}}\eta),
		\end{equation}
		where the projectors $\Uppi$ are defined in~\eqref{notation for the Fourier projection operators} and we write
		\begin{equation}\label{definition of N}
			\mathcal{N}_\phi(q,\eta)=\phi(-\mathfrak{g}\m{id}_{\R^n}\cdot e_n+q+\mathfrak{g}(I-\mathcal{E})\eta) 
			\text{ and }
			\mathcal{N}_\phi(0,0)=\phi(-\mathfrak{g}\m{id}_{\R^n}\cdot e_n)\index{\textbf{Nonlinear maps}!75@$\mathcal{N}_\phi$}.
		\end{equation}
	\end{lem}
	\begin{proof}
		We begin by choosing $\rho_I$. Set $R_I=\m{dist}(\R\setminus I,[-\mathfrak{g}b,0])/2\in(0,\infty]$. Thanks to the supercritical Sobolev embeddings, properties of anisotropic Sobolev from Proposition~\ref{proposition on frequency splitting}, and the boundedness of the extension operator $\mathcal{E}$ from Lemma~\ref{lem on mapping properties of the Poisson extension operator variants}, we see that the map
		\begin{equation}
			H^{2+\tfloor{n/2}}(\Omega)\times\mathcal{H}^{2+\tfloor{n/2}}(\Sigma) \ni (q,\eta)\mapsto q+\mathfrak{g}(I-\mathcal{E})\eta  \in L^\infty(\Omega)  
		\end{equation}
		is well-defined and continuous. Hence, the preimage of $B_{L^\infty}(0,R_I)$ under this map is an open set $U_I$ containing the origin. We take $\rho_I=\m{dist}(0,\pd U_I)/\sqrt{2}$ so that
		\begin{equation}
			B_{H^{2+\tfloor{n/2}}}(0,\rho_I)\times B_{\mathcal{H}^{2+\tfloor{n/2}}}(0,\rho_I)\subseteq U_I.
		\end{equation}

		Now let $(q,\eta)\in B_{H^{2+\tfloor{n/2}}}(0,\rho_I)\times B_{\mathcal{H}^{2+\tfloor{n/2}}}(0,\rho_I)$. Note that $\mathcal{N}_{\phi}(q,\eta)$, as defined in~\eqref{definition of N}, is well-defined as a function from $\Omega$ to $\R$.  We write
		\begin{equation}
			\mathcal{N}_\phi(q,\eta)=\phi(g_0+g+\varphi),
		\end{equation}
		where $g_0=-\mathfrak{g}\m{id}_{\R^n}\cdot e_n$, $g=\mathfrak{g}(I-\m{id}_{\R^n}\cdot e_n/b)\Uppi^1_{\m{L}}\eta$, and $\varphi=q+\mathfrak{g}(I-\mathcal{E})\Uppi^1_{\m{H}}\eta$. Hence,
		\begin{equation}
			\mathcal{N}_\phi(q,\eta)-\mathcal{N}_\phi(0,0)=(\phi(g_0+g+\varphi)-\phi(g_0+g))+(\phi(g_0+g)-\phi(g_0))=\bf{I}+\bf{II}.
		\end{equation}
		For $\bf{I}$ we use the fundamental theorem of calculus to express
		\begin{equation}
			\bf{I}=\varphi\int_0^1\phi'(g_0+g+\tau\varphi)\;\m{d}\tau,
		\end{equation}
		whereas for $\bf{II}$ we use Taylor's formula with integral remainder to write
		\begin{equation}
			\bf{II}=\sum_{j=1}^{r_d-1}\f{\phi^{(j)}(g_0)}{j!}g^j+g^{r_d}\int_0^1\f{(1-\tau)^{r_d-1}}{(r_d-1)!}\phi^{(r_d)}(g_0+\tau g)\;\m{d}\tau,
		\end{equation}
		where $r_d = r_{n-1} \ge 1$ is given by \eqref{the definition of rd} and sums over empty intervals are understood as zero.      This leads us to define
		\begin{equation}
			\mathcal{N}^{(1)}_{\phi}(q,\eta)=\varphi\int_0^1\phi'(g_0+g+\tau\varphi)\;\m{d}\tau+g^{r_d}\int_0^1\f{(1-\tau)^{r_d-1}}{(r_d-1)!}\phi^{(r_d)}(g_0+\tau g)\m{d}\tau
		\end{equation}
		and
		\begin{equation}\label{absolute definition of N2 no questions about it}
			\mathcal{N}^{(2)}_\phi(\Uppi^1_{\m{L}}\eta)=\sum_{j=1}^{r_d-1}\f{\phi^{(j)}(g_0)}{j!}g^j.
		\end{equation}
		The computation leading up to these definitions shows that decomposition~\eqref{taylor expansion trick} holds and that $\mathcal{N}^{(1)}_\phi$ and $\mathcal{N}_\phi^{(2)}$ vanish at the origin.
		
		Now we analyze the tame smoothness of $\mathcal{N}^{(1)}_\phi$ and $\mathcal{N}^{(2)}_\phi$. The former is ${_{\m{s}}}T^\infty_{0,2+\tfloor{n/2}}$ thanks to the results on addition of smooth tame maps, Lemma~\ref{lemma on addition of smooth tame maps}, smooth tameness of superposition multipliers, Lemma~\ref{lem on smooth tameness of superposition multipliers}, composition of smooth tame maps, Lemma~\ref{lem on composition of tame maps}, and the third version of smooth tameness of products, Lemma~\ref{lem onsmooth tameness of products 3}. On the other hand, the map $\mathcal{N}_\phi^{(2)}$ is ${_{\m{s}}}T^\infty_{0,0}$ as a consequence of Lemma~\ref{lem onsmooth tameness of products 3}.
	\end{proof}
	
	We now use the previous result to understand the mapping properties of the vector field argument of the divergence in the continuity equation of~\eqref{The nonlinear equations in the right form}.
	
	\begin{lem}[A vector field decomposition]\label{lem on fundamental decomposition of continuity equation-like vector fields}
		Let $I\subseteq\R$ be an open interval containing $[-\mathfrak{g}b,0]$, $\phi\in C^\infty(I)$, and $\rho_I$ be as in Lemma~\ref{lem on taylor expansion trick}. Let the open set $O_I$ be given by $B_{H^{2+\tfloor{n/2}}}(0,\rho_I)\times H^{2+\tfloor{n/2}}(\Omega;\R^n)\times B_{\mathcal{H}^{3+\tfloor{n/2}}}(0,\rho_I)$, and consider the Banach scale  
		\begin{equation}
		\bf{J}  =\tcb{H^s(\Omega)\times H^s(\Omega;\R^n)\times\mathcal{H}^{1+s}(\Sigma)}_{s\in\N}.    
		\end{equation}
		 There exist
		\begin{equation}
			\mathcal{M}_{\phi}^{(1)}\in{_{\m{s}}}T^{\infty}_{0,2+\tfloor{n/2}}(O_I,\bf{J};\bf{H}(\Omega;\R^n))\index{\textbf{Nonlinear maps}!95@$\mathcal{M}^{(1)}_\phi$}
		\end{equation}
		and
		\begin{equation}\label{mapping properties of the second piece in teh decomposition}
			e_1\cdot\mathcal{M}^{(2)}_{\phi}\in{_{\m{s}}}T^\infty_{0,0}(\bm{\mathcal{H}}_{(1)}(\Sigma);\bf{W}\bm{\mathcal{H}}_{(r_{n-1})})\index{\textbf{Nonlinear maps}!100@$e_1\cdot\mathcal{M}_\phi^{(2)}$}
		\end{equation}
		such that  $\mathcal{M}^{(1)}_{\phi}(0,0,0)=0$, $e_1\cdot\mathcal{M}^{(2)}_\phi(0)=0$,
		\begin{equation}\label{respect the trace or get the claw}
			\m{Tr}_{\pd\Omega}(\mathcal{M}^{(1)}_{\phi}(q,u,\eta)\cdot e_n)=\m{Tr}_{\pd\Omega}(\mathcal{N}_\phi(q,\eta))\sp{(\m{Tr}_{\Sigma}(u\cdot e_n)+\pd_1\eta)\mathds{1}_{\Sigma}+\m{Tr}_{\Sigma_0}(u\cdot e_n)\mathds{1}_{\Sigma_0}},
		\end{equation}
		and
		\begin{equation}\label{fundamental decomposition of vector fields}
			\mathcal{M}_\phi(q,u,\eta)-\mathcal{M}_\phi(0,0,0)=\mathcal{M}^{(1)}_\phi(q,u,\eta)+e_1\cdot\mathcal{M}_\phi^{(2)}(\Uppi^1_{\m{L}}\eta)e_1,
		\end{equation}
		where we write
		\begin{equation}  \index{\textbf{Nonlinear maps}!90@$\mathcal{M}_\phi$}
			\mathcal{M}_\phi(q,u,\eta)=\mathcal{N}_\phi(q,\eta)(u- M_\eta e_1)
			\text{ and }
			\mathcal{M}_\phi(0,0,0)=-\mathcal{N}_{\phi}(0,0)e_1
		\end{equation}
		for $\mathcal{N}_\phi(q,\eta)$, $M_{\eta}$, and $\Uppi$  given by~\eqref{definition of N}, \eqref{Mississippi}, and~\eqref{notation for the Fourier projection operators}, respectively.
	\end{lem}
	\begin{proof}
		The precise form of the decomposition \eqref{fundamental decomposition of vector fields}, which uses Lemma~\ref{lem on taylor expansion trick} as well as the identities $M_\eta e_1 = (1+\partial_n \mathcal{E} \eta)e_1 - \partial_1 \mathcal{E}\eta e_n$ and $\pd_n\mathcal{E}\Uppi^1_{\m{L}}\eta=\Uppi^1_{\m{L}}\eta/b$, is given by
		\begin{equation}\label{M^1 I II splitting}
			\mathcal{M}^{(1)}_{\phi}(q,u,\eta)=\mathcal{N}_\phi(q,\eta)(u-\pd_n\mathcal{E}\Uppi^1_{\m{H}}\eta e_1+\pd_1\mathcal{E}\eta e_n)-\mathcal{N}^{(1)}_{\phi}(q,\eta)(1+\Uppi^1_{\m{L}}\eta/b)e_1
		\end{equation}
		and
		\begin{equation}\label{the definition of the M^2 piece}
			e_1\cdot\mathcal{M}_\phi^{(2)}(\Uppi^1_{\m{L}}\eta)=-\tp{\mathcal{N}_\phi(0,0)\Uppi^1_{\m{L}}\eta/b+\mathcal{N}^{(2)}_\phi(\Uppi^1_{\m{L}}\eta)(1+\Uppi^1_{\m{L}}\eta/b)}.
		\end{equation}
		That $\mathcal{M}_\phi^{(1)}$ is ${_{\m{s}}}T^\infty_{0,2+\tfloor{n/2}}$ is a consequence of the composition of tame maps, Lemma~\ref{lem on composition of tame maps}, the second product result, Lemma~\ref{lem on smooth tameness of products 2}, the mapping properties of $\mathcal{N}_\phi$, Lemma~\ref{lem on taylor expansion trick}, and the properties of anisotropic Sobolev spaces enumerated in~\eqref{the norm on the anisotropic Sobolev spaces} and Proposition~\ref{proposition on frequency splitting}. That $e_1\cdot \mathcal{M}_\phi^{(2)}$ is ${_{\m{s}}}T^\infty_{0,0}$ is a consequence of Lemmas~\ref{lem onsmooth tameness of products 3} and~\ref{lem on taylor expansion trick}.
	\end{proof}
	
	\subsection{Smooth tameness of the nonlinear operator}\label{section on smooth tameness of the nonlinear map}

	We begin by defining a nonlinear operator associated with the PDE~\eqref{The nonlinear equations in the right form}. We will employ the Banach scales defined in Section~\ref{section on scales of Banach spaces}.  For some $r\ge 2+\lfloor n/2\rfloor$, $0<\rho\le\rho_{\m{WD}}$, with $\rho_{\m{WD}}$ determined by Theorem~\ref{thm on smooth tameness of the nonlinear operator}, and $s\ge r$, we set
	\begin{equation}\label{the nonlinear map equation}
		\Bar{\Psi}:(B_{\X^{r}}(0,\rho)\cap\X^{1+s})\times\W_s\times\R^+\to\Y^s\times\W_s\times\R
	\end{equation}
	via 
	\begin{equation}\label{the definition of the nonlinear operator} \index{\textbf{Nonlinear maps}!00@$\Bar{\Psi}$}
		\Bar{\Psi}(q,u,\eta,\mathcal{T},\mathcal{G},\mathcal{F},\gam)=(\Psi(q,u,\eta,\gam)+\Phi(q,u,\eta,\mathcal{T},\mathcal{G},\mathcal{F}),\mathcal{T},\mathcal{G},\mathcal{F},\gam), 
	\end{equation}
	where
	\begin{equation}
		\Psi:(B_{\X^r}(0,\rho)\cap\X^{1+s})\times\R^+\to\Y^s \index{\textbf{Nonlinear maps}!10@$\Psi$}
		\text{ and }
		\Phi:(B_{\X^{r}}(0,\rho)\cap\X^s)\times\W_s\to\Y^s\index{\textbf{Nonlinear maps}!20@$\Phi$}
	\end{equation}
	are the principal and auxiliary parts of $\Bar{\Psi}$, which are defined via
	\begin{equation}\label{Nebraska}
		\Psi(q,u,\eta,\gam)=(\Psi_1(q,u,\eta),\Psi_2(q,u,\eta,\gam),\Psi_3(q,u,\eta,\gam))
	\end{equation}
	and
	\begin{equation}\label{hey i think this may be the last thing I label}
		\Phi(q,u,\eta,\mathcal{T},\mathcal{G},\mathcal{F})=(0,\Phi_2(q,\eta,\mathcal{G},\mathcal{F}),\Phi_3(\eta,\mathcal{T})),
	\end{equation}
	where
	\begin{equation}\label{definition of Psi1}
		\Psi_1(q,u,\eta)=\grad\cdot(\sig_{q,\eta}(u- M_{\eta}e_1)),
	\end{equation}
	\begin{equation}\label{definition of the nonlinear map Psi2}
		\Psi_2(q,u,\eta,\gam)=\gam^2\sig_{q,\eta}M_{\eta}^{-\m{t}}(\tp{u-\gam M_{\eta}e_1}\cdot\grad(M_{\eta}^{-1}u))+\sig_{q,\eta}\grad(q+\mathfrak{g}\eta)-\gam M_{\eta}^{-\m{t}}\grad\cdot(\S^{\sig_{q,\eta}}_{\mathcal{A}_\eta}(M_{\eta}^{-1}u)M_{\eta}^{\m{t}}),
	\end{equation}
	\begin{equation}\label{definition of the nonlinear map Psi3}
		\Psi_3(q,u,\eta,\gam)=-((P-P_{\m{ext}})\circ\sig_{q,\eta}-\gam\S^{\sig_{q,\eta}}_{\mathcal{A}_\eta}(M_{\eta}^{-1}u))M_{\eta}^{\m{t}}e_n-\varsigma\mathscr{H}(\eta)M_{\eta}^{\m{t}}e_n,
	\end{equation}
	\begin{equation}\label{definition of the nonlinear map Phi2}
		\Phi_2(q,\eta,\mathcal{G},\mathcal{F})=-J_\eta M_{\eta}^{-\m{t}}\tp{\sig_{q,\eta}\mathcal{G}\circ\mathfrak{F}_\eta+\mathcal{F}\circ\mathfrak{F}_\eta},
	\end{equation}
	and
	\begin{equation}\label{definition of the nonlinear map Phi3}
		\Phi_3(\eta,\mathcal{T})=-\mathcal{T}\circ\mathfrak{F}_\eta M_{\eta}^{\m{t}}e_n.
	\end{equation}
    Here we recall that $\mathfrak{F}_\eta$ is defined by \eqref{flattening_map_def}, $\mathcal{A}_\eta$ and $J_\eta$ by \eqref{geometry_and_jacobian_def},  $\S^{\uptau}_{\mathcal{A}}$ by \eqref{flat_stress_def},  $M_{\eta}$ by \eqref{Mississippi}, and $\sig_{q,\eta}$  by \eqref{sigma_q_eta_def}. We also recall the assumptions on $\varsigma$, $\upmu$, and $\uplambda$ stated in~\eqref{parameter_assumptions}.
    
	Armed with the results from the previous two subsections, we now endeavor to study the smooth tameness of the map $\Bar{\Psi}$ from \eqref{the nonlinear map equation}. Our strategy is to handle the $\Psi$ and $\Phi$ pieces separately and component-wise.  Our first result studies the continuity equation piece, $\Psi_1$, which we recall is defined in~\eqref{definition of Psi1}. This $\Psi_1$ term is the source of the derivative loss and also requires the most careful analysis among the smooth tameness verification results.
 
	\begin{prop}[Smooth tameness of the continuity equation]\label{prop on smooth tameness of the continuity equation}
		There exists a $\rho_{\m{cont}}\in(0,\infty)$, depending only on the domain of the inverse enthalpy (see~\eqref{domain of the inverse enthalpy}) such that the following hold
  \begin{enumerate}
    \item There exists constants $c,C\in\R^+$ such that for all $(q,u,\eta)\in B_{\X^{1+\tfloor{n/2}}}(0,\rho_{\m{cont}})$ we have the estimate $c\le\sig_{q,\eta}\le C$ where $\sig_{q,\eta}$ is defined in~\eqref{sigma_q_eta_def}.
    \item $\Psi_1\in{_{\m{s}}}T^\infty_{1,1+\tfloor{n/2}}(B_{\X^{1+\tfloor{n/2}}}(0,\rho_{\m{cont}}),\pmb{\X};\tcb{\hat{H}^{1+s}(\Omega)}_{s\in\N})$, where we recall that the spaces $\hat{H}^s(\Omega)$ are defined in~\eqref{Rhode Island} (see also Remark~\ref{remarkable remark is named what else other than shark}).
  \end{enumerate}

	\end{prop}
	\begin{proof}
    We begin by proving the second item. First, we use Lemma~\ref{lem on fundamental decomposition of continuity equation-like vector fields} to set $\rho_{\m{cont}}=\rho_{(H_{\m{min}},H_{\m{max}})}$, where $H_{\m{min}}$ and $H_{\m{max}}$ are defined by~\eqref{domain of the inverse enthalpy}, to see that $\grad\cdot\mathcal{M}_{H^{-1}}(0,0,0)=0$, and to equate
		\begin{equation}
			\Psi_1(q,u,\eta)=\grad\cdot(\mathcal{M}_{H^{-1}}^{(1)}(q,u,\eta))+\pd_1\tp{e_1\cdot\mathcal{M}^{(2)}_{H^{-1}}(\Uppi^1_{\m{L}}\eta)}=\Psi^{\bf{I}}_{1}(q,u,\eta)+\Psi_1^{\bf{II}}(\Uppi^1_{\m{L}}\eta).
		\end{equation}
		In the above $H^{-1}$ refers to the inverse enthalpy function (see~\ref{enthalpy_def}). As a direct consequence of Lemma~\ref{lem on fundamental decomposition of continuity equation-like vector fields}, we obtain the inclusion
		\begin{equation}\label{dishwashers are so convenient}
			\Psi_1^{\bf{I}}\in{_{\m{s}}}T^\infty_{1,1+\tfloor{n/2}}(B_{\X^{1+\tfloor{n/2}}}(0,\rho_{\m{cont}}),\pmb{\X};\tcb{H^{1+s}(\Omega)}_{s\in\N}).
		\end{equation}
		On the other hand, the boundary conditions built into the definition of the space $\X^{1+\tfloor{n/2}}$ in \eqref{domain banach scales}, together with condition~\eqref{respect the trace or get the claw}, imply that
		\begin{equation}
			\int_0^b\Psi^{\bf{I}}_1(q,u,\eta)(\cdot,y)\;\m{d}y=(\grad_{\|},0)\cdot\int_0^b\mathcal{M}^{(1)}_{H^{-1}}(q,u,\eta)(\cdot,y)\;\m{d}y,
		\end{equation}
		and hence we deduce that $\Psi_1^{\bf{I}}$ also maps smoothly into the space $\hat{H}^0(\Omega)$. This fact combined with~\eqref{dishwashers are so convenient} proves that $\Psi^{\bf{I}}_1$ is ${_{\m{s}}}T^\infty_{1,1+\tfloor{n/2}}$ with respect to the Banach scales stated in the hypotheses.
		
    We next handle the $\Psi_1^{\bf{II}}$ piece.  By the norm in the anisotropic Sobolev spaces given in~\eqref{the norm on the anisotropic Sobolev spaces}, we readily see that the embedding $W^{0,\infty}((0,b);\pd_1\mathcal{H}^0_{(r_{n-1})}(\Sigma))\emb H^0(\Omega)$ holds and hence we deduce from conclusion~\eqref{mapping properties of the second piece in teh decomposition} of Lemma~\ref{lem on fundamental decomposition of continuity equation-like vector fields} that $\Psi_1^{\bf{II}}\in{_{\m{s}}}T^\infty_{0,0}(\bm{\mathcal{H}}_{(1)}(\Sigma);\bf{H}(\Omega))$.
  
  By inspection of the anisotropic norm from~\eqref{the norm on the anisotropic Sobolev spaces} again, we also deduce the embedding
  \begin{equation}
      W^{0,\infty}((0,b);\pd_1\mathcal{H}^0_{(r_{n-1})}(\Sigma))\emb W^{0,\infty}((0,b);\dot{H}^{-1}(\Sigma)),
  \end{equation}
  and hence the map
		\begin{equation}
			\mathcal{H}^0_{(1)}(\Sigma) \ni \zeta 
			    \mapsto 
			\int_0^b\Psi_1^{\bf{II}}(\zeta)(\cdot,y)\;\m{d}y
			\in \dot{H}^{-1}(\Sigma)
		\end{equation}
		is well-defined and smooth. These facts merge to show that $\Psi_1^{\bf{II}}$ is ${_{\m{s}}}T^{\infty}_{0,0}$ with respect to the Banach scales $\bm{\mathcal{H}}_{(1)}(\Sigma)$ and $\tcb{\hat{H}^s(\Omega)}_{s\in\N}$. We synthesize our results for $\Psi_1^{\bf{I}}$ and $\Psi_1^{\bf{II}}$ to complete the proof of the second item.

  We now deduce the bounds stated in the first item. By unpacking the meaning of $\rho_{\m{cont}}=\rho_{I}$ from Lemma~\ref{lem on taylor expansion trick}, for $I=(H_{\m{min}},H_{\m{max}})$,  we find that for $(q,u,\eta)\in B_{\X^{1+\tfloor{n/2}}}(0,\rho_{\m{cont}})$ we have the inclusion
  \begin{equation}
      (-\mathfrak{g}\m{id}_{\R^n}\cdot e_n + q+\mathfrak{g}(I-\mathcal{E})\eta)(\Omega)\subseteq(h_\star,h^\star)\Subset (H_{\m{min}},H_{\m{max}}),
  \end{equation}
  for some universal $h_\star,h^\star\in(H_{\m{min}},H_{\m{max}})$.
  Hence, by applying the inverse enthalpy, we find that $H^{-1}(h_\star)\le\sig_{q,\eta}\le H^{-1}(h^\star)$, which is the first item  with $c=H^{-1}(h_\star),C=H^{-1}(h^\star)\in\R^+$.
	\end{proof}

	Next we examine the $\Psi_2$ piece of the momentum equation, which we recall is defined in~\eqref{definition of the nonlinear map Psi2}.
	
\begin{prop}[Smooth tameness of the momentum equation, 1]\label{prop on smooth tameness of the momentum equation 1}
		There exits a $\rho_{\m{mome}}\in(0,\rho_{\m{cont}}]$ such that the following  hold.
  \begin{enumerate}
        \item For $(q,u,\eta)\in B_{\X^{1+\tfloor{n/2}}}(0,\rho_{\m{mome}})$, we have the bounds $1/2\le J_\eta\le 3/2$, where $J_\eta$ is defined in~\eqref{geometry_and_jacobian_def}.
      \item $\Psi_2\in{_{\m{s}}}T^\infty_{0,1+\tfloor{n/2}}(B_{\X^{1+\tfloor{n/2}}}(0,\rho_{\m{mome}})\times\R^+,\pmb{\X}\times\R;\bf{H}(\Omega;\R^n))$.
  \end{enumerate}
	\end{prop}
	\begin{proof}
		We begin by selecting $\rho_{\m{mome}}$. The map
		\begin{equation}
			\mathcal{H}^{3/2+\tfloor{n/2}}(\Sigma)\ni\eta\mapsto J_\eta-1=\pd_n\mathcal{E}\eta\in L^\infty(\Omega)
		\end{equation}
		is bounded thanks to the supercritical Sobolev embedding and the continuity properties of $\mathcal{E}$ from Proposition~\ref{lem on mapping properties of the Poisson extension operator variants}. Denote the preimage of the set $B_{L^\infty}(0,1/2)$ under this map by $O\subseteq\mathcal{H}^{3/2+\tfloor{n/2}}(\Sigma)$. $O$ is open and we set $\tilde{\rho}_{\m{mome}}=\m{dist}(0,\pd O)\in\R^+$ and $\rho_{\m{mome}}=\min\tcb{\tilde{\rho}_{\m{mome}},\rho_{\m{cont}}}$, where the latter is defined in Proposition~\ref{prop on smooth tameness of the continuity equation}.
		
		With this choice of $\rho_{\m{mome}}$ we have that $1/2\le J_\eta\le3/2$, and hence the inverse matrix $M_\eta^{-1}$ exists whenever $\norm{\eta}_{\mathcal{H}^{3/2+\tfloor{n/2}}} < \rho_{\m{mome}}$. Consequently, all of the expressions appearing in $\Psi_2$ are pointwise well-defined. We next decompose
		\begin{equation}\label{Montana}
			\Psi_2(q,u,\eta,\gam)=\gam^2\Psi_2^{\bf{I}}(q,u,\eta)+\Psi_2^{\bf{II}}(q,\eta)+\gam\Psi_2^{\bf{III}}(q,u,\eta),
		\end{equation}
		where
		\begin{equation}\label{Missouri}
			\begin{cases}
				\Psi_2^{\bf{I}}(q,u,\eta)=\sig_{q,\eta}M_{\eta}^{-\m{t}}[(u-M_{\eta}e_1)\cdot\grad(M_{\eta}^{-1}u)],\\
				\Psi_2^{\bf{II}}(q,\eta)=\sig_{q,\eta}\grad(q+\mathfrak{g}\eta),\\
				\Psi_2^{\bf{III}}(q,u,\eta)=-M_{\eta}^{-\m{t}}\grad\cdot(\S^{\sig_{q,\eta}}_{\mathcal{A}_\eta}(M_{\eta}^{-1}u)M_{\eta}^{\m{t}}).
			\end{cases}
		\end{equation}
		The $\gamma$ dependence in~\eqref{Montana} is simple, and it is sufficient to study the pieces of~\eqref{Missouri} individually. For the first piece, we use Lemma~\ref{lem on taylor expansion trick} to write
		\begin{equation}\label{decomposition of the density}
			\sig_{q,\eta}=\varrho+\mathcal{N}_{H^{-1}}^{(1)}(q,\eta)+\mathcal{N}^{(2)}_{H^{-1}}(\Uppi^1_{\m{L}}\eta),
		\end{equation}
		where $H^{-1}$ is the inverse of the enthalpy  (see~\eqref{enthalpy_def}).  The embeddings of the second item of Proposition~\ref{proposition on frequency splitting} allow us to  view the latter map as $\mathcal{N}^{(2)}_{H^{-1}}\in{_{\m{s}}}T^{\infty}_{0,0}(\bm{\mathcal{H}}_{(1)}(\Sigma);\bf{W}(\Omega))$.  Thanks to the smooth tameness of superposition multipliers proved in Proposition~\ref{lem on smooth tameness of superposition multipliers}, the map $(u,\eta)\mapsto \grad(M_\eta^{-1}u)$ is ${_{\m{s}}}T_{0,1+\tfloor{n/2}}^\infty$ with respect to the Banach scales $\tcb{H^{2+s}(\Omega;\R^n)\times\mathcal{H}^{5/2+s}(\Sigma)}_{s\in\N}$ and $\tcb{H^s(\Omega;\R^n)}_{s\in\N}$. By applying the product of tame maps result from Lemma~\ref{lem on smooth tameness of products 2} and the composition of smooth tame maps result from  Lemma~\ref{lem on composition of tame maps}, we then see that the map $(u,\eta)\mapsto(u-M_{\eta}e_1)\cdot\grad(M_{\eta}^{-1}u)$ is also ${_{\m{s}}}T_{0,1+\tfloor{n/2}}^\infty$ with respect to the aforementioned Banach scales. By invoking Lemmas~\ref{lem on smooth tameness of superposition multipliers} and~\ref{lem on composition of tame maps} again, we get that the same conclusion is true for the map $(u,\eta)\mapsto M^{-\m{t}}_\eta[(u-M_{\eta}e_1)\cdot\grad(M_{\eta}^{-1}u)]$. Finally, by using Lemmas~\ref{lem on taylor expansion trick}, \ref{lem on smooth tameness of products 2}, and~\ref{lem on composition of tame maps}, we deduce that $(q,u,\eta)\mapsto(\varrho+\mathcal{N}_{H^{-1}}^{(1)}(q,\eta)+\mathcal{N}^{(2)}_{H^{-1}}(\Uppi^1_{\m{L}}\eta))M^{-\m{t}}_\eta[(u- M_{\eta}e_1)\cdot\grad(M_{\eta}^{-1}u)]$, i.e. $\Psi_2^{\bf{I}}$, is ${_{\m{s}}}T_{0,1+\tfloor{n/2}}^\infty$ with respect to the Banach scales $\tcb{\X_s}_{s\in\N}$ and $\tcb{H^s(\Omega;\R^n)}_{s\in\N}$.
		
		Now we examine the $\Psi_2^{\bf{II}}$ piece.  We again employ the density decomposition in~\eqref{decomposition of the density}.  Then we see that $\Psi_2^{\bf{II}}$ is ${_{\m{s}}}T^\infty_{0,1+\tfloor{n/2}}$ for the Banach scales $\tcb{H^{1+s}(\Omega)\times\mathcal{H}^{5/2+s}(\Sigma)}_{s\in\N}$ and $\tcb{H^s(\Omega;\R^n)}_{s\in\N}$ as a consequence of Lemmas~\ref{lem on taylor expansion trick} and~\ref{lem on smooth tameness of products 2}.
		
		Finally, we examine the $\Psi_2^{\bf{III}}$ piece,  decomposing further:
		\begin{equation}
			\Psi_2^{\bf{III}}(q,u,\eta)=\Psi_2^{\bf{III_1}}(q,u,\eta)+\Psi_2^{\bf{III_2}}(q,u,\eta)+\Psi_2^{\bf{III_3}}(q,u,\eta)+\Psi_2^{\bf{III_4}}(q,u,\eta),
		\end{equation}
		where
		\begin{equation}
			\begin{cases}
				\Psi_2^{\bf{III_1}}(q,u,\eta)=-M_{\eta}^{-\m{t}}\grad\cdot\sp{\upmu(\sig_{q,\eta})J_\eta^{-1}\grad(M_\eta^{-1}u)M_\eta M_{\eta}^{\m{t}}},\\
				\Psi_2^{\bf{III_2}}(q,u,\eta)=-M_{\eta}^{-\m{t}}\grad\cdot(\upmu(\sig_{q,\eta})J_{\eta}^{-1}M_\eta^{\m{t}}\grad(M_{\eta}^{-1}u)^{\m{t}}M_{\eta}^{\m{t}}),\\
				\Psi_2^{\bf{III_3}}(q,u,\eta)=2M_{\eta}^{-\m{t}}\grad\cdot\tp{\upmu(\sig_{q,\eta})J_{\eta}^{-1}(\grad\cdot u)M_{\eta}^{\m{t}}},\\
				\Psi_2^{\bf{III}_4}(q,u,\eta)=-M_{\eta}^{-\m{t}}\grad\cdot(\uplambda(\sig_{q,\eta})J_\eta^{-1}(\grad\cdot u)M_{\eta}^{\m{t}}).
			\end{cases}
		\end{equation}
		We handle these four terms in more or less the same way. The $\upmu$ and $\uplambda$ viscosity coefficients are decomposed as
		\begin{equation}
			\begin{cases}
				\upmu(\sig_{q,\eta})=\upmu(\varrho)+\mathcal{N}^{(1)}_{\upmu\circ H^{-1}}(q,\eta)+\mathcal{N}^{(2)}_{\upmu\circ H^{-1}}(\Uppi^1_{\m{L}}\eta),\\
				\uplambda(\sig_{q,\eta})=\uplambda(\varrho)+\mathcal{N}^{(1)}_{\uplambda\circ H^{-1}}(q,\eta)+\mathcal{N}^{(2)}_{\uplambda\circ H^{-1}}(\Uppi^1_{\m{L}}\eta),
			\end{cases}
		\end{equation}
		via Lemma~\ref{lem on taylor expansion trick}, while the $M_\eta$, $M_{\eta}^{-1}$, $M_{\eta}^{\m{t}}$, $M_{\eta}^{-\m{t}}$, and $J_\eta^{-1}$ terms are viewed as superposition multipliers and thus are handled via Lemma~\ref{lem on smooth tameness of superposition multipliers}. Hence, the fact that each $\Psi_2^{\bf{III}_j}$, $j\in\tcb{1,\dots,4}$, is ${_{\m{s}}}T^\infty_{0,1+\tfloor{n/2}}$ with respect to the Banach scales $\tcb{\X_s}_{s\in\N}$ and $\tcb{H^s(\Omega;\R^n)}_{s\in\N}$ is a consequence of the aforementioned lemmas, the second result on products, Lemma~\ref{lem on smooth tameness of products 2}, and the result on composition of smooth tame maps, Lemma~\ref{lem on composition of tame maps}. 
	\end{proof}
	
	Now we examine the $\Psi_3$ piece of the dynamic boundary condition, which we recall is defined in~\eqref{definition of the nonlinear map Psi3}.

	\begin{prop}[Smooth tameness of the dynamic boundary condition, 1]\label{prop on smooth tameness of the dynamic boundary condition 1}
		With $\rho_{\m{dyna}}=\rho_{\m{mome}}\in\R^+$ as in Proposition~\ref{prop on smooth tameness of the momentum equation 1}, we have
		\begin{equation}
			\Psi_3\in{_{\m{s}}}T^\infty_{0,1+\tfloor{n/2}}(B_{\X^{1+\tfloor{n/2}}}(0,\rho_{\m{dyna}})\times\R^+,\pmb{\X}\times\R;\tcb{H^{1/2+s}(\Sigma;\R^n)}_{s\in\N}).
		\end{equation}
	\end{prop}
	\begin{proof}
		As before, we begin the proof by decomposing
		\begin{equation}
			\Psi_3(q,u,\eta,\gam)=\Psi_3^{\bf{I}}(q,\eta)+\gam\Psi_3^{\bf{II}}(q,u,\eta)+\Psi_3^{\bf{III}}(\eta)
		\end{equation}
		where
		\begin{equation}
			\begin{cases}
				\Psi_3^{\bf{I}}(q,\eta)=-(P\circ\sig_{q,\eta}-P_{\m{ext}})M_{\eta}^{\m{t}}e_n,\\
				\Psi_3^{\bf{II}}(q,u,\eta)=\S^{\sig_{q,\eta}}_{\mathcal{A}_{\eta}}(M_{\eta}^{-1}u)M_{\eta}^{\m{t}}e_n,\\
				\Psi_3^{\bf{III}}(\eta)=-\varsigma\mathscr{H}(\eta)M_{\eta}^{\m{t}}e_n.
			\end{cases}
		\end{equation}
		For the first piece we recall that $H^{-1}$ is the inverse of the enthalpy, as in~\eqref{enthalpy_def}, and use the fact that $\sig_{q,\eta}=H^{-1}(-\mathfrak{g}b+q)$ and $P_{\m{ext}}=P\circ\varrho(b)=P\circ H^{-1}(-\mathfrak{g}b)$ on $\Sigma$ to rewrite
		\begin{equation}
			\Psi_3^{\bf{I}}(q,\eta)=(P\circ H^{-1}(\cdot-\mathfrak{g}b)-P\circ H^{-1}(-\mathfrak{g}b))(q)M_{\eta}^{\m{t}}e_n.
		\end{equation}
		From this, the operator $\Psi_3^{\bf{I}}$ is readily seen to belong to  ${_{\m{s}}}T^\infty_{0,1+\tfloor{n/2}}$ with respect to the Banach scales $\tcb{H^{1+s}(\Omega)\times\mathcal{H}^{5/2+s}(\Sigma)}_{s\in\N}$ and $\tcb{H^{1/2+s}(\Sigma;\R^n)}_{s\in\N}$, as a consequence of the first item of Lemma \ref{lem on smooth tameness of superposition} regarding the tame smoothness of superposition, the result on tame smoothness of superposition multipliers, Lemma~\ref{lem on smooth tameness of superposition multipliers}, and the result on the composition of smooth tame maps, Lemma~\ref{lem on composition of tame maps}.
		
		The fact that $\Psi_3^{\bf{II}}$ is ${_{\m{s}}}T^\infty_{0,1+\tfloor{n/2}}$ for the scales $\tcb{\X_s}_{s\in\N}$ and $\tcb{H^{1/2+s}(\Sigma;\R^n)}_{s\in\N}$ follows from an argument similar to the analysis of the $\Psi_2^{\bf{III}}$-term from the proof of Lemma~\ref{prop on smooth tameness of the momentum equation 1}.
		
		Finally, we handle $\Psi_3^{\bf{III}}$. The mean curvature operator has the expression $\mathscr{H}(\eta)=\grad_{\|}\cdot(\tilde{\mathscr{H}}(\grad_{\|}\eta))$, where the  map $\tilde{\mathscr{H}}\in C^\infty(\R^{n-1};\R^{n-1})$ is given by $\tilde{\mathscr{H}}(v) = \tbr{v}^{-1} v$. 		From this we deduce that $\Psi_3^{\bf{III}}$ is ${_{\m{s}}}T^\infty_{0,1+\tfloor{n/2}}$ by combining the conclusions of Lemmas~\ref{lem on smooth tameness of superposition} (the first item), \ref{lem on smooth tameness of superposition multipliers}, and~\ref{lem on composition of tame maps}.
	\end{proof}
	
	We now synthesize our previous results to deduce the smooth tameness of the principal part  nonlinear operator from~\eqref{Nebraska}.

	\begin{thm}[Smooth tameness of the principal part nonlinear operator]\label{thm on smooth tameness of the principal part nonlinear operator}
		There exists a $\rho_{\m{prin}}\in\R^+$ such that
		\begin{equation}
			\Psi\in{_{\m{s}}}T^\infty_{1,1+\tfloor{n/2}}(B_{\X^{1+\tfloor{n/2}}}(0,\rho_{\m{prin}})\times\R^+,\pmb{\X}\times\R;\pmb{\Y}).
		\end{equation}
	\end{thm}
	\begin{proof}
		We set $\rho_{\m{prin}}=\min\tcb{\rho_{\m{cont}},\rho_{\m{mome}},\rho_{\m{dyna}}}$ and apply Propositions~\ref{prop on smooth tameness of the continuity equation}, \ref{prop on smooth tameness of the momentum equation 1}, and~\ref{prop on smooth tameness of the dynamic boundary condition 1}.
	\end{proof}

	The remainder of this subsection is devoted to the study of the auxiliary piece $\Phi$ of the nonlinear operator $\Bar{\Psi}$. The next result handles the $\Phi_2$ piece, which we recall is defined in~\eqref{definition of the nonlinear map Phi2}.
	
	\begin{prop}[Smooth tameness of the momentum equation, 2]\label{prop on smooth tameness of the momentum equation 2}
		There exists a $\rho_{\m{bulk}}\in\R^+$ such that for every $m\in\N$,
		\begin{equation}
			\Phi_2\in{_{\m{s}}}T^m_{0,2+\tfloor{n/2}}(B_{\X^{2+\tfloor{n/2}}}(0,\rho_{\m{bulk}})\times\W_{m+2+\tfloor{n/2}},\tcb{\X^s\times\W_{m+s}}_{s\in\N};\bf{H}(\Omega;\R^n)).
		\end{equation}
	\end{prop}
	\begin{proof}
		Thanks to the smooth tameness of superposition multipliers, Lemma~\ref{lem on smooth tameness of superposition multipliers}, the decomposition of $\sig_{q,\eta}=\mathcal{N}^{(1)}_{H^{-1}}(q,\eta)+\mathcal{N}^{(2)}_{H^{-1}}(\Uppi^1_{\m{L}}\eta)+\varrho$ from Lemma~\ref{lem on taylor expansion trick}, and the result on composition of smooth tame maps, Lemma~\ref{lem on composition of tame maps}, we see that it is sufficient to prove that the nonlinear operator
		\begin{equation}\label{eqref mappity dappity}
			\Lambda:H^{m+2+\tfloor{n/2}}(\R^n;\R^n)\times B_{\mathcal{H}^{9/2+\tfloor{n/2}}(\Sigma)}(0,\rho)\to H^0(\Omega;\R^n) 
			  \text{ given by }
			\Lambda(\mathcal{I},\eta)=\mathcal{I}\circ\mathfrak{F}_\eta
		\end{equation}
		is, for some $\rho>0$, well-defined and ${_{\m{s}}}T^m_{0,2+\tfloor{n/2}}$ for the scales $\tcb{H^{m+s}(\R^n;\R^n)\times\mathcal{H}^{5/2+s}(\Sigma)}_{s\in\N}$ and $\tcb{H^s(\Omega;\R^n)}$. For this we let $\mathfrak{E}_\Omega$ and $\mathfrak{R}_\Omega$ denote the Stein extension (see Definition~\ref{defn Stein-extension operator}) and the restriction operators for $\Omega$, respectively, and note that we have the equivalent formula
		\begin{equation}
			\Lambda(\mathcal{I},\eta)=\mathfrak{R}_\Omega\mathcal{I}(\m{id}_{\R^n}+\mathfrak{E}_\Omega\mathcal{E}(\Uppi^1_{\m{L}}\eta+\Uppi^1_{\m{H}}\eta)).
		\end{equation}
		Hence, according to Lemma~\ref{comp_C2} and properties of the maps $\mathcal{E}$ and $\mathfrak{E}_\Omega$ (see Lemma~\ref{lem on mapping properties of the Poisson extension operator variants} and Example~\ref{example on Sobolev spaces on domains}), there exists a $\rho_{\m{comp}}\in\R^+$ (depending only on the dimension and $\Omega$) such that whenever $\rho\le\rho_{\m{comp}}$ the map~\eqref{eqref mappity dappity} is  ${_{\m{s}}}T^m_{0,2+\tfloor{n/2}}$ with respect to the aforementioned Banach scales.
	\end{proof}
	
	The penultimate result of this subsection considers the $\Phi_3$ piece of $\Bar{\Psi}$, which is defined in~\eqref{definition of the nonlinear map Phi3}.
	
	\begin{prop}[Smooth tameness of the dynamic boundary condition, 2]\label{smooth tameness of dynamic boundary condition 2}
		For $\rho_{\m{surf}}=\rho_{\m{bulk}}\in\R^+$, where the latter is from Proposition~\ref{prop on smooth tameness of the momentum equation 2}, we have that for every $m\in\N$,
		\begin{equation}
			\Phi_3\in{_{\m{s}}}T^m_{0,2+\tfloor{n/2}}(B_{\X^{2+\tfloor{n/2}}}(0,\rho_{\m{surf}})\times\W_{m+2+\tfloor{n/2}},\tcb{\X^s\times\W_{m+s}}_{s\in\N};\tcb{H^{1/2+s}(\Sigma;\R^n)}_{s\in\N}).
		\end{equation}
	\end{prop}
	\begin{proof}
		We argue is in the proof of Proposition~\ref{prop on smooth tameness of the momentum equation 2} and reduce to studying the map
		\begin{equation}
			\tilde{\Lambda}:H^{m+2+\tfloor{n/2}}(\R^n;\R^n)\times B_{\mathcal{H}^{9/2+\tfloor{n/2}}(\Sigma)}(0,\rho)\to H^0(\Sigma;\R^n) 
		\end{equation}
		given by $\tilde{\Lambda}(\mathcal{T},\eta)=\m{Tr}_{\Sigma}(\mathcal{T}\circ\mathfrak{F}_\eta)$. We then use the equivalent formula
		\begin{equation}
			\tilde{\Lambda}(\mathcal{T},\eta)=\m{Tr}_{\Sigma}\mathfrak{R}_\Omega\mathcal{T}(\m{id}_{\R^n}+\mathfrak{E}_\Omega\mathcal{E}(\Uppi^1_{\m{L}}\eta+\Uppi^1_{\m{H}}\eta)),
		\end{equation}
		and conclude as before.
	\end{proof}

 \begin{rmk}
     The parameter $m\in\N$ appearing in Propositions~\ref{prop on smooth tameness of the momentum equation 2} and~\ref{smooth tameness of dynamic boundary condition 2} (and subsequently in Theorem~\ref{thm on smooth tameness of the nonlinear operator}) gives the data terms $(\mathcal{T},\mathcal{G},\mathcal{F})$ $m$-extra derivatives to ensure that the composition type nonlinearities of $\Phi_2$ and $\Phi_3$ are $C^m$ into the correct codomain space. One sees that this extra regularity is necessary upon inspection of the derivative formulae~\eqref{comp_C2_01} and~\eqref{comp_C2_02}.
 \end{rmk}
	
	At last, we are ready to deduce the smooth tameness of the nonlinear operator $\Bar{\Psi}$, which we recall is defined in~\eqref{the nonlinear map equation}. In the following statement $\m{WD}$ is an acronym for `well-defined'.
	
	\begin{thm}[Smooth tameness of the nonlinear operator]\label{thm on smooth tameness of the nonlinear operator}
		There exists a $\rho_{\m{WD}}\in\R^+$\index{\textbf{Miscellaneous}!30@$\rho_{\m{WD}}$} such that the following hold.
  \begin{enumerate}
    \item For $\N\ni s\ge 2+\tfloor{n/2}$ and $(q,u,\eta)\in \X^s\cap B_{\X^{2+\tfloor{n/2}}}(0,\rho_{\m{WD}})$ we have that the flattening map $\mathfrak{F}_\eta$ defined in~\eqref{flattening_map_def} is a smooth diffeomorphism from $\Omega$ to $\Omega[\eta]$ that extends to a $C^{s+2-\tfloor{n/2}}$ diffeomorphism from $\Bar{\Omega}$ to  $\Bar{\Omega[\eta]}$.
    
    \item For every $m\in\N$, $\Bar{\Psi}\in{_{\m{s}}}T^m_{1,2+\tfloor{n/2}}(B_{\X^{2+\tfloor{n/2}}}(0,\rho_{\m{WD}})\times\W_{m+1+\tfloor{n/2}}\times\R^+,\bf{O}_m;\bf{P}_m)$, where we have denoted $\bf{O}_m=\tcb{\X^s\times\W_{m-1+s}\times\R}_{s\in\N}$ and $\bf{P}_m=\tcb{\Y^s\times\W_{m+s-1}\times\R}_{s\in\N}$.
  \end{enumerate}
	\end{thm}
	\begin{proof}
		We set $\rho_{\m{WD}}=\min\tcb{\rho_{\m{prin}},\rho_{\m{bulk}},\rho_{\m{surf}}}$ and apply Theorem~\ref{thm on smooth tameness of the principal part nonlinear operator} and Propositions~\ref{prop on smooth tameness of the momentum equation 2} and~\ref{smooth tameness of dynamic boundary condition 2}. This immediately gives the second item. For the first item, we note that Proposition~\ref{prop on smooth tameness of the momentum equation 1} guarantees that $J_\eta>0$ and $J_\eta,1/J_\eta\in L^\infty(\Omega)$, which means that $\mathfrak{F}_\eta$ is a continuous bijection from $\Bar{\Omega}$ to $\Bar{\Omega[\eta]}$.  On the other hand, the identity $\pd_n(\mathfrak{F}_\eta\cdot e_n)=J_\eta=\det(\grad\mathfrak{F}_\eta)$ and the inverse function theorem  guarantee that $\mathfrak{F}_\eta:\Omega\to\Omega[\eta]$ is a smooth diffeomorphism. The regularity of $\mathfrak{F}_\eta:\Bar{\Omega}\to\Bar{\Omega[\eta]}$ and its inverse now follow from Sobolev embeddings and Proposition~\ref{proposition on frequency splitting}.
	\end{proof}
	
	\begin{rmk}\label{I feel hung up and I dont know why, its alright}
		In the case $m=2$, the second item of Theorem~\ref{thm on smooth tameness of the nonlinear operator} is stating that $\Bar{\Psi}\in{_{\m{s}}}T^2_{1,2+\tfloor{n/2}}$ for the Banach scales $\pmb{\E}$ and $\pmb{\F}$, which are introduced in~\eqref{North Carolina}.
	\end{rmk}

	\subsection{Derivative splitting}\label{subsection on derivative splitting and notation for linear analysis}

	We now turn our attention to the study of the derivative of the map $\Bar{\Psi}$ from \eqref{the nonlinear map equation}.  When written in full, it is rather complicated, so our focus now is to identify a principal part and handle the remainder terms. Let
	\begin{equation}
		(q_0,u_0,\eta_0)\in B_{\X^{2+\tfloor{n/2}}}(0,\rho_{\m{WD}}),\quad(\mathcal{T}_0,\mathcal{G}_0,\mathcal{F}_0)\in\W_{3+\tfloor{n/2}},
	\end{equation}
	\begin{equation}
		(q,u,\eta)\in\X^{2+\tfloor{n/2}},\quad(\mathcal{T},\mathcal{G},\mathcal{F})\in\W_{2+\tfloor{n/2}},\quad \gam_0\in\R^+,\;\gam\in\R,
	\end{equation}
	and write the $6-$tuple
	\begin{equation}\label{theta_0_def}
	\theta_0=(q_0,u_0,\eta_0,\mathcal{T}_0,\mathcal{G}_0,\mathcal{F}_0).    
	\end{equation}
	The derivative of the map $\Bar{\Psi}$ has the formula
	\begin{equation}\label{formula for the derivative for like he had no place to go psi}
		D\Bar{\Psi}(\theta_0,\gam_0)(q,u,\eta,\mathcal{T},\mathcal{G},\mathcal{F},\gam)=\sp{D\Psi(q_0,u_0,\eta_0,\gam_0)(q,u,\eta,\gam)+D\Phi(\theta_0)(q,u,\eta,\mathcal{T},\mathcal{G},\mathcal{F}),\mathcal{T},\mathcal{G},\mathcal{F},\gam}.
	\end{equation}
	We will decompose the above into a principal part, $A$, and remainder terms, $P$, $Q$, and $R$ (these symbols will come equipped with various adornments, but for brevity we will often refer to them without these in the main text). The derivative $D\Psi$ is split as follows:
	\begin{equation}\label{A+P_splitting_of_DPsi}
		D\Psi(w_0,\gam_0)(q,u,\eta,\gam)=\overset{w_0,\gam_0}{A}(q,u,\eta)+\overset{w_0,\gam_0}{P}(q,u,\eta,\gam), 
	\end{equation}
	where for brevity we define the triple
	\begin{equation}\label{w_0_def}
	    w_0=(q_0,u_0,\eta_0).	    
	\end{equation}
     The $A$ piece is meant to be as close as possible to $D\Psi(0,0,0,\gam_0)$, but to retain the entirety of the structure responsible for derivative loss in the continuity equation.  To that end, we set
	\begin{equation}\label{definition of the principal part operator}\index{\textbf{Linear maps}!30@$\overset{w_0,\gam_0}{A}$}
		\overset{w_0,\gam_0}{A}(q,u,\eta)=\sp{\overset{w_0}{A^1}(q,u,\eta),\overset{\gam_0}{A^2}(q,u,\eta),\overset{\gam_0}{A^3}(q,u,\eta)},
	\end{equation}
	where
	\begin{equation}\label{connecticut}
		\overset{w_0}{A^1}(q,u,\eta)=\grad\cdot(\varrho u)+\grad\cdot(v_{w_0}(q+\mathfrak{g}\eta)),
	\end{equation}
	\begin{equation}\label{deleware}
		\overset{\gam_0}{A^2}(q,u,\eta)=D\Psi_2(0,0,0,\gam_0)(q,u,\eta,0)=-\gam_0^2\varrho\pd_1u+\varrho\grad(q+\mathfrak{g}\eta)-\gam_0\grad\cdot\S^\varrho u,
	\end{equation}
	and
	\begin{equation}
		\overset{\gam_0}{A^3}(q,u,\eta)=D\Psi_3(0,0,0,\gam_0)(q,u,\eta,0)=-(\varrho q-\gam_0\S^\varrho u)e_n-\varsigma\Delta_{\|}\eta e_n.
	\end{equation}
	Here $v_{w_0} : \Omega \to \R^n$ is the vector field 
	\begin{equation}\label{definition of the vector field vnaught}
		v_{w_0}=\mathcal{M}_{(H^{-1})'}(q_0,u_0,\eta_0)=(H^{-1})'(-\mathfrak{g}\m{id}_{\R^n}\cdot e_n+q_0+\mathfrak{g}(I-\mathcal{E})\eta_0)(u_0-M_{\eta_0} e_1),
	\end{equation}
	where the $\mathcal{M}$ notation is from Lemma~\ref{lem on fundamental decomposition of continuity equation-like vector fields} and $(H^{-1})'$ refers to the derivative of the inverse enthalpy (see equation~\eqref{enthalpy_def}). Note that since $\varrho(y)=H^{-1}(-\mathfrak{g}y)$ for $y\in[0,b]$ we have that
	\begin{equation}
	v_{0}=-(H^{-1})'(-\mathfrak{g}\m{id}_{\R^n}\cdot e_n)e_1=\mathfrak{g}^{-1}\varrho'e_1.	    
	\end{equation}

 The $P$ piece is, of course, the remainder of the derivative of $\Psi$ and has the formula
	\begin{equation}\label{the first bit}\index{\textbf{Linear maps}!31@$\overset{w_0,\gam_0}{P}$}
		\overset{w_0,\gam_0}{P}(q,u,\eta,\gam)=\sp{\overset{w_0}{P^1}(q,u,\eta),\overset{w_0,\gam_0}{P^2}(q,u,\eta,\gam),\overset{w_0,\gam_0}{P^3}(q,u,\eta,\gam)}
	\end{equation}
	where
	\begin{equation}\label{salt and vinegar 1}
		\overset{w_0}{P^1}(q,u,\eta)=\grad\cdot((\sig_{q_0,\eta_0}-\varrho)(u-\dot{M}[\eta]e_1))-\mathfrak{g}\grad\cdot((v_{w_0}-\varrho' e_1/\mathfrak{g})\mathcal{E}\eta)-\grad\cdot(\varrho\dot{M}[\eta]e_1)-\grad\cdot(\varrho'\mathcal{E}\eta e_1),
	\end{equation}
	$\dot{M}[\eta] : \Omega \to \R^{n \times n}$ is the matrix field given by
	\begin{equation}\label{Mdot_def}\index{\textbf{Linear maps}!07@$\dot{M}[\eta]$}
		\dot{M}[\eta]=\bpm \pd_n\mathcal{E}\eta I_{(n-1)\times(n-1)}&0_{(n-1)\times 1}\\-\mathcal{E}(\grad_{\|}\eta)&0\epm,
	\end{equation}
	and
	\begin{equation}
		\overset{w_0,\gam_0}{P^j}(q,u,\eta,\gam)=\tp{D\Psi_j(w_0,\gam_0)-D\Psi_j(0,0,0,\gam_0)}(q,u,\eta,0)+D\Psi_j(w_0,\gam_0)(0,0,0,\gam)
	\end{equation}
	for $j\in\tcb{2,3}$.  

	Before we define a similar decomposition of the $D\Phi$ piece of $D\Bar{\Psi}$, we will prove some basic properties about the $A+P$ decomposition. First, we have the following lemma which, aside from the final item, is a reprise of Lemma~\ref{lem on fundamental decomposition of continuity equation-like vector fields}.
	
	\begin{lem}[Properties of the derivative loss vector field]\label{properties of the principal parts vector field}
		Let $0<\rho\le\rho_{\m{WD}}$, where the latter is defined in Theorem~\ref{thm on smooth tameness of the nonlinear operator}, and $w_0=(q_0,u_0,\eta_0)\in B_{\X^{1+\tfloor{n/2}}}(0,\rho)\cap\X^\infty$.  Define the vector field $v_{w_0}: \Omega \to \R^n $ as in \eqref{definition of the vector field vnaught}.  There exists a decomposition
		\begin{equation}\label{fundamental decomposition of the vector field v}\index{\textbf{Nonlinear maps}!160@$v_{w_0}$}
			v_{w_0}=\mathfrak{g}^{-1}\varrho'e_1+v^{(1)}_{q_0,u_0,\eta_0}+v^{(2)}_{\eta_0}
		\end{equation}
		such that the following hold.
		\begin{enumerate}
			\item The vector field $v^{(1)}_{q_0,u_0,\eta_0}$ has vanishing normal trace, $\m{Tr}_{\pd\Omega}(v^{(1)}_{q_0,u_0,\eta_0}\cdot e_n)=0$, satisfies the inclusion $v^{(1)}_{q_0,u_0,\eta_0}\in H^\infty(\Omega;\R^n)$, and obeys the estimates
			\begin{equation}
				\tnorm{v_{q_0,u_0,\eta_0}^{(1)}}_{H^{1+s}}\lesssim\begin{cases}
					\rho&\text{if }s=1+\tfloor{n/2},\\
					\tnorm{q_0,u_0,\eta_0}_{\X_{s}}&\text{if }s>1+\tfloor{n/2}.
				\end{cases}
			\end{equation}
			\item The vector field $v^{(2)}_{\eta_0}$ is parallel to $e_1$, satisfies the inclusion
			\begin{equation}
				v^{(2)}_{\eta_0}\cdot e_1\in W^{\infty,\infty}((0,b);\mathcal{H}^{0}_{(r_{n-1})}(\Sigma))\emb W^{\infty,\infty}(\Omega),
			\end{equation}
			and obeys the estimates
			\begin{equation}
				\tnorm{v^{(2)}_{\eta_0}\cdot e_1}_{W^{s,\infty}((0,b);\mathcal{H}^0_{(r_{n-1})}(\Sigma))}\lesssim\rho \text{ for } s\ge 1+\tfloor{n/2},
			\end{equation}
			where $r_{n-1}\in\N$ is from the second item of Proposition~\ref{proposition on algebra properties} with $d=n-1$.
			\item We have the inclusion $\grad\cdot v_{w_0}\in H^{\infty}(\Omega;\R^n)$ along with the estimates
			\begin{equation}
				\tnorm{\grad\cdot v_{w_0}}_{H^{s}}\lesssim\begin{cases}
					\rho&\text{if }s=1+\tfloor{n/2},\\
					\tnorm{q_0,u_0,\eta_0}_{\X_{s}}&\text{if }s>1+\tfloor{n/2}.
				\end{cases}
			\end{equation}
			\item If $s\in\N$, $q\in H^{1+s}(\Omega)$, and $\eta\in\mathcal{H}^{1+s}(\Sigma)$, then we have that $\grad\cdot(v_{w_0}(q+\mathfrak{g}\eta))\in\hat{H}^{1+s}(\Omega)$, where the latter space is defined in~\eqref{Rhode Island} (see also Remark~\ref{remarkable remark is named what else other than shark}), as well as the estimates
			\begin{equation}\label{well-definedness estimates for the funny bit in the principal part of the continuity equation}
				\tnorm{\grad\cdot(v_{w_0}(q+\mathfrak{g}\eta))}_{\hat{H}^{s}}\lesssim\tnorm{q,\eta}_{H^{1+s}\times\mathcal{H}^{1+s}}+\begin{cases}
					0&\text{if }s\le\tfloor{n/2},\\
					\tbr{\tnorm{q_0,u_0,\eta_0}_{\X_{s}}}\tnorm{q,\eta}_{H^{1+\tfloor{n/2}}\times\mathcal{H}^{1+\tfloor{n/2}}}&\text{if }\tfloor{n/2}<s.
				\end{cases}
			\end{equation}
		\end{enumerate}
		In the above, the implicit constants depend on the dimension, the physical parameters, $s$, and $\rho_{\m{WD}}$.
	\end{lem}
	\begin{proof}
		We apply Lemma~\ref{lem on fundamental decomposition of continuity equation-like vector fields} to obtain identity~\eqref{fundamental decomposition of the vector field v} with $\varrho' e_1/\mathfrak{g}=\mathcal{M}_{(H^{-1})'}(0,0,0)$, $v^{(1)}_{q_0,u_0,\eta_0}=\mathcal{M}^{(1)}_{(H^{-1})'}(q_0,u_0,\eta_0)$, and $v^{(2)}_{\eta_0}=e_1\cdot\mathcal{M}^{(2)}_{(H^{-1})'}(\Uppi^1_{\m{L}}\eta_0) e_1$. The qualitative smoothness assertions in the first, second, and third items now follow immediately the lemma. We will prove the quantitative bounds via the fundamental theorem of calculus. Since $\mathcal{M}^{(1)}_{(H^{-1})'}(0,0,0)=0$, we have that
		\begin{equation}
			v^{(1)}_{q_0,u_0,\eta_0}=\int_0^1D\mathcal{M}^{(1)}_{(H^{-1})'}(tq_0,tu_0,t\eta_0)(q_0,u_0,\eta_0)\;\m{d}t,
		\end{equation}
		and hence, by using the strong smooth tameness assertions in Lemma~\ref{lem on fundamental decomposition of continuity equation-like vector fields} and the derivative estimates on smooth tame maps from Lemma~\ref{lem on derivative estimates on tame maps}, we find that for $s\ge 1+\tfloor{n/2}$,
		\begin{equation}
			\tnorm{v^{(1)}_{q_0,u_0,\eta_0}}_{H^{1+s}}\le\int_0^1\tnorm{D\mathcal{M}^{(1)}_{(H^{-1})'}(tq_0,tu_0,t\eta_0)[q_0,u_0,\eta_0]}_{H^{1+s}}\;\m{d}t\lesssim\tnorm{q_0,u_0,\eta_0}_{\X_{s}}.
		\end{equation}
		The same technique proves the quantitative bounds asserted in the second item. 
		
		Next, we justify the divergence estimates of the third item. By using \eqref{fundamental decomposition of the vector field v}, we compute that
		\begin{equation}
			\grad\cdot v_{w_0}=\grad\cdot v^{(1)}_{q_0,u_0,\eta_0}+\pd_1(v^{(2)}_{\eta_0}\cdot e_1).
		\end{equation}
		From this identity, the quantitative estimates of the first and second items, and the properties of band limited anisotropic Sobolev spaces from  Proposition~\ref{proposition on frequency splitting}, we deduce that for $s\ge 1+\tfloor{n/2}$,
		\begin{equation}
			\tnorm{\grad\cdot v^{(1)}_{q_0,u_0,\eta_0}}_{H^s}\lesssim\tnorm{v^{(1)}_{q_0,u_0,\eta_0}}_{H^{1+s}}\lesssim\tnorm{q_0,u_0,\eta_0}_{\X_{s}}
		\end{equation}
		and
		\begin{equation}
			\tnorm{\pd_1(v^{(2)}_{\eta_0}\cdot e_1)}_{H^{1+s}}\lesssim\tnorm{v^{(2)}_{\eta_0}\cdot e_1}_{W^{1+s,\infty}((0,b);\mathcal{H}^0_{(r_{n-1})}(\Sigma))}\lesssim\rho.
		\end{equation}
		These complete the proof of the third item.
		
		Finally, to prove the fourth item we first note that \eqref{fundamental decomposition of the vector field v} yields the formula
		\begin{multline}
			\grad\cdot(v_{w_0}(q+\mathfrak{g}\eta))=\grad\cdot(v_{w_0}(q+\mathfrak{g}\Uppi^1_{\m{H}}\eta))+\mathfrak{g}\grad\cdot(v_{q_0,u_0,\eta_0}^{(1)}\Uppi^1_{\m{L}}\eta)\\
			+\mathfrak{g}\pd_1(e_1\cdot v^{(2)}_{\eta_0}\Uppi^1_{\m{L}}\eta)+\varrho'\pd_1\Uppi^1_{\m{L}}\eta=\bf{I}+\bf{II}+\bf{III}+\bf{IV}.
		\end{multline}
		For $\bf{I}$ we use Proposition~\ref{proposition on frequency splitting}, the first and second items above, and Corollary~\ref{corollary on tame estimates on simple multipliers} to estimate
		\begin{multline}\label{zappy piglet 1}
			\tnorm{\bf{I}}_{H^s}\lesssim\tnorm{v_{w_0}(q+\mathfrak{g}\Uppi^1_{\m{H}}\eta)}_{H^{1+s}}\le\tnorm{v^{(1)}_{q_0,u_0,\eta_0}(q+\mathfrak{g}\Uppi^1_{\m{H}}\eta)}_{H^{1+s}}+\tnorm{(v^{(2)}_{\eta_0}+\varrho'e_1/\mathfrak{g})(q+\mathfrak{g}\eta_{\m{H}})}_{H^{1+s}}\\\lesssim\tnorm{q,\eta}_{H^{1+s}\times\mathcal{H}^{1+s}}+\begin{cases}
				0&\text{if }s\le\tfloor{n/2},\\
				\tbr{\tnorm{q_0,u_0,\eta_0}_{\X_s}}\tnorm{q,\eta}_{H^{1+\tfloor{n/2}}\times\mathcal{H}^{1+\tfloor{n/2}}}&\text{if }\tfloor{n/2}<s.
			\end{cases}
		\end{multline}
		On the other hand, since $\m{Tr}_{\pd\Omega}(v_{q_0,u_0,\eta_0}\cdot e_n)=0$, we use estimate~\eqref{driving down the road yesterday} from Proposition~\ref{prop on refined divergence compatibility estimate} to bound
		\begin{equation}\label{zappy piglet 2}
			\bsb{\int_0^b\bf{I}(\cdot,y)\;\m{d}y}_{\dot{H}^{-1}}\lesssim\tnorm{v_{w_0}(q+\mathfrak{g}\Uppi^1_{\m{H}}\eta)}_{L^2}\lesssim\tnorm{q,\eta}_{L^2\times\mathcal{H}^0}.
		\end{equation}
		From~\eqref{zappy piglet 1} and~\eqref{zappy piglet 2} we deduce that
		\begin{equation}
			\tnorm{\bf{I}}_{\hat{H}^{s}}\lesssim\tnorm{q,\eta}_{H^{1+s}\times\mathcal{H}^{1+s}}+\begin{cases}
				0&\text{if }s\le\tfloor{n/2},\\
				\tbr{\tnorm{q_0,u_0,\eta_0}_{\X_s}}\tnorm{q,\eta}_{H^{1+\tfloor{n/2}}\times\mathcal{H}^{1+\tfloor{n/2}}}&\text{if }\tfloor{n/2}<s.
			\end{cases}
		\end{equation}
		For $\bf{II}$, we use the first item along with Proposition~\ref{proposition on frequency splitting} and Corollary~\ref{corollary on tame estimates on simple multipliers} again to estimate
		\begin{equation}\label{zappy piglet 3}
			\tnorm{\bf{II}}_{H^s}\lesssim\tnorm{v^{(1)}_{q_0,u_0,\eta_0}}_{L^2}\tnorm{\Uppi^1_{\m{L}}\eta}_{W^{1+s,\infty}}+\tnorm{v^{(1)}_{q_0,u_0,\eta_0}}_{H^{1+s}}\tnorm{\Uppi^1_{\m{L}}\eta}_{L^\infty}
			\lesssim\tnorm{q_0,u_0,\eta_0}_{\X_{s}}\tnorm{\Uppi^1_{\m{L}}\eta}_{\mathcal{H}^0}.
		\end{equation}
		On the other hand, since $\m{Tr}_{\pd\Omega}(v_{q_0,u_0,\eta_0}^{(1)}\cdot e_n)=0$, we have
		\begin{equation}\label{zappy piglet 4}
			\bsb{\int_0^b\bf{II}(\cdot,y)\;\m{d}y}_{\dot{H^{-1}}}\lesssim\tnorm{v^{(1)}_{q_0,u_0,\eta_0}\Uppi^1_{\m{L}}\eta}_{L^2}\lesssim\tnorm{\Uppi^1_{\m{L}}\eta}_{\mathcal{H}^0}.
		\end{equation}
		From~\eqref{zappy piglet 3} and~\eqref{zappy piglet 4} we deduce that $\tnorm{\bf{II}}_{\hat{H}^s}$ is controlled by the right hand side of~\eqref{well-definedness estimates for the funny bit in the principal part of the continuity equation}. We now estimate $\bf{III}$, using the second item and the algebraic properties of the anisotropic Sobolev spaces from Proposition~\ref{proposition on algebra properties} to get
		\begin{equation}\label{zappy piglet 5}
			\tnorm{\bf{III}}_{H^s}\lesssim\tnorm{e_1\cdot v^{(2)}_{\eta_{0}}\Uppi^1_{\m{L}}\eta}_{W^{s,\infty}((0,b);\mathcal{H}^0_{(r_{n-1}+1)}(\Sigma))}\lesssim\tnorm{e_1\cdot v^{(2)}_{\eta_0}}_{W^{s,\infty}\mathcal{H}^0_{(r_{n-1})}}\tnorm{\Uppi^1_{\m{L}}\eta}_{\mathcal{H}^0}\lesssim\tnorm{\Uppi^1_{\m{L}}\eta}_{\mathcal{H}^0}
		\end{equation}
		and
		\begin{equation}\label{zappy piglet 6}
			\bsb{\int_0^b\bf{III}(\cdot,y)\;\m{d}y}_{\dot{H}^{-1}}\le\bnorm{\Uppi^1_{\m{L}}\eta\int_0^be_1\cdot v^{(2)}_{\eta_0}(\cdot,y)\;\m{d}y}_{\mathcal{H}^0}\lesssim\int_0^b\tnorm{e_1\cdot v^{(2)}_{\eta_0}(\cdot,y)}_{\mathcal{H}^0}\;\m{d}y\cdot\tnorm{\Uppi^1_{\m{L}}\eta}_{\mathcal{H}^0}\lesssim\tnorm{\Uppi^1_{\m{L}}\eta}_{\mathcal{H}^0}.
		\end{equation}
		The bounds \eqref{zappy piglet 5} and~\eqref{zappy piglet 6} imply that $\tnorm{\bf{III}}_{\hat{H}^s}$ is controlled by the right hand side of~\eqref{well-definedness estimates for the funny bit in the principal part of the continuity equation}. The norm $\tnorm{\bf{IV}}_{\hat{H}^s}$ is trivially controlled by the same quantity thanks to~\eqref{the norm on the anisotropic Sobolev spaces}. The proof of the fourth item is then complete upon synthesizing these bounds on $\bf{I}$, $\bf{II}$, $\bf{III}$, and $\bf{IV}$.
	\end{proof}

	We now introduce a new scale of adapted spaces pertinent to the analysis of the $A$ piece of $D\Psi$. As we will see, these are the natural domains on which $A$ is an isomorphism of Banach spaces. To define the spaces, we simply build an extra condition into the spaces $\tcb{\X^s}_{s\in\N}$, which we recall were defined in \eqref{domain banach scales}, to create a new family in which the effect of the derivative loss is mitigated. What is notable is that these domains depend on the background point $w_0=(q_0,u_0,\eta_0)$ in a non-trivial way. For $s\in\tcb{-1,0}\cup\R^+$ and $(q_0,u_0,\eta_0)$ as in the hypotheses of Lemma~\ref{properties of the principal parts vector field} we define the space
	\begin{equation}\label{Oregon}\index{\textbf{Function spaces}!12@$\overset{q_0,u_0,\eta_0}{\X^s}$}
		\overset{q_0,u_0,\eta_0}{\X^s}=\tcb{(q,u,\eta)\in\X^s\;:\;\grad\cdot\tp{v_{w_0} q}\in H^{1+s}(\Omega)}.
	\end{equation}
	and equip it with the norm
	\begin{equation}\label{adapted_space_norm_def}
	\tnorm{q,u,\eta}_{\overset{q_0,u_0,\eta_0}{\X^s}}=\sqrt{\tnorm{q,u,\eta}_{\X_s}^2+\tnorm{\grad\cdot(v_{w_0}q)}_{H^{1+s}}^2}.	    
	\end{equation}

 \begin{rmk}
     A simple modification of the  proof of the first item of Proposition~\ref{lem on density of bounded support functions in adapted Sobolev spaces} reveals that the spaces $\overset{q_0,u_0,\eta_0}{\X^s}$ are Hilbert.
 \end{rmk}

	The following result captures that the $A+P$ decomposition of $D\Psi$ from \eqref{A+P_splitting_of_DPsi} is such that $A$ contains all of the derivative loss and $P$ is a small correction term without derivative loss. Note that the conclusions in what follows are stronger than what can be deduced from a direct application of the smooth-tameness verification result, namely Theorem~\ref{thm on smooth tameness of the principal part nonlinear operator}.
	
	\begin{prop}[Properties of the $A+P$ decomposition of $D\Psi$]\label{prop on properties of the A+P decomposition}
		Let $\rho\in\R^+$ and $w_0=(q_0,u_0,\eta_0)$ be as in the hypotheses of Lemma~\ref{properties of the principal parts vector field} and let $I\Subset\R^+$ be an interval with $\gam_0\in I$. The following hold.
		\begin{enumerate}
			\item For every $s\in\N$, the map $\overset{w_0,\gam_0}{A}:\overset{q_0,u_0,\eta_0}{\X^s}\to\Y^s$ is well-defined and continuous.
			\item For every $\N\ni s\ge 1+\tfloor{n/2}$, the map $\overset{w_0,\gam_0}{P}:\X^s\times\R\to\Y^s$ is well-defined, continuous, and obeys the estimates
			\begin{equation}
				\snorm{\overset{w_0,\gam_0}{P}(q,u,\eta,\gam)}_{\Y^s}\lesssim\rho\tnorm{q,u,\eta}_{\X^s}+\tnorm{q_0,u_0,\eta_0}_{\X_{1+s}}\tnorm{q,u,\eta,\gam}_{\X_{\tfloor{n/2}}\times\R}.
			\end{equation}
		\end{enumerate}
		The implicit constants depend on the dimension, the physical parameters, $s$, $\rho_{\m{WD}}$, and $I$.
	\end{prop}
	\begin{proof}
		Upon inspecting the components of $\overset{w_0,\gam_0}{A}$ in~\eqref{definition of the principal part operator}, we see that the second and third are trivially well-defined and continuous, so to prove the first item we heed to Remark~\ref{remarkable remark is named what else other than shark} and reduce to showing that  the map
		\begin{equation}\label{vanilla chai tea}
			\overset{q_0,u_0,\eta_0}{\X^s}\ni(q,u,\eta)\mapsto\overset{w_0}{A^1}(q,u,\eta)\in\hat{H}^{1+s}(\Omega)
		\end{equation}
		is well-defined and continuous. As a consequence of the fourth item of Lemma~\ref{properties of the principal parts vector field}, the identity
		\begin{equation}
			\int_0^b\grad\cdot(\varrho u)(\cdot,y)\;\m{d}y=(\grad_{\|},0)\cdot\int_0^b(\varrho u)(\cdot,y)\;\m{d}y-\varrho(b)\pd_1\eta
		\end{equation}
		(which is true by virtue of the boundary condition build into the $\X^s$ spaces, as defined in \eqref{domain banach scales}), and the norm equivalence \eqref{the norm on the anisotropic Sobolev spaces} from Proposition~\ref{proposition on spatial characterization of anisobros}, we see that the operator in~\eqref{vanilla chai tea} continuously maps into $\hat{H}^0(\Omega)$.  On the other hand, by employing the first, second, and third items of Lemma~\ref{properties of the principal parts vector field}, the equality
		\begin{equation}\label{lechee tastes to me as i taste to the lechee}
			\grad\cdot(v_{w_0}(q+\mathfrak{g}\eta))=\mathfrak{g}\grad\cdot(v_{w_0}\eta)+\grad\cdot(v_{w_0}q),
		\end{equation}
		and the definition of the norm in \eqref{adapted_space_norm_def}, we readily deduce that $\overset{w_0}{A^1}$ maps boundedly into $H^{1+s}(\Omega)$. This completes the proof of the first item.
		
		We now turn to the proof of the second item. Recall from~\eqref{the first bit} that $\overset{w_0,\gam_0}{P}$ has components $\overset{w_0,\gam_0}{P^i}$ for $i \in \{1,2,3\}$.  We divide the remainder of the proof into three steps, each of which handles one of these components.

		\textbf{Step 1}: Estimates on $\overset{w_0}{P^1}$. We first aim to exploit some hidden cancellation by observing that the sum of the final two terms in~\eqref{salt and vinegar 1} vanishes. Indeed, since $\dot{M}[\eta]e_1=\pd_n\mathcal{E}\eta e_1-\pd_1\mathcal{E}\eta e_n$ we have that $\grad\cdot(\dot{M}[\eta]e_1)=0$, and therefore
		\begin{equation}
			-\grad\cdot(\varrho'  \mathcal{E}\eta e_1) - \grad\cdot(\varrho\dot{M}[\eta]e_1)=-\varrho'\pd_1\mathcal{E}\eta-\varrho'(\pd_n\mathcal{E}\eta e_1-\pd_1\mathcal{E}\eta e_n)\cdot e_n=0.
		\end{equation}
		Thus, we have the more accessible formula
		\begin{equation}\label{the more accessible formula}
			\overset{w_0}{P^1}(u,\eta)=\grad\cdot((\sig_{q_0,\eta_0}-\varrho)(u-\dot{M}[\eta]e_1))-\mathfrak{g}\grad\cdot((v_{w_0}-\varrho'e_1/\mathfrak{g})\mathcal{E}\eta)=\bf{I}+\bf{II},
		\end{equation}
		and we will deal with $\bf{I}$ and $\bf{II}$ separately.
		
		The term $\bf{I}$ is handled with the Taylor expansion trick of Lemma~\ref{lem on taylor expansion trick}, i.e. we decompose $\sig_{q_0,\eta_0}-\varrho=\sig^{(1)}_{q_0,\eta_0}+\sig^{(2)}_{\eta_0}$, where
		\begin{equation}
			\sig^{(1)}_{q_0,\eta_0}=\mathcal{N}^{(1)}_{H^{-1}}(q_0,\eta_0)
			 \text{ and }  
			\sig^{(2)}_{\eta_0}=\mathcal{N}^{(2)}_{H^{-1}}(\Uppi^1_{\m{L}}\eta).
		\end{equation} 
		Employing the above decomposition and recalling \eqref{Mdot_def}, we may further rewrite
		\begin{equation}\label{overnight oats}
			\bf{I}=\grad\cdot((\sig_{q_0,\eta_0}-\varrho)(u+\pd_1\mathcal{E}\eta e_n)-\sig^{(1)}_{q_0,\eta_0}\pd_n\mathcal{E}\eta e_1-\sig^{(2)}_{\eta_0}\pd_n\mathcal{E}_0\Uppi^1_{\m{H}}\eta e_1)-\pd_1(\sig^{(2)}_{\eta_0}\Uppi^1_{\m{L}}\eta)/b=\grad\cdot\bf{I}_1+\pd_1\bf{I}_2.
		\end{equation}
		To handle $\nabla \cdot \bf{I}_1$ we will first derive a Sobolev estimate for $\bf{I}_1$.  Indeed, thanks to multiple applications of Corollary~\ref{corollary on tame estimates on simple multipliers} and Lemma~\ref{lem on mapping properties of the Poisson extension operator variants} , we are free to bound, for any $s\in\N$,
		\begin{multline}\label{oh i repeat}
			\tnorm{\bf{I}_1}_{H^s}\lesssim\tnorm{\sig_{q_0,\eta_0}^{(1)},\sig_{\eta_0}^{(2)}}_{H^{1+\lfloor n/2\rfloor}\times W^{1+\lfloor n/2\rfloor,\infty}}\tnorm{u,\eta}_{H^s\times\mathcal{H}^{1/2+s}}\\+\begin{cases}
				0&\text{if }s\le 1+\lfloor n/2\rfloor,\\
				\tnorm{\sig^{(1)}_{q_0,\eta_0},\sig^{(2)}_{\eta_0}}_{H^s\times W^{s,\infty}}\tnorm{u,\eta}_{H^{1+\lfloor n/2\rfloor}\times\mathcal{H}^{3/2+\lfloor n/2\rfloor}}&\text{if }1+\lfloor n/2\rfloor<s.
			\end{cases}
		\end{multline}
		Then we use Lemma~\ref{lem on taylor expansion trick} combined with the fundamental theorem of calculus as in the first part of the proof of Lemma~\ref{properties of the principal parts vector field} to acquire the bounds
		\begin{equation}\label{dont ya think thats kinda neat}
			\tnorm{\sig^{(1)}_{q_0,\eta_0},\sig^{(2)}_{\eta_0}}_{H^s\times W^{s,\infty}}\lesssim\begin{cases}
				\tnorm{q_0,u_0,\eta_0}_{\X_{1+\lfloor n/2\rfloor}}&\text{if }s\le 2+\lfloor n/2\rfloor,\\
				\tnorm{q_0,u_0,\eta_0}_{\X_{s-1}}&\text{if }2+\lfloor n/2\rfloor<s.
			\end{cases}
		\end{equation}
		Upon synthesizing the bounds~\eqref{oh i repeat} and~\eqref{dont ya think thats kinda neat}, we find that
		\begin{equation}\label{give thanks}
			\tnorm{\bf{I}_1}_{H^s}\lesssim\rho\tnorm{q,u,\eta}_{\X_{s-2}}+\begin{cases}
				0&\text{if }s\le2+\lfloor n/2\rfloor,\\
				\tnorm{q_0,u_0,\eta_0}_{\X_{s-1}}\tnorm{q,u,\eta}_{\X_{\lfloor n/2\rfloor-1}}&\text{if }2+\lfloor n/2\rfloor<s,
			\end{cases}
		\end{equation}
		which is the aforementioned Sobolev estimate.  Then, in light of \eqref{give thanks}, the fact that $\m{Tr}_{\pd\Omega}(\bf{I}_1\cdot e_n)=0$, and divergence - normal trace compatibility estimates from Proposition~\ref{prop on divergence-normal trace}, we arrive at the bound
		\begin{equation}\label{a well respected man}
			\tnorm{\grad\cdot\bf{I}_1}_{\hat{H}^{1+s}}\le\tnorm{\bf{I}_1}_{H^{2+s}}\lesssim\rho\tnorm{q,u,\eta}_{\X_{s}}+\begin{cases}
				0&\text{if }s\le\lfloor n/2\rfloor,\\
				\tnorm{q_0,u_0,\eta_0}_{\X_{1+s}}\tnorm{q,u,\eta}_{\X_{\lfloor n/2\rfloor-1}}&\text{if }\lfloor n/2\rfloor<s.
			\end{cases}
		\end{equation}		
		Having dispatched $\bf{I}_1$, we turn our attention to the $\bf{I}_2$ term in~\eqref{overnight oats}.  The product $\sig^{(2)}_{\eta_0}\Uppi^1_{\m{L}}\eta$ is handled via Lemma~\ref{lem on taylor expansion trick} and the algebra properties of band limited members of $\mathcal{H}^0(\Sigma)$ enumerated in Proposition~\ref{proposition on algebra properties}, applied to $\sig^{(2)}_{\eta_0}(\cdot,y) \Uppi^1_{\m{L}}\eta$ for $y \in [0,b]$:
		\begin{multline}\label{she said in a dark brown voice}
			\tnorm{\pd_1\bf{I}_2}_{\hat{H}^{1+s}} \lesssim \bsb{ \partial_1 \bp{\Uppi^1_{\m{L}}\eta\int_0^b\sig^{(2)}_{\eta_0}(\cdot,y)\;\m{d}y} }_{\dot{H}^{-1}} +\max_{0\le j \le 1+\upnu}\sup_{y\in[0,b]} \tnorm{\pd_1\pd_n^j\bf{I}_2(\cdot,y)}_{H^{1+s}(\R^{n-1})}\\
			\lesssim\bnorm{\int_0^b\sig^{(2)}_{\eta_0}(\cdot,y)\;\m{d}y}_{\mathcal{H}^0}\tnorm{\Uppi^1_{\m{L}}\eta}_{\mathcal{H}^0}
			+\max_{0\le j\le 1+s} \sup_{y\in[0,b]} \tsb{\pd_1\pd_n^j(\sig^{(2)}_{\eta_0}(\cdot,y) \Uppi^1_{\m{L}}\eta)}_{\dot{H}^{-1}}\\
			\lesssim \max_{0\le j \le 1+s} \sup_{y\in[0,b]} \tnorm{\pd_n^j\sig^{(2)}_{\eta_0}(\cdot,y)}_{\mathcal{H}^0}\tnorm{\Uppi^1_{\m{L}}\eta}_{\mathcal{H}^0}\lesssim\rho\tnorm{\Uppi^1_{\m{L}}\eta}_{\mathcal{H}^0},
		\end{multline}
		where in deducing the final inequality  we again use the tame mapping properties of $\sig^{(2)}_{\eta_0}$ from Lemma~\ref{lem on taylor expansion trick}, combined with the fundamental theorem of calculus as before. This completes the handling of $\bf{I}_2$, and hence of $\bf{I}$.
		
		Next, we consider $\bf{II}$ from equation~\eqref{the more accessible formula}. Thanks to the decomposition of $v_{w_0}$ from Lemma~\ref{properties of the principal parts vector field}, $\bf{II}$ has the further decomposition
		\begin{equation}\label{he ignores the girl}
			\bf{II}=-\mathfrak{g}\grad\cdot((v^{(1)}_{q_0,u_0,\eta_0}+v^{(2)}_{\eta_0})\mathcal{E}_0\Uppi^1_{\m{H}}\eta+v^{(1)}_{q_0,u_0,\eta_0}\mathcal{E}\Uppi^1_{\m{L}}\eta)-\mathfrak{g}\pd_1(\m{id}_{\R^n}\cdot e_n v^{(2)}_{\eta_0}\cdot e_1\Uppi^1_{\m{L}}\eta)/b=\grad\cdot\bf{II}_1+\pd_1\bf{II}_2.
		\end{equation}
		As in the proof of~\eqref{give thanks}, we use Corollary~\ref{corollary on tame estimates on simple multipliers} and the mapping properties of $v^{(1)}_{q_0,u_0,\eta_0}$ and $v^{(2)}_{\eta_0}$ from Lemma~\ref{properties of the principal parts vector field} to see that for any $s\in\N$,
		\begin{multline}\label{just remembered}
			\tnorm{\bf{II}_1}_{H^s}\lesssim\tnorm{v^{(1)}_{q_0,u_0,\eta_0},v^{(2)}_{\eta_0}}_{H^{1+\lfloor n/2\rfloor}\times W^{1+\lfloor n/2\rfloor,\infty}}\tnorm{\eta}_{\mathcal{H}^{s-1/2}}\\+\begin{cases}
				0&\text{if }s\le 1+\lfloor n/2\rfloor,\\
				\tnorm{v^{(1)}_{q_0,u_0,\eta_0},v^{(2)}_{\eta_0}}_{H^s\times W^{s,\infty}}\tnorm{\eta}_{\mathcal{H}^{1/2+\lfloor n/2\rfloor}}&\text{if }1+\lfloor n/2\rfloor<s,
			\end{cases}\\
			\lesssim\rho\tnorm{q,u,\eta}_{\X_{s-3}}+\begin{cases}
				0&\text{if }s\le 2+\lfloor n/2\rfloor,\\
				\tnorm{q_0,u_0,\eta_0}_{\X_{s-1}}\tnorm{q,u,\eta}_{\X_{\lfloor n/2\rfloor-2}}&\text{if }2+\lfloor n/2\rfloor<s.
			\end{cases}
		\end{multline}
		In turn, \eqref{just remembered}, the fact that $\m{Tr}_{\pd\Omega}(\bf{II}_1\cdot e_n)=0$, and Proposition~\ref{prop on divergence-normal trace} provide the bound
		\begin{equation}
			\tnorm{\grad\cdot\bf{II}_1}_{H^{1+s}}\lesssim\tnorm{\bf{II}_1}_{H^{2+s}}\lesssim\rho\tnorm{q,u,\eta}_{\X_{s-1}}+\begin{cases}
				0&\text{if }\upnu\le\lfloor n/2\rfloor,\\
				\tnorm{q_0,u_0,\eta_0}_{\X_{1+\upnu}}\tnorm{q,u,\eta}_{\X_{\lfloor n/2\rfloor-1}}&\text{if }\lfloor n/2\rfloor<\upnu.
			\end{cases}
		\end{equation}
		For $\bf{II}_2$, we argue as in \eqref{she said in a dark brown voice}, estimating the product $(\m{id}_{\R^n}\cdot e_n)(v^{(2)}_{\eta_0}\cdot e_1) \eta_{\m{L}}$ by invoking the second item from Lemma~\ref{properties of the principal parts vector field} as well as the algebra properties of band-limited members of $\mathcal{H}^0(\R^{n-1})$ from Proposition~\ref{proposition on algebra properties}; this results in the bound
		\begin{equation}\label{where do i go walk}
			\tnorm{\pd_1\bf{II}_2}_{\hat{H}^{1+s}}\lesssim\rho\tnorm{\Uppi^1_{\m{L}}\eta}_{\mathcal{H}^0}.
		\end{equation}
		
		Finally, we synthesize \eqref{the more accessible formula}, \eqref{overnight oats}, \eqref{a well respected man}, \eqref{she said in a dark brown voice}, \eqref{he ignores the girl}, \eqref{just remembered}, and \eqref{where do i go walk} to get the $\overset{w_0}{P^1}$ estimate
		\begin{equation}
			\snorm{\overset{q_0,u_0,\eta_0}{P^1}(u,\eta)}_{\hat{H}^{1+s}}\lesssim\rho\tnorm{q,u,\eta}_{\X_s}+\begin{cases}
				0&\text{if }s\le\lfloor n/2\rfloor,\\
				\tnorm{q_0,u_0,\eta_0}_{\X_{1+s}}\tnorm{q,u,\eta}_{\X_{\lfloor n/2\rfloor-1}}&\text{if }\lfloor n/2\rfloor<s.
			\end{cases}
		\end{equation}

		\textbf{Step 2}: Estimates on $\overset{w_0,\gam_0}{P^2}$. This is simpler than the previous step, as we can use the tame calculus conclusions of Proposition~\ref{prop on smooth tameness of the momentum equation 1}. Recall that
		\begin{equation}
			\overset{w_0,\gam_0}{P^2}(q,u,\eta,\gam)=(D\Psi_2(w_0,\gam_0)-D\Psi_2(0,0,0,\gam_0))(q,u,\eta,0)+D\Psi_2(w_0,\gam_0)(0,0,0,\gam)=\bf{J}_1+\bf{J}_2,
		\end{equation}
		where $\Psi_2$ is the nonlinear map defined in~\eqref{definition of the nonlinear map Psi2}. We may use that $\Psi_2$ is  $C^2$ paired with the fundamental theorem of calculus to write
		\begin{equation}
			\bf{J}_1=\int_0^1D^2\Psi_2(tw_0,\gam_0)((q_0,u_0,\eta_0,0),(q,u,\eta,0))\;\m{d}t.
		\end{equation}
		Hence, by the strong tame estimates on the second derivative from  Proposition~\ref{prop on smooth tameness of the momentum equation 1}, the log-convexity of the norm of the $\X_s$-spaces (see Lemma~\ref{lem on log-convexity of the norms}), and Young's inequality we get that for $\upnu\ge 1+\tfloor{n/2}$:
		\begin{equation}\label{Ohio}
			\tnorm{\bf{J}_1}_{H^s}\lesssim\rho\tnorm{q,u,\eta}_{\X_s}+\tnorm{q_0,u_0,\eta_0}_{\X_s}\tnorm{q,u,\eta}_{\X_{1+\tfloor{n/2}}}.
		\end{equation}
		On the other hand, by inspection we see that $D\Psi_2(0,\gam_0)(0,0,0,\gam)=0$. Thus by a similar fundamental theorem of calculus argument, we have that
		\begin{equation}
			\bf{J}_2=\int_0^1D^2\Psi_2(tw_0,\gam_0)((w_0,0),(0,0,0,\gam))\;\m{d}t.
		\end{equation}
		This lets us use the strong tame estimates on the second derivative again to bound
		\begin{equation}\label{Oklahoma}
			\tnorm{\bf{J}_2}_{H^s}\lesssim|\gam|\tnorm{q_0,u_0,\eta_0}_{\X_{s}}.
		\end{equation}
		By combining equations~\eqref{Ohio} and~\eqref{Oklahoma}, we obtain the stated bounds for $\overset{w_0,\gam_0}{P^2}$.
		
		\textbf{Step 3}: Estimates on $\overset{w_0,\gam_0}{P^3}$. We perform the same analysis, utilizing the fundamental theorem of calculus and the tame $C^2$-estimates as in the second step, but this time we employ Proposition~\ref{prop on smooth tameness of the dynamic boundary condition 1} in place of Proposition~\ref{prop on smooth tameness of the momentum equation 1}.
	\end{proof}

	The remainder of this subsection is devoted to the decomposition of the $D\Phi$ piece of $D\Bar{\Psi}$, which is derived from equations~\eqref{definition of the nonlinear map Phi2} and~\eqref{definition of the nonlinear map Phi3}, into subcomponents $Q$ and $R$. The key observation here is that $\Phi:B_{\X_{2+\tfloor{n/2}}}(0,\rho_{\m{WD}})\times\W_{3+\tfloor{n/2}}\to\Y^0$ is linear in the second factor, and hence for $\theta_0$ as in \eqref{theta_0_def} we can write, 
	\begin{equation}
		D\Phi(\theta_0)(q,u,\eta,\mathcal{T},\mathcal{G},\mathcal{F})=\overset{\theta_0}{Q}(q,u,\eta) + \overset{q_0,u_0,\eta_0}{R}(\mathcal{T},\mathcal{G},\mathcal{F})
	\end{equation}
	where
	\begin{equation}\label{definition of the operator Q}\index{\textbf{Linear maps}!40@$\overset{\theta_0}{Q}$}
		\overset{\theta_0}{Q}(q,u,\eta)= D\Phi(\theta_0)(q,u,\eta,0,0,0)
	\end{equation}
	and
	\begin{equation}\label{definition of the operator R}\index{\textbf{Linear maps}!41@$\overset{q_0,u_0,\eta_0}{R}$}
		\overset{q_0,u_0,\eta_0}{R}(\mathcal{T},\mathcal{G},\mathcal{F})=D\Phi(q_0,u_0,\eta_0,0,0,0)(0,0,0,\mathcal{T},\mathcal{G},\mathcal{F}).
	\end{equation}
 
	The following result reveals the point of the $Q+R$ decomposition of $D\Phi$: $Q$ is small when the background is small, and $R$ is independent of $(q,u,\eta)$ and thus enjoys a useful decoupling. In contrast with the $A+P$ decomposition (Proposition~\ref{prop on properties of the A+P decomposition}), the next result is less delicate and only relies on the smooth-tameness results from Propositions~\ref{prop on smooth tameness of the momentum equation 2} and~\ref{smooth tameness of dynamic boundary condition 2}.

	\begin{prop}[Properties of the $Q+R$ decomposition of $D\Phi$]\label{prop on the Q+R decomposition of DPhi}
		Let $0<\rho\le\rho_{\m{WD}}$, where $\rho_{\m{WD}}$ is defined in Theorem~\ref{thm on smooth tameness of the nonlinear operator}, and
		\begin{equation}
			\theta_0=(q_0,u_0,\eta_0,\mathcal{T}_0,\mathcal{G}_0,\mathcal{F}_0)\in \tp{B_{\X_{2+\tfloor{n/2}}}(0,\rho)\times B_{\W_{3+\tfloor{n/2}}}(0,\rho)}\cap(\X_\infty\times\W_{\infty}).
		\end{equation}
		The following hold for $\N\ni s\ge 2+\tfloor{n/2}$.
		\begin{enumerate}
			\item  The map $\overset{\theta_0}{Q}:\X_{s}\to\Y^s$ is well-defined, continuous, and obeys the estimate 
			\begin{equation}
				\snorm{\overset{\theta_0}{Q}(q,u,\eta)}_{\Y^s}\lesssim\rho\tnorm{q,u,\eta}_{\X_s}+\tnorm{\theta_0}_{\X_s\times\W_{1+s}}\tnorm{q,u,\eta}_{\X_{2+\tfloor{n/2}}}.
			\end{equation}
			\item The map $\overset{q_0,u_0,\eta_0}{R}:\W_{1+s}\to\Y^{s}$ is well-defined, continuous, and obeys the estimate
			\begin{equation}
				\snorm{\overset{q_0,u_0,\eta_0}{R}(\mathcal{T},\mathcal{G},\mathcal{F})}_{\Y^s}\lesssim\tnorm{\mathcal{T},\mathcal{G},\mathcal{F}}_{\W_{1+s}}+\tbr{\tnorm{q_0,u_0,\eta_0}_{\X_s}}\tnorm{\mathcal{T},\mathcal{G},\mathcal{F}}_{\W_{3+\tfloor{n/2}}}.
			\end{equation}
		\end{enumerate}
		In the above the implicit constants depend on the dimension, the physical parameters, $s$, and $\rho_{\m{WD}}$.
	\end{prop}
	\begin{proof}

		That $\overset{\theta_0}{Q}$ is well-defined, continuous, and, when $\rho=\rho_{\m{WD}}$ and $s\ge 2+\tfloor{n/2}$, obeys a tame estimate of the form
		\begin{equation}\label{basic tame estimate}
			\snorm{\overset{\theta_0}{Q}(q,u,\eta)}_{\Y^s}\lesssim\tnorm{q,u,\eta}_{\X_s}+\tbr{\tnorm{\theta_0}_{\X_s\times\W_{1+s}}}\tnorm{q,u,\eta}_{\X_{2+\tfloor{n/2}}}
		\end{equation}
		follows from the formula~\eqref{definition of the operator Q}, the tame smoothness assertions of Propositions~\ref{prop on smooth tameness of the momentum equation 2} and~\ref{smooth tameness of dynamic boundary condition 2}, and the tame estimates on derivatives from Lemma~\ref{lem on derivative estimates on tame maps}. We can improve the above estimate by leveraging the fact that for fixed $(q_0,u_0,\eta_0)$ the map $(\mathcal{T}_0,\mathcal{G}_0,\mathcal{F}_0)\mapsto\overset{\theta_0}{Q}$ is linear; indeed, by arguing as in the proof of Lemma~\ref{lem on multilinear tame maps}, we deduce from estimate~\eqref{basic tame estimate} that actually
		\begin{equation}
			\snorm{\overset{\theta_0}{Q}(q,u,\eta)}_{\Y^s}\lesssim\tnorm{\mathcal{T}_0,\mathcal{G}_0,\mathcal{F}_0}_{\W_{3+\tfloor{n/2}}}\tnorm{q,u,\eta}_{\X_s}+\tnorm{\theta_0}_{\X_s\times\W_{1+s}}\tnorm{q,u,\eta}_{\X_{2+\tfloor{n/2}}}.
		\end{equation}
		The first item now follows. The assertions and estimate of the second item are a direct consequence of formula~\eqref{definition of the operator R}, the tame smoothness assertions of Propositions~\ref{prop on smooth tameness of the momentum equation 2} and~\ref{smooth tameness of dynamic boundary condition 2},and the tame estimates on derivatives from Lemma~\ref{lem on derivative estimates on tame maps}. 
	\end{proof}
	
	\subsection{Prelude to linear analysis}\label{subsection on notation and overview for linear analysis}
	
	Recall that $\Bar{\Psi}$ is the nonlinear operator associated with the PDE~\eqref{The nonlinear equations in the right form} and is defined in \eqref{the definition of the nonlinear operator}. Our goal now is to prove that this map satisfies the hypotheses of the inverse function theorem, which are enumerated in Definition~\ref{defn of the mapping hypotheses}, Theorem~\ref{thm on nmh}, and Theorem~\ref{thm on further conclusions of the inverse function theorem}. The previous analysis of this section, in particular Lemma~\ref{lem on tameness of domain and codomain} and Theorem~\ref{thm on smooth tameness of the nonlinear operator}, shows that the `nonlinear hypotheses,' other than a trivial issue with the domain and the first item of Definition~\ref{defn of the mapping hypotheses}, of this inverse function theorem are satisfied.
	
	It only remains, then, to show that the `linear hypotheses' given in the third item of Definition~\ref{defn of the mapping hypotheses} are satisfied. In other words, we aim to prove the following assertion about the derivative, $D\Bar{\Psi}$.  Given $\N\ni s\ge2+\tfloor{n/2}$ and an interval $I\Subset\R^+$, there exists an \emph{existence and estimates} parameter $\rho_{\m{EE}}(s)\in(0,\rho_{\m{WD}}]$ (also depending on $I$), where $\rho_{\m{WD}}\in\R^+$ is from Theorem~\ref{thm on smooth tameness of the nonlinear operator}, with the property that whenever (recall that $\theta_0$ is defined in \eqref{theta_0_def})
	\begin{equation}
		(\theta_0,\gam_0)\in \tp{B_{\X^{2+\tfloor{n/2}}}(0,\rho_{\m{EE}}(s))\times B_{\W_{3+\tfloor{n/2}}}(0,\rho_{\m{EE}}(s))\times I}\cap\tp{\X^{1+s}\times\W_{2+s}\times\R^+}
	\end{equation}
	and
	\begin{equation}
		(g,f,k,\mathcal{T},\mathcal{G},\mathcal{F},\gam)\in\Y^s\times\W_{1+s}\times\R,
	\end{equation}
	there exists a unique $(q,u,\eta)\in\X^s$ solving $D\Bar{\Psi}(\theta_0,\gam_0)(q,u,\eta,\mathcal{T},\mathcal{G},\mathcal{F},\gam)=(g,f,k,\mathcal{T},\mathcal{G},\mathcal{F},\gam)$; moreover, if we set $\Xi=(g,f,k,\mathcal{T},\mathcal{G},\mathcal{F},\gam)$ and $\Theta=(q,u,\eta,\mathcal{T},\mathcal{G},\mathcal{F},\gam)$ then the solution tuple obeys the tame estimate
	\begin{equation}
		\tnorm{\Theta}_{\X^s\times\W_{1+s}\times\R}\lesssim\tnorm{\Xi}_{\Y^s\times\W_{1+s}\times\R}+\tbr{\tnorm{\theta_0}_{\X^{1+s}\times\W_{2+s}}}\tnorm{\Xi}_{\X^{2+\tfloor{n/2}}\times\W_{3+\tfloor{n/2}}\times\R}
	\end{equation}
	for an implicit constant depending only on the dimension, the various physical parameters, $s$, $I$, and $\rho_{\m{EE}}(s)$.

	Thanks to the analysis of the previous subsection, namely Propositions~\ref{prop on properties of the A+P decomposition} and~\ref{prop on the Q+R decomposition of DPhi}, we have that the behavior of $D\Bar{\Psi}$ is governed by the principal part operator $A$ defined in~\eqref{definition of the principal part operator}. Thus our above goal is essentially achieved, modulo minor supplementary analysis, as soon as we prove the following assertion. For any $s\in\N$ and $\gam_0\in I\Subset\R^+$an interval, there exists an \emph{existence and estimates principal part} parameter $\rho_{\m{EEP}}(s)\in(0,\rho_{\m{WD}}]$ (also depending on $I$) with the property that whenever
	\begin{equation}\label{such that whenever wine is red}
		w_0=(q_0,u_0,\eta_0)\in B_{\X^{2+\tfloor{n/2}}}(0,\rho_{\m{EEP}}(s))\cap\X^\infty 
		\text{ and }
		(g,f,k)\in\Y^s,
	\end{equation}
	there exists a unique $(q,u,\eta)\in\X^s$ solving $\overset{w_0,\gam_0}{A}(q,u,\eta)=(g,f,k)$; moreover, the solution obeys the tame estimates
	\begin{equation}\label{tames estimate goal for the principal part}
		\tnorm{q,u,\eta}_{\X^s}\lesssim\tnorm{g,f,k}_{\Y^s}+\begin{cases}
			0&\text{if }s\le\tfloor{n/2},\\
			\tbr{\tnorm{q_0,u_0,\eta_0}_{\X_{1+s}}}\tnorm{q,u,\eta}_{\X_{\tfloor{n/2}}}&\text{if }\tfloor{n/2}<s.
		\end{cases}
	\end{equation}
	Again we allow for implied constants to depend on the dimension, the physical parameters, $s$, $I$, and $\rho_{\m{EEP}}(s)$. We will see that a necessary condition for the estimate \eqref{tames estimate goal for the principal part} to hold is that the operator $\overset{w_0,\gam_0}{A}:\overset{q_0,u_0,\eta_0}{\X^s}\to\Y^s$ is a Banach isomorphism, where we recall the adapted spaces are defined in~\eqref{Oregon}. This expresses one of the core difficulties of the linear analysis: we either have that the family of operators $\tcb{\overset{w_0,\gam_0}{A}}\subset\mathcal{L}(\X^{1+s};\Y^s)$ are defined on the common Banach space $\X^{1+s}$ but \emph{not} isomorphisms, or else each individual operator $\overset{w_0,\gam_0}{A}$ is defined on a larger adapted Banach space $\overset{q_0,u_0,\eta_0}{\X^s}$ making it an isomorphism onto $\Y^s$, but the spaces $\tcb{\overset{q_0,u_0,\eta_0}{\X^{s}}}$ are \emph{inequivalent}. Consequently, we cannot port invertibility from one operator to the next via the method of continuity, even if we were to establish the estimates of~\eqref{tames estimate goal for the principal part} a priori.

	We overcome this difficulty via an elliptic regularization procedure which we now describe. For $m\in\N^+$ we define the $2m^{\m{th}}$-order linear elliptic differential operator
	\begin{equation}\label{the operator Lm}\index{\textbf{Linear maps}!12@$L_m$}
		L_m=(-1)^m\sum_{j=1}^n\pd_j^{2m}.
	\end{equation}
	We will consider a sequence of operators obtained by adding a vanishing contribution of this operator to the continuity equation component of $A$ and also adding a vanishing contribution of $(-\Delta_{\|})^{m-1/4}$ to the kinematic boundary condition. These new operators have the following domains: for $m,N\in\N^+$ and $s\in\tcb{-1}\cup\N$ we define the space\index{\textbf{Function spaces}!14@$\X^s_{m,N}$}
	\begin{multline}\label{the regularized spaces}
		\X^s_{m,N}=\{(q,u,\eta)\in\X_s\;:\;\m{Tr}_{\Sigma_0}(u)=0,\;\m{Tr}_{\Sigma}(u\cdot e_n)+\pd_1\eta=N^{-1}(-\Delta_{\|})^{m-1/4}\eta\\
		(q,\eta)\in H^{1+s+2m}(\Omega)\times\mathcal{H}^{1+s+2m}(\Sigma),\;\m{Tr}_{\pd\Omega}(\pd_n^mq)=\cdots=\m{Tr}_{\pd\Omega}(\pd_n^{2m-1}q)=0\}
	\end{multline}
	and endow it with the norms
	\begin{equation}
		\tnorm{q,u,\eta}_{\X^s_{m,N}}=\sqrt{\tnorm{q,u,\eta}_{\X_s}^2+N^{-2}\tnorm{q,\eta}^2_{H^{1+s+2m}\times\mathcal{H}^{1+s+2m}}},
	\end{equation}
	and
	\begin{equation}\label{the adapted norm on the regularized spaces}
		\tnorm{q,u,\eta}_{\overset{q_0,u_0,\eta_0}{\X^s_{m,N}}}=\sqrt{\tnorm{q,u,\eta}_{\X^s_{m,N}}^2+\tnorm{\grad\cdot(v_{q_0,u_0,\eta_0}q)}_{H^{1+s}}^2}.
	\end{equation}
 We note that, in light of the fourth item of Lemma~\ref{properties of the principal parts vector field},  $\tnorm{\cdot}_{\overset{q_0,u_0,\eta_0}{\X^s_{m,N}}}$ is a norm equivalent to $\tnorm{\cdot}_{\X^s_{m,N}}$,  with equivalence constants depending on $s$, $m$, $N$, and $\rho_{\m{WD}}$.

	For $s \in \N$ we then define the regularized principal part operator 
	\begin{equation}\label{the map of the regularized principal part operator}\index{\textbf{Linear maps}!31@$\overset{w_0,\gam_0}{A_{m,N}}$}
		\overset{w_0,\gam_0}{A_{m,N}}:\X^s_{m,N}\to\Y^s
	\end{equation}
	via
	\begin{equation}\label{the application of the map of the reg princ part oper and also will be known as ziggy mcpiggson pigface ponylicker frogs made of wicker but not a football kicker boy i sure do have a problem i need to stop labeling like this}
		\overset{w_0,\gam_0}{A_{m,N}}(q,u,\eta)=\tp{N^{-1}L_m(q+\mathfrak{g}\eta)+\overset{w_0,\gam_0}{A^1}(q,u,\eta),\overset{\gam_0}{A^2}(q,u,\eta),\overset{\gam_0}{A^3}(q,u,\eta)},
	\end{equation}
    where the $A^i$ terms are the same as in \eqref{definition of the principal part operator}.
    
	The upshot is that the operators $\tcb{\overset{w_0,\gam_0}{A_{m,N}}}$ are all contained in the space $\mathcal{L}(\X^s_{m,N},\Y^s)$ and, as we will see, obey a family of nice a priori estimates. Thus the method of continuity is available for the regularized operators' existence theory.
	
	Once we have existence for the regularized operators, we would like to show that we can pass to the limit as $N\to\infty$ and obtain existence for $A$. This is achieved by carefully proving $N$-independent a priori estimates for $A_{m,N}$ by inductively working up from the base case of a priori estimates for weak solutions. This inductive estimate procedure will also be done in tandem with the operator $A$ to obtain the sought-after a priori estimates of~\eqref{tames estimate goal for the principal part}.
	
	As a result, we will need to set notation for weak formulations of the operators $A$ and $A_{m,N}$. First we recall the definition of the space ${_0}H^1$ from~\eqref{zero on the left}. The operator associated with the weak formulation of the momentum equation is 
	\begin{equation}\label{definition of the I functional}\index{\textbf{Linear maps}!70@$\overset{\gam_0}{\mathscr{I}}$}
		\overset{\gam_0}{\mathscr{I}}:\X_{-1}\to({_0}H^1(\Omega;\R^n))^\ast
	\end{equation}
	defined by
	\begin{multline}
		\tbr{\overset{\gam_0}{\mathscr{I}}(q,u,\eta),w}_{({_0}H^1)^\ast,{_0}H^1}=\int_{\Omega}-\gam_0^2\varrho\pd_1u\cdot w-q\grad\cdot(\varrho w)+\mathfrak{g}\varrho\grad\eta\cdot w+\gam_0\S^\varrho u:\grad w\\
		-\varsigma\tbr{\Delta_{\|}\eta,\m{Tr}_{\Sigma}(w\cdot e_n)}_{H^{-1/2},H^{1/2}}
	\end{multline}
	for $(q,u,\eta)\in\X_{-1}$ and $w\in{_0}H^1(\Omega;\R^n)$.  We next define a family of operators associated with the weak formulation of the full problem by setting
	\begin{equation}\label{weak formulation operator}\index{\textbf{Linear maps}!71@$\overset{w_0\gam_0}{\mathscr{J}}$}
		\overset{w_0,\gam_0}{\mathscr{J}}:\overset{q_0,u_0,\eta_0}{\X^{-1}}\to\Y^{-1}
	\end{equation}
	via
	\begin{equation}
		\overset{w_0,\gam_0}{\mathscr{J}}(q,u,\eta)=\sp{\grad\cdot(\varrho u)+\grad\cdot(v_{w_0}(q+\mathfrak{g}\eta)),\overset{\gam_0}{\mathscr{I}}(q,u,\eta)},
	\end{equation}
	where we recall that the vector field $v_{w_0}$ is defined in~\eqref{definition of the vector field vnaught}. Furthermore, given $\tau\in[0,1]$ and $m,N\in\N^+$ with $m\ge 2$ we define the map
	\begin{equation}\label{regularized weak formulation operator}\index{\textbf{Linear maps}!72@$\overset{w_0\gam_0}{\mathscr{J}^\tau_{m,N}}$}
		\overset{w_0,\gam_0}{\mathscr{J}^\tau_{m,N}}:\X_{m,N}^{-1}\to\Y^{-1}
	\end{equation}
	via
	\begin{equation}
		\overset{w_0,\gam_0}{\mathscr{J}^\tau_{m,N}}(q,u,\eta)=\tp{\grad\cdot(\varrho u)+\tau\grad\cdot(v_{w_0}(q+\mathfrak{g}\eta))+N^{-1}L_m(q+\mathfrak{g}\eta),\overset{\gam_0}{\mathscr{I}}(q,u,\eta)}.
	\end{equation}

	We now give a basic result on the well-definedness and continuity of these operators.
	
	\begin{lem}[Well-definedness check for linear analysis]\label{lem on well-definedness check for linear analysis}
		Let $\rho\in\R^+$ and $w_0=(q_0,u_0,\eta_0)$ be as in the hypotheses of Lemma~\ref{properties of the principal parts vector field}. Let $I\Subset\R^+$ be an interval with $\gam_0\in I$. The following hold.
		\begin{enumerate}
			\item The weak formulation operator of~\eqref{weak formulation operator} is well-defined and bounded.
			\item For $\N\ni m\ge 2$, $N\in\N^+$, and $\tau\in[0,1]$, the regularized weak formulation operator of~\eqref{regularized weak formulation operator} is well-defined and bounded.
			\item For $\N\ni m\ge 2$ and $N\in\N^+$ the regularization of the principal part operator defined by \eqref{the map of the regularized principal part operator} is well-defined and bounded.
		\end{enumerate}
	\end{lem}
	\begin{proof}
		For the first item we see that it suffices to check that the first component of $\overset{\gam_0,w_0}{\mathscr{J}}$ is well-defined and maps boundedly into $\hat{H}^0(\Omega)$. For this we can argue in a manner similar to the first part of the proof of Proposition~\ref{prop on properties of the A+P decomposition}, where we studied~\eqref{vanilla chai tea}. The only difference is in handling the $\grad\cdot(v_{w_0}q)$ term. That this term maps boundedly into $L^2(\Omega)$ is immediate from the definition of the norm on $\overset{q_0,u_0,\eta_0}{\X^{-1}}$. We then use Proposition~\ref{prop on refined divergence compatibility estimate}, paired with the fact that $v_{w_0}\in L^\infty(\Omega;\R^n)$, to obtain that $q\mapsto\int_0^{b}\grad\cdot(v_{w_0}q)(\cdot,y)\;\m{d}y$ maps boundedly into $\dot{H}^{-1}(\Omega)$.
			
		The second and third items follow as soon as we check that 
		\begin{equation}\label{jo jo was a man}
			\X^{-1}_{m,N}\ni(q,u,\eta)\mapsto\grad\cdot(\varrho u)+\tau\grad\cdot(v_{w_0}(q+\mathfrak{g}\eta))+N^{-1}L_m(q+\mathfrak{g}\eta)\in\hat{H}^0(\Omega)
		\end{equation}
		is a well-defined and bounded linear map.  We handle the $\tau\grad\cdot(v_{w_0}(q+\mathfrak{g}\eta))$ term in the same way that we handled~\eqref{lechee tastes to me as i taste to the lechee} from Proposition~\ref{prop on properties of the A+P decomposition}. For the $N^{-1}L_m(q+\mathfrak{g}\eta)$ term we see that it obviously maps into $L^2(\Omega)$. To acquire the divergence compatibility condition, we use the Neumann boundary conditions built into the domain $\X^{-1}_{m,N}$ to compute
		\begin{equation}
			\int_0^b\f{1}{N}L_m(q+\mathfrak{g}\eta)(\cdot,y)\;\m{d}y
			= \f{1}{N}\sum_{j=1}^{n-1}\pd_j^{2m}\bp{\mathfrak{g}b\eta+\int_0^bq(\cdot,y)\;\m{d}y},
		\end{equation}
		but since $m\ge 1$, this provides the bound
		\begin{equation}
			\bsb{\int_0^b\f{1}{N}L_m(q+\mathfrak{g}\eta)(\cdot,y)\;\m{d}y}_{\dot{H}^{-1}}\lesssim\f{1}{N}\tp{\tnorm{\grad_{\|}\eta}_{H^{2m-2}}+\tnorm{q}_{H^{2m-1}}}\lesssim\tnorm{q,u,\eta}_{\X^{-1}_{m,N}}.
		\end{equation}
		Finally we handle the $\grad\cdot(\varrho u)$ term of~\eqref{jo jo was a man}. Evidently, this maps into $L^2(\Omega)$, leaving us to check the divergence compatibility condition. By integrating, we learn that
		\begin{equation}
			\int_0^b\grad\cdot(\varrho u)(\cdot,y)\;\m{d}y=(\grad_{\|},0)\cdot\int_0^b(\varrho u)(\cdot,y)\;\m{d}y-\varrho(b)\pd_1\eta+\f{\varrho(b)}{N}(-\Delta_{\|})^{m-1/4}\eta,
		\end{equation}
		and since $m\ge 2$ this yields the estimate
		\begin{equation}
			\bsb{\int_0^b\grad\cdot(\varrho u)(\cdot,y)\;\m{d}y}_{\dot{H}^{-1}}\lesssim\tnorm{u}_{L^2}+\tnorm{\eta}_{\mathcal{H}^0}+\f{1}{N}\tnorm{\grad_{\|}\eta}_{H^{2m-5/2}}\lesssim\tnorm{q,u,\eta}_{\X^{-1}_{m,N}}.
		\end{equation}
	\end{proof}
	
	We now consider the relationship between the weak and strong operators, $\mathscr{J}$, $\mathscr{J}^\tau_{m,N}$ and $A$, $A_{m,N}$,  by introducing a functional that maps the strong form of the data to the weak formulation of the data.  For $s\in\N$ we define the strong-to-weak data map
	\begin{equation}\label{definition of the K functional dude lies here rest in peace and serrano peppers}\index{\textbf{Linear maps}!80@$\mathscr{K}$}
		\mathscr{K}:H^{s}(\Omega;\R^n)\times H^{1/2+s}(\Sigma;\R^n)\to({_0}H^1(\Omega;\R^n))^\ast
	\end{equation}
	given by
	\begin{equation}
		\tbr{\mathscr{K}(f,k),w}_{({_0}H^1)^\ast,{_0}H^1}=\int_\Omega f\cdot w+\int_{\Sigma}k\cdot w,
	\end{equation}
	where $w\in{_0}H^1(\Omega;\R^n)$, $(f,k)\in H^{s}(\Omega;\R^n)\times H^{1/2+s}(\Sigma;\R^n)$, and $s\in\N$.
	
	\begin{lem}[Strong and weak solutions]\label{lem on strong solutions are weak solutions}
		Under the hypotheses of Lemma~\ref{lem on well-definedness check for linear analysis}, the following hold for $(g,f,k)\in\Y^0$.
		\begin{enumerate}
			\item If $(q,u,\eta)\in\overset{q_0,u_0,\eta_0}{\X^0}$, then $\overset{w_0,\gam_0}{A}(q,u,\eta)=(g,f,k)$ if and only if $\overset{w_0,\gam_0}{\mathscr{J}}(q,u,\eta)=(g,\mathscr{K}(f,k))$.
			\item If $m,N\in\N^+$, $m\ge 2$, and $(q,u,\eta)\in\X^0_{m,N}$, then $\overset{w_0,\gam_0}{A_{m,N}}(q,u,\eta)=(g,f,k)$ if and only if $\overset{w_0,\gam_0}{\mathscr{J}^1_{m,N}}(q,u,\eta)=(g,\mathscr{K}(f,k))$.
		\end{enumerate}
	\end{lem}
	\begin{proof}
		These follow directly from integration by parts.
	\end{proof}
	
	This section is concluded with the following simple lemma on log-convexity, which we recall is defined in Section \ref{sec_notational_conventions}. 
	 
	\begin{lem}[Log-convexity of the norms]\label{lem on log-convexity of the norms}
		The following Banach scales are log-convex for the stated norms:
		\begin{align}
			&\sp{\X^s,\tnorm{\cdot}_{\X^s}}_{s\in\N}, & &\sp{\overset{q_0,u_0,\eta_0}{\X^s},\tnorm{\cdot}_{\overset{q_0,u_0,\eta_0}{\X^s}}}_{s\in\N}, & &\sp{\X^s_{m,N},\tnorm{\cdot}_{\X^s_{m,N}}}_{s\in\N}, & \nonumber\\&\sp{\X^s_{m,N},\tnorm{\cdot}_{\overset{q_0,u_0,\eta_0}{\X^s_{m,N}}}}_{s\in\N}, & &\sp{\Y^s,\tnorm{\cdot}_{\Y^s}}_{s\in\N},
		\end{align}
		where we take $m,N\in\N$ and $(q_0,u_0,\eta_0)$ be as in the hypotheses of Lemma~\ref{properties of the principal parts vector field}.
	\end{lem}
	\begin{proof} 
		Lemma~\ref{lem on log-convexity in ansiotropic Sobolev spaces} shows that the anisotropic Sobolev spaces $\tcb{\mathcal{H}^s(\R^d)}_{s\in\N}$ are log-convex.  The result then follows from the log-convexity of standard Sobolev spaces (see Theorem \ref{thm gagliardo nirenberg} and Corollary \ref{gagliardo nirenberg interpolation in domains}) and the evident preservation of log-convexity under products.
	\end{proof}
	
	% - space - space - outer - % - space - space - outer - % - space - space - outer - % - space - space - outer - % - space - space - outer - % - space - space - outer - % - space - space - outer - % - space - space - outer - % - space - space - outer - % - space - space - outer - % - space - space - outer - % - space - space - outer -
	
	\section{Analysis of steady transport equations and their regularizations}\label{Section: Analysis of Regularized Steady Transport Equations}

    In the previous section we introduced the regularized principal part operator $\overset{w_0,\gam_0}{A_{m,N}}$ (defined in equation~\eqref{the map of the regularized principal part operator}), which perturbs the operator $\overset{w_0,\gam_0}{A}$ (defined in equation~\eqref{definition of the principal part operator}) by adding a multiple of the elliptic operator $L_m$ (defined in~\eqref{the operator Lm}) to the part of $\overset{w_0,\gam_0}{A}$ corresponding to the linearized continuity equation.  In this section we focus our attention on solutions to such regularized steady transport equations, with the aim of deriving precise estimates that are independent of the regularization parameter.  These will play an essential role in our subsequent linear analysis.   Along the way, we also develop some estimates of solutions to the standard steady transport equation.
    
To be precise, we now study equations of the type
 \begin{equation}\label{sunday morning i just gotta give it up}
     \al f+\grad\cdot(v f)+\ep L_m f=g \text{ in } \Omega,
 \end{equation}
 where $\alpha, g: \Omega \to \R$ are given functions, $v: \Omega \to \R^n$ is a given vector field, $\ep\ge0$ is small, $L_m$ is the linear elliptic operator given by \eqref{the operator Lm}, and $f: \Omega \to \R$ is the unknown. The key results of this section are Proposition~\ref{proposition on a priori estimates for steady transport} and Theorem~\ref{theorem on estimates for regularized steady transport}.
	
In what follows, a key player is a vector field $X\in W^{\infty,\infty}(\Omega;\R^n)$ such that $\m{Tr}_{\pd\Omega}(X\cdot e_n)=0$. Assume $X_0\in H^\infty(\Omega;\R^n)$, $X_1\in W^{\infty,\infty}(\Omega;\R^n)$, $\rho_{\m{max}}\in\R^+$ is arbitrary but fixed, $0<\rho\le\rho_{\m{max}}$, $r\in\N$,
	\begin{equation}\label{assumptions on the vector field X}
		X=X_0+X_1,\text{ and }(DX_0,DX_1)\in B_{H^r}(0,\rho)\times B_{W^{r,\infty}}(0,\rho).
	\end{equation}
	We will frequently reference this  equation when we need to quantify a vector field of this type.
	
	\subsection{Preliminary tame estimates}\label{slip into the shade and sip their lemonaide}
	
    Given $m\in\N^+$, we define a bilinear form associated to $L_m$ via
	\begin{equation}\label{copy that pasta is it linguini or roasted tolueney}
	\index{\textbf{Nonlinear maps}!1@$B_m$}
		B_m:H^m(\Omega)\times H^m(\Omega)\to\R \text{ via } B_m(\varphi_0,\varphi_1)=\sum_{j=1}^n\int_{\Omega}\pd_j^m\varphi_0\pd_j^m\varphi_1.
	\end{equation}
	We have recorded a number of basic properties of $B_m$ in Appendix \ref{appendix_elliptic_tools}.  Our first result here considers an estimate for $B_m$ in which we are able to a save a derivative thanks to integration by parts.
	
	\begin{lem}[A bilinear estimate]\label{lem on a bilinear estimate}
		Let $m\in\N^+$. Suppose that $\varphi\in H^m(\Omega)$, $X\in W^{\infty,\infty}(\Omega;\R^n)$ satisfies $\m{Tr}_{\pd\Omega}(X\cdot e_n)=0$ as well as \eqref{assumptions on the vector field X} with $r=1+\tfloor{n/2}$ and $0<\rho\le\rho_{\m{max}}$, and that $\grad\cdot(X\varphi)\in H^m(\Omega)$. Then we have the estimates
		\begin{equation}\label{Bm L2 low}
			|B_m(\varphi,\grad\cdot(X\varphi))|\lesssim\rho\tnorm{\varphi}_{H^m}^2+\rho^{-1-\tfloor{n/2}}\tbr{\tnorm{DX_0,DX_1}_{H^m\times W^{m,\infty}}}^{2+\lfloor n/2\rfloor}\tnorm{\varphi}_{H^m}\tnorm{\varphi}_{L^2}
		\end{equation}
		and
		\begin{equation}\label{Bm tame}
			|B_m(\varphi,\grad\cdot(X\varphi))|\lesssim\rho\tnorm{\varphi}^2_{H^m}+\tnorm{\varphi}_{H^m}\begin{cases}
				0&\text{if }m\le 1+\lfloor n/2\rfloor,\\
				\tnorm{DX_0,DX_1}_{H^m\times W^{m,\infty}}\tnorm{\varphi}_{H^{1+\lfloor n/2\rfloor}}&\text{if }1+\lfloor n/2\rfloor<m.
			\end{cases}
		\end{equation}
		Here the implicit constants depend on $m$, $\rho_{\m{max}}$, and the dimension.
	\end{lem}
	\begin{proof}
		Since $\varphi, \nabla\cdot(X \varphi) \in H^m(\Omega)$ and $X \in W^{\infty,\infty}(\Omega;\R^n)$, we can apply the Leibniz rule to see that 
		\begin{equation}
			\nabla\cdot (X \partial_j^m \varphi) = \partial_j^m \nabla \cdot(X \varphi) - \sum_{k=1}^m \binom{m}{k} \nabla \cdot(\partial_j^{k} X \partial_j^{m-k} \varphi  )  \in L^2(\Omega),
		\end{equation}
		and so $\partial_j^m \varphi \in H^0_X(\Omega)$, the space defined by \eqref{adapted X space def} in Appendix \ref{appendix on adapted Sobolev spaces}.      By introducing commutators and employing Proposition \ref{prop on divergence trick}, we obtain the identity
		\begin{equation}
			B_m(\grad\cdot(X\varphi),\varphi)=\sum_{j=1}^n\int_{\Omega}\bp{\f{\grad\cdot X}{2}\pd_j^m\varphi+\grad\cdot([\pd_j^m,X]\varphi)}\pd_j^m\varphi.
		\end{equation}
		Hence, we have the estimate
		\begin{equation}\label{return to this equation}
			|B_m(\grad\cdot(X\varphi),\varphi)|\lesssim\bp{\tnorm{DX}_{L^\infty}\tnorm{\varphi}_{H^m}+\sum_{j=1}^n\tnorm{[\pd_j^m,X]\varphi}_{H^1}}\tnorm{\varphi}_{H^m}.
		\end{equation}
		According to Corollary~\ref{coro on tames estimates on commutators}, we have the estimate
		\begin{equation}\label{X0}
			\tnorm{[\pd_j^m,X_0]\varphi}_{H^1}\lesssim\tnorm{\pd_jX_0}_{H^{1+\lfloor n/2\rfloor}}\tnorm{\varphi}_{H^m}+\begin{cases}
				0&\text{if }m\le1+\lfloor n/2\rfloor,\\
				\tnorm{\pd_jX_0}_{H^m}\tnorm{\varphi}_{H^{1+\lfloor n/2\rfloor}}&\text{if }1+\lfloor n/2\rfloor<m,  
				
			\end{cases}
		\end{equation}
		and
		\begin{equation}\label{X1}
			\tnorm{[\pd_j^m,X_1]\varphi}_{H^1}\lesssim\tnorm{\pd_jX_1}_{L^\infty}\tnorm{\varphi}_{H^m}+\tnorm{\pd_jX_1}_{W^{m,\infty}}\tnorm{\varphi}_{L^2}.
		\end{equation}
		Estimates \eqref{X0} and~\eqref{X1} combine to show that
		\begin{multline}\label{oh so}
			\sum_{j=1}^n\tnorm{[\pd_j^m,X]\varphi}_{H^1}\lesssim\tnorm{DX_0,DX_1}_{H^{1+\lfloor n/2\rfloor}\times W^{1+\lfloor n/2\rfloor,\infty}}\tnorm{\varphi}_{H^m}\\
			+\begin{cases}
				0&\text{if }m\le1+\lfloor n/2\rfloor,\\
				\tnorm{DX_0,DX_1}_{H^m\times W^{m,\infty}}\tnorm{\varphi}_{H^{1+\lfloor n/2\rfloor}}&\text{if }1+\lfloor n/2\rfloor<m.
			\end{cases}
		\end{multline}
		This establishes~\eqref{Bm tame}. 
		
		We now continue to prove~\eqref{Bm L2 low}. In the latter case of~\eqref{oh so}, we use interpolation and Young's inequality to bound
		\begin{multline}\label{shines so}
			\tnorm{DX_0,DX_1}_{H^m\times W^{m,\infty}}\tnorm{\varphi}_{H^{1+\lfloor n/2\rfloor}}\lesssim\rho^{-\f{1+\lfloor n/2\rfloor}{m-1-\lfloor n/2\rfloor}}\tnorm{DX_0,DX_1}_{H^m\times W^{m,\infty}}^{\f{m}{m-1-\lfloor n/2\rfloor}}\tnorm{\varphi}_{L^2}+\rho\tnorm{\varphi}_{H^m}\\
			\lesssim_{\rho_{\m{max}}}\rho^{-1-\lfloor n/2\rfloor}\tbr{\tnorm{DX_0,DX_1}_{H^m\times W^{m,\infty}}}^{2+\lfloor n/2\rfloor}\tnorm{\varphi}_{L^2}+\rho\tnorm{\varphi}_{H^m}.
		\end{multline}
		Together, \eqref{oh so} and~\eqref{shines so} provide the bound
		\begin{equation}\label{chez_nicos}
			\sum_{j=1}^n\tnorm{[\pd_j^m,X]\varphi}_{H^1}\lesssim\rho\tnorm{\varphi}_{H^m}+\rho^{-1-\lfloor n/2\rfloor}\tbr{\tnorm{DX_0,DX_1}_{H^m\times W^{m,\infty}}}^{2+\lfloor n/2\rfloor}\tnorm{\varphi}_{L^2}.
		\end{equation}
		Now we return to~\eqref{return to this equation} and plug in \eqref{chez_nicos} to derive the estimate
		\begin{equation}
			|B_m(\grad\cdot(X\varphi),\varphi)|\lesssim\rho\tnorm{\varphi}_{H^m}^2+\rho^{-1-\lfloor n/2\rfloor}\tbr{\tnorm{DX_0,DX_1}_{H^m\times W^{m,\infty}}}^{2+\lfloor n/2\rfloor}\tnorm{\varphi}_{H^m}\tnorm{\varphi}_{L^2},
		\end{equation}
		which is \eqref{Bm L2 low}.
	\end{proof}
	
	For the next two results we use the notation $\grad_{X}\varphi=X\cdot\grad\varphi$. The following rather technical lemma considers traces of the normal derivatives of $\grad_X\varphi$ when $X$ has vanishing normal trace and $\varphi$ has some vanishing normal derivative traces. The point is that the higher norms on normal derivative traces depend only on $X$ and  a lower norm of $\varphi$. 
	
	\begin{lem}[A trace estimate]\label{lemma on a trace estimate}
		Suppose that $X\in W^{\infty,\infty}(\Omega;\R^n)$
		satisfies $\m{Tr}_{\pd\Omega}(X\cdot e_n)=0$ and decomposes as $X = X_0 + X_1$ with $DX_0\in H^\infty(\Omega;\R^{n\times n})$  and $DX_1\in W^{\infty,\infty}(\Omega;\R^{n\times n})$.  Assume additionally that $\varphi\in H^{2m+1}(\Omega)$ and satisfies the Neumann conditions
		\begin{equation}
			\pd_n^m\varphi=\cdots=\pd_n^{2m-1}\varphi=0\quad\text{on }\pd\Omega.
		\end{equation}
		The we have the following normal derivative estimates for $\ell\in\tcb{0,1,\dots,m-1}$:
		\begin{multline}\label{version 1}
			\tnorm{\m{Tr}_{\pd\Omega}(\pd_n^{m+\ell}\grad_X\varphi)}_{H^{-1/2}}\lesssim\tnorm{DX_0,DX_1}_{H^{2+\lfloor n/2\rfloor+\ell}\times W^{2+\lfloor n/2\rfloor+\ell,\infty}}\tnorm{\varphi}_{H^m}\\+
			\begin{cases}
				0&\text{if }m\le1+\lfloor n/2\rfloor,\\
				\tnorm{DX_0,DX_1}_{H^{1+m+\ell}\times W^{1+m+\ell,\infty}}\tnorm{\varphi}_{H^{1+\lfloor n/2\rfloor}}&\text{if }1+\lfloor n/2\rfloor<m,
			\end{cases}
		\end{multline}
		and
		\begin{multline}\label{version 2}
			\tnorm{\m{Tr}_{\pd\Omega}(\pd_n^{m+\ell}\grad_X\varphi)}_{H^{-1/2}}\lesssim\tbr{\tnorm{DX_0,DX_1}_{H^{2+\lfloor n/2\rfloor+\ell}\times W^{2+\lfloor n/2\rfloor+\ell,\infty}}}\tnorm{\varphi}_{H^m}\\+\tbr{\tnorm{DX_0,DX_1}_{H^{1+m+\ell}\times W^{1+m+\ell,\infty}}}^{2+\lfloor n/2\rfloor}\tnorm{\varphi}_{L^2}.
		\end{multline}
	Here the implicit constants depend on the domain, the dimension, $m$ and $\ell$.
	\end{lem}
	\begin{proof}
		We begin by computing the argument of the trace in the stated estimate. Fix $\ell\in\tcb{0,1,\dots,m-1}$. According to the Leibniz rule, we have
		\begin{multline}
			\pd_n^{m+\ell}\grad_X\varphi=\sum_{k=0}^{m+\ell}\bn{m+\ell}{k}\pd_n^{m+\ell-k}X\cdot\grad\pd_n^k\varphi\\
			=\sum_{k=0}^{m+\ell}\bn{m+\ell}{k}\pd_n^{m+\ell-k}X\cdot(\grad_{\|},0)\pd_n^k\varphi+\sum_{k=0}^{m+\ell}\bn{m+\ell}{k}\pd_n^{m+k-\ell}(X\cdot e_n)\pd_n^{k+1}\varphi.
		\end{multline}
		Applying $\m{Tr}_{\pd\Omega}$ and using that $\pd_n^m\varphi=\cdots=\pd_n^{2m-1}\varphi=0$ and $X\cdot e_n=0$ on $\pd\Omega$, we compute
		\begin{multline}
			\m{Tr}_{\pd\Omega}(\pd_n^{m+\ell}\grad_X\varphi)=\sum_{k=0}^{m-1}\bn{m+\ell}{k}\m{Tr}_{\pd\Omega}(\pd_n^{m+\ell-k}X\cdot(\grad_{\|},0)\pd_n^k\varphi)\\+\sum_{k=0}^{m-2}\bn{m+\ell}{k}\m{Tr}_{\pd\Omega}(\pd_n^{m+\ell-k}(X\cdot e_n)\pd_n^{k+1}\varphi)+\begin{cases}
				0&\text{if }\ell<m-1,\\
				\m{Tr}_{\pd\Omega}(X\cdot e_n\pd_n^{2m}\varphi)&\text{if }\ell=m-1.
			\end{cases}
			\\=\sum_{k=0}^{m-2}\bn{m+\ell}{k}\m{Tr}_{\pd\Omega}(\pd_n^{m+\ell-k}X\cdot\grad\pd_n^k\varphi)+\bn{m+\ell}{m-1}\m{Tr}_{\pd\Omega}(\pd_n^{\ell+1}X\cdot(\grad_{\|},0) \pd_n^{m-1}\varphi)=\bf{I}+\bf{II}.
		\end{multline}
		We then take $H^{-1/2}$-norms and handle $\bf{I}$ and $\bf{II}$ separately.

		We bound $\bf{I}$ via the embedding $H^{1/2}\emb H^{-1/2}$ and the $H^1\to H^{1/2}$ continuity of the trace map:
		\begin{multline}
			\tnorm{\bf{I}}_{H^{-1/2}}\lesssim\sum_{k=0}^{m-2}\tnorm{\pd_n^{m+\ell-k}X\cdot\grad\pd_n^k\varphi}_{H^1}
			\lesssim\sum_{k=0}^{m-2}\tnorm{D^{m+\ell-k}X\otimes D^{k+1}\varphi}_{L^2}\\+\sum_{k=0}^{m-1}\tnorm{D^{m+1+\ell-k}X\otimes D^{k+1}\varphi}_{L^2}=\bf{III}+\bf{IV}.
		\end{multline}
		For $\bf{III}$ and $\bf{IV}$ we employ the splitting $X = X_0 +X_1$ and Corollary~\ref{corollary on tame estimates on simple multipliers}: 
		\begin{multline}
			\bf{III}\le\sum_{k=0}^{m-2}\sp{\tnorm{D^{m-1-k}(D^{1+\ell}X_0)\otimes D^{k+1}\varphi}_{L^2}+\tnorm{D^{m-1-k}(D^{1+\ell}X_1)\otimes D^{k+1}\varphi}_{L^2}}\\
			\lesssim\tnorm{D^{1+\ell}X_0,D^{1+\ell}X_1}_{H^{1+\lfloor n/2\rfloor},W^{1+\lfloor n/2\rfloor,\infty}}\tnorm{\varphi}_{H^m}\\+\begin{cases}
				0&\text{if }m\le1+\lfloor n/2\rfloor,\\
				\tnorm{D^{1+\ell}X_0,D^{1+\ell}X_1}_{H^m\times W^{m,\infty}}\tnorm{\varphi}_{H^{1+\lfloor n/2\rfloor}}&\text{if }1+\lfloor n/2\rfloor<m,
			\end{cases}
		\end{multline}
		and similarly,
		\begin{multline}
			\bf{IV}\le\sum_{k=0}^{m-1}\sp{\tnorm{D^{m-1-k}(D^{2+\ell}X_0)\otimes D^{k+1}\varphi}_{L^2}+\tnorm{D^{m-1-k}(D^{2+\ell}X_1)\otimes D^{k+1}\varphi}_{L^2}}\\
			\lesssim\tnorm{D^{2+\ell}X_0,D^{2+\ell}X_1}_{H^{1+\lfloor n/2\rfloor},W^{1+\lfloor n/2\rfloor,\infty}}\tnorm{\varphi}_{H^m}\\+\begin{cases}
				0&\text{if }m\le1+\lfloor n/2\rfloor,\\
				\tnorm{D^{2+\ell}X_0,D^{2+\ell}X_1}_{H^m\times W^{m,\infty}}\tnorm{\varphi}_{H^{1+\lfloor n/2\rfloor}}&\text{if }1+\lfloor n/2\rfloor<m.
			\end{cases}
		\end{multline}
		Hence, the estimate on $\bf{I}$ we obtain is
		\begin{multline}
			\tnorm{\bf{I}}_{H^{-1/2}}\lesssim\tnorm{DX_0,DX_1}_{H^{2+\lfloor n/2\rfloor+\ell}\times W^{2+\lfloor n/2\rfloor+\ell,\infty}}\tnorm{\varphi}_{H^m}\\
			+\begin{cases}
				0&\text{if }m\le1+\lfloor n/2\rfloor,\\
				\tnorm{DX_0,DX_1}_{H^{1+m+\ell}\times W^{1+m+\ell,\infty}}\tnorm{\varphi}_{H^{1+\lfloor n/2\rfloor}}&\text{if }1+\lfloor n/2\rfloor<m.
			\end{cases}
		\end{multline}
		
		For $\bf{II}$ we make a perfect divergence to exploit the negative norm in the $H^{-1/2}$ space, i.e.
		\begin{equation}
			\pd_n^{\ell+1}X\cdot(\grad_{\|},0) \pd_n^{m-1}\varphi = (\grad_{\|},0)\cdot\tp{\pd_n^{\ell+1}X\pd_n^{m-1}\varphi}-(\grad_{\|},0)\cdot(\pd_n^{\ell+1}X)\pd_n^{m-1}\varphi.
		\end{equation}
		Then again by the continuity of the trace map and Corollary~\ref{corollary on tame estimates on simple multipliers}, we find that
		\begin{multline}
			\tnorm{\bf{II}}_{H^{-1/2}}\lesssim\tnorm{\pd_n^{\ell+1}X\pd_n^{m-1}\varphi}_{H^1}+\tnorm{(\grad_{\|},0)\cdot(\pd_n^{\ell+1}X)\pd_n^{m-1}\varphi}_{H^1}
			\lesssim\tnorm{D^{\ell+1}X\otimes D^{m-1}\varphi}_{H^1}\\+\tnorm{D^{\ell+2}X\otimes D^{m-1}\varphi}_{H^1}\lesssim\tnorm{DX_0,DX_1}_{H^{2+\lfloor n/2\rfloor+\ell}\times W^{2+\lfloor n/2\rfloor+\ell,\infty}}\tnorm{\varphi}_{H^m}\\+
			\begin{cases}
				0&\text{if }m\le1+\lfloor n/2\rfloor,\\
				\tnorm{DX_0,DX_1}_{H^{1+m+\ell}\times W^{1+m+\ell,\infty}}\tnorm{\varphi}_{H^{1+\lfloor n/2\rfloor}}&\text{if }1+\lfloor n/2\rfloor<m.
			\end{cases}
		\end{multline}
		Combining the estimates for $\bf{I}$ and $\bf{II}$ completes the proof of estimate~\eqref{version 1}.
		
		For the remaining estimate, we first take $1+\lfloor n/2\rfloor<m$ and use interpolation and Young's inequality to bound
		\begin{multline}
			\tnorm{DX_0,DX_1}_{H^{1+m+\ell}\times W^{1+m+\ell,\infty}}\tnorm{\varphi}_{H^{1+\lfloor n/2\rfloor}}\lesssim\tnorm{\varphi}_{H^m}+\tnorm{DX_0,DX_1}_{H^{1+m+\ell}\times W^{1+m+\ell,\infty}}^{\f{m}{m-(1+\lfloor n/2\rfloor)}}\tnorm{\varphi}_{L^2}\\\lesssim\tnorm{\varphi}_{H^m}+\tbr{\tnorm{DX_0,DX_1}_{H^{1+m+\ell}\times W^{1+m+\ell,\infty}}}^{2+\lfloor n/2\rfloor}\tnorm{\varphi}_{L^2}.
		\end{multline}
		Hence, for any $m$ we see that~\eqref{version 2} follows.
	\end{proof}
	
	As an application of the previous result, we develop the following estimate for the map $B_m$. As was the case for Lemma~\ref{lem on a bilinear estimate}, the point of this estimate is that the right hand side depends on fewer derivatives that one might expect from a na\"ive inspection of the left hand side.
	
	\begin{lem}[Another bilinear estimate]\label{lem on another bilinear estimate}
		Suppose that $m\in\N^+$ and $X\in W^{\infty,\infty}(\Omega;\R^n)$ satisfies $\m{Tr}_{\pd\Omega}(X\cdot e_n)=0$ and decomposes as $X = X_0 + X_1$, where $DX_0\in H^\infty(\Omega;\R^{n\times n})$  and $DX_1\in W^{\infty,\infty}(\Omega;\R^n)$.
		Assume that $\varphi\in H^{3m}(\Omega)$ satisfies the Neumann conditions $\pd_n^{m}\varphi=\cdots=\pd_n^{2m-1}\varphi=0$ on $\pd\Omega$. Then 
		\begin{multline}\label{movie title the estimate action packed film}
			|B_m(\grad_X\varphi,L_m\varphi)|\lesssim\tbr{\tnorm{DX_0,DX_1}_{H^{2+\lfloor n/2\rfloor}\times W^{2+\lfloor n/2\rfloor,\infty}}}\tnorm{\varphi}_{H^{3m}}\tnorm{\varphi}_{H^{m}}\\
			+\tbr{\tnorm{DX_0,DX_1}_{H^{2m}\times W^{2m,\infty}}}^{2+\lfloor n/2\rfloor}\tnorm{\varphi}_{H^{3m}}\tnorm{\varphi}_{L^2}\\
			+\tbr{\tnorm{DX_0,DX_1}_{H^{1+\lfloor n/2\rfloor+m}\times W^{1+\lfloor n/2\rfloor+m,\infty}}}^{6}\tnorm{\varphi}_{H^m}\tnorm{\varphi}_{L^2}.
		\end{multline}
		Here the implicit constant depends on the domain, the dimension, and $m$.
	\end{lem}
	\begin{proof}
		We begin by integrating by parts. Thanks to Lemma~\ref{lemma on integration by parts}, we have the identity
		\begin{equation}
			B_m(\grad_X\varphi,L_m\varphi)=\int_{\Omega}L_m(\grad_X\varphi)L_m\varphi+\sum_{\ell=0}^{m-1}\bp{\int_{\Sigma}-\int_{\Sigma_0}}(-1)^{\ell}\pd_n^{m+\ell}(\grad_{X}\varphi)\pd_n^{m-1-\ell}L_m\varphi.
		\end{equation}
		Hence, by trace theory, we have the estimate
		\begin{equation}
			|B_m(\grad_X\varphi,L_m\varphi)|\lesssim\babs{\int_{\Omega}L_m(\grad_X\varphi)L_m\varphi}
			+ \sum_{\ell=0}^{m-1} \tnorm{\m{Tr}_{\pd\Omega} (\pd_n^{m+\ell}(\grad_{X} \varphi)) }_{H^{-1/2}} \tnorm{L_m\varphi}_{H^{m-\ell}}=\bf{I}+\bf{II}.
		\end{equation}
		We will handle $\bf{I}$ and $\bf{II}$ separately.

		For $\bf{I}$ we first use $\nabla_X \varphi = \nabla\cdot(\varphi X) - \varphi \nabla \cdot X$ to expand 
		\begin{multline}
			\int_{\Omega}L_m(\grad_X\varphi)L_m\varphi  
			= \int_{\Omega} \nabla \cdot L_m (\varphi X) L_m \varphi -  L_m( \varphi (\nabla \cdot X)) L_m \varphi \\
			= \int_{\Omega} \nabla \cdot (X  L_m \varphi) L_m \varphi + \nabla \cdot ([L_m,X]\varphi) L_m \varphi -  L_m( \varphi (\nabla \cdot X)) L_m \varphi, 
		\end{multline}
		and then employ Proposition~\ref{prop on divergence trick} on the first term on the right. This yields the bound
		\begin{multline}\label{it is time}
			\bf{I}  \le\babs{\int_{\Omega}\f{\grad\cdot X}{2}(L_m\varphi)^2}+\babs{\int_{\Omega}\grad\cdot([L_m,X]\varphi)L_m\varphi}
			+ \babs{\int_{\Omega} L_m( \varphi (\nabla \cdot X)) L_m \varphi} \\
			\lesssim\bp{\tnorm{DX}_{L^\infty}\tnorm{\varphi}_{H^{2m}}+\sum_{j=1}^n\tnorm{[\pd_j^{2m},X]\varphi}_{H^1}   +\tnorm{(\grad\cdot X)\varphi}_{H^{2m}}  }\tnorm{\varphi}_{H^{2m}}.
		\end{multline}
		
		We next use Corollary~\ref{corollary on tame estimates on simple multipliers}, followed by interpolation and Young's inequality, to estimate
		\begin{multline}\label{for me to}
			\tnorm{(\grad\cdot X)\varphi}_{H^{2m}}\lesssim\tnorm{DX_0,DX_1}_{H^{1+\lfloor n/2\rfloor}\times W^{1+\lfloor n/2\rfloor,\infty}} \tnorm{\varphi}_{H^{2m}}\\+\begin{cases}
				0&\text{if }2m\le 1+\lfloor n/2\rfloor,\\
				\tnorm{DX_0,DX_1}_{H^{2m}\times W^{2m,\infty}}\tnorm{\varphi}_{H^{1+\lfloor n/2\rfloor}}&\text{if }1+\lfloor n/2\rfloor<2m,
			\end{cases}\\
			\lesssim\tbr{\tnorm{DX_0,DX_1}_{H^{1+\lfloor n/2\rfloor}\times W^{1+\lfloor n/2\rfloor,\infty}}}\tnorm{\varphi}_{H^{2m}}+\tbr{\tnorm{DX_0,DX_1}_{H^{2m}\times W^{2m,\infty}}}^{2+\lfloor n/2\rfloor}\tnorm{\varphi}_{L^2}.
		\end{multline}
		On the other hand, for $j\in\tcb{1,\dots,n}$, we use Corollary~\ref{coro on tames estimates on commutators} and another interpolation and Young inequality argument to estimate
		\begin{multline}\label{create a bunch of stupid labels}
			\tnorm{[\pd_j^{2m},X]\varphi}_{H^1}\lesssim\tnorm{DX_0,DX_1}_{H^{1+\lfloor n/2\rfloor}\times W^{1+\lfloor n/2\rfloor,\infty}}\tnorm{\varphi}_{H^{2m}}\\
			+\begin{cases}
				0&\text{if }2m\le 1+\lfloor n/2,\\
				\tnorm{DX_0,DX_1}_{H^{2m}\times W^{2m,\infty}}\tnorm{\varphi}_{H^{1+\lfloor n/2\rfloor}}&\text{if }1+\lfloor n/2\rfloor<2m,
			\end{cases}\\
			\lesssim\tbr{\tnorm{DX_0,DX_1}_{H^{1+\lfloor n/2\rfloor}\times W^{1+\lfloor n/2\rfloor, \infty}}} \tnorm{\varphi}_{H^{2m}}+\tbr{\tnorm{DX_0,DX_1}_{H^{2m}\times W^{2m,\infty}}}^{2+\lfloor n/2\rfloor}\tnorm{\varphi}_{L^2}.
		\end{multline}
		Upon synthesizing~\eqref{it is time},\eqref{for me to}, and~\eqref{create a bunch of stupid labels}, we learn that
		\begin{equation}\label{eye bound}
			\bf{I}\lesssim\tbr{\tnorm{DX_0,DX_1}_{H^{1+\lfloor n/2\rfloor}\times W^{1+\lfloor n/2\rfloor, \infty}}} \tnorm{\varphi}_{H^{2m}}^2+\tbr{\tnorm{DX_0,DX_1}_{H^{2m}\times W^{2m,\infty}}}^{2+\lfloor n/2\rfloor}\tnorm{\varphi}_{H^{2m}}\tnorm{\varphi}_{L^2}.
		\end{equation}
		
		Now we turn our attention to $\bf{II}$. First we input the estimate~\eqref{version 2} from Lemma~\ref{lemma on a trace estimate}:
		\begin{multline}
			\bf{II}\le\tnorm{\varphi}_{H^m}\sum_{\ell=0}^{m-1}\tbr{\tnorm{DX_0,DX_1}_{H^{2+\lfloor n/2\rfloor+\ell}\times W^{2+\lfloor n/2\rfloor+\ell,\infty}}}\tnorm{L_m\varphi}_{H^{m-\ell}}\\
			+\tnorm{\varphi}_{L^2}\sum_{\ell=0}^{m-1}\tbr{\tnorm{DX_0,DX_1}_{H^{1+m+\ell}\times W^{1+m+\ell,\infty}}}^{2+\lfloor n/2\rfloor}\tnorm{L_m\varphi}_{H^{m-\ell}}.
		\end{multline}
		Next we utilize the interpolation inequalities
		\begin{equation}
			\tnorm{L_m\varphi}_{H^{m-\ell}}\lesssim\tnorm{L_m\varphi}_{H^1}^{\f{\ell}{m-1}}\tnorm{L_m\varphi}_{H^m}^{\f{m-1-\ell}{m-1}},
		\end{equation}
		\begin{multline}
			\tnorm{DX_0,DX_1}_{H^{2+\lfloor n/2\rfloor+\ell}\times W^{2+\lfloor n/2\rfloor+\ell, \infty}} \\\lesssim\tnorm{DX_0,DX_1}_{H^{2+\lfloor n/2\rfloor}\times W^{2+\lfloor n/2\rfloor,\infty}}^{\f{m-1-\ell}{m-1}} \tnorm{DX_0,DX_1}_{H^{1+\lfloor n/2\rfloor+m}\times W^{1+\lfloor n/2\rfloor+m,\infty}}^{\f{\ell}{m-1}},
		\end{multline}
		and
		\begin{equation}
			\tnorm{DX_0,DX_1}_{H^{1+m+\ell}\times W^{1+m+\ell,\infty}}\lesssim\tnorm{DX_0,DX_1}_{H^{1+m}\times W^{1+m,\infty}}^{\f{m-1-\ell}{m-1}}\tnorm{DX_0,DX_1}_{H^{2m}\times W^{2m,\infty}}^{\f{\ell}{m-1}},
		\end{equation}
		together with Young's inequality to bound
		\begin{multline}\label{eye eye bound}
			\bf{II}\le\tnorm{\varphi}_{H^m}\big(\tbr{\tnorm{DX_0,DX_1}_{H^{2+\lfloor n/2\rfloor}\times W^{2+\lfloor n/2\rfloor,\infty}}}\tnorm{L_m\varphi}_{H^{m}}\\+\tbr{\tnorm{DX_0,DX_1}_{H^{1+\lfloor n/2\rfloor+m}\times W^{1+\lfloor n/2\rfloor+m,\infty}}}\tnorm{L_m\varphi}_{H^{1}}\big)\\
			+\tnorm{\varphi}_{L^2}\big(\tbr{\tnorm{DX_0,DX_1}_{H^{1+m}\times W^{1+m,\infty}}}^{2+\lfloor n/2\rfloor}\tnorm{L_m\varphi}_{H^{m}}\\+\tbr{\tnorm{DX_0,DX_1}_{H^{2m}\times W^{2m,\infty}}}^{2+\lfloor n/2\rfloor}\tnorm{L_m\varphi}_{H^{1}}\big).
		\end{multline}
		
		Upon combining \eqref{eye bound} and~\eqref{eye eye bound}, we see that
		\begin{multline}\label{this guy is an equation}
			\bf{I}+\bf{II}\lesssim\tbr{\tnorm{DX_0,DX_1}_{H^{2+\lfloor n/2\rfloor}\times W^{2+\lfloor n/2\rfloor,\infty}}}\tp{\tnorm{\varphi}_{H^{2m}}^2+\tnorm{\varphi}_{H^m}\tnorm{\varphi}_{H^{3m}}}\\+\tbr{\tnorm{DX_0,DX_1}_{H^{2m}\times W^{2m,\infty}}}^{2+\lfloor n/2\rfloor}\tnorm{\varphi}_{H^{3m}}\tnorm{\varphi}_{L^2}\\
			+\tbr{\tnorm{DX_0,DX_1}_{H^{1+\lfloor n/2\rfloor+m}\times W^{1+\lfloor n/2\rfloor+m, \infty}}} \tnorm{\varphi}_{H^m}\tnorm{\varphi}_{H^{2m+1}}.
		\end{multline}
		We then interpolate again: $\tnorm{\varphi}_{H^{2m}}^2\lesssim\tnorm{\varphi}_{H^m}\tnorm{\varphi}_{H^{3m}}$, and if $m\ge2$ then we use Young's inequality to bound
		\begin{multline}
			\tbr{\tnorm{DX_0,DX_1}_{H^{1+\lfloor n/2\rfloor+m}\times W^{1+\lfloor n/2\rfloor+m}}}\tnorm{\varphi}_{H^{2m+1}}\\\lesssim\tnorm{\varphi}_{H^{3m}}+\tbr{\tnorm{DX_0,DX_1}_{H^{1+\lfloor n/2\rfloor+m}\times W^{1+\lfloor n/2\rfloor+m}}}^{\f{3m}{m-1}}\tnorm{\varphi}_{L^2}.
		\end{multline}
		By combining these  with~\eqref{this guy is an equation}, we prove \eqref{movie title the estimate action packed film}.
	\end{proof}
	
	\subsection{Some results on steady transport equations and elliptic regularizations}
	
	With the lemmas from Section~\ref{slip into the shade and sip their lemonaide} in hand, we may now begin to derive the required precise estimates for steady transport equations and their regularizations. The former is much simpler, and our first result gives all of the necessary a priori estimates.
	
	\begin{prop}[A priori estimates for steady transport]\label{proposition on a priori estimates for steady transport}
		Let $m\in\N$. Suppose that $\varphi\in H^m(\Omega)$, $X\in W^{\infty,\infty}(\Omega;\R^n)$ satisfies~\eqref{assumptions on the vector field X} with $r=1+\tfloor{n/2}$ and $0<\rho\le\rho_{\m{max}}$, and that $\grad\cdot(X\varphi)\in H^m(\Omega)$. Suppose additionally that $\Lambda\in W^{\infty,\infty}(\Omega)$ is such that $\Lambda>0$, $1/\Lambda\in L^\infty(\Omega)$, and 
		\begin{equation}\label{zigmund_das_schwein_7} 
			\Lambda\varphi+\grad\cdot(X\varphi)=\psi  \text{ in }\Omega.
		\end{equation}
		There exists a $\rho^{(m)}\in\R^+$, depending only on $m$, $\Lambda$, and the dimension, such that if $0<\rho\le\rho^{(m)}$, then 
		\begin{equation}\label{a priori estimate is in da houz}
			\tnorm{\varphi}_{H^m}\lesssim\tnorm{\psi}_{H^m}+\begin{cases}
				0&\text{if }m\le 1+\lfloor n/2\rfloor,\\
				\tnorm{DX_0,DX_1}_{H^m\times W^{m,\infty}}\tnorm{\varphi}_{H^{1+\lfloor n/2\rfloor}}&\text{if }1+\lfloor n/2\rfloor<m.
			\end{cases}
		\end{equation}
		Here the implicit constant depends only on $m$, $\rho^{(m)}$, $\Lambda$, and the dimension.
	\end{prop}
	\begin{proof}
		We begin by multiplying \eqref{zigmund_das_schwein_7} by $\varphi$ and applying Proposition~\ref{prop on divergence trick} to obtain the identity
		\begin{equation}
			\int_{\Omega}\bp{\Lambda+\f{\grad\cdot X}{2}}\varphi^2=\int_\Omega \varphi\psi.
		\end{equation}
		Selecting $\rho^{(0)}$ sufficiently small, depending on $\tnorm{1/\Lambda}_{L^\infty}$, and using Cauchy-Schwarz on the right hand side, we obtain the a priori estimate
		\begin{equation}\label{bcaprioriestimate}
			\tnorm{\varphi}_{L^2}\lesssim\tnorm{\psi}_{L^2}.
		\end{equation}
		Now suppose that $m\ge 1$ and $0<\rho\le\rho^{(0)}$.  Plugging  \eqref{zigmund_das_schwein_7} into the bilinear form $B_m$ from~\eqref{copy that pasta is it linguini or roasted tolueney}, we have the identity $B_m(\Lambda\varphi+\grad\cdot(X\varphi),\varphi)=B_m(\psi,\varphi)$, and hence
		\begin{equation}\label{the identity is cat}
			B_m(\Lambda\varphi,\varphi)\le\tnorm{\psi}_{H^m}\tnorm{\varphi}_{H^m}+|B_m(\grad\cdot(X\varphi),\varphi)|.
		\end{equation}
		By applying G\aa rding's inequality, Lemma~\ref{lemma on garding inequality for Bm}, we thus obtain the bound
		\begin{equation}
			\tnorm{\varphi}_{H^m}^2\lesssim\tnorm{\psi}_{H^m}\tnorm{\varphi}_{H^m}+\tnorm{\varphi}_{L^2}^2+|B_m(\grad\cdot(X\varphi),\varphi)|.
		\end{equation}
		Now we invoke Cauchy's inequality for the $\tnorm{\psi}_{H^m}\tnorm{\varphi}_{H^m}$ term, estimate \eqref{bcaprioriestimate} for the $\tnorm{\varphi}^2_{L^2}$ term, and estimate \eqref{Bm tame} from Lemma~\ref{lem on a bilinear estimate} for the final term above; these then yield the bound
		\begin{multline}\label{begotten}
			\tnorm{\varphi}_{H^m}^2\lesssim\tnorm{\psi}^2_{H^m}+\rho\tnorm{\varphi}^2_{H^m}\\+\tnorm{\varphi}_{H^m}\begin{cases}
				0&\text{if }m\le 1+\lfloor n/2,\\
				\tnorm{DX_0,DX_1}_{H^{m}\times W^{m,\infty}}\tnorm{\varphi}_{H^{1+\lfloor n/2\rfloor}}&\text{if }1+\lfloor n/2\rfloor<m.
			\end{cases}
		\end{multline}
		We define $\rho^{(m)}\in(0,\rho^{(0)}]$ to be sufficiently small so that when taking $\rho\le\rho^{(m)}$ we can absorb the right hand side's $\tnorm{\varphi}_{H^m}^2$-contribution with the left. Once this is done, estimate~\eqref{a priori estimate is in da houz} follows from one last application of Cauchy's inequality to the final term in~\eqref{begotten}.
	\end{proof}
	
	The remainder of this subsection develops corresponding a priori estimates for regularized steady transport equations, with the right uniformities with respect to the approximation parameters. This is more complicated than Proposition~\ref{proposition on a priori estimates for steady transport} since the theory now has to account for a rather unhappy marriage of elliptic and hyperbolic structures. Our first result in this direction handles the case of low regularity, namely $L^2$,  data.
	
	\begin{prop}[A priori estimate for regularized steady transport with data in $L^2$]\label{prop on L2 steady transport a prioris}
		Let $m\in\N^+$ and $L_m$ be the operator given in~\eqref{the operator Lm}.  Suppose that $\psi\in L^2(\Omega)$, $\varphi\in H^{2m}(\Omega)$, $X\in W^{\infty,\infty}(\Omega;\R^n)$ satisfies~\eqref{assumptions on the vector field X} with $r=1+\tfloor{n/2}$ and $0<\rho\le\rho_{\m{max}}$, $\Lambda\in W^{\infty,\infty}(\Omega)$ satisfies $\Lambda>0$ and $1/\Lambda\in L^\infty(\Omega)$, and $N\in\N^+$. Assume additionally that the equations
		\begin{equation}\label{regularized steady transport}
			\begin{cases}
				N^{-1}\Lambda L_m\varphi+\grad\cdot(X\varphi)=\psi&\text{in }\Omega,\\
				\pd_n^m\varphi=\cdots=\pd_n^{2m-1}\varphi=0&\text{on }\pd\Omega
			\end{cases}
		\end{equation}
		are satisfied.   If
		\begin{equation}\label{N gtrsim estimate}
			N\gtrsim\tbr{\tnorm{DX_0,DX_1}_{H^m\times W^{m,\infty}}}^{4+2\lfloor n/2\rfloor},
		\end{equation}
		then we have the a priori estimate
		\begin{equation}
			\tnorm{\varphi}_{H^{2m}}\lesssim N\tnorm{\varphi,\psi}_{L^2\times L^2},
		\end{equation}
		where the implied constants depend on $m$, the dimension, $\Lambda$, and $\rho_{\m{max}}$.
	\end{prop}
	\begin{proof}
		We multiply the first equation in \eqref{regularized steady transport} by $N^{-1}L_m\varphi$ and integrate over $\Omega$.  By applying Lemma~\ref{lemma on integration by parts} on the $\grad\cdot(X\varphi)$-term, we acquire the identity
		\begin{equation}
			\f{1}{N^2}\int_{\Omega}\Lambda(L_m\varphi)^2+\f{1}{N}B_m(\grad\cdot(X\varphi),\varphi)=\int_{\Omega}\psi\f{1}{N}L_m\varphi.
		\end{equation}
		Now we use the hypotheses on $\Lambda$ paired with Cauchy's inequality and absorption:
		\begin{equation}\label{at long last we return to this equation}
			\f{1}{N^2}\tnorm{L_m\varphi}_{L^2}^2\lesssim\tnorm{\psi}^2_{L^2}+\f{1}{N}|B_m(\grad\cdot(X\varphi),\varphi)|.
		\end{equation}
		We now handle the $B_m(\grad\cdot(X\varphi),\varphi)$ term by using estimate~\eqref{Bm L2 low} from Lemma~\ref{lem on a bilinear estimate} (with $\rho=\rho_{\m{max}}$) and then interpolating:
		\begin{multline}
			\f{1}{N}|B_m(\grad\cdot(X\varphi),\varphi)|\lesssim\f{1}{N}\tnorm{\varphi}_{H^m}^2+\f{1}{N}\tbr{\tnorm{DX_0,DX_1}_{H^m\times W^{m,\infty}}}^{2+\lfloor n/2\rfloor}\tnorm{\varphi}_{H^m}\tnorm{\varphi}_{L^2}\\
			\lesssim\f{1}{N}\tnorm{\varphi}_{H^{2m}}\tnorm{\varphi}_{L^2}+\f{1}{N^{1/2}}\tbr{\tnorm{DX_0,DX_1}_{H^m\times W^{m,\infty}}}^{2+\lfloor n/2\rfloor}\f{\tnorm{\varphi}_{H^{2m}}^{1/2}}{N^{1/2}}\tnorm{\varphi}_{L^2}^{3/2}.
		\end{multline}
		Therefore, if~\eqref{N gtrsim estimate} is satisfied, then
		\begin{equation}
			\f{1}{N}|B_m(\grad\cdot(X\varphi),\varphi)|\lesssim\f{1}{N}\tnorm{\varphi}_{H^{2m}}\tnorm{\varphi}_{L^2}+\f{\tnorm{\varphi}_{H^{2m}}^{1/2}}{N^{1/2}}\tnorm{\varphi}_{L^2}^{3/2}.
		\end{equation}
		Upon returning to \eqref{at long last we return to this equation}, we see that
		\begin{equation}
			\tnorm{L_m\varphi}_{L^2}\lesssim N\tnorm{\psi}_{L^2}+N^{1/2}\tnorm{\varphi}_{H^{2m}}^{1/2}\tnorm{\varphi}_{L^2}^{1/2}+N^{3/4}\tnorm{\varphi}_{H^{2m}}^{1/4}\tnorm{\varphi}_{L^2}^{3/4}.
		\end{equation}
		As the boundary conditions in~\eqref{regularized steady transport} are satisfied, we may then invoke the a priori estimate for $L_m$ from Lemma~\ref{lem on a priori estimates for Lm} to acquire the bound
		\begin{equation}
			\tnorm{\varphi}_{H^{2m}}\lesssim\tnorm{\varphi,L_m\varphi}_{L^2\times L^2}\lesssim N\tnorm{\varphi,\psi}_{L^2\times L^2}+N^{1/2}\tnorm{\varphi}_{H^{2m}}^{1/2}\tnorm{\varphi}_{L^2}^{1/2}+N^{3/4}\tnorm{\varphi}_{H^{2m}}^{1/4}\tnorm{\varphi}_{L^2}^{3/4}.
		\end{equation}
		The proof is complete upon applying Young's inequality in order to absorb the $\tnorm{\varphi}_{H^{2m}}$ terms from the right onto the left.
	\end{proof}
	
	Our next two results consider what happens when the data for a regularized steady transport equation is of higher regularity than $L^2$. The first of these analyzes the bonus regularity of the solution gained from the vanishing elliptic term.

	\begin{prop}[A priori estimate for regularized steady transport with high regularity data, 1]\label{proposition on Hm a prioris for regularized steady transport}
		Let $m\in\N^+$. Suppose that $\psi\in H^m(\Omega)$, $\varphi\in H^{3m}(\Omega)$, $X\in W^{\infty,\infty}(\Omega;\R^n)$ satisfies~\eqref{assumptions on the vector field X} with $r=2+\tfloor{n/2}$ and $0<\rho\le\rho_{\m{max}}$, $\Lambda\in W^{\infty,\infty}(\Omega)$ satisfies $\Lambda>0$ and $1/\Lambda\in L^\infty(\Omega)$, and $N\in\N^+$. Assume additionally that~\eqref{regularized steady transport} is satisfied. If
		\begin{equation}\label{this is gonna be referenced}
			N\gtrsim\tbr{\tnorm{DX_0,DX_1}_{H^{1+\lfloor n/2\rfloor +m}\times W^{1+\lfloor n/2\rfloor+m,\infty}}}^{4+2\lfloor n/2\rfloor},
		\end{equation}
		then we have the a priori estimate
		\begin{equation}
			\tnorm{\varphi}_{H^{3m}}\lesssim N\tnorm{\psi,\varphi}_{H^m\times H^m}+N\tbr{\tnorm{DX_0,DX_1}_{H^{2m}\times W^{2m,\infty}}}^{2+\lfloor n/2\rfloor}\tnorm{\varphi}_{L^2}.
		\end{equation}
		Here the implied constants depend only on $m$, the dimension, $\Lambda$, and $\rho_{\m{max}}$.
	\end{prop}
	\begin{proof}
		We rearrange the first equation in~\eqref{regularized steady transport} as
		\begin{equation}
			N^{-1}\Lambda L_m\varphi=\psi-(\grad\cdot X)\varphi-\grad_X\varphi
		\end{equation}
		and then take the $B_m$-product of the above equation with $N^{-1}L_m \varphi$. After applying G\aa rding's inequality, Lemma~\ref{lemma on garding inequality for Bm}, followed by the $L^2$-a priori estimate of Proposition~\ref{prop on L2 steady transport a prioris}, we find that
		\begin{equation}
			N^{-2}\tnorm{L_m\varphi}^2_{H^m}\lesssim\tnorm{\varphi,\psi}_{L^2\times L^2}^2+N^{-1}\tnorm{\psi,(\grad\cdot X)\varphi}_{H^m\times H^m}\tnorm{L_m\varphi}_{H^m}+N^{-1}|B_m(\grad_X\varphi,L_m\varphi)|.
		\end{equation}
		Hence, by Cauchy's inequality we get the bound
		\begin{equation}
			N^{-2}\tnorm{L_m\varphi}^2_{H^m}\lesssim\tnorm{\varphi,\psi}_{L^2\times L^2}^2+\tnorm{\psi,(\grad\cdot X)\varphi}_{H^m\times H^m}^2+N^{-1}|B_m(\grad_X\varphi,L_m\varphi)|.
		\end{equation}
		Now we apply the a priori estimates for $L_m$ of Lemma~\ref{lem on a priori estimates for Lm} to learn that
		\begin{equation}\label{until daddy takes the t-bird away now}
			N^{-2}\tnorm{\varphi}_{H^{3m}}^2\lesssim\tnorm{\varphi,\psi}_{L^2\times L^2}^2+\tnorm{\psi,(\grad\cdot X)\varphi}_{H^m\times H^m}^2+N^{-1}|B_m(\grad_X\varphi,L_m\varphi)|.
		\end{equation}
		By Corollary~\ref{corollary on tame estimates on simple multipliers} and a familiar argument using  interpolation and Young's inequality, we deduce that
		\begin{multline}\label{noone can ever bite their own dust - say freddy mercurium}
			\tnorm{(\grad\cdot X)\varphi}_{H^m}\lesssim\tnorm{\varphi}_{H^m}+\begin{cases}
				0&\text{if }m\le 1+\lfloor n/2\rfloor,\\
				\tnorm{DX_0,DX_1}_{H^m\times W^{m,\infty}}\tnorm{\varphi}_{H^{1+\lfloor n/2\rfloor}}&\text{if }1+\lfloor n/2\rfloor<m,
			\end{cases}\\
			\lesssim\tnorm{\varphi}_{H^m}+\tbr{\tnorm{DX_0,DX_1}_{H^m\times W^{m,\infty}}}^{2+\lfloor n/2\rfloor}\tnorm{\varphi}_{L^2}.
		\end{multline}
		On the other hand, Lemma~\ref{lem on another bilinear estimate} shows that
		\begin{multline}\label{another one bites the dust}
			|B_m(\grad_X\varphi,L_m\varphi)|\lesssim\tnorm{\varphi}_{H^{3m}}\tnorm{\varphi}_{H^{m}}+\tbr{\tnorm{DX_0,DX_1}_{H^{2m}\times W^{2m,\infty}}}^{2+\lfloor n/2\rfloor}\tnorm{\varphi}_{H^{3m}}\tnorm{\varphi}_{L^2}\\
			+\tbr{\tnorm{DX_0,DX_1}_{H^{1+\lfloor n/2\rfloor+m}\times W^{1+\lfloor n/2\rfloor+m}}}^{6}\tnorm{\varphi}_{H^m}\tnorm{\varphi}_{L^2}.
		\end{multline}
		Inserting~\eqref{noone can ever bite their own dust - say freddy mercurium} and~\eqref{another one bites the dust} into~\eqref{until daddy takes the t-bird away now} and using Cauchy's inequality then shows that
		\begin{multline}
			N^{-2}\tnorm{\varphi}^2_{H^{3m}}\lesssim\tnorm{\varphi,\psi}_{H^m\times H^m}^2+\tbr{\tnorm{DX_0,DX_1}_{H^{2m}\times W^{2m,\infty}}}^{4+2\lfloor n/2\rfloor}\tnorm{\varphi}_{L^2}^2\\
			+N^{-1}\tbr{\tnorm{DX_0,DX_1}_{H^{1+\lfloor n/2\rfloor+m}\times W^{1+\lfloor n/2\rfloor+m,\infty}}}^{6}\tnorm{\varphi}_{H^m}^2,
		\end{multline}
		and upon combining this with hypothesis~\eqref{this is gonna be referenced} we deduce the bound
		\begin{equation}
			N^{-2}\tnorm{\varphi}^2_{H^{3m}}\lesssim\tnorm{\varphi,\psi}_{H^m\times H^m}^2+\tbr{\tnorm{DX_0,DX_1}_{H^{2m}\times W^{2m,\infty}}}^{4+2\lfloor n/2\rfloor}\tnorm{\varphi}_{L^2}^2.
		\end{equation}
		This is the stated estimate.
	\end{proof}

	Next, we aim to estimate the norm of the solution to a regularized steady transport equation in the same space as the regular data, but independently of the approximation parameter.

	\begin{prop}[A priori estimate for regularized steady transport with high regularity data, 2]\label{prop on hm aprioris for regularized steady transport 2}
		Let $m,N\in\N^+$. Suppose that $\psi\in H^m(\Omega)$, $\varphi\in H^{3m}(\Omega)$, $X\in W^{\infty,\infty}(\Omega;\R^n)$ satisfies~\eqref{assumptions on the vector field X} with $r=1+\tfloor{n/2}$ and $0<\rho\le\rho_{\m{max}}$.  Further suppose that, for $i\in\tcb{0,1}$, $\Lambda_i\in W^{\infty,\infty}(\Omega)$ satisfies $\Lambda_i>0$ and $1/\Lambda_i\in L^\infty(\Omega)$, and that the equations
		\begin{equation}\label{assume that the equations}
			\begin{cases}
				\Lambda_0\varphi+N^{-1}\Lambda_1L_m\varphi+\grad\cdot(X\varphi)=\psi&\text{in }\Omega,\\
				\pd_n^m\varphi=\cdots=\pd_n^{2m-1}\varphi=0&\text{on }\pd\Omega,
			\end{cases}
		\end{equation}
		are satisfied. There exists a $\tilde{\rho}^{(m)}\in\R^+$, depending only on the dimension, $m$, $\Lambda_0$, and $\Lambda_1$ such that if $0<\rho\le\tilde{\rho}^{(m)}$, then we have the a priori estimate 
		\begin{equation}
			\tnorm{\varphi}_{H^m}\lesssim\tnorm{\psi}_{H^m}+\tbr{\tnorm{DX_0,DX_1}_{H^{m}\times W^{m,\infty}}}^{2+\lfloor n/2\rfloor}\tnorm{\varphi}_{L^2},
		\end{equation}
		where the implicit constant depends only on $m$, $\tilde{\rho}^{(m)}$, $\Lambda_0$, $\Lambda_1$, and the dimension.
	\end{prop}
	\begin{proof}
		We test the first equation in~\eqref{assume that the equations} with $L_m\varphi$ in the $L^2(\Omega)$ inner product and integrate by parts via Lemma \ref{lemma on integration by parts} to obtain the identity
		\begin{equation}
			B_m(\Lambda_0\varphi,\varphi)+N^{-1}\tnorm{ \sqrt{\Lambda_1}  L_m\varphi}_{L^2}^2 = B_m(\psi,\varphi)+B_m(\grad\cdot(X\varphi),\varphi).
		\end{equation}
		Hence, $B_m(\Lambda_0\varphi,\varphi) \le \tnorm{\psi}_{H^m}\tnorm{\varphi}_{H^m}+|B_m(\grad\cdot(X\varphi),\varphi)|$. This is exactly inequality~\eqref{the identity is cat} from Proposition~\ref{proposition on a priori estimates for steady transport}. We can therefore argue in exactly the same way here to reach the desired conclusion.
	\end{proof}

	The final result of this section is the culmination of our steady transport analysis.  Essentially, we interpolate between the low regularity estimate of Proposition~\ref{prop on L2 steady transport a prioris} and the high regularity estimates of Propositions~\ref{proposition on Hm a prioris for regularized steady transport} and~\ref{prop on hm aprioris for regularized steady transport 2}. 
	In doing so, we aim to preserve a key structure of the previous estimates: all norms involving the vector field are multiplied by the lowest regularity norms of the data or solution.  This requires some fine interpolation results, which are proved in Appendix~\ref{subsection on on interpolation of Sobolev spaces}.
	
	\begin{thm}[Estimates for regularized steady transport]\label{theorem on estimates for regularized steady transport}
		Let $m,N\in\N^+$ and $j\in\tcb{1,\dots,m}$. Suppose that $\psi\in H^j(\Omega)$, $\varphi\in H^{j+2m}(\Omega)$, and $X\in W^{\infty,\infty}(\Omega;\R^n)$ satisfies~\eqref{assumptions on the vector field X} with $r=2+\tfloor{n/2}$ and $0<\rho\le\rho_{\m{max}}$.  Further suppose $\Lambda_i\in W^{\infty,\infty}(\Omega)$ satisfies $\Lambda_i>0$ and $1/\Lambda_i\in L^\infty(\Omega)$ for $i\in\tcb{0,1}$, and that the equations
		\begin{equation}
			\begin{cases}
				\Lambda_0\varphi+N^{-1}L_m\varphi+\grad\cdot(\Lambda_1X\varphi)=\psi&\text{in }\Omega,\\
				\pd_n^m\varphi=\cdots=\pd_n^{2m-1}\varphi=0&\text{on }\pd\Omega
			\end{cases}
		\end{equation}
		are satisfied. There exists a $\Bar{\rho}^{(m)}\in\R^+$, depending only on $m$, $\Lambda_0$, and $\Lambda_1$, such that if
		\begin{equation}
			\max\tcb{\rho,\tnorm{X_0\cdot\grad\Lambda_1}_{H^{1+\lfloor n/2\rfloor}},\tnorm{X_1\cdot\grad\Lambda_1}_{W^{1+\lfloor n/2\rfloor,\infty}}}\le\Bar{\rho}^{(m)}
		\end{equation}
		and
		\begin{equation}\label{the large N condition}
			N\gtrsim\tbr{\tnorm{X_0,X_1}_{H^{2+\lfloor n/2\rfloor+m}\times W^{2+\lfloor n/2\rfloor+m,\infty}}}^{4+2\lfloor n/2\rfloor},
		\end{equation}
		then we have the a priori estimate
		\begin{equation}
			\tnorm{\varphi,\grad\cdot(\Lambda_1X\varphi),N^{-1}\varphi}_{H^{j}\times H^j\times H^{j+2m}}\lesssim\tnorm{\psi}_{H^j}+\tbr{\tnorm{X_0,X_1}_{H^{1+2m}\times W^{1+2m,\infty}}}^{2+\lfloor n/2\rfloor}\tnorm{\psi}_{L^2}.
		\end{equation}
		Here the implied constants depend only on $\Lambda_0$, $\Lambda_1$, $m$, the dimension, and $\Bar{\rho}^{(m)}$.
	\end{thm}
	\begin{proof}
		We let $\Bar{\rho}\in\R^+$ be such that
		\begin{equation}\label{definition of rho}
			\max\tcb{\rho,\tnorm{X_0\cdot\grad\Lambda_1}_{H^{1+\lfloor n/2\rfloor}},\tnorm{X_1\cdot\grad\Lambda_1}_{W^{1+\lfloor n/2\rfloor,\infty}}}\le\Bar{\rho}.
		\end{equation}
		Throughout the proof we will take $\Bar{\rho}$ to be ever smaller to meet various criteria.  Ultimately, we then define $\Bar{\rho}^{(m)}$ to be the value for $\Bar{\rho}$ we have at the end of the proof.
		
		For $\ell \in \N$, we will also make use of the Banach space
		\begin{equation}
			Z(m,\ell) = \tcb{\phi\in H^{\ell+2m}(\Omega)\;:\;\m{Tr}_{\pd\Omega}(\pd_n^m\phi)=\cdots=\m{Tr}_{\pd\Omega}(\pd_n^{2m-1}\phi)=0}
		\end{equation}
		and the bounded linear map $L: Z(m,\ell) \to H^\ell(\Omega)$   defined via $L\varphi=\Lambda_0\varphi+N^{-1}L_m\varphi+\grad\cdot(\Lambda_1X\varphi)$. We divide the remainder of the proof into several steps. 
		
		\textbf{Step 1}: Establishing invertibility.   We claim that $L$ is a Banach isomorphism for every $\ell \in \N$.  This follows from standard elliptic theory arguments once we establish the fact that the bilinear form $\mathfrak{B}$ defined by
		\begin{equation}
			H^m(\Omega)\times H^{m}(\Omega)\ni(\varphi_0,\varphi_0)\overset{\mathfrak{B}}{\mapsto}\int_{\Omega}\Lambda_0\varphi_0\varphi_1+\grad\cdot(\Lambda_1X\varphi_0)\varphi_1+N^{-1}B_m(\varphi_0,\varphi_1)\in\R
		\end{equation}
		is coercive when $\Bar{\rho}$ is sufficiently small.  To prove coercivity, we first use Proposition \ref{prop on divergence trick} to compute
		\begin{equation}
		   \mathfrak{B}(\varphi,\varphi)=\int_{\Omega}\sp{\Lambda_0+2^{-1}\grad\cdot(\Lambda_1X)}\varphi^2+N^{-1}B_m(\varphi,\varphi).  
		\end{equation}
		 In light of \eqref{assumptions on the vector field X}, \eqref{definition of rho}, and the Sobolev embeddings, we may bound
		\begin{equation}
			\norm{\grad\cdot(\Lambda_1 X ) }_{L^\infty(\Omega)} \le  \norm{\Lambda_1\grad\cdot X }_{L^\infty(\Omega) }+\norm{\grad\Lambda_1\cdot X}_{L^\infty(\Omega) } \lesssim \Bar{\rho},
		\end{equation}
		where the implicit constant depends on $\Lambda_0$ and $\Lambda_1$.  Therefore, if we take $0<\bar{\rho}\le\Bar{\rho}^{(0)}$ sufficiently small, we get that
		\begin{equation}\label{coercivity equation station}
			\mathfrak{B}(\varphi,\varphi)\gtrsim\tnorm{\varphi}^2_{L^2}+N^{-1}B_m(\varphi,\varphi),
		\end{equation}
		and hence  Lemma~\ref{lemma on garding inequality for Bm} shows that $\mathfrak{B}$ is coercive.
		
		\textbf{Step 2}: Low-norm estimates on the inverse. Assume that $\psi\in L^2(\Omega)$, $\varphi \in Z(m,0)$, and that $L\varphi=\psi$. According to \eqref{coercivity equation station}, we have the estimate $\tnorm{\varphi}_{L^2}^2\lesssim\mathfrak{B}(\varphi,\varphi)=\int_{\Omega}\psi\varphi$, and hence
		\begin{equation}\label{aloha piano}
			\tnorm{L^{-1}\psi}_{L^2} = \tnorm{\varphi}_{L^2}  \lesssim\tnorm{\psi}_{L^2}.
		\end{equation}
		Now we may invoke Proposition~\ref{prop on L2 steady transport a prioris} (the hypotheses of which are satisfied thanks to~\eqref{the large N condition}) followed by~\eqref{aloha piano} to see that $\tnorm{L^{-1}\psi}_{H^{2m}}=\tnorm{\varphi}_{H^{2m}}\lesssim N\tnorm{\psi-\Lambda_0\varphi,\varphi}_{L^2\times L^2}\lesssim N\tnorm{\psi}_{L^2}$. Thus, we have shown that $L^{-1}$ maps $L^2(\Omega)$ into $Z(m,0)\subset H^{2m}(\Omega)$ with the operator bounds
		\begin{equation}\label{base case L2 data boundz}
			\tnorm{L^{-1}\psi,N^{-1}L^{-1}\psi}_{L^2\times H^{2m}}\lesssim\tnorm{\psi}_{L^2}.
		\end{equation}

		\textbf{Step 3}: High-norm estimates on the inverse. Now assume that $\psi\in H^m(\Omega)$ and  $\varphi\in Z(m,m) \subset H^{3m}(\Omega)$ satisfy $L\varphi=\psi$. Note that this equation is equivalent to
		\begin{equation}
			\widetilde{\Lambda}_0\varphi+\grad\cdot(X\varphi)+N^{-1}\widetilde{\Lambda}_1L_m\varphi=\widetilde{\psi},
		\end{equation}
		where $\widetilde{\Lambda}_0=\Lambda_0/\Lambda_1$, $\widetilde{\Lambda}_1=1/\Lambda_1$, and $\widetilde{\psi}=(\psi-\varphi X\cdot\grad\Lambda_1)/\Lambda_1$. We take $\Bar{\rho}\le\widetilde{\rho}^{(m)}$, where the latter is given by Proposition~\ref{prop on hm aprioris for regularized steady transport 2}, and then apply the proposition to gain the bound
		\begin{equation}\label{mozie one}
			\tnorm{L^{-1}\psi}_{H^{m}}\lesssim\tnorm{\widetilde{\psi}}_{H^m}+\tbr{\tnorm{DX_0,DX_1}_{H^m\times W^{m,\infty}}}^{2+\lfloor n/2\rfloor}\tnorm{\widetilde{\psi}}_{L^2}.
		\end{equation}
		Thanks to Corollary~\ref{corollary on tame estimates on simple multipliers}, interpolation, and Young's inequality, we have that
		\begin{equation}\label{mozie two}
			\tnorm{\widetilde{\psi}}_{H^m}\lesssim\tnorm{\psi}_{H^m}+\Bar{\rho}\tnorm{\varphi}_{H^m}+\tbr{\tnorm{X_0\cdot\grad\Lambda_1,X_1\cdot\grad\Lambda_1}_{H^m\times W^{m,\infty}}}^{2+\lfloor n/2\rfloor}\tnorm{\varphi}_{L^2}.
		\end{equation}
		On the other hand, \eqref{base case L2 data boundz} provides the bound
		\begin{equation}\label{mozie three}
			\tnorm{\widetilde{\psi}}_{L^2}\lesssim\tnorm{\psi}_{L^2}.
		\end{equation}
		We combine~\eqref{aloha piano}, \eqref{mozie one}, \eqref{mozie two}, and~\eqref{mozie three} and then take $\Bar{\rho}\le\Bar{\rho}^{(m)}\le\Bar{\rho}^{(0)}$ to be sufficiently small so that the right hand side's $\tnorm{\varphi}_{H^m}$-contribution can be absorbed by the left; this results in the bound
		\begin{equation}\label{mozie none}
			\tnorm{L^{-1}\psi}_{H^{m}}\lesssim\tnorm{\psi}_{H^m}+\tbr{\tnorm{X_0,X_1}_{H^{1+m}\times W^{1+m,\infty}}}^{2+\lfloor n/2\rfloor}\tnorm{\psi}_{L^2}.
		\end{equation}
		
		We next apply Proposition~\ref{proposition on Hm a prioris for regularized steady transport} followed by estimates~\eqref{mozie none} and~\eqref{base case L2 data boundz}:
		\begin{multline}
			N^{-1}\tnorm{\varphi}_{H^{3m}}\lesssim\tnorm{\psi-\Lambda_0 \varphi,\varphi}_{H^m}+\tbr{\tnorm{D(\Lambda_1X_0),D(\Lambda_1X_1)}_{H^{2m}\times W^{2m,\infty}}}^{2+\lfloor n/2\rfloor}\tnorm{\varphi}_{L^2}\\
			\lesssim\tnorm{\psi}_{H^m}+\tbr{\tnorm{X_0,X_1}_{H^{1+2m}\times W^{1+2m,\infty}}}^{2+\lfloor n/2\rfloor}\tnorm{\psi}_{L^2}.
		\end{multline}
		The culmination of this analysis is that we have shown that $L^{-1}$ maps $H^m(\Omega)$ into $Z(m,m) \subset H^{3m}(\Omega)$ with the operator bounds
		\begin{equation}
			\tnorm{L^{-1}\psi,N^{-1}L^{-1}\psi}_{H^m\times H^{3m}}\lesssim\tnorm{\psi}_{H^m}+\tbr{\tnorm{X_0,X_1}_{H^{1+2m}\times W^{1+2m,\infty}}}^{2+\lfloor n/2\rfloor}\tnorm{\psi}_{L^2}.
		\end{equation}

		\textbf{Step 4}: Interpolation and conclusion. Let $\mathfrak{E}_\Omega$ denote a Stein extension operator for $\Omega$ (see Definition~\ref{defn Stein-extension operator}) and let $\mathfrak{R}_\Omega$ denote the operator given by restriction to $\Omega$ of functions defined on $\R^n$ (see Example~\ref{example on Sobolev spaces on domains}). We define a map $T\in\mathcal{L}(L^2(\R^n);H^{2m}(\R^n))\cap\mathcal{L}(H^m(\R^n);H^{3m}(\R^n))$ via the formula $T=\mathfrak{E}_\Omega L^{-1}\mathfrak{R}_\Omega$. Thanks to the previous steps, we know that for $A=\tbr{\tnorm{X_0,X_1}_{H^{1+2m}\times W^{1+2m,\infty}}}^{2+\lfloor n/2\rfloor}$, $T$ satisfies the operator bounds
		\begin{equation}
			\begin{cases}
				\tnorm{T \psi }_{L^2}\lesssim\tnorm{\psi}_{L^2}&\text{for all }\psi\in L^2(\R^n),\\
				\tnorm{T \psi}_{H^{m}}\lesssim\tnorm{\psi}_{H^m}+A\tnorm{\psi}_{L^2}&\text{for all }\psi\in H^m(\R^n),
			\end{cases}
		\end{equation}
		and
		\begin{equation}
			\begin{cases}
				\tnorm{T\psi}_{H^{2m}}\lesssim N\tnorm{\psi}_{L^2}&\text{for all }\psi\in L^2(\R^n),\\
				\tnorm{T\psi}_{H^{3m}}\lesssim N\tnorm{\psi}_{H^m}+NA\tnorm{\psi}_{L^2}&\text{for all }\psi\in H^m(\R^n).
			\end{cases}
		\end{equation}
		We are therefore in a position to apply Proposition~\ref{proposition on refined interpolation of Sobolev spaces} to deduce that $T\in\mathcal{L}(H^j(\R^n);H^{j+2m}(\R^n))$ for $j\in\tcb{0,1,\dots,m}$, with the estimates
		\begin{equation}\label{estimate for the extendy boy}
			\tnorm{T\psi,N^{-1}T\psi}_{H^{j}\times H^{j+2m}}\lesssim\tnorm{\psi}_{H^j}+A\tnorm{\psi}_{H^{j-m}} \text{ for all }\psi\in H^j(\R^n).
		\end{equation}
		By utilizing that $L^{-1}=\mathfrak{R}_\Omega T\mathfrak{E}_\Omega$, we can port~\eqref{estimate for the extendy boy} to an estimate on $L^{-1}$, namely:
		\begin{equation}\label{the establishment}
			\tnorm{L^{-1}\psi,N^{-1}L^{-1}\psi}_{H^j\times H^{j+2m}}\lesssim\tnorm{\psi}_{H^j}+\tbr{\tnorm{X_0,X_1}_{H^{1+2m}\times W^{1+2m,\infty}}}^{2+\lfloor n/2\rfloor}\tnorm{\psi}_{L^2}
		\end{equation}
		for all $\psi\in H^j(\Omega)$.
		
		It remains to obtain an estimate on $\grad\cdot(\Lambda_1X L^{-1}\psi)$. For this we rearrange the equation as $\grad\cdot(\Lambda_1X\varphi)=\psi-\Lambda_0\varphi-N^{-1}L_m\varphi$, take the norm in $H^j(\Omega)$ of both sides, and apply the established estimates of~\eqref{the establishment}. This yields the bound
		\begin{equation}\label{the middle class}
  \tnorm{\grad\cdot(\Lambda_1X\varphi)}_{H^j}\lesssim\tnorm{\psi}_{H^j}+\tbr{\tnorm{X_0,X_1}_{H^{1+2m}\times W^{1+2m,\infty}}}^{2+\lfloor n/2\rfloor}\tnorm{\psi}_{L^2}.
		\end{equation}
		The proof is complete upon combining \eqref{the establishment} and~\eqref{the middle class}.
	\end{proof}

	% - space - space - outer - % - space - space - outer - % - space - space - outer - % - space - space - outer - % - space - space - outer - % - space - space - outer - % - space - space - outer - % - space - space - outer - % - space - space - outer - % - space - space - outer - % - space - space - outer - % - space - space - outer -
	
	\section{Analysis of weak solutions to the principal part linear equations}\label{chilean sea bass}
	
	In this section we study weak solutions to the PDE
	\begin{equation}\label{principal part of the linearization}
		\begin{cases}
			\grad\cdot(\varrho u)+\grad\cdot(v_{w_0}(q+\mathfrak{g}\eta))=g&\text{in }\Omega,\\
			-\gam_0^2\varrho\pd_1 u+\varrho\grad(q+\mathfrak{g}\eta)-\gam_0\grad\cdot\S^{\varrho}u=f&\text{in }\Omega,\\
			-(\varrho q-\gam_0\S^{\varrho}u)e_n-\varsigma\Delta_{\|}\eta e_n=k&\text{on }\Sigma,\\
			u\cdot e_n+\pd_1\eta=0&\text{on }\Sigma,\\
			u=0&\text{on }\Sigma_0.
		\end{cases}
	\end{equation}
	Here the given data are $g:\Omega\to\R$, $f:\Omega\to\R^n$, and $k:\Sigma\to\R^n$, as well as  $\gam_0\in\R^+$ and a vector field $v_{w_0}:\Omega\to\R^n$ defined via a fixed triple $w_0=(q_0, u_0,\eta_0)$ as in~\eqref{definition of the vector field vnaught} (see also Lemma~\ref{properties of the principal parts vector field}).  The unknowns are  $q:\Omega\to\R$, $u:\Omega\to\R^n$, and $\eta:\Sigma\to\R$. In other words, we are interested in the weak formulation principal part linear operator $\overset{w_0,\gam_0}{\mathscr{J}}$, which we recall is defined in~\eqref{weak formulation operator}.
	
	The above system is not elliptic in the sense of Agmon, Douglis, and Nirenberg~\cite{MR162050}. Because of this and various other linear effects of the derivative loss, we are led to consider the following regularized version of \eqref{principal part of the linearization} with parameters $\tau\in[0,1]$, $m,N\in\N^+$, and $m\ge 2$:
	\begin{equation}\label{regularization of the prinpal part of the linearization}
		\begin{cases}
			\grad\cdot(\varrho u)+\tau\grad\cdot(v_{w_0}(q+\mathfrak{g}\eta))+N^{-1}L_m(q+\mathfrak{g}\eta)=g&\text{in }\Omega,\\
			-\gam_0^2\varrho\pd_1u+\varrho\grad(q+\mathfrak{g}\eta)-\gam_0\grad\cdot\S^\varrho u=f&\text{in }\Omega,\\
			-(\varrho q-\gam_0\S^\varrho u)e_n-\varsigma\Delta_{\|}\eta e_n=k&\text{on }\Sigma,\\
			u\cdot e_n+\pd_1\eta=N^{-1}(-\Delta_{\|})^{m-1/4}\eta&\text{on }\Sigma,\\
			u=0&\text{on }\Sigma_0,\\
			\pd_n^mq=\cdots=\pd_n^{2m-1}q=0&\text{on }\pd\Omega,
		\end{cases}
	\end{equation}
	where the linear elliptic operator $L_m$ is defined in~\eqref{the operator Lm}. In other words, we are also considering in this section the weak formulation regularized principal part linear operators $\overset{w_0,\gam_0}{\mathscr{J}_{m,N}^\tau}$, which are defined in~\eqref{regularized weak formulation operator}.
	
	The strategy is as follows. We begin in Section~\ref{gonna carry that weight a long time} by proving a priori estimates for weak solutions to the systems~\eqref{principal part of the linearization} and~\eqref{regularization of the prinpal part of the linearization} that are appropriately uniform with respect to the background solution, $N$, and $\tau$.  In Section~\ref{please dont spoil my day, im miles away and} we develop the existence theory for~\eqref{regularization of the prinpal part of the linearization} from our a priori estimates and the method of continuity, which is the reason for including the homotopy parameter $\tau$.  As it turns out, the problem with $\tau=0$ can be solved by taking a two parameter limit of solutions to similar equations, which we solve with the help of the Lax-Milgram lemma.
	
	\subsection{Estimates}\label{gonna carry that weight a long time}
	
	In this subsection we prove a priori estimates for weak solutions to~\eqref{principal part of the linearization} and~\eqref{regularization of the prinpal part of the linearization}. We first require the following technical lemma.
	
	\begin{lem}\label{a technical lemma that will become rel}
		Let $0<\rho\le\rho_{\m{WD}}$, where the latter is defined in Theorem~\ref{thm on smooth tameness of the nonlinear operator}, and let $w_0=(q_0,u_0,\eta_0)$ be as in Lemma~\ref{properties of the principal parts vector field}. Suppose that $\eta\in\mathcal{H}^{0}(\Sigma)$ has Fourier support in the punctured ball $\Bar{B_{\R^{n-1}}(0,1)}\setminus\tcb{0}$. Then we have the estimate
		\begin{equation}\label{chez_nicos_2}
			\babs{\int_{\Omega}(\grad\cdot v_{w_0})\eta^2}\lesssim\rho\tnorm{\eta}^2_{\mathcal{H}^0},
		\end{equation}
		where the implied constant depends only on the dimension, the various physical parameters, and $\rho_{\m{WD}}$.
	\end{lem}
	\begin{proof}
		First note that the support hypotheses on $\eta$ imply that $\eta \in H^\infty(\Sigma)$ (see, for instance, Proposition~\ref{proposition on frequency splitting}).  Second, we compute $\grad\cdot v_{w_0}=\grad\cdot(v_{w_0}-\varrho'/\mathfrak{g} e_1)$, and use this, integration by parts, Fubini-Tonelli, the fundamental theorem of calculus, and the fact that $\m{Tr}_{\pd\Omega}(v_{w_0}\cdot e_n)=0$ to rewrite
		\begin{equation}
			\f12\int_{\Omega}\grad\cdot(v_{w_0}-\varrho'/\mathfrak{g}e_1)\eta^2=\int_{\Omega}\grad\cdot((v_{w_0}-\varrho'/\mathfrak{g}e_1)\eta)\eta=\int_{\R^{n-1}}\bp{(\grad_{\|},0)\cdot\int_0^b(v_{w_0}-\varrho'/\mathfrak{g}e_1)\eta}\eta.
		\end{equation}
		In turn, this readily implies that 
		\begin{equation}
			\babs{\f12\int_{\Omega}(\grad\cdot v_{w_0})\eta^2}\le\bsb{(\grad_{\|},0)\cdot\int_0^b(v_{w_0}-\varrho'/\mathfrak{g}e_1)\eta}_{\dot{H}^{-1}}\tsb{\eta}_{\dot{H}^1}.
		\end{equation}
		By Proposition~\ref{proposition on spatial characterization of anisobros}, specifically~\eqref{the norm on the anisotropic Sobolev spaces}, $\tsb{\eta}_{\dot{H}^1}\lesssim\tnorm{\eta}_{\mathcal{H}^0}$, so it remains to estimate the $\dot{H}^{-1}$ term on the right.  For this we use the decomposition~\eqref{fundamental decomposition of the vector field v} of Lemma~\ref{properties of the principal parts vector field}, which allows us to rewrite
		\begin{equation}
			(\grad_{\|},0)\cdot\int_0^b(v_{w_0}-\varrho'/\mathfrak{g}e_1)\eta=(\grad_{\|},0)\cdot\int_0^bv^{(1)}_{q_0,u_0,\eta_0}\eta+\pd_1\bp{\eta\int_0^bv^{(2)}_{\eta_0}(\cdot,y)\cdot e_1\;\m{d}y}
			=\bf{I}_1+\bf{I}_2.
		\end{equation}
		We bound $\bf{I}_1$ by using the fact that the integrand belongs to $L^2(\Omega)$ (and tacitly using Proposition~\ref{proposition on frequency splitting}, Remark~\ref{remark_about_Lp_inclusion_for_anisos}, and the first item of Lemma~\ref{properties of the principal parts vector field}):
		\begin{equation}
			\tsb{\bf{I}_1}_{\dot{H}^{-1}}\lesssim\tnorm{v_{q_0,u_0,\eta_0}^{(1)}\eta}_{L^2}\lesssim\tnorm{v^{(1)}_{q_0,u_0,\eta_0}}_{L^2}\tnorm{\eta}_{L^\infty}\lesssim\rho\tnorm{\eta}_{\mathcal{H}^0}.
		\end{equation}
		We bound $\bf{I}_2$ using the algebra properties of the specialized Sobolev spaces (see Proposition~\ref{proposition on algebra properties}) and the second item of Lemma~\ref{properties of the principal parts vector field}:
		\begin{equation}
			\tsb{\bf{I}_2}_{\dot{H}^{-1}}\le\bnorm{\eta\bp{\int_0^bv_{\eta_0}^{(2)}(\cdot,y)\;\m{d}y}}_{\mathcal{H}^0}\lesssim\bnorm{\int_0^bv_{\eta_0}^{(2)}(\cdot,y)\;\m{d}y}_{\mathcal{H}^0}\tnorm{\eta}_{\mathcal{H}^0}\lesssim\rho\tnorm{\eta}_{\mathcal{H}^0}.
		\end{equation}
		Combining these bounds yields \eqref{chez_nicos_2}.
	\end{proof}
	
	With the lemma in hand, we are ready to study estimates of weak solutions to the principal part equations~\eqref{principal part of the linearization}. Recall that the spaces $\overset{q_0,u_0,\eta_0}{\X^{-1}}$ and $\Y^{-1}$ are defined in equations~\eqref{domain banach scales} and~\eqref{Y^s_def}, while the weak formulation operator $\overset{w_0,\gam_0}{\mathscr{J}}$ is defined in~\eqref{weak formulation operator}. Also recall the surface tension $\varsigma$ and viscosity $\upmu,\uplambda$ hypotheses set forth in~\eqref{parameter_assumptions}.
	
	\begin{prop}[A priori estimates for weak solutions]\label{prop on a priori estimates for weak solutions} Let $0<\rho\le\rho_{\m{WD}}$, where the latter is defined in Theorem~\ref{thm on smooth tameness of the nonlinear operator}, let $w_0=(q_0,u_0,\eta_0)$ be as in Lemma~\ref{properties of the principal parts vector field}, and let $\gam_0\in I$, where $I\Subset\R^+$ is some interval. Suppose that $(q,u,\eta)\in\overset{q_0,u_0,\eta_0}{\X^{-1}}$ and $(g,F)\in\Y^{-1}$ satisfy the equation
		\begin{equation}\label{definition of weak solutions}
			\overset{w_0,\gam_0}{\mathscr{J}}(q,u,\eta)=(g,F).
		\end{equation}
		There exists a $\rho_{\m{weak}}\in\R^+$ such that if $0<\rho\le\rho_{\m{weak}}$, then we have the a priori estimate
		\begin{equation}\label{a priori estimate for weak solutions equation}
			\tnorm{q,u,\eta}_{\overset{q_0,u_0,\eta_0}{\X^{-1}}}\lesssim\tnorm{g,F}_{\Y^{-1}}.
		\end{equation}
		The implicit constants and $\rho_{\m{weak}}$ depend on the physical parameters, the dimension, $\rho_{\m{WD}}$, and $I$.
	\end{prop}
	\begin{proof}
		We divide the proof into several steps.
		
		\textbf{Step 1}: Reduction to $g=0$. We claim first that it suffices to prove the result under the specialized assumption that $g=0$.  Indeed, suppose this special case has been proved and let $(q,u,\eta)$, $w_0=(q_0,u_0,\eta_0)$, $\gam_0$, and $(g,F)$ be related as in~\eqref{definition of weak solutions}. With the help of the operator $\mathcal{B}_0$ from Proposition~\ref{Bogovskii main}, we define $w\in{_0}H^1(\Omega;\R^n)$ via $w=u-\varrho^{-1} \mathcal{B}_0g$.   Since $\mathcal{B}_0g\in H^1_0(\Omega;\R^n)$, we have that $\m{Tr}_{\Sigma}(w\cdot e_n)=\m{Tr}_{\Sigma}(u\cdot e_n)=-\pd_1\eta$, and hence $(q,w,\eta) \in \overset{q_0,u_0,\eta_0}{\X^{-1}}$.  Recalling that $\overset{\gam_0}{\mathscr{I}}$ is defined in~\eqref{definition of the I functional},  we calculate that
		\begin{equation}
			\overset{\gam_0}{\mathscr{I}}(q,w,\eta)=F-\overset{\gam_0}{\mathscr{I}}(0,\mathcal{B}_0g/\varrho,0) \text{ and }
			\grad\cdot(\varrho w)+\grad\cdot(v_{w_0}(q+\mathfrak{g}\eta))=0.
		\end{equation}
		 Thus, we can apply the special case to obtain the estimate
		\begin{equation}
			\tnorm{q,w,\eta}_{\overset{q_0,u_0,\eta_0}{\X^{-1}}}\lesssim\tnorm{F-\overset{\gam_0}{\mathscr{I}}(0,\mathcal{B}_0g/\varrho,0)}_{({_0}H^1)^\ast}\lesssim\tnorm{g,F}_{\Y^{-1}}.
		\end{equation}
		To switch from $w$ to $u$ in this bound we use the estimate provided by Proposition~\ref{Bogovskii main}, namely $\tnorm{u}_{H^1} \lesssim \tnorm{w}_{H^1} + \tnorm{g}_{\hat{H}^0}$. By chaining this together with the previous estimate we prove the result in general, which completes the proof of the claim. In the remaining steps we will prove the result in the special case that $g=0$.
		
		\textbf{Step 2}: A priori bound on $u$.  We claim that the bound
		\begin{equation}\label{the bound of step 1}
			\tnorm{u}_{H^1}^2\lesssim\tnorm{F}^2_{({_0}H^1)^\ast}+\rho\tp{\tnorm{q}_{L^2}^2+\tnorm{\eta}_{\mathcal{H}^{3/2}}^2}
		\end{equation}
		holds.  We make the following notational simplification for the Fourier space decompositions of $\eta$ from~\eqref{notation for the Fourier projection operators}:
		\begin{equation}\label{fourier space decomposition of eta}
			\eta_{\kappa,\m{L}}=\Uppi^\kappa_{\m{L}}\eta 
			\text{ and }
			\eta_{\kappa,\m{H}}=\Uppi^\kappa_\m{H}\eta.
		\end{equation}
		According to Proposition~\ref{proposition on frequency splitting}, we have that $\eta_{\kappa,\m{H}}\in L^2(\Sigma)$ and $\eta_{\kappa,\m{L}}\in\mathcal{H}^\infty(\Sigma)$. We test the equation
		\begin{equation}\label{weak formulation with eta splitting}
			\overset{\gam_0}{\mathscr{I}}(q,u,\eta_{\kappa,\m{H}})=F-\overset{\gam_0}{\mathscr{I}}(0,0,\eta_{\kappa,\m{L}})
		\end{equation}
		with $u$; the left hand side of the resulting identity reads
		\begin{multline}\label{testing with u}
			\tbr{\overset{\gam_0}{\mathscr{I}}(q,u,\eta_{\kappa,\m{H}}),u}=\int_{\Omega}\gam_0\bp{\f{\upmu(\varrho)}{2}|\mathbb{D}^0u|^2+\uplambda(\varrho)(\grad\cdot u)^2}-(q+\mathfrak{g}\eta_{\kappa,\m{H}})\grad\cdot(\varrho u)\\+\tbr{(\mathfrak{g}\varrho-\varsigma\Delta_{\|})\eta_{\kappa,\m{H}},\m{Tr}_{\Sigma}(u\cdot e_n)}_{H^{-1/2},H^{1/2}},
		\end{multline}
		and we next aim to estimate the latter two terms on the right side this expression.  To this end, we recall that
		\begin{equation}\label{this is being labeled to ensure students recollection}
			\m{Tr}_{\Sigma}(u\cdot e_n)=-\pd_1\eta 
			\text{ and }
			-\grad\cdot(\varrho u)=\grad\cdot(v_{w_0}(q+\mathfrak{g}\eta)).
		\end{equation}
		These allow us to compute
		\begin{equation}\label{can o vanish}
			\tbr{(\mathfrak{g}\varrho-\varsigma\Delta_{\|})\eta_{\kappa,\m{H}},\m{Tr}_{\Sigma}(u\cdot e_n)}_{H^{-1/2},H^{1/2}}=-\tbr{(\mathfrak{g}\varrho-\varsigma\Delta_{\|})\eta_{\kappa,\m{H}},\pd_1\eta_{\kappa,\m{H}}}_{H^{-1/2},H^{1/2}}=0
		\end{equation}
		and  (employing the integration by parts trick of Proposition~\ref{prop on divergence trick})
		\begin{equation}\label{first instance of the I_1 dood}
			\int_{\Omega}-(q+\mathfrak{g}\eta_{\kappa,\m{H}})\grad\cdot(\varrho u)=\int_{\Omega}\f{\grad\cdot v_{w_0}}{2}(q+\mathfrak{g}\eta_{\kappa,\m{H}})^2+\grad\cdot(v_{w_0}\eta_{\kappa,\m{L}})(q+\mathfrak{g}\eta_{\kappa,\m{H}})=\bf{I}_1+\bf{I}_2.
		\end{equation}
		
		We handle $\bf{I}_1$ first by expanding
		\begin{multline}
			\bf{I}_1=\int_{\Omega}\f{\grad\cdot v_{w_0}}{2}(q+\mathfrak{g}\eta_{1,\m{H}})^2+\mathfrak{g}\int_{\Omega}\grad\cdot v_{w_0}(q+\mathfrak{g}\eta_{1,\m{H}})(\eta_{\kappa,\m{H}}-\eta_{1,\m{H}})\\+\int_{\Omega}\f{\mathfrak{g}^2\grad\cdot v_{w_0}}{2}(\eta_{\kappa,\m{H}}-\eta_{1,\m{H}})^2
			=\bf{I}_{1,1}+\bf{I}_{1,2}+\bf{I}_{1,3}.
		\end{multline}
		According to the third item of Lemma~\ref{properties of the principal parts vector field} and Proposition~\ref{proposition on frequency splitting}, we may estimate
		\begin{equation}
			|\bf{I}_{1,1}|\le\tnorm{\grad\cdot v_{w_0}}_{L^\infty}\tnorm{q+\mathfrak{g}\eta_{1,\m{H}}}_{L^2}^2\lesssim\rho\tp{\tnorm{q}_{L^2}^2+\tnorm{\eta}^2_{\mathcal{H}^0}}
		\end{equation}
		and
		\begin{equation}
			|\bf{I}_{1,2}|\lesssim\tnorm{\grad\cdot v_{w_0}}_{L^2}\tnorm{q+\mathfrak{g}\eta_{1,\m{H}}}_{L^2}\tnorm{\eta_{\kappa,\m{H}}-\eta_{1,\m{H}}}_{L^\infty}\lesssim\rho\tp{\tnorm{q}_{L^2}^2+\tnorm{\eta}_{\mathcal{H}^0}^2}.
		\end{equation}
		For $\bf{I}_{1,3}$, we instead use Lemma~\ref{a technical lemma that will become rel}:
		\begin{equation}
			|\bf{I}_{1,3}|\lesssim\rho\tnorm{\eta}_{\mathcal{H}^0}^2.
		\end{equation}
		Upon piecing together the previous three estimates, we deduce that
		\begin{equation}
			|\bf{I}_1|\lesssim\rho\tp{\tnorm{q}_{L^2}^2+\tnorm{\eta}^2_{\mathcal{H}^0}}.
		\end{equation}
		
		Now we turn our attention to the term $\bf{I}_2$. Again, we first decompose
		\begin{equation}\label{cite1}
			\bf{I}_2=\int_{\Omega}\grad\cdot(v_{w_0}\eta_{\kappa,\m{L}})(q+\mathfrak{g}\eta_{1,\m{H}})+\mathfrak{g}\int_{\Omega}\grad\cdot(v_{w_0}\eta_{\kappa,\m{L}})(\eta_{\kappa,\m{H}}-\eta_{1,\m{H}})=\bf{I}_{2,1}+\bf{I}_{2,2}.
		\end{equation}
		$\bf{I}_{2,1}$ is handled via fourth item of Lemma~\ref{properties of the principal parts vector field}:
		\begin{equation}\label{cite2}
			|\bf{I}_{2,1}|\le\tnorm{\grad\cdot(v_{w_0}\eta_{\kappa,\m{L}})}_{L^2}\tnorm{q+\mathfrak{g}\eta_{1,\m{H}}}_{L^2}\lesssim\tnorm{\eta_{\kappa,\m{L}}}_{\mathcal{H}^0}\tp{\tnorm{q}_{L^2}+\tnorm{\eta}_{\mathcal{H}^0}}.
		\end{equation}
		For $\bf{I}_{2,2}$ we instead integrate by parts and use Fubini-Tonelli:
		\begin{equation}\label{cite3}
			|\bf{I}_{2,2}|=\mathfrak{g}\babs{\int_{\R^{n-1}}\bp{(\grad_{\|},0)\cdot\int_0^b v_{w_0}\eta_{\kappa,\m{L}}}(\eta_{\kappa,\m{H}}-\eta_{1,\m{H}})}\lesssim\bsb{(\grad_{\|},0)\cdot\int_0^bv_{w_0}\eta_{\kappa,\m{L}}}_{\dot{H}^{-1}}\tsb{\eta_{\kappa,\m{H}}-\eta_{1,\m{H}}}_{\dot{H^1}}.
		\end{equation}
		By decomposing $v_{w_0}=\varrho'\mathfrak{g}^{-1}e_1+v^{(1)}_{q_0,u_0,\eta_0}+v^{(2)}_{\eta_0}$ (see~\eqref{fundamental decomposition of the vector field v}) and arguing as in the proof of Lemma~\ref{a technical lemma that will become rel}, we acquire the bound
		\begin{equation}\label{cite4}
			\bsb{(\grad_{\|},0)\cdot\int_0^bv_{w_0}\eta_{\kappa,\m{L}}}_{\dot{H}^{-1}}\lesssim\tnorm{\eta_{\kappa,\m{L}}}_{\mathcal{H}^0}.
		\end{equation}
		Hence,
		\begin{equation}\label{cite5}
			|\bf{I}_2|\lesssim\tnorm{\eta_{\kappa,\m{L}}}_{\mathcal{H}^0}\tp{\tnorm{q}_{L^2}+\tnorm{\eta}_{\mathcal{H}^0}},
		\end{equation}
		and upon combining this with the $\bf{I}_1$ estimate we deduce that
		\begin{equation}\label{ziggy_piggy}
			\babs{\int_{\Omega}-(q+\mathfrak{g}\eta_{\kappa,\m{H}})\grad\cdot(\varrho u)}\lesssim\rho\tp{\tnorm{q}^2_{L^2}+\tnorm{\eta}^2_{\mathcal{H}^0}}+\tnorm{\eta_{\kappa,\m{L}}}_{\mathcal{H}^0}(\tnorm{q}_{L^2}+\tnorm{\eta}_{\mathcal{H}^0}).
		\end{equation}

		With \eqref{ziggy_piggy} and \eqref{can o vanish} in hand, we return to~\eqref{testing with u} to  obtain the inequality
		\begin{multline}
			\int_{\Omega}\gam_0\bp{\f{\upmu(\varrho)}{2}|\mathbb{D}^0u|^2+\uplambda(\varrho)(\grad\cdot u)^2}\lesssim\tnorm{F-\mathscr{I}(0,0,\eta_{\kappa,\m{L}})}_{({_0}H^1)^\ast}\tnorm{u}_{H^1}+\rho\tp{\tnorm{q}^2_{L^2}+\tnorm{\eta}^2_{\mathcal{H}^0}}\\+\tnorm{\eta_{\kappa,\m{L}}}_{\mathcal{H}^0}(\tnorm{q}_{L^2}+\tnorm{\eta}_{\mathcal{H}^0}).
		\end{multline}
		This holds for all $\kappa\in(0,1)$ and the implicit constant is independent of $\kappa$. Thus we may send $\kappa\to0$ and use the fact that, as a consequence of the dominated convergence theorem and the definition of the norm on the anisotropic Sobolev spaces~\eqref{there's one for you nineteen for me}, $\tnorm{\eta_{\kappa,\m{L}}}_{\mathcal{H}^0}\to0$ to arrive at the bound
		\begin{equation}\label{chez_nicos_3}
			\gam_0\int_{\Omega}\f{\upmu(\varrho)}{2}|\mathbb{D}^0u|^2+\uplambda(\varrho)(\grad\cdot u)^2\lesssim\tnorm{F}_{({_0}H^1)^\ast}\tnorm{u}_{H^1}+\rho(\tnorm{q}^2_{L^2}+\tnorm{\eta}^2_{\mathcal{H}^0}).
		\end{equation}
		Recall that the assumptions on $\upmu,\uplambda\in C^\infty(\R^+)$ are that $\upmu>0$ and $\uplambda>0$ if $n=2$, while $\upmu>0$ and $\uplambda\ge 0$ if $n\ge 3$.  Thus, the bound \eqref{the bound of step 1} follows from \eqref{chez_nicos_3}, the inclusion $\gam_0\in I\Subset\R^+$, and either Proposition~\ref{prop on deviatoric Korn's inequality} in the case $n\ge 3$ or else Proposition~\ref{prop on Korn's inequality} in the case $n=2$. 
		
		\textbf{Step 3}: A priori bound on $\grad_{\|}\eta$. Next we claim that we have the a priori bound
		\begin{equation}\label{the step 3 estimate}
			\tnorm{\grad_{\|}\eta}_{H^{1/2}}\lesssim\tnorm{u}_{H^1}+\tnorm{F}_{({_0}H^1)^\ast}.
		\end{equation}
		We first consider the case that surface tension is positive: $\varsigma>0$.
		To prove this we will utilize the operator $\mathcal{B}_2$ from Corollary~\ref{Bogovskii 2}. Recall from the previous step that for $\kappa\in(0,1)$ we have the decomposition $\eta=\eta_{\kappa,\m{L}}+\eta_{\kappa,\m{H}}$ defined in~\eqref{fourier space decomposition of eta}. We test identity~\eqref{weak formulation with eta splitting} with $w_\kappa\in{_0}H^1(\Omega;\R^n)$, defined via 
		\begin{equation}
			w_\kappa=-\varrho^{-1}\mathcal{B}_2(\varrho(b)\tbr{\grad_{\|}}^{-1}\Delta_{\|}\eta_{\kappa,\m{H}}),
		\end{equation}
		noting that $\grad\cdot(\varrho w_\kappa)=0$, $\m{Tr}_{\Sigma}(w_\kappa\cdot e_n)=-\tbr{\grad_{\|}}^{-1}\Delta_{\|}\eta_{\kappa,\m{H}}$, and
		\begin{equation}\label{estimate on the H1 norm of dubya kappa}
			\tnorm{w_\kappa}_{H^1}\lesssim\tnorm{\tbr{\grad_{\|}}^{-1}\Delta_{\|}\eta_{\kappa,\m{H}}}_{\dot{H}^{-1}\cap H^{1/2}}\asymp\tnorm{\tbr{\grad_{\|}}^{1/2}|\grad_{\|}|\eta_{\kappa,\m{H}}}_{L^2}\lesssim\tnorm{\grad_{\|}\eta_{\kappa,\m{H}}}_{H^{1/2}}. 
		\end{equation}
		The result is the identity
		\begin{multline}
			\tbr{\overset{\gam_0}{\mathscr{I}}(q,u,\eta_{\kappa,\m{H}}),w_{\kappa}}=\int_{\Omega}\gam_0\tp{\gam_0\varrho u\otimes e_1+\S^{\varrho}u}:\grad w_\kappa+\mathfrak{g}\varrho\grad\eta_{\kappa,\m{L}}\cdot w_\kappa\\
			+\tbr{(\mathfrak{g}\varrho-\varsigma\Delta_{\|})\eta_{\kappa,\m{H}},-\tbr{\grad_{\|}}^{-1}\Delta_{\|}\eta_{\kappa,\m{H}}}_{H^{-1/2},H^{1/2}}.
		\end{multline}
		As
		\begin{equation}
			\tnorm{\grad_{\|}\eta_{\kappa,\m{H}}}_{H^{1/2}}^2\lesssim\tbr{(\mathfrak{g}\varrho-\varsigma\Delta_{\|})\eta_{\kappa,\m{H}},-\tbr{\grad_{\|}}^{-1}\Delta_{\|}\eta_{\kappa,\m{H}}}_{H^{-1/2},H^{1/2}},
		\end{equation}
		we obtain the estimate
		\begin{equation}
			\tnorm{\grad_{\|}\eta_{\kappa,\m{H}}}^2_{H^{1/2}}\lesssim\tp{\tnorm{u}_{H^1}+\tnorm{F}_{({_0}H^1)^\ast}+\tnorm{\eta_{\kappa,\m{L}}}_{\mathcal{H}^{3/2}}}\tnorm{w_\kappa}_{H^1}.
		\end{equation}
		Then \eqref{the step 3 estimate} follows from this and \eqref{estimate on the H1 norm of dubya kappa} upon sending $\kappa\to0$, which is valid since the implicit constants are independent of $\kappa$.  This proves the claim in the case of positive surface tension.
		
		Now we consider the case of vanishing surface tension, $\varsigma=0$, in dimension $n=2$. For this we simply look to the boundary condition satisfied by $u$, namely $\pd_1\eta=-\m{Tr}_{\Sigma}(u\cdot e_n)$. Since $n=2$ and $\eta$ is defined on $\Sigma\simeq\R$, we have $\pd_1\eta=\grad_{\|}\eta$. Therefore we have $\tnorm{\grad_{\|}\eta}_{H^{1/2}}=\tnorm{\m{Tr}_{\Sigma}(u\cdot e_n)}_{H^{1/2}}\lesssim\tnorm{u}_{H^1}$, so~\eqref{the step 3 estimate} holds.

		\textbf{Step 4}: A priori bound on $q$.  We claim that we have the a priori bound
		\begin{equation}\label{a priori bound on q}
			\tnorm{q}_{L^2}\lesssim\tnorm{u}_{H^1}+\tnorm{\grad_{\|}\eta}_{H^{1/2}}+\tnorm{F}_{({_0}H^1)^\ast}.
		\end{equation}
		To see this, we first let $w=\varrho^{-1}\mathcal{B}q \in{_0}H^1(\Omega;\R^n)$,  where $\mathcal{B}$ is constructed in Corollary~\ref{Bogovskii 1}. By construction, we have that $\grad\cdot(\varrho w)=q$ and 
		\begin{equation}\label{ziggy_piggy_2}
			\tnorm{w}_{H^1}\lesssim\tnorm{q}_{L^2}.
		\end{equation}
		Then we we test the identity $\overset{\gam_0}{\mathscr{I}}(q,u,\eta)=F$ with $w$ to see that
		\begin{equation}
			\int_{\Omega}\gam_0(\gam_0\varrho u\otimes e_1+\S^{\varrho}u):\grad w+\mathfrak{g}\varrho\grad\eta\cdot w-q^2-\varsigma\tbr{\Delta_{\|}\eta,\m{Tr}_{\Sigma}(w\cdot e_n)}_{H^{-1/2},H^{1/2}}
			=\tbr{F,w}_{({_0}H^1)^\ast,{_0}H^1}.
		\end{equation}
		Estimate~\eqref{a priori bound on q} readily follows from this, \eqref{ziggy_piggy_2}, and the estimates established in the previous steps.  This proves the claim.

		\textbf{Step 5}: A priori bound on $\pd_1\eta$.  Next we claim that 
		\begin{equation}\label{a priori bound on eta}
			\tsb{\pd_1\eta}_{\dot{H}^{-1}}\lesssim\tnorm{u}_{L^2}+\tnorm{q}_{L^2}+\rho\tnorm{\eta}_{\mathcal{H}^0}.
		\end{equation}
		First, we note that by using the decomposition of $v_{w_0}$ from Lemma~\ref{properties of the principal parts vector field}, the continuity equation is equivalently written as
		\begin{equation}
			0=\grad\cdot(\varrho u+v_{w_0}q)+\mathfrak{g}\grad\cdot(v^{(1)}_{q_0,u_0,\eta_0}\eta)+\mathfrak{g}\pd_1(v^{(2)}_{\eta_0}\cdot e_1\eta)+\varrho'\pd_1\eta.
		\end{equation}
		We integrate this in the $n^{\m{th}}$ coordinate over $(0,b)$;  after recalling the identities $\m{Tr}_\Sigma(u\cdot e_n)+\pd_1\eta=0$ and \eqref{song for a favorite flour} from Proposition~\ref{prop on refined divergence compatibility estimate}, this results in the equality
		\begin{equation}
			\varrho(0)\pd_1\eta=(\grad_{\|},0)\cdot\int_{0}^b\tp{\varrho u+v_{w_0}q+\mathfrak{g}v_{q_0,u_0,\eta_0}^{(1)}\eta}+\mathfrak{g}\pd_1\bp{\eta\bp{\int_0^bv^{(2)}_{\eta_0}(\cdot,y)\cdot e_1\;\m{d}y}}.
		\end{equation}
		Hence, we may use the estimates from Lemma~\ref{properties of the principal parts vector field}, the fact that $\int_0^bv^{(2)}_{\eta_0}(\cdot,y)\;\m{d}y$ has $r_{n-1}$ as a band limit, the algebra properties of the anisotropic Sobolev spaces in Proposition~\ref{proposition on algebra properties}, and estimate~\eqref{driving down the road yesterday} from Proposition~\ref{prop on refined divergence compatibility estimate} to bound
		\begin{equation}
			\varrho(0)\tsb{\pd_1\eta}_{\dot{H}^{-1}}\lesssim\tnorm{\varrho +v_{w_0}q+\mathfrak{g}v^{(1)}_{q_0,u_0,\eta_0}\eta}_{L^2}+\mathfrak{g}\bnorm{\bp{\int_0^bv_{\eta_0}^{(2)}(\cdot,y)\;\m{d}y}\eta}_{\mathcal{H}^0}
			\lesssim\tnorm{u}_{L^2}+\tnorm{q}_{L^2}+\rho\tnorm{\eta}_{\mathcal{H}^0}.
		\end{equation}
		This proves the claim since $\varrho(0)>0$.
		
		\textbf{Step 6}: Conclusion.  We now synthesize the claims of the previous steps to conclude.  First, we take the bound from~\eqref{the bound of step 1} and plug it into the right hand side of~\eqref{the step 3 estimate}; this yields the inequality
		\begin{equation}\label{better estimate on gradient eta}
			\tnorm{\grad_{\|}\eta}_{H^{1/2}}\lesssim\tnorm{F}_{({_0}H^1)^\ast}+\sqrt{\rho}\tp{\tnorm{q}_{L^2}+\tnorm{\eta}_{\mathcal{H}^{3/2}}}.
		\end{equation}
		Second, we take~\eqref{better estimate on gradient eta} and~\eqref{the bound of step 1} and insert them into the right hand side of~\eqref{a priori bound on q} to get
		\begin{equation}\label{better a priori bound on q}
			\tnorm{q}_{L^2}\lesssim\tnorm{F}_{({_0}H^1)^\ast}+\sqrt{\rho}\tp{\tnorm{q}_{L^2}+\tnorm{\eta}_{\mathcal{H}^{3/2}}}.
		\end{equation}
		Now, while heeding to~\eqref{the norm on the anisotropic Sobolev spaces}, we sum the estimates~\eqref{the bound of step 1}, ~\eqref{a priori bound on eta}, \eqref{better estimate on gradient eta}, and \eqref{better a priori bound on q}, and then use ~\eqref{the bound of step 1} and~\eqref{better a priori bound on q} on the right hand side; the resulting estimate is
		\begin{equation}
			\tnorm{q,u,\eta}_{\X_{-1}}\lesssim\tnorm{F}_{({_0}H^1)^\ast}+\sqrt{\rho}\tnorm{q,u,\eta}_{\X_{-1}}.
		\end{equation}
		We choose $\rho_{\m{weak}}\in\R^+$ sufficiently small so that when taking $\rho\le \rho_{\m{weak}}$ we can absorb the $\X_{-1}$ contribution onto the left side and obtain the clean a priori bound
		\begin{equation}
			\tnorm{q,u,\eta}_{\X_{-1}}\lesssim\tnorm{F}_{({_0}H^1)^\ast}.
		\end{equation}
		It remains to only estimate the $L^2(\Omega)$-norm of $\grad\cdot(v_{w_0}q)$ in terms of the data.  This is now a simple matter since we can isolate it in the continuity equation via $-\grad\cdot(v_{w_0}q)=\grad\cdot(\varrho u)+\mathfrak{g}\grad\cdot(v_{w_0}\eta)$ and then note that the right hand side is controlled by $\tnorm{q,u,\eta}_{\X_{-1}}$ (see in particular the fourth item of Lemma~\ref{properties of the principal parts vector field}).
	\end{proof}
	
	Before our next a priori estimates result, we need a lemma from the theory of regularized steady transport equations.
	
	\begin{lem}[Regularized steady transport lemma]\label{regularized steady transport equations lemma 1}
		Let $0<\rho\le\rho_{\m{WD}}$, where the latter is defined in Theorem~\ref{thm on smooth tameness of the nonlinear operator}, $w_0=(q_0,u_0,\eta_0)$ be as in Lemma~\ref{properties of the principal parts vector field}, $m,N\in\N^+$, $\tau\in[0,1]$, $\varphi\in H^{2m}(\Omega)$, and $\psi\in L^2(\Omega)$.  Suppose that  
		\begin{equation}
			\begin{cases}
				N^{-1}L_m\varphi+\tau\grad\cdot(v_{w_0}\varphi)=\psi&\text{in }\Omega,\\
				\pd_n^m\varphi=\cdots=\pd_n^{2m-1}\varphi=0&\text{on }\pd\Omega,
			\end{cases}
		\end{equation}
		where $L_m$ is defined in~\eqref{the operator Lm}.  If $N\gtrsim\tbr{\tau\tnorm{q_0,u_0,\eta_0}_{\X_m}}^{4+2\lfloor n/2\rfloor}$, then we have the a priori estimate 
		\begin{equation}\label{chez_nicos_4}
		\tnorm{\varphi}_{H^{2m}}\lesssim N\tnorm{\varphi,\psi}_{L^2\times L^2},    
		\end{equation}
		 where the implicit constant depends only on the physical parameters, the dimension, $m$, and $\rho_{\m{WD}}$.
	\end{lem}
	\begin{proof}
		Most of the work in verifying this result has already been executed in Section~\ref{Section: Analysis of Regularized Steady Transport Equations}. We invoke Proposition~\ref{prop on L2 steady transport a prioris} with $\Lambda=1$ and the decomposed vector field $X = X_0 + X_1$, where $X=\tau v_{w_0}$,  $X_0=\tau v^{(1)}_{q_0,u_0,\eta_0}$,   and  $X_1=\tau\tp{v^{(2)}_{\eta_0}+\mathfrak{g}^{-1}\varrho'e_1}$,
		with $v^{(1)}_{q_0,u_0,\eta_0}$ and $v^{(2)}_{\eta_0}$ as in Lemma~\ref{properties of the principal parts vector field}. The estimate \eqref{chez_nicos_4} then follows from properties of the vector field $v_{q_0,u_0,\eta_0}$ stated in Lemma~\ref{properties of the principal parts vector field}, namely:  for $s\in\tcb{1+\lfloor n/2\rfloor,m}$ we have $\tnorm{DX_0,DX_1}_{H^s\times W^{s,\infty}}\le\tnorm{X_0,X_1}_{H^{1+s}\times W^{1+s,\infty}}\lesssim\tau\tbr{\tnorm{q_0,u_0,\eta_0}_{\X_s}}$.
	\end{proof}
	
	Our next result studies estimates on weak solutions to the regularized equations~\eqref{regularization of the prinpal part of the linearization}. Recall that the spaces $\X^{-1}_{m,N}$, the norms $\tnorm{\cdot}_{\overset{q_0,u_0,\eta_0}{\X^{-1}_{m,N}}}$, and the mappings $\overset{w_0,\gam_0}{\mathscr{J}^\tau_{m,N}}$ are defined in~\eqref{the regularized spaces}, \eqref{the adapted norm on the regularized spaces}, and~\eqref{regularized weak formulation operator}, respectively.
	
	\begin{prop}[A priori estimates for regularized weak solutions]\label{prop on a priori estimates for regularized weak solutions}
		Let $0<\rho\le\rho_{\m{WD}}$, where the latter is from Theorem~\ref{thm on smooth tameness of the nonlinear operator}, $w_0=(q_0,u_0,\eta_0)$ be as in Lemma~\ref{properties of the principal parts vector field}, and $\gam_0\in I$, where $I\Subset\R^+$ is some interval. Suppose that $m,N\in\N^+$, $\tau\in[0,1]$, $(q,u,\eta)\in\X^{-1}_{m,N}$, and that $(g,F)\in\Y^{-1}$ satisfy the equation $\overset{w_0,\gam_0}{\mathscr{J}^\tau_{m,N}}(q,u,\eta)=(g,F)$. There exists a $\rho_{\m{weak,reg}}\in\R^+$ such that if
		\begin{equation}\label{the labeled conditions to be satisfaction}
			0<\rho\le\rho_{\m{weak,reg}}\text{ and }N\gtrsim\tbr{\tnorm{q_0,u_0,\eta_0}_{\X_m}}^{4+2\lfloor n/2\rfloor},
		\end{equation}
		then we have the a priori estimate
		\begin{equation}\label{prop on a priori estimates for regularized weak solutions EST}
			\tnorm{q,u,\eta}_{\overset{q_0,u_0,\eta_0}{\X^{-1}_{m,N}}}\lesssim\tnorm{g,F}_{\Y^{-1}}.
		\end{equation}
		The implicit constants and $\rho_{\m{weak,reg}}$ depend on the dimension, physical parameters, $m$, $\rho_{\m{WD}}$, and $I$.
	\end{prop}
	\begin{proof}
		We proceed in much the same way as in the proof of Proposition~\ref{prop on a priori estimates for weak solutions}, breaking to steps that mirror the structure of the argument used there.
		
		\textbf{Step 1}: Reduction to the case $g=0$.  We claim that it suffices to prove the result in the special case that $g=0$.  Indeed, the exact same argument used in the first step of Proposition~\ref{prop on a priori estimates for weak solutions} proves the claim here. In the remaining steps we will prove the result in the special case that $g=0$.
		
		\textbf{Step 2}: A priori bound on $u$.  We claim that we have the a priori bound
		\begin{equation}\label{this is a surely an a priori bound on u}
			\tnorm{u}_{H^1}^2\lesssim\tnorm{F}_{({_0}H^1)^\ast}^2+\rho\tp{\tnorm{q}_{L^2}^2+\tnorm{\eta}^2_{\mathcal{H}^{3/2}}}.
		\end{equation}
		To prove the claim we again use the Fourier space decomposition of~\eqref{fourier space decomposition of eta}, which again yields \eqref{weak formulation with eta splitting}.   We then test~\eqref{weak formulation with eta splitting} with $u$ to arrive at~\eqref{testing with u} as in the proof of Proposition~\ref{prop on a priori estimates for weak solutions}.   Note, though, that in the present context we have the identities
		\begin{equation}\label{regularized trace condition}
			\m{Tr}_{\Sigma}(u\cdot e_n)=-\pd_1\eta+N^{-1}(-\Delta_{\|})^{m-1/4}\eta
		\end{equation}
		and
		\begin{equation}\label{regularized continuity equation}
			-\grad\cdot(\varrho u)=\tau\grad\cdot(v_{w_0}(q+\mathfrak{g}\eta))+N^{-1}L_m(q+\mathfrak{g}\eta),
		\end{equation}
		which are somewhat different from~\eqref{this is being labeled to ensure students recollection}. We insert \eqref{regularized trace condition} and~\eqref{regularized continuity equation} into~\eqref{testing with u} and integrate by parts  to see that 
		\begin{multline}\label{ziggy_piggy_3}
			\int_{\Omega}\gamma_0\bp{\f{\upmu(\varrho)}{2}|\mathbb{D}^0u|^2+\uplambda(\varrho)(\grad\cdot u)^2}+\f{\tau\grad\cdot v_{w_0}}{2}(q+\mathfrak{g}\eta_{\kappa,\m{H}})^2+\f{1}{N}\sum_{j=1}^n(\pd_j^m(q+\mathfrak{g}\eta_{\kappa,\m{H}}))^2\\
			+\mathfrak{g}\int_{\Omega}(q+\mathfrak{g}\eta_{\kappa,\m{H}})\bp{\tau\grad\cdot(v_{w_0}\eta_{\kappa,\m{L}})
				+\f{1}{N} L_{m,\|}\eta_{\kappa,\m{L}}} + \f{1}{2}\tnorm{(\mathfrak{g}\varrho-\varsigma\Delta_{\|})^{1/2}(\Delta_{\|})^{m/2-1/8}\eta_{\kappa,\m{H}}}^2_{L^2}\\
			=\tbr{F-\mathscr{I}(0,0,\eta_{\kappa,\m{L}}),u}_{({_0}H^1)^\ast,{_0}H^1}.
		\end{multline}
		
		By combining \eqref{ziggy_piggy_3}, the Korn inequalities from Propositions~\ref{prop on deviatoric Korn's inequality} and \ref{prop on Korn's inequality} as well as the end of the second step of Proposition~\ref{prop on a priori estimates for weak solutions}, and the fact that $\tau \in [0,1]$, we obtain the estimate
		\begin{multline}\label{this 0}
			\tnorm{u}_{H^1}^2 \lesssim 
			\babs{\int_{\Omega}\f{\grad\cdot v_{w_0}}{2}(q+\mathfrak{g}\eta_{\kappa,\m{H}})^2}
			+\mathfrak{g}\babs{\int_{\Omega}(q+\mathfrak{g}\eta_{\kappa,\m{H}})\bp{\tau\grad\cdot(v_{w_0}\eta_{\kappa,\m{L}})+\f{1}{N} L_{m,\|} \eta_{\kappa,\m{L}}}}\\
			+\tp{\tnorm{F}_{({_0}H^1)^\ast}+\tnorm{\eta_{\kappa,\m{L}}}_{\mathcal{H}^{3/2}}}\tnorm{u}_{H^1}
			=\bf{I}_1+\bf{I}_2+\bf{I}_3.
		\end{multline}
		The term $\bf{I}_1$ is identical to the term $\bf{I}_1$ that appeared in \eqref{first instance of the I_1 dood}, so we may use the same argument used in the second step of the proof of Proposition~\ref{prop on a priori estimates for weak solutions} to arrive at the estimate
		\begin{equation}\label{this3}
			\bf{I}_1\lesssim\rho\tp{\tnorm{q}_{L^2}^2+\tnorm{\eta}^2_{\mathcal{H}^0}}.
		\end{equation}
		For the term $\bf{I}_2$ we bound
		\begin{equation}
			\bf{I}_2\le  \mathfrak{g}\babs{\int_{\Omega}(q+\mathfrak{g}\eta_{\kappa,\m{H}})\tau\grad\cdot(v_{w_0}\eta_{\kappa,\m{L}})}+\f{\mathfrak{g}}{N}\babs{\int_{\Omega}(q+\mathfrak{g}\eta_{\kappa,\m{H}})L_{m,\|} \eta_{\kappa,\m{L}}}=\bf{I}_{2,1}+\bf{I}_{2,2}.
		\end{equation}
		Arguing as in \eqref{cite1}--\eqref{cite5}, we find that
		\begin{equation}\label{this1}
			\bf{I}_{2,1}\lesssim\tnorm{\eta_{\kappa,\m{L}}}_{\mathcal{H}^0}\tp{\tnorm{q}_{L^2}+\tnorm{\eta}_{\mathcal{H}^0}}.
		\end{equation}
		For the remaining piece we exploit the fact that $\eta_{\kappa,\m{L}}$ and $\eta_{\kappa,\m{H}}$ have disjoint Fourier supports, together with~\eqref{the norm on the anisotropic Sobolev spaces} to see that 
		\begin{equation}\label{this2}  
			\bf{I}_{2,2}=\f{g}{N}\babs{\int_{\Omega}q L_{m,\|} \eta_{\kappa,\m{L}}}
			\lesssim \f{1}{N}\tnorm{q}_{L^2} \tnorm{ L_{m,\|} \eta_{\kappa,\m{L}}}_{L^2}.
			\lesssim \f{1}{N}\tnorm{q}_{L^2}\tnorm{\eta_{\kappa,\m{L}}}_{\mathcal{H}^0}.
		\end{equation}
		Finally, we plug \eqref{this3}, \eqref{this1}, and \eqref{this2} into \eqref{this 0}  and then use Cauchy's inequality on $\bf{I}_3$ to deduce that
		\begin{equation}
			\tnorm{u}^2_{H^1}\lesssim\tnorm{F}^2_{({_0}H^1)^\ast}+\rho\tp{\tnorm{q}_{L^2}^2+\tnorm{\eta}_{\mathcal{H}^0}^2}+\tnorm{\eta_{\kappa,\m{L}}}_{\mathcal{H}^0}\tp{\tnorm{q}_{L^2}+\tnorm{\eta}_{\mathcal{H}^0}}.
		\end{equation}
		The implicit constant is independent of $\kappa$, so we may send $\kappa\to 0$ to obtain~\eqref{this is a surely an a priori bound on u} from this.  The claim is proved.

		\textbf{Step 3}: A priori bound on $\grad_{\|}\eta$.  We claim that we have the a priori bound
		\begin{equation}\label{a priori grad eta woo}
			\tnorm{\grad_{\|}\eta}_{H^{1/2}}\lesssim\tnorm{u}_{H^1}+\tnorm{F}_{({_0}H^1)^\ast}.
		\end{equation}
		Indeed, the claim is proved in exactly the same way as the claim from third step of the proof of Proposition~\ref{prop on a priori estimates for weak solutions}.

		\textbf{Step 4}: A priori bound on $q$. We claim that we have the a priori bound
		\begin{equation}\label{a priori grad q woo}
			\tnorm{q}_{L^2}\lesssim\tnorm{u}_{H^1}+\tnorm{\grad_{\|}\eta}_{H^{1/2}}+\tnorm{F}_{({_0}H^1)^\ast}.
		\end{equation}
		Again, this is proved in exactly the same way as the fourth step in the proof of Proposition~\ref{prop on a priori estimates for weak solutions}.
		
		\textbf{Step 5}: High norm a priori bounds on $\eta$ and $q$.  We next claim that we have the bound
		\begin{equation}\label{the high norm estimates}
			\tnorm{\eta_{1,\m{H}}}_{H^{2m}}+\tnorm{q}_{H^{2m}}\lesssim N\tp{\tnorm{u}_{H^1}+\tnorm{q}_{L^2}+\tnorm{\grad_{\|}\eta}_{H^{1/2}}+\rho\tnorm{\eta}_{\mathcal{H}^0}}.
		\end{equation}
		To see this, we begin by rewriting the equation
		\begin{equation}
			\grad\cdot(\varrho u)+\tau\grad\cdot(v_{w_0}(q+\mathfrak{g}\eta))+N^{-1}L_m(q+\mathfrak{g}\eta)=0
		\end{equation}
		with the help of the splitting~\eqref{fourier space decomposition of eta} with $\kappa=1$:
		\begin{equation}
			N^{-1}L_m(q+\mathfrak{g}\eta_{1,\m{H}})+\tau\grad\cdot(v_{w_0}(q+\mathfrak{g}\eta_{1,\m{H}})=-\grad\cdot(\varrho u)-\mathfrak{g}\tau\grad\cdot(v_{w_0}\eta_{1,\m{L}})-\mathfrak{g}N^{-1}L_{m,\|} \eta_{1,\m{L}}.
		\end{equation}
		As $\pd_n^m(q+\mathfrak{g}\eta_{1,\m{H}})=\cdots=\pd_n^{2m-1}(q+\mathfrak{g}\eta_{1,\m{H}})=0$, we are in a position to apply Lemma~\ref{regularized steady transport equations lemma 1}, provided that  $N\gtrsim\tbr{\tnorm{q_0,u_0,\eta_0}_{\X_m}}^{4+2\lfloor n/2\rfloor}$, which we are free to assume; this yields the bound
		\begin{equation}\label{the sum}
			\tnorm{q+\mathfrak{g}\eta_{1,\m{H}}}_{H^{2m}}\lesssim N\tp{\tnorm{q}_{L^2}+\tnorm{u}_{H^1}+\rho\tnorm{\eta}_{\mathcal{H}^0}+\tnorm{\grad_{\|}\eta}_{H^{1/2}}}.
		\end{equation}
		
		Next we obtain a bound on $\eta_{1,\m{H}}$ in $H^{2m}(\Sigma)$ through the normal trace boundary condition.  Indeed, we test equation~\eqref{regularized trace condition} with $N^{-1}(-\Delta_{\|})^{m+1/4}\eta_{1,\m{H}}$ to see that
		\begin{equation}
			N^{-2}\tnorm{(-\Delta_{\|})^{m}\eta_{1,\m{H}}}^2_{L^2}\le N^{-1}\tnorm{\m{Tr}_\Sigma(u\cdot e_n)}_{H^{1/2}}\tnorm{(-\Delta_{\|})^{m+1/4}\eta_{1,\m{H}}}_{H^{-1/2}},
		\end{equation}
		and hence
		\begin{equation}\label{eta alone}
			\tnorm{\eta_{1,\m{H}}}_{H^{2m}}\lesssim N\tnorm{u}_{H^1}.
		\end{equation}
		We then obtain~\eqref{the high norm estimates} by combining~\eqref{the sum} and~\eqref{eta alone}.  The claim is proved.

		\textbf{Step 6}: A priori bounds on $\pd_1\eta$.  We claim that we have the a priori bound
		\begin{equation}\label{aprioriboundonpd1eta}
			\tsb{\pd_1\eta}_{\dot{H}^{-1}}\lesssim\tnorm{u}_{L^2}+\tnorm{q}_{L^2}+N^{-1}\tnorm{q}_{H^{2m}}+\rho\tnorm{\eta}_{\mathcal{H}^0}.
		\end{equation}
		As in the fifth step of the proof of Proposition~\ref{prop on a priori estimates for weak solutions}, we integrate equation~\eqref{regularized continuity equation} in the $n^{\m{th}}$ coordinate from $0$ to $b$ and isolate the non-small $\pd_1\eta$ contributions:
		\begin{multline}\label{ziggy_piggy_5}
			((1-\tau)\varrho(b)+\tau\varrho(0))\pd_1\eta=(\grad_{\|},0)\cdot\int_0^b(\varrho u+\tau v_{w_0}q+\tau\mathfrak{g}v^{(1)}_{q_0,u_0,\eta_0}\eta)+\tau\mathfrak{g}\pd_1\bp{\eta\int_0^b v^{(2)}_{\eta_0}(\cdot,y)\cdot e_1\;\m{d}y}\\+\f{1}{N}L_{m,\|}\int_0^b(q+\mathfrak{g}\eta)+\f{\varrho(b)}{N}(-\Delta_{\|})^{m-1/4}\eta=\bf{I}_1+\bf{I}_2+\bf{I}_3+\bf{I}_4.
		\end{multline}
		Fix $\kappa\in(0,1)$ and set $\eta_{\kappa}=(\Uppi^{1/\kappa}_{\m{L}}-\Uppi^\kappa_{\m{L}})\eta$. We take the $L^2$-inner product of \eqref{ziggy_piggy_5} with $|\grad_{\|}|^{-2}\pd_1\eta_\kappa$.  The terms involving $\bf{I}_1$ and $\bf{I}_2$ are estimated via Cauchy-Schwarz
		\begin{equation}
			\tbr{\bf{I}_1 + \bf{I}_2,|\grad_{\|}|^{-2}\pd_1\eta_\kappa}_{L^2,L^2}
			\lesssim 
			\tsb{\bf{I}_1 +\bf{I}_2 }_{\dot{H}^{-1} }\tsb{\pd_1\eta_\kappa}_{\dot{H}^{-1}}, 
		\end{equation}
		but then the argument used in the fifth step of the proof of Proposition~\ref{prop on a priori estimates for weak solutions} shows that
		\begin{equation}
			\tsb{\bf{I}_1}_{\dot{H}^{-1}}+\tsb{\bf{I}_2}_{\dot{H}^{-1}}\lesssim
			\tnorm{u}_{L^2}+\tnorm{q}_{L^2}+\rho\tnorm{\eta}_{\mathcal{H}^0},
		\end{equation}
		so 
		\begin{equation}
			\tbr{\bf{I}_1 + \bf{I}_2,|\grad_{\|}|^{-2}\pd_1\eta_\kappa}_{L^2,L^2} \lesssim \left(\tnorm{u}_{L^2}+\tnorm{q}_{L^2}+\rho\tnorm{\eta}_{\mathcal{H}^0} \right) \tsb{\pd_1\eta_\kappa}_{\dot{H}^{-1}}.
		\end{equation}
		On the other hand, for the term involving $\bf{I}_3$ we have
		\begin{equation}
			\tbr{\bf{I}_3,|\grad_{\|}|^{-2}\pd_1\eta_\kappa}_{L^2,L^2}=\f{1}{N}\bbr{|\grad_{\|}|^{-1}L_{m,\|}\int_0^bq,|\grad_{\|}|\pd_1\eta_\kappa}_{L^2,L^2}\lesssim\f{1}{N}\tnorm{q}_{H^{2m-1}}\tsb{\eta_{\kappa}}_{\dot{H^{-1}}}.
		\end{equation}
		while for the $\bf{I}_4$ term we compute $\tbr{\bf{I}_4,|\grad_{\|}|^{-2}\pd_1\eta_{\kappa}}_{L^2,L^2}=0$. All together, these combine to show that
		\begin{equation}
			\tsb{\pd_1\eta_\kappa}_{\dot{H}^{-1}}\lesssim\tnorm{u}_{L^2}+\tnorm{q}_{L^2}+N^{-1}\tnorm{q}_{H^{2m}}+\rho\tnorm{\eta}_{\mathcal{H}^0},
		\end{equation}
		and since the implicit constant is independent of $\kappa$, we can send  $\kappa\to0$ to arrive at~\eqref{aprioriboundonpd1eta}.  The claim is proved.
		
		\textbf{Step 7}: Conclusion. Finally, we derive the desired estimate \eqref{prop on a priori estimates for regularized weak solutions EST} by combining the bounds from the previous steps and taking $\rho$ to be sufficiently small to absorb, exactly as in the proof of the sixth step of Proposition~\ref{prop on a priori estimates for weak solutions}.
	\end{proof}
	
	\subsection{Existence of solutions to the regularization}\label{please dont spoil my day, im miles away and}
	
	In this subsection we prove the existence of weak solutions to the regularized problem~\eqref{regularization of the prinpal part of the linearization} and then quickly deduce qualitative regularity. First we have our main existence result for weak solutions. 
	
	\begin{thm}[Existence of regularized weak solutions]\label{thm on existence of regularized weak solutions}
		Let $0<\rho\le\rho_{\m{WD}}$, where the latter is from Theorem~\ref{thm on smooth tameness of the nonlinear operator}, $w_0=(q_0,u_0,\eta_0)$ be as in Lemma~\ref{properties of the principal parts vector field}, $m,N\in\N^+$ with $m\ge 2$, $\tau\in[0,1]$, $(g,F)\in\Y^{-1}$, and $\gam_0\in I$ for some interval $I\Subset\R^+$. If~\eqref{the labeled conditions to be satisfaction} is satisfied, then there exists a unique $(q,u,\eta)\in\X^{-1}_{m,N}$ satisfying 
		\begin{equation}\label{proposition on existence of regularized weak solutions EQN}
			\overset{w_0,\gam_0}{\mathscr{J}^\tau_{m,N}}(q,u,\eta)=(g,F).
		\end{equation}
	\end{thm}
	\begin{proof}
		In functional analytic terms, we aim to prove that for every $\tau \in [0,1]$ the operator $ \overset{w_0,\gam_0}{\mathscr{J}_{m,N}^\tau} \in\mathcal{L}(\X_{m,N}^{-1};\Y^{-1})$, which is well-defined thanks to Lemma~\ref{lem on well-definedness check for linear analysis},  is an isomorphism.  We divide the proof of this into several steps.

		\textbf{Step 1}: Reduction to proving existence with  $g=0$ and $\tau=0$.  We will achieve the reduction to $\tau=0$ through the method of continuity (see, for instance, Theorem 5.2 in Gilbarg and Trudinger~\cite{MR1814364}).  Indeed, the convex homotopy of operators $[0,1]\ni\tau\mapsto \overset{w_0,\gam_0}{\mathscr{J}_{m,N}^\tau}
		\in\mathcal{L}(\X_{m,N}^{-1};\Y^{-1})$ satisfies $\tau$-uniform a priori estimates thanks to Proposition~\ref{prop on a priori estimates for regularized weak solutions}.  Thus, the method of continuity guarantees that $\overset{w_0,\gam_0}{\mathscr{J}_{m,N}^\tau}$ is an isomorphism for every $\tau \in [0,1]$ if and only if $\overset{w_0,\gam_0}{\mathscr{J}_{m,N}^0}$ is an isomorphism, so we reduce to proving that $\overset{w_0,\gam_0}{\mathscr{J}_{m,N}^0}$ is an isomorphism.
		
		Next, we note that the a priori estimate of Proposition \ref{prop on a priori estimates for regularized weak solutions} guarantees that $\overset{w_0,\gam_0}{\mathscr{J}_{m,N}^0}$ is injective, so we further reduce, by way of the bounded inverse theorem, to proving that this map is surjective.  In turn, we reduce to proving the existence of solutions to  \eqref{proposition on existence of regularized weak solutions EQN} with $g =0$ and $\tau=0$ with the help of the operator $\mathcal{B}_0$, exactly as in the first step in the proof of Proposition~\ref{prop on a priori estimates for weak solutions}.
		
		We have now shown that it suffices to establish the existence of solutions to \eqref{proposition on existence of regularized weak solutions EQN} with $g=0$ and $\tau =0$.  Thus, in the remainder of the proof we set $g=0$ and $\tau=0$.

		\textbf{Step 2}: Existence of a two parameter family of approximate solutions.  We introduce the approximation parameters $M,K\in\N^+$ and fix data $F\in({_0}H^1(\Omega;\R^n))^\ast$.  We then claim that there exists a collection
		\begin{equation}\label{the two paramter approximation sequence}
			\tcb{(q_{M,K},u_{M,K},\eta_{M,K})}_{M,K \in \N^+} \subset H^{2m}(\Omega)\times{_0}H^1(\Omega;\R^n)\times H^{2m}(\Sigma)
		\end{equation}
		satisfying
		\begin{equation}\label{weak formulation satisfied}
			\overset{\gam_0}{\mathscr{I}}(q_{M,K},u_{M,K},\eta_{M,K})=F,
		\end{equation}
		\begin{equation}\label{equations for q+geta}
			\begin{cases}
				N^{-1}\sp{K^{-1}+L_m}(q_{M,K}+\mathfrak{g}\eta_{M,K})=-\grad\cdot(\varrho u_{M,K})&\text{in }\Omega,\\
				\pd_n^m(q_{M,K}+\mathfrak{g}\eta_{M,K})=\cdots=\pd_n^{2m-1}(q_{M,K}+\mathfrak{g}\eta_{M,K})=0&\text{on }\pd\Omega,
			\end{cases}
		\end{equation}
		and
		\begin{equation}\label{equation for eta}
			\m{Tr}_{\Sigma}(u_{M,K}\cdot e_n)+\pd_1\eta_{M,K}=N^{-1}\tp{M^{-1}+(-\Delta_{\|})^{m-1/4}}\eta_{M,K}.
		\end{equation}
		The high-level idea for producing these approximate solutions is as follows.   First, we note that equation~\eqref{equation for eta} determines $\eta_{M,K}$ as a function of $u_{M,K}$. Then equation~\eqref{equations for q+geta} determines $q_{M,K}$ as a function of $u_{M,K}$ and $\eta_{M,K}$. These allow us to rewrite~\eqref{weak formulation satisfied} as an equation relating $F$ and $u_{M,K}$ alone, and it turns out that this can be solved by utilizing the Lax-Milgram lemma.
		
		To prove the claim we begin by recalling that the space ${_0}H^1(\Omega;\R^n)$ is defined by~\eqref{zero on the left} and defining the bounded linear maps $\m{p}_K: {_0}H^1(\Omega;\R^n)\to H^{2m}(\Omega)$ and
		$\upeta_M:{_0}H^1(\Omega;\R^n)\to H^{2m}(\Sigma)$ via
		\begin{equation}
		\m{p}_K(u)=-N(K^{-1}+L_m)^{-1}(\grad\cdot(\varrho u)) \text{ and }    \upeta_M(u)=N\tp{M^{-1}+(-\Delta_{\|})^{m-1/4}- N\pd_1}^{-1}\m{Tr}_\Sigma(u\cdot e_n).
		\end{equation}
		The map $\upeta_M$ is well-defined and bounded in light of the symbol inversion result in Lemma~\ref{lemma on simple symbol inversion}, while  $p_{\m{K}}$ is well-defined and bounded by virtue of Lemma~\ref{elliptic theory lemma for existence}.  These maps allow us to define the bilinear form $B_{M,K}:{_0}H^1(\Omega;\R^n)\times{_0}H^1(\Omega;\R^n)\to\R$ via 
		\begin{equation}
    		B_{M,K}(u,w)=\tbr{\overset{\gam_0}{\mathscr{I}}(\m{p}_K(u)-\mathfrak{g}\upeta_M(u),u,\upeta_M(u)),w}_{({_0}H^1)^\ast,{_0}H^1}.    
		\end{equation}
		 Thanks to integration by parts and the definitions of $\m{p}_K$ and $\upeta_M$, we have the equivalent formulation
		\begin{multline}
			B_{M,K}(u,w)=\int_{\Omega}\gam_0\tp{\gam_0\varrho u\otimes e_1+\S^{\varrho}u}:\grad w-\m{p}_{K}(u)\grad\cdot(\varrho w)
			\\+\tbr{(\mathfrak{g}\varrho-\varsigma\Delta_{\|})\upeta_{M}(u),\m{Tr}_{\Sigma}(w\cdot e_n)}_{H^{-1/2},H^{1/2}}\\
			=\int_{\Omega}\gam_0(\gam_0\varrho u\otimes e_1+\S^{\varrho}u):\grad w+\f{1}{N}\int_{\Omega}\f{1}{K}\m{p}_K(u)\m{p}_K(w)+\sum_{j=1}^n\pd_j^m\m{p}_K(u)\pd_j^m\m{p}_K(w)\\
			+\tbr{(\mathfrak{g}\varrho-\varsigma\Delta_{\|})\upeta_{M}(u),-\pd_1\upeta_M(w)+N^{-1}(M^{-1}+(-\Delta_{\|})^{m-1/4})\upeta_{M}(w)}_{H^{-1/2},H^{1/2}}.
		\end{multline}
		
		The boundedness of $B_{M,K}$ is straightforward to check, but it is also coercive since by anti-symmetry we have 
		\begin{equation}
		 \int_{\Omega}\gam_0^2\varrho u\otimes e_1:\grad u=0,   
		\end{equation}
		by orthogonality  we have 
		\begin{equation}
		\int_{\Omega}\S^{\varrho} u:\grad u=\int_{\Omega}\f{\upmu(\varrho)}{2}|\mathbb{D}^0u|^2+\uplambda(\varrho)(\grad\cdot u)^2,    
		\end{equation}
		by squaring we have
		\begin{equation}
		\f{1}{N}\int_{\Omega}\f{1}{K}\m{p}_K(u)^2+\sum_{j=1}^n(\pd_j^m\m{p}_K(u))^2\ge0,    
		\end{equation}
		and by anti-symmetry and symmetry considerations again we have
		\begin{multline}
			\tbr{(\mathfrak{g}\varrho-\varsigma\Delta_{\|})\upeta_{M}(u),-\gam\pd_1\upeta_M(w)+N^{-1}(M^{-1}+(-\Delta_{\|})^{m-1/4})\upeta_{M}(w)}_{H^{-1/2},H^{1/2}}\\=\tbr{(\mathfrak{g}\varrho-\varsigma\Delta_{\|})\upeta_M(u),N^{-1}(M^{-1}+(-\Delta_{\|})^{m-1/4})\upeta_M(u)}\ge0,
		\end{multline}
		which together with the Korn inequalities of Propositions~\ref{prop on deviatoric Korn's inequality} and~\ref{prop on Korn's inequality} imply the coercivity estimate
		\begin{equation}\label{the coercive inequality}
			B_{M,K}(u,u)\ge\gam_0\int_{\Omega}\f{\upmu(\varrho)}{2}|\mathbb{D}^0u|^2+\uplambda(\varrho)(\grad\cdot u)^2  \gtrsim \norm{u}_{H^1}^2.
		\end{equation}
		Recall that our conditions on $\upmu,\uplambda\in C^\infty(\R^+)$ are that $\upmu,\uplambda>0$ if $n=2$, and $\upmu>0$, $\uplambda\ge0$ if $n\ge 3$.
		The hypotheses of the Lax-Milgram lemma (see, for instance Theorem 6 in Chapter 6 of Lax~\cite{MR1892228}) are satisfied,  so we are therefore granted $u_{M,K}\in{_0}H^1(\Omega;\R^n)$ with the property that
		\begin{equation}\label{definition of uMK}
			B_{M,K}(u_{M,K},w)=\tbr{F,w}_{({_0}H^1)^\ast,{_0}H^1}  \text{ for all } w\in{_0}H^1(\Omega;\R^n).
		\end{equation}
		We then set $\eta_{M,K}=\upeta_K(u_{M,K}) \in H^{2m}(\Sigma)$ and $q_{M,K} = \m{p}_K(u_{M,K})-\mathfrak{g}\upeta_M(u_{M,K}) \in  H^{2m}(\Omega)$. By construction, the collection $\tcb{(q_{M,K},u_{M,K},\eta_{M,K})}_{M,K \in \N^+}$ satisfies the claimed inclusions and satisfies \eqref{weak formulation satisfied}, \eqref{equations for q+geta}, and \eqref{equation for eta}.
		
		\textbf{Step 3}: Estimates on the two parameter family of approximate solutions.  We claim that the approximate solutions \eqref{the two paramter approximation sequence} obey the $K$-independent bounds
		\begin{equation}\label{i never remember how to spell tomorrow good thing today it is}
			\tnorm{q_{M,K},u_{M,K},\eta_{M,K}}_{H^{2m}\times H^1\times H^{2m}}\lesssim_{m,M,N}\tnorm{F}_{({_0}H^1)^\ast}.
		\end{equation}
		To prove the claim, we first take $w=u_{M,K}$ in~\eqref{definition of uMK} and then use the coercive inequality~\eqref{the coercive inequality}  to get the control
		\begin{equation}\label{e niniu ai}
			\sup_{M,K \in \N^+ } \tnorm{u_{M,K}}_{H^1}\lesssim\tnorm{F}_{({_0}H^1)^\ast},
		\end{equation}
		where the implied constant only depends on the physical parameters.  Next, we employ the continuity of the map $\upeta_M$ to see that
		\begin{equation}\label{today I am labeling an equation with these keystrokes}
			\tnorm{\eta_{M,K}}_{H^{2m}}\lesssim_MN\tnorm{u_{M,N}}_{H^1}\lesssim_{M}N\tnorm{F}_{({_0}H^1)^\ast}.
		\end{equation}
		
		Obtaining a $K$-independent bound on the $q_{M,K}$ is slightly more involved.  We test the weak formulation \eqref{definition of uMK} with $w=\varrho^{-1}\mathcal{B}(q_{M,K}+\mathfrak{g}\eta_{M,K}) \in  {_0}H^1(\Omega;\R^n)$, where $\mathcal{B}$ is the right inverse to the divergence constructed in Corollary~\ref{Bogovskii 1}. Since $\grad\cdot(\varrho w) = q_{M,K} + \mathfrak{g}\eta_{M,K} = \m{p}_K(u_{M,K})$, the identity $B_{M,K}(u_{M,K},w)=\tbr{F,w}$ and the bounds \eqref{e niniu ai} and \eqref{today I am labeling an equation with these keystrokes} imply that
		\begin{equation}
			\tnorm{q_{M,K}+\mathfrak{g}\eta_{M,K}}_{L^2}\lesssim\tnorm{u_{M,K}}_{H^1}+\tnorm{\eta_{M,K}}_{H^{3/2}}+\tnorm{F}_{({_0}H^1)^\ast}\lesssim_{N,M}\tnorm{F}_{({_0}H^1)^\ast}.
		\end{equation}
		To get a higher-regularity estimate, we rewrite the first equation in~\eqref{equations for q+geta} as $(1+L_m)(q_{M,K}+\mathfrak{g}\eta_{M,K})=\f{K-1}{K}(q_{M,K}+\mathfrak{g}\eta_{M,K})-N\grad\cdot(\varrho u_{M,K})$ and then apply Lemma~\ref{elliptic theory lemma for existence} with $\kappa =1$ (and the already established $K$-independent bounds) to arrive at the estimate
		\begin{equation}\label{creativity is key when labeling these very important equation}
			\tnorm{q_{M,K}+\mathfrak{g}\eta_{M,K}}_{H^{2m}} \lesssim_m (1-K^{-1})\tnorm{q_{M,K}+\mathfrak{g}\eta_{M,K}}_{L^2}+N\tnorm{u_{M,K}}_{H^1}\lesssim_{m,M,N}\tnorm{F}_{({_0}H^1)^\ast}.
		\end{equation}
		Then \eqref{today I am labeling an equation with these keystrokes} and \eqref{creativity is key when labeling these very important equation} imply that $\tnorm{q_{M,K}}_{H^{2m}}  \lesssim_{m,M,N} \tnorm{F}_{({_0}H^1)^\ast},$ and we then combine this with~\eqref{e niniu ai} and~\eqref{today I am labeling an equation with these keystrokes} to complete the verification of~\eqref{i never remember how to spell tomorrow good thing today it is}, and hence the proof of the claim.

		\textbf{Step 4}: Existence of a one parameter family of approximate solutions.  With the fixed data $F\in({_0}H^1(\Omega;\R^n))^\ast$ as in the previous steps, we claim that there exists a sequence
		\begin{equation}\label{i was at the airport and i met an onion}
			\tcb{(q_M,u_M,\eta_M)}_{M \in \N^+ } \subset H^{2m}(\Omega)\times{_0}H^1(\Omega;\R^n)\times H^{2m}(\Sigma)
		\end{equation}
		such that 
		\begin{equation}\label{better}
			\overset{\gam_0}{\mathscr{I}}(q_M,u_M,\eta_M)=F,
		\end{equation}
		\begin{equation}\label{call}
			\begin{cases}
				N^{-1}L_m(q_M+\mathfrak{g}\eta_M)=-\grad\cdot(\varrho u_M)&\text{in }\Omega,\\
				\pd_n^m(q_M+\mathfrak{g}\eta_M)=\cdots=\pd_n^{2m-1}(q_M+\mathfrak{g}\eta_M)=0&\text{on }\pd\Omega,
			\end{cases}
		\end{equation}
		and
		\begin{equation}\label{saul}
			\m{Tr}_{\Sigma}(u_M\cdot e_n)+\pd_1\eta_M=N^{-1}\tp{M^{-1}+(-\Delta_{\|})^{m-1/4}}\eta_{M}.
		\end{equation}
		The existence of this sequence follows by taking a weak subsequential limit in the $K$ parameter in our previously constructed collection.  More precisely, for each fixed $M\in\N^+$ we have established in the previous step that the $K$-independent bounds~\eqref{i never remember how to spell tomorrow good thing today it is} hold. Thus, by weak compactness, there exist $(q_M,u_M,\eta_M)\in H^{2m}(\Omega)\times{_0}H^1(\Omega;\R^n)\times H^{2m}(\Sigma)$ that are a weak subsequential limit of the sequence $\tcb{(q_{M,K},u_{M,K},\eta_{M,K})}_{K \in \N^+}$. Routine weak convergence arguments applied to the  identities~\eqref{weak formulation satisfied}, \eqref{equations for q+geta}, and~\eqref{equation for eta} then show that \eqref{better}, \eqref{call}, and~\eqref{saul} hold.  This completes the construction and the proof of the claim.
		
		\textbf{Step 5}: Estimates on the one parameter family of approximate solutions.  We claim that the one parameter family of approximate solutions from \eqref{i was at the airport and i met an onion} obeys the $M$-independent bounds
		\begin{equation}\label{medium rare sought after bounds are delicious}
			\tnorm{q_M,u_M,\eta_M}_{H^{2m}\times H^1\times\mathcal{H}^{2m}}\lesssim_{m,N}\tnorm{F}_{({_0}H^1)^\ast}.
		\end{equation}
		To see this, we first employ the weak sequential lower semicontinuity of the norm and \eqref{e niniu ai} to obtain the estimate
		\begin{equation}\label{apple of my ogre}
			\sup_{M \in \N^+}\tnorm{u_M}_{H^1}\lesssim\tnorm{F}_{({_0}H^1)^\ast}.
		\end{equation}
		
		Next, as in the third step of the proof of Proposition~\ref{prop on a priori estimates for weak solutions}, in the case that $\varsigma>0$, we test~\eqref{better} against the function $w_M=-\varrho^{-1}\mathcal{B}_2(\varrho(b)\tbr{\grad_{\|}}^{-1}\Delta_{\|}\eta_M)\in{_0}H^1(\Omega;\R^n),$ where $\mathcal{B}_2$ is defined in Corollary~\ref{Bogovskii 2}.  This yields the identity
		\begin{equation}
			\tbr{F,w_M}_{({_0}H^1)^\ast,{_0}H^1}=\int_{\Omega}\gam_0(\gam_0\varrho u_M\otimes e_1+\S^{\varrho}u_M):\grad w_M+\tbr{(\mathfrak{g}\varrho-\varsigma\Delta_{\|})\eta_M,-\tbr{\grad_{\|}}^{-1}\Delta_{\|}\eta_M}_{H^{-1/2},H^{1/2}},
		\end{equation}
		which in turn yields the estimate
		\begin{equation}
			\tnorm{\grad_{\|}\eta_M}_{H^{1/2}}^2\lesssim\tp{\tnorm{F}_{({_0}H^1)^\ast}+\tnorm{u_M}_{H^1}}\tnorm{w_M}_{H^1}\lesssim\tnorm{F}_{({_0}H^1)^\ast}\tnorm{w_M}_{H^1}.
		\end{equation}
		The continuity of $\mathcal{B}_2$ shows that $\tnorm{w_M}_{H^1}\lesssim\tnorm{\grad_{\|}\eta_M}_{H^{1/2}}$, and hence
		\begin{equation}\label{Pennsylvania}
			\tnorm{\grad_{\|}\eta_M}_{H^{1/2}}\lesssim\tnorm{F}_{({_0}H^1)^\ast}.
		\end{equation}
		On the other hand, if $\varsigma=0$ and $n=2$, estimate~\eqref{Pennsylvania} follows by testing~\eqref{saul} with $\pd_1\tbr{\grad_{\|}}\eta_M=\grad_{\|}\tbr{\grad_{\|}}\eta_M\in H^{-1/2}(\Sigma)$ and using orthogonality, Cauchy-Schwarz, boundedness of traces, and~\eqref{apple of my ogre}.
		
		Now, as in the fourth step of Proposition~\ref{prop on a priori estimates for weak solutions}, we employ $w_M=\varrho^{-1}\mathcal{B}q \in{_0}H^1(\Omega;\R^n)$, where $\mathcal{B}$ is constructed in Corollary~\ref{Bogovskii 1}, as a test function in \eqref{better} in order to see that
		\begin{multline}
			\tbr{F,w_M}_{({_0}H^1)^\ast,{_0}H^1}=\int_{\Omega}\gam_0(\gam_0\varrho u_M\otimes e_1+\S^{\varrho}u_M):\grad w_M+\mathfrak{g}\varrho\grad\eta_M\cdot w_M-\int_{\Omega}q_M^2\\-\varsigma\tbr{\Delta_{\|}\eta_M,\m{Tr}_{\Sigma}(w_M\cdot e_n)}.
		\end{multline}
		From this we readily deduce the bound $\tnorm{q_M}_{L^2}^2\lesssim\tp{\tnorm{F}_{({_0}H^1)^\ast}+\tnorm{u_M}_{H^1}+\tnorm{\grad_{\|}\eta_M}_{H^{1/2}}}\tnorm{w_M}_{H^1}$, but this combines with our already established bounds and the continuity estimate $\tnorm{w_M}_{H^1}\lesssim\tnorm{q_M}_{L^2}$ to show the low regularity bound $\tnorm{q_M}_{L^2}\lesssim\tnorm{F}_{({_0}H^1)^\ast}.$
		
		We now have low regularity estimates on $q_M$ and $\eta_M$.  To promote these to high regularity bounds, we proceed as in the fifth step of the proof of Proposition~\ref{prop on a priori estimates for regularized weak solutions}.  The first identity in \eqref{call} is equivalent to $N^{-1}L_m(q_M+\mathfrak{g}\eta_M^{\m{H}}) = -\grad\cdot(\varrho u_M) - \mathfrak{g}N^{-1} L_{m,\|} \eta_{M}^{\m{L}}$, where we have decomposed $\eta_{M}^{\m{L}}=\Uppi^1_{\m{L}}\eta_M$ and $\eta_M^{\m{H}}=\Uppi^1_{\m{H}}\eta_M$. Since $\pd_n^m(q_M+\mathfrak{g}\eta_M^{\m{H}})=\cdots=\pd_n^{2m-1}(q_M+\mathfrak{g}\eta_M^{\m{H}})=0$, we can apply Lemma~\ref{regularized steady transport equations lemma 1} with $\tau=0$ and exploit the Fourier supports of $\eta_M^{\m{L}}$ and $\eta_M^{\m{H}}$ to obtain the estimate
		\begin{multline}\label{if i could make}
			\tnorm{q_M+\mathfrak{g}\eta_M^{\m{H}}}_{H^{2m}}
			\lesssim N\tp{\tnorm{u_M}_{H^1}   + \tnorm{ L_{m,\|} \eta_M^{\m{L}} }_{L^2} + \tnorm{q_M+\mathfrak{g}\eta_M^{\m{H}}}_{L^2}} \\
			\lesssim N\tp{\tnorm{u_M}_{H^1}+\tnorm{\grad_{\|}\eta_M}_{H^{1/2}}+\tnorm{q_M}_{L^2}}.
		\end{multline}
		Now we test the identity \eqref{saul} against $N^{-1}(-\Delta_{\|})^{m+1/4}\eta_{M}^{\m{H}} \in H^{-1/2}(\Sigma)$ and employ the Fourier support of $\eta_M^{\m{H}}$ as well as the bound \eqref{medium rare sought after bounds are delicious} to arrive at the estimate
		\begin{equation}\label{and strange as what}
			\tnorm{\eta_{M}^{\m{H}}}_{H^{2m}}\lesssim N\tnorm{u_M}_{H^1}\lesssim N\tnorm{F}_{({_0}H^1)^\ast}.
		\end{equation}
		Together, \eqref{if i could make} and~\eqref{and strange as what} imply that
		\begin{equation}\label{is whenever im with}
			\tnorm{q_M,\eta_M^{\m{H}}}_{H^{2m}\times H^{2m}}\lesssim N\tnorm{F}_{({_0}H^1)^\ast}.
		\end{equation}
		
		In light of Proposition~\ref{proposition on frequency splitting} and equation~\eqref{the norm on the anisotropic Sobolev spaces}, it remains only to obtain $M$-uniform  bounds on $\tsb{\pd_1\eta_M}_{\dot{H}^{-1}}$.  First, we integrate~\eqref{call} over $(0,b)$ in the $n^{\m{th}}$-coordinate and recall~\eqref{saul} to acquire the equality
		\begin{equation}\label{goodbye goodbye goodbye}
			\varrho(b)\pd_1\eta_M=\f{1}{N}L_{m,\|}\int_0^b(q_M+\mathfrak{g}\eta_M)+\f{\varrho(b)}{N}\bp{\f{1}{M}+(-\Delta_{\|})^{m-1/4}}\eta_M+(\grad_{\|},0)\cdot\int_0^b\varrho u_M.
		\end{equation}
		For $\kappa\in(0,1)$ we write $\eta_M^\kappa=(\Uppi^{1/\kappa}_{\m{L}}-\Uppi^\kappa_{\m{L}})\eta_M\in H^\infty(\Sigma)$ and then take the $L^2$ inner product of \eqref{goodbye goodbye goodbye} with $|\grad_{\|}|^{-2}\pd_1\eta_M^\kappa \in H^\infty(\Sigma)$.  The right hand side $\eta_M$ terms all vanish due to the $\partial_1$ operator, and we arrive at the equality
		\begin{equation}
			\varrho(b)\tsb{\pd_1\eta_M^\kappa}_{\dot{H}^{-1}}^2=\bbr{\f{|\grad_{\|}|^{-1}}{N}L_{m,\|}\int_0^bq_M+|\grad_{\|}|^{-1}(\grad_{\|},0)\cdot\int_0^b\varrho u_M,|\grad_{\|}|^{-1}\pd_1\eta_M^\kappa}_{L^2,L^2},
		\end{equation}
		from which we readily deduce that
		\begin{equation}\label{hay is for horses}
			\tsb{\pd_1\eta_M}_{\dot{H}^{-1}}=\lim_{\kappa\to0}\tsb{\pd_1\eta_M^\kappa}_{\dot{H}^{-1}}\lesssim_mN^{-1}\tnorm{q_M}_{H^{2m-1}}+\tnorm{u_M}_{L^2}\lesssim_m\tnorm{F}_{({_0}H^1)^\ast}.
		\end{equation}
		Then  \eqref{medium rare sought after bounds are delicious} follows by combining~\eqref{apple of my ogre}, \eqref{is whenever im with}, and~\eqref{hay is for horses} and recalling the equivalent norm  on $\mathcal{H}^{2m}(\Sigma)$ given in \eqref{the norm on the anisotropic Sobolev spaces}.  The claim is proved.
		
		\textbf{Step 6}: Conclusion.  The $M$-uniform bounds  \eqref{medium rare sought after bounds are delicious} guarantee the existence of a weak subsequential limit for the sequence~\eqref{i was at the airport and i met an onion}, say $(q,u,\eta)\in H^{2m}(\Omega)\times{_0}H^1(\Omega;\R^n)\times\mathcal{H}^{2m}(\Sigma)$. Routine weak convergence arguments applied to the identities~\eqref{better}, \eqref{call}, and~\eqref{saul} then reveal that $(q,u,\eta)$ satisfy $\overset{\gam_0}{\mathscr{I}}(q,u,\eta)=F$,
		\begin{equation}
			\begin{cases}
				N^{-1}L_m(q+\mathfrak{g}\eta)=-\grad\cdot(\varrho u)&\text{in }\Omega,\\
				\pd_n^m(q+\mathfrak{g}\eta)=\cdots=\pd_n^{2m-1}(q+\mathfrak{g}\eta)=0&\text{on }\pd\Omega,
			\end{cases}
		\end{equation}
		and $\m{Tr}_{\Sigma}(u\cdot e_n)+\pd_1\eta=N^{-1}(-\Delta_{\|})^{m-1/4}\eta$. Therefore $(q,u,\eta)\in\X^{-1}_{m,N}$ satisfy $\overset{w_0,\gam_0}{\mathscr{J}^0_{m,N}}(q,u,\eta)=(0,F)$, and so the proof is complete in light of the first step.
	\end{proof}

	As a consequence of the existence of weak solutions to the regularization, we now show that the operators associated to the strong formulation of the regularization, namely $\overset{w_0,\gam_0}{A_{m,N}}$ defined in~\eqref{the map of the regularized principal part operator}, are automatically isomorphisms. The catch is that at this point we cannot guarantee the inverses come with estimates independent of $N$, so will will have to work harder in subsequent sections to verify this. 
	
	\begin{coro}[Isomorphisms induced by the regularization]\label{coro on regularity of the regularization}
		Let $0<\rho\le\rho_{\m{WD}}$, where the latter is from Theorem~\ref{thm on smooth tameness of the nonlinear operator}, $w_0=(q_0,u_0,\eta_0)$ be as in Lemma~\ref{properties of the principal parts vector field}, $m,N\in\N^+$ with $m\ge 2$, and $\gam_0\in I$ for an interval $I\Subset\R^+$. Suppose that \eqref{the labeled conditions to be satisfaction} holds.  Then for $\upnu\in\N^+$, the map
		\begin{equation}\label{girls go to jupiter to get more frogs}
			\overset{w_0,\gam_0}{A_{m,N}}:\X_{m,N}^\upnu\to\Y^\upnu
		\end{equation}
		defined by~\eqref{the application of the map of the reg princ part oper and also will be known as ziggy mcpiggson pigface ponylicker frogs made of wicker but not a football kicker boy i sure do have a problem i need to stop labeling like this}  is a Banach isomorphism.
	\end{coro}
	\begin{proof}
		That the map~\eqref{girls go to jupiter to get more frogs} is well-defined is a consequence of the third item of Lemma~\ref{lem on well-definedness check for linear analysis}. We begin by proving injectivity. Suppose that $(q,u,\eta)\in\m{ker}\overset{w_0,\gam_0}{A_{m,N}}\subseteq\X_{m,N}^{\upnu}$. Lemma~\ref{lem on strong solutions are weak solutions} shows that strong solutions are weak solutions, and hence $\overset{w_0,\gam_0}{\mathscr{J}^1_{m,N}}(q,u,\eta)=(0,0)$.  We then deduce that $(q,u,\eta)=0$ by invoking the a priori estimates of Proposition~\ref{prop on a priori estimates for regularized weak solutions}.
		
		We now turn to the proof of surjectivity. Suppose that
		\begin{equation}\label{the data is included so it doesnt feel left out}
			(g,f,k)\in\Y^\upnu,
		\end{equation}
		and, by utilizing Theorem~\ref{thm on existence of regularized weak solutions} and the map from \eqref{definition of the K functional dude lies here rest in peace and serrano peppers}, define $(q,u,\eta)$ via
		\begin{equation}\label{ocean man take me by the ocean oops}
			(q,u,\eta)=\sp{\overset{w_0,\gam_0}{\mathscr{J}^1_{m,N}}}^{-1}(g,\mathscr{K}(f,k))  \in\X^{-1}_{m,N}.
		\end{equation}
		We claim that we have the higher-regularity inclusion $(q,u,\eta)\in\X^{\upnu}_{m,N}$. Once this is shown, another application of Lemma~\ref{lem on strong solutions are weak solutions}  reveals that $\overset{w_0,\gam_0}{A_{m,N}}(q,u,\eta)=(g,f,k)$, which establishes surjectivity.
		
		To prove the claim we will employ a finite induction argument to promote the regularity of the triple $(q,u,\eta)$ one step at a time.  To this end, it is useful to unpack the definition of weak solution in a more helpful way.  Equation~\eqref{ocean man take me by the ocean oops} is equivalent to: $(q,u,\eta)\in\X_{-1}$,
        \begin{equation}\label{qgeta systa}
			\begin{cases}
				N^{-1}L_m q = g - \grad\cdot(v_{w_0}(q+\mathfrak{g}\eta))  - \nabla \cdot(\varrho u) - 
				\mathfrak{g}N^{-1}L_{m,\|}\eta &\text{in }\Omega,\\
				\pd_n^m q =\cdots=\pd_n^{2m-1} q=0&\text{on }\pd\Omega,
			\end{cases}
		\end{equation}
		\begin{equation}\label{eta systa}
			N^{-1}(-\Delta_{\|})^{m-1/4}\Uppi^1_{\m{H}}\eta = - N^{-1}(-\Delta_{\|})^{m-1/4}\Uppi^1_{\m{L}}\eta +  \pd_1\eta + \m{Tr}_{\Sigma}(u\cdot e_n),\quad\m{Tr}_{\Sigma_0}(u)=0,
		\end{equation}
		and for all $w\in{_0}H^1(\Omega;\R^n)$ it holds that
		\begin{equation}\label{ziggy_piggy_7}
			\int_{\Omega}\gam_0\tp{\gam_0\varrho u\otimes e_1+\S^{\varrho}u}:\grad w =\tbr{\widetilde{F},w}_{({_0}H^1)^\ast,{_0}H^1},
		\end{equation}
		where $\tilde{F}=\mathscr{K}(\tilde{f},\tilde{k})$ for $\widetilde{f}=f-\varrho\grad(q+\mathfrak{g}\eta)$ and $\widetilde{k}=k+(\varrho q+\varsigma\Delta_{\|}\eta)e_n$.
		
		The identity \eqref{ziggy_piggy_7} means, in other words, that $u$ is a weak solution to the elliptic boundary value problem
		\begin{equation}\label{move all of the nonvelocity to the other side}
			\begin{cases}
				-\gam_0^2\varrho\pd_1u-\gam_0\grad\cdot\S^{\varrho}u=\tilde{f}&\text{in }\Omega,\\
				\gam_0\S^{\varrho}ue_n=\tilde{k}&\text{on }\Sigma,\\
				u=0&\text{on }\Sigma_0.
			\end{cases}
		\end{equation}
		From the inclusions \eqref{the data is included so it doesnt feel left out} and $(q,\eta)\in H^{2m}(\Omega)\times\mathcal{H}^{2m}(\Sigma)$ and the norm~\eqref{the norm on the anisotropic Sobolev spaces} we deduce that
		\begin{equation}
			\widetilde{f}\in H^{\min\tcb{\upnu,2m-1}}(\Omega;\R^n) 
			\text{ and }
			\widetilde{k}\in H^{\min\tcb{1/2+\upnu,2m-2}}(\Sigma;\R^n),
		\end{equation}
		and so the standard elliptic regularity gain for the problem \eqref{move all of the nonvelocity to the other side} (see, for instance, Agmon, Douglis, and Nirenberg~\cite{MR162050}) guarantees the inclusion
		\begin{equation}\label{base case promotion of u}
			u\in H^{\min\tcb{2+\upnu,2m-1/2}}(\Omega;\R^n)\emb H^2(\Omega;\R^n),
		\end{equation}
		where the embedding holds since $m\ge 2$.  With the improved $u$ regularity from \eqref{base case promotion of u} in hand, we return to~\eqref{eta systa} to see that $\Uppi^1_{\m{H}}\eta \in H^{1+2m}(\Sigma)$, and hence, by Proposition~\ref{proposition on frequency splitting}, $\eta  \in\mathcal{H}^{1+2m}(\Sigma)$. In turn, we use use the improved $u$ and $\eta$ regularity together with the fourth item of Lemma~\ref{properties of the principal parts vector field} in~\eqref{qgeta systa}, appealing to Lemma~\ref{lem on a priori estimates for Lm} to deduce the improvement $q\in H^{1+2m}(\Omega)$.
		
		Proceeding by finite induction, we assume now that for some $0\le\nu\le\upnu-1$ we have the inclusion
		\begin{equation}\label{this is an equation station boyz}
			(q,u,\eta)\in H^{\nu+1+2m}(\Omega)\times H^{\nu+2}(\Omega;\R^n)\times\mathcal{H}^{\nu+1+2m}(\Sigma).
		\end{equation}
		This implies that $\widetilde{f}\in H^{\min\tcb{\upnu,\nu+2m}}(\Omega;\R^n)$ and $\widetilde{k}\in H^{\min\tcb{1/2+\upnu,\nu-1+2m}}(\Sigma;\R^n)$, and so elliptic regularity for \eqref{move all of the nonvelocity to the other side} implies that $u\in H^{\min\tcb{2+\upnu,\nu+1/2+2m}}(\Omega;\R^n)\emb H^{\nu+3}(\Omega;\R^n)$. We then argue exactly as above to use the improved $u$ regularity to promote to $\eta\in\mathcal{H}^{\nu+2+2m}(\Sigma)$ and $q\in H^{\nu+2+2m}(\Omega)$. Thus, \eqref{this is an equation station boyz} holds with $\nu$ replaced by $\nu+1$.  By finite induction, \eqref{this is an equation station boyz} then also holds for $\nu = \upnu$, and so $(q,u,\eta)\in\X^\upnu_{m,N}$, which completes the proof of the claim.
	\end{proof}
	
	% - space - space - outer - % - space - space - outer - % - space - space - outer - % - space - space - outer - % - space - space - outer - % - space - space - outer - % - space - space - outer - % - space - space - outer - % - space - space - outer - % - space - space - outer - % - space - space - outer - % - space - space - outer -
	
	\section{Analysis of strong solutions to the linearization}\label{alaskan black cod}

	Previously, we have established estimates for weak solutions to~\eqref{principal part of the linearization} and \eqref{regularization of the prinpal part of the linearization}, and for the latter we proved existence and qualitative regularity in Corollary~\ref{coro on regularity of the regularization}. The majority of this section is devoted to building a series of tools to aid in the estimation of the higher regularity norms of these weak solutions when given sufficiently regular data. In other words, we require a specific quantitative understanding of the regularity for systems~\eqref{principal part of the linearization} and~\eqref{regularization of the prinpal part of the linearization}; for the former we will develop tame estimates of the solution norms with respect to the background $w_0=(q_0,u_0,\eta_0)$ and $\gam_0$, while for the latter we will prove high regularity bounds that are independent of the approximation parameter.
	
	We proceed as follows. Section~\ref{running everywhere at such a speed} analyzes the equations satisfied by the tangential derivatives of the solutions to systems~\eqref{principal part of the linearization} and~\eqref{regularization of the prinpal part of the linearization}. Section~\ref{section on analysis of normal systems} reduces the regularity promotion of solutions to estimates on tangential derivatives via analysis of the so-called normal system. Section~\ref{till they find, there's no need} combines results from the tangential derivative and normal system analysis to derive the sought-after precise a priori estimates for the principal part equations. We then conclude in Section~\ref{dotting his socks in the night} by deducing the existence of strong solutions to~\eqref{principal part of the linearization} and then to the full linearization.
	
	\subsection{Analysis of tangential derivatives}\label{running everywhere at such a speed}
	
	In order to study the tangential derivatives of solutions to ~\eqref{principal part of the linearization} and~\eqref{regularization of the prinpal part of the linearization}, we must first understand the commutators of the operators $\overset{w_0,\gam_0}{A}$, $\overset{w_0,\gam_0}{A_{m,N}}$, $\overset{w_0,\gam_0}{\mathscr{J}}$, and $\overset{w_0,\gam_0}{\mathscr{J}^1_{m,N}}$ with the tangential derivatives $\pd_j$ for $j\in\tcb{1,\dots,n-1}$.  These operators are nearly tangentially translation invariant, and so nearly commute with the tangential derivatives; the failure of each to commute is precisely due to the appearance of $\grad\cdot(v_{w_0}(q+\mathfrak{g}\eta))$ in their continuity equations. The following definition and subsequent lemma capture this almost tangential translation invariance. We recall that the $\hat{H}^s(\Omega)$ spaces are defined in~\eqref{Rhode Island}.
	
	\begin{defn}[The principal part commutator]\label{commutator_def}
		Let $0<\rho\le\rho_{\m{WD}}$, where the latter is defined in Theorem~\ref{thm on smooth tameness of the nonlinear operator}, and let $w_0=(q_0,u_0,\eta_0)$ be as in Lemma~\ref{properties of the principal parts vector field}. For $j\in\tcb{1,\dots,n-1}$ we define the map $\overset{w_0}{\mathscr{C}^j}:H^{1+s}(\Omega)\times\mathcal{H}^{1+s}(\Sigma)\to\hat{H}^{s}(\Omega)$ via $\overset{w_0}{\mathscr{C}^j}(q,\eta)=\grad\cdot(\pd_jv_{w_0}(q+\mathfrak{g}\eta))$.
	\end{defn}
	
	We now check that the $\mathscr{C}$ maps are well-defined and then explore their mapping properties.
	
	\begin{lem}[Mapping properties of $\mathscr{C}$]\label{lem on mapping properties of the C}
		Under the hypotheses of Definition~\ref{commutator_def}, we have the following estimates for $s\in\N$:
		\begin{equation}\label{sought me and you}
			\snorm{\overset{w_0}{\mathscr{C}^j}(q,\eta)}_{\hat{H}^{s}}\lesssim\rho\tnorm{q,\eta}_{H^{1+s}\times\mathcal{H}^{1+s}}+\begin{cases}
				0&\text{if }s\le\tfloor{n/2},\\
				\tnorm{q_0,u_0,\eta_0}_{\X_{1+s}}\tnorm{q,\eta}_{H^{1+\tfloor{n/2}}\times\mathcal{H}^{1+\tfloor{n/2}}}&\text{if }\tfloor{n/2}<s.
			\end{cases}
		\end{equation}
		Here the implicit constant depends only on $s$, the physical parameters, the dimension, and $\rho_{\m{WD}}$.
	\end{lem}
	\begin{proof}
		We first use Lemma~\ref{properties of the principal parts vector field} (in particular the splitting~\eqref{fundamental decomposition of the vector field v}) to compute 
		\begin{equation}\label{art}
			\pd_jv_{w_0}=\pd_jv_{q_0,u_0,\eta_0}^{(1)}+\pd_jv^{(2)}_{\eta_0}.
		\end{equation}
		Then, given $\N\ni s\ge 1+\lfloor n/2\rfloor$, we use the first item of Lemma~\ref{properties of the principal parts vector field} to estimate
		\begin{equation}\label{deco}
			\tnorm{\pd_jv^{(1)}_{q_0,u_0,\eta_0}}_{H^s}\le\tnorm{v_{q_0,u_0,\eta_0}^{(1)}}_{H^{1+s}}\lesssim\tnorm{q_0,u_0,\eta_0}_{\X_s}.
		\end{equation}
		For the other piece, we may use the second item of Lemma~\ref{properties of the principal parts vector field} to deduce that $v^{(2)}_{\eta_0}=(v^{(2)}_{\eta_0}\cdot e_1)e_1$ and that $v^{(2)}_{\eta_0}$ has $r_{n-1}$ as a band-limit. This, in addition to the second item of the aforementioned lemma, \eqref{the norm on the anisotropic Sobolev spaces}, and Proposition~\ref{proposition on frequency splitting}, allow us to estimate
		\begin{equation}\label{empire state building}
			\tnorm{\pd_jv_{\eta_0}^{(2)}}_{H^s} \lesssim \sup_{0\le p\le s}\sup_{y\in[0,b]}\tnorm{\pd_j\pd_n^pv_{\eta_0}^{(2)}(\cdot,y)}_{H^{s-p}(\R^{n-1})}\le\sup_{0\le p\le s}\sup_{y\in[0,b]}\tnorm{\pd_n^pv_{\eta_0}(\cdot,y)\cdot e_1}_{\mathcal{H}^0}\lesssim\tnorm{\Uppi^1_{\m{L}}\eta_0}_{\mathcal{H}^0}.
		\end{equation}
		Together, \eqref{art}, \eqref{deco}, and~\eqref{empire state building} imply the bound $\tnorm{\pd_jv_{w_0}}_{H^s}\lesssim\tnorm{q_0,u_0,\eta_0}_{\X_s}$ for $s\ge 1+\lfloor n/2\rfloor$. With this established, proving the stated estimates is a simple application of Corollary~\ref{corollary on tame estimates on simple multipliers}.  Indeed, it shows that for any $s\in\N$ and $(q,\eta)\in H^{1+s}(\Omega)\times\mathcal{H}^{1+s}(\Sigma)$, we have
		\begin{multline}\label{high norm estimate}
			\snorm{\overset{w_0}{\mathscr{C}^j}(q,\eta)}_{H^s}\lesssim\tnorm{\pd_jv_{w_0}(q+\mathfrak{g}\eta)}_{H^{1+s}}\lesssim\rho\tnorm{q,\Uppi^1_{\m{H}}\eta,\Uppi^1_{\m{L}}\eta}_{H^{1+s}\times H^{1+s}\times W^{1+s,\infty}}\\
			+\begin{cases}
				0&\text{if }s<\lfloor n/2\rfloor,\\
				\tnorm{\pd_jv_{q_0,u_0,\eta_0}}_{H^{1+s}}\tnorm{q,\Uppi^1_{\m{H}}\eta,\Uppi^1_{\m{L}}\eta}_{H^{1+\lfloor n/2\rfloor}\times H^{1+\lfloor n/2\rfloor}\times W^{1+\lfloor n/2\rfloor,\infty}}&\text{if }\lfloor n/2\rfloor\le s,
			\end{cases}
			\\\lesssim\rho\tnorm{q,\eta}_{H^{1+s}\times\mathcal{H}^{1+s}}+\begin{cases}
				0&\text{if }s<\lfloor n/2\rfloor,\\
				\tnorm{q_0,u_0,\eta_0}_{\X_{1+s}}\tnorm{q,\eta}_{H^{1+\lfloor n/2\rfloor}\times\mathcal{H}^{1+\lfloor n/2\rfloor}}&\text{if }\lfloor n/2\rfloor\le s.
			\end{cases}
		\end{multline}
		Additionally, since $\m{Tr}_{\pd\Omega}(v_{w_0}\cdot e_n)=0$ and  $\pd_jv_{w_0}\in(L^\infty\cap L^2)(\Omega)$, we have the estimate
		\begin{equation}\label{funny extra bit}
			\bsb{\int_0^b\overset{w_0}{\mathscr{C}^j}(q,\eta)(\cdot,y)\;\m{d}y}_{\dot{H}^{-1}}\lesssim\tnorm{\pd_jv_{w_0}(q+\mathfrak{g}\eta)}_{L^2}\lesssim\rho\tnorm{q,\Uppi^1_{\m{H}}\eta,\Uppi^1_{\m{L}}\eta}_{L^2\times L^2\times L^\infty}\lesssim\rho\tnorm{q,\eta}_{L^2\times\mathcal{H}^0}.
		\end{equation}
		Then~\eqref{sought me and you} follows by combining~\eqref{high norm estimate} and~\eqref{funny extra bit}.
	\end{proof}
	
	We now apply the mapping properties of $\mathscr{C}$ to obtain low regularity estimates for first order tangential derivatives.

	\begin{prop}[Low norm estimates on tangential derivatives]\label{prop on commutators 1}
		Under the hypotheses of Definition~\ref{commutator_def}, the following hold for $\gam_0\in I$ with $I\Subset\R^+$ an interval, $(g,f,k)\in\Y^0$, and $j \in \{1,\dotsc,n-1\}$.
		\begin{enumerate}
			\item If $(q,u,\eta)\in\overset{q_0,u_0,\eta_0}{\X^{0}}$ satisfies $\overset{w_0,\gam_0}{A}(q,u,\eta)=(g,f,k)$, then $(\pd_jq,\pd_ju,\pd_j\eta) \in \overset{q_0,u_0,\eta_0}{\X^{-1}}$ and obeys the estimate 
			\begin{equation}\label{arizona}
				\snorm{\overset{w_0,\gam_0}{\mathscr{J}}(\pd_jq,\pd_ju,\pd_j\eta)}_{\Y^{-1}}\lesssim\rho\tnorm{q,u,\eta}_{\X_{0}}+\tnorm{g,f,k}_{\Y^0}.
			\end{equation}
			
			\item If $m,N\in\N^+$ with $m\ge 2$ and $(q,u,\eta)\in\X^0_{m,N}$ satisfy $\overset{w_0,\gam_0}{A_{m,N}}(q,u,\eta)=(g,f,k)$, then  $(\pd_jq,\pd_ju,\pd_j\eta) \in \X^{-1}_{m,N}$ and obeys the estimate
			\begin{equation}\label{arkansas}
				\snorm{\overset{w_0,\gam_0}{\mathscr{J}^1_{m,N}}(\pd_jq,\pd_ju,\pd_j\eta)}_{\Y^{-1}}\lesssim\rho\tnorm{q,u,\eta}_{\X_0}+\tnorm{g,f,k}_{\Y^0}.
			\end{equation}
		\end{enumerate}
		Here the implicit constants depend on the dimension, the physical parameters, $\rho_{\m{WD}}$, and $I$.
	\end{prop}
	\begin{proof}
		We will prove only the first item;  the proof of the second follows from a nearly identical argument.  We begin with three observations.
		
		First, recall the bounded linear map  $\mathscr{K} : L^2(\Omega;\R^n) \times H^{1/2}(\Sigma;\R^n)  \to ({_0}H^1(\Omega;\R^n))^\ast$ defined by \eqref{definition of the K functional dude lies here rest in peace and serrano peppers}.  For each $j \in \{1,\dotsc,n-1\}$ we construct a related bounded linear map $\mathscr{K}_j : L^2(\Omega;\R^n) \times H^{1/2}(\Sigma;\R^n)  \to ({_0}H^1(\Omega;\R^n))^\ast$ as follows.  For $(f,k) \in L^2(\Omega;\R^n) \times H^{1/2}(\Sigma;\R^n)$ we initially define the linear map $\mathscr{K}_j(f,k) : {_0}H^1(\Omega;\R^n) \cap H^2(\Omega;\R^n) \to \R$ via $\tbr{\mathscr{K}_j(f,k) ,w }= \tbr{\mathscr{K}(f,k) , -\pd_j w }_{({_0}H^1)^\ast, {_0}H^1}  =  -\int_{\Omega} f\cdot \pd_j w - \int_{\Sigma}k\cdot \pd_j w$. For such $f$, $k$, and $w$, we may bound 
		\begin{equation}
			\tabs{\tbr{\mathscr{K}_j(f,k),w}} \le \tnorm{f}_{L^2}\tnorm{\pd_j w}_{L^2} + \tnorm{k}_{H^{1/2}} \tnorm{\pd_j\m{Tr}_{\Sigma} w}_{H^{-1/2}}
			\lesssim \tnorm{f,k}_{L^2\times H^{1/2}}\tnorm{ w}_{H^1},
		\end{equation}
		and from this and the density of ${_0}H^1(\Omega;\R^n) \cap H^2(\Omega;\R^n)$ in ${_0}H^1(\Omega;\R^n)$ we deduce that $\mathscr{K}_j(f,k)$ uniquely extends to an element $\mathscr{K}_j(f,k) \in ({_0}H^1(\Omega;\R^n))^\ast$ such that 
		\begin{equation}\label{zigmund_das_schwein_8}
			\tbr{\mathscr{K}_j(f,k) ,w }_{({_0}H^1)^\ast, {_0}H^1} = \tbr{\mathscr{K}(f,k) , -\pd_j w }_{({_0}H^1)^\ast, {_0}H^1}
		\end{equation}
		for all $w \in H^2(\Omega; \R^n) \cap {_0}H^1(\Omega;\R^n)$ and $\tnorm{\mathscr{K}_j(f,k) }_{({_0}H^1)^\ast} \lesssim  \tnorm{f,k}_{L^2\times H^{1/2}}$.  In particular, the latter estimate shows that the induced map $\mathscr{K}_j$ is bounded and linear with the domain and codomain stated above.
		
		Second, we recall the map $\overset{\gam_0}{\mathscr{I}} : \X_{-1} \to ({_0}H^1(\Omega;\R^n))^\ast$ defined by \eqref{definition of the I functional}.  Suppose that $(q,u,\eta) \in \X_0$.  For $j \in \{1,\dotsc,n-1\}$ we have that $(\pd_j q, \pd_j u, \pd_j \eta) \in \X_{-1}$,  which means that $\overset{\gam_0}{\mathscr{I}}(\pd_j q, \pd_j u, \pd_j \eta)$ defines an element of $({_0}H^1(\Omega;\R^n))^\ast$.  We may compute the action of this functional on any  $w \in H^2(\Omega;\R^n) \cap {_0}H^1(\Omega;\R^n)$ by integrating by parts:  
		\begin{multline}\label{zigmund_das_schwein_9}
			\tbr{\overset{\gam_0}{\mathscr{I}}(q,u,\eta), -\pd_j w}_{({_0}H^1)^\ast,{_0}H^1} 
			=    -\int_{\Omega}-\gam_0^2\varrho\pd_1u\cdot \pd_jw - q\grad\cdot(\varrho\pd_jw) + \mathfrak{g}\varrho\grad\eta\cdot\pd_j w + \gam_0\S^{\varrho}u : \grad \pd_jw \\
			+ \varsigma\tbr{\Delta_{\|}\eta,\m{Tr}_{\Sigma}(\pd_jw\cdot e_n)}_{H^{-1/2},H^{1/2}}  
			= \int_{\Omega}-\gam_0^2\varrho\pd_1\pd_ju\cdot w-\pd_jq\grad\cdot(\varrho w)+\mathfrak{g}\varrho\grad\pd_j\eta\cdot w\\+\int_{\Omega}\gam_0\S^{\varrho}\pd_ju:\grad w
			-\varsigma\tbr{\Delta_{\|}\pd_j\eta,\m{Tr}_{\Sigma}(w\cdot e_n)}_{H^{-1/2},H^{1/2}}=\tbr{\mathscr{I}(\pd_jq,\pd_ju,\pd_j\eta),w}_{({_0}H^1)^\ast,{_0}H^1}.
		\end{multline}

		For the third observation, suppose that $q\in H^1(\Omega)$ satisfies $\grad\cdot(v_{w_0}q)\in H^{1}(\Omega)$.   Testing against members of $C^\infty_c(\Omega)$ and appealing to Lemma~\ref{lem on mapping properties of the C}, we see that for each $j \in \{1,\dotsc,n-1\}$ we have the distributional identity 
		\begin{equation}\label{zigmund_das_schwein_10}
			\grad\cdot(v_{w_0}\pd_jq) = \pd_j\grad\cdot(v_{w_0}q)-\overset{w_0}{\mathscr{C}^j}(q,0) \in L^2(\Omega).
		\end{equation}
		In particular, this identity establishes that  $\pd_j q\in H^0_{v_{w_0}}(\Omega)$, where this space is defined in  Appendix~\ref{appendix on adapted Sobolev spaces}.
		
		Having established these three observations, we are now ready to prove the first item.  Suppose that  $(q,u,\eta) \in \overset{q_0,u_0,\eta_0}{\X^{0}}$ and $(g,f,k) \in \Y^0$ satisfy the strong from equation $\overset{w_0,\gam_0}{A}(q,u,\eta)=(g,f,k)$, which (due to the assumed level of regularity) is equivalent to the weak form equation $\overset{w_0,\gam_0}{\mathscr{J}}(q,u,\eta) = (g,\mathscr{K}(f,k))$ (see Lemma~\ref{lem on strong solutions are weak solutions}).  In turn, the weak formulation unpacks into the pair of equations
		\begin{equation}\label{zigmund_das_schwein_11}
			\grad\cdot(\varrho  u) + \grad\cdot(v_{w_0}  (q+\mathfrak{g}\eta))= g \text{ in } L^2(\Omega)  \text{ and } \mathscr{I}(q,u,\eta) = \mathscr{K}(f,k) \text{ in } ({_0}H^1(\Omega;\R^n))^\ast.
		\end{equation}
		Let $j \in \{1,\dotsc,n-1\}$.  For the first equation in \eqref{zigmund_das_schwein_11} we apply $\pd_j$ and use \eqref{zigmund_das_schwein_10} from the third observation to see that
		\begin{equation}\label{alabama}
			\grad\cdot(\varrho\pd_ju)+\grad\cdot(v_{w_0}\pd_j(q+\mathfrak{g}\eta))=\pd_jg-\overset{w_0}{\mathscr{C}^j}(q,\eta).
		\end{equation}
		For the second equation in \eqref{zigmund_das_schwein_11} we let $w \in H^2(\Omega;\R^n) \cap {_0}H^1(\Omega;\R^n)$, test against $-\pd_j w$, and use \eqref{zigmund_das_schwein_8} and \eqref{zigmund_das_schwein_9} from the first and second observations to see that 
		\begin{multline}
			\tbr{\overset{\gam_0}{\mathscr{I}}(\pd_jq,\pd_ju,\pd_j\eta),w}_{({_0}H^1)^\ast,{_0}H^1} 
			= \tbr{\overset{\gam_0}{\mathscr{I}}(q,u,\eta), -\pd_j w}_{({_0}H^1)^\ast,{_0}H^1} \\
			= \tbr{\mathscr{K}(f,k), -\pd_j w}_{({_0}H^1)^\ast,{_0}H^1} 
			= \tbr{\mathscr{K}_j(f,k) ,w }_{({_0}H^1)^\ast, {_0}H^1}.
		\end{multline}
		Since $H^2(\Omega;\R^n) \cap {_0}H^1(\Omega;\R^n)$ is dense in ${_0}H^1(\Omega;\R^n)$, we deduce that 
		\begin{equation}\label{alaska}
			\overset{\gam_0}{\mathscr{I}}(\pd_j q,\pd_j u,\pd_j \eta)= \mathscr{K}_{j}(f,k).
		\end{equation}
		By combining \eqref{alabama} and \eqref{alaska}, we deduce that
  \begin{equation}\label{did i hear you say that you dislike my labels}
      \overset{w_0,\gam_0}{\mathscr{J}}(\pd_jq,\pd_ju,\pd_j\eta)=\tp{\pd_jg-\overset{w_0}{\mathscr{C}^j}(q,\eta), \mathscr{K}_j(f,k)}.
  \end{equation}
 Therefore, \eqref{arizona} follows by taking the norm in $\Y^{-1}$ of \eqref{did i hear you say that you dislike my labels} and applying both Lemma~\ref{lem on mapping properties of the C} and the boundedness of $\mathscr{K}_j$ established in the first observation.
	\end{proof}
	
	The next result is a higher regularity version of the previous one.
	
	\begin{prop}[High norm estimates on tangential derivatives]\label{prop on commutators 2}
		Under the hypotheses of Definition~\ref{commutator_def}, the following hold for $s\in\N^+$, $j \in \{1,\dotsc,n-1\}$, $\gam_0\in I\Subset\R^+$ for $I$ an interval, and $(g,f,k)\in\Y^s$.
		\begin{enumerate}
			\item If $(q,u,\eta)\in\overset{q_0,u_0,\eta_0}{\X^s}$ satisfies $\overset{w_0,\gam_0}{A}(q,u,\eta)=(g,f,k)$, then  $(\pd_jq,\pd_ju,\pd_j\eta)  \in \overset{q_0,u_0,\eta_0}{\X^{s-1}}$ and obeys the estimate
			\begin{multline}\label{california}
				\snorm{\overset{w_0,\gam_0}{A}(\pd_jq,\pd_ju,\pd_j\eta)}_{\Y^{s-1}}\lesssim\rho\tnorm{q,u,\eta}_{\X_s}\\+\tnorm{g,f,k}_{\Y^s}+\begin{cases}
					0&\text{if }s\le\tfloor{n/2},\\
					\tnorm{q_0,u_0,\eta_0}_{\X_{1+s}}\tnorm{q,u,\eta}_{\X_{\tfloor{n/2}}}&\text{if }\tfloor{n/2}< s.
				\end{cases}
			\end{multline}
			\item If $m,N\in\N^+$ with $m\ge 2$ and $(q,u,\eta)\in{\X^s_{m,N}}$ satisfy $\overset{w_0,\gam_0}{A_{m,N}}(q,u,\eta)=(g,f,k)$, then $(\pd_jq,\pd_ju,\pd_j\eta) \in \X^{s-1}_{m,N}$ and obeys the estimate
			\begin{multline}
				\snorm{\overset{w_0,\gam_0}{A_{m,N}}(\pd_jq,\pd_ju,\pd_j\eta)}_{\Y^{s-1}}\lesssim\rho\tnorm{q,u,\eta}_{\X_s}\\+\tnorm{g,f,k}_{\Y^s}+\begin{cases}
					0&\text{if }s\le\tfloor{n/2},\\
					\tnorm{q_0,u_0,\eta_0}_{\X_{1+s}}\tnorm{q,u,\eta}_{\X_{\tfloor{n/2}}}&\text{if }\tfloor{n/2}< s.
				\end{cases}
			\end{multline}
		\end{enumerate}
		Here the implicit constants depend on $s$, the dimension, the physical parameters, $\rho_{\m{WD}}$, and $I$.
	\end{prop}
	\begin{proof}
		We will only prove the first item, as the second follows from a nearly identical argument.  Employing the identity $\grad\cdot(v_{w_0}\pd_jq) = \pd_j\grad\cdot(v_{w_0}q)-\overset{w_0}{\mathscr{C}^j}(q,0)$ and Lemma~\ref{lem on mapping properties of the C}, it is evident that the differentiated triple $(\pd_jq,\pd_ju,\pd_j\eta)$ belongs to the space $\overset{q_0,u_0,\eta_0}{\X^{s-1}}$.   Applying $\pd_j$ to the equations in $\overset{w_0,\gam_0}{A}(q,u,\eta)=(g,f,k)$ and rearranging provides the identity $
			\overset{w_0,\gam_0}{A}(\pd_jq,\pd_ju,\pd_j\eta)=\sp{\pd_jg-\overset{w_0}{\mathscr{C}^j}(q,\eta),\pd_jf,\pd_jk}$. 	Then \eqref{california} follows by taking the $\Y^{s-1}$ norm of both sides of this identity and applying the estimates from Lemma~\ref{lem on mapping properties of the C}.
	\end{proof}	
	
	To conclude this subsection, we iterate Proposition~\ref{prop on commutators 2} and combine with Proposition~\ref{prop on commutators 1} to obtain low norm estimates on high-order tangential derivatives.
	
	\begin{thm}[Synthesis of tangential derivative analysis]\label{thm on tangential derivative analysis}
		Let $0<\rho\le\rho_{\m{WD}}$, where the latter is as in Theorem~\ref{thm on smooth tameness of the nonlinear operator}, $w_0=(q_0,u_0,\eta_0)$ be as in Lemma~\ref{properties of the principal parts vector field}, $s\in\N$, $\gam_0\in I\Subset\R^+$ for an interval $I$, $(g,f,k)\in\Y^s$, and $\al\in\N^n\setminus\tcb{0}$ be a multiindex such that $\al\cdot e_n=0$ and $|\al|\le 1+s$. The following hold.
		\begin{enumerate}
			\item If $(q,u,\eta)\in\overset{q_0,u_0,\eta_0}{\X^s}$ satisfies $\overset{w_0,\gam_0}{A}(q,u,\eta)=(g,f,k)$, then  $(\pd^\al q,\pd^\al u,\pd^\al q)  \in \overset{w_0,\gam_0}{\X^{-1}}$ and we have the estimate
			\begin{multline}\label{higher order tangential derivatives}
				\snorm{\overset{w_0,\gam_0}{\mathscr{J}}(\pd^\al q,\pd^\al u,\pd^\al\eta)}_{\Y^{-1}}\lesssim\rho\tnorm{q,u,\eta}_{\X_s}\\+\tnorm{g,f,k}_{\Y^s}+\begin{cases}
					0&\text{if }s\le\tfloor{n/2},\\
					\tnorm{q_0,u_0,\eta_0}_{\X_{1+s}}\tnorm{q,u,\eta}_{\X_{\tfloor{n/2}}}&\text{if }\tfloor{n/2}< s.
				\end{cases}
			\end{multline}
			\item If $m,N\in\N^+$ with $m\ge2$ and $(q,u,\eta)\in\X^s_{m,N}$ satisfy $\overset{w_0,\gam_0}{A_{m,N}}(q,u,\eta)=(g,f,k)$, then   $(\pd^\al q,\pd^\al u,\pd^\al\eta)  \in \X^{-1}_{m,N}$ and we have the estimate
			\begin{multline}
				\snorm{\overset{w_0,\gam_0}{\mathscr{J}^1_{m,N}}(\pd^\al q,\pd^\al u,\pd^\al\eta)}_{\Y^{-1}}\lesssim\rho\tnorm{q,u,\eta}_{\X_s}\\+\tnorm{g,f,k}_{\Y^s}+\begin{cases}
					0&\text{if }s\le\tfloor{n/2},\\
					\tnorm{q_0,u_0,\eta_0}_{\X_{1+s}}\tnorm{q,u,\eta}_{\X_{\tfloor{n/2}}}&\text{if }\tfloor{n/2}< s.
				\end{cases}
			\end{multline}
		\end{enumerate}
		The implicit constants depend on $s$, the dimension, the physical parameters, $\rho_{\m{WD}}$, and $I$.
	\end{thm}
	\begin{proof}
		Again we only prove the first item, as the second follows from a nearly identical argument. First we claim that if $(q,u,\eta)\in\overset{q_0,u_0,\eta_0}{\X^s}$ and $\al\in\N^n$ satisfies $1 \le |\al| \le s$ and $\al\cdot e_n=0$, then  $(\pd^\al q,\pd^\al u,\pd^\al\eta)\in\overset{q_0,u_0,\eta_0}{\X^{s-|\al|}}$ and
		\begin{multline}\label{induction function whats your conjunction}
			\snorm{\overset{w_0,\gam_0}{A}(\pd^\al q,\pd^\al u,\pd^\al\eta)}_{\Y^{s-|\al|}}\lesssim\rho\tnorm{q,u,\eta}_{\X_s}\\+\snorm{\overset{w_0,\gam_0}{A}(q,u,\eta)}_{\Y^s}+\begin{cases}
				0&\text{if }s\le\tfloor{n/2},\\
				\tnorm{q_0,u_0,\eta_0}_{\X_{1+s}}\tnorm{q,u,\eta}_{\X_{\tfloor{n/2}}}&\text{if }\tfloor{n/2}<s.
			\end{cases}
		\end{multline}
		We establish this via strong induction on $|\al|$.

		The base case of the claim, $|\al|=1$, was already established in Proposition~\ref{prop on commutators 2}. Suppose now that for a fixed $1\le \upnu \le s-1$ the claim holds for all $\al\in\N^n$ such that $\al\cdot e_n=0$ and $1\le |\al|\le\upnu$. Let $\al\in\N^n$ with $\al\cdot e_n=0$ be such that $|\al|=\upnu+1$. Let $\be,\gam\in\N^n\setminus\tcb{0}$ be such that $\al=\be+\gam$ and $|\gamma|=1$. By repeated applications of the induction hypothesis, it follows that $(\pd^\be q,\pd^\be u,\pd^\be \eta)\in\overset{q_0,u_0,\eta_0}{\X^{s-|\be|}}$ and, in turn, that $(\pd^\al q,\pd^\al u,\pd^\al \eta)\in\overset{q_0,u_0,\eta_0}{\X^{s-|\al|}}$. Moreover, the induction hypothesis also provides the estimates
		\begin{multline}\label{i like to eat fish}
			\snorm{\overset{w_0,\gam_0}{A}(\pd^\al q,\pd^\al u,\pd^\al\eta)}_{\Y^{s-|\al|}}\lesssim\rho\tnorm{\pd^\be q,\pd^\be u,\pd^\be\eta}_{\X_{s-|\be|}}+\snorm{\overset{w_0,\gam_0}{A}(\pd^\be q,\pd^\be u,\pd^\be\eta)}_{\Y^{s-|\be|}}\\+\begin{cases}
				0&\text{if }s-|\be|\le\tfloor{n/2},\\
				\tnorm{q_0,u_0,\eta_0}_{\X_{1+s-|\be|}}\tnorm{\pd^\be q,\pd^\be u,\pd^\be\eta}_{\X_{\tfloor{n/2}}}&\text{if }\tfloor{n/2}<s-|\be|,
			\end{cases}
			\\\lesssim\rho\sp{\tnorm{q,u,\eta}_{\X_s}+\tnorm{\pd^\be q,\pd^\be u,\pd^\be\eta}_{\X_{s-|\be|}}}+\snorm{\overset{w_0,\gam_0}{A}(q,u,\eta)}_{\Y^s}\\+\begin{cases}
				0&\text{if }s\le\tfloor{n/2},\\
				\tnorm{q_0,u_0,\eta_0}_{\X_{1+s}}\tnorm{q,u,\eta}_{\X_{\tfloor{n/2}}}&\text{if }\tfloor{n/2}<s,
			\end{cases}
			\\+\begin{cases}
				0&\text{if }s-|\be|\le\tfloor{n/2},\\
				\tnorm{q_0,u_0,\eta_0}_{\X_{1+s-|\be|}}\tnorm{\pd^\be q,\pd^\be u,\pd^\be\eta}_{\X_{\tfloor{n/2}}}&\text{if }\tfloor{n/2}<s-|\be|.
			\end{cases}
		\end{multline}
		To morph this estimate into the correct form, we note the following two facts. First, the $\partial^\be$-continuity estimates: $\tnorm{\pd^\be q,\pd^\be u,\pd^\be\eta}_{\X_{s-|\be|}}\lesssim\tnorm{q,u,\eta}_{\X_s}$ and $\tnorm{\pd^\be q,\pd^\be u,\pd^\be\eta}_{\X_{\tfloor{n/2}}}\lesssim\tnorm{q,u,\eta}_{\X_{|\be|+\tfloor{n/2}}}$. Second, by the log-convexity of the norm in the $\X$-spaces (see Lemma~\ref{lem on log-convexity of the norms}) and Young's inequality, we have that if $\tfloor{n/2}<s-|\be|$, then
		\begin{multline}\label{fish likes to eat me}
			\tnorm{q_0,u_0,\eta_0}_{\X_{1+s-|\be|}}\tnorm{\pd^\be q,\pd^\be u,\pd^\be\eta}_{\X_{\tfloor{n/2}}}\lesssim\tnorm{q_0,u_0,\eta_0}_{\X_{1+s-|\be|}}\tnorm{q,u,\eta}_{\X_{|\be|+\tfloor{n/2}}}\\
			\lesssim\sp{\tnorm{q_0,u_0,\eta_0}_{\X_{1+\tfloor{n/2}}}\tnorm{q,u,\eta}_{\X_{s}}}^{\f{|\be|}{s-\tfloor{n/2}}}\sp{\tnorm{q_0,u_0,\eta_0}_{\X_{1+s}}\tnorm{q,u,\eta}_{\X_{\tfloor{n/2}}}}^{1-\f{|\be|}{s-\tfloor{n/2}}}
			\\\lesssim\rho\tnorm{q,u,\eta}_{\X_s}+\tnorm{q_0,u_0,\eta_0}_{\X_{1+s}}\tnorm{q,u,\eta}_{\X_{\tfloor{n/2}}}.
		\end{multline}
		Upon combining these facts with~\eqref{i like to eat fish}, we verify that~\eqref{induction function whats your conjunction} holds, which completes the proof of the inductive step and hence the claim.
		
		Now, if $\al\in\N^n$ is such that $\al\cdot e_n$ and $1\le |\al|\le s$, then the sought-after estimate~\eqref{higher order tangential derivatives} is true thanks to the estimate \eqref{induction function whats your conjunction} established in the above claim and the trivial bounds
		\begin{equation}
			\snorm{\overset{w_0,\gam_0}{\mathscr{J}}(q,u,\eta)}_{\Y^{-1}}\lesssim\snorm{\overset{w_0,\gam_0}{A}(q,u,\eta)}_{\Y^{0}}\le\snorm{\overset{w_0,\gam_0}{A}(q,u,\eta)}_{\Y^{s-|\al|}}.
		\end{equation}
		It then only remains to prove \eqref{higher order tangential derivatives} in the case that $|\al|=s+1$.  In this case we may write $\pd^\al=\pd_j\pd^\be$ for some $j\in\tcb{1,\dots,n-1}$  and $|\be|=s$. Applying Proposition~\ref{prop on commutators 1}, followed by~\eqref{induction function whats your conjunction}, we arrive at the estimate
		\begin{multline}
			\snorm{\overset{w_0,\gam_0}{\mathscr{J}}(\pd^\al q,\pd^\al u,\pd^\al\eta)}_{\Y^{-1}}\lesssim\rho\tnorm{\pd^\be q,\pd^\be u,\pd^\be\eta}_{\X_0}+\tnorm{\overset{w_0,\gam_0}{A}(\pd^\be q,\pd^\be u,\pd^\be\eta)}_{\Y^0}\\
			\lesssim\rho\tnorm{q,u,\eta}_{\X_s}+\tnorm{\overset{w_0,\gam_0}{A}(q,u,\eta)}_{\Y^s}
			+\begin{cases}
				0&\text{if }s\le\tfloor{n/2},\\
				\tnorm{q_0,u_0,\eta_0}_{\X_{1+s}}\tnorm{q,u,\eta}_{\X_{\tfloor{n/2}}}&\text{if }\tfloor{n/2}<s,
			\end{cases}
		\end{multline}
		which completes the proof in the case $|\alpha|=s+1$.
	\end{proof}
	
	\subsection{Analysis of normal systems}\label{section on analysis of normal systems}
	
	A useful technique in the study of the dynamic compressible Navier-Stokes system, originally developed by Matsumura and Nishida~\cite{MR713680}, is to take linear combinations of a normal, or vertical, derivative of the continuity equation with certain components of the momentum equation in order to reveal a subtle dissipative structure for the normal derivative of the density. Our goal now is to implement a version of this technique for our traveling wave problem. The result will essentially be a bound on various high norms of a solution in terms of the data and norms of tangential derivatives alone.

	We begin with a computation that motivates the definition of the normal system. Suppose that $\overset{w_0,\gam_0}{A^1}(q,u,\eta)=g$ and $\overset{\gam_0}{A^2}(q,u,\eta)=f$, where the operators $A^i$ are as defined in \eqref{connecticut} and~\eqref{deleware}, and consider the linear combination
	\begin{equation}\label{special linear combination}
		\pd_n\overset{w_0}{A^1}(q,u,\eta)+\f{\varrho}{\gam_0(2(1-1/n)\upmu(\varrho)+\uplambda(\varrho))}\overset{\gam_0}{A^2}(q,u,\eta)\cdot e_n=\pd_ng+\f{\varrho}{\gam_0(2(1-1/n)\upmu(\varrho)+\uplambda(\varrho))}f\cdot e_n.
	\end{equation}
	On the left hand side the $\pd_n^2(u\cdot e_n)$-terms cancel each other out. We may then rearrange the above equation to create a (differentiated) steady transport equation in $q$, i.e.~\eqref{special linear combination} is equivalent to
	\begin{equation}\label{normal system part 1}
		\pd_n\bp{\f{\varrho^2}{\gam_0(2(1-1/n)\upmu(\varrho)+\uplambda(\varrho))}q+\grad\cdot(v_{w_0}q)}=\overset{w_0,\gam_0}{\mathscr{N}^0}(q,u,\eta,g,f),
	\end{equation}
	where
	\begin{multline}\label{N0_map_defs}
		\overset{w_0,\gam_0}{\mathscr{N}^0}(q,u,\eta,g,f)=\bp{\f{\varrho^2}{\gam_0(2(1-1/n)\upmu(\varrho)+\uplambda(\varrho))}}'q+\pd_ng\\-\pd_n(\varrho'u\cdot e_n)-\varrho'\pd_nu\cdot e_n-(\grad_{\|},0)\cdot\pd_n(\varrho u)-\mathfrak{g}\pd_n\grad\cdot(v_{w_0}\eta)\\
		+\f{\varrho}{\gam_0(2(1-1/n)\upmu(\varrho)+\uplambda(\varrho))}\big(f\cdot e_n+\gam_0^2\pd_1u\cdot e_n+\gam_0\upmu(\varrho)\Delta_{\|}u\cdot e_n\\+\gam_0(\upmu(\varrho)(1-2/n)+\uplambda(\varrho))(\grad_{\|},0)\cdot\pd_nu+\gam_0\varrho'(\upmu'(\varrho)\mathbb{D}^0ue_n\cdot e_n+\gam_0\uplambda'(\varrho)\grad\cdot u)\big).
	\end{multline}
	The key point  is that~\eqref{N0_map_defs} depends only on tangential derivatives of $q,u,\pd_nu,\eta$ along with $\pd_ng$ and $f$.
	
	If we perform the same manipulations under the regularized hypotheses $N^{-1}L_m(q+\mathfrak{g}\eta)+\overset{w_0}{A^1}(q,u,\eta)=g$ and $\overset{\gam_0}{A^2}(q,u,\eta)=f$, where we recall that $L_m$ is defined in~\eqref{the operator Lm}, then we obtain the equation
	\begin{equation}\label{normal system part 1, regularized}
		\pd_n\bp{\f{\varrho^2}{\gam_0(2(1-1/n)\upmu(\varrho)+\uplambda(\varrho))}q+\grad\cdot(v_{w_0}q)+\f{1}{N}L_mq}=\overset{w_0,\gam_0}{\mathscr{N}^0}(q,u,\eta,g,f).
	\end{equation}
	
	The identities \eqref{normal system part 1} and~\eqref{normal system part 1, regularized} are only half of the normal system in that they only allow us to gain control of the normal derivative of $q$ in terms of lower order and tangential derivatives.  Next, we see how to obtain similar control of the normal derivatives $u$.  For this we only need to examine the equation $\overset{\gam_0}{A^2}(q,u,\eta)=f$. If $j\in\tcb{1,\dots,n-1}$, then we take the $j^{\m{th}}$-component of this equation and isolate the $\pd_n^2(u\cdot e_j)$ term: $\pd_n^2(u\cdot e_j)=\overset{\gam_0}{\mathscr{N}^j}(q,u,\eta,f)$, where
	\begin{multline}\label{Nj_map_defs}
		\overset{\gam_0}{\mathscr{N}^j}(q,u,\eta,f)=\f{1}{\gam_0\upmu(\varrho)}\big(-\gam_0^2\varrho\pd_1u\cdot e_j+\varrho\pd_j(q+\mathfrak{g}\eta)-\gam_0\upmu(\varrho)\Delta_{\|}u\cdot e_j \\ -\gam_0(\upmu(\varrho)(1-2/n)+\uplambda(\varrho))\pd_j\grad\cdot u
		-\gam_0\mathbb{D}^0ue_n\cdot e_j(\upmu(\varrho))'-f\cdot e_j\big).
	\end{multline}
	On the other hand, if we take the $n^{\m{th}}$-component of $\overset{\gam_0}{A^2}(q,u,\eta)=f$ and isolate the $\pd_n^2( u\cdot e_n)$ contribution, then we get the equation $\pd_n^2(u\cdot e_n)=\overset{\gam_0}{\mathscr{N}^n}(q,u,\eta,f)$, where
	\begin{multline}\label{Nn_map_defs}
		\overset{\gam_0}{\mathscr{N}^n}(q,u,\eta,f)=\f{1}{\gam_0(2(1-1/n)\upmu(\varrho)+\uplambda(\varrho))}\big(-\gam_0^2\varrho\pd_1u\cdot e_n+\varrho\pd_nq-\gam_0\upmu(\varrho)\Delta_{\|}u\cdot e_n\\-\gam_0(\upmu(\varrho)(1-2/n)+\uplambda(\varrho))(\grad_{\|},0)\cdot\pd_nu-\gam_0\varrho'(\upmu'(\varrho)\mathbb{D}^0ue_n\cdot e_n+\uplambda'(\varrho)\grad\cdot u)-f\cdot e_n\big).
	\end{multline}
	
	We record the output of these calculations in the following lemma.
	
	\begin{lem}[Existence of the normal systems]\label{lemma on the existence of the normal systems}
		Let $0<\rho\le\rho_{\m{WD}}$, where the latter is defined in Theorem~\ref{thm on smooth tameness of the nonlinear operator}, $w_0=(q_0,u_0,\eta_0)$ be as in Lemma~\ref{properties of the principal parts vector field}, $\gam_0\in I$ for $I\Subset\R+$ an interval, and $(g,f,k)\in\Y^0$. The following hold upon setting $\Lambda_{\gam_0}(\varrho)=\gam_0^{-1}\varrho^2(2(1-1/n)\upmu(\varrho)+\uplambda(\varrho))^{-1}$.
		\begin{enumerate}
			\item If $(q,u,\eta)\in\overset{q_0,u_0,\eta_0}{\X^{0}}$ satisfies $\overset{w_0,\gam_0}{A}(q,u,\eta)=(g,f,k)$, then
			\begin{equation}\label{normal system zeroth component identity}
				\pd_n\tp{\Lambda_{\gam_0}(\varrho)q+\grad\cdot(v_{w_0}q)}=\overset{w_0,\gam_0}{\mathscr{N}^0}(q,u,\eta,f)
			\end{equation}
			and
			\begin{equation}\label{country roads take me home}
				\pd_n^2u=\tp{\overset{\gam_0}{\mathscr{N}^1},\dots,\overset{\gam_0}{\mathscr{N}^n}}(q,u,\eta,f).
			\end{equation}
			\item If $m,N\in\N^+$, with $m\ge 2$, and $(q,u,\eta)\in\X^0_{m,N}$ satisfy the equations $\overset{w_0,\gam_0}{A_{m,N}}(q,u,\eta)=(g,f,k)$, then
			\begin{equation}
				\pd_n\tp{\Lambda_{\gam_0}(\varrho)q+\grad\cdot(v_{w_0}q)+N^{-1}L_mq}=\overset{w_0,\gam_0}{\mathscr{N}^0}(q,u,\eta,f),
			\end{equation}
			and~\eqref{country roads take me home} holds.
		\end{enumerate}
	\end{lem}

	We next study the mapping properties of the $\mathscr{N}$ mappings, starting with $\mathscr{N}^0$.
	
	\begin{lem}[Boundedness of $\mathscr{N}^0$]\label{prop on boundedness of N^0}
		Let $0<\rho\le\rho_{\m{WD}}$, where the latter is defined in Theorem~\ref{thm on smooth tameness of the nonlinear operator}, $w_0=(q_0,u_0,\eta_0)$ be as in Lemma~\ref{properties of the principal parts vector field}, $\gam_0\in I$ for $I\Subset\R^+$ an interval, and $s\in\N$. The linear map $\overset{w_0,\gam_0}{\mathscr{N}^0}:\X_s\times H^{1+s}(\Omega)\times H^s(\Omega;\R^n)\to H^s(\Omega)$ from \eqref{N0_map_defs} is well-defined, bounded, and obeys the following estimate.
		\begin{multline}
			\snorm{\overset{w_0,\gam_0}{\mathscr{N}^0}(q,u,\eta,g,f)}_{H^s}\lesssim\sum_{\sig=0}^1\sum_{j=1}^{n-1}\tnorm{\pd_j^\sig q,\pd_j^\sig u,\pd_j^\sig\eta}_{\X_{s-1}}\\+\tnorm{g,f}_{H^{1+s}\times H^s}+\begin{cases}
				0&\text{if }s<\tfloor{n/2},\\
				\tbr{\tnorm{q_0,u_0,\eta_0}_{\X_{1+s}}}\tnorm{\eta}_{\mathcal{H}^{1+\tfloor{n/2}}}&\text{if }\tfloor{n/2}\le s.
			\end{cases}
		\end{multline}
		Here the implicit constant depends on the dimension, the physical parameters, $\rho_{\m{WD}}$, and $I$.
	\end{lem}
	\begin{proof}
		We decompose
		\begin{equation}
			\overset{w_0,\gam_0}{\mathscr{N}^0}(q,u,\eta,g,f)=\overset{0,\gam_0}{\mathscr{N}^0}(q,u,\eta,g,f)-\mathfrak{g}\pd_n\grad\cdot((v_{w_0}-\varrho'e_1/\mathfrak{g})\eta).
		\end{equation}
		By inspection, it is clear that for any $s\ge0$ we have the estimate
		\begin{equation}\label{zigmund_das_schwein_2}
			\snorm{\overset{0,\gam_0}{\mathscr{N}^0}(q,u,\eta,g,f)}_{H^s}\lesssim\sum_{\sig=0}^1\sum_{j=1}^{n-1}\tnorm{\pd_j^\sig q,\pd_j^\sig u,\pd_j^\sig \eta}_{\X_{s-1}}+\tnorm{g,f}_{H^{1+s}\times H^s}.
		\end{equation}
		For the remaining piece we appeal to the fourth item of Lemma~\ref{properties of the principal parts vector field} to estimate
		\begin{multline}
			\tnorm{\pd_n\grad\cdot((v_{w_0}-\varrho'e_1/\mathfrak{g})\eta)}_{H^s}\le\tnorm{\grad\cdot(v_{w_0}\eta)}_{H^{1+s}}+\tnorm{\pd_1\eta}_{H^s}\\
			\lesssim\tnorm{\eta}_{\mathcal{H}^{2+s}}+\begin{cases}
				0&\text{if }s<\tfloor{n/2},\\
				\tbr{\tnorm{q_0,u_0,\eta_0}_{\X_{1+s}}}\tnorm{\eta}_{\mathcal{H}^{1+\tfloor{n/2}}}&\text{if }\tfloor{n/2}\le s.
			\end{cases}
		\end{multline}
		We then conclude by noting that
		\begin{equation}\label{SQG is also an equation did yua know?}
			\tnorm{\eta}_{\mathcal{H}^{5/2+s}}\lesssim\sum_{\sig=0}^1\sum_{j=1}^{n-1}\tnorm{\pd_j^\sig\eta}_{\mathcal{H}^{3/2+s}}.
		\end{equation}
	\end{proof}
	
	Next, we study the $\mathscr{N}^j$ maps for $1 \le j \le n$, as defined by \eqref{Nj_map_defs} and \eqref{Nn_map_defs}.
	
	\begin{lem}[Boundedness of $\mathscr{N}^j$, for $1\le j\le n$]\label{boundedness of the remaining N^j maps}
		Suppose that  $\gam_0\in I$ for $I \Subset\R^+$ an interval, $s\in\N$, and $i\in\tcb{1,\dots,n}$.  Then the linear map $\overset{\gam_0}{\mathscr{N}^i}:\X_s\times H^s(\Omega;\R^n)\to H^s(\Omega)$ from \eqref{Nj_map_defs} when $j < n$ or \eqref{Nn_map_defs} when $j=n$ is well-defined, bounded, and satisfies the following estimates.  If $1\le i\le n-1$, then
		\begin{equation}
			\tnorm{\overset{\gam_0}{\mathscr{N}^i}(q,u,\eta,f)}_{H^s}\lesssim\sum_{\sig=0}^1\sum_{j=1}^{n-1}\tnorm{\pd_j^\sig q,\pd_j^\sig u,\pd_j^\sig\eta}_{\X_{s-1}}+\tnorm{f}_{H^s},
		\end{equation}
		and if $j=n$, then
		\begin{equation}
			\tnorm{\overset{\gam_0}{\mathscr{N}^n}(q,u,\eta,f)}_{H^s}\lesssim\sum_{\sig=0}^{1}\sum_{j=1}^{n-1}\tnorm{\pd_j^\sig q,\pd_j^\sig u,\pd_j^\sig\eta}_{\X_{s-1}}+\tnorm{\pd_nq}_{H^s}+\tnorm{f}_{H^s}.
		\end{equation}
	\end{lem}
	\begin{proof}
		These estimates are clear by inspection.
	\end{proof}
	
	The next result is an application of some of the analysis from Section~\ref{Section: Analysis of Regularized Steady Transport Equations} to the specific steady transport structure appearing here in the normal system.
	
	\begin{lem}[Steady transport estimate]\label{lemma on steady transport estimate}
		Let $0<\rho\le\rho_{\m{WD}}$, where the latter is defined in Theorem~\ref{thm on smooth tameness of the nonlinear operator}, $w_0=(q_0,u_0,\eta_0)$ be as in Lemma~\ref{properties of the principal parts vector field}, $\gam_0\in I$ for $I\Subset\R^+$ an interval, and $\upnu\in\N$. Suppose that $\varphi,\psi\in H^\upnu(\Omega)$ satisfy $\grad\cdot(v_{w_0}\varphi)\in H^\upnu(\Omega)$ and
		\begin{equation}\label{steady freddy}
			\Lambda_{\gam_0}(\varrho)\varphi+\grad\cdot(v_{w_0}\varphi)=\psi\quad\text{in }\Omega,
		\end{equation}
		where $\Lambda_{\gam_0}(\varrho)$ is defined as in Lemma~\ref{lemma on the existence of the normal systems}. There exists a $\rho_{\m{ST},\upnu}\in\R^+$, depending only on the physical parameters, $\upnu$, and the dimension, and $I$, such that if $\rho\le\rho_{\m{ST},\upnu}$ then we have the estimate
		\begin{equation}\label{everybody had a hard year}
			\tnorm{\varphi,\grad\cdot(v_{w_0}\varphi)}_{H^\upnu\times H^{\upnu}}\lesssim\tnorm{\psi}_{H^{\upnu}}+\begin{cases}
				0&\text{if }\upnu\le 1+\tfloor{n/2},\\
				\tnorm{q_0,u_0,\eta_0}_{\X_{\upnu}}\tnorm{\varphi}_{H^{1+\tfloor{n/2}}}&\text{if }1+\tfloor{n/2}<\upnu.
			\end{cases}
		\end{equation}
		The implicit constant depends on $\upnu$, the dimension, the physical parameters, $\rho_{\m{ST},\upnu}$, and $I$.
	\end{lem}
	\begin{proof}
		Most of the work was already carried out in Section~\ref{Section: Analysis of Regularized Steady Transport Equations} in the sense that we endeavor to apply Proposition~\ref{proposition on a priori estimates for steady transport} with the decomposed vector field $X = \f{\mathfrak{g}}{\varrho'}v_{w_0} = X_0 + X_1$, where $X_0 = \f{\mathfrak{g}}{\varrho'}v^{(1)}_{q_0,u_0,\eta_0}$ and $X_1 = e_1+\f{\mathfrak{g}}{\varrho'}v^{(2)}_{\eta_0}$,
		which is split according to the decomposition of $v_{w_0}$ from Lemma~\ref{properties of the principal parts vector field}.  Write $\Uplambda=\mathfrak{g}\Lambda_{\gam_0}(\varrho)/\gam\varrho'$. By hypothesis,  $\widetilde{\varphi}=\varrho'\varphi/\mathfrak{g}\in H^\upnu(\Omega)$ satisfies
		\begin{equation}
			\Uplambda\widetilde{\varphi}+\grad\cdot(X\widetilde{\varphi})=\psi \text{ in }\Omega 
			\text{ and } 
			\grad\cdot(X\widetilde{\varphi})=\grad\cdot(v_{w_0}\varphi)\in H^\upnu(\Omega).
		\end{equation}
		Thus, we may employ Proposition~\ref{proposition on a priori estimates for steady transport} to see that there exists $\rho^{(\upnu)}\in\R^+$ (depending only $\upnu$, the physical parameters, and $I$) such that if $\tnorm{DX_0,DX_1}_{H^{1+\lfloor n/2\rfloor}\times W^{1+\lfloor n/2\rfloor,\infty}}\le \rho^{(\upnu)}$, then we have the estimate
		\begin{equation}
			\tnorm{\widetilde{\varphi}}_{H^\upnu}\lesssim\tnorm{\psi}_{H^\upnu}+\begin{cases}
				0&\text{if }\upnu\le 1+\lfloor n/2\rfloor,\\
				\tnorm{DX_0,DX_1}_{H^\upnu\times W^{\upnu,\infty}}\tnorm{\widetilde{\varphi}}_{H^{1+\lfloor n/2\rfloor}}&\text{if }1+\lfloor n/2\rfloor<\upnu.
			\end{cases}
		\end{equation}
		Next we note that, with implicit constants depending only on the physical parameters and $s\in\tcb{\upnu,1+\lfloor n/2\rfloor}$, we have the estimates
		\begin{equation}
			\tnorm{\varphi}_{H^s}\asymp\tnorm{\tilde{\varphi} }_{H^s}
			\text{ and }
			\tnorm{DX_0,DX_1}_{H^s\times W^{s,\infty}}\lesssim\tnorm{v^{(1)}_{q_0,u_0,\eta_0},v^{(2)}_{\eta_0}}_{H^{1+s}\times W^{1+s,\infty}}.
		\end{equation}
		By invoking the first and second items of Lemma~\ref{properties of the principal parts vector field}, we also see that
		\begin{equation}
			\tnorm{v^{(1)}_{q_0,u_0,\eta_0},v^{(2)}_{\eta_0}}_{H^{1+s}\times W^{1+s,\infty}}\lesssim\tnorm{q_0,u_0,\eta_0}_{\X_s}.
		\end{equation}
		The claimed bound on $\varphi$ in the $H^\upnu(\Omega)$ norm now follows by combining these bounds and taking $\rho_{\m{ST},\upnu}$ small enough. 
		
		It remains to establish a bound on $\grad\cdot(v_{w_0}\varphi)$ in the same space. The equation~\eqref{steady freddy} is equivalent to $\grad\cdot(v_{w_0}\varphi)=\psi-\Uplambda(\rho)\varphi$, so by taking the norm in $H^\upnu(\Omega)$ and utilizing the established bounds on $\tnorm{\varphi}_{H^\upnu}$, we complete the proof of~\eqref{everybody had a hard year}.
	\end{proof}

	We are now ready to identify a recursive estimate for the norm of a solution to the principal part equations. The previous normal system identification and boundedness results merge in this next proposition and allow us to control the solution's norm in terms of the data and a lower norm of the tangentially differentiated solution.

	\begin{prop}[Synthesis of normal system results, 1]\label{first normal system synthesis}
		Let $0<\rho\le\rho_{\m{WD}}$, where the latter is from Theorem~\ref{thm on smooth tameness of the nonlinear operator}, $w_0=(q_0,u_0,\eta_0)$ be as in Lemma~\ref{properties of the principal parts vector field}, $\gam_0\in I$ for $I\Subset\R^+$ an interval, and $\upnu\in\N$.  Suppose that $(q,u,\eta)\in\overset{q_0,u_0,\eta_0}{\X^\upnu}$ and $(g,f,k)\in\Y^\upnu$ satisfy $\overset{w_0,\gam_0}{A}(q,u,\eta)=(g,f,k)$. There exists a $\rho_{\m{normal},\upnu}\in\R^+$
		, depending only on the physical parameters, the dimension, $\upnu$, and $I$, such that that if $\rho\le\rho_{\m{normal},\upnu}$ then we have the estimate 
		\begin{multline}\label{normal system estimate no 1}
			\tnorm{q,u,\eta}_{\overset{q_0,u_0,\eta_0}{\X^\upnu}}\lesssim\sum_{\sig=0}^1\sum_{j=1}^{n-1}\tnorm{\pd_j^\sig q,\pd_j^\sig u,\pd_j^\sig\eta}_{\overset{q_0,u_0,\eta_0}{\X^{\upnu-1}}}\\+\tnorm{g,f,k}_{\Y^{\upnu}}+\begin{cases}
				0&\text{if }\upnu\le\tfloor{n/2},\\
				\tbr{\tnorm{q_0,u_0,\eta_0}_{\X_{1+\upnu}}}\tnorm{q,u,\eta}_{\X_{\tfloor{n/2}}}&\text{if }\tfloor{n/2}<\upnu.
			\end{cases}
		\end{multline}
		The implicit constant depends on $\upnu$, $\rho_{\m{normal},\upnu}$, the dimension, the physical parameters, and $I$.
	\end{prop}
	\begin{proof}
		Suppose first that $\rho\le\rho_{\m{ST},1+\upnu}$, where this parameter is from Lemma~\ref{lemma on steady transport estimate}. We begin by proving the stated bounds on $q$. Denote $\psi=\Lambda_{\gam_0}(\varrho)q+\grad\cdot(v_{w_0}q)$, where $\Lambda_{\gam_0}(\varrho)$ is from Lemma~\ref{lemma on the existence of the normal systems}. According to the steady transport estimate, Lemma~\ref{lemma on steady transport estimate}, we have that
		\begin{equation}\label{combine -1}
			\tnorm{q,\grad\cdot(v_{w_0}q)}_{H^{1+\upnu}\times H^{1+\upnu}}\lesssim\begin{cases}
				\tnorm{\psi}_{H^{1+\upnu}}&\text{if }\upnu\le \lfloor n/2\rfloor,\\
				\tnorm{\psi}_{H^{1+\upnu}}+\tbr{\tnorm{q_0,u_0,\eta_0}_{\X^{1+\upnu}}}\tnorm{q}_{H^{1+\lfloor n/2\rfloor}}&\text{if }\lfloor n/2\rfloor<\upnu.
			\end{cases}
		\end{equation}
		To go from this to the desired bounds on $q$ we will carefully estimate $\norm{\psi}_{H^{1+\nu}}$ by splitting into three pieces:
		\begin{equation}\label{here, there, and everywhere}
			\tnorm{\psi}_{H^{1+\upnu}}\lesssim \tnorm{\psi}_{L^2}+\tnorm{\grad\psi}_{H^{\upnu}}
			\lesssim \tnorm{\psi}_{L^2} + \sum_{j=1}^{n-1}\tnorm{\pd_j\psi}_{H^\upnu}
			+\tnorm{\pd_n\psi}_{H^{\upnu}}.
		\end{equation}
		
		For the $L^2$-norm in \eqref{here, there, and everywhere} we bound via the definition of the norm on $\overset{q_0,u_0,\eta_0}{\X^{-1}}$:
		\begin{equation}\label{combine 1}
			\tnorm{\psi}_{L^2}\lesssim\tnorm{q}_{L^2}+\tnorm{\grad\cdot(v_{w_0}q)}_{L^2}\lesssim\tnorm{q,u,\eta}_{\overset{q_0,u_0,\eta_0}{\X^{-1}}}.
		\end{equation}
		Next, we consider the tangential ($1\le j\le n-1$) derivative terms in  \eqref{here, there, and everywhere} by employing the commutator operator from Definition \ref{commutator_def} to write
		\begin{equation}\label{exhibit uno}
			\pd_j\psi=\Lambda_{\gam_0}(\varrho)\pd_jq+\grad\cdot(v_{w_0}\pd_jq)+\overset{w_0}{\mathscr{C}^j}(q,0).
		\end{equation}
		The first and second terms are readily dealt with using the definition of the space $\overset{q_0,u_0,\eta_0}{\X^{\upnu-1}}$: 
		\begin{equation}\label{zigmund_das_schwein_3}
			\tnorm{\Lambda_{\gam_0}(\varrho)\pd_jq+\grad\cdot(v_{w_0}\pd_jq)}_{H^{\nu}} \lesssim \tnorm{\pd_jq,\pd_ju,\pd_j\eta}_{\overset{q_0,u_0,\eta_0}{\X^{\upnu-1}}}.
		\end{equation}
		We next use Lemma~\ref{lem on mapping properties of the C} on the $\overset{w_0}{\mathscr{C}^{j}}(q,0)$-term in~\eqref{exhibit uno} and then combine with estimate~\eqref{zigmund_das_schwein_3} to see that
		\begin{equation}\label{combine 2}
			\tnorm{\pd_j\psi}_{H^{\upnu}}\lesssim \tnorm{\pd_jq,\pd_ju,\pd_j\eta}_{\overset{q_0,u_0,\eta_0}{\X^{\upnu-1}}}
			+ \rho\tnorm{q}_{H^{1+\upnu}}+\begin{cases}
				0&\text{if }\upnu\le \lfloor n/2\rfloor,\\
				\tbr{\tnorm{q_0,u_0,\eta_0}_{\X_{1+\upnu}}}\tnorm{q}_{H^{1+\lfloor n/2\rfloor}}&\text{if } \lfloor n/2\rfloor<\upnu,
			\end{cases}
		\end{equation}
		for all $1 \le j \le n-1$.
		
		In order to handle the normal ($j=n$) derivative in \eqref{here, there, and everywhere}, we use the existence of the normal system, Lemma~\ref{lemma on the existence of the normal systems}, which shows that $\pd_n\psi=\overset{w_0,\gam_0}{\mathscr{N}^0}(q,u,\eta,g,f)$. Then we use the boundedness of the linear map $\mathscr{N}^0$, Proposition~\ref{prop on boundedness of N^0}, to bound
		\begin{multline}\label{combine 3}
			\tnorm{\pd_n\psi}_{H^\upnu}\lesssim\sum_{\sig=0}^1\sum_{j=1}^n\tnorm{\pd_j^\sig q,\pd_j^\sig u,\pd_j^\sig\eta}_{\X_{\upnu-1}}\\
			+\tnorm{g,f}_{H^{1+\upnu}\times H^\upnu}+\begin{cases}
				0&\text{if }\upnu\le\tfloor{n/2},\\
				\tbr{\tnorm{q_0,u_0,\eta_0}_{\X_{1+\upnu}}}\tnorm{\eta}_{\mathcal{H}^{1+\tfloor{n/2}}}&\text{if }\tfloor{n/2}<\upnu.
			\end{cases}
		\end{multline}
		We have now handled all of the terms on the right side of \eqref{here, there, and everywhere};  by combining \eqref{combine -1}, \eqref{here, there, and everywhere}, \eqref{combine 1}, \eqref{zigmund_das_schwein_3}, \eqref{combine 2}, and \eqref{combine 3}, we deduce the  estimate
		\begin{multline} \label{who what when where why}
			\tnorm{q,\grad\cdot(v_{w_0}q)}_{H^{1+\upnu}\times H^{1+\upnu}}\lesssim\rho\tnorm{q}_{H^{1+\upnu}}\\+\sum_{\sig=0}^1\sum_{j=1}^{n-1}\tnorm{\pd_jq,\pd_ju,\pd_j\eta}_{\overset{q_0,u_0,\eta_0}{\X^\upnu-1}}
			+\tnorm{g,f,k}_{\Y^\upnu}+\begin{cases}
				0&\text{if }\upnu\le\tfloor{n/2},\\
				\tbr{\tnorm{q_0,u_0,\eta_0}_{\X_{1+\upnu}}}\tnorm{q,u,\eta}_{\X_{\tfloor{n/2}}}&\text{if }\tfloor{n/2}<\upnu.
			\end{cases}
		\end{multline}
		
		Next, we turn our attention to the estimate of $u$ in~\eqref{normal system estimate no 1}. Note that 
		\begin{equation}\label{throught the hole and out the door}
			\tnorm{u}_{H^{2+\upnu}}\lesssim\tnorm{u}_{L^2}+\sum_{j=1}^{n-1}\tnorm{\pd_ju}_{H^{1+\upnu}}+\tnorm{\pd_n u}_{H^{1+\upnu}}.
		\end{equation}
		The $L^2$-norm and the sum of tangential ($1\le j\le n-1$) derivatives are trivially controlled by the right hand side of~\eqref{normal system estimate no 1}.
		For the normal $(j=n)$ derivative, we split as before:
		\begin{equation}\label{clear blue island sky}
			\tnorm{\pd_nu}_{H^{1+\upnu}}\le\tnorm{\pd_nu}_{L^2}+\sum_{j=1}^{n-1}\tnorm{\pd_j\pd_nu}_{H^\upnu}+\tnorm{\pd_n^2u}_{H^{\upnu}}.
		\end{equation}
		The first two terms in~\eqref{clear blue island sky} are again trivially bounded by the right hand side of~\eqref{normal system estimate no 1}. According to Lemma~\ref{lemma on the existence of the normal systems}, $\pd_n^2 u$ satisfies~\eqref{country roads take me home}, so we may employ the boundedness of the linear maps $\mathscr{N}^1,\dotsc,\mathscr{N}^n$, proved in Lemma~\ref{boundedness of the remaining N^j maps}, to estimate
		\begin{equation}\label{what a wonderful world}
			\tnorm{\pd_n^2u}_{H^{\upnu}}\lesssim\sum_{\sig=0}^1\sum_{j=1}^{n-1}\tnorm{\pd_j^\sig q,\pd_j^\sig u,\pd_j^\sig\eta}_{\X_{\upnu-1}}+\tnorm{\pd_nq}_{H^\upnu}+\tnorm{f}_{H^\upnu}.
		\end{equation}
		For the $\tnorm{\pd_nq}_{H^\upnu}$ term on the right we insert the already established bounds for $\tnorm{q}_{H^{1+\upnu}}$ from \eqref{who what when where why}. We then synthesize these bounds to deduce the sought-after estimate of $u$ in~\eqref{normal system estimate no 1}. In remains only to handle $\eta$.  However, the estimate for $\eta$ in \eqref{normal system estimate no 1} follows directly from \eqref{SQG is also an equation did yua know?}.
	\end{proof}
	
	\begin{rmk}\label{nonincreasing remark 1}
		We are free to choose the function $\N\ni\upnu\mapsto\rho_{\m{normal},\upnu}\in\R^+$ to be nonincreasing without altering the statement or conclusion of Proposition~\ref{first normal system synthesis}.
	\end{rmk}
	
	Next, we consider the normal system corresponding to the regularized principal part operator $\overset{w_0,\gam_0}{A_{m,N}}$. We require the following preliminary result, which is analogous to Lemma~\ref{lemma on steady transport estimate} from the unregularized case.

	\begin{lem}[Regularized steady transport estimate]\label{lemma on regularized steady transport estimate}
		Let $0<\rho\le\rho_{\m{WD}}$, where the latter is defined in Theorem~\ref{thm on smooth tameness of the nonlinear operator}, $w_0=(q_0,u_0,\eta_0)\in B_{\X^{2+\tfloor{n/2}}}(0,\rho)\cap\X^\infty$, $\gam_0\in I$ for $I \Subset\R^+$ an interval, $m,N\in\N^+$, and $\upnu\in\tcb{1,\dots,m}$.  Suppose that $\varphi\in H^{\upnu+2m}(\Omega)$  and $\psi\in H^{\upnu}(\Omega)$ satisfy
		\begin{equation}
			\begin{cases}\Lambda_{\gam_0}(\varrho)\varphi+\grad\cdot(v_{w_0}\varphi)+N^{-1}L_m\varphi=\psi,&\text{in }\Omega,\\
				\pd_n^m\varphi=\cdots=\pd_n^{2m-1}\varphi=0&\text{in }\pd\Omega,
			\end{cases}
		\end{equation}
		where $\Lambda_{\gam_0}(\varrho)$ and $L_m$ are defined in Lemma~\ref{lemma on the existence of the normal systems} and \eqref{the operator Lm}, respectively.  There exists a $\rho_{\m{RST},m}\in\R^+$, depending only on the physical parameters, the dimension, $m$, and $I$, such that if $\rho\le\rho_{\m{RST},m}$ and $N\gtrsim\tbr{\tnorm{q_0,u_0,\eta_0}_{\X_{1+\lfloor n/2\rfloor+m}}}^{4+2\lfloor n/2\rfloor}$, then we have the estimate
		\begin{equation}
			\tnorm{\varphi,\grad\cdot(v_{w_0}\varphi),N^{-1}\varphi}_{H^\upnu\times H^{\upnu}\times H^{\upnu+2m}}\lesssim\tnorm{\psi}_{H^{\upnu}}+\tbr{\tnorm{q_0,u_0,\eta_0}_{\X_{2m}}}^{2+\lfloor n/2\rfloor}\tnorm{\psi}_{L^2},
		\end{equation}
		where the implied constants depend on $m$, the dimension, the physical parameters, $\rho_{\m{RST},m}$, and $I$.
	\end{lem}
	\begin{proof}
		Once more, the goal is to invoke the work from Section~\ref{Section: Analysis of Regularized Steady Transport Equations} by applying Theorem~\ref{theorem on estimates for regularized steady transport}. To this end, we set $\Lambda_0=\Lambda_{\gam_0}(\varrho)$ and $\Lambda_1=\mathfrak{g}^{-1}\varrho'$,
		and we use Lemma~\ref{properties of the principal parts vector field} to define the decomposed vector field $X=X_0+X_1$, where $X_0=\f{\mathfrak{g}}{\varrho'}v^{(1)}_{q_0,u_0,\eta_0}$ and $X_1=e_1+\f{\mathfrak{g}}{\varrho'}v^{(2)}_{\eta_0}$. Thanks to the first and second items of the lemma, we have that 
		\begin{equation}
			\tnorm{X_0,X_1}_{H^{1+s}\times W^{1+s,\infty}}\lesssim\tbr{\tnorm{q_0,u_0,\eta_0}_{\X_{s}}}
			\text{ for }
			\N\ni s\ge 1+\lfloor n/2\rfloor.
		\end{equation}
		By the additional fact that $\grad\Lambda_1=\gam\mathfrak{g}^{-1}\varrho''e_n$, we see that
		\begin{multline}
			\max\tcb{\tnorm{DX_0,DX_1}_{H^{2+\lfloor n/2\rfloor}\times W^{2+\lfloor n/2\rfloor,\infty}},\tnorm{X_0\cdot\grad\Lambda_1}_{H^{1+\lfloor n/2\rfloor}},\tnorm{X_1\cdot\grad\Lambda_1}_{W^{1+\lfloor n/2\rfloor,\infty}}}\\\lesssim\tnorm{q_0,u_0,\eta_0}_{\X_{2+\lfloor n/2\rfloor}}\le\rho.
		\end{multline}
		Thus, we may take $0<\rho_{\m{RST},m}\lesssim\Bar{\rho}^{(m)}$, where the latter is defined in Theorem~\ref{theorem on estimates for regularized steady transport}, in order to reach the asserted conclusion.
	\end{proof}
	
	The next result is the analog of Proposition~\ref{first normal system synthesis} for the $m,N$-regularized principal part operator. There are a few key differences, the most glaring of which is the finite range, depending on $m$, in which the estimate holds. Another more minor distinction is that while the following estimates are no longer tame (due to the appearance of the $\tbr{\cdot}^{2+\tfloor{n/2}}$ term), they still place all of the high norms of the background onto the low norms of the solution.

	\begin{prop}[Synthesis of normal system results, 2]\label{second normal system synthesis}
		Let $0<\rho\le\rho_{\m{WD}}$, where the latter is defined in Theorem~\ref{thm on smooth tameness of the nonlinear operator}, $w_0=(q_0,u_0,\eta_0)\in B_{\X^{2+\tfloor{n/2}}}(0,\rho)\cap\X^\infty$, $\gam_0\in I$ for $I\Subset\R^+$ an interval, $m,N\in\N^+$ with $m\ge2$, and $\upnu\in\tcb{0,\dots,m-1}$.  Suppose that $(q,u,\eta)\in\X^\upnu_{m,N}$ and $(g,f,k)\in\Y^\upnu$ satisfy $\overset{w_0,\gam_0}{A_{m,N}}(q,u,\eta)=(g,f,k)$. There exists a $\rho_{\m{reg,normal},\upnu,m}\in\R^+$, depending only on the physical parameters, the dimension, $\upnu$, $m$, and $I$ such that if $\rho\le\rho_{\m{reg,normal},\upnu,m}$ and
		\begin{equation}\label{big boy}
			N\gtrsim\tbr{\tnorm{q_0,u_0,\eta_0}_{\X_{1+\lfloor n/2\rfloor+m}}}^{4+2\lfloor n/2\rfloor},
		\end{equation}
		then we have the estimate
		\begin{multline}\label{hi'ilawe}
			\tnorm{q,u,\eta}_{\overset{q_0,u_0,\eta_0}{\X^{\upnu}_{m,N}}}\lesssim\sum_{\sig=0}^1\sum_{j=1}^{n-1}\tnorm{\pd_j^\sig q,\pd_j^\sig u,\pd_j^\sig \eta}_{\overset{q_0,u_0,\eta_0}{\X^{\upnu-1}_{m,N}}}+\tnorm{g,f,k}_{\Y^\upnu}\\+\tbr{\tnorm{q_0,u_0,\eta_0}_{\X_{\max\tcb{2m,1+\lfloor n/2\rfloor+m}}}}^{2+\lfloor n/2\rfloor}\tnorm{q,u,\eta}_{\overset{q_0,u_0,\eta_0}{\X^{-1}_{m,N}}}.
		\end{multline}
		The implicit constants depend on the dimension, the physical parameters, $m$, $\rho_{\m{reg,normal},\upnu,m}$, and $I$.
	\end{prop}
	\begin{proof}
		The proof largely mirrors that of Proposition \ref{first normal system synthesis}, but since we have to appeal to estimates for solutions to the regularized steady transport equation, the estimates we get are somewhat different in a few key places.  As such, we work through most of the argument in detail, appealing to the proof of Proposition \ref{first normal system synthesis} when possible.
		
		Suppose first that $\rho\le\rho_{\m{RST},m}$.  We begin with the bounds on $q$ by setting $\psi=\Lambda_{\gam_0}(\varrho)q+\grad\cdot(v_{w_0}q)+N^{-1}L_mq$. According to the regularized steady transport estimate, Lemma~\ref{lemma on regularized steady transport estimate}, we have the bound
		\begin{equation}\label{how are you stranger}
			\tnorm{q,\grad\cdot(v_{w_0}q),N^{-1}q}_{H^{1+\upnu}\times H^{1+\upnu}\times H^{1+\upnu+2m}}\lesssim\tnorm{\psi}_{H^{1+\upnu}}+\tbr{\tnorm{q_0,u_0,\eta_0}_{\X_{2m}}}^{2+\lfloor n/2\rfloor}\tnorm{\psi}_{L^2},
		\end{equation}
		and to parlay this into our desired estimate we split
		\begin{equation}\label{pastoral}
			\tnorm{\psi}_{H^{1+\upnu}}\le\tnorm{\psi}_{L^2}+\sum_{j=1}^{n-1}\tnorm{\pd_j\psi}_{H^\upnu}+\tnorm{\pd_n\psi}_{H^\upnu}.
		\end{equation}
		For the $L^2$-norm we bound as in~\eqref{combine 1}:
		\begin{equation}\label{are you real dah dah dah}
			\tnorm{\psi}_{L^2}\lesssim\tnorm{q,u,\eta}_{\overset{q_0,u_0,\eta_0}{\X_{m,N}^{-1}}}.
		\end{equation}
		We study the tangential $(1\le j\le n-1)$ derivatives in \eqref{pastoral} via the identity
		\begin{equation}
			\pd_j\psi=\Lambda_{\gam_0}(\varrho)\pd_jq+\grad\cdot(v_{w_0}\pd_jq)+N^{-1}L_m\pd_jq+\overset{w_0}{\mathscr{C}^j}(q,0),
		\end{equation}
		where $\overset{w_0}{\mathscr{C}^{j}}$ is as in Definition \ref{commutator_def}, which shows that  
		\begin{equation}\label{all we are saying}
			\tnorm{\pd_j\psi}_{H^{\upnu}}\lesssim\tnorm{\pd_jq,\pd_ju,\pd_j\eta}_{\overset{q_0,u_0,\eta_0}{\X^{\upnu-1}_{m,N}}}+\snorm{\overset{w_0}{\mathscr{C}^{j}}(q,0)}_{H^{\upnu}}.
		\end{equation}
		By the boundedness of $\mathscr{C}^j$ established in Lemma~\ref{lem on mapping properties of the C}, we obtain the estimate
		\begin{equation}
			\tnorm{\overset{w_0}{\mathscr{C}^j}(q,0)}_{H^{\upnu}}\lesssim\rho\tnorm{q}_{H^{1+\upnu}}+\begin{cases}
				0&\text{if }\upnu\le\lfloor n/2\rfloor,\\
				\tnorm{q_0,u_0,\eta_0}_{\X_{1+\upnu}}\tnorm{q}_{H^{1+\lfloor n/2\rfloor}}&\text{if }\lfloor n/2\rfloor<\upnu.
			\end{cases}
		\end{equation}
		In the case that $\lfloor n/2\rfloor<\upnu$, we use interpolation (see Lemma~\ref{lem on log-convexity of the norms}) and Young's inequality to further estimate $\tnorm{q_0,u_0,\eta_0}_{\X_{1+\upnu}}\tnorm{q}_{H^{1+\lfloor n/2\rfloor}}\lesssim\rho\tnorm{q}_{H^{1+\upnu}}+\tnorm{q_0,u_0,\eta_0}_{\X_{2+\lfloor n/2\rfloor+\upnu}}\tnorm{q}_{L^2}$. Hence, for any $\upnu$ we have the bound
		\begin{equation}\label{peace a chance}
			\snorm{\overset{w_0}{\mathscr{C}^j}(q,0)}_{H^\upnu}\lesssim\rho\tnorm{q}_{H^{1+\upnu}}+\tnorm{q_0,u_0,\eta_0}_{\X_{2+\lfloor n/2\rfloor+\upnu}}\tnorm{q}_{L^2}
		\end{equation}
		Upon combining~\eqref{all we are saying} and \eqref{peace a chance} and then summing over $1\le j\le n-1$, we acquire the tangential bound
		\begin{equation}\label{mona lisa}
			\sum_{j=1}^{n-1}\tnorm{\pd_j\psi}_{H^\upnu}\lesssim\rho\tnorm{q}_{H^{1+\upnu}}+\sum_{j=1}^{n-1}\tnorm{\pd_jq,\pd_ju,\pd_j\eta}_{\overset{q_0,u_0,\eta_0}{\X^{\upnu-1}_{m,N}}}+\tnorm{q_0,u_0,\eta_0}_{\X_{2+\lfloor n/2\rfloor+\upnu}}\tnorm{q,u,\eta}_{\X_{-1}}.
		\end{equation}
		
		It remains to estimate the $\tnorm{\pd_n\psi}_{H^{\upnu}}$-term in \eqref{pastoral}. For this, we recall the regularized normal system, Lemma~\ref{lemma on the existence of the normal systems}, which says that $\pd_n\psi=\overset{w_0,\gam_0}{\mathscr{N}^0}(q,u,\eta,g,f)$. Therefore, we may apply the continuity properties of $\mathscr{N}^0$, Lemma~\ref{prop on boundedness of N^0}, to see that
		\begin{multline}
			\tnorm{\pd_n\psi}_{H^\upnu}\lesssim\sum_{\sig=0}^1\sum_{j=1}^{n-1}\tnorm{\pd_j^\sig q,\pd_j^\sig u,\pd_j^\sig\eta}_{\X_{\upnu-1}}+\tnorm{g,f}_{H^{1+\upnu}\times H^\upnu}\\+\begin{cases}
				0&\text{if }\upnu\le\lfloor n/2\rfloor,\\
				\tbr{\tnorm{q_0,u_0,\eta_0}_{\X_{1+\upnu}}}\tnorm{\eta}_{\mathcal{H}^{1+\lfloor n/2\rfloor}}&\text{if }\lfloor n/2\rfloor<\upnu.
			\end{cases}
		\end{multline}
		Again, in the latter case of $\lfloor n/2\rfloor<\upnu$ we use interpolation and Young's inequality to bound
		\begin{equation}
			\tbr{\tnorm{q_0,u_0,\eta_0}_{\X_{1+\upnu}}}\tnorm{\eta}_{\mathcal{H}^{1+\lfloor n/2\rfloor}}\lesssim\tbr{\tnorm{q_0,u_0,\eta_0}_{\X_{\lfloor n/2\rfloor+1/2}}}\tnorm{\eta}_{\mathcal{H}^{3/2+\upnu}}+\tbr{\tnorm{q_0,u_0,\eta_0}_{\X_{\lfloor n/2\rfloor+1/2+\upnu}}}\tnorm{\eta}_{\mathcal{H}^{3/2}},
		\end{equation}
		and hence in all cases
		\begin{equation}\label{twinkle in your eye}
			\tnorm{\pd_n\psi}_{H^\upnu}\lesssim\sum_{\sig=0}^1\sum_{j=1}^{n-1}\tnorm{\pd_j^\sig q,\pd_j^\sig u,\pd_j^\sig\eta}_{\X_{\upnu-1}}+\tnorm{g,f}_{H^{1+\upnu}\times H^\upnu}+\tbr{\tnorm{q_0,u_0,\eta_0}_{\X_{\lfloor n/2\rfloor+1/2+\upnu}}}\tnorm{q,u,\eta}_{\X_{-1}}.
		\end{equation}
		
		Upon piecing together~\eqref{pastoral}, \eqref{are you real dah dah dah}, \eqref{mona lisa}, and~\eqref{twinkle in your eye}, we arrive at the bound
		\begin{multline}\label{the sandpiper returns}
			\tnorm{\psi}_{H^{1+\upnu}}\lesssim\rho\tnorm{q}_{H^{1+\upnu}}+\sum_{\sig=0}^1\sum_{j=1}^{n-1}\tnorm{\pd_j^\sig q,\pd_j^\sig u,\pd_j^\sig \eta}_{\overset{q_0,u_0,\eta_0}{\X^{\upnu-1}_{m,N}}}\\+\tnorm{g,f,k}_{\Y^\upnu}+\tbr{\tnorm{q_0,u_0,\eta_0}_{\X_{2+\lfloor n/2\rfloor+\upnu}}}\tnorm{q,u,\eta}_{\overset{q_0,u_0,\eta_0}{\X^{-1}_{m,N}}}.
		\end{multline}
		Now we insert~\eqref{are you real dah dah dah} and~\eqref{the sandpiper returns} into \eqref{how are you stranger} and choose $\rho_{\m{reg,normal},\upnu,m}\in\R^+$ such that $\rho_{\m{reg,normal},\upnu,m}\le\rho_{\m{RST},m}$ and sufficiently small to allow absorption of the $\tnorm{q}_{H^{1+\upnu}}$ term on the right by the left hand side. This yields
		\begin{multline}\label{we have so far prooved that}
			\tnorm{q,\grad\cdot({v_{w_0}}q),N^{-1}q}_{H^{1+\upnu}\times H^{1+\upnu}\times H^{1+\upnu+2m}}\lesssim\sum_{\sig=0}^1\sum_{j=1}^{n-1}\tnorm{\pd_j^\sig q,\pd_j^\sig u,\pd_j^\sig\eta}_{\overset{q_0,u_0,\eta_0}{\X^{\upnu-1}_{m,N}}}\\+\tnorm{g,f,k}_{\Y^\upnu}+\tbr{\tnorm{q_0,u_0,\eta_0}_{\X_{\max\tcb{2m,1+\lfloor n/2\rfloor+m}}}}^{2+\lfloor n/2\rfloor}\tnorm{q,u,\eta}_{\overset{q_0,u_0,\eta_0}{\X^{-1}_{m,N}}},
		\end{multline}
		and this implies the asserted bound on $q$ in~\eqref{hi'ilawe}.  
		
		Finally, the proof of the estimates for $u$ and $\eta$ is mostly a reprise of the latter part of the proof of Proposition~\ref{first normal system synthesis}.  We get equations~\eqref{throught the hole and out the door}, \eqref{clear blue island sky}, and~\eqref{what a wonderful world} exactly as before.  Then we insert our new bound for $q$ from~\eqref{we have so far prooved that} into the $\tnorm{\pd_nq}_{H^\upnu}$-term.  This yields the $u$ bound of~\eqref{hi'ilawe}. The $\eta$ bound is again trivial.
	\end{proof}
	
	\begin{rmk}\label{nonincreasing remark 2}
		For fixed $m$, we are free to choose the function $\tcb{0,\dots,m-1}\ni\upnu\mapsto\rho_{\m{reg,normal,\upnu,m}}\in\R^+$ to be nonincreasing without altering the statement or conclusion of Proposition~\ref{second normal system synthesis}.
	\end{rmk}
	
	The final result of this subsection iterates the conclusions of Propositions~\ref{first normal system synthesis} and~\ref{second normal system synthesis} to alter the form of the estimate's right hand side to one in which higher order tangential derivatives appear in a low norm.
	
	\begin{thm}[Synthesis of normal system results, 3]\label{third synthesis of normal system results}
		Let $0<\rho\le\rho_{\m{WD}}$, where the latter is defined in Theorem~\ref{thm on smooth tameness of the nonlinear operator}, $r\in\tcb{1+\tfloor{n/2},2+\tfloor{n/2}}$, $w_0=(q_0,u_0,\eta_0)\in B_{\X^r}(0,\rho)\cap\X^\infty$, $\gam_0\in I$ for $I\Subset\R^+$ an interval, $\upnu\in\N$, and $(g,f,k)\in\Y^\upnu$. The following hold.
		\begin{enumerate}
			\item If $r=1+\tfloor{n/2}$, $\rho\le\rho_{\m{normal},\upnu}$, and $(q,u,\eta)\in\overset{q_0,u_0,\eta_0}{\X^\upnu}$ satisfy $\overset{w_0,\gam_0}{A}(q,u,\eta)=(g,f,k)$, then we have the estimate
			\begin{multline}\label{noone around you to carry the blame for you}
				\tnorm{q,u,\eta}_{\overset{q_0,u_0,\eta_0}{\X^{\upnu}}}\lesssim\rho\tnorm{q,u,\eta}_{\X_\upnu}+\sum_{\substack{|\al|\le1+\upnu\\\al\cdot e_n=0}}\tnorm{\pd^\al q,\pd^\al u,\pd^\al\eta}_{\overset{q_0,u_0,\eta_0}{\X^{-1}}}\\+\tnorm{g,f,k}_{\Y^\upnu}+\begin{cases}
					0&\text{if }\upnu\le\tfloor{n/2},\\
					\tbr{\tnorm{q_0,u_0,\eta_0}_{\X_{1+\upnu}}}\tnorm{q,u,\eta}_{\X_{\tfloor{n/2}}}&\text{if }\tfloor{n/2}<\upnu.
				\end{cases}
			\end{multline}
			The implied constant depends on $\upnu$, the dimension, the physical parameters, $\rho_{\m{normal},\upnu}$, and $I$.
			
			\item If $r=2+\tfloor{n/2}$, $\N\ni m\ge\max\tcb{2,1+\upnu}$, $N\in\N$ satisfies~\eqref{big boy}, $\rho\le\rho_{\m{reg,normal},\upnu,m}$, and $(q,u,\eta)\in\X^\upnu_{m,N}$ satisfy $\overset{w_0,\gam_0}{A_{m,N}}(q,u,\eta)=(g,f,k)$, then there exists $\chi_{\upnu,m}\in\R^+$, depending only on $\upnu$, $m$, and the dimension, and $C_{1/\rho}\in\R^+$, depending only on $1/\rho$, $\upnu$, $m$, such that we have the estimate
			\begin{multline}\label{how high can you leap}
				\tnorm{q,u,\eta}_{\overset{q_0,u_0,\eta_0}{\X_{m,N}^\upnu}}\lesssim\rho\tnorm{q,u,\eta}_{\overset{q_0,u_0,\eta_0}{\X^\upnu_{m,N}}}+\sum_{\substack{|\al|\le1+\upnu\\\al\cdot e_n}}\tnorm{\pd^\al q,\pd^\al u,\pd^\al\eta}_{\overset{q_0,u_0,\eta_0}{\X_{m,N}^{-1}}}\\
				+\tnorm{g,f,k}_{\Y^{\upnu}}+C_{1/\rho}\tbr{\tnorm{q_0,u_0,\eta_0}_{\X_{\max\tcb{2m,1+\tfloor{n/2}+m}}}}^{\chi_{\upnu,m}}\tnorm{q,u,\eta}_{\overset{q_0,u_0,\eta_0}{\X^{-1}_{m,N}}}.
			\end{multline}
			Here the implied constant depends on $m$, the dimension, the various physical parameters, $\rho_{\m{reg,normal},\upnu,m}$, and $I$. 
		\end{enumerate}
	\end{thm}
	\begin{proof}
		We begin by proving the first item via an induction argument.  The proposition to be proved inductively is as follows: if $s\in\N^+$ satisfies $s \le \upnu +1$ for some $\upnu \in \N$, $\rho\le\rho_{\m{normal},\upnu}$, and $(q,u,\eta)\in\overset{q_0,u_0,\eta_0}{\X^\upnu}$, then  we have the estimate
		\begin{multline}\label{the inductive hypothesis here and now for all to take no we're not gonna take it kitty bag anymore noone none fun thumb bath tub}
			\tnorm{q,u,\eta}_{\overset{q_0,u_0,\eta_0}{\X^\upnu}}\lesssim\rho\tnorm{q,u,\eta}_{\X_\upnu}+\sum_{\substack{|\al|\le s\\\al\cdot e_n=0}}\tnorm{\pd^\al q,\pd^\al u,\pd^\al\eta}_{\overset{q_0,u_0,\eta_0}{\X^{\upnu-s}}}\\+\snorm{\overset{w_0,\gam_0}{A}(q,u,\eta)}_{\Y^\upnu}+\begin{cases}
				0&\text{if }\upnu\le\tfloor{n/2},\\
				\tbr{\tnorm{q_0,u_0,\eta_0}_{\X_{1+\upnu}}}\tnorm{q,u,\eta}_{\X_{\tfloor{n/2}}}&\text{if }\tfloor{n/2}<\upnu.
			\end{cases}
		\end{multline}
		The base case, $s=1$, was established in Proposition~\ref{first normal system synthesis}. Suppose now that $\N\ni s\ge 2$ is such that the proposition holds at the level $s-1$. We will now verify it at level $s$. Suppose that $s \le \upnu +1$ for some $\upnu \in \N$, $\rho\le\rho_{\m{normal},\upnu}$, and $(q,u,\eta)\in\overset{q_0,u_0,\eta_0}{\X^\upnu}$. We first invoke Proposition~\ref{first normal system synthesis} to obtain the estimate
		\begin{multline}\label{normal system estimate no 1, wah wah}
			\tnorm{q,u,\eta}_{\overset{q_0,u_0,\eta_0}{\X^\upnu}}\lesssim\sum_{\sig=0}^1\sum_{j=1}^{n-1}\tnorm{\pd_j^\sig q,\pd_j^\sig u,\pd_j^\sig\eta}_{\overset{q_0,u_0,\eta_0}{\X^{\upnu-1}}}\\+\snorm{\overset{w_0,\gam_0}{A}(q,u,\eta)}_{\Y^{\upnu}}+\begin{cases}
				0&\text{if }\upnu\le\tfloor{n/2},\\
				\tbr{\tnorm{q_0,u_0,\eta_0}_{\X_{1+\upnu}}}\tnorm{q,u,\eta}_{\X_{\tfloor{n/2}}}&\text{if }\tfloor{n/2}<\upnu.
			\end{cases}
		\end{multline}
		Fix $\sig\in\tcb{0,1}$ and $j\in\tcb{1,\dots,n-1}$. To handle the $\tnorm{\pd^\sig_jq,\pd_j^\sig u,\pd^\sig_j\eta}_{\overset{q_0,u_0,\eta_0}{\X^{\upnu-1}}}$ term appearing in~\eqref{normal system estimate no 1, wah wah}, we invoke the induction hypothesis at level $s-1$ (while heeding to Remark~\ref{nonincreasing remark 1}) in order to learn that
		\begin{multline}\label{no use sitting in another chair}
			\tnorm{\pd^\sig_jq,\pd_j^\sig u,\pd_j^\sig\eta}_{\overset{q_0,u_0,\eta_0}{\X^{\upnu-1}}}\lesssim\rho\tnorm{\pd_j^\sig q,\pd_j^\sig u,\pd^\sig_j\eta}_{\X_{\upnu-1}}+\sum_{\substack{|\al|\le s-1\\\al\cdot e_n=0}}\tnorm{\pd^\al\pd_j^\sig q,\pd^\al\pd^\sig_ju,\pd^\al\pd_j^\sig\eta}_{\overset{q_0,u_0,\eta_0}{\X^{\upnu-s}}}\\+\snorm{\overset{w_0,\gam_0}{A}(\pd_j^\sig q,\pd_j^\sig u,\pd_j^\sig\eta)}_{\Y^{\upnu-1}}+\begin{cases}
				0&\text{if }\upnu\le 1+\tfloor{n/2},\\
				\tbr{\tnorm{q_0,u_0,\eta_0}_{\X_{\upnu}}}\tnorm{\pd_j^\sig q,\pd_j^\sig u,\pd_j^\sig \eta}_{\X_{\tfloor{n/2}}}&\text{if }1+\tfloor{n/2}<\upnu.
			\end{cases}
		\end{multline}
		We estimate the first term on the right hand side trivially via $\tnorm{\pd_j^\sig q,\pd_j^\sig u,\pd_j^\sig\eta}_{\X_{\upnu-1}}\lesssim\tnorm{q,u,\eta}_{\X_\upnu}$. For the term involving the operator $\overset{w_0,\gam_0}{A}$, we invoke the first item of Proposition~\ref{prop on commutators 2} (noting that $\upnu>0$) to bound
		\begin{multline}\label{looking like i dont care}
			\snorm{\overset{w_0,\gam_0}{A}(\pd_j^\sig q,\pd_j^\sig u,\pd_j^\sig\eta)}_{\Y^{\upnu-1}}\lesssim\rho\tnorm{q,u,\eta}_{\X_\upnu}\\+\snorm{\overset{w_0,\gam_0}{A}(q,u,\eta)}_{\Y^\upnu}+\begin{cases}
				0&\text{if }\upnu\le\tfloor{n/2},\\
				\tnorm{q_0,u_0,\eta_0}_{\X_{1+\upnu}}\tnorm{q,u,\eta}_{\X_{\tfloor{n/2}}}&\text{if }\tfloor{n/2}<s.
			\end{cases}
		\end{multline}
		Upon combining inequalities~\eqref{normal system estimate no 1, wah wah}, \eqref{no use sitting in another chair}, and~\eqref{looking like i dont care} with the fact that
		\begin{multline}
			\begin{cases}
				0&\text{if }\upnu\le 1+\tfloor{n/2},\\
				\tbr{\tnorm{q_0,u_0,\eta_0}_{\X_{\upnu}}}\tnorm{\pd_j^\sig q,\pd_j^\sig u,\pd_j^\sig \eta}_{\X_{\tfloor{n/2}}}&\text{if }1+\tfloor{n/2}<\upnu,
			\end{cases} \\
			\lesssim\rho\tnorm{q,u,\eta}_{\X_{\upnu}}
			+\begin{cases}
				0&\text{if }\upnu\le\tfloor{n/2},\\
				\tnorm{q_0,u_0,\eta_0}_{\X_{1+\upnu}}\tnorm{q,u,\eta}_{\X_{\tfloor{n/2}}}&\text{if }\tfloor{n/2}<s,
			\end{cases}
		\end{multline}
		we establish ~\eqref{the inductive hypothesis here and now for all to take no we're not gonna take it kitty bag anymore noone none fun thumb bath tub} for the level $s$, and thus prove the proposition at the $s$ level.  The proposition then holds for all $s\in \N^+$ by induction.  By taking $s=1+\upnu$ in~\eqref{the inductive hypothesis here and now for all to take no we're not gonna take it kitty bag anymore noone none fun thumb bath tub} we obtain \eqref{noone around you to carry the blame for you}, which completes the proof of the first item.
		
		Next, we turn our attention to the second item. The strategy here is the same as in the first item, but due to the dissimilarities between~\eqref{noone around you to carry the blame for you} and~\eqref{how high can you leap}, we cannot apply precisely the same argument. Our new proposition to be proved inductively is as follows: if $s\in\N^+$, then for all $\N\ni\upnu\ge s-1$ and $\N\ni m\ge\max\tcb{2,1+\upnu}$, there exists $\chi\in\R^+$, depending on $s$, $m$, and $\upnu$, and $C_{1/\rho}$, depending on $s$, $1/\rho$, $\upnu$, and $m$, such that for all $N\in\N$ satisfying~\eqref{big boy}, $\rho\le\rho_{\m{reg,normal},\upnu,m}$, and $(q,u,\eta)\in\X^{\upnu}_{m,N}$ we have the estimate
		\begin{multline}\label{the second inductive hypothesis}
			\tnorm{q,u,\eta}_{\overset{q_0,u_0,\eta_0}{\X^\upnu_{m,N}}}\lesssim\rho\tnorm{q,u,\eta}_{\overset{q_0,u_0,\eta_0}{\X^\upnu_{m,N}}}+\sum_{\substack{|\al|\le s\\\al\cdot e_n=0}}\tnorm{\pd^\al q,\pd^\al u,\pd^\al\eta}_{\overset{q_0,u_0,\eta_0}{\X^{\upnu-s}_{m,N}}}\\+\snorm{\overset{w_0,\gam_0}{A_{m,N}}(q,u,\eta)}_{\Y^\upnu}+C_{1/\rho}\tbr{\tnorm{q_0,u_0,\eta_0}_{\X_{\max\tcb{2m,1+\tfloor{n/2}+m}}}}^{\chi}\tnorm{q,u,\eta}_{\overset{q_0,u_0,\eta_0}{\X^{-1}_{m,N}}}.
		\end{multline}
		As before, the base case, $s=1$, was established already, this time in Proposition~\ref{second normal system synthesis}. 
		
		Suppose now that $\N\ni s\ge 2$ is such that the induction proposition holds at level $s-1$. We will now verify it at level $s$. Suppose that $\N\ni\upnu\ge s-1$, $\N\ni m\ge\max\tcb{2,1+\upnu}$, $N\in\N$ satisfies~\eqref{big boy}, $\rho\le\rho_{\m{reg,normal},\upnu,m}$, and $(q,u,\eta)\in\X^{\upnu}_{m,N}$. As before, we first invoke Proposition~\ref{second normal system synthesis} and obtain the estimate
		\begin{multline}\label{hi'ilawe number 2}
			\tnorm{q,u,\eta}_{\overset{q_0,u_0,\eta_0}{\X^{\upnu}_{m,N}}}\lesssim\sum_{\sig=0}^1\sum_{j=1}^{n-1}\tnorm{\pd_j^\sig q,\pd_j^\sig u,\pd_j^\sig \eta}_{\overset{q_0,u_0,\eta_0}{\X^{\upnu-1}_{m,N}}}+\snorm{\overset{w_0,\gam_0}{A_{m,N}}(q,u,\eta)}_{\Y^\upnu}\\+\tbr{\tnorm{q_0,u_0,\eta_0}_{\X_{\max\tcb{2m,1+\lfloor n/2\rfloor+m}}}}^{2+\lfloor n/2\rfloor}\tnorm{q,u,\eta}_{\overset{q_0,u_0,\eta_0}{\X^{-1}_{m,N}}}.
		\end{multline}
		For $\sig\in\tcb{0,1}$ and $j\in\tcb{1,\dots,n-1}$, we invoke the $(s-1)$-induction hypotheses while again heeding Remark~\ref{nonincreasing remark 2} to bound
		\begin{multline}\label{harpsichord}
			\tnorm{\pd_j^\sig q,\pd_j^\sig u,\pd_j^\sig\eta}_{\overset{q_0,u_0,\eta_0}{\X^{\upnu-1}_{m,N}}}\lesssim\rho\tnorm{\pd_j^\sig q,\pd_j^\sig u,\pd_j^\sig\eta}_{\overset{q_0,u_0,\eta_0}{\X^{\upnu-1}_{m,N}}}+\sum_{\substack{|\al|\le s-1\\\al\cdot e_n=0}}\tnorm{\pd^\al\pd_j^\sig q,\pd^\al\pd_j^\sig  u,\pd^\al\pd_j^\sig\eta}_{\overset{q_0,u_0,\eta_0}{\X^{\upnu-s}_{m,N}}}\\
			+\snorm{\overset{w_0,\gam_0}{A_{m,N}}(\pd_j^\sig q,\pd_j^\sig u,\pd_j^\sig\eta)}_{\Y^{\upnu-1}}+C_{1/\rho}\tbr{\tnorm{q_0,u_0,\eta_0}_{\X_{\max\tcb{2m,1+\tfloor{n/2}+m}}}}^{\chi}\tnorm{\pd_j^\sig q,\pd_j^\sig u,\pd_j^\sig\eta}_{\overset{q_0,u_0,\eta_0}{\X^{-1}_{m,N}}}.
		\end{multline}
		Now we make three estimates. First, we again have the trivial estimate
		\begin{equation}\label{first of three}
			\tnorm{\pd_j^\sig q,\pd_j^\sig u,\pd_j^\sig\eta}_{\overset{q_0,u_0,\eta_0}{\X^{\upnu-1}_{m,N}}}\lesssim\tnorm{q,u,\eta}_{\overset{q_0,u_0,\eta_0}{\X^{\upnu}_{m,N}}}.
		\end{equation}
		Second, by invoking the second item of Proposition~\ref{prop on commutators 2}, the log-convexity of the $\X$-norms from Lemma~\ref{lem on log-convexity of the norms}, and Young's inequality, we get the bounds
		\begin{multline}\label{second of three}
			\snorm{\overset{w_0,\gam_0}{A_{m,N}}(\pd_j^\sig q,\pd_j^\sig u,\pd_j^\sig\eta)}_{\Y^{\upnu-1}}\lesssim\rho\tnorm{q,u,\eta}_{\X_{\upnu}}\\+\snorm{\overset{w_0,\gam_0}{A_{m,N}}(q,u,\eta)}_{\Y^\upnu}+\begin{cases}
				0&\text{if }\upnu\le\tfloor{n/2},\\
				\tnorm{q_0,u_0,\eta_0}_{\X_{1+\upnu}}\tnorm{q,u,\eta}_{\X_{\tfloor{n/2}}}&\text{if }\tfloor{n/2}<\upnu,
			\end{cases}\\
			\lesssim\rho\tnorm{q,u,\eta}_{\X_\upnu}+\snorm{\overset{w_0,\gam_0}{A}(q,u,\eta)}_{\Y^{\upnu}}+\rho^{-1-\tfloor{n/2}}\tbr{\tnorm{q_0,u_0,\eta_0}_{\X_{1+\upnu}}}^{2+\tfloor{n/2}}\tnorm{q,u,\eta}_{\X_{-1}}.
		\end{multline}
		Third, by the log-convexity of the $\X$-norms again and Young's inequality, we get the bound
		\begin{multline}\label{third of three}
			C_{1/\rho}\tbr{\tnorm{q_0,u_0,\eta_0}_{\X_{\max\tcb{2m,1+\tfloor{n/2}+m}}}}^{\chi}\tnorm{\pd_j^\sig q,\pd_j^\sig u,\pd_j^\sig\eta}_{\overset{q_0,u_0,\eta_0}{\X^{-1}_{m,N}}}\lesssim\rho\tnorm{q,u,\eta}_{\overset{q_0,u_0,\eta_0}{\X^\upnu_{m,N}}}\\+\rho^{-1/\upnu}C_{1/\rho}^{(1+\upnu)/\upnu}\tbr{\tnorm{q_0,u_0,\eta_0}_{\X_{\max\tcb{2m,1+m+\tfloor{n/2}}}}}^{\chi(1+\upnu)/\upnu}\tnorm{q,u,\eta}_{\overset{q_0,u_0,\eta_0}{\X^{-1}_{m,N}}}.
		\end{multline}
		Now we combine estimates~\eqref{hi'ilawe number 2} and \eqref{harpsichord} with the trio \eqref{first of three}, \eqref{second of three}, and~\eqref{third of three}; this shows that the induction proposition holds at the level $s$, and hence for all $s \in \N^+$ by induction.   Estimate~\eqref{how high can you leap} from the second item now follows by taking $s=1+\upnu$ in~\eqref{the second inductive hypothesis}.
	\end{proof}
	
	\subsection{Estimates and existence for the principal part}\label{till they find, there's no need}
	
	In this subsection we first prove a priori estimates for systems~\eqref{principal part of the linearization} and~\eqref{regularization of the prinpal part of the linearization}. After this, we derive the existence theory for the former. 
	
	\begin{thm}[A priori estimates for the principal part and the regularization]\label{thm on a priori estimates for the principal part and the regularization}
		Let $0<\rho\le\rho_{\m{WD}}$, where the latter is defined in Theorem~\ref{thm on smooth tameness of the nonlinear operator}, $r\in\tcb{1+\tfloor{n/2},2+\tfloor{n/2}}$, $w_0=(q_0,u_0,\eta_0)\in B_{\X^r}(0,\rho)\cap\X^\infty$, $\gam_0\in I$ for $I\Subset\R^+$ an interval, $\upnu\in\N$, and $(g,f,k)\in\Y^{\upnu}$. The following hold.
		\begin{enumerate}
			\item Let $r=1+\tfloor{n/2}$. There exists a $\rho_{\m{est},\upnu}\in\R^+$, depending only on $\upnu$, the dimension, the various physical parameters, and $I$, such that if $\rho\le\rho_{\m{est},\upnu}$, then for $(q,u,\eta)\in\overset{q_0,u_0,\eta_0}{\X^\upnu}$ satisfying $\overset{w_0,\gam_0}{A}(q,u,\eta)=(g,f,k)$ we have the a priori estimate
			\begin{equation}\label{clarity}
				\tnorm{q,u,\eta}_{\overset{q_0,u_0,\eta_0}{\X^\upnu}}\lesssim\tnorm{g,f,k}_{\Y^\upnu}+\begin{cases}
					0&\text{if }\upnu\le\tfloor{n/2},\\
					\tbr{\tnorm{q_0,u_0,\eta_0}_{\X_{1+\upnu}}}\tnorm{q,u,\eta}_{\X_{\tfloor{n/2}}}&\text{if }\tfloor{n/2}<\upnu.
				\end{cases}
			\end{equation}
			The implicit constants depend on  $\upnu$, the dimension, the physical parameters, $\rho_{\m{est},\upnu}$, and $I$.
			\item Let $r=2+\tfloor{n/2}$ and $\N\ni m\ge\max\tcb{2,1+\upnu}$. There exists a $\rho_{\m{reg},\upnu,m}\in\R^+$, depending only on $\upnu$, $m$, the various physical parameters, and $I$ such that if $\rho\le\rho_{\m{reg},\upnu,m}$, then for $(q,u,\eta)\in\X^\upnu_{m,N}$ satisfying $\overset{w_0,\gam_0}{A_{m,N}}(q,u,\eta)=(g,f,k)$ we have the a priori estimate
			\begin{equation}\label{i couldnt even see the floor}
				\tnorm{q,u,\eta}_{\overset{q_0,u_0,\eta_0}{\X^\upnu_{m,N}}}\lesssim\tnorm{g,f,k}_{\Y^\upnu}+\tbr{\tnorm{q_0,u_0,\eta_0}_{\X_{\max\tcb{2m,1+\tfloor{n/2}+m}}}}^{\chi_{\upnu,m}}\tnorm{q,u,\eta}_{\overset{q_0,u_0,\eta_0}{\X^{-1}_{m,N}}},
			\end{equation}
			where $\chi_{\upnu,m}\in\R^+$ is from the second item of Theorem~\ref{third synthesis of normal system results}. Here the implicit constant depends on $\upnu$, $m$, the dimension, the physical parameters, $\rho_{\m{reg},\upnu,m}$, and $I$.
		\end{enumerate}
	\end{thm}
	\begin{proof}
		We begin by proving the first item. Assume that $\rho\le\min\tcb{\rho_{\m{weak}},\rho_{\m{normal},\upnu}}$, where these smallness parameters are from Propositions~\ref{prop on a priori estimates for weak solutions} and~\ref{first normal system synthesis}, respectively. By combining the a priori estimates for weak solutions from Proposition~\ref{prop on a priori estimates for weak solutions} with the tangential derivative estimates from the first item of Theorem~\ref{thm on tangential derivative analysis}, we learn that
		\begin{multline}\label{the above estimate}
			\sum_{\substack{|\al|\le 1+\upnu\\\al\cdot e_n=0}}\tnorm{\pd^\al q,\pd^\al u,\pd^\al\eta}_{\overset{q_0,u_0,\eta_0}{\X^{-1}}}\lesssim\rho\tnorm{q,u,\eta}_{\X_\upnu}\\+\tnorm{g,f,k}_{\Y^\upnu}+\begin{cases}
				0&\text{if }\upnu\le\tfloor{n/2},\\
				\tnorm{q_0,u_0,\eta_0}_{\X_{1+\upnu}}\tnorm{q,u,\eta}_{\X_{\tfloor{n/2}}}&\text{if }\tfloor{n/2}<\upnu.
			\end{cases}
		\end{multline}
		We insert the bound~\eqref{the above estimate} into conclusion~\eqref{noone around you to carry the blame for you} of the first item of Theorem~\ref{third synthesis of normal system results} to acquire the estimate
		\begin{equation}
			\tnorm{q,u,\eta}_{\overset{q_0,u_0,\eta_0}{\X^\upnu}}\lesssim\rho\tnorm{q,u,\eta}_{\X_\upnu}+\tnorm{g,f,k}_{\Y^\upnu}+\begin{cases}
				0&\text{if }\upnu\le\tfloor{n/2},\\
				\tnorm{q_0,u_0,\eta_0}_{\X_{1+\upnu}}\tnorm{q,u,\eta}_{\X_{\tfloor{n/2}}}&\text{if }\tfloor{n/2}<\upnu.
			\end{cases}
		\end{equation}
		Hence, we may define $\rho_{\m{est},\upnu}\in\R^+$ to be sufficiently small so that when $0<\rho\le\rho_{\m{est},\upnu}$, we may absorb the right hand side's $\tnorm{q,u,\eta}_{\X_{\upnu}}$-contribution by the left and obtain~\eqref{clarity}. This completes the proof of the first item.
		
		Next we consider the second item. The argument is basically the same, but with an extra step. Assume $0<\rho\le\min\tcb{\rho_{\m{weak,reg}},\rho_{\m{reg,normal},\upnu,m}}$, where these smallness parameters are from Propositions~\ref{prop on a priori estimates for regularized weak solutions} and~\ref{second normal system synthesis}, respectively. As in the proof of the first item, we combine the conclusions of Proposition~\ref{prop on a priori estimates for regularized weak solutions}, the second item of Theorem~\ref{thm on tangential derivative analysis}, and the second item of Theorem~\ref{third synthesis of normal system results}; however, the resulting estimate is slightly different:
		\begin{multline}\label{crackerbox palace}
			\tnorm{q,u,\eta}_{\overset{q_0,u_0,\eta_0}{\X^\upnu_{m,N}}}\lesssim\rho\tnorm{q,u,\eta}_{\overset{q_0,u_0,\eta_0}{\X^{\upnu}_{m,N}}}+\tnorm{g,f,k}_{\Y^\upnu}\\+C_{1/\rho}\tbr{\tnorm{q_0,u_0,\eta_0}_{\X_{\max\tcb{2m,1+\tfloor{n/2}+m}}}}^{\chi_{\upnu,m}}\tnorm{q,u,\eta}_{\overset{q_0,u_0,\eta_0}{\X^{-1}_{m,N}}}\\+\begin{cases}
				0&\text{if }\upnu\le\tfloor{n/2},\\
				\tnorm{q_0,u_0,\eta_0}_{\X_{1+\upnu}}\tnorm{q,u,\eta}_{\X_{\tfloor{n/2}}}&\text{if }\tfloor{n/2}<\upnu.
			\end{cases}
		\end{multline}
		To reach the estimate~\eqref{i couldnt even see the floor}, we note that the log-convexity of the $\X$-norms from Lemma~\ref{lem on log-convexity of the norms}, paired with Young's inequality grants the bound
		\begin{multline}\label{thirty three and a third}
			\begin{cases}
				0&\text{if }\upnu\le\tfloor{n/2},\\
				\tnorm{q_0,u_0,\eta_0}_{\X_{1+\upnu}}\tnorm{q,u,\eta}_{\X_{\tfloor{n/2}}}&\text{if }\tfloor{n/2}<\upnu,
			\end{cases}\lesssim\rho\tnorm{q,u,\eta}_{\overset{q_0,u_0,\eta_0}{\X^\upnu_{m,N}}}\\+\rho^{-1-\tfloor{n/2}}\tbr{\tnorm{q_0,u_0,\eta_0}_{\X_{1+\upnu}}}^{2+\tfloor{n/2}}\tnorm{q,u,\eta}_{\overset{q_0,u_0,\eta_0}{\X^{-1}_{m,N}}},
		\end{multline}
		and thus we may combine estimates~\eqref{crackerbox palace} and~\eqref{thirty three and a third}, and then take $0<\rho\le\rho_{\m{reg},\upnu,m}$ sufficiently small to obtain~\eqref{i couldnt even see the floor}.  
	\end{proof}
	
	We now are in a position to give an existence result for the principal part of the linearization, system~\eqref{principal part of the linearization}.

	\begin{thm}[Existence for the principal part]\label{thm on existence for the principal part}
		Let $0<\rho\le\rho_{\m{WD}}$, where the latter is defined in Theorem~\ref{thm on smooth tameness of the nonlinear operator}, and let $w_0=(q_0,u_0,\eta_0)\in B_{\X^{2+\tfloor{n/2}}}(0,\rho)\cap\X^\infty$, $\gam_0\in I$ for $I\Subset\R^+$ an interval. For each $\upnu\in\N$, there exists a $\rho_{\m{exi},\upnu}\in\R^+$, depending on $\upnu$, the physical parameters, and $I$, such that if $\rho\le\rho_{\m{exi},\upnu}$ then the map
		\begin{equation}\label{that the map}
			\overset{w_0,\gam_0}{A}:\overset{q_0,u_0,\eta_0}{\X^\upnu}\to\Y^{\upnu},
		\end{equation}
		the action of which is given via~\eqref{definition of the principal part operator}, is a Banach isomorphism.
	\end{thm}
	\begin{proof}
		Take $\rho_{\m{exi},\upnu}=\min\tcb{\rho_{\m{est},\upnu},\rho_{\m{reg},\upnu,\max\tcb{2,1+\upnu}}}$, where these smallness parameters are from the first and second items of Theorem~\ref{thm on a priori estimates for the principal part and the regularization}. That the map~\eqref{that the map} is well-defined is a consequence of the first item of Proposition~\ref{prop on properties of the A+P decomposition}. This map is injective as a consequence of a priori estimate~\eqref{clarity} from the first item of Theorem~\ref{thm on a priori estimates for the principal part and the regularization}. It remains only to verify surjectivity. 
		
		The proof of the second item of Theorem~\ref{thm on a priori estimates for the principal part and the regularization} shows that $\rho_{\m{exi},\upnu}\le\rho_{\m{reg,weak}}$, and hence we are in a position to apply Corollary~\ref{coro on regularity of the regularization}, which tells us that the maps $\overset{w_0,\gam_0}{A_{\max\tcb{2,1+\upnu},N}}:\X^\upnu_{\max\tcb{2,1+\upnu},N}\to\Y^\upnu$ are Banach isomorphisms for $N\in\N$ sufficiently large, say $N\ge\Bar{N}$. Thus, given $(g,f,k)\in\Y^\upnu$ we can define the sequence $\tcb{(q_N,u_N,\eta_N)}_{N=\Bar{N}}^\infty\subset\X_\upnu$ via $(q_N,u_N,\eta_N)=\sp{\overset{w_0,\gam_0}{A_{\max\tcb{2,1+\upnu},N}}}^{-1}(g,f,k)$. In this we are tacitly using that $\X^\upnu_{m,N}\emb\X_\upnu$ for any $m,N\in\N^+$; in fact, this embedding is non-expansive.  Therefore, by applying the a priori estimate~\eqref{i couldnt even see the floor} from the second item of Theorem~\ref{thm on a priori estimates for the principal part and the regularization}, followed by the regularized weak solution a priori estimate from Proposition~\ref{prop on a priori estimates for regularized weak solutions} (and by also invoking Lemma~\ref{lem on strong solutions are weak solutions}), we obtain the uniform bounds
		\begin{equation}
			\sup_{N\ge\Bar{N}}\sp{\tnorm{q_N,u_N,\eta_N}_{\X_\upnu}+\tnorm{\grad\cdot(v_{w_0}q_N)}_{H^{1+\upnu}}}\lesssim\tbr{\tnorm{q_0,u_0,\eta_0}_{\X_{\max\tcb{2m,1+\tfloor{n/2}+m}}}}^{\chi_{\upnu,m}}\tnorm{g,f,k}_{\Y^\upnu}.
		\end{equation}
		Therefore, we may extract a weak limit  $(q,u,\eta)\in\X_{\upnu}$ with $\grad\cdot(v_{w_0}q)\in H^{1+\upnu}(\Omega)$ such that along some unlabeled subsequence we have that
		\begin{equation}
			(q_N,u_N,\eta_N)\rightharpoonup(q,u,\eta)\text{ in }\X_\upnu
			\text{ and }
			\grad\cdot\tp{v_{w_0}q_N}\rightharpoonup\grad\cdot(v_{w_0}q)\text{ in }H^{1+\upnu}(\Omega).
		\end{equation}
		By routine weak convergence arguments, we readily deduce that $(q,u,\eta)\in\overset{q_0,u_0,\eta_0}{\X^\upnu}$ and $\overset{w_0,\gam_0}{A}(q,u,\eta)=(g,f,k)$, which completes the proof of surjectivity.
	\end{proof}
	
	\subsection{Synthesis of linear analysis}\label{dotting his socks in the night}
	
	In this final subsection of linear analysis, we turn to the study of the full derivative of the nonlinear map $\Bar{\Psi}$ from~\eqref{the nonlinear map equation}, which is associated to the PDE~\eqref{The nonlinear equations in the right form}. In other words, we consider the question of existence and tame estimates for the system
	\begin{equation}
		D\Bar{\Psi}(\theta_0,\gam_0)(q,u,\eta,\mathcal{T},\mathcal{G},\mathcal{F},\gam)=(g,f,k,\mathcal{T},\mathcal{G},\mathcal{F},\gam).
	\end{equation}
	As usual, the unknowns are $q:\Omega\to\R$, $u:\Omega\to\R^n$, and $\eta:\Sigma\to\R$, while the given data are $g:\Omega\to\R$, $f:\Omega\to\R^n$, $k:\Sigma\to\R^n$, $\mathcal{G},\mathcal{F}:\R^n\to\R^n$, $\mathcal{T}:\R^n\to\R^{n\times n}$, and $\gam\in\R$, as well as the background tuple $(\theta_0,\gam_0)$. We now state our main linear analysis result.

	\begin{thm}[Analysis of the linearization]\label{thm on analysis of the linearization 1}
		Let $0<\rho\le\rho_{\m{WD}}$ (recall that the latter is defined in Theorem~\ref{thm on smooth tameness of the nonlinear operator}), $I\Subset\R^+$ be an interval, and
		\begin{equation}\label{the hypotheses on the background solution is a pumpkin}
			(\theta_0,\gam_0)=(q_0,u_0,\eta_0,\mathcal{T}_0,\mathcal{G}_0,\mathcal{F}_0,\gam_0)\in\tp{ B_{\X^{3+\tfloor{n/2}}}(0,\rho)\times B_{\W_{4+\tfloor{n/2}}}(0,\rho)\times I}\cap(\X^\infty\times\W_\infty\times\R^+).
		\end{equation}
		For every $\N\ni\upnu\ge 3+\tfloor{n/2}$ there exists a $\rho_\upnu\in\R^+$, depending on $\upnu$, the physical parameters, and $I$, such that if $0<\rho\le\rho_\upnu$, then the following hold.
		\begin{enumerate}
			\item The map
			\begin{equation}\label{parental supervision}
				D\Bar{\Psi}(\theta_0,\gam_0):\overset{q_0,u_0,\eta_0}{\X^\upnu}\times\W_{1+\upnu}\times\R\to\Y^\upnu\times\W_{1+\upnu}\times\R
			\end{equation}
			is well-defined and a Banach isomorphism.
			\item Assume that $\Xi=(g,f,k,\mathcal{T},\mathcal{G},\mathcal{F},\gam)\in\Y^\upnu\times\W_{1+\upnu}\times\R$ and $\Theta=(D\Bar{\Psi}(\theta_0,\gam_0))^{-1}\Xi\in\overset{q_0,u_0,\eta_0}{\X^\upnu}\times\W_{1+\upnu}\times\R$. We have the tame estimate
			\begin{equation}
				\tnorm{\Theta}_{\overset{q_0,u_0,\eta_0}{\X^\upnu}\times\W_{1+\upnu}\times\R}\lesssim\tnorm{\Xi}_{\Y^\upnu\times\W_{1+\upnu}\times\R}+\tbr{\tnorm{\theta_0}_{\X^{1+\upnu}\times\W_{1+\upnu}}}\tnorm{\Xi}_{\Y^{2+\tfloor{n/2}}\times\W_{3+\tfloor{n/2}}\times\R},
			\end{equation}
			where the implicit constant depends on $\upnu$, the dimension, the physical parameters, $\rho_{\upnu}$, and $I$.
		\end{enumerate}
	\end{thm}
	\begin{proof}
		Suppose initially that $\N\ni\upnu\ge 2+\tfloor{n/2}$. The idea of the proof is to prove a priori estimates, uniform with respect to a parameter, and utilize the method of continuity. To this end, for $\tau\in[0,1]$ we define the convex homotopy of operators
		\begin{equation}
			L^\tau(\theta_0,\gam_0):\overset{q_0,u_0,\eta_0}{\X^\upnu}\times\W_{1+\upnu}\times\R\to\Y^\upnu\times\W_{1+\upnu}\times\R
		\end{equation}
		via 
		\begin{equation}\label{given via the formulae}
			L^\tau(\theta_0,\gam_0)(q,u,\eta,\mathcal{T},\mathcal{G},\mathcal{F},\gam)=\bpm \overset{w_0,\gam_0}{A}(q,u,\eta)+\tau\sp{\overset{w_0,\gam_0}{P}+\overset{\theta_0}{Q}}(q,u,\eta,\gam)+\overset{q_0,u_0,\eta_0}{R}(\mathcal{T},\mathcal{G},\mathcal{F})\\\mathcal{T},\mathcal{G},\mathcal{F},\gam\epm.
		\end{equation}
		Note that $L^1(\theta_0,\gam_0)=D\Bar{\Psi}(\theta_0,\gam_0)$ from~\eqref{formula for the derivative for like he had no place to go psi} and that the mapping properties of $A$, $P$, $Q$, and $R$ established in Propositions~\ref{prop on properties of the A+P decomposition} and~\ref{prop on the Q+R decomposition of DPhi} ensure well-definedness and continuity of the maps $L^\tau(\theta_0,\gam_0)$. 
		
		We now aim to prove $\tau$-uniform a priori estimates for $\tcb{L^\tau(\theta_0,\gam_0)}_{\tau\in[0,1]}$. Assume that
		\begin{equation}
			(q,u,\eta,\mathcal{T},\mathcal{G},\mathcal{F},\gam)\in\overset{q_0,u_0,\eta_0}{\X^\upnu}\times\W_{1+\upnu}\times\R
			\text{ and } 
			(g,f,k)\in\Y^\upnu
		\end{equation} 
		are related via
		\begin{equation}
			L^\tau(\theta_0,\gam_0)(q,u,\eta,\mathcal{T},\mathcal{G},\mathcal{F},\gam)=(g,f,k,\mathcal{T},\mathcal{G},\mathcal{F},\gam).
		\end{equation}
		The first three components of the above equation are equivalent to
		\begin{equation}\label{utilization station}
			\overset{w_0,\gam_0}{A}(q,u,\eta)=(g,f,k)-\tau\sp{\overset{w_0,\gam_0}{P}+\overset{\theta_0}{Q}}(q,u,\eta,\gam)-\overset{q_0,u_0,\eta_0}{R}(\mathcal{T},\mathcal{G},\mathcal{F})=(\tilde{g},\tilde{f},\tilde{k}).
		\end{equation}
		Assume that $0<\rho\le\rho_{\m{est},\upnu}$; then we may invoke the principal part estimates from the first item of Theorem~\ref{thm on a priori estimates for the principal part and the regularization} to see that
		\begin{multline}\label{you give me your wah wah}
			\tnorm{q,u,\eta}_{\overset{q_0,u_0,\eta_0}{\X^\upnu}}\lesssim\tnorm{g,f,k}_{\Y^\upnu}+\snorm{\overset{w_0,\gam_0}{P}(q,u,\eta,\gam)}_{\Y^\upnu}+\snorm{\overset{\theta_0}{Q}(q,u,\eta)}_{\Y^\upnu}\\+\snorm{\overset{q_0,u_0,\eta_0}{R}(\mathcal{T},\mathcal{G},\mathcal{F})}_{\Y^\upnu}+\tbr{\tnorm{q_0,u_0,\eta_0}_{\X_{1+\upnu}}}\tnorm{q,u,\eta}_{\X_{\tfloor{n/2}}}.
		\end{multline}
		According to the second item of Proposition~\ref{prop on properties of the A+P decomposition} and the first and second items of Proposition~\ref{prop on the Q+R decomposition of DPhi}, we have the bounds
		\begin{equation}\label{bounds for P}
			\snorm{\overset{w_0,\gam_0}{P}(q,u,\eta,\gam)}_{\Y^\upnu}\lesssim\rho\tnorm{q,u,\eta}_{\X_{\upnu}}+\tnorm{q_0,u_0,\eta_0}_{\X_{1+\upnu}}\tnorm{q,u,\eta,\gam}_{\X_{\tfloor{n/2}}\times\R},
		\end{equation}
		\begin{equation}\label{bounds for Q}
			\snorm{\overset{\theta_0}{Q}(q,u,\eta)}_{\Y^\upnu}\lesssim\rho\tnorm{q,u,\eta}_{\X_\upnu}+\tnorm{\theta_0}_{\X_{\upnu}\times\W_{1+\upnu}}\tnorm{q,u,\eta}_{\X_{2+\tfloor{n/2}}},
		\end{equation}
		and
		\begin{equation}\label{bounds for R}
			\snorm{\overset{q_0,u_0,\eta_0}{R}(\mathcal{T},\mathcal{G},\mathcal{F})}_{\Y^\upnu}\lesssim\tnorm{\mathcal{T},\mathcal{G},\mathcal{F}}_{\W_{1+\upnu}}+
			\tbr{\tnorm{q_0,u_0,\eta_0}_{\X_{\upnu}}}\tnorm{\mathcal{T},\mathcal{G},\mathcal{F}}_{\W_{3+\tfloor{n/2}}}.
		\end{equation}
		Inserting~\eqref{bounds for P}, \eqref{bounds for Q}, and~\eqref{bounds for R} into the right hand side of~\eqref{you give me your wah wah}, we obtain the  estimate
		\begin{multline}\label{fishcake by the sea all my friends}
			\tnorm{q,u,\eta}_{\overset{q_0,u_0,\eta_0}{\X^\upnu}}\lesssim\rho\tnorm{q,u,\eta}_{\X_\upnu}+\tnorm{g,f,k,\mathcal{T},\mathcal{G},\mathcal{F}}_{\Y^\upnu\times\W_{1+\upnu}}\\+\tbr{\tnorm{\theta_0}_{\X_{1+\upnu}\times\W_{1+\upnu}}}\tnorm{q,u,\eta,\mathcal{T},\mathcal{G},\mathcal{F},\gam}_{\X_{2+\tfloor{n/2}}\times\W_{3+\tfloor{n/2}}\times\R}.
		\end{multline}
		Now we take  $0<\rho\le\min\tcb{\rho_{\m{est},\upnu},\rho_{\m{exi},\upnu}}$, where the latter smallness parameters are from Theorems~\ref{thm on a priori estimates for the principal part and the regularization} and~\ref{thm on existence for the principal part}, to be sufficiently small so that we can absorb as usual in~\eqref{fishcake by the sea all my friends} to see that
		\begin{equation}\label{just dance it'll be ok}
			\tnorm{q,u,\eta}_{\overset{q_0,u_0,\eta_0}{\X^\upnu}}\lesssim\tnorm{g,f,k,\mathcal{T},\mathcal{G},\mathcal{F}}_{\Y^\upnu\times\W_{1+\upnu}}+\tbr{\tnorm{\theta_0}_{\X_{1+\upnu}\times\W_{1+\upnu}}}\tnorm{q,u,\eta,\mathcal{T},\mathcal{G},\mathcal{F},\gam}_{\X_{2+\tfloor{n/2}}\times\W_{3+\tfloor{n/2}}\times\R}.
		\end{equation}
		
		We need to handle the norm of $(q,u,\eta)$ in $\X_{2+\tfloor{n/2}}$ on the right hand side of \eqref{just dance it'll be ok}.  According to the first item of Theorem~\ref{thm on a priori estimates for the principal part and the regularization} and~\eqref{the hypotheses on the background solution is a pumpkin},  we know that
		\begin{equation}\label{equation number 1}
			\tnorm{q,u,\eta}_{\X_{2+\tfloor{n/2}}}\le\tnorm{q,u,\eta}_{\overset{q_0,u_0,\eta_0}{\X^{2+\tfloor{n/2}}}}\lesssim\tnorm{\tilde{g},\tilde{f},\tilde{k}}_{\Y^{2+\tfloor{n/2}}},
		\end{equation}
		where the latter is defined in~\eqref{utilization station}. By arguing as in \eqref{bounds for P}, \eqref{bounds for Q}, and~\eqref{bounds for R} again but using~\eqref{the hypotheses on the background solution is a pumpkin}, we learn that
		\begin{equation}\label{equation number 10}
			\tnorm{\tilde{g},\tilde{f},\tilde{k}}_{\Y^{2+\tfloor{n/2}}}\lesssim\rho\tnorm{q,u,\eta,\gam}_{\X_{2+\tfloor{n/2}}\times\R}+\tnorm{g,f,k}_{\Y^{2+\tfloor{n/2}}}+\tnorm{\mathcal{T},\mathcal{G},\mathcal{F}}_{\W_{3+\tfloor{n/2}}}.
		\end{equation}
		By combining~\eqref{equation number 1} and~\eqref{equation number 10} and taking $\rho$ smaller, if necessary, we gain the estimate
		\begin{equation}\label{now insert}
			\tnorm{q,u,\eta}_{\X_{2+\tfloor{n/2}}}\lesssim\tnorm{g,f,k,\mathcal{T},\mathcal{G},\mathcal{F},\gam}_{\Y^{2+\tfloor{n/2}}\times\W_{3+\tfloor{n/2}}\times\R}.
		\end{equation}
		Finally, we insert~\eqref{now insert} into~\eqref{just dance it'll be ok} to acquire the estimate
		\begin{multline}\label{you'd better hurry}
			\tnorm{q,u,\eta,\mathcal{T},\mathcal{G},\mathcal{F},\gam}_{\overset{q_0,u_0,\eta_0}{\X^\upnu}\times\W_{1+\upnu}\times\R}\lesssim\tnorm{L^\tau(\theta_0,\gam_0)(q,u,\eta,\mathcal{T},\mathcal{G},\mathcal{F},\gam)}_{\Y^\upnu\times\W_{1+\upnu}\times\R}\\+\tbr{\tnorm{\theta_0}_{\X_{1+\upnu}\times\W_{1+\upnu}}}\tnorm{L^\tau(\theta_0,\gam_0)(q,u,\eta,\mathcal{T},\mathcal{G},\mathcal{F},\gam)}_{\Y^{2+\tfloor{n/2}}\times\W_{3+\tfloor{n/2}}\times\R}.
		\end{multline}
		
		The bound \eqref{you'd better hurry} establishes the desired $\tau$-uniform  a prior bounds for $\tcb{L^\tau(\theta_0,\gam_0)}_{\tau\in[0,1]}$, as long as $\rho\le\rho_\upnu$ for some $\rho_{\upnu}\in\R^+$. Therefore, by the method of continuity (see, for instance, Theorem 5.2 in Gilbarg and Trudinger~\cite{MR1814364}), the invertibility of $L^1(\theta_0,\gam_0)=D\Bar{\Psi}(\theta_0,\gam_0)$ is established as soon as we know that $L^0(\theta_0,\gam_0)$ is invertible.  The latter holds thanks to Theorem~\ref{thm on existence for the principal part} (and the fact that $\rho_{\upnu}\le\rho_{\m{exi},\upnu}$) since
		\begin{equation}
			(L^0(\theta_0,\gam_0))^{-1}(g,f,k,\mathcal{T},\mathcal{G},\mathcal{F},\gam)=\sp{\sp{\overset{w_0,\gam_0}{A}}^{-1}\sp{(g,f,k)-\overset{q_0,u_0,\eta_0}{R}(\mathcal{T},\mathcal{G},\mathcal{F})},\mathcal{T},\mathcal{G},\mathcal{F},\gam}.
		\end{equation}
		This completes the proof of the first item. The second item follows from the first and estimate~\eqref{you'd better hurry} at $\tau=1$.
	\end{proof}
	
	% - space - space - outer - % - space - space - outer - % - space - space - outer - % - space - space - outer - % - space - space - outer - % - space - space - outer - % - space - space - outer - % - space - space - outer - % - space - space - outer - % - space - space - outer - % - space - space - outer - % - space - space - outer -

 \section{Conclusion}\label{she could steal, but she could not rob}

In this section we prove our main result, Theorem~\ref{thm on main}.  We begin with Theorem~\ref{main thm 1 on traveling wave solutions to the free boundary compressible Navier-Stokes equations}, which is an abstract construction in the Nash-Moser inverse function theorem framework from Section~\ref{section on NMH}. We then employ the abstract construction for the PDE-style result in Theorem~\ref{main thm 2 on traveling wave solutions to the free boundary compressible Navier-Stokes equations}.

 \subsection{Abstract construction}

Recall the scales of Banach spaces from~\eqref{main_thm_banach_scales_def}, the nonlinear operator $\Bar{\Psi}$ from~\eqref{the nonlinear map equation}, and the parameter $\rho_{\m{WD}} \in \R^+$ from Theorem \ref{thm on smooth tameness of the nonlinear operator}.  To state the next theorem we introduce the following notation. First, we set $\be=9+2\tfloor{n/2}$. Second, for $\N\ni s\ge\be$ and $\rho\in\R^+$ we define the open sets
\begin{equation}\label{before stegosaurus take the t-bird away now}
    E_s(\rho)=B_{\E^\be}(0,\rho)\cap\E^s
    \text{ and }
    F_s(\rho)=B_{\F^\be}(0,\rho)\cap\F^s,
\end{equation}
where $\E^s$ and $\F^s$ are defined by \eqref{North Carolina}.  Finally, it is convenient to  introduce a translation of $\Bar{\Psi}$; namely, for $\Gamma\in\R^+$ we set
\begin{equation}\label{fester uncle pester}
    \Bar{\Psi}_{\Gamma}(q,u,\eta,\mathcal{T},\mathcal{G},\mathcal{F},\gam)=\Bar{\Psi}(q,u,\eta,\mathcal{T},\mathcal{G},\mathcal{F},\Gamma+\gam)-(0,0,0,0,0,0,\Gamma),
\end{equation}
defined for $(q,u,\eta)\in B_{\X^\be}(0,\rho_{\m{WD}})$, $(\mathcal{T},\mathcal{G},\mathcal{F})\in\W_{1+\be}$, and $\R^+\ni\gam>-\Gamma$.  The utility of this definition is that $\Bar{\Psi}_{\Gamma}(0)=0$.

\begin{thm}[Traveling waves for free boundary compressible Navier-Stokes, 1]\label{main thm 1 on traveling wave solutions to the free boundary compressible Navier-Stokes equations}
    Assume that the parameters $\upmu$, $\uplambda$, and $\varsigma$ satisfy \eqref{parameter_assumptions} and that $P$ satisfies $P'>0$ and~\eqref{equilibrium_ccs}.   Let $\Gamma\in\R^+$, and let $\beta,$ $E_s(\rho)$, $F_s(\rho)$, and $\bar{\Psi}_{\Gamma}$ be as defined above.  There exists a nonincreasing sequence $\tcb{\ep_\nu(\Gamma)}_{\nu=0}^\infty\subset(0,\min\tcb{\rho_{\m{WD}},\Gamma/2})$ and $\kappa(\Gamma)\in\R^+$ such that the following hold.
    \begin{enumerate}
        \item Existence and uniqueness: Given $\bf{f}\in F_\be(\ep_0(\Gamma))$, there exists a unique $\bf{e}\in E_\be(\kappa(\Gamma)\ep_0(\Gamma))$ such that $\Bar{\Psi}_\Gamma(\bf{e})=\bf{f}$.  This induces the local inverse map
        \begin{equation}
            \Bar{\Psi}_{\Gamma}^{-1} : F_\be(\ep_0(\Gamma)) \to \Bar{\Psi}_{\Gamma}^{-1}(F_\be(\ep_0(\Gamma))) \subseteq E_\be(\kappa(\Gamma)\ep_0(\Gamma)).
        \end{equation}
        
        \item Higher regularity, given low norm smallness: If $\nu\in\N$ and $\bf{f}\in F_{\be+\nu}(\ep_\nu(\Gamma))$, then $\Bar{\Psi}_{\Gamma}^{-1}(\bf{f})\in\E^{\be+\nu}$, and we have the tame estimate
        \begin{equation}
            \tnorm{\Bar{\Psi}_\Gamma^{-1}(\bf{f})}_{\E^{\be+\nu}}\lesssim\tnorm{\bf{f}}_{\F^{\be+\nu}},
        \end{equation}
        for an implicit constant depending only on the dimension, the physical parameters, $\nu$, $\rho_{\m{WD}}$, and $\Gamma$.
        
        \item Continuous dependence: For every $\nu\in\N$, the restricted map
        \begin{equation}
            \Bar{\Psi}^{-1}_\Gamma:F_{\be+\nu}(\ep_\nu(\Gamma))\to\E^{\be+\nu}
        \end{equation}
        is continuous with respect to the norms on $\F^{\be+\nu}$ and $\E^{\be+\nu}$.

        \item Continuous differentiability: For every $\nu\in\N$, the restricted map
        \begin{equation}
            \Bar{\Psi}_\Gamma^{-1}:F_{\be+\nu}(\ep_{\nu+2}(\Gamma))\to\E^{\be+\nu-1}
        \end{equation}
        is differentiable.  Moreover, if we view $D\Bar{\Psi}_\Gamma^{-1}$ as a map
        \begin{equation}\label{Kansas}
            D\Bar{\Psi}_\Gamma^{-1}:F_{\be+\nu}(\ep_{\nu+2}(\Gamma))\times\F^{\be+\nu-1}\to\E^{\be+\nu-1},
        \end{equation}
        then $D\Bar{\Psi}_\Gamma^{-1}$ is continuous.
    \end{enumerate}
\end{thm}
\begin{proof}
Our aim is to show that the hypotheses of Theorem~\ref{thm on nmh} are satisfied by $\bar{\Psi}_\Gamma$ with the Banach scales $\pmb{\E}$ and $\pmb{\F}$ defined in \eqref{main_thm_banach_scales_def}. Thanks to Lemma~\ref{lem on tameness of domain and codomain}, we have that condition $II$ from the hypotheses of Theorem~\ref{thm on nmh} is satisfied by $\pmb{\E}$ and $\pmb{\F}$. In verifying the rest of the hypotheses we consider the triple $(\pmb{\E},\pmb{\F},\Bar{\Psi}_\Gamma)$ with parameters  $(\mu,r,R)=(1,3+\tfloor{n/2},17+3\tfloor{n/2})$, which obey the requisite inequality $2(r+\mu)+1=\be<(r+R)/2$ with $\beta=9+2\tfloor{n/2}$.  

We next set $\del_r=\min\tcb{\rho_{\m{WD}},\Gamma/2}$.  Then Theorem~\ref{thm on smooth tameness of the nonlinear operator} and Remark~\ref{I feel hung up and I dont know why, its alright} prove that the second item ($C^2$ and $\mu$-tameness) of Definition~\ref{defn of the mapping hypotheses} is satisfied.  Invoking Theorem~\ref{thm on analysis of the linearization 1} with $\upnu=R$ and interval $I=[\Gamma/2,2\Gamma]$ then shows that the third item (derivative inversion) of Definition~\ref{defn of the mapping hypotheses} holds with $\del_R=\min\tcb{\rho_{\upnu},\del_r}$ (where $\rho_{\upnu}$ is given by Theorem~\ref{thm on analysis of the linearization 1}) and that the remainder of the LRI mapping hypotheses from the definition are also satisfied.

Hence, the hypotheses of Theorem~\ref{thm on nmh} are satisfied.  We thus obtain an $\ep = \ep_0(\Gamma)$ such that the first item holds, and we have the estimate
\begin{equation}\label{Indiana}
    \tnorm{\Bar{\Psi}_\Gamma^{-1}(\bf{f})}_{\E^{\be}}\lesssim\tnorm{\bf{f}}_{\F^{\be}} 
\end{equation}
for all $\bf{f}\in F_\be(\ep_0(\Gamma))$.  In fact, by the second item of Theorem~\ref{thm on nmh}, we have that for all $s\in[\be,R+r-\be]\cap\N=[9+\tfloor{n/2},11+3\tfloor{n/2}]\cap\N$ it holds that if $\bf{f}\in F_\be(\ep_0(\Gamma))\cap\F^s$, then $\Bar{\Psi}_\Gamma^{-1}(\bf{f})\in\E^s$ and estimate~\eqref{Indiana} holds with $\be$ replaced by $s$. Therefore, we may set $\ep_\nu(\Gamma)=\ep_0$ for $\nu\in\tcb{0,1,\dots,2+2\tfloor{n/2}=R+r-2\be}$.

Now, given $\N\ni\nu>2+2\tfloor{n/2}$, we define $\ep_\nu(\Gamma)\in(0,\ep_0]$ to be the smallness parameter $\ep$ provided by Theorem~\ref{thm on nmh} with parameter triple $(\mu,r,\tilde{R})$, where $\tilde{R}=R+\nu-(2+2\tfloor{n/2})$. Note that the LRI mapping hypotheses are satisfied in this case if we take $\del_{\tilde{R}}=\min\tcb{\rho_{R+\nu-(2+2\tfloor{n/2})},\del_r}$, where the former parameter is from Theorem~\ref{thm on analysis of the linearization 1} (with $\upnu=R+\nu-(2+2\tfloor{n/2})$ and $I=[\Gamma/2,2\Gamma]$). The second item   now follows from this definition of $\tcb{\ep_\nu(\Gamma)}_{\nu=0}^\infty$, the second item of Theorem~\ref{thm on nmh}, and uniqueness.  Finally, the third and fourth items now follow directly from Theorem~\ref{thm on further conclusions of the inverse function theorem}.
\end{proof}

\begin{rmk}\label{abstract_higher_reg_remark}
    As a consequence of the third conclusion of Theorem~\ref{thm on further conclusions of the inverse function theorem}, the local inverse map $\Bar{\Psi}_\Gamma^{-1}$ produced in Theorem~\ref{main thm 1 on traveling wave solutions to the free boundary compressible Navier-Stokes equations} satisfies certain higher-order differentiability assertions beyond the basic continuous differentiability result in the third item of the previous theorem. However, the precise statements become cumbersome to enumerate due to a compounding of the derivative loss in the formulas for higher-order derivatives of the inverse map.  As such, we have chosen not to state these precisely.
\end{rmk}

 \subsection{PDE construction}  We utilize Theorem \ref{main thm 1 on traveling wave solutions to the free boundary compressible Navier-Stokes equations} to prove our main result about the free boundary compressible Navier-Stokes equations, system \eqref{The nonlinear equations in the right form}.

\begin{thm}[Traveling wave for free boundary compressible Navier-Stokes, 2]\label{main thm 2 on traveling wave solutions to the free boundary compressible Navier-Stokes equations}
    Assume that the parameters $\upmu$, $\uplambda$, and $\varsigma$ satisfy \eqref{parameter_assumptions} and that $P$ satisfies $P'>0$ and~\eqref{equilibrium_ccs}. 
    Let $\be=9+2\tfloor{n/2}$.  There exist a family $\tcb{\mathcal{V}(\gamma)}_{\gamma\in\R^+}$ of open sets of $\X^\be$ and a nonincreasing sequence $\tcb{\mathcal{U}_s}_{s=\be}^\infty$ of open sets of $\W_{1+\be}\times\R^+$ such that the following hold.
    \begin{enumerate}
        \item Nondegeneracy: We have that $\tcb{0}\times\R^+ \subseteq \bigcap_{s=\be}^\infty \mathcal{U}_s$ and $0\in\bigcap_{\gamma\in\R^+}\mathcal{V}(\gamma)$.

        \item Existence and uniqueness:  For all $(\mathcal{T},\mathcal{G},\mathcal{F},\gam)\in\mathcal{U}_\be$ there exists a unique $(q,u,\eta)\in\mathcal{V}(\gam)$ such that the traveling wave formulation for the free boundary compressible Navier-Stokes equations, system~\eqref{The nonlinear equations in the right form}, is classically satisfied with wave speed $\gam$, data $(\mathcal{T},\mathcal{G},\mathcal{F})$, and solution $(q,u,\eta)$.
        
        \item Higher regularity, given low norm smallness: If $\N\ni s\ge\be$ and $(\mathcal{T},\mathcal{G},\mathcal{F},\gam)\in\mathcal{U}_{s}\cap(\W_{1+s}\times\R)$, then the corresponding solution satisfies $(q,u,\eta)\in\mathcal{V}(\gam)\cap\X^{1+s}$.
        
        \item Continuous dependence: For any $\N\ni s\ge\be$, the solution  map
        \begin{equation}
        \mathcal{U}_s\cap(\W_{1+s}\times\R)\ni(\mathcal{T},\mathcal{G},\mathcal{F},\gam)\mapsto(q,u,\eta)\in\X^{1+s}
        \end{equation}
        is continuous with respect to the $\W_{1+s}\times\R$ and $\X^{1+s}$ norms.
        \item No vacuum formation: There exists positive constants $c,C\in\R^+$ such that for all $(q,u,\eta)\in\bigcup_{\gam\in\R^+}\mathcal{V}(\gam)$ we have that $c\le \sig_{q,\eta}\le C$, where $\sig_{q,\eta}$ is defined in~\eqref{sigma_q_eta_def}.

        \item Flattening map diffeomorphism: For any $\N\ni s\ge\be$ and $(q,u,\eta)\in \X^s\cap\bigcup_{\gam\in\R^+}\mathcal{V}(\gam)$ we have that the flattening map $\mathfrak{F}_\eta$ from~\eqref{flattening_map_def} is a smooth diffeomorphism from $\Omega$ to $\Omega[\eta]$ that extends to a         $C^{s+2-\tfloor{n/2}}$ diffeomorphism from $\Bar{\Omega}$ to $\Bar{\Omega[\eta]}$.

    \end{enumerate}
\end{thm}
\begin{proof}
For each $\Gamma \in \R^+$, we many invoke Theorem~\ref{main thm 1 on traveling wave solutions to the free boundary compressible Navier-Stokes equations} and acquire a nonincreasing sequence $\tcb{\ep_\nu(\Gamma)}_{\nu=0}^\infty\subset(0,\min\tcb{\rho_{\m{WD}},\Gamma/2})$ and a $\kappa(\Gamma)\in\R^+$ such that the various conclusions of the theorem hold. We then define the open sets 
\begin{equation}\label{pester uncle fester}
    \mathcal{U}_s=\bigcup_{\Gamma\in\R^+}B_{\W_{1+\be}\times\R}((0,\Gamma),\ep_{s-\be+1}(\Gamma)/C) \text{ for } \N\ni s\ge\be
\end{equation}
for some constant $C\ge 1$ (independent of $s$) to be determined, and 
\begin{equation}\label{pester uncle fester 2}
        \mathcal{V}(\Gamma)=B_{\X^\be}(0,\kappa(\Gamma)\ep_0(\Gamma)) \text{ for } \Gamma \in \R^+.    
\end{equation}
 The first item is now clear by inspection.
    
    For the second item, we apply the first conclusion of Theorem~\ref{main thm 1 on traveling wave solutions to the free boundary compressible Navier-Stokes equations}. Indeed, for $(\mathcal{T},\mathcal{G},\mathcal{F},\gam)\in\mathcal{U}_\be$ we have existence by setting $(q,u,\eta)$ to be the first three components of the tuple   $\Bar{\Psi}_\gam^{-1}(0,0,0,\mathcal{T},\mathcal{G},\mathcal{F},0)$. Uniqueness in $\mathcal{V}(\gamma)$ follows from the fact that if $\Bar{\Psi}(q,u,\eta,\mathcal{T},\mathcal{G},\mathcal{F},\gam)=\Bar{\Psi}(\tilde{q},\tilde{u},\tilde{\eta},\mathcal{T},\mathcal{G},\mathcal{F},\gam)$, then $\Bar{\Psi}_\gam(q,u,\eta,\mathcal{T},\mathcal{G},\mathcal{F},0)=\Bar{\Psi}_\gam(\tilde{q},\tilde{u},\tilde{\eta},\mathcal{T},\mathcal{G},\mathcal{F},0)$ and hence $(q,u,\eta)=(\tilde{q},\tilde{u},\tilde{\eta})$ by the uniqueness assertions of  Theorem~\ref{main thm 1 on traveling wave solutions to the free boundary compressible Navier-Stokes equations}.

    We now prove the third item. An immediate consequence of the second conclusion of Theorem~\ref{main thm 1 on traveling wave solutions to the free boundary compressible Navier-Stokes equations} is that if $\N\ni s\ge\be$ and $(\mathcal{T},\mathcal{G},\mathcal{F},\gam)\in\mathcal{U}_s\cap(\W_{1+s}\times\R)$, then the corresponding solution satisfies $(q,u,\eta)\in\mathcal{V}(\gam)\cap\X^s$.  In fact,  the solution is one degree more regular.  To see this, we recall the map $\Phi$ from~\eqref{hey i think this may be the last thing I label} and note that the equation $\Bar{\Psi}(q,u,\eta,\mathcal{T},\mathcal{G},\mathcal{F},\gam)=(0,0,0,\mathcal{T},\mathcal{G},\mathcal{F},\gam)$ is equivalent to $\Bar{\Psi}(q,u,\eta,0,0,0,\gam)=(-\Phi(q,u,\eta,\mathcal{T},\mathcal{G},\mathcal{F}),0,0,0,\gam)$, which in turn is equivalent to
    \begin{equation}\label{on second thought I might need to reference this equation too}
        \Bar{\Psi}_\gam(q,u,\eta,0,0,0,0)=(-\Phi(q,u,\eta,\mathcal{T},\mathcal{G},\mathcal{F}),0,0,0,0).
    \end{equation}
    Note that $\Phi(q,u,\eta,\mathcal{T},\mathcal{G},\mathcal{F})$ is linear in $(\mathcal{T},\mathcal{G},\mathcal{F})$; consequently, Propositions~\ref{prop on smooth tameness of the momentum equation 2} and~\ref{smooth tameness of dynamic boundary condition 2} (with $m=0$) show that
    \begin{equation}\label{round round get around I get around}
        \tnorm{\Phi(q,u,\eta,\mathcal{T},\mathcal{G},\mathcal{F})}_{\Y^\be}\le c\tnorm{\mathcal{T},\mathcal{G},\mathcal{F}}_{\W_\be}
        \text{ and }
        \tnorm{\Phi(q,u,\eta,\mathcal{T},\mathcal{G},\mathcal{F})}_{\Y^{1+s}}\lesssim_{\tnorm{q,u,\eta}_{\X^s}}\tnorm{\mathcal{T},\mathcal{G},\mathcal{F}}_{\W_{1+s}}
    \end{equation}
    for a constant $c$ depending only on the dimension, the physical parameters, and $\rho_{\m{WD}}$. Assume the constant $C\ge 1$ from \eqref{pester uncle fester} satisfies $C\gtrsim c$.  We then use the inclusion $(\mathcal{T},\mathcal{G},\mathcal{F},\gam)\in\mathcal{U}_s\cap(\W_{1+s}\times\R)$ and \eqref{round round get around I get around} to see that
    \begin{equation}\label{everybody gonna surfin, surfin US and A}
        (-\Phi(q,u,\eta,\mathcal{T},\mathcal{G},\mathcal{F}),0,0,0,\gam)\in\F^{1+s}\cap\bigcup_{\Gamma\in\R^+}B_{\F^{\be}\times\R}((0,\Gamma),\ep_{s-\be+1}(\Gamma)).
    \end{equation}
    Hence, we can apply the first and second conclusions of Theorem~\ref{main thm 1 on traveling wave solutions to the free boundary compressible Navier-Stokes equations} to deduce from~\eqref{on second thought I might need to reference this equation too} and~\eqref{everybody gonna surfin, surfin US and A} that the third item  holds.

    It remains to justify the fourth item.  Identity~\eqref{fester uncle pester}  and the third conclusion of Theorem~\ref{main thm 1 on traveling wave solutions to the free boundary compressible Navier-Stokes equations} show that for $\N\ni s\ge\be$, the solution map 
        \begin{equation}\label{this continuity}
        \mathcal{U}_{s-1}\cap(\W_{1+s}\times\R)\ni(\mathcal{T},\mathcal{G},\mathcal{F},\gam)\mapsto(q,u,\eta)\in\X^{s}
        \end{equation}
        is continuous. In fact, we can do one derivative better for the target space topology by arguing as in the proof of the third item above.  Indeed, by Propositions~\ref{prop on smooth tameness of the momentum equation 2} and~\ref{smooth tameness of dynamic boundary condition 2} again, along with~\eqref{this continuity}, we find that the map
        \begin{equation}\label{pauls broken a glass broken a glass}
            \mathcal{U}_s\cap(\W_{1+s}\times\R)\ni(\mathcal{T},\mathcal{G},\mathcal{F},\gam)\mapsto(-\Phi(q,u,\eta,\mathcal{T},\mathcal{G},\mathcal{F}),0,0,0,\gam)\in\F^{1+s}
        \end{equation}
        is continuous.  In light of the identity
        \begin{equation}\label{thats an interesting}
            \Bar{\Psi}(q,u,\eta,0,0,0,\gam)=(-\Phi(q,u,\eta,\mathcal{T},\mathcal{G},\mathcal{F}),0,0,0,\gam),
        \end{equation}
        the third conclusion of Theorem~\ref{main thm 1 on traveling wave solutions to the free boundary compressible Navier-Stokes equations} (which implies that $\Bar{\Psi}^{-1}$ is continuous), the inclusion \eqref{everybody gonna surfin, surfin US and A} (which implies that the right side of~\eqref{thats an interesting} is in the domain of $\Bar{\Psi}^{-1}$), the continuity of the map ~\eqref{pauls broken a glass broken a glass}, and the continuity of compositions, we then complete the verification of the fourth item.

        The fifth and sixth items are immediate consequences of the first conclusions of Proposition~\ref{prop on smooth tameness of the continuity equation} and Theorem~\ref{thm on smooth tameness of the nonlinear operator}. This completes the proof.
\end{proof}

We conclude this section by recording some simple consequences of Theorem~\ref{main thm 2 on traveling wave solutions to the free boundary compressible Navier-Stokes equations}.

\begin{coro}[Open set of data for fixed wave speed]\label{coro on rain}
For each $\gam\in\R^+$ there exists a nonempty open set $(0,0,0)\in\mathcal{W}(\gam)\subset\W_{1+\be}$ with the property that for all stress-force data tuples $(\mathcal{T},\mathcal{G},\mathcal{F})\in\mathcal{W}(\gam)$ there exists a unique $(q,u,\eta)\in\mathcal{V}(\gam)$ such that system~\eqref{The nonlinear equations in the right form} is satisfied with solution $(q,u,\eta)$, wave speed $\gam$, and data $(\mathcal{T},\mathcal{G},\mathcal{F})$.
\end{coro}
\begin{proof}
    Given $\gamma\in\R^+$ we take $\mathcal{W}(\gam)=\tcb{(\mathcal{T},\mathcal{G},\mathcal{F})\;:\;(\mathcal{T},\mathcal{G},\mathcal{F},\gam)\in\mathcal{U}_\be}$.
\end{proof}

\begin{coro}[On the Eulerian formulation]\label{blowing in the wind}
    Each solution to the flattened, perturbative enthalpy formulation of the traveling wave problem for free boundary compressible Navier-Stokes, i.e system~\eqref{The nonlinear equations in the right form}, produced by Theorem~\ref{main thm 2 on traveling wave solutions to the free boundary compressible Navier-Stokes equations}, gives rise to a classical solution to the traveling Eulerian formulation of the problem given by system~\eqref{compressible navier-stokes traveling wave equations rescaled}.
\end{coro}
\begin{proof}
    The sixth item of Theorem~\ref{main thm 2 on traveling wave solutions to the free boundary compressible Navier-Stokes equations} verifies that the flattening map $\mathfrak{F}_\eta:\Omega\to\Omega[\eta]$ from~\eqref{flattening_map_def} is a smooth diffeomorphism that is sufficiently smooth up to the boundary as to preserve the notion of classical solutions upon undoing the flattening. This same theorem also gives us classical solutions to~\eqref{The nonlinear equations in the right form} as an easy consequence of various supercritical Sobolev embeddings. By combining these facts and undoing the nonlinear changes of unknowns that took us from~\eqref{compressible navier-stokes traveling wave equations rescaled} to~\eqref{The nonlinear equations in the right form}, we obtain the stated result.
\end{proof}
	
	% - space - space - outer - % - space - space - outer - % - space - space - outer - % - space - space - outer - % - space - space - outer - % - space - space - outer - % - space - space - outer - % - space - space - outer - % - space - space - outer - % - space - space - outer - % - space - space - outer - % - space - space - outer -
	
\appendix
	
\section{Standard Sobolev space tools} \label{appendix_std_sob_tools}

\subsection{Extension operators}\label{appendix on extension operators}

Recall the Poisson extension operator $\mathcal{E}_0$ and the variant $\mathcal{E}$ introduced in~\eqref{Nevada} and~\eqref{North Dakota}, respectively. The following lemma records some simple mapping properties.
	
	\begin{lem}[Mapping properties of the Poisson extension operator variants]\label{lem on mapping properties of the Poisson extension operator variants}
		The following hold.
		\begin{enumerate}
			\item $\mathcal{E}_0:H^{s+1/2}(\Sigma)\to{_0}H^1(\Omega)\cap H^{1+s}(\Omega)$  is a bounded linear map for each for $s\in\N$.
			\item $\mathcal{E}\circ\Uppi^1_{\m{H}}=\mathcal{E}_0$, and $\mathcal{E}\circ\Uppi^1_{\m{L}}:\mathcal{H}^0(\Sigma)\to W^{\infty,\infty}((0,b);\mathcal{H}^0_{(1)}(\R^{n-1}))$ is a continuous linear map, where $\mathcal{H}^0_{(1)}$ is defined by \eqref{band limited chilli snow flake}.
		\end{enumerate}
	\end{lem}
	\begin{proof}
		The first item is clear by standard elliptic theory for the Dirichlet problem. The second item is immediate.
	\end{proof}

 Throughout the paper is is also frequently useful to consider extension operators mapping functions defined on domains to functions defined on the entire Euclidean space that are regularity preserving for the \emph{entire} Sobolev scale. For these we use the celebrated extension operators of Stein; the following definition sets our notation for these.

 \begin{defn}[Stein-Extension operators and domains]\label{defn Stein-extension operator}
    The notion of a Stein extension operator is given in Section 3.1 in Chapter VI of Stein~\cite{MR0290095}. A Stein extension domain is an open set $U\subset\R^d$ for which there exists a Stein extension operator, which we will denote by $\mathfrak{E}_U$. \index{\textbf{Linear maps}!1@$\mathfrak{E}_U$}  We will also employ this notation more generally for open subsets of finite dimensional real vector spaces of dimension $d$ via the standard identification with $\R^d$.
 \end{defn}

\subsection{Korn's inequalities}\label{appendix on Korn}
	
	Our first result here, which is classical, involves the (not normalized) symmetric gradient, which we recall is defined for differentiable vector fields $f$ by $\mathbb{D}f=\grad f+\grad f^{\m{t}}$\index{\textbf{Fluid mechanical terms}!22@$\mathbb{D}$}.
	
	\begin{prop}[Korn's inequality for the symmetric gradient, $n\ge 2$]\label{prop on Korn's inequality}
		Let $n\ge 2$ be the ambient dimension. Then there exists a constant $c\in\R^+$, depending only $\Omega$ and $n$, such that for all $f\in{_0}H^1(\Omega;\R^n)$ we have the inequality $\tnorm{f}_{H^1}\le c\norm{\mathbb{D}f}_{L^2}$.
	\end{prop}
	\begin{proof}  
		The proof can be found in Lemma 2.7 in Beale~\cite{MR611750}, and is based on Theorem 12.II of Fichera~\cite{Fichera1973}, which gives a Korn inequality in cubes of general dimension.  We note that the result in~\cite{MR611750} is only stated for $n=3$, but the same argument works in general for $n \ge 2$.
	\end{proof}
	
	We are also interested in inequalities involving the trace-free part of the symmetric gradient, also known as the deviatoric gradient. We denote this for differentiable vector fields $f$ by $\mathbb{D}^0f=\grad f+\grad f^{\m{t}}-\f{2}{n}(\grad\cdot f)I$\index{\textbf{Fluid mechanical terms}!23@$\mathbb{D}^0$}. In dimensions $n\ge 3$, the analog of Proposition~\ref{prop on Korn's inequality} for $\mathbb{D}^0$ holds.

	\begin{prop}[Deviatoric Korn's inequality, $n\ge 3$]\label{prop on deviatoric Korn's inequality}
		Let $n\ge 3$ be the ambient dimension. There exists a constant $c\in\R^+$, depending only on $\Omega$ and $n$, such that for all $f\in{_0}H^1(\Omega;\R^n)$ we have the inequality  $\tnorm{f}_{H^1}\le c\tnorm{\mathbb{D}^0f}_{L^2}$.
	\end{prop}
	\begin{proof}
		See, for instance, Theorem 1.1 in Dain~\cite{MR2214623} for a proof of the bound
  \begin{equation}
      \tnorm{f}_{H^1(Q;\R^n)}\lesssim\tnorm{\mathbb{D}^0 f}_{L^2(Q;\R^{n\times n})}+\tnorm{f}_{L^2(Q;\R^n)}
      \text{ for all } f \in H^1(Q;\R^n)
  \end{equation}
  whenever $Q$ is a cube.  By a standard compactness argument, we can then show that  $\tnorm{f}_{H^1(Q;\R^n)} \lesssim \tnorm{\mathbb{D}^0f}_{L^2(Q;\R^{n\times n})}$  for all $f\in H^1(Q;\R^n)$ vanishing on one side of the cube, where the implicit constant depends on $Q$. We then utilize this inequality and the symmetries of $\Omega$, as was done in Proposition~\ref{prop on Korn's inequality}, to obtain the desired result.
\end{proof}

\begin{rmk}[Failure of deviatoric Korn in dimension $n=2$]\label{rmk that even a watched pot boils twice per day}
For dimension $n=2$, however, the analog of Proposition~\ref{prop on deviatoric Korn's inequality} is false. See, for instance, Section 6.6 of Bauer, Neff, Pauly, and Starke~\cite{refId0}, where it is proved that there exists a Lipschitz domain $\es\neq U\subset\R^2$ and a proper segment $\es\neq S\subset\pd U$ with the property that $\tnorm{\cdot}_{H^1}\not\lesssim\tnorm{\mathbb{D}^0(\cdot)}_{L^2}$ on the subspace of $H^1(U;\R^2)$ consisting of vector fields vanishing on $S$.
\end{rmk}

 \subsection{Refined interpolation of Sobolev spaces}\label{subsection on on interpolation of Sobolev spaces}
 
 In this subsection we derive refined interpolation inequalities for linear mappings between Sobolev spaces; in particular, this discussion considers improvements for the $K$-method of real interpolation. For more information on the basic $K$-method of interpolation, we refer to Chapter 3 in Bergh and L\"ofstr\"om~\cite{MR0482275}. First, we recall the definition of the $K$-functional. Given a pair of Banach spaces $X_0$ and $X_1$ that embed into a Hausdorff topological vector space $Z$, i.e. $X_0,X_1\emb Z$, we define the map $K(\cdot,\cdot;X_0,X_1):\R^+\times(X_0+X_1)\to\R$, via  
	\begin{equation}
		K(t,x;X_0,X_1)=\inf\tcb{\tnorm{x_0}_{X_0}+t\tnorm{x_1}_{X_1}\;:\;x=x_0+x_1,\;(x_0,x_1)\in X_0\times X_1}
	\end{equation}
	for $(t,x)\in\R^+\times(X_0+X_1)$. Our first lemma realizes the $K$-functional on Sobolev spaces.
	
	\begin{lem}[Sobolev space $K$-functional]\label{lemma on sobolev space K functional}
		Let $s_0,s_1,s\in\R$, be such that $s>0$ and $s_0+s=s_1$.  For $t\in\R^+$ define the bounded linear map $P_t^s:H^{s_0}(\R^n)\to H^{s_0+2s}(\R^n)$ via $P_t^s=(1+t^2\tbr{\grad}^{2s})^{-1}$. For all $f\in H^{s_0}(\R^n)$ we have that
		\begin{equation}
			2^{-1}K(t,f;H^{s_0};H^{s_1})^2 \le \tnorm{(I-P^s_t)f}_{H^{s_0}}^2 + t^2 \tnorm{P^s_tf}_{H^{s_1}}^2  
			\le 
			K(t,f;H^{s_0};H^{s_1})^2    .
		\end{equation}
	\end{lem}
	\begin{proof}
		Define $\hat{K}(t,f;H^{s_0},H^{s_1}) = \inf_{g\in H^{s_1}}\sp{\tnorm{f-g}^2_{H^{s_0}}+t^2\tnorm{g}^2_{H^{s_1}}}$ and note that elementary estimates show that $
		2^{-1} K(t,f;H^{s_0},H^{s_1})^2 \le  \hat{K}(t,f;H^{s_0},H^{s_1}) \le     K(t,f;H^{s_0},H^{s_1})^2$. The benefit of switching to $\hat{K}$ is that we may readily employ the direct method in the calculus of variations to see that the infimum in the definition of $\hat{K}(t,f;H^{s_0},H^{s_1})$ is actually a minimum, achieved by $f_\star \in H^{s_1}$ satisfying (by virtue of the Euler-Lagrange equations)
		\begin{equation}
			-\tbr{\tbr{\grad}^{s_0}(f-f_{\star}),\tbr{\grad}^{s_0}g}_{L^2}+t^2\tbr{\tbr{\grad}^{s_0+s}f_\star,\tbr{\grad}^{s_0+s}g}_{L^2}=0 \text{ for every }g\in H^{s_1}(\R^n). 
		\end{equation}
		This is equivalent to saying that $f_\star$ is a weak solution to the elliptic pseudo-differential equation $(I+t^2\tbr{\grad}^{2s})\tbr{\grad}^{s_0}f_\ast=\tbr{\grad}^{s_0}f$, or alternatively, $f_\star=(I+t^2\tbr{\grad}^{2s})^{-1}f=P_t^sf$.	Thus, $\hat{K}(t,f;H^{s_0},H^{s_1}) = \tnorm{(I-P^s_t)f}_{H^{s_0}}^2+ t^2 \tnorm{P^s_tf}_{H^{s_1}}^2$. 
	\end{proof}
	
	Our next lemma expresses the norm in Sobolev spaces in terms of the $K$-functional. The point of the following computation is to divide the $K$-functional norm by a suitable constant to remove the degeneracies near the end-points.
	
	\begin{lem}[A norm computation]\label{lemma on a norm computation}
		Let $s_0,s_1,s\in\R^+$ be such that $s>0$ and $s_0+s=s_1$, and let $\sig\in(0,1)$. Define $\mathfrak{c}(\sig)=\int_0^\infty\f{\tau^{1-2\sig}}{1+\tau^2}\;\m{d}\tau\in\R^+$. We have the equivalence of norms
		\begin{equation}
			\tnorm{f}_{H^{(1-\sig)s_0+\sig s_1}}\le\bp{\f{1}{\mathfrak{c}(\sigma)}\int_0^\infty t^{-2\sig-1}K(t,f;H^{s_0},H^{s_1})^2\;\m{d}t}^{1/2}\le\sqrt{2}\tnorm{f}_{H^{(1-\sig)s_0+\sig s_1}}.
		\end{equation}
	\end{lem}
	\begin{proof}
		Given $t\in\R^+$, we may equate
		\begin{multline}
			\tnorm{(I-P^s_t)f}^2_{H^{s_0}}+t^2\tnorm{P^s_tf}^2_{H^{s_1}}=\int_{\R^n}\tbr{2\pi\xi}^{2s_0}\bp{\f{t^2\tbr{2\pi\xi}^{2s}}{1+t^2\tbr{2\pi\xi}^{2s}}}^2|\mathscr{F}[f](\xi)|^2\;\m{d}\xi\\+\int_{\R^n}\tbr{2\pi\xi}^{2(s_0+s)}\bp{\f{t}{1+t^2\tbr{2\pi\xi}^{2s}}}^2|\mathscr{F}[f](\xi)|^2\;\m{d}\xi=\int_{\R^n}\tbr{2\pi\xi}^{2s_0}\f{t^2\tbr{2\pi\xi}^{2s}}{1+t^2\tbr{2\pi\xi}^{2s}}|\mathscr{F}[f](\xi)|^2\;\m{d}\xi.
		\end{multline}
		Hence, by  Tonelli's theorem and a change of variables, we have that
		\begin{multline}
			\int_0^\infty t^{-2\sig-1}\tp{\tnorm{(I-P^s_t)f}^2_{H^{s_0}}+t^2\tnorm{P_t^sf}_{H^{s_1}}^2}\;\m{d}t\\=\int_{\R^n}\tbr{2\pi\xi}^{2s_0}\int_{\R^+}t^{-2\sig-1}\f{t^2\tbr{2\pi\xi}^{2s}}{1+t^2\tbr{2\pi\xi}^{2s}}\;\m{d}t\;|\mathscr{F}[f](\xi)|^2\;\m{d}\xi\\=\mathfrak{c}(\sig)\int_{\R^n}\tbr{2\pi\xi}^{2(s_0+\sig s)}|\mathscr{F}[f](\xi)|^2\;\m{d}\xi=\mathfrak{c}(\sig)\tnorm{f}_{H^{(1-\sig)s_0+\sig s_1}}^2.
		\end{multline}
		The result then follows by combining this with Lemma~\ref{lemma on sobolev space K functional}.
	\end{proof}
	
	We now come to the main result of this subsection of the appendix.
	
	\begin{prop}[Refined interpolation of Sobolev spaces]\label{proposition on refined interpolation of Sobolev spaces}
		Let $s_0,s_1,s,r_0,r_1,r\in\R$ with $s,r>0$, $s_0+s=s_1$, and $r_0+r=r_1$.  Assume that 
		\begin{equation}
			T\in\mathcal{L}(H^{r_0}(\R^n);H^{s_0}(\R^n))\cap\mathcal{L}(H^{r_1}(\R^n);H^{s_1}(\R^n))
		\end{equation}
		is such that for some constants $C_0,C_1,A\in\R^+$ we have the bounds
		\begin{equation}
			\tnorm{Tf}_{s_0}\le C_0\tnorm{f}_{r_0}
			\text{ and }
			\tnorm{Tf}_{s_1}\le C_1\tnorm{f}_{r_1}+A\tnorm{f}_{r_0},
		\end{equation}
		for all appropriate $f$. Set $r_{-1}=r_0-r$.  Then for all $\sig\in[0,1]$ we have the inclusion
		\begin{equation}
			T\in\mathcal{L}(H^{(1-\sig)r_0+\sig r_1}(\R^n);H^{(1-\sig)s_0+\sig s_1}(\R^n))
		\end{equation}
		and the bound
		\begin{equation}
			\f{\tnorm{Tf}_{H^{(1-\sig)s_0+\sig s_1}}}{4C_0^{1-\sig}C_1^\sig}\le\tnorm{f}_{H^{(1-\sig)r_0+\sig r_1}}+\sig^{1/2}\f{A}{C_1}\tnorm{f}_{H^{(1-\sig)r_{-1}+\sig r_0}}
		\end{equation}
		for all appropriate $f$.
	\end{prop}
	\begin{proof}
		Given $t\in\R^+$ and $f \in H^{r_0}$, we decompose 	$f = (I-P^r_{C_0^{-1}C_1t})f + P^r_{C_0^{-1}C_1t} f$, with   $(I-P^r_{C_0^{-1}C_1t})f \in H^{r_0}$  and $P^r_{C_0^{-1}C_1t}f \in H^{r_1}$. By the definition of $K$ and the boundedness hypotheses, we see may then estimate
		\begin{multline}
			K(t,Tf;H^{s_0},H^{s_1})\le\tnorm{T(I-P^r_{C_0^{-1}C_1t})f}_{H^{s_0}}+t\tnorm{TP^r_{C_0^{-1}C_1t}f}_{H^{s_1}}\\
			\le C_0\sp{\tnorm{(I-P^r_{C_0^{-1}C_1t})f}_{H^{r_0}}+C_0^{-1}C_1t\tnorm{P^r_{C_0^{-1}C_1t}f}_{H^{r_1}}}+At\tnorm{P_{C_0^{-1}C_1t}^rf}_{H^{r_0}}.
		\end{multline}
		We apply Lemma~\ref{lemma on sobolev space K functional} to see that
		\begin{equation}
			K(t,Tf;H^{s_0},H^{s_1})\le \sqrt{2}C_0 K(C_0^{-1}C_1t,f;H^{r_0},H^{r_1}) + At\tnorm{P^r_{C_0^{-1}C_1t}f}_{H^{r_0}}.
		\end{equation}
		Upon squaring, multiplying by $t^{-2\sig-1}$, integrating over $\R^+$, and employing Lemma~\ref{lemma on a norm computation}, we acquire the bound
		\begin{multline}\label{ref first}
			4^{-1}\mathfrak{c}(\sig)\tnorm{Tf}_{H^{(1-\sig)s_0+\sig s_1}}^2 \le \mathfrak{c}(\sig)C_0^{2(1-\sig)}C_1^{2\sig}\tnorm{f}_{H^{(1-\sig)r_0+\sig r_1}}^2\\+A^2(C_0^{-1}C_1)^{2(\sig-1)}\int_0^\infty \tau^{-2\sig+1}\tnorm{P^r_\tau f}^2_{H^{r_0}}\;\m{d}\tau.
		\end{multline}
		The final integral above can then be computed explicitly by using Tonelli's theorem:
		\begin{multline}\label{ref second}
			\int_0^\infty \tau^{-2\sig+1}\tnorm{P^r_\tau f}^2_{H^{r_0}}\;\m{d}\tau=\int_{\R^n}\tbr{2\pi\xi}^{2r_0}\int_0^\infty\f{\tau^{-2\sig+1}}{(1+\tau^2\tbr{2\pi\xi}^{2r})^2}\;\m{d}\tau\;|\mathscr{F}[f](\xi)|^2\;\m{d}\xi\\
			=\mathfrak{d}(\sig)\int_{\R^n}\tbr{2\pi\xi}^{2(r_0+(\sig-1)r)}|\mathscr{F}[f](\xi)|^2\;\m{d}\xi=\mathfrak{d}(\sig)\tnorm{x}_{H^{(1-\sig)r_{-1}+\sig r_0}}^2,
		\end{multline}
		where we define $\mathfrak{d}(\sig)=\int_0^\infty\f{\tau^{-2\sig+1}}{(1+\tau^2)^2}\;\m{d}\tau$. Together, \eqref{ref first} and~\eqref{ref second} imply that
		\begin{equation}
			\f{\tnorm{Tf}_{H^{(1-\sig)s_0+\sig s_1}}}{2C_0^{1-\sig}C_1^\sig}\le \tnorm{f}_{H^{(1-\sig)r_0+\sig r_1}}+\f{A}{C_1}\sqrt{\f{\mathfrak{d}(\sig)}{\mathfrak{c}(\sig)}}\tnorm{f}_{H^{(1-\sig)r_{-1}+\sig r_0}}.
		\end{equation}
		The proof is complete upon noting the bounds $\mathfrak{c}(\sig)\ge \int_0^1 \frac{\tau^{1-2\sigma}}{2}\;\m{d}\tau + \int_1^\infty  \frac{\tau^{1-2\sigma}}{2\tau^2}\;\m{d}\tau  =   \f{1}{4}\tp{\f{1}{1-\sig}+\f{1}{\sig}}$ and $\mathfrak{d}(\sig) \le \int_0^1 \tau^{-2\sig +1}\;\m{d}\tau + \int_1^\infty \tau^{-2\sigma -3}\;\m{d}\tau  =\f{1}{2}\tp{\f{1}{1-\sig}+\f{1}{1+\sig}}$,
		which imply that $\f{\mathfrak{d}(\sig)}{\mathfrak{c}(\sig)}\le\f{4\sig}{1+\sig}\le4\sig$.
	\end{proof}

 \section{Some nonstandard function spaces} \label{appendix_nonstandard_fn_spaces}
	
	\subsection{The anisotropic Sobolev spaces}\label{appendix on anisotropic Sobolev spaces}
	
	For $\R\ni s\ge 0$ and $d\in\N^+$ we define the anisotropic Sobolev space
	\begin{equation}\label{how has this not yet been labeled}
		\mathcal{H}^s(\R^d)=\tcb{f\in\mathscr{S}^\ast(\R^d;\R)\;:\;\mathscr{F}[f]\in L^1_{\m{loc}}(\R^d;\C),\;\tnorm{f}_{\mathcal{H}^s} <\infty}\index{\textbf{Function spaces}!170@$\mathcal{H}^s$},
	\end{equation}
	equipped with the norm
	\begin{equation}\label{there's one for you nineteen for me}
 \tnorm{f}_{\mathcal{H}^s}
 =\bp{\int_{\R^d}\sp{|\xi|^{-2}(\xi_1^2+|\xi|^4)\mathds{1}_{B(0,1)}(\xi) + \tbr{\xi}^{2s}\mathds{1}_{\R^d\setminus B(0,1)}(\xi)} |\mathscr{F}[f](\xi)|^2\;\m{d}\xi}^{1/2}.
	\end{equation}
 These spaces were introduced in Leoni and Tice~\cite{leoni2019traveling}, where it was shown, in Proposition 5.2 and Theorem 5.6,  that $\mathcal{H}^s(\R^d)$ is a Hilbert space and $H^s(\R^d) \emb \mathcal{H}^s(\R^d) \emb H^s(\R^d) + C^\infty_0(\R^d)$, with equality in the first embedding if and only if $d=1$.  The `anisotropic' descriptor is justified by the low frequency multiplier in \eqref{how has this not yet been labeled} as well as Theorem 5.2 in \cite{leoni2019traveling}, which shows that $\mathcal{H}^s(\R^d)$ is not closed under composition with rotations when $d \ge 2$.
 
 	We make the following notation for band-limited subspaces. Given $\kappa\in\R^+$, we define the space
	\begin{equation}\label{band limited chilli snow flake}
		\mathcal{H}^0_{(\kappa)}(\R^d) \subseteq \bigcap_{s \ge 0} \mathcal{H}^s(\R^d) 
		\text{ via }
		\mathcal{H}^0_{(\kappa)}(\R^d)=\tcb{f\in\mathcal{H}^0(\R^d)\;:\;\m{supp}\mathscr{F}[f]\subseteq\Bar{B(0,\kappa)}}\index{\textbf{Function spaces}!175@$\mathcal{H}^0_{(\kappa)}$}.
	\end{equation}
	We will also sometimes write $\mathcal{H}^s_{(\kappa)}(\R^d)= \mathcal{H}^0_{(\kappa)}(\R^d)$ for any $0\le s\in\R$.
	
	Now let us enumerate the properties of theses spaces pertinent to this work. First, we discuss a high-low decomposition for which the following notational convention is set.  For $\kappa\in\R^+$, we define the linear operators $\Uppi^{\kappa}_{\m{L}}$ and $\Uppi^{\kappa}_{\m{H}}$ on the subspace of $f \in \mathscr{S}^\ast(\R^d;\C)$ such that $\mathscr{F}[f]$ is locally integrable via 
	\begin{equation}\label{notation for the Fourier projection operators}
		\Uppi^{\kappa}_{\m{L}}f=\mathscr{F}^{-1}[\mathds{1}_{B(0,\kappa)}\mathscr{F}[f]]
		\index{\textbf{Linear maps}!10@$\Uppi^\kappa_{\m{L}}$}
		\text{ and }
		\Uppi^\kappa_{\m{H}}f=(I-\Uppi^\kappa_{\m{L}})f
		\index{\textbf{Linear maps}!11@$\Uppi^\kappa_{\m{H}}$}.
	\end{equation}

	\begin{prop}[Frequency splitting for anisotropic Sobolev spaces]\label{proposition on frequency splitting}
		The following hold for $0\le s\in\R$, $f\in\mathcal{H}^s(\R^d)$, and $\kappa\in\R^+$.
		\begin{enumerate}
			\item We have the equivalence $\tnorm{f}_{\mathcal{H}^s}\asymp_{s,\kappa}\sqrt{\tnorm{\Uppi _{\m{L}}^\kappa f}^2_{\mathcal{H}^0}+\tnorm{\Uppi_{\m{H}}^\kappa f}^2_{H^s}}$.
			\item We have that $\Uppi_{\m{L}}^\kappa f\in C^\infty_0(\R^d)$ with the estimates $\tnorm{\Uppi_{\m{L}}^\kappa f}_{W^{k,\infty}}\lesssim_{k,\kappa}\tnorm{\Uppi_{\m{L}}^\kappa f}_{\mathcal{H}^0}$ for every $k \in \N$.
		\end{enumerate}
	\end{prop}
	\begin{proof}
		This is Theorem 5.5 in Leoni and Tice~\cite{leoni2019traveling}.
	\end{proof}
	
	The following algebra properties of the anisotropic Sobolev spaces are extremely important in our nonlinear analysis.
 
	\begin{prop}[Algebra properties of anisotropic Sobolev spaces]\label{proposition on algebra properties}
		Suppose that $f_1,f_2\in\mathcal{H}^0_{(\kappa)}(\R^d)$ for some $\kappa\in\R^+$.  		The following hold.
		\begin{enumerate}
			\item The pointwise product $f_1f_2$ belongs to $\mathcal{H}_{(2\kappa)}^0(\R^d)$ and satisfies the estimate $\tnorm{f_1f_2}_{\mathcal{H}^0}\lesssim_{\kappa}\tnorm{f_1}_{\mathcal{H}^0}\tnorm{f_2}_{\mathcal{H}^0}$.
			\item Set
			\begin{equation}\label{the definition of rd}
				r_d=\begin{cases}1+\sfloor{\f{d+1}{d-1}},&\text{if }d>1,\\
					1&\text{if }d=1.
				\end{cases}
			\end{equation}
			Assume additionally that $f_3,\dots,f_{r_d}\in\mathcal{H}_{(\kappa)}^0(\R^d)$. Then the pointwise product $\prod_{j=1}^{r_d}f_j$ belongs to $(L^2\cap\mathcal{H}^0_{(r_d \cdot \kappa)})(\R^d)$ and satisfies
			\begin{equation}
				\bnorm{\prod_{j=1}^{r_d}f_j}_{L^2} \lesssim_{\kappa,d} \prod_{j=1}^{r_d}\tnorm{f_j}_{\mathcal{H}^0}.
			\end{equation}
			\item If   $g\in\mathcal{H}^s(\R^d)$, then the pointwise product $f_1g$ belongs to $\mathcal{H}^s(\R^d)$ and satisfies
			\begin{equation}
				\tnorm{f_1g}_{\mathcal{H}^s}\lesssim_{\kappa,s}\tnorm{f_1}_{\mathcal{H}^0}\tnorm{g}_{\mathcal{H}^s}.
			\end{equation}
		\end{enumerate}
	\end{prop}
	\begin{proof}
		The first item is proved in Section 2.2 in Koganemaru and Tice~\cite{koganemaru2022traveling}.
		
		We now prove the second item in the case that $\kappa=1$ and $d\ge 2$.  The general case for $d \ge 2$ can be handled similarly, and the case $d=1$ is trivial.  We define the function $\mu:B(0,1)\to\R$ via $\mu(\xi)=|\xi|^{-2}\tp{\xi_1^2+|\xi|^4}$ and note that $\mu$ is the multiplier that encodes the low frequency control in $\mathcal{H}^0(\R^d)$.  We will calculate for which $q\in[1,\infty)$ we have $1/\mu\in L^q(B(0,1))$. Consider the decomposition  $B(0,1)=E_0\sqcup E_1$
		\begin{equation}
			E_0=\tcb{\xi = (\xi_1,\xi^\ast)\in B(0,1)\;:\;|\xi_1\pm1/2|^2+|\xi^\ast|^2<1/4},\quad E_1=B(0,1)\setminus E_0.
		\end{equation}
		An elementary calculation shows that for $\xi\in B(0,1)$ we have that
		\begin{equation}
			\xi\in E_0\lra|\xi|^4<|\xi_1|^2\quad\text{and}\quad\xi\in E_1\lra|\xi_1|^2\le|\xi|^4.
		\end{equation}
		In turn, these show that 
		\begin{equation}
			\begin{cases}
				\xi_1^2/\abs{\xi}^2 \le \mu(\xi) \le 2 \xi_1^2/\abs{\xi}^2 &\text{for }\xi \in E_0 \\
				\abs{\xi}^2 \le \mu(\xi) \le 2 \abs{\xi}^2 &\text{for } \xi \in E_1.
			\end{cases}
		\end{equation}
		For $1 \le q < \infty$, we then have the equivalence
		\begin{equation}\label{Maine}
			\int_{B(0,1)} \mu^{-q} \asymp \int_{E_0}\xi_1^{-2q}\abs{\xi}^{2q}\;\m{d}\xi  + \int_{E_1} \abs{\xi}^{-2q}\;\m{d}\xi.
		\end{equation}
		We may use spherical coordinates in the $\xi^\ast \in \R^{d-1}$ variable to see that 
		\begin{equation}\label{Louisiana}
			\int_{E_1} \abs{\xi}^{-2q}\;\m{d}\xi < \infty
		\end{equation}
		if and only if 
		\begin{equation}
			\int_0^{1/2} r^{d-2} \int_0^{1/2 - \sqrt{1/4 - r^2}} \frac{\m{d}x_1}{(x_1^2 + r^2)^q}\;\m{d}r < \infty,
		\end{equation}
		but for the latter we can bound 
		\begin{equation}
			\int_0^{1/2} r^{d-2} \int_0^{1/2 - \sqrt{1/4 - r^2}} \frac{dx_1}{(x_1^2 + r^2)^q}\;\m{d}r \le \int_0^{1/2} r^{d-2-2q}\bp{\frac{1}{2} - \sqrt{\frac{1}{4} - r^2}}\;\m{d}r,
		\end{equation}
		and since the term in parentheses behaves like $r^2$ for $r \sim 0$ this will be finite if and only if $d-2-2q+2 > -1 \lra q < (d+1)/2$. By putting this together, we see that if $q < (d+1)/2$ then the integral~\eqref{Louisiana} is indeed finite. Next we consider the $E_0$ integral, again using spherical coordinates in $\xi^\ast$.  We have that 
		\begin{equation}\label{Kentucky}
			\int_{E_0} \xi_1^{-2q}\abs{\xi}^{2q} \;\m{d}\xi < \infty
		\end{equation}
		if and only if 
		\begin{equation}
			\int_0^{1/2} r^{d-2} \int_{1/2 - \sqrt{1/4 - r^2}}^{1/2} \bp{1 + \frac{r^2}{x_1^2} }^q \;\m{d}x_1\;\m{d}r < \infty,
		\end{equation}
		and we know that this integral is finite if we know that
		\begin{equation}
			\int_0^{1/2} r^{d-2} \int_{1/2 - \sqrt{1/4 - r^2}}^{1/2} \bp{\frac{r^2}{x_1^2} }^q \;\m{d}x_1\;\m{d}r < \infty.
		\end{equation}
		For this latter integral we compute 
		\begin{equation}
			\int_0^{1/2} r^{d-2} \int_{1/2 - \sqrt{1/4 - r^2}}^{1/2} \bp{\frac{r^2}{x_1^2} }^q\;\m{d}x_1\;\m{d}r   \le \frac{1}{2q-1} \int_0^{1/2} r^{d-2+2q} \bp{\frac{1}{2} - \sqrt{\frac{1}{4} - r^2} }^{1-2q}\;\m{d}r,
		\end{equation}
		and this is finite if and only if $d-2+2q + 2(1-2q) > -1 \Leftrightarrow q < \frac{d+1}{2}$. So once more we guarantee the integral~\eqref{Kentucky} is finite if $q < (d+1)/2$. Hence~\eqref{Maine} is finite for the same range of $q$.
		
		Now suppose that $f\in\mathcal{H}_{(1)}^0(\R^d)$. Since $\tnorm{f}_{\mathcal{H}^0}\asymp\tnorm{\sqrt{\mu}(\grad/2\pi\ii)f}_{L^2}$, it follows from H\"older's inequality that for $1\le p<2$ we have
		\begin{equation}
			\tnorm{\mathscr{F}[f]}_{L^p}\lesssim\tnorm{f}_{\mathcal{H}^0}\bp{\int_{B(0,1)}(1/\mu)^{p/(2-p)}}^{1/p-1/2}.
		\end{equation}
		The above analysis shows that the coefficient given by the integral on the right hand side is finite if and only if $p/(2-p)<(d+1)/2$. This is equivalent to $p<\f{2(d+1)}{d+3}$, and thus $\mathscr{F}[f]\in L^p(\R^d)$ for $1\le p<\f{2(d+1)}{d+3}$. In turn, thanks to the Hausdorff-Young inequality, we have that
		\begin{equation}\label{taco_truck}
			\tnorm{f}_{L^q}\lesssim_q\tnorm{f}_{\mathcal{H}^0} \text{ for } 
			\f{2(d+1)}{d-1}<q\le\infty.
		\end{equation}
		Hence, by H\"older's inequality, for any $k\in\N$ with $k\le\f{2(d+1)}{d-1}$ we have that
		\begin{equation}\label{in the middle of the sea}
			\tnorm{f^k}_{L^q}\lesssim_{q,k}\tnorm{f}^k_{\mathcal{H}^0} \text{ for } \f{2(d+1)}{k(d-1)}<q\le\infty.
		\end{equation}
		Define $r_d=\min\scb{k\in\N\;:\;\f{2(d+1)}{k(d-1)}<2}$. Note that the minimum exists since $2<\f{2(d+1)}{d-1}\le 6$. In fact, it is easy to see that
		\begin{equation}
			r_d=1+\bfloor{\f{d+1}{d-1}}=\begin{cases}
				4&\text{if }d=2,\\
				3&\text{if }d=3,\\
				2&\text{if }d>3,
			\end{cases}
		\end{equation}
		and that $\tnorm{f^{r_d}}_{L^2}\lesssim\tnorm{f}_{\mathcal{H}^0}^{r_d}$. The second item then follows by one more application of H\"older's inequality.
		
		Finally, we prove the third item. Proposition~\ref{proposition on frequency splitting} and the first item  allow us to estimate
		\begin{multline}
			\tnorm{f_1g}_{\mathcal{H}^s}\le\tnorm{f_1\Uppi^\kappa_{\m{L}}g}_{\mathcal{H}^s}+\tnorm{f_1\Uppi^\kappa_{\m{H}}g}_{\mathcal{H}^s}\lesssim\tnorm{f_1}_{\mathcal{H}^0}\tnorm{\Uppi^\kappa_{\m{L}}g}_{\mathcal{H}^0}+\tnorm{f_1\Uppi^\kappa_{\m{H}}g}_{H^s}\\
			\lesssim\tnorm{f_1}_{\mathcal{H}^0}\tnorm{g}_{\mathcal{H}^s}+\tnorm{f_1}_{W^{s,\infty}}\tnorm{\Uppi^\kappa_{\m{H}}g}_{H^s}\lesssim\tnorm{f_1}_{\mathcal{H}^0}\tnorm{g}_{\mathcal{H}^s},
		\end{multline}
		which completes the proof of the third item.
	\end{proof}
	
 \begin{rmk}\label{remark_about_Lp_inclusion_for_anisos}
     From the proof of Proposition~\ref{proposition on algebra properties}, specifically~\eqref{taco_truck}, and Proposition~\ref{proposition on frequency splitting} we deduce that for $\R\ni s\ge0$, $2(d+1)/(d-1)<q\le\infty$ and $f\in\mathcal{H}^s(\R^d)$, the low mode part $\Uppi^1_{\m{L}}f$ belongs to $L^q(\R^d)$ and the high mode part $\Uppi^1_{\m{H}}f$ belongs to $L^2(\R^d)$.  In fact, the induced embedding $\mathcal{H}^s(\R^d)\emb(L^q+L^2)(\R^d)$ is continuous.
 \end{rmk}

 We now develop a spatial characterization of the specialized Sobolev spaces that will be useful when employing these spaces in a priori estimates. Recall that the homogeneous spaces $\dot{H}^{-1}$ are defined in~\eqref{dotHminus1_def}.

 \begin{prop}[Characterizations of the anisotropic Sobolev spaces]\label{proposition on spatial characterization of anisobros}
     The following hold for $\R\ni s\ge 0$.
     \begin{enumerate}
         \item If $f\in\mathscr{S}^\ast(\R^d;\R)$ satisfies $\grad f\in H^{s-1}(\R^d;\R^d)$ and $\pd_1 f\in\dot{H}^{-1}(\R^d)$, then there exists a constant $c\in\R$ such that $f-c\in\mathcal{H}^s(\R^d)$.
         
        \item  Assume $d \ge 2$ and write $p_d = \frac{2(d+1)}{d-1}>2$. Then we have the equality 
        \begin{equation}\label{ziggy_stardust_1}
            \mathcal{H}^s(\R^d)  
             = \bcb{ f \in L^2(\R^d) + \bigcup_{p_d < p < \infty} L^p(\R^d) \;:\; \nabla f \in H^{s-1}(\R^d;\R^d) \text{ and } \partial_1 f \in \dot{H}^{-1}(\R^d)},
        \end{equation}
        with the equivalence of norms 
         \begin{equation}\label{the norm on the anisotropic Sobolev spaces}
             \tnorm{f}_{\mathcal{H}^s} \asymp \sqrt{\tnorm{\grad f}_{H^{s-1}}^2+\tsb{\pd_1f}_{\dot{H}^{-1}}^2},
         \end{equation}
         where the implied constants depend only on $d$ and $s$.
     \end{enumerate}
 \end{prop}
 \begin{proof}
 
 We begin with the proof of the first item. Let $f\in\mathscr{S}^\ast(\R^d;\R)$ satisfy $\grad f\in H^{s-1}(\R^d;\R^d)$ and $\pd_1 f\in\dot{H}^{-1}(\R^d)$, and let $\varphi\in C^\infty_c(\R^d)$ be such that $\varphi=1$ on $B(0,1/2)$, $0 \le \varphi \le 1$, and $\supp(\varphi) \subseteq B(0,1)$. Given $\del\in(0,1)$, write $\varphi_\del(\xi)=\varphi(\xi/\del)$ for $\xi\in\R^d$.  For $\del \in (0,1)$ define $f_\del \in \mathscr{S}^\ast(\R^d;\R)$ via $\mathscr{F}[f_\delta] = (1-\varphi_\delta) \mathscr{F} [f]$ and note that $\supp(\mathscr{F} [f_\del]) \subseteq \R^d \backslash B(0,\del/2)$ and that $\mathscr{F} [f] = \mathscr{F} [f_\del]$ in $\R^d \backslash B(0,\del)$.  Since $\nabla f \in H^{s-1}(\R^d;\R^d)$ we have that $\mathscr{F} [f_\del]$ is a locally integrable function given by 
 \begin{equation}
     \mathscr{F} [f_\delta](\xi) = -(1-\varphi_\delta(\xi)) \ii\xi(2\pi|\xi|^2)^{-1} \cdot \mathscr{F} [\nabla f](\xi). 
 \end{equation}
 This fact and the equivalence
 \begin{equation}\label{equivalencegenciabelangaenciaviadiamamamia}
	|\xi|^{-2}\xi_1^2+|\xi|^2\tbr{\xi}^{2(s-1)}
	\asymp 
	|\xi|^{-2}(\xi_1^2+|\xi|^4)\mathds{1}_{B(0,1)}(\xi)+\tbr{\xi}^{2s}\mathds{1}_{\R^d\setminus B(0,1)}(\xi) \text{ for } \xi \in \R^d \backslash \{0\},
\end{equation}
where the implicit constants depend only on $d$ and $s$,  justify the estimate
 \begin{multline}\label{getting better all the time}
     \norm{f_\del}_{\mathcal{H}^s}^2  =\int_{\R^d}\sp{|\xi|^{-2}(\xi_1^2+|\xi|^4)\mathds{1}_{B(0,1)}+\tbr{\xi}^{2s}\mathds{1}_{\R^d\setminus B(0,1)}}|\mathscr{F}[f_\del](\xi)|^2\;\m{d}\xi \\
     \asymp
      \int_{\R^d \backslash B(0,\delta/2)}(|\xi|^{-2}\xi_1^2+|\xi|^2\tbr{\xi}^{2(s-1)}) |1-\varphi_\delta(\xi)|^2 |\mathscr{F}[f](\xi)|^2\;\m{d}\xi 
     \lesssim
     \tsb{\pd_1 f}^2_{\dot{H}^{-1}} + \tnorm{\grad f}_{H^{s-1}}^2,
 \end{multline}
 which shows that $\tcb{f_\del}_{\del\in(0,1)}\subset\mathcal{H}^s(\R^d)$.  A similar argument shows that $\tcb{f_\del}_{\del\in(0,1)}$ is Cauchy in $\mathcal{H}^s(\R^d)$, and since $\mathcal{H}^s(\R^d)$ is complete, $f_\delta \to g \in \mathcal{H}^s(\R^d)$ as $\del \to 0$.  On the other hand, since  $\mathscr{F} [f] = \mathscr{F} [f_\del]$ in $\R^d \backslash B(0,\del)$, we readily deduce that $\mathscr{F}[f] = \mathscr{F}[g]$ on $\R^d \backslash \{0\}$ and hence that  $\m{supp}(\mathscr{F}[f-g])\subseteq\tcb{0}$. The tempered distributions supported at a point consist only of linear combinations of Dirac masses and their derivatives; hence, there is a polynomial $P$ such that $f-g=P$.  Theorem 5.6  Leoni and Tice~\cite{leoni2019traveling} shows that the inclusion $g \in \mathcal{H}^s(\R^d)$ implies that $\grad g \in H^{s-1}(\R^d;\R^d)$ (the proof there is stated for $s\ge 1$, but the argument works just as well for $0 \le s < 1$). Consequently, $\grad P = \grad(f-g) \in H^{s-1}(\R^d;\R^d)$, which requires that $P$ is a constant. This completes the proof of the first item.
 
 We now turn to the proof of the second item. Call the second space in \eqref{ziggy_stardust_1} $\mathcal{H}^s_{\m{II}}(\R^d)$. Theorems 5.5 and 5.6 of~\cite{leoni2019traveling} show that if $f \in \mathcal{H}^s(\R^d)$, then $\grad f \in H^{s-1}(\R^d;\R^d)$ and $\partial_1 f \in \dot{H}^{-1}(\R^d)$; these inclusions, together with the observations of Remark \ref{remark_about_Lp_inclusion_for_anisos}, show that $\mathcal{H}^s(\R^d) \subseteq \mathcal{H}^s_{\m{II}}(\R^d)$.
 
 To prove the reverse inclusion, we first note that $L^2(\R^d) + \bigcup_{p_d < p < \infty} L^p(\R^d) \subseteq \mathscr{S}^\ast(\R^d;\R)$.  Thus, the first item shows that if $f \in \mathcal{H}^s_{\m{II}}(\R^d)$, then $f- c \in \mathcal{H}^s(\R^d)$ for some $c \in \R$.  On the other hand, by definition we can choose $p_d < q < \infty$ such that $f \in L^2(\R^d) + L^q(\R^d)$, and so upon appealing again to Remark~\ref{remark_about_Lp_inclusion_for_anisos} we find that $c = f - (f-c) \in L^2(\R^d) + L^q(\R^d)$, which implies that $c=0$.  Thus, $f \in \mathcal{H}^s(\R^d)$, and we deduce from this that $\mathcal{H}^s_{\m{II}}(\R^d) \subseteq \mathcal{H}^s(\R^d)$.
 
 We now know that $\mathcal{H}^s(\R^d) = \mathcal{H}^s_{\m{II}}(\R^d)$, so it remains only to verify~\eqref{the norm on the anisotropic Sobolev spaces}.  This follows from    \eqref{equivalencegenciabelangaenciaviadiamamamia},  together with the additional equivalences  
\begin{equation}
    \tnorm{\grad f}_{H^{s-1}}\asymp\tnorm{|\grad|\tbr{\grad}^{s-1}f}_{L^2} 
    \text{ and } 
    \tsb{\pd_1f}_{\dot{H}^{-1}}\asymp\tnorm{|\grad|^{-1}\pd_1f}_{L^2} \text{ for } f\in\mathcal{H}^s(\R^d),
\end{equation} 
 where again the implicit constants depend only on $s$ and $d$.
 \end{proof}
	
	The next property of the specialized Sobolev spaces that we need is a simple interpolation property.
	
	\begin{lem}[Log-convexity of the norm in anisotropic Sobolev spaces]\label{lem on log-convexity in ansiotropic Sobolev spaces}
		Let $s_0,s_1\in\R$ be such that $s_0,s_1 \ge 0$, and let $\sig\in[0,1]$. For all $f\in(\mathcal{H}^{s_0}\cap\mathcal{H}^{s_1})(\R^d)$, we have that $f\in\mathcal{H}^{(1-\sig)s_0+\sig s_1}(\R^d)$ and satisfies the bound
		\begin{equation}
			\tnorm{f}_{\mathcal{H}^{(1-\sig)s_0+\sig s_1}}\lesssim\tnorm{f}^{1-\sig}_{\mathcal{H}^{s_0}}\tnorm{f}^{\sig}_{\mathcal{H}^{s_1}}.
		\end{equation}
	\end{lem}
	\begin{proof}
		We define an auxiliary function
		\begin{equation}\label{Minnesota}
			\chi:\R^d\to\R
			\text{ via }
			\chi(\xi)=\tp{|\xi\cdot e_1||\xi|^{-1}+|\xi|}\mathds{1}_{B(0,1)}(\xi)+\mathds{1}_{\R^d\setminus B(0,1)}(\xi).
		\end{equation}
		From the definition of the norm on $\mathcal{H}^s(\R^d)$ it is clear that if $\eta$ belongs to this space then $\chi(\grad/2\pi\ii)\eta\in H^s(\R^d)$ and we have the equivalence
		\begin{equation}\label{Idaho}
			\tnorm{\eta}_{\mathcal{H}^s}\asymp\tnorm{\chi(\grad/2\pi\ii)\eta}_{H^s}.
		\end{equation}
		From~\eqref{Idaho}, we see immediately that the log-convexity of the $\mathcal{H}^s$ spaces follows from the usual log-convexity of $L^2$-based Sobolev spaces.
	\end{proof}
	
	We now study smoothing operators on the scales of anisotropic Sobolev spaces.
	\begin{lem}[LP-smoothability]\label{lem on lp smoothability of anisotropic Sobolev spaces}
		The Banach scale $\tcb{\mathcal{H}^s(\R^d)}_{s\in\N}$ is LP-smoothable in the sense of Definition~\ref{defn of smoothable and LP-smoothable Banach scales}.
	\end{lem}
	\begin{proof}
		For the smoothing operators we make the obvious choice $S_j\eta=\mathds{1}_{B(0,2^j)}(\grad/2\pi\ii)\eta$ for $j\in\N^+$. The continuity of $S_j:\mathcal{H}^0(\R^d)\to\mathcal{H}^\infty(\R^d)$  is clear. Moreover, the Littlewood-Paley condition~\eqref{smoothing_LP_2} is trivially satisfied with $A=1$. Thus it remains only to verify the smoothing inequalities from the first item of Definition~\eqref{defn of smoothable and LP-smoothable Banach scales}. We again consider the auxiliary function $\chi$ from~\eqref{Minnesota}.  Since the smoothing operators $\tcb{S_j}_{j=0}^\infty$ commute with the Fourier multiplication operator $\chi(\grad/2\pi\ii)$, we see that the smoothing inequalities of~\eqref{smoothing_1}--\eqref{smoothing_4} follow directly from LP-smoothability of $\tcb{H^s(\R^d)}_{s \in \N}$ (see Example~\ref{example of Euclidean Soblev spaces}) and~\eqref{Idaho}.
	\end{proof}
	
	\subsection{Adapted Sobolev spaces}\label{appendix on adapted Sobolev spaces}
	
	Given $X\in   L^{\infty}(\Omega;\R^n)$,  we define the adapted Sobolev space
	\begin{equation}\label{adapted X space def}
		H^0_X(\Omega)=\tcb{\varphi\in L^2(\Omega)\;:\;\grad\cdot(X\varphi)\in L^2(\Omega)},
	\end{equation}
	where the divergence is understood in the sense of distributions, and equip it with the obvious norm $\tnorm{\varphi}_{H^0_X}=\sqrt{\tnorm{\varphi}_{L^2}^2+\tnorm{\grad\cdot(X\varphi)}_{L^2}^2}$ and associated inner-product.

	\begin{lem}[Basic properties of $H^0_X(\Omega)$]\label{lem on density of bounded support functions in adapted Sobolev spaces}
		Let $X\in L^{\infty}(\Omega;\R^n)$.  Then the following hold.
		\begin{enumerate}
			\item $H^0_X(\Omega)$ is a Hilbert space.
			\item If $\psi \in C^\infty(\Omega) \cap W^{1,\infty}(\Omega)$ and $\varphi \in H^0_X(\Omega)$, then we have the inclusion $\varphi \psi \in H^0_X(\Omega)$ and the estimate $\norm{\varphi \psi}_{H^0_X} \lesssim \br{ \norm{X}_{L^\infty}} \norm{\psi}_{W^{1,\infty}} \norm{\varphi}_{H^0_X}$.
			\item The space $\{\varphi \in H^0_X(\Omega) \;:\; \supp(\varphi) \text{ is bounded} \}$ is dense in $H^0_X(\Omega)$.
		\end{enumerate}

	\end{lem}
	\begin{proof}
		Suppose $\{\varphi_k\}_{k=0}^\infty \subseteq H^0_X(\Omega)$ is Cauchy. Then the sequences $\{\varphi_k\}_{k=0}^\infty$ and $\{\zeta_k\}_{k=0}^\infty$ are Cauchy in $L^2$, where $\zeta_k = \nabla \cdot (\varphi_k X)$.  Thus, there exist $\varphi,\zeta \in L^2(\Omega)$ such that $\varphi_k \to \varphi$ and $\zeta_k \to \zeta$ in $L^2$.  However, for $\chi \in C_c^\infty(\Omega)$,
		\begin{equation}
			\int_\Omega \zeta \chi = \lim_{k \to \infty} \int_{\Omega}\zeta_k \chi = \lim_{k \to \infty} \int_\Omega -\varphi_k X \cdot \nabla \chi = \int_\Omega -\varphi X \cdot \nabla \chi,
		\end{equation}
		so $\zeta = \nabla\cdot(\varphi X)$ in the sense of distributions, and we deduce that $H^0_X(\Omega)$ is complete and hence a Hilbert space.  The first item is proved.
		
		Now we prove the second item.  Let $\varphi\in H^0_X(\Omega)$ and $\psi\in W^{1,\infty}(\Omega) \cap C^\infty(\Omega)$. We begin by claiming that
		\begin{equation}\label{the product rule in adapted for cutoff}
			\grad\cdot(X\psi\varphi)=\psi\grad\cdot(X\varphi)+X\varphi\cdot\grad\psi.
		\end{equation}
		Once the claim is proved, the inclusion $\varphi \psi \in H^0_X(\Omega)$ and the stated estimate follow immediately.  To prove the claim, let $\tcb{\varphi_\ell}_{\ell\in\N}\subseteq H^1(\Omega)\cap C^\infty(\Omega)$ be a sequence such that $\varphi_\ell\to\varphi$ in $L^2(\Omega)$ as $\ell\to\infty$. By the product rule, we have the identity
		\begin{equation}
			\grad\cdot(X\psi\varphi_\ell)=\psi\grad\cdot(X\varphi_\ell)+X\varphi_\ell\cdot\grad\psi \text{ for every } \ell \in \N,
		\end{equation}
		or more precisely, for any $\chi\in C^\infty_c(\Omega)$ we have that
		\begin{equation}
			-\int_{\Omega}X\psi\varphi_\ell\cdot\grad\chi 
			= -\int_{\Omega}\grad(\psi\chi)\cdot X\varphi_\ell
			+\int_{\Omega}\chi X\varphi_\ell\cdot\grad\psi.
		\end{equation}
		Now we send $\ell\to\infty$ and use the convergence in $L^2(\Omega)$ and the inclusion $\grad\cdot(X\varphi)\in L^2(\Omega)$ to deduce that
		\begin{equation}
			-\int_{\Omega}X\psi\varphi\cdot\grad\chi=\int_{\Omega}(\psi\grad\cdot(X\varphi)+X\varphi\cdot\grad\psi)\chi,
		\end{equation}
		which is \eqref{the product rule in adapted for cutoff}.  This completes the proof of the claim and the second item.
		
		Finally, we prove the density assertion.  Fix $\psi\in C^\infty_c(\R^n)$ such that $\psi=1$ in $B(0,1)$. Define $\tcb{\psi_N}_{N\in\N^+}\subset H^1(\Omega)$ via $\psi_N=\psi(N^{-1}\cdot)$.   Let  $\varphi\in H^0_X(\Omega)$, and write $\varphi_N=\psi_N\varphi$. The proof is complete as soon as we show that $\varphi_N\to\varphi$ as $N\to\infty$ in the $H^0_X(\Omega)$-norm.  The dominated convergence theorem shows that $\varphi_N\to\varphi$ in $L^2$ as $N\to\infty$. On the other hand, formula~\eqref{the product rule in adapted for cutoff} implies that $\grad\cdot(X\psi_N\varphi)=\psi_N\grad\cdot(X\varphi)+N^{-1}X\varphi(\grad\psi)(N^{-1}\cdot)$. Since $\grad\cdot(X\varphi)\in L^2(\Omega)$, we have that $\psi_N\grad\cdot(X\varphi)\to\grad\cdot(X\varphi)$ as $N\to\infty$ in $L^2(\Omega)$. Since $\grad\psi$ is bounded, we have that $N^{-1}X\varphi(\grad\psi)(N^{-1}\cdot)\to 0$ as $N\to\infty$ in $L^2(\Omega)$. Together, these convergences imply that $\grad\cdot(X\varphi_N)\to\grad\cdot(X\varphi)$ as $N\to\infty$ in $L^2(\Omega)$. Hence, we have shown that $\varphi_N\to\varphi$ as $N\to\infty$ in the $H^0_X(\Omega)$-norm, and the third item is proved.
	\end{proof}
	
	Next we prove a density result.
	
	\begin{lem}[Density of smooth functions in adapted Sobolev spaces]\label{lem on density of smooth functions in adapted sobolev spaces}
		For $X\in W^{1,\infty}(\Omega;\R^n)$, the space $C^\infty_c(\bar{\Omega})$ is dense in  $H^0_X(\Omega)$.
	\end{lem}
	\begin{proof}
		First note that the inclusion  $C^\infty_c(\bar{\Omega}) \subseteq H^0_X(\Omega)$ follows from the extra assumption that $X \in W^{1,\infty}(\Omega;\R^n)$ and Rademacher's theorem.  To prove density, we endeavor to apply Proposition III.2.8 and Theorem III.2.10 in Boyer and Fabrie~\cite{MR2986590}. These results show that, given a bounded Lipschitz domain $U\subseteq\R^n$ such that $U \subseteq B(0,R)$, there exist mollification operators $\tcb{\mathcal{S}_\ep}_{\ep>0}$ such that if $f\in L^2(U)$ and $\al\in W^{1,\infty}(U;\R^n)$ satisfy $\grad\cdot(\al f)\in L^2(U)$, then we have that $\tcb{\mathcal{S}_\ep f}_{\ep>0}\subseteq C^\infty_c(\Bar{U})$ with $\supp(\mathcal{S}_\ep f) \subseteq R+1$, $\mathcal{S}_\ep f\overset{\ep\to0}{\to} f$ in $L^2(U)$, and $\grad\cdot(\al\mathcal{S}_\ep f)\overset{\ep\to0}{\to}\grad\cdot(\al f)$ in $L^2(U)$.
		
		Note that $\Omega\subset\R^n$ does not have a compact boundary, so we cannot directly apply this result. However, thanks to Lemma~\ref{lem on density of bounded support functions in adapted Sobolev spaces} we can reduce to this case. Let $\varphi\in H^0_X(\Omega)$ and let $\kappa\in\R^+$. Thanks to the aforementioned lemma there exists a $\tilde{\varphi}\in H^0_X(\Omega)$ with bounded support such that $\tnorm{\varphi-\tilde{\varphi}}_{H^0_X}\le\kappa/2$. For some $3 \le R\in\R^+$, we then have that $\m{supp}(\tilde{\varphi})\subseteq U_{R}=B_{\R^n}(0,R)\cap\Omega$. The set $U_{2R}$ is a Lipschitz domain with compact boundary, and hence it admits smoothing operators $\tcb{\mathcal{S}_\ep}_{\ep>0}$ as described above. By the support conditions on $\tilde{\varphi}$ and the properties of these smoothing operators, there exits an $\ep_0\in\R^+$ such that for all $0<\ep\le\ep_0$ we have that $\m{supp}(\mathcal{S}_\ep\tilde{\varphi})\subseteq U_{3R/2}$ and
		\begin{equation}
			\sqrt{\tnorm{\mathcal{S}_{\ep}\tilde{\varphi}-\tilde{\varphi}}^2_{L^2(U_{2R})}+\tnorm{\grad\cdot(X\mathcal{S_\ep}\tilde{\varphi})-\grad\cdot(X\tilde{\varphi})}_{L^2(U_{2R})}^2}\le\kappa/2.
		\end{equation}
		Let $\Bar{\varphi}=\mathcal{S}_{\ep_0}\tilde{\varphi}$. By support considerations on $\tilde{\varphi}$ and $\Bar{\varphi}$, we have that actually $\Bar{\varphi}\in C^\infty_c(\Bar{\Omega})$ and
		\begin{equation}
			\begin{cases}
				\tnorm{\Bar{\varphi}-\tilde{\varphi}}_{L^2(U_{2R})}=\tnorm{\Bar{\varphi}-\tilde{\varphi}}_{L^2(\Omega)},\\
				\tnorm{\grad\cdot(X\Bar{\varphi})-\grad\cdot(X\tilde{\varphi})}_{L^2(U_{2R})}=\tnorm{\grad\cdot(X(\Bar{\varphi}-\tilde{\varphi}))}_{L^2(\Omega)}.
			\end{cases}
		\end{equation}
		Hence, by the previous two equations we get $\tnorm{\Bar{\varphi}-\tilde{\varphi}}_{H^0_X}\le\kappa/2$. The estimate $\tnorm{\varphi-\Bar{\varphi}}_{H^0_X}\le\kappa$ now follows from the triangle inequality, and density is proved.
	\end{proof}
	
	The next result records some integration by parts formulas.
	
	\begin{prop}[Some integration by parts]\label{prop on divergence trick}
		Let $X\in W^{1,\infty}(\Omega;\R^n)$ be such that  $\m{Tr}_{\pd\Omega}(X\cdot e_n)=0$.  The following hold.
		
		\begin{enumerate}
			\item If  $\varphi \in H^0_X(\Omega)$, then 
			\begin{equation}\label{the identity}
				\int_{\Omega} \frac{\nabla \cdot X}{2} \varphi^2 = \int_{\Omega} \varphi \nabla \cdot(\varphi X).
			\end{equation}
			\item If $\varphi\in H^0_X(\Omega)$ and $\psi\in H^1(\Omega)$, then $\int_{\Omega}\grad\cdot(\varphi X)\psi=-\int_{\Omega}\varphi X\cdot\grad\psi$.
		\end{enumerate}
	\end{prop}
	\begin{proof}
		We begin by proving the first item. Assume initially that $\varphi\in H^1(\Omega)$. By using the product rule in two ways, we compute that
  \begin{equation}
  2^{-1}\grad\cdot(\varphi^2X)=\grad\cdot(\varphi^2X)-2^{-1}\grad\cdot(\varphi^2X)=\varphi\grad\cdot(\varphi X)-2^{-1}(\grad\cdot X)\varphi^2.
  \end{equation}
		We then integrate this over $\Omega$, apply the divergence theorem, and use the condition $\m{Tr}_{\pd\Omega}(X\cdot e_n)=0$ to get~\eqref{the identity}. The case of general $\varphi\in H^0_X(\Omega)$ follows from a density argument via Lemma~\ref{lem on density of smooth functions in adapted sobolev spaces}. The second item is proved via a similar density argument.
	\end{proof}
	
	Next we prove a useful estimate.
	
	\begin{prop}[Refined divergence compatibility estimate]\label{prop on refined divergence compatibility estimate}
		Suppose that $X\in W^{1,\infty}(\Omega)$ and $\varphi\in H^0_X(\Omega)$. Then the function $\int_0^b\grad\cdot(X\varphi)(\cdot,y)\;\m{d}y : \R^{n-1} \to \R$ belongs to $\dot{H}^{-1}(\R^{n-1})$, satisfies the equality
		\begin{equation}\label{song for a favorite flour}
			\int_0^b\grad\cdot(X\varphi)(\cdot,y)\;\m{d}y=(\grad_{\|},0)\cdot\int_0^b(X\varphi)(\cdot,y)\;\m{d}y,
		\end{equation}
		and obeys the estimate
		\begin{equation}\label{driving down the road yesterday}
			\bsb{\int_0^b\grad\cdot(X\varphi)(\cdot,y)\;\m{d}y }_{\dot{H}^{-1}}\lesssim\tnorm{X\varphi}_{L^2}
		\end{equation}
		for an implicit constant depending only on $b$ and the dimension.
	\end{prop}
	\begin{proof}
		Let $\psi\in C^\infty_c(\R^{n-1})$ and view $\psi\in C^\infty_c(\Bar{\Omega})$ via the trivial extension $\psi(x,y)=\psi(x)$ for $(x,y)\in\R^{n-1}\times(0,b)$. Thanks to the second item of Proposition~\ref{prop on divergence trick} and Fubini's theorem, we have that
		\begin{multline}
			\int_{\R^{n-1}}\bp{\int_0^b\grad\cdot(X\varphi)(\cdot,y)\;\m{d}y}\psi=\int_{\Omega}\grad\cdot(X\varphi)\psi=-\int_{\Omega}\varphi X\cdot\grad\psi\\
			=-\int_{\R^{n-1}}\bp{\int_0^b(\varphi X)(\cdot,y);\m{d}y}\cdot(\grad_{\|}\psi,0).
		\end{multline}
		Since $\psi$ was arbitrary, the identity~\eqref{song for a favorite flour} is established. The estimate~\eqref{driving down the road yesterday}  readily follows from \eqref{song for a favorite flour} and the definition of the norm on $\dot{H}^{-1}(\R^{n-1})$.
	\end{proof}

 \section{A selection of PDE tools}	\label{appendix_pde_tools}
	
	\subsection{The divergence boundary value problem}
	
	Our focus in this subsection is the construction of right-inverses of the divergence operator in various spaces.  	We first make the following notation: for $s\in[0,\infty)$, we define the function space
	\begin{equation}\label{Rhode Island}
		\hat{H}^s(\Omega)=\bcb{\phi\in H^s(\Omega)\;:\;\int_0^b\phi(\cdot,y)\;\m{d}y\in\dot{H}^{-1}(\Sigma)}\index{\textbf{Function spaces}!160@$\hat{H}^s$},
	\end{equation}
	where the homogeneous spaces $\dot{H}^{-1}$ are defined in~\eqref{dotHminus1_def}.  This space is Hilbert for the norm $\tnorm{\phi}_{\hat{H}^s}=\sp{\tnorm{\phi}^2_{H^s}+\ssb{\int_0^b\phi(\cdot,y)\;\m{d}y}_{\dot{H}^{-1}}^2}^{1/2}$. 

 We now provide a useful compatibility estimate related to the divergence and the normal trace.

 \begin{prop}[Divergence-normal trace compatibility condition]\label{prop on divergence-normal trace}
     If $u\in H^1(\Omega;\R^n)$, then
     \begin{equation}\label{eigen is inferior to propre}
     \bsb{ \m{Tr}_{\Sigma}u\cdot e_n -\m{Tr}_{\Sigma_0}u\cdot e_n -\int_0^b(\grad\cdot u)(\cdot,y)\;\m{d}y}_{\dot{H}^{-1}(\R^{n-1})}\lesssim\tnorm{u}_{L^2}.
     \end{equation}
 \end{prop}
 \begin{proof}
 Thanks to the fundamental theorem of calculus, we have
 \begin{equation}
     \m{Tr}_{\Sigma}u\cdot e_n -\m{Tr}_{\Sigma_0}u\cdot e_n-\int_0^b(\grad\cdot u)(\cdot,y)\;\m{d}y=-(\grad_{\|},0)\cdot\int_0^bu(\cdot,y)\;\m{d}y.
 \end{equation}
Estimate~\eqref{eigen is inferior to propre} now follows by definition of the norm on $\dot{H}^{-1}(\R^{n-1})$ (see~\eqref{dotHminus1_def}).
 \end{proof}

The rest of this subsection is devoted building solution operators to divergence boundary value problems.
	
	\begin{prop}[Solution operators to divergence boundary value problems]\label{Bogovskii main}
		There exist bounded linear operators $\mathcal{B}_0$ and $\mathcal{B}_1$ such that for $k\in\N$,
		\begin{equation}
			\mathcal{B}_0: \hat{H}^k(\Omega)
			\to H^{1+k}(\Omega;\R^n)\cap H^1_0(\Omega;\R^n)
		\end{equation}
		and
		\begin{equation}
			\mathcal{B}_1:\tcb{\varphi\in H^{1/2+k}(\Sigma;\R^n)\;:\;\varphi\cdot e_n\in\dot{H}^{-1}(\Sigma)}\to H^{1+k}(\Omega;\R^n)\cap{_0}H^1(\Omega;\R^n),
		\end{equation}
		with the properties
		\begin{equation}\label{Bogovskii main 0}
			\begin{cases}
				\grad\cdot\mathcal{B}_0\psi=\psi&\text{in }\Omega,\\
				\mathcal{B}_0\psi=0&\text{on }\pd\Omega,
			\end{cases}
			\quad\text{and}\quad
			\begin{cases}
				\grad\cdot\mathcal{B}_1\varphi=0&\text{in }\Omega,\\
				\mathcal{B}_1\varphi=\varphi&\text{on }\Sigma,\\
				\mathcal{B}_1\varphi=0&\text{on }\Sigma_0.
			\end{cases}
		\end{equation}
	\end{prop}
	\begin{proof}
		We begin with the construction of $\mathcal{B}_0$.  Fix $k \in \N$ and $\psi \in \hat{H}^k(\Omega)$.  According to the Tonelli and Parseval theorems, we have that 
		\begin{equation}\label{Bogovskii main 1}
			\int_{\R^{n-1}} \sum_{j=0}^k \int_0^b \br{\xi}^{2(k-j)} \tabs{\partial_n^j \mathscr{F}[\psi](\xi,y)}^2\;\m{d}y\;\m{d}\xi \asymp \norm{\psi}_{H^k}^2,
		\end{equation}
		and in particular this implies that for almost every $\xi \in \R^{n-1}$ we have that 
		\begin{equation}\label{Bogovskii main 2}
			\sum_{j=0}^k \int_0^b \br{\xi}^{2(k-j)} \tabs{\partial_n^j \mathscr{F}[\psi](\xi,y)}^2\;\m{d}y  < \infty.
		\end{equation}
		Without loss of generality, we may assume that, in fact, this quantity is finite for every $\xi \in \R^{n-1}$.  We then define the measurable function $\zeta : \Omega \to \C$ via 
		\begin{multline}
			\zeta(\xi,y) = \frac{\cosh(2\pi \abs{\xi} y)}{2\pi \abs{\xi} \sinh(2\pi \abs{\xi} b)} \int_0^b \mathscr{F}[\psi](\xi,t) \cosh(2\pi \abs{\xi}(b-t))\;\m{d}t \\
			- \int_0^y \mathscr{F}[\psi](\xi,t) \frac{\sinh(2\pi \abs{\xi}(y-t))}{2\pi \abs{\xi}}\;\m{d}t.
		\end{multline}
		It is a simple matter to verify that for each $\xi \in \R^{n-1}\backslash \{0\}$ we have the inclusion $\zeta(\xi,\cdot) \in H^2((0,b);\C)$ and that 
		\begin{equation}\label{Bogovskii main 3}
			\begin{cases}
				-\partial_n^2 \zeta(\xi,\cdot) + 4 \pi^2 \abs{\xi}^2 \zeta(\xi,\cdot) = \mathscr{F}[\psi](\xi,\cdot) & \text{in }(0,b) \\
				\partial_n \zeta(\xi,t) =0 &\text{for }t \in \{0,b\}.
			\end{cases}
		\end{equation}
		
		Multiplying the ODE satisfied by $\zeta$ by $\bar{\zeta}$ and integrating by parts over $(0,b)$, we see that 
		\begin{multline}
			\int_0^b \tp{ \abs{\partial_n \zeta(\xi,y)}^2 + 4 \pi^2 \abs{\xi}^2 \abs{\zeta(\xi,y)}^2}\;\m{d}y  = \int_0^b \mathscr{F}[\psi](\xi,y) \Bar{\zeta(\xi,y)}\;\m{d}y\\
			=\int_0^b \mathscr{F}[\psi](\xi,y) \tp{\Bar{\zeta(\xi,y)} - \tbr{\Bar{\zeta(\xi,\cdot)}}_{(0,b)}}\;\m{d}y + \tbr{\Bar{\zeta(\xi,\cdot)}}_{(0,b)} \int_0^b \mathscr{F}[\psi](\xi)\;\m{d}y.
		\end{multline}
		In the above we have written $\br{\cdot}_{(0,b)}$ to be the integral average over $(0,b)$. By the Poincar\'{e}, Cauchy-Schwarz, and Cauchy inequalities, we may then bound 
		\begin{multline}
			\babs{\int_0^b \mathscr{F}[\psi](\xi,y) \Bar{\zeta(\xi,y)}\;\m{d}y}
			\le \frac{1}{2} \int_0^b \tp{\abs{\partial_n \zeta(\xi,y)}^2 + 4\pi^2 \abs{\xi}^2 \abs{\zeta(\xi,y)}^2}\;\m{d}y \\
			+ C \int_0^b \tabs{\mathscr{F}[\psi](\xi,y)}^2\;\m{d}y + C \babs{\int_0^b \frac{\mathscr{F}[\psi](\xi,y)}{\abs{\xi}}\;\m{d}y}^2
		\end{multline}
		for a constant $C >0$ depending only on $b$.  On the other hand, a similar argument shows that 
		\begin{multline}
			\int_0^b \tp{ \abs{\xi}^{2(k+1)} \abs{\partial_n \zeta(\xi,y)}^2 + 4 \pi^2 \abs{\xi}^{2(k+2)} \abs{\zeta(\xi,y)}^2 }\;\m{d}y 
			= \int_0^b \abs{\xi}^k \mathscr{F}[\psi](\xi,y) \abs{\xi}^{k+2}\Bar{\zeta(\xi,y)}\;\m{d}y \\
			\le \frac{1}{2} \int_0^b  4 \pi^2 \abs{\xi}^{2(k+2)} \abs{\zeta(\xi,y)}^2\;\m{d}y + C \int_0^b   \abs{\xi}^{2k} \tabs{\mathscr{F}[\psi](\xi,y)}^2\;\m{d}y
		\end{multline}
		for a constant $C >0$ depending only on $b$.  Upon combining these bounds, we deduce that  $\zeta(\xi,\cdot)$ obeys the estimate 
		\begin{equation}\label{Bogovskii main 4}
			\br{\xi}^{2(k+1)}  \int_0^b \tp{\abs{\partial_n \zeta(\xi,y)}^2 + 4 \pi^2 \abs{\xi}^2   \tabs{\zeta(\xi,y)}^2}\;\m{d}y  \lesssim  \br{\xi}^{2k} \int_0^b \tabs{\mathscr{F}[\psi](\xi,y)}^2\;\m{d}y +  \babs{\int_0^b \frac{\mathscr{F}[\psi](\xi,y)}{\abs{\xi}}\;\m{d}y}^2.  
		\end{equation}
		
		Returning to \eqref{Bogovskii main 3} and exploiting \eqref{Bogovskii main 2} and a simple finite induction argument, we deduce that $\zeta(\xi,\cdot) \in H^{2+k}((0,b);\C)$ and that for $0 \le j \le k$ we have the bound
		\begin{multline}
			\br{\xi}^{2(k-j)}  \int_0^b  \tabs{\partial_n^{2+j} \zeta(\xi,y)}^2\;\m{d}y \lesssim 
			\br{\xi}^{2(k-j)}  \int_0^b  \tabs{\partial_n^j \mathscr{F}[\psi](\xi,y)}^2\;\m{d}y \\+ \br{\xi}^{2(k-j+1)} 4 \pi^2 \tabs{\xi}^2  \int_0^b  \tabs{\partial_n^j \zeta(\xi,y)}^2\;\m{d}y.
		\end{multline}
		We may then take appropriate linear combinations of these estimates, for $0 \le j \le k$, and \eqref{Bogovskii main 4}, in order to deduce the bound 
		\begin{multline}\label{Bogovskii main 5}
			\br{\xi}^{2(k+1)}  \int_0^b  4 \pi^2 \abs{\xi}^2   \tabs{\zeta(\xi,y)}^2  \;\m{d}y +   \sum_{j=1}^{2+k} \br{\xi}^{2(k+2-j)} \int_0^b \tabs{\partial_n^{j} \zeta(\xi,y)}^2  \;\m{d}y \\
			\lesssim
			\sum_{j=0}^k \int_0^b \br{\xi}^{2(k-j)} \tabs{\partial_n^j \mathscr{F}[\psi](\xi,y)}^2\;\m{d}y + \babs{\int_0^b \frac{\mathscr{F}[\psi](\xi,y)}{\abs{\xi}}\;\m{d}y}^2.
		\end{multline}
		
		Next, we define the vector field $\Xi : \Omega \to \C^n$ via $\Xi(\xi,y) = (-2\pi i \xi \zeta(\xi,y) , -\partial_n \zeta(\xi,y))$. Since $\psi$ is real-valued we know that $\Bar{\mathscr{F}[\psi](\xi,y)} = \mathscr{F}[\psi](-\xi,y)$ for $\xi \in \R^{n-1}$ and $y \in (0,b)$, and this readily implies that $\Bar{\Xi(\xi,y)} = \Xi(-\xi,y)$ as well.  The bound \eqref{Bogovskii main 5}, integrated over $\xi \in \R^{n-1}$, implies that $\Xi \in L^2(\Omega; \C^n)$, and so we may define $X \in L^2(\Omega ;\R^n)$ via $X = \mathscr{F}^{-1}[\Xi]$.  In turn,  \eqref{Bogovskii main 5} and \eqref{Bogovskii main 1} further imply that $X \in H^{1+k}(\Omega;\R^n)$ and 
		\begin{equation}\label{Bogovskii main 6}
			\norm{X}_{H^{1+k}} \lesssim\bp{ \norm{\psi}_{H^k}^2 + \bsb{\int_0^b\psi(\cdot,y)\;\m{d}y}_{\dot{H}^{-1}}^2 }^{1/2} = \norm{\psi}_{\hat{H}^k}.
		\end{equation}
		We use \eqref{Bogovskii main 3} to compute $\mathscr{F}[\nabla \cdot X](\xi,y) = (2\pi i \xi ,0)\cdot \Xi(\xi,y) + \partial_n \Xi\cdot e_n(\xi,y) = -\partial_n^2 \zeta(\xi,y) + 4 \pi^2 \abs{\xi}^2 \zeta(\xi,y) = \mathscr{F}[\psi](\xi,y)$ and $\mathscr{F}[X](\xi,t) = \Xi\cdot e_n(\xi,t) = 0 \text{ for } t \in \{0,b\}$. Thus, $X$ satisfies 
		\begin{equation}\label{Bogovskii main 7}
			\begin{cases}
				\nabla \cdot X = \psi & \text{in } \Omega, \\
				X \cdot e_n = 0 & \text{on } \partial \Omega.
			\end{cases}
		\end{equation}

		We next recall that standard Sobolev trace theory (see for instance Chapter 7 in Adams and Fournier~\cite{MR2424078} or Chapter 18 in Leoni~\cite{MR3726909}) shows that the trace map 
		\begin{equation}
			H^{2+k}(\Omega) \ni f \mapsto (\m{Tr}_{\pd\Omega}[f], \m{Tr}_{\pd\Omega}[\partial_n f]) \in H^{3/2+k}(\partial \Omega) \times H^{1/2+k}(\partial \Omega)
		\end{equation}
		is bounded and linear and admits a bounded right inverse, $\mathcal{R}$.  Consequently, we may define $\omega \in H^{k+2}(\Omega;\R^{n-1})$ via $\omega=(\mathcal{R}(0,-X\cdot e_1),\dots,\mathcal{R}(0,-X\cdot e_{n-1}))$. Finally, this allows us to define $\mathcal{B}_0\psi \in H^{1+k}(\Omega;\R^n)$ via $\mathcal{B}_0\psi = X + (\partial_n \omega, \grad_{\|} \cdot \omega)$. Then from \eqref{Bogovskii main 7} and the construction of $\omega$ we have that $\nabla \cdot \mathcal{B}_0 \psi = \nabla \cdot X + \nabla_{\|} \cdot \partial_n \omega - \partial_n \nabla_{\|} \cdot \omega = \psi$ in $\Omega,$ while on $\partial \Omega$ we have for $j\in\tcb{1,\dots,n-1}$ $(\mathcal{B}_0 \psi)\cdot e_j = X\cdot e_j + \partial_n \omega\cdot e_j = X\cdot e_j - X\cdot e_j =0$.
		Additionally we have $(\mathcal{B}_0 \psi) \cdot e_n = X\cdot e_n- \nabla_{\|} \cdot \omega = X\cdot e_n =0$. Thus, $\mathcal{B}_0 \psi$ satisfies the properties stated in \eqref{Bogovskii main 0}.  Moreover, the construction of $\mathcal{B}_0 \psi$ shows that the resulting map $\psi \mapsto \mathcal{B}_0 \psi$ is linear and satisfies $\norm{\mathcal{B}_0 \psi}_{H^{k+1}} \lesssim \norm{\psi}_{\hat{H}^k}$. This completes the construction of the bounded linear map $\mathcal{B}_0$.
		
		We now  build $\mathcal{B}_1$.  To this end, we write $L$ for a bounded right inverse of the trace map
		\begin{equation}
			H^{1+k}(\Omega;\R^n) \cap{_0}H^1(\Omega;\R^n)   \ni u \mapsto \m{Tr}_{\Sigma}(u) \in H^{1/2+k}(\Sigma;\R^n), 
		\end{equation}
		which again exists by standard trace theory.  Note that if $\varphi \in H^{1/2+k}(\Sigma;\R^n)$ and $\varphi\cdot e_n \in \dot{H}^{-1}(\Sigma;\R^n)$, then 
		\begin{equation}
			\int_0^b \nabla \cdot (L\varphi)(\cdot,y)\;\m{d}y =  (\grad_{\|},0)\cdot\int_0^b (L\varphi)(\cdot,y)\;\m{d}y  + \varphi\cdot e_n(\cdot),
		\end{equation}
		and we readily deduce from this that 
		\begin{equation}
			\bsb{\int_0^b\nabla \cdot (L\varphi)(\cdot,y)\;\m{d}y}_{\dot{H}^{-1}} \lesssim \tnorm{L\varphi}_{L^2} +     \tsb{\varphi\cdot e_n}_{\dot{H}^{-1}} \lesssim \norm{\varphi}_{H^{1/2+k}} +     \tsb{\varphi\cdot e_n}_{\dot{H}^{-1}}.
		\end{equation}
		We may thus define the linear map $\mathcal{B}_1$ via $\mathcal{B}_1 \varphi = L\varphi - \mathcal{B}_0 \nabla \cdot L \varphi$. It is then a simple matter to verify that $\mathcal{B}_1$ is bounded and  $\mathcal{B}_1 \varphi$ satisfies the conditions stated in \eqref{Bogovskii main 0}.  This completes the construction of $\mathcal{B}_1$.
	\end{proof}
	
	We next record a couple corollaries.  The first constructs an operator on $H^k(\Omega)$.
	
	\begin{coro}[Right inverse to the divergence]\label{Bogovskii 1}
		There exists a bounded linear operator such that for $k\in\N$, $\mathcal{B}:H^k(\Omega)\to H^{1+k}(\Omega;\R^n)\cap{_0}H^1(\Omega;\R^n)$
		with $\grad\cdot\mathcal{B}\psi =\psi$ for all $\psi\in H^k(\Omega)$.
	\end{coro}
	\begin{proof}
		If for $p \in \N$  we define $A_p \psi : \Omega \to \R$ via $A_p \psi(x,y) = \frac{y^p}{b} \int_0^b \psi(x,t)\;\m{d}t$, then $A_p \psi \in H^k(\Omega)$ and $\norm{A_p \psi}_{H^k} \lesssim_p \norm{\psi}_{H^k}$.  In  other words, $A_p \in \mathcal{L}(H^k(\Omega))$.  
		
		Note that for $\psi \in H^k(\Omega)$ we have that $\int_0^b \tp{ \psi(\cdot,y) - A_0 \psi(\cdot,y)}\;\m{d}y = \int_0^b \psi(\cdot,y)\;\m{d}y - \int_0^b \psi(\cdot,y)\;\m{d}y =0$. Consequently, we may construct $\mathcal{B}_0 (\psi - A_0 \psi) \in H^{1+k}(\Omega;\R^n) \cap H^1_0(\Omega;\R^n)$ by using the operator $\mathcal{B}_0$ from Proposition \ref{Bogovskii main}.  We then define  $\mathcal{B}:H^k(\Omega)\to H^{1+k}(\Omega;\R^n)\cap{_0}H^1(\Omega;\R^n)$ via $\mathcal{B} \psi = \mathcal{B}_0 (\psi - A_0 \psi) + (A_1 \psi) e_n$. It is then a simple matter to check that this is bounded and linear and satisfies $\nabla \cdot \mathcal{B}\psi = \psi$ for all $\psi \in H^k(\Omega)$.
	\end{proof}
	
	The second corollary to Proposition~\ref{Bogovskii main} constructs an operator on $H^{1/2+k}(\Sigma)\cap\dot{H}^{-1}(\Sigma)$.
	
	\begin{coro}[Solenoidal extension operator]\label{Bogovskii 2}
		There exists a bounded linear operator such that for $k\in\N$, $\mathcal{B}_2:H^{1/2+k}(\Sigma)\cap\dot{H}^{-1}(\Sigma)\to{_0}H^1(\Omega;\R^n)$ with $\grad\cdot\mathcal{B}_2\chi=0$, $\m{Tr}_\Sigma(\mathcal{B}_2\chi\cdot e_n)=\chi$, and $(I-e_n\otimes e_n)\m{Tr}_\Sigma(\mathcal{B}_2\chi)=0$.
	\end{coro}
	\begin{proof}
		Let $\mathcal{B}_1$ be as in Proposition \ref{Bogovskii main}, and define $\mathcal{B}_2$ via $\mathcal{B}_2 \chi = \mathcal{B}_1 (\chi e_n)$.
	\end{proof}
	
	\subsection{Elliptic theory tools}\label{appendix_elliptic_tools}

	We begin this subsection by recording some results about the bilinear form $B_m$ from \eqref{copy that pasta is it linguini or roasted tolueney}.  While $B_m$ is not coercive, the next result shows that it satisfies G\aa rding's inequality.
	
	\begin{lem}[$B_m$ G\aa rding inequalities]\label{lemma on garding inequality for Bm}
		Let $m\in\N$ and $\Lambda\in W^{\infty,\infty}(\Omega)$ with $\Lambda>0$ and $1/\Lambda\in L^\infty(\Omega)$. There exist constants $c,C\in\R^+$, depending only on $\Lambda$, the dimension, $\Omega$, and $m$, such that $c\tnorm{\varphi}_{H^m}^2\le B_m(\Lambda\varphi,\varphi)+C\tnorm{\varphi}_{L^2}^2$ for all $\varphi\in H^m(\Omega)$.  
	\end{lem}
	\begin{proof}
		Consider first the case that $\Lambda=1$. We decompose $B_m=B_{m,\|}+B_{m,n}$ where
		\begin{equation}
			B_{m,\|}(\varphi_0,\varphi_1)=\sum_{j=1}^{n-1}\int_{\Omega}\pd_j^m\varphi_0\pd_j^m\varphi_1
			\text{ and }
			B_{m,n}(\varphi_0,\varphi_1)=\int_{\Omega}\pd_n^m\varphi_0\pd_n^m\varphi_1.
		\end{equation}
		For $B_{m,\|}$, we use Fubini's theorem followed by Plancherel's theorem to estimate
		\begin{multline}
			B_{m,\|}(\varphi,\varphi)=\sum_{j=1}^{ n-1}\int_0^b\int_{\R^{n-1}}|(2\pi\ii\xi_j)^m\mathscr{F}[\varphi(\cdot,y)](\xi)|^2\;\m{d}\xi\;\m{d}y\\
			=\int_0^b\int_{\R^{n-1}}\f{\sum_{j=1}^{n-1}|(2\pi\ii\xi_j)^m|^2}{\sum_{|\al|=m}|(2\pi\ii\xi)^\al|^2}\sum_{|\al|=m}|(2\pi\ii\xi)^\al\mathscr{F}[\varphi(\cdot,y)](\xi)|^2\;\m{d}\xi\;\m{d}y\\
			\gtrsim_m\sum_{\substack{\be\in\N^{n-1}\\|\be|=m}}\tnorm{\pd^{(\be,0)}\varphi}_{L^2}^2=\tnorm{\varphi}_{\dot{H}^m(\R^{n-1};L^2(0,b))}^2.
		\end{multline}
		Hence, by the fact that $\dot{H}^m\cap L^2=H^m$, we get that
		\begin{equation}\label{first equation to combine}
			B_{m,\|}(\varphi,\varphi)+\tnorm{\varphi}_{L^2}^2\gtrsim_m\tnorm{\varphi}_{H^m(\R^{n-1};L^2(0,b))}^2.
		\end{equation}
		On the other hand, for the $B_{m,n}$ piece, we see that
		\begin{equation}
			B_{m,n}(\varphi,\varphi)=\tnorm{\varphi}_{L^2(\R^{n-1};\dot{H}^m(0,b))}^2.
		\end{equation}
		Thus, by again using that $\dot{H}^m\cap L^2=H^m$, we get the bound
		\begin{equation}\label{second equation to combine}
			B_{m,n}(\varphi,\varphi)+\tnorm{\varphi}^2_{L^2}\gtrsim_{\Omega,m}\tnorm{\varphi}_{L^2(\R^{n-1};H^m(0,b))}^2.
		\end{equation}
		Upon combining \eqref{first equation to combine} and~\eqref{second equation to combine}, we arrive at the estimate
		\begin{equation}
			B_m(\varphi,\varphi)+\tnorm{\varphi}_{L^2}^2\gtrsim_{m,\Omega}\tnorm{\varphi}_{H^mL^2\cap L^2H^m}^2.
		\end{equation}
		The proof in the case that $\Lambda=1$ is then complete as soon as we note the slicing characterization
		\begin{equation}\label{it reeks in here}
			H^m(\Omega)=H^m(\R^{n-1};L^2(0,b))\cap L^2(\R^{n-1};H^m(0,b)),
		\end{equation}
		with norm equivalence, which can be proved as follows.  The slicing characterization for $\R^n$ in place of $\Omega$ is proved in Lemma A.6 from Leoni and Tice~\cite{leoni2019traveling}.  We then reduce~\eqref{it reeks in here} to this case with the help of a Stein extension operator $\mathfrak{E}_\Omega$ (see Definition~\ref{defn Stein-extension operator}), which continuously maps both  $H^{m}(\Omega)\to H^m(\R^n)$ and $H^m(\R^{n-1};L^2(0,b))\cap L^2(\R^n;H^{m}(0,b))\to H^m(\R^{n-1};L^2(\R))\cap L^2(\R^n;H^{m}(\R))$.
		
		We next consider the case of general $\Lambda$ satisfying the stated hypotheses. We introduce commutators as follows:
		\begin{equation}
			B_m(\Lambda\varphi,\varphi)=\sum_{j=1}^n\int_{\Omega}\Lambda(\pd_j^m\varphi)^2+[\Lambda,\pd_j^m]\varphi\pd_j^m\varphi.
		\end{equation}
		Hence, we have the estimate
		\begin{equation}
			c B_m(\varphi,\varphi)\le B_m(\Lambda\varphi,\varphi)+\sqrt{B_m(\varphi,\varphi)}\sqrt{\sum_{j=1}^n\tnorm{[\Lambda,\pd_j^m]\varphi}_{L^2}^2}.
		\end{equation}
		By Cauchy's inequality and the boundedness of the commutator operator (see Corollary~\ref{coro on tames estimates on commutators}), 
		\begin{equation}
			cB_m(\varphi,\varphi)\le B_m(\Lambda\varphi,\varphi)+C\tnorm{\varphi}_{H^{m-1}}^2.
		\end{equation}
		Invoking the special case $\Lambda=1$ then shows that
		\begin{equation}
			c\tnorm{\varphi}^2_{H^{m}}\le C\tnorm{\varphi}_{L^2}^2+B_m(\Lambda\varphi,\varphi)+C\tnorm{\varphi}^2_{H^{m-1}}.
		\end{equation}
		For the final term above, we use the log-convexity of the norm and Young's inequality:  for $\kappa\in(0,1)$ we estimate $\tnorm{\varphi}_{H^{m-1}}\lesssim \kappa^{-m}\tnorm{\varphi}_{L^2}+\kappa\tnorm{\varphi}_{H^m}$ and then choose $\kappa$ sufficiently small to absorb the right hand side's $\tnorm{\varphi}_{H^m}$-contribution by the left.
	\end{proof}
	
	Our next lemma involves integration by parts.
	\begin{lem}[$B_m$ integration by parts]\label{lemma on integration by parts}
		Suppose that $\psi\in H^{2m}(\Omega)$ and $\phi\in H^m(\Omega)$.  Then
		\begin{equation}
			B_m(\phi,\psi)=\int_{\Omega}\phi L_m\psi+\sum_{j=0}^{m-1}\bp{\int_{\Sigma}-\int_{\Sigma_0}}(-1)^{m-1-j}\pd_n^j\phi\pd_n^{2m-j-1}\psi,
		\end{equation}
		where $L_m$ is the linear elliptic operator defined in~\eqref{the operator Lm}.
	\end{lem}
	\begin{proof}
		This is a simple exercise in integration by parts.
	\end{proof}
	
	Now we prove a priori estimates for the natural Neumann problem associated with the operator $L_m$.
	
	\begin{lem}[A priori estimates for $L_m$]\label{lem on a priori estimates for Lm}
		Suppose that $m\in\N^+$, $k\in\N$, $\psi\in H^k(\Omega)$, and $\varphi\in H^{k+2m}(\Omega)$ are related via the equations
		\begin{equation}\label{PDE for Lm with NNC}
			\begin{cases}
				L_m\varphi=\psi&\text{in }\Omega,\\
				\pd_n^m\varphi=\cdots=\pd_n^{2m-1}\varphi=0&\text{on }\pd\Omega.
			\end{cases}
		\end{equation}
		Then we have the a priori estimate
		\begin{equation}\label{the a priori estimate for Lm}
			\tnorm{\varphi}_{H^{k+2m}}\lesssim\tnorm{\varphi}_{L^2}+\tnorm{\psi}_{H^k},
		\end{equation}
		where the implicit constant depends only on $k$, $m$, and $\Omega$.
	\end{lem}
	\begin{proof}
		We begin by proving the case $k=0$. By taking the $L^2$-inner product of the equation with $\varphi$ and utilizing Lemma~\ref{lemma on integration by parts},  we are left with
		\begin{equation}
			B_m(\varphi,\varphi)=\int_{\Omega}\psi\varphi.
		\end{equation}
		For the right hand side we use the Cauchy-Schwarz inequality, while for the left hand side we use the G\aa rding inequality from Lemma~\ref{lemma on garding inequality for Bm}.  This yields the inequality
		\begin{equation}
			\tnorm{\varphi}_{H^{2m}}\lesssim\tnorm{\varphi,\psi}_{L^2\times L^2}.
		\end{equation}
		
		We now induct on $k \in \N$.  We have already established the base case, $k=0$. Suppose that $k\in\N$ and that whenever $\varphi\in H^{k+2m}(\Omega)$ and $\psi\in H^k(\Omega)$ are related via the equations~\eqref{PDE for Lm with NNC}, we have the a priori estimate
		\begin{equation}
			\tnorm{\varphi}_{H^{k+2m}}\lesssim\tnorm{\varphi,\psi}_{H^k\times H^k}.
		\end{equation}
		We now prove the above necessarily holds for $k+1$ as well. Assume $\varphi\in H^{1+k+2m}(\Omega)$ and $\psi\in H^{1+k}(\Omega)$ satisfy~\eqref{PDE for Lm with NNC}. Let $j\in\tcb{1,\dots,n-1}$ and apply $\pd_j$ to these equations. We then see that
		\begin{equation}
			\begin{cases}
				L_m\pd_j\varphi=\pd_j\psi&\text{in }\Omega,\\
				\pd_n^m\pd_j\varphi=\cdots=\pd_n^{2m-1}\pd_j\varphi=0&\text{on }\pd\Omega.
			\end{cases}
		\end{equation}
		Hence, we can apply the inductive hypothesis and sum over $j$ to arrive at the bounds
		\begin{equation}\label{this guys got is 1}
			\sum_{j=1}^{n-1}\tnorm{\pd_j\varphi}_{H^{k+2m}}\lesssim\sum_{j=1}^{n-1}\tnorm{\pd_j\varphi,\pd_j\psi}_{H^k\times H^k}\le\tnorm{\varphi,\psi}_{H^{1+k}\times H^{1+k}}.
		\end{equation}
		On the other hand, we can rearrange the PDE satisfied by $\varphi$ and $\psi$ to see that $(-1)^m\pd_n^{2m}\varphi=\psi-(-1)^m\sum_{j=1}^{n-1}\pd_j^{2m}\varphi$. Thus,
		\begin{equation}\label{this guys got is 2}
			\tnorm{\pd_n^{2m}\varphi}_{H^{1+k}}\le\tnorm{\psi}_{H^{1+k}}+\sum_{j=1}^{n-1}\tnorm{\pd_j^{2m}\varphi}_{H^{1+k}}\le\tnorm{\varphi}_{H^{1+k}}+\sum_{j=1}^{n-1}\tnorm{\pd_j\varphi}_{H^{k+2m}}.
		\end{equation}
		Combining~\eqref{this guys got is 1} and~\eqref{this guys got is 2} then shows that
		\begin{equation}
			\sqrt{B_{1+k+2m}(\varphi,\varphi)}\le\sum_{j=1}^{n-1}\tnorm{\pd_j\varphi}_{H^{k+2m}}+\tnorm{\pd_n^{2m}\varphi}_{H^{1+k}}\lesssim\tnorm{\varphi,\psi}_{H^{1+k}\times H^{1+k}}.
		\end{equation}
		Finally, we use G\aa rding's inequality, Lemma~\ref{lemma on garding inequality for Bm}, to see that
		\begin{equation}
			\tnorm{\varphi}_{H^{1+k+2m}}\lesssim\tnorm{\varphi,\psi}_{H^{1+k}\times H^{1+k}}.
		\end{equation}
		This completes the induction argument. 
		
		From the induction, we know that a weaker version of \eqref{the a priori estimate for Lm} holds, namely the same estimate but with $\tnorm{\varphi}_{L^2}$ replaced by $\tnorm{\varphi}_{H^k}$.  However, we can readily derive the desired bound from this weaker one by using interpolation, Young's inequality, and absorption: for $\kappa\in\R^+$ we have that
		\begin{equation}
			\tnorm{\varphi}_{H^k}\lesssim\tnorm{\varphi}_{L^2}^{\f{2m}{k+2m}}\tnorm{\varphi}_{H^{k+2m}}^{\f{k}{k+2m}}\lesssim\kappa\tnorm{\varphi}_{H^{k+2m}}+\kappa^{-\f{k}{2m}}\tnorm{\varphi}_{L^2}.
		\end{equation}
	\end{proof}
	
	If we add a $0^{\m{th}}$ term to $L_m$, we obtain an existence theory.
	
	\begin{lem}[Existence theory for $L_m$]\label{elliptic theory lemma for existence}
		Let $m\in\N^+$ and $\kappa \in \R^+$. The operator
		\begin{equation}
			\kappa +L_m:\tcb{\varphi\in H^{2m}(\Omega)\;:\;\m{Tr}_{\pd\Omega}(\pd_n^m\varphi)=\cdots=\m{Tr}_{\pd\Omega}(\pd_n^{2m-1}\varphi)=0}\to L^2(\Omega)
		\end{equation}
		is a Banach isomorphism.
	\end{lem}
	\begin{proof}
		The proof of Lemma~\ref{lemma on garding inequality for Bm}, paired with the semi-definiteness of $B_m$, shows that the symmetric bilinear map
		\begin{equation}
			H^m(\Omega)\times H^{m}(\Omega)\ni(\varphi_0,\varphi_1)\mapsto 
			\int_{\Omega} \kappa  \varphi_0 \varphi_1
			+ \sum_{j=1}^n\int_{\Omega}\pd_j^m\varphi_0\pd_j^m\varphi_1 \in\R
		\end{equation}
		is bounded and coercive, and thus defines an inner-product on $H^m(\Omega)$.  The Riesz representation theorem then provides for the existence of unique weak solutions, but then standard elliptic regularity arguments (which are elementary in $\Omega$ since we can use horizontal difference quotients without cutoffs) allow us to promote the regularity of weak solutions to $H^{2m}(\Omega)$, provided the data lie in $L^2(\Omega)$.  In turn, by integrating by parts, we verify that such weak solutions satisfy the boundary conditions.  From this we readily deduce that the operator is an isomorphism between the stated spaces.
	\end{proof}
	
	We next record a result about inverting a particular pseudodifferential operator.
	
	\begin{lem}[Simple symbol inversion]\label{lemma on simple symbol inversion}
		Fix $s,\sig\in\R$ with $\sig\ge0$ and $M,N\in\N^+$. Suppose that $\psi\in H^s(\R^d)$. There exists a unique $\varphi\in H^{s+\sig}(\R^d)$ such that 
		\begin{equation}\label{Maryland}
			\tp{N^{-1}(M^{-1}+(-\Delta)^{\sig/2})-\pd_1}\varphi=\psi.
		\end{equation}
		Moreover we have the estimate
		\begin{equation}
			\tnorm{(NM)^{-1}\varphi,N^{-1}|\grad|^\sig\varphi}_{H^s\times H^{s}}\lesssim\tnorm{\psi}_{H^s}
		\end{equation}
		for implicit constants depending only on $\sig$.
	\end{lem}
	\begin{proof}
		The symbol of the pseudodifferential operator in~\eqref{Maryland} is
		\begin{equation}
			\R^d\ni\xi\mapsto a(\xi)=(MN)^{-1}\tbr{2\pi\sqrt{M|\xi|^\sig}}^2-2\pi\ii\xi\cdot e_1\in\C.
		\end{equation}
		The stated estimate then follows readily from the elementary estimate  
		\begin{equation}
			\f{1}{|a(\xi)|}\le\f{MN}{\tbr{2\pi\sqrt{M|\xi|^\sig}}^2}\le\f{N}{2\pi}\min\scb{M,|\xi|^{-\sig}} \text{ for } \xi\in\R^d.
		\end{equation}
	\end{proof}

\subsection{Dissipation calculation for traveling compressible Navier-Stokes}\label{lift my head, im still yawning}

This subsection explores the role of forcing in the traveling wave problem~\eqref{The nonlinear equations in the right form}. We will require the following density result.

\begin{lem}[Density of smooth functions with bounded support]\label{lem on density of do}
    For any  $s \in \{-1,0\} \cup \R^+$,  the subspace
    \begin{equation}\label{Xs_dense_subspace}
        \{(q,u,\eta)\in C_c^\infty(\Bar{\Omega})\times C^\infty_c(\Bar{\Omega};\R^n)\times C^\infty_c(\R^{n-1})\;:\; 
        \m{Tr}_{\Sigma_0}u=0, 
        \text{ and } \m{Tr}_\Sigma u\cdot e_n+\pd_1\eta=0 \}
    \end{equation}
    is dense in $\X^s$, as defined by \eqref{domain banach scales}.
\end{lem}
\begin{proof}
    First, we note that Theorem 5.2 and the second item of Theorem 5.6 in Leoni and Tice~\cite{leoni2019traveling} imply that  $H^r(\R^{n-1})\subset \mathcal{H}^r(\R^{n-1})$ is a continuous and  dense inclusion for $r \ge 0$.  On the other hand, $C^\infty_c(\R^{n-1})$ is dense in $H^r(\R^{n-1})$. These facts combine to give the density of $C^\infty_c(\R^{n-1})$ in $\mathcal{H}^r(\R^{n-1})$.  In turn, we deduce that
    \begin{equation}\label{tomorrow}
     C^\infty_c(\Bar{\Omega})\times C^\infty_c(\Bar{\Omega};\R^n)\times C^\infty_c(\R^{n-1}) \subset \X_s
    \end{equation}
    is a dense inclusion, where the latter space is defined by~\eqref{Illinois}.

    It remains to handle the boundary conditions that define the subspace $\X^s \subset \X_s$. For this we will modify the map $\rho_{\X}$ from Lemma~\ref{lem on tameness of domain and codomain}. Let $\psi\in C^\infty_c(\R)$ be such that $0\le\psi\le 1$, $\supp(\psi) \subseteq (-2,2)$, and $\psi=1$ on $(-1,1)$.  For $1\le\nu\in\R$ we define  $\psi_\nu\in C^\infty_c(\Bar{\Omega})$ via $\psi_\nu(x,y)=\psi(|x/\nu|^2)$ for $(x,y)\in\R^{n-1}\times(0,b)$.  Next, we define $\Pi_\nu:\X_s\to\X_s$ via
    \begin{equation}
        \Pi_\nu(q,u,\eta)=(q,u-\psi_\nu\mathcal{E}_1(\m{Tr}_{\Sigma_0}u,\m{Tr}_{\Sigma}(u\cdot e_n)+\pd_1\eta),\eta),
    \end{equation}
    where $\mathcal{E}_1$ is the map defined in the proof of Lemma~\ref{lem on tameness of domain and codomain}. In addition to being linear and bounded (with  $\sup_{\nu \ge 1} \norm{\Pi_\nu}_{\mathcal{L}(\X_s)} < \infty$), $\Pi_\nu$ has the property that if $(q,u,\eta)\in \X_s$ is smooth and supported in a ball of radius $\nu$ (appropriately interpreted for each element of the tuple), then $\Pi_\nu(q,u,\eta)$ belongs to $\X^s$ and remains smooth and supported in a ball of radius $2\nu$.

    Fix $(q,u,\eta)\in\X^s$. From the dense inclusion~\eqref{tomorrow}, we are assured of the existence of a sequence $\tcb{(\tilde{q}_N,\tilde{u}_N,\tilde{\eta}_N)}_{N=1}^\infty\subset C^\infty_c(\Bar{\Omega})\times C^\infty_c(\Bar{\Omega};\R^n)\times C^\infty_c(\R^{n-1})$  with the $N^{\m{th}}$ element of the sequence consisting of functions supported in a ball of radius $1\le R_N\in\R^+$ and with the additional property that $(\tilde{q}_N,\tilde{u}_N,\tilde{\eta}_N)\to(q,u,\eta)$ in $\X_s$ as $N\to\infty$. Set $(q_N,u_N,\eta_N)=\Pi_{R_N}(\tilde{q}_N,\tilde{u}_N,\tilde{\eta}_N)$. Thanks to the aforementioned properties of $\Pi_{R_N}$, we know that $\tcb{(q_N,u_N,\eta_N)}_{N=1}^\infty\subset\X^s$ and that this sequence consists of smooth functions with compact support.

    By using the boundary conditions implied by the inclusion $(q,u,\eta)\in\X^s$, we now check that
    \begin{multline}\label{it can hurt you}
        (q,u,\eta)-(q_N,u_N,\eta_N)=(q-\tilde{q}_N,u-\tilde{u}_N,\eta-\tilde{\eta}_N)\\-(0,\psi_{R_N}\mathcal{E}_1(\m{Tr}_{\Sigma_0}(\tilde{u}_N-u),\m{Tr}_{\Sigma}(\tilde{u}_N-u)\cdot e_n+\pd_1(\tilde{\eta}_N-\eta)),0).
    \end{multline}
    Then, in light of the convergence of $\tcb{(\tilde{q}_N,\tilde{u}_N,\tilde{\eta}_N)}_{N=1}^\infty$ to $(q,u,\eta)$, we conclude from \eqref{it can hurt you} that $(q_N,u_N,\eta_N) \to (q,u,\eta)$ in $\X^s$ as $N \to \infty$.  
\end{proof}

We now record an important calculation.

\begin{thm}[Dissipation-power balance for traveling compressible Navier-Stokes]\label{thm on power-dissipation}
Suppose that $\gam\in\R^+$, $\N\ni s\ge1+\tfloor{n/2}$, $(g,f,k)\in\Y^s$, and $(q,u,\eta)\in B_{\X^{1+\tfloor{n/2}}}(0,\rho_{\m{prin}})\cap\X^{1+s}$, where $\rho_{\m{prin}}$ is defined in Theorem~\ref{thm on smooth tameness of the principal part nonlinear operator}.  Further suppose that 
\begin{equation}\label{whole foods}
    \begin{cases}
        \grad\cdot(\sig_{q,\eta}(u-M_{\eta}e_1))=g&\text{in }\Omega,\\
        \gam^2\sig_{q,\eta}M_{\eta}^{-\m{t}}(((u-M_{\eta}e_1)\cdot\grad)(M_\eta^{-1}u))+\sig_{q,\eta}\grad(q+\mathfrak{g}\eta)&\\\quad-\gam M_{\eta}^{-\m{t}}\grad\cdot(\mathbb{S}^{\sig_{q,\eta}}_{\mathcal{A}_\eta}(M_{\eta}^{-1}u)M_{\eta}^{\m{t}})=f&\text{in }\Omega,\\
        -\tp{(P-P_{\m{ext}})\circ\sig_{q,\eta}-\gam\S^{\sig_{q,\eta}}_{\mathcal{A}_\eta}(M_\eta^{-1}u)}M_\eta^{\m{t}}e_n-\varsigma\mathscr{H}(\eta)M_\eta^{\m{t}}e_n=k&\text{in }\Sigma,\\
        u\cdot e_n+\pd_1\eta=0&\text{on }\Sigma,\\
        u=0&\text{on }\Sigma_0.
    \end{cases}
\end{equation}
In other words, $\Psi(q,u,\eta,\gam)=(g,f,k)$, where $\Psi$ is the operator given in~\eqref{Nebraska}. Then
    \begin{multline}\label{DP}
        \int_{\Omega}\f{\gam}{J_\eta}\bp{\f{\upmu(\sig_{q,\eta})}{2}|\mathbb{D}^0_{M_\eta^{\m{t}}}(M_{\eta}^{-1}u)|^2+\uplambda(\sig_{q,\eta})|\grad\cdot u|^2}=\int_{\Omega}f\cdot u+g(\gam^2|M_{\eta}^{-1}u|^2/2+q)\\
        +\mathfrak{g} \int_{\R^{n-1}}\bp{|\grad_{\|}|^{-1}\int_0^bg(\cdot,y)\;\m{d}y}|\grad_{\|}|\eta
        +\int_{\Sigma}k\cdot M_{\eta}^{-1}u,
    \end{multline}
    where for  $w : \Omega \to \R^n$ differentiable and $M : \Omega \to \R^{n \times n}$ we write 
    \begin{equation}
        \mathbb{D}^0_{M}w=\grad w M^{\m{t}}+M\grad w^{\m{t}}-\f{2}{n}(M\grad)\cdot w I.
    \end{equation}
\end{thm}
\begin{proof}
    First, we claim that it suffices to prove~\eqref{DP} under the additional assumption that $(q,u,\eta)$ belongs to the space in \eqref{Xs_dense_subspace}.  Indeed, assume that the identity holds in this special case, and let $(q,u,\eta)\in B_{\X^{1+\tfloor{n/2}}}(0,\rho_{\m{prin}})\cap\X^{1+s}$ be generic. Thanks to Lemma~\ref{lem on density of do}, there exists a sequence $\tcb{(q_N,u_N,\eta_N)}_{N=1}^\infty$, belonging to the space in \eqref{Xs_dense_subspace}, such that  $(q_N,u_N,\eta_N) \to (q,u,\eta)$ in $\X^{1+s}$ as $N\to\infty$.  Due to this convergence, we may assume without loss of generality that the sequence is contained in $B_{\X^{1+\tfloor{n/2}}}(0,\rho_{\m{prin}})$.  We then rewrite the identity $\Psi(q,u,\eta,\gam)=(g,f,k)$ as $\Psi(q_N,u_N,\eta_N,\gam)=(g_N,f_N,k_N)$, where $g_N = g+(\Psi_1(q_N,u_N,\eta_N)-\Psi_1(q,u,\eta))$, $f_N = f + (\Psi_2(q_N,u_N,\eta_N,\gam)-\Psi_2(q,u,\eta,\gam))$, and $k_N = k + (\Psi_3(q_N,u_N,\eta_N,\gam)-\Psi_3(q,u,\eta,\gam))$. Thanks to the continuity of the map $\Psi$ established in Theorem~\ref{thm on smooth tameness of the principal part nonlinear operator}, we have that $(g_N,f_N,k_N)\to(g,f,k)$ in the space $\Y^s$ as $N\to\infty$.  Using the special case, we  have the identity
        \begin{multline}\label{DPN}
        \int_{\Omega}\f{\gam}{J_{\eta_N}}\bp{\f{\upmu(\sig_{q_N,\eta_N})}{2}|\mathbb{D}^0_{M_{\eta_N}^{\m{t}}}(M_{\eta_N}^{-1}u_N)|^2+\uplambda(\sig_{q_N,\eta_N})|\grad\cdot u_N|^2}
        = \int_{\Omega}f_N\cdot u_N+g_N(\gam^2|M_{\eta_N}^{-1}u_N|^2/2+q_N) \\
        + \int_{\Sigma}k_N\cdot M_{\eta_N}^{-1}u_N
        + \mathfrak{g} \int_{\R^{n-1}}\bp{|\grad_{\|}|^{-1}\int_0^b g_N(\cdot,y)\;\m{d}y}|\grad_{\|}|\eta_N
    \end{multline}
    for every $N$.  Since  $(q_N,u_N,\eta_N)\to(q,u,\eta)$ in $\X^{1+s}$ and $(g_N,f_N,k_N)\to(g,f,k)$ in $\Y^s$ as $N \to \infty$ and $s\ge1+\tfloor{n/2}$, it is a simple matter to send $N \to \infty$ in~\eqref{DPN} to obtain the desired equality~\eqref{DP}.  This completes the proof of the claim.

    We now establish~\eqref{DP} in the special case.  We begin by taking the inner product of the second equation in~\eqref{whole foods} with $u$ in $L^2(\Omega;\R^n)$. Since the vector field $\sig_{q,\eta}(u-M_{\eta}e_1)$ has divergence $g$ and vanishing normal trace, the contribution of the advective derivative is
    \begin{multline}
        \int_{\Omega}\sig_{q,\eta}M_{\eta}^{-\m{t}}(((u-M_{\eta}e_1)\cdot\grad)(M_\eta^{-1}u))\cdot u=-\int_{\Omega}\f{1}{2}\grad\cdot(\sig_{q,\eta}(u-M_\eta e_1))|M_{\eta}^{-1}u|^2\\+\int_{\Sigma}\f{1}{2}\sig_{q,\eta}(u-M_{\eta}e_1)\cdot e_n|M_{\eta}^{-1}u|^2=-\f{1}{2}\int_{\Omega}g|M_{\eta}^{-1}u|^2.
    \end{multline}
    The contribution of the viscous stress term is
    \begin{multline}
        \int_{\Omega}M_{\eta}^{-\m{t}}\grad\cdot(\mathbb{S}^{\sig_{q,\eta}}_{\mathcal{A}_{\eta}}(M_{\eta}^{-1}u)M_{\eta}^{\m{t}})\cdot u=\int_{\Sigma}\mathbb{S}^{\sig_{q,\eta}}_{\mathcal{A}_\eta}(M_{\eta}^{-1}u)M_{\eta}^{\m{t}}e_n\cdot M_{\eta}^{-1}u\\-\int_{\Omega}\mathbb{S}_{\mathcal{A}_\eta}^{\sig_{q,\eta}}(M_{\eta}^{-1}u)M_{\eta}^{\m{t}}:\grad(M_{\eta}^{-1}u)=\int_{\Sigma}\mathbb{S}^{\sig_{q,\eta}}_{\mathcal{A}_\eta}(M_{\eta}^{-1}u)M_{\eta}^{\m{t}}e_n\cdot M_{\eta}^{-1}u\\-\int_{\Omega}\f{1}{J_\eta}\bp{\f{\upmu(\sig_{q,\eta})}{2}|\mathbb{D}^0_{M_\eta^{\m{t}}}(M_{\eta}^{-1}u)|^2+\uplambda(\sig_{q,\eta})|\grad\cdot u|^2},
    \end{multline}
    where in the final identity we have used the fact that $\grad\cdot(J_\eta \mathcal{A}_{\eta}) =0$.   For the pressure contribution, we first write $\sig_{q,\eta}u=\sig_{q,\eta}(u-M_{\eta}e_1)+\sig_{q,\eta}M_{\eta}e_1$, and then integrate by parts to see that
    \begin{equation}
        \int_{\Omega}\sig_{q,\eta}\grad(q+\mathfrak{g}\eta)\cdot u=-\int_{\Omega}g(q+\mathfrak{g}\eta)+\int_{\Omega}\sig_{q,\eta}M_{\eta}e_1\cdot\grad(q+\mathfrak{g}\eta).
    \end{equation}
    To handle the final term above, we use the definition of $\sig_{q,\eta}$, which appears in~\eqref{sigma_q_eta_def}, to express
    \begin{equation}
        q+\mathfrak{g}\eta=H\circ\sig_{q,\eta}-H\circ\varrho(\mathfrak{F}_\eta\cdot e_n).
    \end{equation}
    Since $\grad(\mathfrak{F}_\eta\cdot e_n)=J_\eta M_{\eta}^{-\m{t}}e_n$, it follows that $M_\eta e_1\cdot\grad(H\circ\varrho(\mathfrak{F}_\eta\cdot e_n))=M_\eta e_1\cdot\grad(P\circ\varrho(\mathfrak{F}_\eta\cdot e_n))=0$. We may then rewrite
    \begin{equation}
    \int_{\Omega}\sig_{q,\eta}M_{\eta}e_1\cdot\grad(q+\mathfrak{g}\eta)=\int_{\Omega}M_\eta e_1\cdot\grad(P(\sig_{q,\eta})-P\circ\varrho(\mathfrak{F}_\eta\cdot e_n)).
    \end{equation}
    Since  $\grad\cdot (M_\eta e_1)=0$ and  $P(\sig_{q,\eta})-P\circ\varrho(\mathfrak{F}_\eta\cdot e_n) \in C_c^\infty(\Bar{\Omega})$, we may integrate by parts once again to see that
    \begin{equation}
        \int_{\Omega}M_\eta e_1\cdot\grad(P(\sig_{q,\eta})-P\circ\varrho(\mathfrak{F}_\eta\cdot e_n))=-\int_{\Sigma}(P-P_{\m{ext}})(\sig_{q,\eta})\pd_1\eta+\int_{\Sigma}(P\circ\varrho(\mathfrak{F}_\eta\cdot e_n)-P_{\m{ext}})\pd_1\eta.
    \end{equation}
    The final term above vanishes since
    \begin{equation}
        \int_{\Sigma}(P\circ\varrho(\mathfrak{F}_\eta\cdot e_n)-P_{\m{ext}})\pd_1\eta=\int_{\Sigma}\pd_1\bp{\int_0^\eta(P\circ\varrho(b+s))-P_{\m{ext}})\;\m{d}s}=0.
    \end{equation}
    Combining these, we deduce that the contribution of the pressure term in the momentum equation is 
    \begin{equation}
        \int_{\Omega}\sig_{q,\eta}\grad(q+\mathfrak{g}\eta)\cdot u=\int_{\Sigma}(P-P_{\m{ext}})(\sig_{q,\eta})u\cdot e_n
        -\int_{\Omega}g(q+\mathfrak{g}\eta).
    \end{equation}
    Upon synthesizing these calculations, we conclude that
    \begin{multline}\label{lebron}
        \int_{\Omega}f\cdot u+g(\gam^2|M_{\eta}^{-1}u|^2/2+q+\mathfrak{g}\eta)\\
        -\int_{\Sigma}((P-P_{\m{ext}})(\sig_{q,\eta})-\gam\mathbb{S}^{\sig_{q,\eta}}_{\mathcal{A}_\eta}(M_{\eta}^{-1}u))M_{\eta}^{\m{t}}e_n\cdot M_{\eta}^{-1}u\\=\int_{\Omega}\f{\gam}{J_\eta}\bp{\f{\upmu(\sig_{q,\eta})}{2}|\mathbb{D}^0_{M_\eta^{\m{t}}}(M_\eta^{-1}u)|^2+\uplambda(\sig_{q,\eta})|\grad\cdot u|^2}.
    \end{multline}
    
    We appeal to the dynamic boundary condition to rewrite
    \begin{multline}\label{james}
        -\int_{\Sigma}((P-P_{\m{ext}})(\sig_{q,\eta})-\gam\mathbb{S}^{\sig_{q,\eta}}_{\mathcal{A}_\eta}(M_{\eta}^{-1}u))M_{\eta}^{\m{t}}e_n\cdot M_{\eta}^{-1}u\\
        =\int_{\Sigma}(k+\varsigma\mathscr{H}(\eta)M_\eta^{\m{t}}e_n)\cdot M_\eta^{-1}u=\int_{\Sigma}k\cdot M_{\eta}^{-1}u+\varsigma\int_{\Sigma}\pd_1(\tbr{\grad_{\|}\eta}-1)=\int_{\Sigma}k\cdot M_{\eta}^{-1}u.
    \end{multline}
    Finally, from Parseval's theorem we have the identity 
    \begin{equation}\label{basketball}
        \int_{\Omega}g\eta=   \int_{\R^{n-1}}\bp{|\grad_{\|}|^{-1}\int_0^bg(\cdot,y)\;\m{d}y}|\grad_{\|}|\eta.
    \end{equation}
    Then~\eqref{DP} follows by combining~\eqref{lebron}, \eqref{james}, and~\eqref{basketball}.
\end{proof}

Theorem~\ref{thm on power-dissipation} immediately leads to the next result.

\begin{coro}\label{flat_diss_power_id}
    Suppose that  $\N\ni s\ge 2+\tfloor{n/2}$, $(q,u,\eta)\in \X^{1+s}\cap B_{\X^{2+\tfloor{n/2}}}(0,\rho_{\m{WD}})$, and $(\mathcal{T},\mathcal{G},\mathcal{F})\in \W_{s}$ satisfy  \eqref{The nonlinear equations in the right form}.  Then 
        \begin{multline}\label{DP2}
        \int_{\Omega}\f{\gam}{J_\eta}\bp{\f{\upmu(\sig_{q,\eta})}{2}|\mathbb{D}^0_{M_\eta^{\m{t}}}(M_\eta^{-1}u)|^2+\uplambda(\sig_{q,\eta})|\grad\cdot u|^2}=\int_{\Sigma}\mathcal{T}\circ\mathfrak{F}_\eta M_\eta^{\m{t}}e_n\cdot M_{\eta}^{-1}u\\+\int_{\Omega}J_\eta(\sig_{q,\eta}\mathcal{G}\circ\mathfrak{F}_\eta+\mathcal{F}\circ\mathfrak{F}_\eta)\cdot M_{\eta}^{-1}u.
    \end{multline}
\end{coro}

We now combine this result with Korn-type bounds to deduce a useful uniqueness result. Recall that $\rho_{\m{WD}}$ is from Theorem~\ref{thm on smooth tameness of the nonlinear operator}.

\begin{coro}\label{trivial_solns_unique}
There exists a $\rho \in (0,\rho_{\m{WD}}]$, depending on $b$, $\mathfrak{g}$, $P$, $\upmu$, $\uplambda$, and $n$, such that if $(q,u,\eta)\in \X^{3+\tfloor{n/2}}\cap B_{\X^{2+\tfloor{n/2}}}(0,\rho)$  satisfies \eqref{The nonlinear equations in the right form} with $\mathcal{T}=0$, $\mathcal{G}=0$, and $\mathcal{F}=0$, then $q=0$, $u=0$, and $\eta =0$.
\end{coro}
\begin{proof}
We begin with some general considerations.  Propositions \ref{prop on smooth tameness of the continuity equation}  and \ref{prop on smooth tameness of the momentum equation 1}  show that if \begin{equation}\label{gladines}
(q,u,\eta) \in B_{\X^{2+\tfloor{n/2}}}(0,\rho_{\m{WD}}),
\end{equation}
then $c \le \sigma_{q,\eta} \le C$ for some constants $c,C \in\R^+$, and $1/2 \le  J_\eta \le 3/2$, which in particular implies that $M_\eta$ is invertible.  Assume~\eqref{gladines} holds.   We may then define
\begin{equation}
    \mathfrak{D}_{\eta} =  \int_{\Omega} \bp{ \f{\upmu(\varrho)}{2} |\mathbb{D}^0_{M_\eta^{\m{t}}}(M_\eta^{-1}u)|^2 + \uplambda(\varrho) |\grad\cdot u|^2}.
\end{equation}
It is a simple matter to check that 
\begin{equation}\label{gladines_2.5}
    \mathfrak{D}_0 \lesssim \mathfrak{D}_\eta + g(\norm{\eta}_{H^{9/2+\tfloor{n/2}}} ) \norm{u}_{H^1}^2 
\end{equation}
for an continuous and increasing function $g: [0,\infty) \to [0,\infty)$ such that $g(0) =0$.  In light of the properties of $g$ and the Korn inequalities from Propositions \ref{prop on Korn's inequality} and \ref{prop on deviatoric Korn's inequality}, we may choose $0 < \rho \le \rho_{\m{WD}}$ such that if $\norm{\eta}_{H^{9/2+\tfloor{n/2}}} < \rho$, then the term $g(\norm{\eta}_{H^{9/2+\tfloor{n/2}}} ) \norm{u}_{H^1}^2$ may be absorbed onto the left side of \eqref{gladines_2.5}, resulting in the bound
\begin{equation}\label{gladines_3}
    \mathfrak{D}_0 \lesssim \mathfrak{D}_\eta. 
\end{equation}

Now  assume that $(q,u,\eta)\in \X^{3+\tfloor{n/2}}\cap B_{\X^{2+\tfloor{n/2}}}(0,\rho)$  satisfies \eqref{The nonlinear equations in the right form} with $\mathcal{T}=0$ and $\mathcal{G}=\mathcal{F}=0$.  Corollary \ref{flat_diss_power_id} and the above bounds on $\sigma_{q,\eta}$ and $J_\eta$ then imply that $\mathfrak{D}_\eta =0$.  Then \eqref{gladines_3} implies that $\mathfrak{D}_0 =0$, and so we may again appeal to the Korn inequalities to see that $u=0$.  The fourth equation in \eqref{The nonlinear equations in the right form} then implies that $\eta =0$, but then the second and third require that $q=0$.
\end{proof}

 \section{Fine tools for nonlinear analysis} \label{appendix_nlin_analysis}
	\subsection{Smoothness of superposition nonlinearities}\label{subsection on analysis of superposition nonlinearities}
	
	This subsection is concerned with operators between Sobolev spaces involving composition nonlinearities. 
	\begin{lem}\label{comp_bilip_lem}
		Let $k = 2+ \tfloor{n/2}$.  Suppose that $\Phi : \R^n \to \R^n$ is a bi-Lipschitz homeomorphism and a $C^1$ diffeomorphism.  Then there exists a $\delta >0$, depending on $n$ and the Lipschitz seminorm $[\Phi]_{C^{0,1}}$, such that if $g \in W^{k,\infty}(\R^n;\R^n)$ and $h \in H^k(\R^n;\R^n)$ satisfy 
		\begin{equation}\label{comp_bilip_lem_0}
			\max\{\norm{g}_{W^{k,\infty}}, \norm{h}_{H^k}\} < \delta,
		\end{equation}
		then $\Phi + g +h$ is also a bi-Lipschitz homeomorphism and a $C^1$ diffeomorphism, satisfying the bounds 
		\begin{equation}\label{comp_bilip_lem_00}
			\tnorm{D(\Phi + g + h)}_{C^0_b} < 3\cdot 2^{-1} \tnorm{D \Phi}_{C^0_b} \text{ and }   \tnorm{D(\Phi + g + h)^{-1}}_{C^0_b} < 2 \tnorm{D \Phi^{-1}}_{C^0_b}.
		\end{equation}
	\end{lem}
	\begin{proof}
		First note that the Banach fixed point theorem implies that if $\Xi: \R^n \to \R^n$ is a Lipschitz map satisfying the bound $[\Xi]_{C^{0,1}} < [\Phi^{-1}]_{C^{0,1}}^{-1}$, then $\Phi + \Xi$ is also a bi-Lipschitz homeomorphism.  The standard Sobolev embeddings provide a constant $C>0$, depending on $n$, such that $[g+h]_{C^{0,1}} \le \norm{g + h}_{C^1_b} \le C\tp{ \norm{g}_{W^{k,\infty}} + \norm{h}_{H^k}}$.
		Thus, if we set $\delta = \tp{ 4 C [\Phi^{-1}]_{C^{0,1}}}^{-1}$,
		then the bound \eqref{comp_bilip_lem_0} implies that $[g+h]_{C^{0,1}} < \frac{1}{2} [\Phi^{-1}]_{C^{0,1}}^{-1}$,
		and so $\Phi + g +h$ is a bi-Lipschitz homeomorphism. 	On the other hand, $[\Phi^{-1}]_{C^{0,1}} = \tnorm{D \Phi^{-1}}_{C^0_b} \text{ and } [g+h]_{C^{0,1}} = \tnorm{Dg + Dh}_{C^0_b}$,
		so the continuous map $D(\Phi + g+h)$ is everywhere invertible, and so $\Phi + g + h$ is then a $C^1$ diffeomorphism by the inverse function theorem.
		
		It remains to prove the bound \eqref{comp_bilip_lem_00}.  To this end first note that $\tnorm{D(g+h)(D\Phi)^{-1}}_{C^0_b}\le\tsb{\Phi^{-1}}_{C^{0,1}}\tsb{g+h}_{C^{0,1}}<2^{-1}$
		From this we deduce that 
		\begin{equation}
			\tnorm{D(\Phi + g +h)}_{C^0_b} \le \norm{D \Phi}_{C^0_b} \tnorm{1+D(g+h)(D\Phi)^{-1}}_{C^0_b} 
			\le 
			 3\cdot 2^{-1} \norm{D \Phi}_{C^0_b}
		\end{equation}
		and 
		\begin{multline}
			\tnorm{D(\Phi + g +h)^{-1}}_{C^0_b} =   \tnorm{(D\Phi)^{-1} (1 + D(g+h) (D\Phi)^{-1})^{-1}  }_{C^0_b}  \\
			\le \tnorm{D\Phi^{-1}}_{C^0_b} \sum_{m=0}^\infty \tnorm{((D(g+h) (D\Phi)^{-1})^{-1})^m }_{C^0_b} 
			<  \tnorm{D\Phi^{-1}}_{C^0_b} \sum_{m=0}^\infty \frac{1}{2^m}
			= 2\tnorm{D\Phi^{-1}}_{C^0_b}.
		\end{multline}
		These prove \eqref{comp_bilip_lem_00}.
	\end{proof}

	The following is a modification of the main argument presented in Inci, Kappeler, and Topalov~\cite{MR3135704}.
	
	\begin{thm}\label{comp_C2}
		Let $\N \ni k \ge2+ \tfloor{n/2}$ and $m \in \N$.  Suppose that $\Phi : \R^n \to \R^n$ is a bi-Lipschitz homeomorphism and a $C^1$ diffeomorphism such that $D\Phi \in W^{k-1,\infty}(\R^n; \R^{n \times n})$.  Let $\delta >0$ be determined by $n$ and $\Phi$ as in Lemma \ref{comp_bilip_lem}.  Define the map 
		\begin{multline}
			\Lambda : H^{m+k}(\R^n;\R^d) \times (B_{W^{2+\tfloor{n/2},\infty}}(0,\delta)\cap W^{k,\infty}(\R^n;\R^n)) \times (B_{H^{2+\tfloor{n/2}}}(0,\delta)\cap H^k(\R^n;\R^n))\\
			\to H^k(\R^n; \R^d)
		\end{multline}
		via $\Lambda(f,g,h) = f(\Phi + g +h)$. Then the following hold.
		\begin{enumerate}
			\item $\Lambda$ is well-defined and continuous.
			\item If $m=1$, then $\Lambda$ is $C^1$ and satisfies
			\begin{equation}\label{comp_C2_01}
				D \Lambda(f,g,h) (f_1,g_1,h_1) = D f(\Phi + g + h) (g_1 + h_1) + f_1(\Phi + g +h).   
			\end{equation}
			\item If $m\ge 2$, then $\Lambda$ is $C^m$ and satisfies
			\begin{multline}\label{comp_C2_02}
				D^m \Lambda(f,g,h) [(f_1,g_1,h_1),\dots,(f_m,g_m,h_m)] =  D^m f(\Phi + g +h)\tp{g_i+h_i}_{i=1}^m \\
				+\sum_{\ell=1}^mD^{m-1}f_\ell(\Phi+g+h)\tp{g_i+h_i}_{i\neq\ell}.
			\end{multline}
		\end{enumerate}

	\end{thm}
	\begin{proof}
		We divide the proof into steps.
		
		\textbf{Step 1}: Preliminary observations. For any $f$, $g$, and $h$ in the domain of $\Lambda$ we have that $\max\{\norm{g}_{W^{2+\tfloor{n/2},\infty}}, \norm{h}_{H^{2+\tfloor{n/2}}}\} < \delta$. Consequently, Lemma \ref{comp_bilip_lem} implies that $\Phi + g +h$ is a $C^1$ diffeomorphism and that 
		\begin{equation}\label{comp_C2_1}
			\tnorm{D (\Phi + g + h)}_{C^0_b} +     \tnorm{D (\Phi + g + h)^{-1}}_{C^0_b} \le A_0,
		\end{equation}
		where $A_0>0$ is a constant depending on $\Phi$.  Throughout the rest of the proof, when we write $\lesssim$ we allow for the implicit constant to depend on $A_0$, and thus on $\Phi$.

		\textbf{Step 2}: Well-definedness and continuity.  We now aim to prove that $\Lambda$ is actually well-defined, i.e. takes values in $H^k(\R^n;\R^d)$, and is a continuous map.  It suffices to prove this when $m=0$, as the cases $m\in \{1,2\}$ follow from this case.  To prove this, we proceed by finite induction on $0 \le j \le k$.  For such $j$ let $\mathbb{P}_j$ denote the proposition that 
		\begin{multline}
			\Lambda : H^{m+j}(\R^n;\R^d) \times (B_{W^{2+\tfloor{n/2},\infty}}(0,\delta)\cap W^{k,\infty}(\R^n;\R^n)) \times (B_{H^{2+\tfloor{n/2}}}(0,\delta)\cap H^k(\R^n;\R^n))\\
			\to H^j(\R^n; \R^d)
		\end{multline}
		is well-defined and continuous and obeys the estimate
		\begin{equation}\label{comp_C2_1.5}
			\norm{\Lambda(f,g,h)}_{H^j} \lesssim \tbr{\norm{g}_{W^{k,\infty}}, \norm{h}_{H^k}}^j \norm{f}_{H^j}.
		\end{equation}
		
		Consider the base case, $j =0$.  A change of variables and the bound \eqref{comp_C2_1} allow us to estimate 
		\begin{equation}\label{comp_C2_2}
			\norm{\Lambda(f,g,h)}_{H^0} = \bp{\int_{\R^n} \abs{f \circ(\Phi+g+h)}^2 }^{1/2} \le \tnorm{\det \nabla  (\Phi +g +h)^{-1}}_{L^\infty}^{1/2} \tnorm{f}_{H^0} \lesssim \tnorm{f}_{H^0}.
		\end{equation}
		This shows that $\Lambda$ is well-defined and establishes \eqref{comp_C2_1.5} when $j=0$. Next we bound
		\begin{equation}
			\tnorm{\Lambda(f,g,h) - \Lambda(\tilde{f},\tilde{g},\tilde{h})}_{H^0} \le \tnorm{\Lambda(f,g,h) - \Lambda(f,\tilde{g},\tilde{h})}_{H^0} + \tnorm{\Lambda(f-\tilde{f},\tilde{g},\tilde{h})}_{H^0}.    
		\end{equation}
		For the latter term we use \eqref{comp_C2_2} to see that $\lim_{(\tilde{f},\tilde{g},\tilde{h}) \to (f,g,h)} \tnorm{\Lambda(f-\tilde{f},\tilde{g},\tilde{h})}_{H^0}=0$,
		while for the former we use the density of $C_c^\infty(\R^n;\R^d)$ in $H^0(\R^n;\R^d)$ together with the fact that if $(\tilde{g},\tilde{h}) \to (g,h)$ then $\Phi + \tilde{g}+\tilde{h} \to \Phi + g + h$ uniformly to deduce that $\lim_{(\tilde{f},\tilde{g},\tilde{h}) \to (f,g,h)} \tnorm{\Lambda(f,g,h) - \Lambda(f,\tilde{g},\tilde{h})}_{H^0}=0$.
		Thus, $\Lambda$ is continuous when $j=0$, and we find that $\mathbb{P}_0$ holds.
		
		Proceeding inductively, we now suppose that $\mathbb{P}_\ell$ holds for all $0\le \ell \le j \le k-1$ and consider the case $j+1 \le k$.  Let $(f,g,h)$ be in the domain of $\Lambda$ in this case.  For any $1\le a \le n$ we compute 
		\begin{equation}
			\partial_a \Lambda(f,g,h) = \sum_{b=1}^n \partial_b f \circ(\Phi+g+h) \partial_a(\Phi+g+h)_b = \sum_{b=1}^n \Lambda(\partial_b f,g,h) \partial_a (\Phi+g+h)_b    
		\end{equation}
		in order to exploit the induction hypothesis and a product estimate (see Corollary~\ref{corollary on tame estimates on simple multipliers}) to bound 
		\begin{multline}
			\norm{\Lambda(f,g,h)}_{H^{j+1}} \lesssim  \tnorm{\Lambda(f,g,h)}_{H^{0}} + \sum_{a=1}^n \tnorm{\partial_a \Lambda(f,g,h)}_{H^{j}} \\
			\lesssim \tnorm{f}_{H^0} + \sum_{a,b=1}^n \tnorm{\Lambda(\partial_b f,g,h) \partial_a (\Phi+g+h)_b}_{H^j}
			\lesssim \tnorm{f}_{H^0} + \tbr{\norm{g}_{W^{k,\infty}}, \tnorm{h}_{H^k}}\sum_{b=1}^n \tnorm{\Lambda(\partial_b f,g,h)}_{H^j} \\
			\lesssim \tnorm{f}_{H^0} + \tbr{\tnorm{g}_{W^{k,\infty}},\tnorm{h}_{H^k}} \tbr{ \norm{g}_{W^{k,\infty}}, \norm{h}_{H^k} }^j \tnorm{f}_{H^{j+1} }
			\lesssim  \tbr{ \norm{g}_{W^{k,\infty}}, \norm{h}_{H^k} }^{j+1} \tnorm{f}_{H^{j+1}},
		\end{multline}
		which shows well-definedness and the bound \eqref{comp_C2_1.5} for  $j+1$.
		
		Next, we establish the continuity assertion of $\mathbb{P}_{j+1}$. We initially compute 
		\begin{multline}
			\partial_a(\Lambda(f,g,h) - \Lambda(\tilde{f},\tilde{g},\tilde{h})  ) 
			= \sum_{b=1}^n \tp{ \Lambda(\partial_b f,g,h) - \Lambda(\partial_b \tilde{f},\tilde{g},\tilde{h}) } \partial_a(\Phi + g+h)_b \\
			+ \sum_{b=1}^n \Lambda(\partial_b \tilde{f},\tilde{g},\tilde{h}) \partial_a( g - \tilde{g} + h - \tilde{h})_b.
		\end{multline}
		Again using basic product estimates (see Corollary~\ref{corollary on tame estimates on simple multipliers}), we may deduce from this that
		\begin{multline}
			\tnorm{\Lambda(f,g,h) - \Lambda(\tilde{f},\tilde{g},\tilde{h}) }_{H^{j+1}}  \lesssim \tnorm{\Lambda(f,g,h) - \Lambda(\tilde{f},\tilde{g},\tilde{h}) }_{H^0} + \sum_{a=1}^n \tnorm{\partial_a \Lambda(f,g,h) - \partial_a \Lambda(\tilde{f},\tilde{g},\tilde{h}) }_{H^j}  
			\lesssim\\
			\tbr{\norm{g,h}_{W^{k,\infty}\times H^k} } \sum_{\abs{\alpha} \le 1} \tnorm{\Lambda(\partial^\alpha f,g,h) - \Lambda(\partial^\alpha \tilde{f},\tilde{g},\tilde{h})}_{H^j}
			+ \tnorm{g-\tilde{g},h-\tilde{h}}_{W^{k,\infty}\times H^k}\sum_{b=1}^n \tnorm{\Lambda(\partial_b \tilde{f},\tilde{g},\tilde{h})}_{H^j},
		\end{multline}
		and this bound and the induction hypothesis readily imply the continuity assertion of $\mathbb{P}_{j+1}.$  Thus $\mathbb{P}_{j+1}$ holds, and the induction argument is complete.  
		
		\textbf{Step 3}: Continuous differentiability when $m=1$. Next we aim to prove that $\Lambda$ is $C^1$ when $m=1$.  To this end fix $(f,g,h)$ in the domain of $\Lambda$ and pick $r >0$ such that $r + \max\{ \norm{g}_{W^{2+\tfloor{n/2},\infty}}, \norm{h}_{H^{2+\tfloor{n/2}}} \} < \delta$.
		Then $(f,g,h) + t (f_1,g_1,h_1)$ belongs to the domain of $\Lambda$ for every $t \in [-1,1]$ and $f_1 \in H^{2+k}(\R^n;\R^d)$, $g_1 \in W^{k,\infty}(\R^n;\R^n)$, and $h_1 \in H^k(\R^n;\R^n)$ such that $\max\{ \norm{g_1}_{W^{2+\tfloor{n/2},\infty}}, \norm{h_1}_{H^{2+\tfloor{n/2}}} \} < r$. By using the fundamental theorem of calculus, we may compute 
		\begin{equation}
			\Lambda(f+f_1,g+g_1, h+h_1) - \Lambda(f,g,h) - \Lambda(f_1,g,h) - \sum_{b=1}^n \Lambda(\partial_b f,g,h)(g_1+h_1)_b =
			\mathcal{R}_1 + \mathcal{R}_2
		\end{equation}
		for remainder terms
		\begin{equation}
			\mathcal{R}_1 = \sum_{b=1}^n (g_1+h_1)_b \int_0^1 \left(\Lambda(\partial_b f,g+tg_1,h+th_1) - \Lambda(\partial_b f, g,h) \right)dt
		\end{equation}
		and 
		\begin{equation}
			\mathcal{R}_2 = \sum_{b=1}^n (g_1+h_1)_b \int_0^1 \Lambda(\partial_b f_1,g+tg_1,h+th_1) dt.
		\end{equation}
		Using the results of the previous step, we readily deduce from these expressions that 
		\begin{equation}
			\frac{\norm{\mathcal{R}_1 + \mathcal{R}_2}_{H^k}}{ \norm{f_1}_{H^{1+k}} + \norm{g_1}_{W^{k,\infty}} + \norm{h_1}_{H^k} } \to 0 \text{ as } (f_1,g_1,h_1) \to 0.
		\end{equation}
		Hence, $\Lambda$ is differentiable on its domain, and 
		\begin{equation}\label{comp_C2_3}
			D \Lambda(f,g,h)(f_1,g_1,h_1) = \Lambda(f_1,g,h) + \sum_{b=1}^n \Lambda(\partial_b f,g,h)(g_1 +h_1)_b,  
		\end{equation}
		which may be rewritten as \eqref{comp_C2_01}.  On the other hand, the expression for $D \Lambda(f,g,h)$ in terms of $\Lambda$ shows that $D \Lambda$ is continuous on its domain when $m=1$, and so $\Lambda$ is actually $C^1$ in this case.
		
		\textbf{Step 4}: $m^{\m{th}}$-order continuous differentiability when $m\ge2$.  Now assume that $m\ge2$.  In this case, the result of the third step still shows that $\Lambda$ is $C^1$ on its domain with the same expression for $D \Lambda$ given in \eqref{comp_C2_3}. However, since $D\Lambda$ can be expressed in terms of standard products and $\Lambda$, we iteratively deduce that $\Lambda$ is actually $C^2$ when $m=2$, $C^3$ when $m=3$, etc.,  and that the formula~\eqref{comp_C2_02} holds for any such $m$.  
	\end{proof}

We also need a variant of Theorem \ref{comp_C2} in which the outer composition map is a fixed element of a $C^{k}_b-$type space, and we are considering smoothness as a mapping into the space of Sobolev multipliers. This latter space is defined as follows. For $W\subseteq\R^n$ an open set and $s\in\N$, we define the space of Sobolev multipliers $\mathcal{M}(H^s(W))$ to be the set of bounded linear maps $L\in\mathcal{L}(H^s(W))$ for which there exists $m\in L^1_{\m{loc}}(W)$ such that $Lf=mf$ for all $f\in H^s(W)$. We endow this space with the natural operator norm: $\tnorm{m}_{\mathcal{M}(H^s)}=\sup_{\tnorm{f}_{H^s}\le1}\tnorm{mf}_{H^s}$. It turns out that $\mathcal{M}(H^s(W))$ is a closed subspace of $\mathcal{L}(H^s(W))$, and is thus complete, and that $\mathcal{M}(H^s(W))\subset H^s_{\m{loc}}(W)$. For the proofs of these facts and more information on the spaces of Sobolev multipliers, we refer the reader to the book by Maz'ya and Shaposhnikova~\cite{MR2457601}.

	\begin{thm}\label{comp_C2_fixed_variant}
		Let $\N \ni k \ge2+\tfloor{n/2}$,  $\varnothing \neq V \subseteq \R^d$ be open, and  $\es\neq O\subseteq\R^n$ be a Stein-extension domain (see Definition~\ref{defn Stein-extension operator}). Suppose that $\varnothing \neq U \subseteq  W^{k,\infty}(O;\R^d) \times  H^k(O;\R^d)$ is an open set with the property that if $(g,h) \in U$, then $(g+h)(O) \subseteq V$.   For any $f \in C^{k}_b(V)$ define the function $\Theta(f;g,h) : O \to \R$ via $\Theta(f;g,h) = f(g +h)$. Then the following hold for each $m \in \N$ and $f \in C^{m+k}_b(V)$.
		\begin{enumerate}
			\item For each $(g,h) \in U$, the function $\Theta(f;g,h)$ defines an element of $\mathcal{M}(H^k(O))$.  Moreover, the induced map $\Theta(f;\cdot) : U  \to \mathcal{M}(H^k(O))$ is $C^m$.
			\item If $m\ge 1$, then the induced map $\Theta(f;\cdot):U \to  \mathcal{M}(H^k(O))$ satisfies
			\begin{equation}\label{comp_C2_fv_01}
				D^m \Theta(f;g,h) [(g_1,h_1),\dotsc,(g_m,h_m)] =  D^mf(g + h) [(g_1,h_1),\dotsc,(g_m,h_m)].   
			\end{equation}
		\end{enumerate}
	\end{thm}
	\begin{proof}

		We divide the proof into three steps.
		
		\textbf{Step 1}: Well-definedness and continuity when $m=0$. Suppose that $m=0$.  For $0 \le j \le k$ let $\mathbb{P}_j$ denote the proposition that if $f \in C^{j}_b(V)$, then $\Theta(f;\cdot) : U \to \mathcal{M}(H^j(O))$ is well-defined, continuous, and obeys the estimate 
		\begin{equation}\label{comp_C2_fv_1}
			\norm{\Theta(f;g,h)}_{\mathcal{M}(H^j)} \lesssim \norm{f}_{C^j_b} \tbr{\norm{g}_{W^{k,\infty}}, \norm{h}_{H^k} }^{j-1}.
		\end{equation}
		We will employ a finite induction to show that $\mathbb{P}_j$ holds for all $0 \le j \le k$, which then proves the first item when $m=0$.
		
		Consider the base case, $j=0$.  Then for $\varphi \in H^0(O)$ we can trivially bound $\norm{\Theta(f;g,h) \varphi}_{H^0} \le \norm{f}_{C^0_b} \norm{\varphi}_{H^0}$ in order to see that $\Theta(f;g,h) \in \mathcal{M}(H^0(O))$ with $\norm{\Theta(f;g,h)}_{\mathcal{M}(H^0)} \le \norm{f}_{C^0_b}$. Thus, $\Theta(f;\cdot)$ is well-defined and obeys the bound \eqref{comp_C2_fv_1} when $j=0$.  To prove continuity, we note that for $(g,h), (\bar{g},\bar{h}) \in U$ we may argue as above to bound $\tnorm{\Theta(f;g,h) - \Theta(f;\bar{g},\bar{h})}_{\mathcal{M}(H^0)} \le \tnorm{f(g+h)-f(\bar{g}+\bar{h})}_{C^0_b}$, but since $k > n/2$, convergence in $W^{k,\infty} \times H^k$ implies uniform convergence, and we deduce that $\lim_{(\bar{g},\bar{h}) \to (g,h)}     \tnorm{\Theta(f;g,h) - \Theta(f;\bar{g},\bar{h})}_{\mathcal{M}(H^0)} =0$. This establishes the continuity of $\Theta(f;\cdot)$ in $U$, and so $\mathbb{P}_0$ is proved.
		
		Now suppose that $\mathbb{P}_\ell$ holds for all $0 \le \ell \le j \le k-1$ and suppose that $f \in C^{j+1}_b(V)$.  Then for $\varphi \in H^{j+1}(O)$ and $(g,h) \in U$ we have that
		\begin{multline}
			\norm{\Theta(f;g,h) \varphi}_{H^{j+1}} \lesssim     \norm{\Theta(f;g,h) \varphi}_{H^{0}} + \sum_{a=1}^n      \norm{\partial_a [\Theta(f;g,h) \varphi]}_{H^{j+1}} \\
			\lesssim     \norm{\Theta(f;g,h) \varphi}_{H^{0}} + \sum_{a=1}^n      \norm{\partial_a[ \Theta(f;g,h) ]\varphi}_{H^{j}} + \sum_{a=1}^n\norm{\Theta(f;g,h) \partial_a \varphi}_{H^{j}}.
		\end{multline}
		By the induction hypothesis, we have the bounds $\norm{\Theta(f;g,h) \varphi}_{H^{0}} \le \norm{\Theta(f;g,h) }_{\mathcal{M}(H^{0})} \norm{\varphi}_{H^0}$ and $    \sum_{a=1}^n\norm{\Theta(f;g,h) \partial_a \varphi}_{H^{j}} \lesssim  \norm{\Theta(f;g,h) }_{\mathcal{M}(H^{j})} \norm{\varphi}_{H^{j+1}}$. On the other hand, 
		\begin{equation}
		\partial_a [\Theta(f;g,h)] = \sum_{b=1}^n \partial_b f(g + h) \partial_a(g+h)_b =  \sum_{b=1}^n \Theta(\partial_b f,g,h) \partial_a(g+h)_b,    
		\end{equation}
		so we may again use the induction hypothesis to bound 
		\begin{multline}
			\sum_{a=1}^n      \norm{\partial_a[ \Theta(f;g,h) ]\varphi}_{H^{j}} \lesssim \sum_{a,b=1}^n \norm{\Theta(\partial_b f,g,h)}_{\mathcal{M}(H^j)} \norm{\partial_a(g+h)_b \varphi}_{H^j} \\
			\lesssim 
			\sum_{b=1}^n \norm{\Theta(\partial_b f,g,h)}_{\mathcal{M}(H^j)}\left(\norm{g}_{W^{k,\infty}} + \norm{h}_{H^k} \right) \norm{\varphi}_{H^j}.
		\end{multline}
		Upon combining these bounds and again using the induction hypothesis, we find that $\Theta(f;g,h) \in \mathcal{M}(H^{j+1}(O))$ with 
		\begin{multline}
			\norm{\Theta(f;g,h)}_{\mathcal{M}(H^j)} \lesssim \norm{\Theta(f;g,h) }_{\mathcal{M}(H^{0})} + \norm{\Theta(f;g,h) }_{\mathcal{M}(H^{j})} \\
			+ \sum_{b=1}^n \norm{\Theta(\partial_b f,g,h)}_{\mathcal{M}(H^j)}\left(\norm{g}_{W^{k,\infty}} + \norm{h}_{H^k} \right)
			\lesssim
			\norm{f}_{C^j_b} \br{\norm{g}_{W^{k,\infty}},\norm{h}_{H^k} }^{j-1} \\
			+ \sum_{b=1}^n     \norm{\partial_b f}_{C^j_b} \br{\norm{g}_{W^{k,\infty}} + \norm{h}_{H^k} }^{j} \lesssim     \norm{f}_{C^{j+1}_b} \br{\norm{g}_{W^{k,\infty}}, \norm{h}_{H^k} }^{(j+1)-1}.
		\end{multline}
		This shows that $\Theta(f;\cdot)$ is well-defined on $U$ and obeys \eqref{comp_C2_fv_1} in the case $j+1$.  
		
		To prove $\mathbb{P}_{j+1},$ it remains only to establish continuity.  For $(g,h),(\bar{g},\bar{h}) \in U$ and $\varphi \in H^{j+1}(\R^n)$ we argue as above to bound 
		\begin{multline}
			\tnorm{[\Theta(f;g,h) - \Theta(f;\bar{g},\bar{h}) ] \varphi}_{H^{j+1}} \lesssim     \tnorm{\Theta(f;g,h) - \Theta(f;\bar{g},\bar{h})}_{\mathcal{M}(H^{0})} \tnorm{\varphi}_{H^0} \\
			+ \sum_{a=1}^n \tnorm{\Theta(f;g,h) -\Theta(f;\bar{g},\bar{h})}_{\mathcal{M}(H^{j})} 
			\tnorm{\partial_a \varphi}_{H^j} 
			+ \sum_{a=1}^n \tnorm{\partial_a[\Theta(f;g,h) -\Theta(f;\bar{g},\bar{h})] \varphi}_{H^j}.
		\end{multline}
		The first two terms here are easy to deal with, but we must rewrite the third: $\partial_a[\Theta(f;g,h) -\Theta(f;\bar{g},\bar{h})] = \sum_{b=1}^n [\Theta(\partial_b f;g,h) -\Theta(\partial_b f;\bar{g},\bar{h})] \partial_a(g+h)_b
		+ \sum_{b=1}^n \Theta(\partial_b f, \bar{g},\bar{h}) \partial_a(g+h-\bar{g}-\bar{h})_b$,   
		which then allows us to bound 
		\begin{multline}
			\sum_{a=1}^n \tnorm{\partial_a[\Theta(f;g,h) -\Theta(f;\bar{g},\bar{h})] \varphi}_{H^j} \lesssim
			\sum_{a,b=1}^n \tnorm{\Theta(\partial_b f;g,h) -\Theta(\partial_b f;\bar{g},\bar{h})}_{\mathcal{M}(H^{j})} \tnorm{\partial_a(g+h)_b \varphi  }_{H^j} \\
			+ \sum_{a,b=1}^n \tnorm{\Theta(\partial_b f;\bar{g},\bar{h})}_{\mathcal{M}(H^{j})} \tnorm{\partial_a(g+h -\bar{g}-\bar{h})_b \varphi  }_{H^j} \\
			\lesssim \sum_{b=1}^n \tnorm{\Theta(\partial_b f;g,h) -\Theta(\partial_b f;\bar{g},\bar{h})}_{\mathcal{M}(H^{j})}\left(\norm{g}_{W^{k,\infty}} + \tnorm{h}_{H^k} \right) \norm{\varphi  }_{H^j}  \\
			+ \sum_{b=1}^n  \norm{\Theta(\partial_b f;\bar{g},\bar{h})}_{\mathcal{M}(H^{j})} \left(\norm{g-\bar{g}}_{W^{k,\infty}} + \norm{h-\bar{h}}_{H^k} \right) \norm{\varphi  }_{H^j}.
		\end{multline}
		Hence, 
		\begin{multline}
			\norm{[\Theta(f;g,h) - \Theta(f;\bar{g},\bar{h}) ] }_{\mathcal{M}(H^{j+1})} \lesssim
			\norm{\Theta(f;g,h) -\Theta(f;\bar{g},\bar{h})}_{\mathcal{M}(H^{0})} \\
			+  \norm{\Theta(f;g,h) -\Theta(f;\bar{g},\bar{h})}_{\mathcal{M}(H^{j})} 
			+\left(\norm{g}_{W^{k,\infty}} + \norm{h}_{H^k} \right) \sum_{b=1}^n \norm{\Theta(\partial_b f;g,h) -\Theta(\partial_b f;\bar{g},\bar{h})}_{\mathcal{M}(H^{j})} \\
			+ \left(\norm{g-\bar{g}}_{W^{k,\infty}} + \norm{h-\bar{h}}_{H^k} \right) \sum_{b=1}^n  \norm{\Theta(\partial_b f;\bar{g},\bar{h})}_{\mathcal{M}(H^{j})}, 
		\end{multline}
		and we readily deduce from this bound and the induction hypothesis that $\Theta(f;\cdot)$ is continuous on $U$ when $f \in C^{j+1}_b(V)$.  Thus, $\mathbb{P}_{j+1}$ holds, and the finite induction argument is complete.

		\textbf{Step 2}: $C^1$ when $f \in C^{1+k}_b(V)$. Now suppose that $m=1$ and $f \in C^{1+k}_b(V)$.  For $(g,h) \in U$ and $(g_1,h_1)\in W^{k,\infty}(O;\R^d) \times H^k(O;\R^d)$ small enough that  $(g+tg_1,h+th_1) \in U$ for all $t\in[0,1]$, we have the identity
		\begin{multline}
			\Theta(f;g+g_1,h+h_1) - \Theta(f;g,h) - \sum_{b=1}^n \Theta(\partial_b f;g,h)(g_1+h_1)_b \\
			= 
			\sum_{b=1}^n (g_1+h_1)_b \int_0^1 \left[\Theta(\partial_b f; g+tg_1,h+th_1) - \Theta(\partial_b f, g,h) \right] dt.
		\end{multline}
		Using this and the results from Step 1, we may argue as in Step 3 of the proof of Theorem \ref{comp_C2} to deduce from this that $\Theta(f;\cdot)$ is differentiable on $U$ and satisfies $D \Theta(f;g,h)(g_1,h_1) = \sum_{b=1}^n \Theta(\partial_b f;g,h)(g_1+h_1)_b$, which may be rewritten as \eqref{comp_C2_fv_01} when $m=1$.  In turn, this formula and the results from the previous step show that $\Theta(f;\cdot)$ is $C^1$.

		\textbf{Step 3}: $C^m$ when $f \in C^{m+k}_b(V)$.  With the $m=1$ case established and formula \eqref{comp_C2_fv_01} in hand, we may then employ a simple induction argument, which we omit for the sake of brevity, to conclude that if $f \in C^{m+k}_b(V)$ for $m \ge 2$, then $\Theta(f;\cdot)$ is $C^m$ on $U$ with the formula \eqref{comp_C2_fv_01} holding for all $m \ge 1$.
	\end{proof}
	
	\subsection{Tools for tame estimates}\label{appendix on tools for tame estimates}

	In this subsection we record some useful tame estimates that form the basis for Section~\ref{only place he's ever been} and numerous other areas of our analysis.  We begin by recalling two versions of the log-convexity result known as Gagliardo-Nirenberg interpolation.  The first is for functions on $\R^n$.
    
	\begin{thm}[Gagliardo-Nirenberg interpolation in full space]\label{thm gagliardo nirenberg}
		Let $V$ be a finite dimensional real vector space, $1\le t,p,q\le\infty$ and $s,r\in\N^+$ be such that $1\le r<s$ and
		\begin{equation}\label{that interpolation equality convex good stuff}
			\f{r}{s}\f{1}{p}+\bp{1-\f{r}{s}}\f{1}{q}=\f{1}{t}.
		\end{equation}
		For all $\varphi\in  L^q(\R^n;V) \cap  \dot{W}^{s,p}(\R^n;V)$ we have the estimate $\tnorm{D^r\varphi}_{L^t}\lesssim\tnorm{\varphi}_{L^q}^{1-r/s}\tnorm{D^s\varphi}^{s/r}_{L^p}$, where the implicit constant depends only the dimension, $r$, $s$, $p$, and $q$.
	\end{thm}
	\begin{proof}
	See, for instance, Theorem 12.85 in Leoni~\cite{MR3726909} for the proof when $V=\R$. The case for general $V$ follows by choosing a basis and working component-wise.
	\end{proof}
	
	The second version of Gagliardo-Nirenberg is for functions in certain domains.
	
	\begin{coro}[Gagliardo-Nirenberg interpolation in domains]\label{gagliardo nirenberg interpolation in domains}
		Suppose that $V$ is a finite dimensional real vector space and $U\subseteq\R^n$ is a Stein extension domain (see Definition~\ref{defn Stein-extension operator}). Let $1\le p,q,t\le\infty$ and $s,r\in\N^+$ be such that $1\le r<s$ and~\eqref{that interpolation equality convex good stuff} holds. For all $\varphi\in(L^q\cap W^{s,p})(U;V)$ we have the estimate $\tnorm{\varphi}_{W^{r,t}}\lesssim\tnorm{\varphi}^{1-r/s}_{L^q}\tnorm{\varphi}_{W^{s,p}}^{s/r}$, where the implicit constant depends on the dimension, $r$, $s$, $p$, $q$, and $U$.
	\end{coro}
	\begin{proof}
		Let $\mathfrak{E}_U$ denote a Stein-extension operator for the domain $U$. Given $\varphi$ as in the hypotheses, we use the Theorem~\ref{thm gagliardo nirenberg} followed by the log-convexity of the norm on the Lebesgue spaces to obtain the desired bound:
		\begin{multline}
			\tnorm{\varphi}_{W^{r,t}(U;V)}\le\tnorm{\mathfrak{E}_U\varphi}_{W^{r,t}(\R^n;V)}\lesssim\tnorm{\mathfrak{E}_U\varphi}_{L^t(\R^n;V)}+\tnorm{D^r\mathfrak{E}_U\varphi}_{L^t(\R^n;V)}\lesssim\tnorm{\varphi}_{L^t(U;V)}\\+\tnorm{\varphi}^{1-r/s}_{L^q(U;V)}\tnorm{\varphi}^{s/r}_{W^{s,p}(U;V)}\le\tnorm{\varphi}_{L^q(U;V)}^{1-r/s}\tnorm{\varphi}^{r/s}_{L^p(U;V)}+\tnorm{\varphi}^{1-r/s}_{L^q(U;V)}\tnorm{\varphi}^{s/r}_{W^{s,p}(U;V)}\\\lesssim\tnorm{\varphi}^{1-r/s}_{L^q(U;V)}\tnorm{\varphi}^{s/r}_{W^{s,p}(U;V)}.
		\end{multline}
	\end{proof}
	
	Armed with the interpolation theorems, we can now prove a series of useful tame estimates. We begin by studying products. The following principal result, which has a rather lengthy statement, is the source of numerous useful and simpler corollaries.
	
	\begin{thm}[Tame estimates on products]\label{theorem on tame estimates on products}
		Let $U\subseteq\R^d$ be a Stein extension domain (see Definition~\ref{defn Stein-extension operator}), and let $\ell,m,k,\al,\be \in\N$ and $\al_1,\dots,\al_\ell,\be_1,\dots,\be_m\in\N^d$ satisfy $\sum_{i=1}^\ell|\al_i|=\al$, $\sum_{j=1}^m|\be_j|=\be$, and $\al+\be=k$.  Suppose that $r,s,t,p_1,\dots,p_\ell,q_1,\dots,q_\ell,a_1,\dots,a_m,b_1,\dots,b_m\in[1,\infty]$ satisfy $1/r=1/s+1/t$ as well as
		\begin{equation}\label{hypothesis madness}
			\f{1}{p_i}-\f{1}{q_i}=\f{k}{\al}\bp{\f{1}{s}-\sum_{\lambda=1}^\ell\f{1}{q_\lambda}}
			\text{ for }i\in\tcb{1,\dots,\ell}
			\text{ and }
			\f{1}{a_j}-\f{1}{b_j}=\f{k}{\be}\bp{\f{1}{t}-\sum_{\mu=1}^m\f{1}{b_\mu}}
			\text{ for }j\in\tcb{1,\dots,m}.
		\end{equation}
		Suppose additionally that
		\begin{equation}
			u_i\in(W^{k,p_i}\cap L^{q_i})(U) 
			\text{ for }i\in\tcb{1,\dots,\ell}
			\text{ and }
			 w_j\in(W^{k,a_j}\cap L^{b_j})(U) 
    		 \text{ for }j\in\tcb{1,\dots,m}.
		\end{equation}
		Then we have  the product estimate
		\begin{equation}\label{theorem on tame estimates on products: est}
			\bnorm{\bp{\prod_{i=1}^\ell\pd^{\al_i}u_i}\bp{\prod_{j=1}^m\pd^{\be_j}w_j}}_{L^r}\lesssim\sum_{n=1}^{m+\ell}\tnorm{v_n}_{W^{k,c_n}}\prod_{\nu \neq n}\tnorm{v_\nu}_{L^{d_n}},
		\end{equation}
		where
		\begin{equation}\label{newnotation}
			v_n=\begin{cases}
				u_n&\text{if }n\le \ell,\\
				w_{n-\ell}&\text{if }\ell<n,
			\end{cases}
			\quad c_n=\begin{cases}
				p_n&\text{if }n\le\ell,\\
				a_{n-\ell}&\text{if }\ell<n,
			\end{cases}
			\text{ and }
			 d_n=\begin{cases}
				q_n&\text{if }n\le\ell,\\
				b_{n-\ell}&\text{if }\ell<n.
			\end{cases}
		\end{equation}
	\end{thm}
	\begin{proof}
		For $i\in\tcb{1,\dots,\ell}$ and $j\in\tcb{1,\dots,m}$, we define $s_i,t_j\in[1,\infty]$ via
		\begin{equation}\label{some intermediate exponents}
			\f{1}{s_i}=\f{|\al_i|/k}{p_i}+\f{1-|\al_i|/k}{q_i}
			\text{ and }
			\f{1}{t_j}=\f{|\be_j|/k}{a_j}+\f{1-|\be_j|/k}{b_j}.
		\end{equation}
		We first claim that $\sum_{i=1}^\ell1/s_i=1/s$ and $\sum_{j=1}^m1/t_j=1/t$.  Indeed, by invoking  \eqref{hypothesis madness}, we see that equation~\eqref{some intermediate exponents} is equivalent to
		\begin{equation}
			\f{1}{s_i}-\f{1}{q_i}=\f{|\al_i|}{k}\cdot\f{k}{\al}\bp{\f{1}{s}-\sum_{\lambda=1}^\ell\f{1}{q_\lambda}}
			\text{ and }
			\f{1}{t_j}-\f{1}{b_j}=\f{|\be_j|}{k}\cdot\f{k}{\be}\bp{\f{1}{t}-\sum_{\mu=1}^m\f{1}{b_\mu}}
		\end{equation}
		for $i\in\tcb{1,\dots,\ell}$ and $j\in\tcb{1,\dots,m}$. Now we sum over $i$ and $j$ in these ranges and use the definition of $\al$ and $\be$ to see that
		\begin{equation}
			\sum_{i=1}^\ell\bp{\f{1}{s_i}-\f{1}{q_i}}=\f{1}{s}-\sum_{i=1}^\ell\f{1}{q_i}
			\text{ and }
			\sum_{j=1}^m\bp{\f{1}{t_j}-\f{1}{b_j}}=\f{1}{t}-\sum_{j=1}^m\f{1}{b_j},
		\end{equation}
		from which the claim immediately follows.
        
        We next note that by hypothesis  $1/r=1/s+1/t$, and so the claim allows us to begin the product estimate by applying H\"older's inequality:
		\begin{equation}\label{application of Holders}
			\bnorm{\bp{\prod_{i=1}^\ell\pd^{\al_i}u_i}\bp{\prod_{j=1}^m\pd^{\be_j}w_j}}_{L^r}\le\bp{\prod_{i=1}^{\ell}\tnorm{u_i}_{W^{|\al_i|,s_i}}}\bp{\prod_{j=1}^m\tnorm{w_j}_{W^{|\be_j|,t_j}}}.
		\end{equation}
		We then use~\eqref{some intermediate exponents} and apply Gagliardo-Nirenberg interpolation, Corollary~\ref{gagliardo nirenberg interpolation in domains}, to bound
		\begin{equation}\label{application of GNI}
			\tnorm{u_i}_{W^{|\al_i|,s_i}}\lesssim\tnorm{u_i}_{W^{k,p_i}}^{|\al_i|/k}\tnorm{u_i}_{L^{q_i}}^{1-|\al_i|,k}
			\text{ and }
			\tnorm{w_j}_{W^{|\al_j|,s_j}}\lesssim\tnorm{w_j}_{W^{k,a_j}}^{|\be_j|/k}\tnorm{w_j}_{L^{b_j}}^{1-|\be_i|/k}.
		\end{equation}
		We set $\gam_n=\al_n$ if $n\le\ell$ and $\gam_n=\be_{n-\ell}$ if $n>\ell$ and note that $\sum_{n=1}^{m+\ell} \abs{\gam_n}/k = (\alpha + \beta)/k =1$.  The latter identity implies that if $\{B_n\}_{n=1}^{m+\ell} \subseteq [0,\infty)$, then 
		\begin{equation}\label{square_product_ident}
		    \prod_{n=1}^{m+\ell} B_n^{1-\abs{\gam_n}/k} =  \prod_{n=1}^{m+\ell} B_n^{\sum_{\nu \neq n} \abs{\gam_\nu}/k} = \prod_{n=1}^{m+\ell} \prod_{\nu \neq n} B_n^{\abs{\gam_\nu}/k} =  \prod_{n=1}^{m+\ell} \prod_{\nu \neq n} B_\nu^{\abs{\gam_n}/k}.
		\end{equation}
		We then combine estimates~\eqref{application of Holders} and~\eqref{application of GNI}, while using the notation of~\eqref{newnotation} and the identity \eqref{square_product_ident}, to arrive at the estimate
		\begin{equation}
			\bnorm{\bp{\prod_{i=1}^\ell\pd^{\al_i}u_i}\prod_{j=1}^m\pd^{\be_j}w_j}_{L^r}
			\lesssim \prod_{n=1}^{m+\ell}\tnorm{v_n}^{|\gam_n|/k}_{W^{k,c_n}}\tnorm{v_n}^{1-|\gam_n|/k}_{L^{d_n}}=\prod_{n=1}^{m+\ell}\bp{\tnorm{v_n}_{W^{k,c_n}}\prod_{\nu \neq n}\tnorm{v_\nu}_{L^{d_n}}}^{|\gam_n|/k}.
		\end{equation}
		Then \eqref{theorem on tame estimates on products: est} follows from this via an application of Young's inequality.
	\end{proof}
	
	We now can derive more familiar-looking high-low type product estimates.

	\begin{coro}[Tame estimates on simple multipliers]\label{corollary on tame estimates on simple multipliers}
		Let $U\subseteq\R^d$ be a Stein extension domain (see Definition~\ref{defn Stein-extension operator}), $k\in\N$, and $\varphi\in H^k(U)$. The following hold.
		\begin{enumerate}
			\item Suppose that $\al,\be\in\N^d$ are such that $|\al|+|\be|=k$ and $\psi\in H^{\max\tcb{1+\tfloor{d/2},k}}(U)$. Then the product $\pd^\al\psi\pd^\be\varphi$ belongs to $L^2(U)$ and obeys the bound
			\begin{equation}\label{a bound for me and you}
				\tnorm{\pd^\al\psi\pd^\be\varphi}_{L^2}\lesssim\tnorm{\psi}_{H^{1+\tfloor{d/2}}}\tnorm{\psi}_{H^k}+\begin{cases}
					0&\text{if }k\le\tfloor{d/2},\\
					\tnorm{\psi}_{H^k}\tnorm{\varphi}_{H^{1+\tfloor{d/2}}}&\text{if }1+\tfloor{d/2}<k.
				\end{cases}
			\end{equation}
			\item Suppose that $\psi\in H^{\max\tcb{1+\tfloor{d/2},k}}(U)$. Then the product $\psi\varphi$ belongs to $H^k(U)$ and obeys the bound
			\begin{equation}
				\tnorm{\psi\varphi}_{H^k}\lesssim\tnorm{\psi}_{H^{1+\tfloor{d/2}}}\tnorm{\psi}_{H^k}+\begin{cases}
					0&\text{if }k\le1+\tfloor{d/2},\\
					\tnorm{\psi}_{H^k}\tnorm{\varphi}_{H^{1+\tfloor{d/2}}}&\text{if }1+\tfloor{d/2}<k.
				\end{cases}
			\end{equation}
			\item Suppose that $\al,\be\in\N^d$ are such that $|\al|+|\be|=k$ and $\psi\in W^{k,\infty}(U)$. Then the product $\pd^\al\psi\pd^\be\varphi$ belongs to $L^2(U)$ and obeys the bound
			\begin{equation}
				\tnorm{\pd^\al\psi\pd^\be\varphi}_{L^2}\lesssim\tnorm{\psi}_{L^\infty}\tnorm{\varphi}_{H^k}+\tnorm{\psi}_{W^{k,\infty}}\tnorm{\varphi}_{L^2}.
			\end{equation}
			\item Suppose that $\psi\in W^{k,\infty}(U)$. Then the product $\psi\varphi$ belongs to $H^k(U)$ and obeys the bound
			\begin{equation}
				\tnorm{\psi\varphi}_{H^k}\lesssim\tnorm{\psi}_{L^\infty}\tnorm{\varphi}_{H^k}+\tnorm{\psi}_{W^{k,\infty}}\tnorm{\varphi}_{L^2}.
			\end{equation}
		\end{enumerate}
	\end{coro}
	\begin{proof}
		The second and fourth items follow directly from the Leibniz rule and the first and third items, respectively.   The third item follows as a direct application of Theorem~\ref{theorem on tame estimates on products} in the case that $\ell=2$, $m=0$, $r=2$, $p_1=2$, $p_2=\infty$, $q_1=2$, $q_2=\infty$. It remains to prove the first item.  In the case $1+\tfloor{d/2}< k$, the estimate of the first item again follows from Theorem~\ref{theorem on tame estimates on products}, applied with $\ell=2$, $m=0$, $r=2$, $p_1=p_2=2$, $q_1=q_2=\infty$, combined with the supercritical Sobolev embedding $H^{1+\tfloor{d/2}}(U)\emb L^\infty(U)$. 
		
		Now consider the remaining case,  $k\le\tfloor{d/2}$. We consider first what happens when $k<d/2$.  In this case we can invoke the subcritical Sobolev embedding $H^k(U)\emb L^{\f{2d}{d-2k}}(U)$ to see that $\varphi\in(W^{k,p_1}\cap L^{q_1})(U)$ for $p_1=2$ and $q_1=\f{2d}{d-2k}$. On the other hand, the supercritical Sobolev embedding, 	Corollary~\ref{gagliardo nirenberg interpolation in domains}, and the estimate $1+\tfloor{d/2}>d/2>k$ show that we have the embeddings 
		\begin{equation}
		H^{1+\tfloor{d/2}}(U)\emb H^{1+\tfloor{d/2}}(U) \cap L^\infty(U) \emb  W^{k,2(1+\tfloor{d/2})/k}(U)\emb W^{k,d/k}(U).    
		\end{equation}
		This implies that $\psi\in(L^\infty\cap H^{1+\tfloor{d/2}})(U)$.  We set $p_2=d/k$ and $q_2=\infty$ and note that the hypotheses of Theorem~\ref{theorem on tame estimates on products} are satisfied with $r=2$, $\ell=2$, and $m=0$ again; the bound \eqref{a bound for me and you} when $k < d/2$ then follows from the product estimate provided by the theorem.
		
		Finally, we consider the case $k=d/2$. Thanks to Corollary~\ref{gagliardo nirenberg interpolation in domains} again, we obtain that $\psi\in H^{1+d/2}(U)\emb W^{k,(4+2d)/d}(U)$, which motivates setting $p_2=\f{4+2d}{d}$ and $q_2=\infty$.  The critical Sobolev embedding implies that $\varphi\in H^{k}(U)\emb L^{2+d}(U)$, which dictates that we set $p_1=2$, $q_1=2+d$, and $r=2$.  Using these parameters, we again apply Theorem~\ref{theorem on tame estimates on products} to get \eqref{a bound for me and you} when $k=d/2$, which completes the proof of the first item.
	\end{proof}
	
	Next we consider tame bounds on commutator expressions.

	\begin{coro}[Tame estimates on commutators]\label{coro on tames estimates on commutators}
		Let $U\subseteq\R^d$ be a Stein extension domain (see Definition~\ref{defn Stein-extension operator}), $k,m\in\N$ with $m\ge 1$, $j\in\tcb{1,\dots,d}$, and $\varphi\in H^{k+m-1}(U)$. The following hold.
		\begin{enumerate}
			\item Suppose that $\psi\in H^{\max\tcb{2+\tfloor{d/2},k+m}}_{\loc}(U)$ satisfies $\pd_j\psi\in H^{\max\tcb{1+\tfloor{d/2},k+m-1}}(U)$. Then the commutator $[\pd_j^m,\psi]\varphi$ belongs to $H^{k}(U)$ and obeys the bound
			\begin{equation}
				\tnorm{[\pd_j^m,\psi]\varphi}_{H^k}\lesssim\tnorm{\pd_j\psi}_{H^{1+\tfloor{d/2}}}\tnorm{\varphi}_{H^{k+m-1}}+\begin{cases}
					0&\text{if }k+m\le2+\tfloor{d/2},\\
					\tnorm{\pd_j\psi}_{H^{k+m-1}}\tnorm{\varphi}_{H^{1+\tfloor{d/2}}}&\text{if }2+\tfloor{d/2}<k+m.
				\end{cases}
			\end{equation}
			\item Suppose that $\psi\in W^{k+m,\infty}_{\loc}(U)$ satisfies $\pd_j\psi\in W^{k+m-1,\infty}(U)$. Then the commutator $[\pd_j^m,\psi]\varphi$ belongs to $H^k(U)$ and obeys the bound
			\begin{equation}
				\tnorm{[\pd_j^m,\psi]\varphi}_{H^k}\lesssim\tnorm{\pd_j\psi}_{L^\infty}\tnorm{\varphi}_{H^{k+m-1}}+\tnorm{\pd_j\psi}_{W^{k+m-1,\infty}}\tnorm{\varphi}_{L^2}.
			\end{equation}
		\end{enumerate}
	\end{coro}
	\begin{proof}
		Fix $\al\in\N^d$ with $|\al|\le k$. By the Leibniz rule we have
		\begin{equation}
			\pd^\al([\pd_j^m,\psi]\varphi)=\sum_{\ell=0}^{m-1}\sum_{\be\le\al}c_{\ell,m,\be,\al}\pd^\be\pd_j^{\ell}\pd_j\psi\pd^{\al-\be}\pd_j^{m-1-\ell}\varphi,
		\end{equation}
		and hence
		\begin{equation}
			\tnorm{[\pd_j^m,\psi]\varphi}_{H^k}\lesssim\sum_{|\al|\le k}\sum_{\ell=0}^{m-1}\sum_{\be\le\al}\tnorm{\pd^\be\pd_j^{\ell}\pd_j\psi\pd^{\al-\be}\pd_j^{m-1-\ell}\varphi}_{L^2}.
		\end{equation}
		To prove the first item, we now apply the first conclusion of Corollary~\ref{corollary on tame estimates on simple multipliers} and use that $|\al|\le k$. To prove the second item, we instead use the third conclusion of the aforementioned corollary and that $\ell\le m-1$ in these sums.
	\end{proof}

	Now we derive tame bounds on superposition multipliers.

	\begin{coro}[Tame estimates on superposition multipliers]\label{coro on tame estimates on superposition multipliers}
		Let $U\subseteq\R^d$ be a Stein extension domain (see Definition~\ref{defn Stein-extension operator}), $k\in\N^+$, and $V$ be a finite dimensional real vector space.  Suppose that $f\in W^{k,\infty}(V)$, $g\in W^{k,\infty}_{\loc}(U;V)$ is such that $Dg\in W^{k-1,\infty}(U;\mathcal{L}(\R^d;V))$, and $\psi\in H^k_{\loc}(U;V)$ is such that $D\psi\in (L^\infty\cap H^{k-1})(U;\mathcal{L}(\R^d;V))$. Let $R\in \R^+$ and $S\in \R^+ \cup \{\infty\}$  satisfy
		\begin{equation}
			\tbr{\tnorm{f}_{W^{1,\infty}},\tnorm{Dg}_{L^\infty},\tnorm{D\psi}_{L^2\cap L^\infty}}\le R \text{ and }
			\tbr{\tnorm{f}_{W^{2+\tfloor{d/2},\infty}},\tnorm{Dg}_{W^{1+\tfloor{d/2},\infty}},\tnorm{D\psi}_{H^{1+\tfloor{d/2}}}}\le S.
		\end{equation}
		Then the following hold.
		\begin{enumerate}
			\item If $\varphi\in(L^\infty\cap H^k)(U)$, then the product $f(g+\psi)\varphi$ belongs to $H^k(U)$ and obeys the estimate
			\begin{equation}\label{the first items estimate}
				\tnorm{f(g+\psi)\varphi}_{H^k}\lesssim R^{k+1}\tnorm{\varphi}_{H^k}+R^k\tnorm{f,Dg,D\psi}_{W^{k,\infty}\times W^{k-1,\infty}\times H^{k-1}}\tnorm{\varphi}_{L^2\cap L^\infty}.
			\end{equation}
			\item If $\varphi\in H^k(U)$ and we assume additionally that $S<\infty$, then the product $f(g+\psi)\varphi$ belongs to $H^k(U)$ and obeys the estimate
			\begin{equation}\label{the introductino is net}
				\tnorm{f(g+\psi)\varphi}_{H^k}\lesssim S^{k+1}\tnorm{\varphi}_{H^k}+S^k\begin{cases}
					0&\text{if }k\le 1+\tfloor{d/2},\\
					\tnorm{f,Dg,D\psi}_{W^{k,\infty}\times W^{k-1,\infty}\times H^{k-1}}\tnorm{\varphi}_{H^{1+\tfloor{d/2}}}&\text{if }1+\tfloor{d/2}<k.
				\end{cases}
			\end{equation}
			\item Suppose that $\varphi\in H^k(U)$, $S<\infty$, and there is a Stein-extension domain $O\subseteq V$ such that $f\in W^{k,\infty}(O)$. If the image of $g+\psi$ is a subset of $O$, then the inclusion and estimate from the second item hold.
		\end{enumerate}
	\end{coro}
	\begin{proof}
		We begin by proving the first item. We start by estimating
		\begin{equation}\label{norm estimate low plus high no inbetween}
			\tnorm{f(g+\psi)\varphi}_{H^k}\lesssim\tnorm{f(g+\psi)\varphi}_{L^2}+\tnorm{D^k[f(g+\psi)\varphi]}_{L^2}
		\end{equation}
		and $\tnorm{f(g+\psi)\varphi}_{L^2}\le\tnorm{f}_{L^\infty}\tnorm{\varphi}_{L^2}$. For the second term in \eqref{norm estimate low plus high no inbetween} we have to work harder. By the differentiation rules for products and compositions, we have that
		\begin{multline}
			D^k[f(g+\psi)\varphi]=f(g+\psi)D^k\varphi+\m{sym}\sum_{j=0}^{k-1}\sum_{\ell=1}^{j+1}\sum_{j_1+\cdots+j_\ell=j+1}c_{j+1,k,\ell,j_1,\dots,j_\ell}
			\\\cdot D^{\ell}f(g+\psi)\tcb{D^{j_1}(g+\psi),\dots,D^{j_\ell}(g+\psi)}\otimes D^{k-1-j}\varphi
			=\bf{I}+\bf{II},
		\end{multline}
		where $\m{sym}$ denotes symmetrization of the multilinear map. We will estimate the $L^2$-norm of $\bf{I}$ and $\bf{II}$ separately. We handle $\bf{I}$ trivially:
		\begin{equation}\label{estimate for I_}
			\tnorm{\bf{I}}_{L^2}\le\tnorm{f}_{L^\infty}\tnorm{\varphi}_{H^k}.
		\end{equation}
		
		For $\bf{II}$, we use the triangle inequality and study each term in the series:
		\begin{equation}
			\tnorm{\bf{II}}_{L^2}\lesssim\sum_{j=0}^{k-1}\sum_{\ell=1}^{j+1}\sum_{j_1+\cdots+j_\ell=j+1}\tnorm{D^{\ell}f(g+\psi)\tcb{D^{j_1}(g+\psi),\dots,D^{j_\ell}(g+\psi)}\otimes D^{k-1-j}\varphi}_{L^2}.
		\end{equation}
		By the $\ell$-multilinearity of $D^\ell f$ and symmetry considerations, we have the upper bound
		\begin{equation}\label{monkberry moon delight try some of this honey}
			\tnorm{\bf{II}}_{L^2}\lesssim\sum_{j=0}^{k-1}\sum_{\ell=1}^{j+1}\sum_{j_1+\cdots+j_\ell=j+1}\sum_{\upnu=0}^\ell\tnorm{D^\ell f(g+\psi)\tcb{D^{j_1}\psi,\dots,D^{j_\upnu}\psi,D^{j_{\upnu+1}}g,\dots,D^{j_\ell}g}\otimes D^{k-1-j}\varphi}_{L^2}.
		\end{equation}
		We would like to apply Theorem~\ref{theorem on tame estimates on products} to the summands in this expression, but due to the appearance of $D^\ell f(g+\psi)$ terms we cannot verify the theorem's hypotheses.  Instead, we modify the argument used to prove the theorem. Given $j\in\tcb{0,\dots,k-1}$, $\ell\in\tcb{1,\dots,j+1}$, $j_1+\cdots+j_\ell=j+1$, and $\upnu\in\tcb{0,\dots,\ell}$, we set $\al=k-1-j+\sum_{\nu=1}^\upnu(j_\nu-1)$, $\be=\ell-1+\sum_{\nu=\upnu+1}^\ell(j_\nu-1)$, $\al+\be=k-1$, $r=s=2$, $t=a_1=b_1=\cdots=a_{\ell-\upnu+1}=b_{\ell-\upnu+1}=\infty$, $p_1=\cdots= p_{\upnu+1}=2$, $q_1=\cdots=q_{\upnu+1}=2\f{(k-1)(\upnu+1)-\al}{k-1-\al}$, $u_1 = \cdots = u_{\upnu}=D\psi$, $u_{\upnu+1}=\varphi$, $w_1 = Df$, and $w_2 = \cdots = w_{\ell-\upnu+1}=Dg$.  The  argument used to prove Theorem \ref{theorem on tame estimates on products} then pushes through for the summands in \eqref{monkberry moon delight try some of this honey} with these parameters, thanks to the trivial bound $\tnorm{D^\ell f(g+\psi)}_{L^\infty}\le\tnorm{D^\ell f}_{L^\infty}$;  this results in the estimate
		\begin{multline}\label{estimate for II_}
			\tnorm{\bf{II}}_{L^2}\lesssim\tbr{\tnorm{Dg,D\psi}_{L^\infty\times(L^2\cap L^\infty)}}^{k-1}\big(\tnorm{Df}_{L^\infty}\tnorm{Dg,D\psi}_{L^\infty\times(L^2\cap L^\infty)}\tnorm{\varphi}_{H^{k-1}}\\+\tnorm{Df}_{L^\infty}\tnorm{Dg,D\psi}_{W^{k-1,\infty}\times H^{k-1}}\tnorm{\varphi}_{L^2\cap L^\infty}+\tnorm{Df}_{W^{k-1,\infty}}\tnorm{Dg,D\psi}_{L^\infty\times(L^2\cap L^\infty)}\tnorm{\varphi}_{L^2\cap L^\infty}\big).
		\end{multline}
		Upon combining estimates~\eqref{norm estimate low plus high no inbetween}, \eqref{estimate for I_}, and~\eqref{estimate for II_}, we arrive at the desired conclusion, estimate~\eqref{the first items estimate}, of the first item. 
		
		The second item in the case that $1+\tfloor{d/2}\le k$ follows from estimate~\eqref{the first items estimate} of the first item, the fact that $1\le R\lesssim S$, and the supercritical Sobolev embedding $H^{1+\tfloor{d/2}}\emb L^2\cap L^\infty$. On the other hand, we have the trivial estimate
  \begin{equation}
      \tnorm{f(g+\psi)\varphi}_{L^2}\le\tnorm{f}_{L^\infty}\tnorm{\varphi}_{L^2}\le S\tnorm{\varphi}_{L^2}.    
  \end{equation}
  This shows that the linear map $\varphi\mapsto f(g+\psi)\varphi$ has $S^{2+\tfloor{d/2}}$ as an upper-bound on the $\mathcal{L}(H^{1+\tfloor{d/2}}(U))$ operator norm and has $S$ as an upper-bound on the $\mathcal{L}(L^2(U))$ operator norm.  Employing operator interpolation (see, for instance, Bergh and L\"ofstr\"om~\cite{MR0482275}) with these bounds, we achieve~\eqref{the introductino is net} in the cases $k\le 1+\tfloor{d/2}$, which completes the proof of the second item.
  
  The third item follows from the second item applied when $f$ is replaced by $\mathfrak{E}_Of$, where $\mathfrak{E}_O$ is a Stein extension operator for $O$, and the observation that the image hypothesis on $g+\psi$ ensures that $(\mathfrak{E}_Of)(g+\psi)=f(g+\psi)$.
	\end{proof}
	
	Now we consider superposition on its own, not as a multiplier.

	\begin{coro}[Tame estimates on superposition]\label{coro on tame estimates on superposition}
		Let $\N\ni k\ge2+\lfloor n/2\rfloor$. The following hold.
		\begin{enumerate}
			\item Let $V$, $W$ be finite dimensional real vector spaces, $O\subseteq V$ be a Stein extension domain (see Definition~\ref{defn Stein-extension operator}) containing $0$ that is star shaped with respect to $0$, and $U\subseteq\R^n$ be a Stein extension domain. If $f\in W_{\loc}^{k,\infty}(O;W)$ is such that $f(0)=0$ and $Df\in W^{k-1,\infty}(O;\mathcal{L}(V;W))$, and $\varphi\in H^k(U;V)$ is such that $\varphi(U)\subseteq O$, then the superposition $f(\varphi)$ belongs to $H^k(U;W)$ and obeys the estimate 
			\begin{equation}
				\tnorm{f(\varphi)}_{H^k}\lesssim S^k\tnorm{\varphi}_{H^k},
			\end{equation}
			for any $S\in\R^+$ satisfying $\tbr{\tnorm{Df}_{W^{k-1,\infty}},\tnorm{\varphi}_{H^{2+\lfloor n/2\rfloor}}}\le S$.
			\item Suppose that $f\in H^k(\R^n;\R^d)$ and $g,\psi\in C^1(\R^n;\R^n)$ are such that $Dg\in W^{k-1,\infty}(\R^n;\R^{n\times n})$, $D\psi\in H^{k-1}(\R^n;\R^{n\times n})$, and $g+\psi$ is a bi-Lipschitz and $C^1$-diffeomorphism of $\R^n$. The superposition $f(g+\psi)$ belongs to $H^k(\R^n;\R^d)$  and satisfies the estimate
			\begin{equation}\label{tame estimates on superposition estimate numero dos}
				\tnorm{f(g+\psi)}_{H^k}\lesssim S^{k+1}\tnorm{f}_{H^k}+S^k\tnorm{Dg,D\varphi}_{W^{k-1,\infty}\times H^{k-1}}\tnorm{f}_{H^{2+\lfloor n/2\rfloor}},
			\end{equation}
			for any $S\in\R^+$ satisfying
			\begin{equation}
				\tbr{\tnorm{\det D(g+\psi)^{-1}}_{L^\infty},\tnorm{Dg,D\psi}_{W^{1+\tfloor{n/2},\infty}\times H^{1+\tfloor{n/2}}}}\le S.
			\end{equation}
		\end{enumerate}
	\end{coro}
	\begin{proof}
		We begin by proving the first item. Since $f(0)=0$ and $O$ is star-shaped with respect to the origin of $V$, we can use the fundamental theorem of calculus to obtain the equality $f(\varphi)=\int_0^1Df(t\varphi)[\varphi]\;\m{d}t$, from which we deduce that $\tnorm{f(\varphi)}_{L^2}\le\tnorm{Df}_{L^\infty}\tnorm{\varphi}_{L^2}$. We next establish an $H^{k-1}$-bound on the derivative of the superposition, $D(f(\varphi))=Df(\varphi)D\varphi$. For this we utilize the third item of Corollary~\ref{coro on tame estimates on superposition multipliers} and obtain the bound $\tnorm{D(f(\varphi))}_{H^{k-1}}\lesssim S^k\tnorm{D\varphi}_{H^{k-1}}$. Together, these estimate give the conclusion of the first item.
		
		We next prove the second item, noting initially that
		\begin{equation}\label{the equation who shall not pass}
			\tnorm{f(g+\psi)}_{H^k}\lesssim\tnorm{f(g+\psi)}_{L^2}+\tnorm{D^k[f(g+\psi)]}_{L^2}.
		\end{equation}
		For the first $L^2$-norm, we use that $(g+\psi)$ is a bi-Lipschitz and $C^1$ diffeomorphism to estimate
		\begin{equation}\label{the equation who will pass}
			\tnorm{f(g+\psi)}_{L^2}=\bp{\int_{\R^n}|f|^2|\det D(g+\psi)^{-1}|}^{1/2}\le S^{1/2}\tnorm{f}_{L^2}.
		\end{equation}
		For the latter term of~\eqref{the equation who shall not pass}, we would like to use Theorem \ref{theorem on tame estimates on products}, but as in the proof of Corollary \ref{coro on tame estimates on superposition multipliers}, we cannot quite do so.  Instead, we use the differentiation rules for composition and argue in a manner similar to the proof of the first item in Corollary~\ref{coro on tame estimates on superposition multipliers}.  Indeed, we have the formula 
		\begin{equation}
			D^k[f(g+\psi)]=\m{sym}\sum_{\ell=1}^k\sum_{j_1+\dots+j_\ell=k}c_{\ell,k,j_1,\dots,j_\ell}D^\ell f(g+\psi)\tcb{D^{j_1}(g+\psi),\dots,D^{j_\ell}(g+\psi)},
		\end{equation}
		where $\m{sym}$ denotes the symmetrization operator for multilinear maps. By multilinearity and symmetry considerations, we then deduce the upper bound
		\begin{equation}\label{i pushed her away, i walked to the door}
			\tnorm{D^k[f(g+\psi)]}_{L^2}\lesssim
			\sum_{\ell=1}^k\sum_{j_1+\cdots+j_\ell=k}\sum_{\upnu=0}^\ell\tnorm{D^\ell f(g+\psi)\tcb{D^{j_1}\psi,\dots D^{j_\upnu}\psi,D^{j_{\upnu+1}}g,\dots,D^{j_\ell}g}}_{L^2}.
		\end{equation}
		To each summand on the right hand side above we apply the argument from the proof of Theorem~\ref{theorem on tame estimates on products} as follows. Given $\ell\in\tcb{1,\dots,k}$, $j_1+\cdots+j_\ell=k$, and $\upnu\in\tcb{0,1,\dots,\ell}$, we define $\al=\ell-1+\sum_{\nu=1}^\upnu(j_\nu-1)$, $\be=\sum_{\nu=\upnu+1}^\ell(j_\nu-1)$, $\al+\be=k-1$, $r=s=2$, $t=a_1=b_1=\cdots=a_{\ell-\upnu}=b_{\ell-\upnu}=\infty$, $p_1=\cdots=p_{\upnu+1}=2$, $q_1=\cdots=q_{\upnu+1}=2\f{(k-1)(\upnu+1)-\al}{k-1-\al}$, $u_1 = Df$, $u_2 = \cdots = u_{\upnu+1}=D\psi$, and $w_1 = \cdots = w_{\ell-\upnu}=Dg$.        The argument used in the theorem pushes through so long as we carry an extra factor of $S$, thanks to the bounds
		\begin{equation}
			\tnorm{D^\ell f(g+\psi)}_{L^\chi}\le S^{1/\chi}\tnorm{D^\ell f}_{L^\chi}\le S\tnorm{D^\ell f}_{L^\chi} \text{ for all }  \chi\in[1,\infty].
		\end{equation}
		Therefore, we deduce the estimate
		\begin{multline}\label{and thats the way that i wanted to stay}
			\tnorm{D^\ell f(g+\psi)\tcb{D^{j_1}\psi,\dots D^{j_\upnu}\psi,D^{j_{\upnu+1}}g,\dots,D^{j_\ell}g}}_{L^2}\\\lesssim S\tbr{\tnorm{Dg,D\psi}_{L^\infty\times(L^2\cap L^\infty)}}^{\ell-1}\big(\tnorm{Df}_{H^{k-1}}\tnorm{Dg,D\psi}_{L^\infty\times (L^2\cap L^\infty)}\\+\tnorm{Df}_{L^2\cap L^\infty}\tnorm{Dg,D\psi}_{W^{k-1,\infty}\times H^{k-1}}\big)
			\lesssim S^{\ell+1}\tnorm{f}_{H^k}+S^\ell\tnorm{Dg,D\psi}_{W^{k-1,\infty}\times H^{k-1}}\tnorm{f}_{H^{2+\lfloor n/2\rfloor}}.
		\end{multline}
		Upon combining~\eqref{the equation who shall not pass}, \eqref{the equation who will pass}, \eqref{i pushed her away, i walked to the door}, and~\eqref{and thats the way that i wanted to stay}, we acquire the desired bound, \eqref{tame estimates on superposition estimate numero dos}. 
	\end{proof}
	
	%-%-% This is the end of the document yay %%%---%%%---%%%%-%-% This is the end of the document yay %%%---%%%---%%%%-%-% This is the end of the document yay %%%---%%%---%%%%-%-% This is the end of the document yay %%%---%%%---%%%%-%-% This is the end of the document yay %%%---%%%---%%%%-%-% This is the end of the document yay %%%---%%%---%%%%-%-% This is the end of the document yay %%%---%%%---%%%%-%-% This is the end of the document yay %%%---%%%---%%%%-%-% This is the end of the document yay %%%---%%%---%%%%-%-% This is the end of the document yay %%%---%%%---%%%%-%-% This is the end of the document yay %%%---%%%---%%%%-%-% This is the end of the document yay %%%---%%%---%%%
        {\small\printindex}
	\bibliographystyle{abbrv}
	\bibliography{cnstw.bib}

\begin{thebibliography}{100}

\bibitem{AbMaRa_1988}
R.~Abraham, J.~E. Marsden, and T.~Ratiu.
\newblock {\em Manifolds, tensor analysis, and applications}, volume~75 of {\em
  Applied Mathematical Sciences}.
\newblock Springer-Verlag, New York, second edition, 1988.

\bibitem{MR2424078}
R.~A. Adams and J.~J.~F. Fournier.
\newblock {\em Sobolev spaces}, volume 140 of {\em Pure and Applied Mathematics
  (Amsterdam)}.
\newblock Elsevier/Academic Press, Amsterdam, second edition, 2003.

\bibitem{MR162050}
S.~Agmon, A.~Douglis, and L.~Nirenberg.
\newblock Estimates near the boundary for solutions of elliptic partial
  differential equations satisfying general boundary conditions. {II}.
\newblock {\em Comm. Pure Appl. Math.}, 17:35--92, 1964.

\bibitem{Auchmuty_Beals_1971}
J.~F.~G. Auchmuty and R.~Beals.
\newblock Variational solutions of some nonlinear free boundary problems.
\newblock {\em Arch. Rational Mech. Anal.}, 43:255--271, 1971.

\bibitem{MR3711883}
P.~Baldi and E.~Haus.
\newblock A {N}ash-{M}oser-{H}\"{o}rmander implicit function theorem with
  applications to control and {C}auchy problems for {PDE}s.
\newblock {\em J. Funct. Anal.}, 273(12):3875--3900, 2017.

\bibitem{refId0}
{Bauer, Sebastian}, {Neff, Patrizio}, {Pauly, Dirk}, and {Starke, Gerhard}.
\newblock Dev-div- and devsym-devcurl-inequalities for incompatible square
  tensor fields with mixed boundary conditions.
\newblock {\em ESAIM: COCV}, 22(1):112--133, 2016.

\bibitem{MR445136}
J.~T. Beale.
\newblock The existence of solitary water waves.
\newblock {\em Comm. Pure Appl. Math.}, 30(4):373--389, 1977.

\bibitem{MR611750}
J.~T. Beale.
\newblock The initial value problem for the {N}avier-{S}tokes equations with a
  free surface.
\newblock {\em Comm. Pure Appl. Math.}, 34(3):359--392, 1981.

\bibitem{MR0482275}
J.~Bergh and J.~L\"{o}fstr\"{o}m.
\newblock {\em Interpolation spaces. {A}n introduction}.
\newblock Grundlehren der Mathematischen Wissenschaften, No. 223.
  Springer-Verlag, Berlin-New York, 1976.

\bibitem{MR2580515}
M.~Berti, P.~Bolle, and M.~Procesi.
\newblock An abstract {N}ash-{M}oser theorem with parameters and applications
  to {PDE}s.
\newblock {\em Ann. Inst. H. Poincar\'{e} C Anal. Non Lin\'{e}aire},
  27(1):377--399, 2010.

\bibitem{MR2986590}
F.~Boyer and P.~Fabrie.
\newblock {\em Mathematical tools for the study of the incompressible
  {N}avier-{S}tokes equations and related models}, volume 183 of {\em Applied
  Mathematical Sciences}.
\newblock Springer, New York, 2013.

\bibitem{MR3892402}
B.~Buffoni and E.~Wahl\'{e}n.
\newblock Steady three-dimensional rotational flows: an approach via two stream
  functions and {N}ash-{M}oser iteration.
\newblock {\em Anal. PDE}, 12(5):1225--1258, 2019.

\bibitem{BSZ_2018}
J.~Burczak, Y.~Shibata, and W.~M. Zaj\c{a}czkowski.
\newblock Local and global solutions for the compressible {N}avier-{S}tokes
  equations near equilibria via the energy method.
\newblock In {\em Handbook of mathematical analysis in mechanics of viscous
  fluids}, pages 1751--1841. Springer, Cham, 2018.

\bibitem{Chandrasekhar_1957}
S.~Chandrasekhar.
\newblock {\em An introduction to the study of stellar structure}.
\newblock Dover Publications, Inc., New York, N.Y., 1957.

\bibitem{MR2372810}
G.-Q. Chen and Y.-G. Wang.
\newblock Existence and stability of compressible current-vortex sheets in
  three-dimensional magnetohydrodynamics.
\newblock {\em Arch. Ration. Mech. Anal.}, 187(3):369--408, 2008.

\bibitem{MR3925528}
G.-Q.~G. Chen, P.~Secchi, and T.~Wang.
\newblock Nonlinear stability of relativistic vortex sheets in
  three-dimensional {M}inkowski spacetime.
\newblock {\em Arch. Ration. Mech. Anal.}, 232(2):591--695, 2019.

\bibitem{CHWWY_2020}
R.~M. Chen, J.~Hu, D.~Wang, T.~Wang, and D.~Yuan.
\newblock Nonlinear stability and existence of compressible vortex sheets in
  2{D} elastodynamics.
\newblock {\em J. Differential Equations}, 269(9):6899--6940, 2020.

\bibitem{CDAD_2011}
Y.~Cho, J.~D. Diorio, T.~R. Akylas, and J.~H. Duncan.
\newblock {Resonantly forced gravity--capillary lumps on deep water. Part 2.
  Theoretical model}.
\newblock {\em J. Math. Fluid Mech.}, 672:288--306, 2011.

\bibitem{MR3139610}
D.~Coutand, J.~Hole, and S.~Shkoller.
\newblock Well-posedness of the free-boundary compressible 3-{D} {E}uler
  equations with surface tension and the zero surface tension limit.
\newblock {\em SIAM J. Math. Anal.}, 45(6):3690--3767, 2013.

\bibitem{Coutand_Shkoller_2011}
D.~Coutand and S.~Shkoller.
\newblock Well-posedness in smooth function spaces for moving-boundary 1-{D}
  compressible {E}uler equations in physical vacuum.
\newblock {\em Comm. Pure Appl. Math.}, 64(3):328--366, 2011.

\bibitem{MR2980528}
D.~Coutand and S.~Shkoller.
\newblock Well-posedness in smooth function spaces for the moving-boundary
  three-dimensional compressible {E}uler equations in physical vacuum.
\newblock {\em Arch. Ration. Mech. Anal.}, 206(2):515--616, 2012.

\bibitem{MR2214623}
S.~Dain.
\newblock Generalized {K}orn's inequality and conformal {K}illing vectors.
\newblock {\em Calc. Var. Partial Differential Equations}, 25(4):535--540,
  2006.

\bibitem{Denisova_1997}
I.~V. Denisova.
\newblock The problem of the motion of two compressible fluids separated by a
  closed free surface.
\newblock {\em Zap. Nauchn. Sem. S.-Peterburg. Otdel. Mat. Inst. Steklov.
  (POMI)}, 243(Kraev. Zadachi Mat. Fiz. i Smezh. Vopr. Teor. Funktsi\u{\i}.
  28):61--86, 338--339, 1997.

\bibitem{Denisova_2000}
I.~V. Denisova.
\newblock Evolution of compressible and incompressible fluids separated by a
  closed interface.
\newblock {\em Interfaces Free Bound.}, 2(3):283--312, 2000.

\bibitem{Denisova_2003}
I.~V. Denisova.
\newblock Solvability in weighted {H}\"{o}lder spaces for a problem governing
  the evolution of two compressible fluids.
\newblock {\em Zap. Nauchn. Sem. S.-Peterburg. Otdel. Mat. Inst. Steklov.
  (POMI)}, 295(Kraev. Zadachi Mat. Fiz. i Smezh. Vopr. Teor. Funkts.
  33):57--89, 244--245, 2003.

\bibitem{Denisova_Solonnikov_2018}
I.~V. Denisova and V.~A. Solonnikov.
\newblock Local and global solvability of free boundary problems for the
  compressible {N}avier-{S}tokes equations near equilibria.
\newblock In {\em Handbook of mathematical analysis in mechanics of viscous
  fluids}, pages 1947--2035. Springer, Cham, 2018.

\bibitem{DCDA_2011}
J.~D. Diorio, Y.~Cho, J.~H. Duncan, and T.~R. Akylas.
\newblock {Resonantly forced gravity--capillary lumps on deep water. Part 1.
  Experiments}.
\newblock {\em J. Math. Fluid Mech.}, 672:268--287, 2011.

\bibitem{Disconzi_Kukavica_2019}
M.~M. Disconzi and I.~Kukavica.
\newblock A priori estimates for the 3{D} compressible free-boundary {E}uler
  equations with surface tension in the case of a liquid.
\newblock {\em Evol. Equ. Control Theory}, 8(3):503--542, 2019.

\bibitem{Disconzi_Luo_2020}
M.~M. Disconzi and C.~Luo.
\newblock On the incompressible limit for the compressible free-boundary
  {E}uler equations with surface tension in the case of a liquid.
\newblock {\em Arch. Ration. Mech. Anal.}, 237(2):829--897, 2020.

\bibitem{MR2765512}
I.~Ekeland.
\newblock An inverse function theorem in {F}r\'{e}chet spaces.
\newblock {\em Ann. Inst. H. Poincar\'{e} C Anal. Non Lin\'{e}aire},
  28(1):91--105, 2011.

\bibitem{MR4275475}
I.~Ekeland and E.~S\'{e}r\'{e}.
\newblock A surjection theorem for maps with singular perturbation and loss of
  derivatives.
\newblock {\em J. Eur. Math. Soc. (JEMS)}, 23(10):3323--3349, 2021.

\bibitem{EvBS_2014}
Y.~Enomoto, L.~von Below, and Y.~Shibata.
\newblock On some free boundary problem for a compressible barotropic viscous
  fluid flow.
\newblock {\em Ann. Univ. Ferrara Sez. VII Sci. Mat.}, 60(1):55--89, 2014.

\bibitem{MR2040667}
E.~Feireisl.
\newblock {\em Dynamics of viscous compressible fluids}, volume~26 of {\em
  Oxford Lecture Series in Mathematics and its Applications}.
\newblock Oxford University Press, Oxford, 2004.

\bibitem{Fichera1973}
G.~Fichera.
\newblock {\em Existence Theorems in Elasticity}, pages 347--389.
\newblock Springer Berlin Heidelberg, Berlin, Heidelberg, 1973.

\bibitem{MR1814364}
D.~Gilbarg and N.~S. Trudinger.
\newblock {\em Elliptic partial differential equations of second order}.
\newblock Classics in Mathematics. Springer-Verlag, Berlin, 2001.
\newblock Reprint of the 1998 edition.

\bibitem{MR4072680}
D.~Ginsberg, H.~Lindblad, and C.~Luo.
\newblock Local well-posedness for the motion of a compressible,
  self-gravitating liquid with free surface boundary.
\newblock {\em Arch. Ration. Mech. Anal.}, 236(2):603--733, 2020.

\bibitem{Groves_2004}
M.~D. Groves.
\newblock {Steady water waves}.
\newblock {\em J. Nonlinear Math. Phys.}, 11(4):435--460, 2004.

\bibitem{GFA_2010}
M.~E. Gurtin, E.~Fried, and L.~Anand.
\newblock {\em The mechanics and thermodynamics of continua}.
\newblock Cambridge University Press, Cambridge, 2010.

\bibitem{Hadzic_Jang_2019}
M.~Had\v{z}i\'{c} and J.~J. Jang.
\newblock A class of global solutions to the {E}uler-{P}oisson system.
\newblock {\em Comm. Math. Phys.}, 370(2):475--505, 2019.

\bibitem{MR656198}
R.~S. Hamilton.
\newblock The inverse function theorem of {N}ash and {M}oser.
\newblock {\em Bull. Amer. Math. Soc. (N.S.)}, 7(1):65--222, 1982.

\bibitem{MR4406719}
S.~V. Haziot, V.~M. Hur, W.~A. Strauss, J.~F. Toland, E.~Wahl\'{e}n, S.~Walsh,
  and M.~H. Wheeler.
\newblock Traveling water waves---the ebb and flow of two centuries.
\newblock {\em Quart. Appl. Math.}, 80(2):317--401, 2022.

\bibitem{MR602181}
L.~H\"{o}rmander.
\newblock The boundary problems of physical geodesy.
\newblock {\em Arch. Rational Mech. Anal.}, 62(1):1--52, 1976.

\bibitem{MR802486}
L.~H\"{o}rmander.
\newblock On the {N}ash-{M}oser implicit function theorem.
\newblock {\em Ann. Acad. Sci. Fenn. Ser. A I Math.}, 10:255--259, 1985.

\bibitem{MR1039355}
L.~H\"{o}rmander.
\newblock The {N}ash-{M}oser theorem and paradifferential operators.
\newblock In {\em Analysis, et cetera}, pages 429--449. Academic Press, Boston,
  MA, 1990.

\bibitem{MR4266110}
Y.~Huang and T.~Luo.
\newblock Compressible viscous heat-conducting surface wave without surface
  tension.
\newblock {\em J. Math. Phys.}, 62(6):Paper No. 061501, 30, 2021.

\bibitem{MR3135704}
H.~Inci, T.~Kappeler, and P.~Topalov.
\newblock On the regularity of the composition of diffeomorphisms.
\newblock {\em Mem. Amer. Math. Soc.}, 226(1062):vi+60, 2013.

\bibitem{Jang_2010}
J.~Jang.
\newblock Local well-posedness of dynamics of viscous gaseous stars.
\newblock {\em Arch. Ration. Mech. Anal.}, 195(3):797--863, 2010.

\bibitem{Jang_Makino_2017}
J.~Jang and T.~Makino.
\newblock On slowly rotating axisymmetric solutions of the {E}uler-{P}oisson
  equations.
\newblock {\em Arch. Ration. Mech. Anal.}, 225(2):873--900, 2017.

\bibitem{Jang_Masmoudi_2009}
J.~Jang and N.~Masmoudi.
\newblock Well-posedness for compressible {E}uler equations with physical
  vacuum singularity.
\newblock {\em Comm. Pure Appl. Math.}, 62(10):1327--1385, 2009.

\bibitem{Jang_Masmoudi_2015}
J.~Jang and N.~Masmoudi.
\newblock Well-posedness of compressible {E}uler equations in a physical
  vacuum.
\newblock {\em Comm. Pure Appl. Math.}, 68(1):61--111, 2015.

\bibitem{MR3537008}
J.~Jang, I.~Tice, and Y.~Wang.
\newblock The compressible viscous surface-internal wave problem: local
  well-posedness.
\newblock {\em SIAM J. Math. Anal.}, 48(4):2602--2673, 2016.

\bibitem{MR3488552}
J.~Jang, I.~Tice, and Y.~Wang.
\newblock The compressible viscous surface-internal wave problem: stability and
  vanishing surface tension limit.
\newblock {\em Comm. Math. Phys.}, 343(3):1039--1113, 2016.

\bibitem{MR2164990}
B.~J. Jin.
\newblock Existence of viscous compressible barotropic flow in a moving domain
  with free upper surface, via {G}alerkin method.
\newblock {\em Ann. Univ. Ferrara Sez. VII (N.S.)}, 49:43--71, 2003.

\bibitem{MR1903004}
B.~J. Jin and M.~Padula.
\newblock In a horizontal layer with free upper surface.
\newblock {\em Commun. Pure Appl. Anal.}, 1(3):379--415, 2002.

\bibitem{MR2076684}
B.~J. Jin and M.~Padula.
\newblock Steady flows of compressible fluids in a rigid container with upper
  free boundary.
\newblock {\em Math. Ann.}, 329(4):723--770, 2004.

\bibitem{MR3281921}
U.~Kadri.
\newblock Wave motion in a heavy compressible fluid: revisited.
\newblock {\em Eur. J. Mech. B Fluids}, 49(part A):50--57, 2015.

\bibitem{koganemaru2022traveling}
J.~Koganemaru and I.~Tice.
\newblock Traveling wave solutions to the inclined or periodic free boundary
  incompressible navier-stokes equations.
\newblock {\em Preprint, arXiv:2207.07702}, 2022.

\bibitem{MR1892228}
P.~D. Lax.
\newblock {\em Functional analysis}.
\newblock Pure and Applied Mathematics (New York). Wiley-Interscience [John
  Wiley \& Sons], New York, 2002.

\bibitem{MR3726909}
G.~Leoni.
\newblock {\em A first course in {S}obolev spaces}, volume 181 of {\em Graduate
  Studies in Mathematics}.
\newblock American Mathematical Society, Providence, RI, second edition, 2017.

\bibitem{leoni2019traveling}
G.~Leoni and I.~Tice.
\newblock Traveling wave solutions to the free boundary incompressible
  navier-stokes equations.
\newblock {\em Comm. Pure Appl. Math.}, 2022.

\bibitem{Li_1991}
Y.~Y. Li.
\newblock On uniformly rotating stars.
\newblock {\em Arch. Rational Mech. Anal.}, 115(4):367--393, 1991.

\bibitem{Lindblad_2003}
H.~Lindblad.
\newblock Well-posedness for the linearized motion of a compressible liquid
  with free surface boundary.
\newblock {\em Comm. Math. Phys.}, 236(2):281--310, 2003.

\bibitem{MR2177323}
H.~Lindblad.
\newblock Well posedness for the motion of a compressible liquid with free
  surface boundary.
\newblock {\em Comm. Math. Phys.}, 260(2):319--392, 2005.

\bibitem{MR2178961}
H.~Lindblad.
\newblock Well-posedness for the motion of an incompressible liquid with free
  surface boundary.
\newblock {\em Ann. of Math. (2)}, 162(1):109--194, 2005.

\bibitem{Lindblad_Luo_2018}
H.~Lindblad and C.~Luo.
\newblock A priori estimates for the compressible {E}uler equations for a
  liquid with free surface boundary and the incompressible limit.
\newblock {\em Comm. Pure Appl. Math.}, 71(7):1273--1333, 2018.

\bibitem{MR1422251}
P.-L. Lions.
\newblock {\em Mathematical topics in fluid mechanics. {V}ol. 1}, volume~3 of
  {\em Oxford Lecture Series in Mathematics and its Applications}.
\newblock The Clarendon Press, Oxford University Press, New York, 1996.
\newblock Incompressible models, Oxford Science Publications.

\bibitem{Long_Morton_1966}
R.~Long and J.~Morton.
\newblock Solitary waves in compressible stratified fluids.
\newblock {\em Tellus}, 18(1):79--85, 1966.

\bibitem{MR40887}
M.~S. Longuet-Higgins.
\newblock A theory of the origin of microseisms.
\newblock {\em Philos. Trans. Roy. Soc. London Ser. A}, 243:1--35, 1950.

\bibitem{MR4439376}
C.~Luo and J.~Zhang.
\newblock Local well-posedness for the motion of a compressible gravity water
  wave with vorticity.
\newblock {\em J. Differential Equations}, 332:333--403, 2022.

\bibitem{MR3218831}
T.~Luo, Z.~Xin, and H.~Zeng.
\newblock Well-posedness for the motion of physical vacuum of the
  three-dimensional compressible {E}uler equations with or without
  self-gravitation.
\newblock {\em Arch. Ration. Mech. Anal.}, 213(3):763--831, 2014.

\bibitem{MR882389}
T.~Makino.
\newblock On a local existence theorem for the evolution equation of gaseous
  stars.
\newblock In {\em Patterns and waves}, volume~18 of {\em Stud. Math. Appl.},
  pages 459--479. North-Holland, Amsterdam, 1986.

\bibitem{MR3379135}
T.~Makino.
\newblock On spherically symmetric motions of the atmosphere surrounding a
  planet governed by the compressible {E}uler equations.
\newblock {\em Funkcial. Ekvac.}, 58(1):43--85, 2015.

\bibitem{MR3569407}
T.~Makino.
\newblock An application of the {N}ash-{M}oser theorem to the vacuum boundary
  problem of gaseous stars.
\newblock {\em J. Differential Equations}, 262(2):803--843, 2017.

\bibitem{MD_2017}
N.~Masnadi and J.~H. Duncan.
\newblock The generation of gravity–capillary solitary waves by a pressure
  source moving at a trans-critical speed.
\newblock {\em J. Fluid Mech.}, 810:448–474, 2017.

\bibitem{MR713680}
A.~Matsumura and T.~Nishida.
\newblock Initial-boundary value problems for the equations of motion of
  compressible viscous and heat-conductive fluids.
\newblock {\em Comm. Math. Phys.}, 89(4):445--464, 1983.

\bibitem{MOM_1995}
v.~Matu\v{s}u-Ne\v{c}asov\'{a}, M.~Okada, and T.~Makino.
\newblock Free boundary problem for the equation of spherically symmetric
  motion of viscous gas. {II}.
\newblock {\em Japan J. Indust. Appl. Math.}, 12(2):195--203, 1995.

\bibitem{MR2457601}
V.~G. Maz'ya and T.~O. Shaposhnikova.
\newblock {\em Theory of {S}obolev multipliers}, volume 337 of {\em Grundlehren
  der mathematischen Wissenschaften [Fundamental Principles of Mathematical
  Sciences]}.
\newblock Springer-Verlag, Berlin, 2009.
\newblock With applications to differential and integral operators.

\bibitem{MKS_1990}
R.~Miesen, L.~Kamp, and F.~Sluijter.
\newblock Long solitary waves in compressible shallow fluids.
\newblock {\em Physics of Fluids A - Fluid Dynamics}, 2(3):359--370, MAR 1990.

\bibitem{MKS_1990_2}
R.~Miesen, L.~Kamp, and F.~Sluijter.
\newblock Solitary waves in compressible deep fluids.
\newblock {\em Physics of fluids A - Fluid Dynamics}, 2(8):1404--1411, AUG
  1990.

\bibitem{Moser_1966}
J.~Moser.
\newblock A rapidly convergent iteration method and non-linear partial
  differential equations. {I}.
\newblock {\em Ann. Scuola Norm. Sup. Pisa Cl. Sci. (3)}, 20:265--315, 1966.

\bibitem{nash_1956}
J.~Nash.
\newblock The imbedding problem for {R}iemannian manifolds.
\newblock {\em Ann. of Math. (2)}, 63:20--63, 1956.

\bibitem{nguyen_tice_2022}
H.~Nguyen and I.~Tice.
\newblock Traveling wave solutions to the one-phase muskat problem: existence
  and stability.
\newblock {\em Preprint, arXiv:2211.06286}, 2022.

\bibitem{MR2084891}
A.~Novotn\'{y} and I.~Stra\v{s}kraba.
\newblock {\em Introduction to the mathematical theory of compressible flow},
  volume~27 of {\em Oxford Lecture Series in Mathematics and its Applications}.
\newblock Oxford University Press, Oxford, 2004.

\bibitem{Makino_1993}
M.~Okada and T.~Makino.
\newblock Free boundary problem for the equation of spherically symmetric
  motion of viscous gas.
\newblock {\em Japan J. Indust. Appl. Math.}, 10(2):219--235, 1993.

\bibitem{PC_2016}
B.~Park and Y.~Cho.
\newblock {Experimental observation of gravity--capillary solitary waves
  generated by a moving air suction}.
\newblock {\em J. Math. Fluid Mech.}, 808:168--188, 2016.

\bibitem{PC_2018}
B.~Park and Y.~Cho.
\newblock {Two-dimensional gravity--capillary solitary waves on deep water:
  generation and transverse instability}.
\newblock {\em J. Math. Fluid Mech.}, 834:92--124, 2018.

\bibitem{MR1046283}
K.~Pileckas and W.~M. Zaj\c{a}czkowski.
\newblock On the free boundary problem for stationary compressible
  {N}avier-{S}tokes equations.
\newblock {\em Comm. Math. Phys.}, 129(1):169--204, 1990.

\bibitem{MR2963679}
P.~Plotnikov and J.~Soko{\l}owski.
\newblock {\em Compressible {N}avier-{S}tokes equations}, volume~73 of {\em
  Instytut Matematyczny Polskiej Akademii Nauk. Monografie Matematyczne (New
  Series) [Mathematics Institute of the Polish Academy of Sciences.
  Mathematical Monographs (New Series)]}.
\newblock Birkh\"{a}user/Springer Basel AG, Basel, 2012.
\newblock Theory and shape optimization.

\bibitem{MR1854060}
P.~I. Plotnikov and J.~F. Toland.
\newblock Nash-{M}oser theory for standing water waves.
\newblock {\em Arch. Ration. Mech. Anal.}, 159(1):1--83, 2001.

\bibitem{Schwartz_1960}
J.~Schwartz.
\newblock On {N}ash's implicit functional theorem.
\newblock {\em Comm. Pure Appl. Math.}, 13:509--530, 1960.

\bibitem{Schwartz_1969}
J.~T. Schwartz.
\newblock {\em Nonlinear functional analysis}.
\newblock Notes on Mathematics and its Applications. Gordon and Breach Science
  Publishers, New York-London-Paris, 1969.
\newblock Notes by H. Fattorini, R. Nirenberg and H. Porta, with an additional
  chapter by Hermann Karcher.

\bibitem{MR697305}
P.~Secchi and A.~Valli.
\newblock A free boundary problem for compressible viscous fluids.
\newblock {\em J. Reine Angew. Math.}, 341:1--31, 1983.

\bibitem{Sergeraert_1972}
F.~Sergeraert.
\newblock Un th\'{e}or\`eme de fonctions implicites sur certains espaces de
  {F}r\'{e}chet et quelques applications.
\newblock {\em Ann. Sci. \'{E}cole Norm. Sup. (4)}, 5:599--660, 1972.

\bibitem{Shibata_2016}
Y.~Shibata.
\newblock On the global well-posedness of some free boundary problem for a
  compressible barotropic viscous fluid flow.
\newblock In {\em Recent advances in partial differential equations and
  applications}, volume 666 of {\em Contemp. Math.}, pages 341--356. Amer.
  Math. Soc., Providence, RI, 2016.

\bibitem{MR1226506}
V.~A. Solonnikov and A.~Tani.
\newblock Evolution free boundary problem for equations of motion of viscous
  compressible barotropic liquid.
\newblock In {\em The {N}avier-{S}tokes equations {II}---theory and numerical
  methods ({O}berwolfach, 1991)}, volume 1530 of {\em Lecture Notes in Math.},
  pages 30--55. Springer, Berlin, 1992.

\bibitem{MR0290095}
E.~M. Stein.
\newblock {\em Singular integrals and differentiability properties of
  functions}.
\newblock Princeton Mathematical Series, No. 30. Princeton University Press,
  Princeton, N.J., 1970.

\bibitem{MR4337506}
N.~Stevenson and I.~Tice.
\newblock Traveling wave solutions to the multilayer free boundary
  incompressible {N}avier-{S}tokes equations.
\newblock {\em SIAM J. Math. Anal.}, 53(6):6370--6423, 2021.

\bibitem{Strauss_2010}
W.~A. Strauss.
\newblock {Steady water waves}.
\newblock {\em Bull. Amer. Math. Soc. (N.S.)}, 47(4):671--694, 2010.

\bibitem{Strauss_Wu_2019}
W.~A. Strauss and Y.~Wu.
\newblock Rapidly rotating stars.
\newblock {\em Comm. Math. Phys.}, 368(2):701--721, 2019.

\bibitem{MR2004291}
N.~Tanaka and A.~Tani.
\newblock Surface waves for a compressible viscous fluid.
\newblock {\em J. Math. Fluid Mech.}, 5(4):303--363, 2003.

\bibitem{Toland_1996}
J.~F. Toland.
\newblock {Stokes waves}.
\newblock {\em Topol. Methods Nonlinear Anal.}, 7(1):1--48, 1996.

\bibitem{MR2560044}
Y.~Trakhinin.
\newblock Local existence for the free boundary problem for nonrelativistic and
  relativistic compressible {E}uler equations with a vacuum boundary condition.
\newblock {\em Comm. Pure Appl. Math.}, 62(11):1551--1594, 2009.

\bibitem{MR4444136}
Y.~Trakhinin and T.~Wang.
\newblock Well-posedness for the free-boundary ideal compressible
  magnetohydrodynamic equations with surface tension.
\newblock {\em Math. Ann.}, 383(1-2):761--808, 2022.

\bibitem{MR0119656}
J.~V. Wehausen and E.~V. Laitone.
\newblock {\em Surface waves}.
\newblock Handbuch der Physik, Vol. 9, Part 3. Springer-Verlag, Berlin, 1960.

\end{thebibliography}
 	
\end{document}